\documentclass[11pt]{amsart}

\usepackage{latexsym,graphicx,amssymb}
\usepackage[tips,matrix,arrow]{xy}
\usepackage{etex}
\usepackage{tikz}
\usepackage{diagbox}
\usepackage{adjustbox}

\usepackage[T1]{fontenc}

\usepackage{xcolor}
\usepackage{ulem}

\definecolor{violet}{cmyk}{1,0,0,0}

\usepackage{hyperref}

\usepackage{setspace}
\usepackage{imakeidx}
\makeindex[name=notations, title=List of Notations]
\makeindex[name=index, title=Index]

\newcommand{\C}{\mathbb{C}}

\newcommand{\R}{\mathbb{R}}
\newcommand{\Z}{\mathbb{Z}}

\newcommand{\fS}{\mathfrak{S}}

\newcommand{\Aut}{\mathrm{Aut}}
\newcommand{\End}{\mathrm{End}}

\newcommand{\id}{\mathrm{id}}

\newcommand{\af}{\mathrm{aff}}

\newcommand{\nr}{\mathrm{nr}}

\newcommand{\dc}{\,\underline{\mathrm{d}}^c}
\newcommand{\mc}{\, \underline{\mathrm{m}}^c}
\newcommand{\W[1]}{W( #1 )}

\newcommand{\resrad}{\mathrm{res}_{\mathrm{rad}(I_F)}}

\newcommand{\w}{\mathrm{w}}
\newcommand{\nw}{\mathrm{nw}}

\newtheorem{thm}{Theorem}[section]
\newtheorem{prop}[thm]{Proposition}
\newtheorem{lemma}[thm]{Lemma}
\newtheorem{cor}[thm]{Corollary}
\newtheorem{defn}{Definition}[section]

\newtheorem{ex}[defn]{Example}
\newtheorem{rem}[thm]{Remark}
\newtheorem{prob}{Problem}
\newtheorem{ack}{Acknowledgement}

\newcommand{\rad}{\operatorname{rad}}

\newcommand{\fg}{\mathfrak{g}}

\newcommand{\vep}{\varepsilon}

\definecolor{mygreen}{cmyk}{1,0,1,0}

\newcommand{\node}[1]{\begin{tikzpicture} \draw [fill=gray!#1] (0,0) circle(0.1); \end{tikzpicture}}

\newcommand{\white}{\begin{tikzpicture} \draw (0,0) circle(0.1); \end{tikzpicture}}

\newcommand{\black}{\begin{tikzpicture} \fill (0,0) circle(0.1); \end{tikzpicture}}
\newcommand{\gray}{\begin{tikzpicture} \draw [black,fill=gray!70] (0,0) circle(0.1); \end{tikzpicture}}

\begin{document}
\title[Classification of mERS with non-reduced quotient]
{Classification of marked elliptic root systems with non-reduced quotient}
\author{A. Fialowski, K. Iohara and Y. Saito}
\address{Alice Fialowski, Department Computer Algebra, 
E\"{o}tv\"{o}s Lor\'{a}nd University, H-1117 Budapest, 
P\'{a}zm\'{a}ny P\'{e}ter s\'{e}t\'{a}ny 1/C, Budapest, Hungary}
\email{alice.fialowski@gmail.com, fialowski@inf.elte.hu}
\address{Kenji Iohara, Universit\'{e} Claude Bernard Lyon 1, CNRS, Institut Camille Jordan UMR 5208, F-69622 Villeurbanne, France}
\email{iohara@math.univ-lyon1.fr}
\address{Yoshihisa Saito, Department of Mathematics, Rikkyo University, Toshima-ku, Tokyo 171-8501, Japan}
\email{yoshihisa@rikkyo.ac.jp}

\thanks{This research was partly supported through the program "Oberwolfach Research Fellows"  by the Mathematisches Forschungsinstitut Oberwolfach in 2023.   Y .S. is partially supported by JSPS KAKENHI Grant Numbers JP20K03568 and JP25K06942.\\}

\begin{abstract} K. Saito (Publ. RIMS 21 (1), 1985, 75-179) has introduced a class of root systems called elliptic root systems which lie in the real vector space $F$ with a metric $I$ whose signature is of type $(l,2,0)$. He also classified the pair $(R,G)$ of  elliptic root systems $R$ with one dimensional subspaces $G$ of the radical of $I$, under the assumption that the quotient root system $R/G$ is reduced. Here, we classify the pairs $(R,G)$ where $R/G$ is non-reduced. 
\end{abstract}


\maketitle

\begin{center}
{\it To the memory of  Professor Kenji Iohara {\rm(}1970-2025{\rm)}}\footnote{We are 
deeply shocked that our good friend and collaborator, a bright and extremely productive 
mathematician and wonderful person, Kenji Iohara passed away on October 1, 2025 
at the age of 55.}.
\end{center}
\vskip 0.3in

\setcounter{tocdepth}{2}
\tableofcontents

\section*{Introduction}

The notion of root system of a Lie algebra has a long history. Originally it was introduced by W. Killing \cite{Killing1888, Killing1889a, Killing1889b, Killing1890} in 1888, in order to classify all simple finite dimensional Lie algebras over $\C$. His work was continued and completed by \'{E}. Cartan \cite{Cartan1894} in 1894. In the early 1930s,  H.S.M. Coxeter \cite{Coxeter1934} initiated the study of discrete groups  generated by reflections, which are now called Coxeter groups. Then comes a miraculous year 1934-35\footnote{See, e.g., MacTutor: \url{https://mathshistory.st-andrews.ac.uk}.}; Coxeter visited Princeton and attended lectures by H. Weyl. They soon realized a surprising connection between Coxeter's finite groups and Weyl's infinite groups, i.e., Coxeter's groups show up as the Weyl group of root systems. At the same time, P. DuVal had been visiting Princeton and  studied isolated rational surface singularities \cite{DuVal1934-I, DuVal1934-II, DuVal1934-III}. They found a curious resemblance between Coxeter graphs and what are today called simple singularities. Hence, a fabulous discovery of connections between Lie groups, Coxeter groups and simple singularities happened in 1934 !
In 1970 at the ICM Congress in Nice, E. Brieskorn \cite{Brieskorn1971} showed,
 that simple singularities of type A, D, E and their semi-universal deformations can be realized purely in terms of simple algebraic groups of type A, D, E over an algebraically closed field with good characteristic. Later P. Slodowy \cite{Slodowy1980a, Slodowy1980b} gave an accessible account of the Brieskorn theory and slightly generalized it, in particular, he also described them in terms of simple Lie algebras. \\

H. Coxeter studied properties of a product of generators (cf. \cite{Coxeter1951}), which plays a prominent role in invariant theory. Any such product is now called a Coxeter transformation. For a simple finite dimensional Lie algebra, B. Kostant \cite{Kostant1959} established a direct relationship between its principal simple 3-dimensional subalgebra, its adjoint representation, and the Coxeter transformation. For more historical account for finite root systems and Coxeter transformations see \cite{Bourbaki(Lie)4-6}. We note that the finiteness of the order of a Coxeter transformation plays an important role in the geometry of isolated surface singularities.  \\

There is a natural extension of a finite Weyl group, which is a semi-direct product of 
some lattice and a finite Weyl group. Such a group is called an affine Weyl group. 
In 1965, N. Iwahori and H. Matsumoto \cite{IwahoriMatsumoto1965} found that an affine 
Weyl group can be realized as the Weyl group of a Chevalley group over a $p$-adic field. 
A few years later, V. Kac \cite{Kac1969} and R. Moody \cite{Moody1969} invented a new 
class of infinite dimensional Lie algebras, called affine Lie algebras, whose Weyl group is 
an affine Weyl group. An affine Lie algebra can also be defined with a generalized Cartan 
matrix, introduced by R. Moody \cite{Moody1968}, and their complete classification was 
achieved by E.B.Vinberg \cite{Vinberg1971}. A. J. Coleman \cite{Coleman1989} studied 
Coxeter transformations for a wider class of Lie algebras, including the affine ones. 
Note that for an affine Weyl group, the Coxeter transformation is of infinite order. 
There are also other approaches to the classification of affine root systems, like one in 
relation with the classification of reductive groups over a local field by F. Bruhat and J. Tits 
\cite{BruhatTits1972}, another in relation with the denominator identity and 
powers of the Dedekind $\eta$-function by I. Macdonald \cite{Macdonald1972}. 
Both classifications give the same result, and the complete classification was given in 
Tits' paper \cite{Tits1979}. A remarkable feature of affine Lie algebras is that the 
character of their integrable highest weight representations can be described in terms of 
the Dedekind $\eta$-function and $\theta$-functions, as was studied by V. Kac and D. 
Peterson \cite{KacPeterson1984}. 
This {phenomenon} was one of the key ingredients in the development 
of theoretical physics such as conformal field theory. \\
 
In 1974 K. Saito \cite{Saito1974} studied and classified simple elliptic singularities. 
Inspired by Brieskorn's theory,he 
initiated 
a research on elliptic root systems in a series of papers {from} the middle of 1980s (cf.  \cite{Saito1985}).
An elliptic root system (called extended affine root system) is a generalized root system realized in
a real vector space of finite dimension equipped with a symmetric bilinear form $I$ of signature
$(l,2,0)$. He gave the classification of marked elliptic root systems $(R,G)$, which are pairs of 
an elliptic root system $R$ and  a $1$-dimensional subspace $G$ of the radical of $I$, 
under the condition that the quotient root system $R/G$ is a 
reduced affine root system. To each pair $(R,G)$, he attached a finite oriented diagram, 
called an elliptic diagram (first called an elliptic Dynkin diagram), and defined an element of the Weyl group, called
a Coxeter element. \\

The important part of this theory is that a Coxeter element here is an element of the Weyl group of finite order. Since then Saito has been added new ideas to his theory in a series of papers, like for example the structure of Weyl groups \cite{SaitoTakebayashi1997} or $\eta$-product \cite{Saito2001}. For each marked elliptic root system $(R,G)$, U. Pollmann, in her master thesis \cite{Pollmann1994}, showed that there exists a Lie algebra whose real root system is isomorphic to $R$. Indeed, she constructed such a Lie algebra via the twisted construction of the pair $(\fg, \sigma)$ of an affine Lie algebra $\fg$ with its finite order diagram automorphism $\sigma$ admitting a fixed point. She also examined $6$ other cases when the finite order diagram automorphisms do not have a fixed point; there are $2$ cases, whose real root systems do not correspond to those in the list of K. Saito \cite{Saito1985}. Here, we shall mention that Allison et al. \cite{AllisonBermanPianzola2014} studied the extended affine Lie algebras (cf. \cite{AllisonAzamBermanGaoPianzola1997}) with nullity $2$  in 2014, which are naturally related to the above Lie algebras. Apparently, such Lie algebras contain certain class of Lie algebras graded by a finite root system, among which $BC_l$-graded cases were studied in \cite{AllisonBenkartGao2002} for $l\geq 2$ and in \cite{BenkartSmirnov2003} for $l=1$. 
However, it is not so clear whether these Lie algebras and the root systems considered in this memoir are related or not.
\\

Beside these works, we mention 3 other directions of elliptic root systems. 
Firstly, Y. Billig \cite{Billig1999}  and K. Iohara, Y. Saito and M. Wakimoto 
\cite{IoharaSaitoWakimoto1999a, IoharaSaitoWakimoto1999b} in 1999 obtained 
soliton equations arriving from a representation of the Lie algebras of type 
$A_l^{(1,1)}, D_l^{(1,1)}, E_6^{(1,1)}, E_7^{(1,1)}$ and $E_8^{(1,1)}$. 
The other direction is the work of Helmke and P. Slodowy in 2004 
\cite{HelmkeSlodowy2004} in which they show the connections of loop groups and 
elliptic Lie algebras via principle bundles over an elliptic curve, in view of simply elliptic 
singularities (see \cite{HelmkeSlodowy2005} for an overview). 
The third direction is elliptic Hecke algebras by Y. Saito and M. Shiota in 2009 
\cite{SaitoShiota2009}. 
They studied the connection with the double affine Hecke algebra, which plays an 
important role to analyze Macdonald's polynomials (cf. \cite{Macdonald2003}). \\


In this article, we complete the classification of marked elliptic root systems $(R,G)$, namely, we classify the so far missing cases where the quotient root system $R/G$ is a non-reduced affine root system. We found 6 series of new reduced marked elliptic root systems (see Theorem \ref{thm_classification-reduced}, and Section \ref{sect_root-data-red} for their root data), among which 4  have been already found by K. Saito \cite{Saito1985}, namely $BC_l^{(1,2)}, BC_l^{(4,2)}, BC_l^{(2,2)\sigma}(1)$ and $BC_l^{(2,2)\sigma}(2)$, but with different marking (see \S \ref{sect:isom-root-sys}) and 34 series and 1 exceptional new non-reduced marked elliptic root systems (see Theorem \ref{thm_classification-non-reduced} and Sections \ref{sect_root-data-non-reduced-1}, \ref{sect_root-data-non-reduced-2}, \ref{sect_root-data-non-reduced-3}, and \ref{sect_root-data-non-reduced-4} for their root data).  It turns out that the $2$ missing cases in Pollmann's construction mentioned above, correspond to the $2$ 
newly found reduced marked elliptic root systems $(R,G)$ with non-reduced affine quotient $R/G$ ! \\

The logical dependencies of the main results of this memoir can be described as follows.
Theorems \ref{thm_classification-reduced}  and \ref{thm_classification-non-reduced} provide a list of marked elliptic root systems with non-reduced affine quotient, whose proofs are given by gluing K. Saito's marked elliptic root systems, thus their explicit construction is given. At this moment, we don't state whether any two marked elliptic root systems in the above are non-isomorphic as marked root systems or not. In Theorems \ref{thm:isom-red-mERS-nred_root-sys} and \ref{thm:isom-nred-mERS-nred_root-sys}, we classify the isomorphism classes of the marked elliptic root systems presented in the above two theorems, as root systems. 
These last theorems allow us to compare the list of our elliptic root systems with those obtained in \cite{Azam2002} and \cite{AzamKhaliliYousofzadeh2005}, as stated in Remark \ref{rem-comparison-Azam-etc}.
After the introduction of the notion of elliptic diagram in \S \ref{sect_elliptic-diagram-II}, where its well-definedness is assured in Proposition \ref{prop:indendence-Gamma}, we again classify the marked elliptic root systems with non-reduced affine quotient (cf. Corollary \ref{cor:classification-thm}) via their diagrams in Theorem \ref{thm:classification-thm}. This latter theorem states possible marked elliptic root systems up to isomorphism, as marked root systems, but their existence is not discussed. 
Comparing this last theorem with Theorems \ref{thm_classification-reduced}  and \ref{thm_classification-non-reduced}, where an explicit construction of each marked elliptic root system is given, we obtain the classification of marked elliptic root systems with non-reduced affine quotient. Thus, our results together with Theorem \ref{thm:classification-red}, due to K. Saito \cite{Saito1985}, 
allow us to conclude the complete classification of marked elliptic root systems. The explicit root data are given in Chapter 4. \\

The structure of the paper is the following. 
In Chapter \ref{chapter:review}, we provide a brief review of the known theory on elliptic root systems.
In Section \ref{sect:general} we recall the basic definitions. In Section \ref{sect:mers} we recall the technical tools necessary for the classification of reduced marked elliptic root systems from K.Saito. In particular, we present his classification with details of the proof to show how it works, and obtain the classification of marked elliptic root systems $(R,G)$ for which its quotient $R/G$ is reduced. For later use, we describe the structure of the group of automorphisms of a marked elliptic root systems $(R,G)$ in Section \ref{sect:auto-mers-reduced-case}. \\
In Chapter \ref{chapter:new-mERS}, we  provide the list of all marked elliptic root systems $(R,G)$ whose affine quotient $R/G$ is non-reduced. 
In Section \ref{sect:MERS-non-red-quot} we describe all the possible marked elliptic root systems $(R,G)$ with non-reduced affine quotient $R/G$,
and state a weak classification theorem, namely, a theorem which states that any such $(R,G)$ is isomorphic to one of those presented in this section. Section \ref{sect_mers-II} is devoted to the proof of these theorems. In Section \ref{sect:more-mERS}, 
we discuss the isomorphism classes of the elliptic root systems, as root systems. For each marked elliptic root system $(R,G)$, 
possible affine quotients of $R$ (not necessarily by $G$) are also given. \\
In Chapter \ref{chapter:classification}, we prove that any two of the marked elliptic root systems $(R,G)$ presented in Chapter \ref{chapter:new-mERS} are non isomorphic. 
For this purpose, another proof of the classification theorem of the marked elliptic root systems, that is completely independent of the previous one, is given, which allows us to associate each marked elliptic root system to a unique diagram. Section \ref{sect:elliptic-diagram} is a preparation for the proof of these theorems and in Section \ref{sect:main}, a strong classification theorem is  given. 
Chapter \ref{chapter:root-data} is the summary of the root systems we obtained in this 
article, namely, we collect the root data of the newly obtained marked elliptic root 
systems in Sections \ref{sect_root-data-red} - \ref{sect_root-data-non-reduced-4}. We 
conclude this article with three Appendices \ref{sect_finite-root-system}, 
\ref{sect_affine-root-system} and \ref{sect_clasfficiation-red-quot}, where we collect 
 root data for the finite, affine root systems and marked elliptic root systems with reduced 
affine quotient.\\
 
{
The present article is a first step in the study of (marked) elliptic root systems with 
non-reduced affine quotients. Therefore, there are many problems that need to be 
studied. For example, the following topics should be considered:
\begin{itemize}
\item[(i)] Coxeter elements for a (marked) elliptic root system $(R,G)$ with 
non-reduced affine quotients;
\item[(ii)] detailed structure of the Weyl group $W(R)$ and the group $\mathrm{Aut}(R)$
of automorphisms for such an $(R,G)$;
\item[(iii)] invariant theory for such an $(R,G)$. 
\end{itemize} 
For (i) and (ii), the results will appear in our forthcoming paper (cf. Remark 
\ref{rem:reconstruction} below).  For (iii), we have to discuss the relationship between 
our results and those of Looijenga
and van der Lek.  
Motivated by his work on semi-universal deformations of simple
elliptic singularities, Looijenga introduced an {\it extended Coxeter group}
$\widetilde{W}$ which is a semidirect product of a Coxeter group $W$ and 
a root lattice $Q$, and a certain tube domain $\Omega$ which has a natural action of 
$\widetilde{W}$ (\cite{Looijenga1976}, \cite{Looijenga1980}). 
When $W$ is an affine Weyl group, the orbit space $\Omega/\widetilde{W}$
appears also in K. Saito's theory by using the language of marked elliptic root 
systems with reduced affine quotients (see \cite{Saito1990})\footnote{In \cite{Saito1990}, 
this orbit space is denoted by $\widetilde{\mathbb{E}}/\widetilde{W}_R$. 
The method for identifying these two spaces is not 
necessarily obvious. However, since it is not relevant to the results of this paper, 
we will omit to give it in detail.}. After Looijenga's work, H. van der Lek 
computed the fundamental group of the regular $\widetilde{W}$-orbit space of 
$\Omega$ (see \cite{Lek1983}).  

Based on our results in this article, the invariant theory and the study on Artin group for 
$(R,G)$ should be extended to the case when $(R,G)$ has non-reduced affine 
quotient. We believe that this is an important problem in this area, and deserves 
future research.

\begin{ack}
{\rm The authors are grateful to the referees for their valuable comments. }
\end{ack}


}

\part{Marked elliptic root systems with reduced quotient}\label{chapter:review}

 In 1985, K. Saito \cite{Saito1985} introduced what one calls now elliptic root systems (\textit{n\'{e}} extended affine root systems). 
  The main purpose of this chapter is to provide a brief review of elliptic root systems and some of the key concepts,  among which marking, tier numbers, counting sets and counting numbers play a key role in the main part of this article. In particular, the classification theorem of marked elliptic root systems with reduced affine quotient is given, to show how the concepts like counting numbers work in the classification. \\
  
  In Section \ref{sect:general}, we recall the notion of generalized root system, state the relations between finite, affine and elliptic root systems, and present the finite root system of type $BC_l$, since this is one of the main objects in this article.  The Dynkin diagrams of finite and affine root systems are also reviewed. Section \ref{sect:mers} is devoted to several concepts in relation with marking such as tier numbers, countings and elliptic diagrams. In particular, we state the classification theorem of marked elliptic root systems with reduced affine quotient and give a proof of it. 
In Section  \ref{sect:auto-mers-reduced-case}, we {determine 
the structure of the automorphism groups of marked elliptic root systems with reduced 
affine quotient. This results will be used in Part 2, for studying marked elliptic root 
systems with non-reduced affine quotient. }



\section{Generalized Root Systems}\label{sect:general}

\subsection{General framework}\label{subsec:general}

\smallskip

Recall some important notions from K Saito \cite{Saito1985}. 
\smallskip

Let $(F,I)$ 
be a pair of a vector space $F$ over $\R$ \index[notations]{r111@$R$} and  a symmetric bilinear form of finite rank $I: F\times F \rightarrow \R$ with signature $(\mu_+, \mu_0, \mu_-)$. For a nonisotropic vector $\alpha \in F$ (i.e. $I(\alpha, \alpha) \neq 0$), define its \textbf{dual} by
$$
\alpha^\vee=\frac{2}{I(\alpha,\alpha)}\alpha 
$$
and the \textbf{reflection} $r_\alpha \in \mathrm{O}(F,I)$ by
$$
r_\alpha(\lambda)=\lambda-I(\alpha^\vee, \lambda)\alpha
$$
where $\mathrm{O}(F,I):={\{\,g \in GL(F)\,\vert\, I \circ g=I\,\}}$ is the orthogonal group of 
the metric $I$. In this article, we use the word \textit{dual} to distinguish from the 
preposition \textit{co} which indicates an object in the dual vector space.
\begin{defn}\label{defn_root-sys}
A subset $R \subset F \setminus \{0\}$ is called a (generalized) 
\textbf{root system} \index[index]{root system} belonging to $(F,I)$ if it satisfies the following:
\begin{enumerate}
\item[(R1)] the subgroup of $F$ generated by $R$, $Q(R):={\Z}R$ (the root lattice of $R$) is a full lattice in $F$, i.e., $\R \otimes_{\Z} Q(R) \cong F$, 
\item[(R2)] for any $\alpha \in R$, $I(\alpha, \alpha) \neq 0$, 
\item[(R3)] for any $\alpha \in R$, $r_\alpha(R)=R$,
\item[(R4)] for any $\alpha, \beta \in R$, $I(\alpha^\vee,\beta) \in \Z$,  
\item[(R5)] irreducibility: if there is a decomposition $R=R_1 \amalg R_2$ and $R_1 \perp R_2$ with respect to $I$ for some subsets $R_i \subset R\; (i=1,2)$, then either $R_1=\emptyset$ or $R_2=\emptyset$.
\end{enumerate}
\end{defn}

It is clear that if $R$ is a root system belonging to $(F,I)$, then so does its \textbf{dual root system}\index[index]{root system!dual@dual --} $R^\vee:=\{\alpha^\vee\, \vert \, \alpha \in R\}$. 

The condition (R3) implies that $\alpha \in F$ belongs to $R$ iff so does 
$-\alpha$. From condition (R4) it follows that if $\alpha \in R$ and $c \alpha \in R$ 
for some $c \in \R$, then $c \in \{ \pm \frac{1}{2}, \pm 1, \pm 2\}$.  
A root system $R$ is said to be 
\textbf{reduced}\index[index]{root system!reduced@reduced --}
if $\alpha \in R$ and $c \alpha \in R$ for some $c \in \R$ implies $c \in \{ \pm 1\}$. Otherwise, $R$ is said to be \textbf{non-reduced}. 
\index[index]{root system!non-reduced@non-reduced --}
\smallskip

The subgroup $W(R)$\index[notations]{w111@$W(R)$} of $\mathrm{O}(F,I)$ generated by the 
reflections $r_\alpha$ for 
$\alpha \in R$ is called the \textbf{Weyl group}\index[index]{Weyl group} of the root 
system $R$. The reflections are clearly the same for the dual root system.

Two root systems $R_1\subset (F_1,I_1)$ and $R_2 \subset (F_2,I_2)$ are \textbf{isomorphic} iff there exists a linear isomorphism 
$\varphi:F_1 \rightarrow F_2$ such that $\varphi(R_1)=R_2$. 
For a root system $R\subset (F,I)$, a linear isomorphism 
$\varphi\in GL(F)$ is called an \textbf{automorphism} of $R$, if 
$\varphi(R)=R$. The group of automorphisms of $R$ is denoted by 
$\mathrm{Aut}(R)$\index[notations]{aut110@$\mathrm{Aut}(R)$}. 
The following lemma is due to K. Saito.
\begin{lemma}[\cite{Saito1985}]\label{lemma:Aut(R)}
The group $\mathrm{Aut}(R)$ is a subgroup of $\mathrm{O}(F,I)$. 
\end{lemma}

In this article, we sometimes denote $I_F$ in place of $I$ to indicate the underlying vector space $F$ on which the symmetric bilinear form $I$ is defined.

\bigskip
 
We consider the case when $\mu_-=0$.
\begin{enumerate}
\item the case $\mu_0=0$ corresponds to the \textbf{finite (classical) root systems};
\index[index]{root system!finite@finite --} this means there are finitely many roots and the Weyl group is finite,
\item the case $\mu_0=1$ corresponds to the \textbf{affine root systems},
\index[index]{root system!affine@affine --} 
\item the case $\mu_0=2$ corresponds to the so-called \textbf{elliptic root systems}.
\index[index]{root system!elliptic@elliptic --}
\end{enumerate}
In such a case, it follows that
\[ I(\alpha^\vee,\beta) \in \Z \cap [-4,4] \qquad \forall\; \alpha, \beta \in R.
\]
\subsection{Dynkin diagrams: finite and affine cases} 
Let $R$ be a finite or affine root system. There exists a minimal generating system $\Pi \subset R$ of $F$ such that for any root $\alpha=\sum_{\beta \in \Pi}c_\beta \beta \in R$, either $c_\beta \in \Z_{\geq 0}$ for all $\beta \in \Pi $  or  $c_\beta \in \Z_{\leq 0}$ for all $\beta \in \Pi$ holds. For a finite root system, any such system $\Pi'$
is $W(R)$-conjugate to $\Pi$ and for an affine root system $\Pi'$ is either in $W(R).\Pi$ or in $W(R).(-\Pi)$. Fix such a system $\Pi$\index[notations]{p111@$\Pi$} and call it a 
\textbf{simple system}\index[index]{simple system} of $R$. One can associate $\Pi$ with 
a graph, called the \textbf{Dynkin diagram}\index[index]{Dynkin diagram} $\Gamma$.

\begin{defn}\label{defn_Dynkin-diag}
A \textbf{Dynkin diagram} $\Gamma$ of a (finite or affine) root system $R$ is a graph whose set of nodes $\vert \Gamma \vert$ consists of $\alpha \in \Pi$ and is represented by
\begin{enumerate}
\item[(Node1)] 
\begin{tikzpicture} \draw (0,0) circle(0.1); \draw (0,0.1) node[above] {$\alpha$}; \end{tikzpicture}
if $\alpha \in R$ but $2 \alpha \not\in R$  and
\item[(Node2)]
\begin{tikzpicture} \fill (0,0) circle(0.1); \draw (0,0.1) node[above] {$\alpha$}; \end{tikzpicture}
if $\alpha \in R$ and $2 \alpha \in R$, 
\end{enumerate}
and for two vertices $\alpha, \beta \in \Pi$, 
we connect them by a bond with arrows according to the following rules:
for $\node{40}, \node{80} \in \{ \white, \black\}$
\begin{enumerate}
\item[(Edge0)] \begin{tikzpicture} \draw [fill=gray!40] (0,0) circle(0.1); \draw (0,0.1) node[above] {$\alpha$};
                                     \draw [fill=gray!80] (1,0) circle(0.1); \draw (1,0.1) node[above] {$\beta$};
           \end{tikzpicture}   
if $I(\alpha, \beta^\vee)=0 \quad \Longleftrightarrow \quad I(\beta, \alpha^\vee)=0$,
\item[(Edge1)] \begin{tikzpicture} \draw [fill=gray!40] (0,0) circle(0.1); \draw (0,0.1) node[above] {$\alpha$};
                                     \draw (0.1,0) -- (0.9,0); 
                                     \draw [fill=gray!80] (1,0) circle(0.1); \draw (1,0.1) node[above] {$\beta$};
           \end{tikzpicture}   
if $I(\alpha, \beta^\vee)=I(\beta, \alpha^\vee)=-1$,
\item[(Edge2)] \begin{tikzpicture} \draw [fill=gray!40] (0,0) circle(0.1); \draw (0,0.1) node[above] {$\alpha$};
                                     \draw (0.1,0) -- (0.9,0);  \draw (0.5,0.05) node[above] {\tiny{$t$}};
                                     \draw (0.55,0.1) -- (0.45,0); \draw (0.55,-0.1) -- (0.45,0); 
                                     \draw [fill=gray!80] (1,0) circle(0.1); \draw (1,0.1) node[above] {$\beta$};
           \end{tikzpicture} 
if $I(\alpha, \beta^\vee)=-1$ and $I(\beta, \alpha^\vee)=-t$ for $t=2,3,4$ and
\item[(Edge3)] \begin{tikzpicture} \draw [fill=gray!40] (0,0) circle(0.1); \draw (0,0.1) node[above] {$\alpha$};
                                     \draw (0.1,0) -- (0.9,0); \draw (0.5,0.03) node[above] {\tiny{$\infty$}};
                                     \draw [fill=gray!80] (1,0) circle(0.1); \draw (1,0.1) node[above] {$\beta$};
           \end{tikzpicture}   
if $I(\alpha, \beta^\vee)=I(\beta, \alpha^\vee)=-2$.
\end{enumerate}
\end{defn}
Some necessary data about the finite and affine root systems are recalled in
the Appendices \ref{sect_finite-root-system} and \ref{sect_affine-root-system}, respectively. 

For any root system $R$, let us denote the subset of short roots, middle length roots 
and long roots by $R_s$\index[notations]{r112@$R_s$}, 
$R_m$\index[notations]{r113@$R_m$} and 
$R_l$\index[notations]{r113@$R_l$}, respectively. 
For a non-reduced root system $R$, let $R^{\dc}$\index[notations]{r114@$R^{\dc}$} and 
$R^{\mc}$\index[notations]{r115@$R^{\mc}$} be the root subsystems of indivisible roots 
and non-multiplicable roots of $R$, respectively (cf. \cite{Macdonald1972}), that is,
\[
R^{\dc}:=
\left\{ \alpha \in R\, \left\vert \, \frac{1}{2}\alpha \not\in R\right\}, \right. \qquad
R^{\mc}:=
\{ \alpha \in R\, \vert\, 2\alpha \not\in R\}.
\]
We call the pair $(R^{\dc}, R^{\mc})$ 
the \textbf{reduced pair}\index[index]{reduced pair} of $R$. 
It can be shown that both $R^{\dc}$ and $R^{\mc}$ are root subsystems of $R$, and the 
Weyl groups $W(R), W(R^{\dc})$ and $W(R^{\mc})$ are the same. 
\medskip
\subsection{Root system of type $BC_l$ ($l\geq 1$)}\label{sect_BC-l}
In this subsection, we present root data of the finite root system $R_f$ of type $BC_l$
for reader's convenience. We realize $R_f$ in the real vector space $F_f=\bigoplus_{i=1}^l \R \vep_i$ equipped with the symmetric bilinear form $I_{F_f}$ satisfying
\begin{equation}\label{normalization-I} 
I_{F_f}(\vep_i, \vep_j)=\delta_{i,j} \qquad 1\leq i,j \leq l. 
\end{equation}
As is given in \S \ref{sect_BC}, the root system $R_f$ is the union of the set of short roots $(R_f)_s$, middle length roots $(R_f)_m$ and long roots $(R_f)_l$, given by
\begin{align*}
(R_f)_s=
&\{ \, \pm \vep_i\, \vert\, 1\leq i\leq l\, \}, \\
(R_f)_m=
&\{ \, \pm(\vep_i \pm \vep_j)\, \vert\, 1\leq i<j\leq l\, \}, \\
(R_f)_l=
&\{\, \pm 2\vep_i\, \vert\, 1\leq i\leq l\, \}.
\end{align*}
Notice that $(R_f)_m=\emptyset$ if $l=1$. 
For $l>1$, $(R_f)^{\dc}=(R_f)_s \cup (R_f)_m$ is a root system of type $B_l$ and $(R_f)^{\mc}=(R_f)_m \cup (R_f)_l$ is a root system of type $C_l$. It should also be mentioned that $(R_f)_m$ is a root system of type $D_l$. Here and in the rest of this article, the root system of type $D_l$ is defined for $l\geq 2$ as follows:
\[ R(D_l)=\{\,\pm (\vep_i \pm \vep_j)\; (1\leq i<j\leq l)\, \}. \]
In particular, we choose its simple roots and its fundamental dual weights as usual:
\begin{align*}
 \alpha_i= 
 &\begin{cases} \; \vep_i-\vep_{i+1}\;, & 1\leq i<l, \\ 
                                         \; \vep_{l-1}+\vep_l\; & i=l, \end{cases} \\
 \varpi_i^\vee=
 &\begin{cases} \; \vep_1+\cdots +\vep_i\; &1\leq i\leq l-2, \\
                                               \;  \frac{1}{2}(\vep_1+\cdots +\vep_{l-1}-\vep_l)\; &i=l-1,  \\
                                               \; \frac{1}{2}(\vep_1+\cdots +\vep_{l-1}+\vep_l)\; &i=l. \end{cases}
\end{align*}
 The rank $l$ of a root system of type $D_l$ is usually assumed to be $l\geq 4$. 
It is known that 
\[ R(D_2) \cong R(A_1) \times R(A_1), \qquad
R(D_3) \cong R(A_3) \]
as abstract root systems. The Weyl group of type $D_l$ is isomorphic to $\fS_l \ltimes (\Z/2\Z)^{l-1}$ for $l \geq 4$. For $l=2$ and $3$, we have
\begin{align*}
& W(D_2) \cong W(A_1) \times W(A_1) \cong 
\fS_2 \times \fS_2 \left( \cong (\Z/2\Z) \times (\Z/2\Z) \right), \\
&  W(D_3) \cong W(A_3)\cong \fS_4. 
\end{align*}

It is clear that the Weyl group $W(R_f)$ is isomorphic to the Weyl group of type $B_l$ and $C_l$, that is, $\fS_l \ltimes (\Z/2\Z)^l$ for $l\geq 2$, and for $l=1$ it is isomorphic to the Weyl group of type $A_1$. \\

Set $\alpha_i=\vep_i-\vep_{i+1}$ for $1\leq i<l$ and $\alpha_l=\vep_l$. Then $\Pi_f^{\dc}:=\{\alpha_i\}_{1\leq i\leq l}$ is a simple system of $(R_f)^{\dc}$ and $\Pi_f^{\mc}:=\{\alpha_i\}_{1\leq i<l} \cup \{2\alpha_l\}$ is a simple system of $R_f^{\mc}$. Notice that 
\[ (\Pi_f^{\dc})^\vee:=\{\, \alpha^\vee\,\vert \, \alpha \in \Pi_f^{\dc}\, \}=\Pi_f^{\mc} \qquad  \text{and} \qquad  (\Pi_f^{\mc})^\vee:=\{ \, \alpha^\vee\, \vert\, \alpha \in \Pi_f^{\mc}\, \}=\Pi_f^{\dc}. 
\]
The simple system $\Pi_f^{\dc}$ is also a simple system of $R_f$, which we also denote by $\Pi_f$. 
Indeed, we have
\begin{align*}
\vep_i=
&\alpha_i+\alpha_{i+1}+\cdots+\alpha_l, \\
\vep_i-\vep_j=
&\alpha_i+\alpha_{i+1}+\cdots+\alpha_{j-1}, \\
\vep_i+\vep_j=
&\alpha_i+\alpha_{i+1}\cdots +\alpha_{j-1}+2(\alpha_j+\alpha_{j+1}+\cdots+\alpha_l), 
\end{align*}
and the root lattice $Q(R_f)$ of $R_f$, which is generated by the elements of $R_f$,  is indeed generated by $\{\alpha_i\}_{1\leq i\leq l}$:  
\[ Q(R_f)=\bigoplus_{\alpha \in \Pi_f} \Z \alpha. \]

As for the dual roots, it follows that
\[ (\vep_i)^\vee=2\vep_i, \qquad (\vep_i\pm \vep_j)^\vee=\vep_i\pm \vep_j, \qquad (2\vep_i)^\vee=\vep_i, \]
in particular, one has $R_f^\vee=R_f$. Hence, the dual root lattice $Q(R_f^\vee)$ and the root lattice $Q(R_f)$ are the same. Since, $\alpha_i^\vee=\alpha_i$ ($1\leq i<l$) and $\alpha_l^\vee=2\alpha_l$, we have
\[ Q(R_f^\vee) =\Z \alpha_1^\vee \oplus \Z \alpha_2^\vee \oplus \cdots \oplus \Z \alpha_{l-1}^\vee \oplus \frac{1}{2}\Z \alpha_l^\vee,
\]
i.e., the dual roots $\{\alpha_i^\vee\}_{1\leq i\leq l}$ \textit{do not} generate $Q(R_f^\vee)$. 
The lattice generated by the dual roots $\{\alpha_i^\vee\}_{1\leq i\leq l}$ is the root lattice of type $C_l$.
\medskip

Recall that the fundamental weights $\varpi_i$  (resp. dual weights $\varpi_i^\vee$) ($1\leq i\leq l$) are the dual base of simple dual roots $\alpha_i^\vee$ (resp. roots $\alpha_i$) ($1\leq i\leq i$). Explicitly, they are given by
\[
\varpi_i=\begin{cases} \vep_1+\vep_2+\cdots +\vep_i \qquad &1\leq i<l, \\
           \frac{1}{2}(\vep_1+\vep_2+\cdots+\vep_l) \quad &i=l, \end{cases}
\qquad
\varpi_i^\vee= \vep_1+\vep_2+\cdots+\vep_i \; (1\leq i\leq l).
\]
In particular, the weight lattice $P(R_f)$, which is the set of all $\lambda \in F_f$ with $I_{F_f}(\lambda, R_f^\vee) \subset \Z$, does not contain $\varpi_l$:
\[ P(R_f)=\Z \varpi_1 \oplus \Z \varpi_2 \oplus \cdots \oplus \Z \varpi_{l-1}\oplus 2\Z \varpi_l,
\]
and the lattice generated by $\{\varpi_i\}_{1\leq i\leq l}$ is the weight lattice of type $B_l$. 
The dual weight lattice $P(R_f^\vee)$, which is the set of all $\lambda \in F_f$ with $I_{F_f}(\lambda, R_f)\subset \Z$, is indeed generated by $\{\varpi_i^\vee\}_{1\leq i\leq l}$:
\[ P(R_f^\vee)=\bigoplus_{i=1}^l \Z \varpi_i^\vee. \]
This lattice is the weight lattice of type $C_l$. 
\begin{rem}\label{rem_normalization}{\rm
\begin{enumerate}
\item With the normalization \eqref{normalization-I}  of $I_{F_f}$ in this section, 
all the $4$ lattices $Q(R_f), \, Q(R_f^\vee), P(R_f)$ and $P(R_f^\vee)$ coincide with
\[ \bigoplus_{i=1}^l \Z \vep_i .
\]
\item In general, we only have 
$P(R_f)=Q(R_f)$, $P(R_f^\vee)=Q(R_f^\vee)$ and $Q(R_f^\vee) \cong Q(R_f)$. 
\end{enumerate}}
\end{rem}


\section{Marked elliptic root systems}\label{sect:mers}
Let $(F,I_F)$ be an $\R$-vector space with a symmetric bilinear form $I_F$ on it. 
\index[notations]{i211@$I_F$}
First we 
introduce the notion of a marking $G$ which is a certain linear subspace of $F$, and 
then we shall restrict ourselves to the elliptic case, i.e., to those where the signature of 
$I_F$ is $(l,2,0)$ for some $l>0$. 
\subsection{Marking}\label{sect_marking}
We say that a subspace $G$ of $F$ is defined over 
${\Z}$ if we have 
$$
\text{rank}_{\R}G=\text{rank}_{\Z} \left(G \cap Q(R)\right).
$$
The radical
\[ \rad(I_F):=\{x \in F \, \vert \, I_F(x,y)=0 \quad \forall\; y \in F\} 
\index[notations]{r211@$\rad(I_F)$}
\] 
is defined over $\Z$, i.e., 
\[\rad_\Z(I_F)=\rad(I_F) \cap Q(R)\index[notations]{r212@$\rad_\Z(I_F)$}\]
 is a full lattice of $\rad(I_F)$. 
A subspace $G$ of $\rad(I_F)$ is called a \textbf{marking}\index[index]{marking} of $R$ 
if it is defined over 
${\Z}$. Let $G$ be a marking, and $\pi_G:F \rightarrow F/G$ the canonical projection. 
It induces a bilinear form $I_{F/G}$\index[notations]{i212@$I_{F/G}$} on $F/G$ 
defined by $I_{F/G}(\pi_G(x), \pi_G(y))=I_F(x,y)$ for 
$x,y \in F$. The image of $R$ in $F/G$, denoted by 
$R/G$\index[notations]{r211@$R/G$}, 
is a root system called the 
\textbf{quotient root system of $R$ by $G$}\index[index]{root system!quotient}.  
It is a root system belonging to $(F/G,I_{F/G})$. 

The pair $(R,G)$ of a root system and a subspace of $\rad(I_F)$ defined over ${\Z}$ is 
called a \textbf{marked root system}. Two marked root systems $(R_1,G_1)$ with 
$R_1 \subset (F_1,I_1)$ and $(R_2,G_2)$ with $R_2 \subset (F_2,I_2)$ are 
\textbf{isomorphic} iff there exists a linear isomorphism $\varphi:F_1 \rightarrow F_2$ 
such that $\varphi(R_1)=R_2$ and $\varphi(G_1)=G_2$. 
In particular, $(R,G)$ is called a 
{\bf marked elliptic root system}\index[index]{root system!marked elliptic
@marked elliptic -- (mERS)} if 
$\dim G=1$.  \\

\begin{rem}\label{rem_quot-root-sys} {\rm
\begin{enumerate}
\item
For a marked root system $(R,G)$, its quotient $R/G$ is still a generalized root 
system $($cf. \S 1.8 in \cite{Saito1985}$)$. In particular, for any  marked elliptic root 
system $(R,G)$, its quotient root system is an affine root system, if $G$ is of rank 
$1$. 
\item Note that, if the quotient $R/G$ is reduced, then so is $R$. Indeed, assume $R$ 
is non-reduced. Then there is a root 
$\alpha\in R$ such that $r\alpha\in R$ for $r=\pm 2,\pm1/2$.  Taking their
images under the canonical projection $\pi_G: F\to F/G$, it follows that
there exists a root $\pi_G(\alpha)\in R/G$ such that 
$r(\pi_G(\alpha))=\pi_G(r\alpha)\in R/G$. That is, $R/G$ is non-reduced. 
\item Therefore, for any marked root system $(R,G)$, one of the 
following three holds{\rm :} 
\begin{itemize}
\item[(i)] both $R$ and $R/G$ are reduced{\rm ;} 
\item[(ii)] $R$ is reduced and $R/G$ is non-reduced{\rm ;} 
\item[(iii)] both $R$ and $R/G$ are non-reduced.  
\end{itemize}
\end{enumerate} }
\end{rem}


In this article we classify marked elliptic root systems  
$(R,G)$ with non-reduced quotient $R/G$. K. Saito \cite{Saito1985} has already classified such root systems in the case 
when the quotient root system $R/G$ is reduced. 
Hence, this gives a complete classification result. For short, we will abbreviate \textbf{marked elliptic root system} to \textbf{mERS}. 
{Here and after, for $\alpha\in R$, 
we sometimes denote $\overline{\alpha}=\pi_G(\alpha)$ for simplicity.}
\index[notations]{a111@$\overline{\alpha}$}
\\

Let $(R,G)$ be a mERS. 
\begin{defn}\label{defn:basis}
A set $\Pi \subset R$ of roots
is called a \textbf{simple system}\index[index]{simple system} of a 
marked elliptic root system $(R,G)$, if its image under the canonical projection 
$\pi_G:F\to F/G$ is a simple system of the quotient affine root system $R/G$. 
\end{defn}
K. Saito \cite{Saito1985} showed, that if  $R/G$ is reduced, then for any two 
{simple systems} $\Pi_1$ and $\Pi_2$ of $(R,G)$, there exists an automorphism $\varphi$ of 
$(R,G)$ such that $\varphi(\Pi_1)=\Pi_2$ (see \cite{Saito1985}, (6.2) 
Corollary, for details).  That is, a {simple system} $\Pi$ of $(R,G)$
is unique up to automorphism of $(R,G)$, if $R/G$ is reduced. In 
his proof, the reducedness of $R/G$ is essentially used. {Here and after,
we denote $\Pi_G:=\pi_G(\Pi)$.}\index[notations]{p112@$\Pi_G$}\\

On the other hand, if $R/G$ is non-reduced, the situation is more 
complicated. For a non-reduced mERS $(R,G)$ belonging to $(F,I_F)$, 
let $(R^{\dc},R^{\mc})$ be the reduced pair of $R$ and $((R/G)^{\dc},(R/G)^{\mc})$ of the 
quotient affine root system $R/G$. Since both $(R^{\dc},G)$ and $(R^{\mc},G)$ are 
(reduced) mERSs, their quotients $R^{\dc}/G$ and $R^{\mc}/G$ are 
affine.
However, since $R^{\dc}/G$ and $R^{\mc}/G$ are
not necessarily reduced, they are again the union of their reduced parts:
\[R^{\dc}/G=(R^{\dc}{/}G)^{\dc}\cup (R^{\dc}/G)^{\mc},\quad R^{\mc}/G=(R^{\mc}/G)^{\dc}\cup (R^{\mc}/G)^{\mc}.\]
Hence, there are six reduced affine root subsystems 
\[(R/G)^{\dc},\ (R/G)^{\mc},\ (R^{\dc}/G)^{\dc},\ (R^{\dc}/G)^{\mc},\ (R^{\mc}/G)^{\dc},\ (R^{\mc}/G)^{\mc}\]
of $R/G$ belonging to $(F/G,I_{F/G})$. A natural question now is the following: 

\vskip 2.5mm
\begin{quotation}
{\bf Question 1}. 
{\it What are the explicit relationships between these six affine root systems\,?}
\end{quotation}
\vskip 2.5mm
A partial answer for this question will be given in 
$\S$ \ref{sect:ans-Q1} (see Lemma \ref{lemma:reduced_pair2}).  \\

The uniqueness of a basis of a non-reduced $(R,G)$, 
after preparing certain terminologies in the first half of Section \ref{sect:elliptic-diagram},  will be treated in 
Proposition \ref{prop:nr-counting-Autom1} in $\S$ \ref{sect:unique-basis}.  \\

The relationship between a simple system of $(R/G)^{\dc}$ and 
of $(R/G)^{\mc}$ is described as follows.
Denote by  $(\Pi_G)^{\dc}$ a simple system of $(R/G)^{\dc}$. By definition (cf. \cite{Macdonald1972}), 
it is also a simple system of $R/G$. 
\begin{lemma} \label{lemma_base} 
There exists a simple system $(\Pi_G)^{\mc}$ of $(R/G)^{\mc}$ such that $(\Pi_G)^{\mc} \setminus (\Pi_G)^{\dc}$ 
contains only $1$ root for $R/G$ being of type $BCC_l, C^\vee BC_l$ and $BB_l^\vee$,
and $2$ roots of type $C^\vee C_l$.
\end{lemma}
\begin{proof}
Write $(\Pi_G)^{\dc}=\{\overline{\alpha_i}\}_{0\leq i\leq l}$ and set
\[ \overline{\beta_i}=\begin{cases} \overline{\alpha_i} \qquad & \text{if} \; 
2\overline{\alpha_i} \not\in R/G, \\
          2\overline{\alpha_i} \qquad & \text{otherwise}. \end{cases}
\]
It can be shown that $\{\overline{\beta_i}\}_{0\leq i\leq l}$ is a simple system of 
$(R/G)^{\mc}$ we are looking for. Indeed, by \cite{Bourbaki(Lie)4-6}, there 
exists an automorphism $\varphi$ of $R/G$ such that $\varphi((\Pi_G)^{\dc})$ is the one 
in \S  \ref{sect_non-red-affine}. By case by case checking, we see that 
$\varphi((\Pi_G)^{\mc})$ satisfies this property.
\end{proof}

Here and after, we fix $a \in G$ in such a way that 
\begin{equation}\label{defn:a}
G \cap Q(R)=\Z a,
\end{equation} 
and 
take $b \in \rad(I_F)$ so that $\rad_\Z(I_F)=\Z b \oplus 
\Z a$.

\begin{rem}{\rm
For a mERS $(R,G)$, an element $a\in \rad(I_F)$ is uniquely determined by 
\eqref{defn:a} up to sign. On the other hand, there are
many choices for $b\in\rad(I_F)$. More precisely, 
only the class $b\,\mathrm{mod}\,(Q(R)\cap G) \in \rad_\Z(I_F)/(Q(R)\cap G)
$ is uniquely determined up to sign. }
\end{rem}
\subsection{Tier numbers}\label{sect_tier-number}
Let $R$ be a root system belonging to $(F, I)$. It is shown in \cite{Saito1985}
 (cf. Lemma 1.9) that 
there exists a non-zero scalar $c \in \R$ such that the set 
$\{ cI(\alpha, \alpha) \, \vert \, \alpha \in R\}$ is a finite subset of $\Z$. Set
\[ t(R)=\max\left\{ \left. \frac{I(\alpha, \alpha)}{I(\beta,\beta)} \right\vert\, \alpha, 
\beta \in R\right\},  \]
and call it the \textbf{total tier number}\index[index]{tier number!total@total --} of $R$. 
With such a constant $c$, we also 
define
\begin{equation}\label{def_normalized-form}
\begin{split}
&I_R=\frac{2c}{\mathrm{gcd}\{ cI(\alpha, \alpha)\, \vert \, \alpha \in R\}}I, \\
&I_{R^\vee}=\frac{\mathrm{lcm}\{ cI(\alpha, \alpha)\, \vert \, \alpha \in R\}}{2c}I.
\end{split}
\end{equation}
Here is the list of total tier numbers: \\

\begin{center}
\begin{tabular}{|c||c|c|c|c|c|c|c|c|} \hline
Type of $R/\rad(I)$ & $A_l$ & $B_l$ & $C_l$ & $BC_l$ & $D_l$ & $E_l$ & $F_4$ & 
$G_2$ \\ \hline
$t(R)$ & $1$ & $2 $ & $2$ & $4$ & $1$ & $1$ & $2$ & $3$ \\ \hline
\end{tabular}
\end{center}
\vskip 0.1in

Let $(R,G)$ be a marked elliptic root system and $(R^\vee, G)$ its dual. Let 
$a^\vee \in \R a$ and $b^\vee \in \rad(I)$ such that $Q(R^\vee) \cap G=\Z a^\vee$ 
and $Q(R^\vee) \cap \rad(I)=\Z a^\vee \oplus \Z b^\vee$. 
The \textbf{first and the second tier numbers}
\index[index]{tier number!first@first --}\index[index]{tier number!second@second --}for $(R,G)$ and 
$(R^\vee,G)$ are 
defined by
\begin{align*}
t_1(R,G)=
&\vert (b^\vee \; \mathrm{mod}\; a^\vee) : (b\; \mathrm{mod}\;a)\vert \times 
(I_{R^\vee}:I), \\
t_1(R^\vee,G)=
&\vert (b\; \mathrm{mod}\;a) : (b^\vee\; \mathrm{mod}\;a^\vee)\vert \times (I_{R}:I), \\
t_2(R,G)=
&\vert (a^\vee: a) \vert \times  (I_{R^\vee}:I), \\
t_2(R^\vee,G)=
&\vert (a: a^\vee) \vert \times  (I_{R}:I),
\end{align*}
where $A : B$\index[notations]{aa112@$A : B$} signifies the constant $c$ such that $A=cB$. 
It is known that 
\[ t(R)=t_1(R,G)t_1(R^\vee,G)=t_2(R,G)t_2(R^\vee,G). \]
\subsection{Counting set}\label{sect:counting}
The classification of the mERSs $(R,G)$ with reduced affine quotient $R/G$ has been carried out 
by K. Saito in \cite{Saito1985}. In that work, even if the result itself is clearly stated, the 
details of its proof is omitted. 
Here we briefly present a proof to show how the notion like counting works. \\

In the first half of this subsection, we return to the general setting. 
Let $R$ be an elliptic root system belonging to $(F,I)$ and $G$ be a 
subspace of $\rad (I)$ defined over $\Z$ (marking). That is, $G$ is {\it not} 
necessarily one dimensional, and $R/G$ is {\it not} assumed to be reduced. 

For each $\alpha \in R$, let us define the subset $K_G(\alpha)$ of the lattice 
$Q(R) \cap G$ by
\begin{equation}\label{defn:counting-set}
 K_G(\alpha)=\{x \in G \, \vert \, \alpha+x \in R\, \}, 
 \index[notations]{k111@$K_G(\alpha)$}
\end{equation}
and call it the \textbf{counting set}\index[index]{counting set} of $\alpha \in R$. 
It has the next nice properties
 (cf. p. 97 in \cite{Saito1985}):

\begin{lemma}\label{lemma-counting}
\begin{enumerate}
\item $0 \in K_G(\alpha)$ and $K_G(\alpha)=-K_G(\alpha)$ for $\alpha \in R$.
\item $K_G(\varphi(\alpha))=K_G(\alpha)$ for an automorphism $\varphi$ of $R$ and 
$\alpha \in R$.
\item $K_G(\alpha)$ is closed under the {reflection} centered at each point of 
$K_G(\alpha)$, 
\newline
i.e. $2x-y \in K_G(\alpha)$ for any $x,y \in K_G(\alpha)$.
\item $K_G(\alpha)$ is closed under the translation by 
$I(\alpha, \beta^\vee)K_G(\beta)$ for $\alpha, \beta \in R$,
\newline
 i.e., $K_G(\alpha) \supset K_G(\alpha)+I(\alpha, \beta^\vee)K_G(\beta)$.
\item $\dfrac{2}{I(\alpha, \alpha)}K_G(\alpha)=K_G^\vee(\alpha^\vee)$, where 
$K_G^\vee$ is the counting set of $(R^\vee, G)$. 
\end{enumerate}
\end{lemma}
In particular, the first statement implies that if $x \in K_G(\alpha)$, then $\Z x \subset 
K_G(\alpha)$. Indeed, as $\alpha, \alpha+x \in R$, $-x \in K_G(\alpha+x)$ which implies 
$x \in K_G(\alpha+x)$, so $\alpha+2x \in R$. In the same way, one can show that 
$\Z_{\geq 0}x \subset K_G(\alpha)$, hence $\Z x \subset K_G(\alpha)$ again by (1).
The statements (2) and (4) imply that if $\beta \in W(R).\alpha$ and $I(\alpha,\beta^\vee) \in 
\{\pm 1\}$, then $K_G(\alpha)=K_G(\beta)$ is additively closed. 

\begin{lemma}\label{lemma_counting-rootsystem}
Suppose that there exists a subspace $L \subset F$ satisfying
\begin{align}
&F=L \oplus G, \label{cond1:lemma_counting-rootsystem}\\
&\pi_G \vert_{R\cap L}: 
R \cap L\rightarrow R/G \quad \text{is surjective}.
\label{cond2:lemma_counting-rootsystem}
\end{align}
Then
\begin{enumerate}
\item $K_G(\alpha)$ contains a full lattice of $G$ for $\alpha \in R$.
\vskip 1mm
\item $R=\amalg_{\alpha \in R\cap L} \left( \alpha+K_G(\alpha)\right)$. 
\vskip 1mm
\item $\bigcup_{\alpha \in R \cap L}K_G(\alpha)$ generates the lattice $Q(R) \cap G$.
\end{enumerate}
\end{lemma}

\begin{rem}\label{rem:dimL}
As we already remarked above,  the subspace $G\subset \rad(I)$ is {\it not} 
necessarily one dimensional. That is, we may assume $G=\rad(I)$. In this setting,
there exists a mERS $(R,G)$ such that the surjective map 
\eqref{cond2:lemma_counting-rootsystem} is not injective {\rm(}see also below{\rm )}.
\end{rem}

Here and after, we assume that $G=\R a$. Recall that we already 
fixed a generator $a$ of the lattice $Q(R)\cap G$ of rank 1:
$Q(R)\cap G=\Z a$ (see \eqref{defn:a}), and define the 
{\bf counting number}\index[index]{counting number} for 
$\alpha\in R$ by
\begin{equation}\label{defn:counting-number}
k(\alpha)=\mathrm{min}\,\{k\in \Z_{>0}\,|\,\alpha+ka\in R\}.
\end{equation}
\index[notations]{k221@$k(\alpha)$}
For $\alpha\in R$, set 
\begin{equation}\label{defn:ast}
\alpha^*=\alpha+k(\alpha)a.\index[notations]{a211@$\alpha^*$}
\end{equation}

Fix a simple system $\{\alpha_i\}_{0\leq i\leq l}$ of the mERS $(R,G)$ in the sense of
Definition \ref{defn:basis}, and define an $(l+1)$-dimensional subspace $F_a$ 
of $F$ by
\begin{equation}\label{defn:L^{l+1}}
F_a=\bigoplus_{i=0}^l \R \alpha_i.
\end{equation}
The following lemma is a generalization of the result of K. Saito 
(\cite{Saito1985}, (6.1), Assertion).

\begin{lemma}\label{lemma:countings1}
Let $\alpha\in R$.
\begin{enumerate}
\item One has $K_G(\alpha)=\Z k(\alpha)a$. That is, 
the counting number $k(\alpha)$ satisfies 
\begin{equation}\label{counting-number}
\alpha+\Z k(\alpha)a=R\cap (\alpha+\Z a).
\end{equation}
\item For any two roots $\alpha, \beta \in R$, 
\begin{equation}\label{countings-division}
k(\beta) \vert I(\beta, \alpha^\vee)k(\alpha). 
\end{equation}
In particular, if $I(\beta, \alpha^\vee)\in \{\pm 1\}$, then
\[ 1 \vert k(\alpha)k(\beta)^{-1} \vert I(\alpha, \beta^\vee). \]
\item For every automorphism $\varphi$ of $R$, one has $k(\varphi(\alpha))=k(\alpha)$. 
\end{enumerate}
For the following two statements, we assume $R/G$ is reduced. 
\begin{itemize}
\item[(4)] For the $(l+1)$-dimensional subspace $F_a$ of $F$ introduced in 
\eqref{defn:L^{l+1}}, one has 
\begin{equation}\label{eq:description-redR}
R=\bigsqcup_{\alpha \in R \cap F_a}(\alpha+\Z k(\alpha)a).
\end{equation} 
\item[(5)] $g.c.d.\{k(\alpha_i)\,|\,0\leq i\leq l\}=1$.
\end{itemize}
\end{lemma}

\begin{rem}\label{rem:countings1}
{\rm (1)}
If $R/G$ is reduced, 
it is known that the canonical projection $\pi_G:F\to F/G$ induces a bijection
\begin{equation}\label{eq:isom-p_G}
\pi_G|_{R\cap F_a}:R \cap F_a\xrightarrow{\sim}R/G
\end{equation}
{\rm (}see \cite{Saito1985}, {\rm (3.3), Note 2}{\rm )}. Therefore, the subspace $F_a$ satisfies the
conditions \eqref{cond1:lemma_counting-rootsystem} and 
\eqref{cond2:lemma_counting-rootsystem}, and the results in Lemma 
\ref{lemma_counting-rootsystem} hold for $F_a$. Indeed, in \cite{Saito1985}, 
K. Saito proved the statements {\rm ((6.1), Assertion)} under the assumption that
$R/G$ is reduced, and used Lemma \ref{lemma_counting-rootsystem}. 
\vskip 1mm
\noindent
{\rm (2)}
Otherwise, the induced map 
$\pi_G|_{R\cap F_a}:R\cap F_a \longrightarrow R/G$ is {\it not} surjective in general.
As a consequence, Lemma \ref{lemma_counting-rootsystem} {\it does not} 
hold in general. Therefore, we have to prove the statements {\rm (1), (2), (3)} 
in the lemma without using Lemma  \ref{lemma_counting-rootsystem}.
\end{rem}

\begin{proof}
Let us prove statement (1). Note that the following equality holds:
\begin{equation}\label{k=kast}
k(\alpha)=k(\alpha^\ast).
\end{equation}
Indeed, since $\alpha,\alpha^\ast=\alpha+k(\alpha)a\in R$, we have 
$-k(\alpha)a\in K_G(\alpha^\ast)$. By Lemma \ref{lemma-counting} (1), 
 it follows that
$k(\alpha)a \in K_G(\alpha^\ast)$. This shows that $k(\alpha^\ast)\leq k(\alpha)$.   
By Lemma \ref{lemma-counting} (1), 
$-k(\alpha^\ast)a\in K_G(\alpha^\ast)$. Hence, 
\[\alpha^*-k(\alpha^\ast)a=\alpha+(k(\alpha)-k(\alpha^\ast))a\in R,\]
from which the minimality of $k(\alpha)$ implies $k(\alpha^\ast) \geq k(\alpha)$,
 thus we have \eqref{k=kast}. By using \eqref{k=kast} repeatedly, we obtain
\[k(\alpha)=k(\alpha+rk(\alpha)a)\quad\text{for every $r\in\Z$}.\]
Thus, statement (1) follows.

Statement (2) is an immediate consequence of statement (1) and Lemma 
\ref{lemma-counting} (4).
Statement (3) follows from Lemma \ref{lemma-counting} (2) and 
statement (1). 
\end{proof}
\vskip 3mm
\subsection{Elliptic diagrams I: reduced affine quotient}\label{sect_elliptic-diagram-I}
Let $(R,G)$ be a mERS belonging to $(F,I_F)$ whose quotient system $R/G$ is 
\textit{reduced}, and fix a simple system $\{\alpha_i\}_{0\leq i\leq l}$ of $(R,G)$.
Recall that the canonical projection $\pi_G:F\to F/G$ induces a bijection
\eqref{eq:isom-p_G}. 
Via this bijection, the set $\{\alpha_i\}_{0\leq i\leq l}$ can be
regarded as the set of nodes of the Dynkin diagram of the affine root system $R/G$. 
Let $\Gamma_a$ be the Dynkin diagram of the affine root system 
$R/G$. {Denote $\overline{\alpha}_i=\pi_G(\alpha_i)$ for $0\leq i\leq l$.
Let $n_i$'s be the positive coprime integers  such that 
$\overline{b}:=\sum_{i=0}^l n_i \overline{\alpha}_i$ generates the lattice 
$Q(R/G)\cap \rad(I_{F/G})$ of rank $1$. Define
\begin{equation}\label{defn:b}
b=\sum_{i=0}^l n_i \alpha_i.
\end{equation}
Then $(b,a)$ is an integral basis of $\rad_\Z(I_F)$, {\it i.e.}, 
$\rad_\Z(I_F)=\Z b\oplus \Z a$.  
Let $k(\alpha_i)$ be the counting of the root $\alpha_i$.
The \textbf{exponent}\index[index]{exponent} $m_{\alpha_i}$ is defined by
\[ m_{\alpha_i}=\frac{I_R(\alpha_i, \alpha_i)}{2k(\alpha_i)}n_i.\]
We denote the set of 
nodes by $\vert \Gamma_a \vert$. \\

Set $m_{\max}=\max\{m_{\alpha_0}, \cdots, m_{\alpha_l}\}$. Let $\Gamma_m$ be 
the subdiagram of $\Gamma_a$  consisting of nodes 
\[ \vert \Gamma_m \vert=\{ \alpha_{\alpha_i} \in \vert \Gamma_a \vert\, 
\vert m_{\alpha_i}=m_{\max}\}. \]
Define
\begin{align*}
&\vert \Gamma_m^\ast \vert=\{\alpha^\ast\, \vert \, \alpha \in \vert\Gamma_m \vert\, \}.
\end{align*}

\begin{defn}\label{defn_elliptic-diagram-I}
The \textbf{elliptic diagram}\index[index]{elliptic diagram} 
$\Gamma(R,G)$\index[notations]{G1111@$\Gamma(R,G)$} 
for a marked elliptic root 
system $(R,G)$ 
is the graph whose set of nodes is 
\[ \vert \Gamma(R,G)\vert :=\vert \Gamma_a \vert \cup \vert \Gamma_m^\ast \vert, \]
and any two nodes $\alpha, \beta \in \vert\Gamma(R,G)\vert $ are connected by the rules 
$($Edge0$)$ - $($Edge3$)$ with the additional case as follows:
\begin{itemize}
\item[(Edge4)] 
\begin{tikzpicture} 
\draw (0,0) circle(0.1); \draw (0,0.1) node[above] {$\alpha$};
\draw  [dashed, double distance=2] (0.1,0,0) -- (0.9,0,0); 
\draw (1,0) circle(0.1); \draw (1,0.1) node[above] {$\beta$};
\end{tikzpicture} if $I_F(\alpha, \beta^\vee)=I_F(\beta, \alpha^\vee)=2$.
\end{itemize}
\end{defn}
\label{sect_elliptic}

\subsection{Classification of mERSs with reduced quotient}
In this subsection, we provide the classification theorem of the mERSs $(R,G)$ with reduced affine quotient $R/G$. Assume that the root system $R$ belongs to a real vector space $F$ equipped with a symmetric bilinear form $I_F$ whose signature is $(l,2,0)$. 
 Fix a basis $(\vep_1,\vep_2,\cdots, \vep_l, b,a)$ of $F$ satisfying $G=\R a$ and 
\begin{equation}\label{def_base-F}
I_F(\vep_i,\vep_j)=\delta_{i,j},\; I_F(F,b)=I_F(F,a)=0.
\end{equation}

Here, we use freely the notations on affine root systems as is given in Appendix \ref{sect_affine-root-system}. Fix an isometric section $F/G \hookrightarrow F$ so that its image is spanned by $\{ \vep_i \}_{1\leq i\leq l} \cup \{ b\}$ and that it induces $R/G \hookrightarrow R$. Via such a section, we identify $F/G$ with its image and regard $R/G$ as a root subsystem of $R$.  \\

Depending on the type of the full quotient $R/\rad(I_F)$, we introduce some mERSs $(R,G)$ as follows (classical type):
\begin{enumerate}
\item $R/\rad(I_F)$: of type $A_l\; (l\geq 1)$, $D_l\; (l\geq 4)$, $E_6$, $E_7$ or $E_8$. \quad
Let $R(X_l)$ be one of such finite root systems. The mERS of type $X_l^{(1,1)}$ is defined by
\[ R(X_l^{(1,1)})=R(X_l^{(1)})+\Z a=R(X_l)+\Z a+\Z b. \]
\item $R/\rad(I_F)$: of type $B_l\; (l\geq 3)$, $C_l\; (l\geq 2)$, $F_4$ or $G_2$. \quad
Let $R(X_l)$ be one of such finite root systems and let $t_1, t_2 \in \{1,2\}$ for $X_l$ of type $B_l, C_l$ or $F_4$ and $t_1,t_2 \in \{1,3\}$ for $X_l$ of type $G_2$. The mERS of type $X_l^{(t_1,t_2)}$ is defined by
\begin{align*}
R(X_l^{(t_1,t_2)})=
&(R(X_l^{(t_1)})_s+\Z a) \cup (R(X_l^{(t_1)})_l+t_2\Z a) \\
=
&(R(X_l)_s+\Z a+\Z b) \cup (R(X_l)_l+t_2\Z a+t_1\Z b). 
\end{align*}
The defining range of the root system $R(X_l^{(t_1,t_2)})$ with respect to $l$ is the same as that of the affine root system $R(X_l^{(t_1)})$. 
\item $R/\rad{I_F}$: of type $BC_l\; (l\geq 1)$.
\quad By the assumption that $R/G$ is reduced, it should be of type $BC_l^{(2)}$. That
is, the first tier number $t_1$ is equal to $2$ in this case. Recall that
\[ R(BC_l^{(2)})=(R(BC_l)_s+\Z b) \cup (R(BC_l)_m+\Z b) \cup (R(BC_l)_l+(1+2\Z)b). \]
Let $t_2 \in \{1,2,4\}$ be a positive integer.
\begin{enumerate}
\item[i)] For $t_2=1$, the root system of type 
$BC_l^{(2,1)}$ ($l\geq 1$) is defined by
\[
R(BC_l^{(2,1)})=R(BC_l^{(2)})+\Z a.
\]
\item[ii)] For $t_2=2$, let $i \in \{1,2\}$ be a positive integer. 
The root system of type $BC_l^{(2,2)}(i)$ 
($l\geq 2$ for $i=1$ and $l \geq 1$ for $i=2$) is defined by
\[
R(BC_l^{(2,2)}(i))=
(R(BC_l^{(2)})_s+\Z a) \cup (R(BC_l^{(2)})_m+i\Z a) \cup (R(BC_l^{(2)})_l+ 2\Z a).
\]
\item[iii)] For $t_2=4$, the root system of type $BC_l^{(2,4)}$ 
($l\geq 1$) is defined by
\[
R(BC_l^{(2,4)})=
(R(BC_l^{(2)})_s+\Z a) \cup (R(BC_l^{(2)})_m+2\Z a) \cup (R(BC_l^{(2)})_l+ 4\Z a).
\]
\end{enumerate}
\end{enumerate}
Surprisingly, there are another type of mERSs, called  $\ast$-type:
\begin{enumerate}
\item The root system of type $A_1^{(1,1)\ast}$:
\[ R(A_1^{(1,1)\ast})=R(A_1)+\{\, ma+nb\, \vert\, mn \equiv 0 [2]\, \}. \]
\item The root system of type $B_l^{(2,2)\ast}\; (l\geq 2)$:
\[ R(B_l^{(2,2)\ast})=(R(B_l)_s+\{\, ma+nb\, \vert\, mn \equiv 0 [2]\, \}) \cup (R(B_l)_l+2\Z a+2\Z b). \]
\item The root system of type $C_l^{(1,1)\ast}\; (l\geq 2)$:
\[ R(C_l^{(1,1)\ast})=(R(C_l)_s+\Z a+\Z b) \cup (R(C_l)_l+\{\, ma+nb\, \vert\, mn \equiv 0 [2]\, \}). \]
\end{enumerate}
Remark that the superscripts are the tier numbers. See Appendix \ref{sect_clasfficiation-red-quot} for detail. \\

K. Saito \cite{Saito1985} obtained the following classification theorem:
\begin{thm}\label{thm:classification-red} Let $(R,G)$ be a marked elliptic root system belonging to a real vector space $F=\R^{l+2}$ equipped with a symmetric bilinear form $I$ whose signature is $(l,2,0)$ for some positive integer $l\in \Z_{>0}$. Suppose that the affine quotient $R/G$ is reduced. Then $(R,G)$ is isomorphic to one of the following marked elliptic root systems:
\begin{align*}
&A_l^{(1,1)}\,(l\geq 1),\;  D_l^{(1,1)}\, (l\geq 4),\;  E_6^{(1,1)}, \; E_7^{(1,1)}, \; E_8^{(1,1)}, \\
&B_l^{(1,t)}\, (l\geq 3),\; C_l^{(1,t)}\, (l\geq 2),\; F_4^{(1,t)} \quad t \in \{1,2\}, \qquad 
  G_2^{(1,t)} \quad t \in \{1,3\}, \\
&B_l^{(2,t)}\, (l\geq 2),\; C_l^{(2,t)}\, (l\geq 3), \; F_4^{(2,t)} \quad t \in \{1,2\}, \qquad
  G_2^{(3,t)}\quad t \in \{1,3\}, \\
&BC_l^{(2,1)}, \; BC_l^{(2,4)} \quad (l\geq 1), \\
&BC_l^{(2,2)}(1)\, (l\geq 2), \; BC_l^{(2,2)}(2)\, (l\geq 1), \\
&A_1^{(1,1)\ast}, \; B_l^{(2,2)\ast}, \; C_l^{(1,1)\ast}\quad (l\geq 2).
\end{align*}
\end{thm}
A sketchy proof of this theorem is given by K. Saito in  \cite{Saito1985}. Here, we provide a simple complete proof as an application of Lemma \ref{lemma-counting}.
It should be noticed that none of the two mERSs in the above theorem are isomorphic, since two tier numbers are different, except for types $BC_l^{(2,2)}(i)$ $(i=1,2)$ where the number of Weyl group orbits on the root systems are different.
Therefore, in practice, the above theorem also gives the classification of the isomorphism classes of mERS with reduced affine quotient. 
\medskip
\subsection{Proof of Theorem \ref{thm:classification-red}}
A key idea of the proof comes from Lemma \ref{lemma-counting}, as the quotient root system $R/G$ is a reduced affine root system. In view of this lemma, it is convenient to regroup the reduced affine root systems, depending on the type of the orbit decomposition of $R/G$ with respect to the affine Weyl group $W(R/G)$, as follows:
\begin{enumerate}
\item $A_l^{(1)}\, (l\geq 2)$, \; $D_l^{(1)}\, (l\geq 4)$,\; $E_6^{(1)}$, \; $E_7^{(1)}$ and 
$E_8^{(1)}$.
\item $B_l^{(1)}\, (l\geq 3), \; F_4^{(1)}, \; G_2^{(1)}, \; C_l^{(2)}\, (l\geq 3), \; F_4^{(2)}$ 
and $G_2^{(3)}$.
\item $A_1^{(1)}, \; C_l^{(1)}\, (l\geq 2)$ and $B_l^{(2)}\, (l\geq 2)$.
\item $BC_l^{(2)}\, (l\geq 1)$.
\end{enumerate}
See Appendix \ref{sect_affine-root-system} for details on the affine root systems. 
\begin{rem} {\rm
Here we use the nomenclature for the reduced affine root systems due to 
K. Saito \cite{Saito1985}, since they reflect the structure of root systems 
themselves more than those of V. Kac \cite{Kac1990}. For the sake of reader's 
convenience, we have recalled them in  Appendix  \ref{sect_affine-root-system}.  
Here is the correspondence$:$ \\
\begin{center}
\scalebox{0.9}{%
\begin{tabular}{|c||c|c|c|c|c||c|} \hline
$\begin{matrix}
\text{\small V. Kac } \\
\text{\small $($1969, \cite{Kac1969}$)$ }
\end{matrix}$ 
& $X_l^{(1)}$ & $\begin{matrix} {\tiny D_{l+1}^{(2)}} \\ {\tiny l \geq 2} \end{matrix}$  & $\begin{matrix} {\tiny A_{2l-1}^{(2)}} \\ {\tiny l\geq 3} \end{matrix}$ & 
 $E_6^{(2)}$ & $D_4^{(3)}$  & $A_{2l}^{(2)}$\\ \hline
$\begin{matrix} \text{\small R. Moody} \\ \text{\small  $($1969, \cite{Moody1969}$)$} 
\end{matrix}$ 
& $X_{l,1}$ & $B_{l,2}$ & $C_{l,2}$  
& $F_{4,2}$ & $G_{2,3}$ & 
$\begin{matrix}
{\tiny A_{1,2}} & \text{\tiny $l=1$}\\ 
{\tiny BC_{l,2}} & \text{\tiny $l>1$} 
\end{matrix}$
\\ \hline
$\begin{matrix}
\text{\small I. Macdonald} \\ \text{\small $($1972, \cite{Macdonald1972}$)$}
\end{matrix}$
& 
$\begin{matrix}
{\tiny X_l=X_l^\vee} & \text{\tiny $X=ADE$}\\ 
{\tiny X_l} &  \text{\tiny $X=BCFG$} 
\end{matrix}$
& $C_{l}^\vee$ & $B_{l}^\vee$  
& $F_{4}^\vee$ & $G_{2}^\vee$ & $BC_l=BC_l^\vee$\\ \hline
$\begin{matrix}
\text{\small K. Saito } \\ 
\text{\small $($1985, \cite{Saito1985}$)$}
\end{matrix}$
& $X_l^{(1)}$  & $B_l^{(2)}$ & $C_l^{(2)}$ 
& $F_4^{(2)}$ & $G_2^{(3)}$ & $BC_l^{(2)}$\\ \hline
$\begin{matrix}
\text{\small R. Carter} \\
\text{\small  $($2005, \cite{Carter2005}$)$}
\end{matrix}$ 
& $\widetilde{X}_l$
& $\widetilde{C}_l^{\mathrm{t}}$ & $\widetilde{B}_l^{\mathrm{t}}$ 
& $\widetilde{F}_4^{\mathrm{t}}$ & $\widetilde{G}_2^{\mathrm{t}}$
& $\widetilde{C}_l'$ \\ \hline
\end{tabular}%
}
\end{center} }
\end{rem}
\vskip 0.1in
From now on, we fix a basis of the vector space $F=\R^{l+2}$ with the symmetric 
bilinear form $I_F$ with signature $(l,2,0)$. Let $(\vep_1,\vep_2,\cdots, \vep_l, b, a)$ be a basis of $F$ satisfying \eqref{def_base-F}. 
\subsubsection{$R/G$ of type $A_l^{(1)}\, (l\geq 2), \; D_l^{(1)}\, (l\geq 4), \; E_6^{(1)}, 
\, E_7^{(1)}, \, E_8^{(1)}$}

\smallskip

Let $X_l^{(1)}$ be one of such root systems. 
In this case, the real root system $R/G$ has a single Weyl group orbit. Hence, by 
Lemma \ref{lemma:countings1} (3), $k(\alpha)=1$ for any $\alpha \in R$. Thus, the 
root system $R$ can be described as follows:
\[ R=R(X_l^{(1)})+\Z a=R(X_l)+\Z b+\Z a.
\]
This is exactly the root system of type $X_l^{(1,1)}$.


\subsubsection{$R/G$ of type $B_l^{(1)}\, (l\geq 3), \; F_4^{(1)}, \; G_2^{(1)}, \; C_l^{(2)}\, 
(l\geq 3), \; F_4^{(2)}, \; G_2^{(3)}$}
Let $X_l^{(r)}$ be one of such root systems. In this case, the decomposition 
$R(X_l^{(r)})=R(X_l^{(r)})_s \amalg R(X_l^{(r)})_l$ gives the decomposition into the Weyl 
group orbits.  Hence the counting number only depends on the length of the roots. 
Moreover, for a long root $\alpha$ and short root $\beta$, one has 
$I_F(\beta, \alpha^\vee) \in \{ \pm 1\}$ if it is not $0$. Thus, by Lemma 
\ref{lemma:countings1} (2), we have $k(\beta)=1$. 
For each $s:=k(\alpha) \in \left\{ 1, \frac{I_F(\alpha, \alpha)}{I_F(\beta, \beta)}\right\}$, we obtain
\begin{align*}
R=
&(R(X_l^{(r)})_s+\Z a) \cup (R(X_l^{(r)})_l+s\Z a)\\
=
&(R(X_l)_s+\Z b +\Z a) \cup (R(X_l)_l+r\Z b+s\Z a),
\end{align*}
which is of type $X_l^{(r,s)}$. This  gives us the next list: \\
\begin{center}
\begin{tabular}{|c||c|c|c|c|c|c|} \hline
Type of $R/G$ & $B_l^{(1)}\; (l\geq 3)$ & $F_4^{(1)}$ & $G_2^{(1)}$ 
                       & $C_l^{(2)}\; (l\geq 3)$ & $F_4^{(2)}$ & $G_2^{(3)}$ \\ \hline
       $s=1$      & $B_l^{(1,1)}$ & $F_4^{(1,1)}$ & $G_2^{(1,1)}$ 
                      & $C_l^{(2,1)}$ & $F_4^{(2,1)}$ & $G_2^{(3,1)}$ \\ \hline
       $s=\dfrac{I_F(\alpha, \alpha)}{I_F(\beta, \beta)}$      
                      & $B_l^{(1,2)}$ & $F_4^{(1,2)}$ & $G_2^{(1,3)}$ 
                      & $C_l^{(2,2)}$ & $F_4^{(2,2)}$ & $G_2^{(3,3)}$ \\ \hline
\end{tabular}
\end{center}

\subsubsection{$R/G$ of type $A_1^{(1)}, \; C_l^{(1)}\, (l\geq 2), \; B_l^{(2)}\, (l\geq 2)$}

In this case,  $R/G$ admits a diagram automorphism $\overline{\varphi}$, which is defined 
by $\overline{\varphi}(\beta_i)=\beta_{l-i}$ ($0\leq i\leq l$) where $\{\beta_0, \cdots, 
\beta_l\}$ is a simple system of $R/G$. If 
$\overline{\varphi}$ lifts up to an automorphism $\varphi$ of $(R,G)$, then the counting 
number depends only on the length of the roots, in which case it follows that
\begin{enumerate}
\item $k(\alpha)=1$ for any $\alpha \in R$ if $R/G$ is of type $A_1^{(1)}$ and
\item $k(\alpha)=1$ for any $\alpha \in R_s$ and $k(\alpha) \in \{1,2\}$ for any 
$\alpha \in R_l$\\
 if $R/G$ is either of type $B_l^{(2)}$ or $C_l^{(1)}$.
\end{enumerate}
Hence, we consider the case when the automorphism $\overline{\varphi}$ of $R/G$ 
cannot be lifted up to an automorphism of $(R,G)$. In this case, one should have 
$\{ k(\alpha_0), \, k(\alpha_l) \}=\{1,2\}$. \\

\noindent{\fbox{The case $(k(\alpha_0), k(\alpha_l))=(1,2)$}} \quad \\
\begin{enumerate}
\item $R/G$: of type $A_1^{(1)}$
\[ R=\{ \alpha+nb+ma\, \vert\, \alpha \in R(A_1), \; m(n-1) \equiv 0\; [2]\, \}. \]
\item $R/G$: of type $B_l^{(2)}$
\[ R_s=\{ \alpha+nb+ma\, \vert\, \alpha \in R(B_l)_s, \; m(n-1) \equiv 0\; [2]\, \}. \]
By Lemma \ref{lemma:countings1} (2), one has $k(\alpha)=2$ for any 
$\alpha \in R_l$. 
\item $R/G$: of type $C_l^{(1)}$
\[ R_l=\{ \alpha+nb+ma\, \vert\, \alpha \in R(C_l)_l, \; m(n-1) \equiv 0\; [2]\, \}. \]
By Lemma \ref{lemma:countings1} (2), we have $k(\alpha)=1$ for any 
$\alpha \in R_s$. 
\end{enumerate}
\vskip+5mm

\noindent{\fbox{The case $(k(\alpha_0), k(\alpha_l))=(2,1)$}} \quad \\
\begin{enumerate}
\item $R/G$ of type $A_1^{(1)}$ 
\[ R=\{ \alpha+nb+ma\, \vert\, \alpha \in R(A_1), \; mn \equiv 0\; [2]\, \}. \]
\item $R/G$ of type $B_l^{(2)}$
\[ R_s=\{ \alpha+nb+ma\, \vert\, \alpha \in R(B_l)_s, \; mn \equiv 0\; [2]\, \}. \]
By Lemma \ref{lemma:countings1} (2), one has $k(\alpha)=2$ for any 
$\alpha \in R_l$. 
\item $R/G$: of type $C_l^{(1)}$. 
\[ R_l=\{ \alpha+nb+ma\, \vert\, \alpha \in R(C_l)_l, \; mn \equiv 0\; [2]\, \}. \]
By Lemma \ref{lemma:countings1} (2), one has $k(\alpha)=1$ for any 
$\alpha \in R_s$. 
\end{enumerate}
\vskip+5mm

Let us examine whether the cases $(k(\alpha_0), k(\alpha_l))=(1,2)$ and 
$(k(\alpha_0), k(\alpha_l))\\ =(2,1)$ provide isomorphic marked root systems or not. \\

In the case when $R/G$ is of type $A_1^{(1)}$, it can be checked that the linear map of 
$F$ defined by 
\[ \alpha\,  \mapsto\,  \alpha+b, \qquad b\,  \mapsto\,  b, \qquad a\,  \mapsto\,  a, \]
with a fixed positive root $\alpha$ of $R(A_1)$, 
induces an isomorphism of the two marked root systems which is of type $A_1^{(1,1)\ast}$.

In the case when $R/G$ is of type $B_l^{(2)}$,  the linear map of $F$ defined by 
\[ \vep_i \,  \mapsto\,  \vep_i+b, \qquad b\,  \mapsto\,  b, \qquad a\,  \mapsto\,  a \]
for $1\leq i\leq l$, 
induces an isomorphism of the two marked root systems which is of type $B_l^{(2,2)\ast}$.

In the case when $R/G$ is of type $C_l^{(1)}$,  the linear map of $F$ defined by 
\[ \vep_i \,  \mapsto\,  \vep_i+\frac{1}{2}b, \qquad b\,  \mapsto\,  b, \qquad a\,  
\mapsto\,  a \]
for $1\leq i\leq l$, 
induces an isomorphism of the two marked root systems which is of type $C_l^{(1,1)\ast}$.\\

Summarizing, we obtain the next table: \\

\begin{center}
\begin{tabular}{|c||c|c|c|} \hline
Type of $R/G$ & $A_1^{(1)}$ & $B_l^{(2)}\; (l\geq 2)$ & $C_l^{(1)}\; (l\geq 2)$ \\ \hline
$k(\alpha_0)k(\alpha_l)=1$
& $A_1^{(1,1)}$ & $B_l^{(2,1)}$ & $C_l^{(1,1)}$ \\ \hline
$k(\alpha_0)k(\alpha_l)=2$
&  $A_1^{(1,1)\ast}$ & $B_l^{(2,2)\ast}$ & $C_l^{(1,1)\ast}$ \\ \hline
$k(\alpha_0)k(\alpha_l)=4$
& & $B_l^{(2,2)}$ & $C_l^{(1,2)}$ \\ \hline
\end{tabular}
\end{center}

\subsubsection{$R/G$ of type $BC_l^{(2)}$}
Denote the finite root system of type $BC_l$ by $R_f$. 
In this case, there can be $3$ different values, say 
$k_s:=k(\vep_i), \; k_m:=k(\vep_i+\vep_j), \; k_l:=k(2\vep_i)$ with any $1\leq i\neq j\leq l$. 
By Lemma \ref{lemma:countings1} (2), we have
\begin{equation}\label{counting_red-red}
k_s \vert k_l \vert 4k_s, \qquad k_s \vert k_m \vert 2k_s, \qquad k_m \vert k_l \vert 2k_m.
\end{equation}
From Lemma \ref{lemma:countings1} (5), it follows that $k_s=1$. Hence the 
above equality implies $k_m \in \{1,2\}$. 
\begin{enumerate}
\item If $k_m=1$, then $k_l \in \{1,2\}$ by \eqref{counting_red-red}.
\begin{enumerate}
\item[i)] If $k_l=1$, then 
\begin{align*}
R=
&\{ \alpha+\Z a+\Z b\, \vert\,  \alpha \in (R_f)_s \cup (R_f)_m \} \\
&\cup
    \{ \alpha+\Z a+(2\Z -1)b\, \vert \, \alpha \in (R_f)_l \} 
\end{align*}
is of type $BC_l^{(2,1)}$, and 
\item[ii)] if $k_l=2$, then 
\begin{align*}
R=
&\{ \alpha+\Z a+\Z b\, \vert\,  \alpha \in (R_f)_s \cup (R_f)_m\} \\
&\cup
  \{ \alpha+2\Z a+(2\Z -1)b\, \vert \, \alpha \in (R_f)_l \} 
\end{align*}
is of type $BC_l^{(2,2)}(1)$.
\end{enumerate}
\item If $k_m=2$, then $k_l \in \{2,4\}$ by \eqref{counting_red-red}.
\begin{enumerate}
\item[i)] If $k_l=2$, then
\begin{align*}
R=
&\{ \alpha+\Z a+\Z b\, \vert\,  \alpha \in (R_f)_s \} \\
&\cup 
   \{ \alpha+2\Z a+\Z b\, \vert \, \alpha \in (R_f)_m \} \\
&\cup
    \{ \alpha+2\Z a+(2\Z -1)b\, \vert \, \alpha \in (R_f)_l \} 
\end{align*}
is of type $BC_l^{(2,2)}(2)$, and 
\item[ii)] if $k_l=4$, then 
\begin{align*}
R=
&\{ \alpha+\Z a+\Z b\, \vert\,  \alpha \in (R_f)_s \} \\
&\cup
    \{ \alpha+2\Z a+\Z b\, \vert \, \alpha \in (R_f)_m\} \\
&\cup
    \{ \alpha+4\Z a+(2\Z -1)b\, \vert \, \alpha \in (R_f)_l \} 
\end{align*}
is of type $BC_l^{(2,4)}$.
\end{enumerate}
\end{enumerate}

\section{Automorphisms of a mERS with reduced quotient}\label{sect:auto-mers-reduced-case}
\subsection{Preliminaries}
For the description of Weyl groups, we recall a notion of 
the Eichler-Siegel map, following \cite{Saito1985}.
Let $V$ be a finite-dimensional real vector space with a symmetric 
bilinear form $J:V\times V\to \R$ on $V$.  Define a semigroup structure 
$\circ$ on $V\otimes V/\mbox{rad}J$ by
\begin{equation*}
\Bigl(\sum_i u_i\otimes v_i\Bigr)\circ \Bigl(\sum_j w_j\otimes x_j\Bigr):=
\sum_i u_i\otimes v_i+\sum_j w_j\otimes x_j-\sum_{i,j}J(v_i,w_j)u_i\otimes x_j.
\end{equation*}
\begin{defn} The Eichler-Siegel map $E_V$ for $(V,J)$ is the map from 
$V\otimes V/\mathrm{rad}(J)$ to $\End(V)$ defined by
\[ E_V(\sum_i u_i \otimes v_i)(u)=u-\sum_i  I(v_i,u)u_i. \]
\end{defn}
From the definition, it follows that
\begin{enumerate}
\item[i)] the map $E_V:V \otimes V/\mathrm{rad}(J) \rightarrow \End(V)$ is
a homomorphism of semi-groups;
\item[ii)] it is injective, and is bijective if and only if $\mathrm{rad}{J}=0$;
\item[iii)] for a non-isotropic $\alpha \in V$, one has $r_\alpha=E_V(\alpha \otimes 
\alpha^\vee)$;
\item[iv)] the inverse of Eichler-Siegel map on the Weyl group $W(R)$ is well-defined;
\item[v)] for $\xi, \eta \in \mathrm{rad}(J) \otimes V/\mathrm{rad}(J)$, one has
\[E_V(\xi+\eta)=E_V(\xi)E_V(\eta). \]
\end{enumerate}
(See \cite{Saito1985}, for details.) \\

We return to the setting in \S \ref{subsec:general}.  Let $(F,I_F)$ be a pair of 
a finite-dimensional real vector space $F$ and a symmetric bilinear form 
$I_F:F\times F\to \R$ on $F$ with signature $(\mu_+,\mu_0,\mu_-)$ and 
$R$ a root system belonging to $(F,I_F)$. Furthermore, let $G$ be a subspace of 
$\rad (I_F)$ defined over $\Z$ (marking). 

Since every element $w\in W(R)$ preserves the subspace $G$ of 
$\mathrm{rad}(I_F)$, the 
canonical projection $\pi_G: F\to F/G$ induces the surjective group homomorphism
$(\pi_G)_*:W(R)\to W(R/G)$. Define a subgroup $T_G(R)$ of $W(R)$ by
\begin{equation}
T_G(R)=W(R)\cap E_F\big(G\otimes (F/\mathrm{rad}\, I_F)\big),
\end{equation}
where $E_F:F\otimes F/\mathrm{rad}\, I_F\to \mathrm{End} (F)
$\index[notations]{e311@$E_F$} is the 
Eichler-Siegel map. The following lemma is useful for describing the structure
of the Weyl group. 

\begin{lemma}[\cite{Saito1985}, (1.15), Assertion and Lemma]
\label{lemma:Saito(1.15)} 
\begin{enumerate}
\item There is an exact sequence
\begin{equation}\label{eq:exact-Weyl}
1 \longrightarrow T_G(R)\longrightarrow W(R) \xrightarrow{(\pi_G)_*} W(R/G) 
\longrightarrow 1.
\end{equation}
\item Assume there exists a vector subspace $L$ of $F$ which has 
the following properties{\rm :}
\begin{itemize}
\item[(i)] $F=L\oplus G${\rm ;}
\item[(ii)] The restriction $(\pi_G)_*\big|_{W(R\cap L)}:W(R\cap L)\to W(R/G)$ is 
surjective. Here, $W(R\cap L)$ is the subgroup of $W(R)$
generated by the reflections $r_\alpha\ (\alpha\in R\cap L)$.
\end{itemize}
Then, the map $(\pi_G)_*\big|_{W(R\cap L)}$ is an isomorphism. 
\vskip 1mm
\noindent
\item Under the assumption that such an $L$ exists, 
the map $\big((\pi_G)_*\big|_{W(R\cap L)}\big)^{-1}:W(R/G)\to W(R\cap L)$
gives a section of \eqref{eq:exact-Weyl}. That is, this exact sequence splits, and
 $W(R)$ can be described as a semidirect product{\rm :}
\begin{equation}
W(R)\cong T_G(R)\rtimes W(R\cap L).
\end{equation}
\vskip 1mm
\noindent
\item Under the same assumption, $E_F^{-1}\big(T_G(R)\big)$ is a lattice of 
$G\otimes F/\mathrm{rad}(I_F)$, which is generated by
\begin{equation}
\alpha_G\otimes \big(\alpha_L^\vee\, \mathrm{mod}\, \mathrm{rad}(I_F)\big)\quad\text{for}\quad \alpha\in R.
\end{equation} 
Here, $\alpha=\alpha_L+\alpha_G$ is the decomposition of $\alpha\in R$ with 
respect to the direct sum decomposition $F=L\oplus G$.
\end{enumerate}
\end{lemma}

\begin{rem}{\rm
In the proof of the above lemma (see \cite{Saito1985}, (1.15), 
Assertion and Lemma), both of the root systems $R$ and $R/G$ 
are not assumed to be reduced. 
}\end{rem}

\subsection{Elliptic Weyl groups}\label{sect:Weyl1}
In  the following two subsections, for a mERS $(R,G)$ belonging to $(F,I_F)$, 
the quotient affine root system $R/G$ is assumed to be reduced. 
In this subsection, under this assumption, we {study 
the explicit structure of the elliptic Weyl group $W(R)$. These results will be used in
Part 2 for studying the case that $R/G$ is non-reduced.} \\
\\
{\bf Step 1}. First, we fix an integral basis $(b,a)$ of $\rad_\Z(I_F)$. Indeed, take a basis $(\alpha_0,\ldots,\alpha_l)$ of 
$(R,G)$ as in Appendix \ref{sect_clasfficiation-red-quot}. Note that the element
$\alpha_0$ is chosen as the condition
\begin{equation*}\label{cond:n_0=1}
n_0=1
\end{equation*}
is satisfied. The elements $a$ and $b$ are introduced in \eqref{defn:a} and 
\eqref{defn:b}, respectively. 
\vskip 3mm
\noindent
{\bf Step 2}.
Second, recall the $(l+1)$-dimensional subspace $F_a=\bigoplus_{i=0}^l \R\alpha_i$
of $F$,  introduced in \eqref{defn:L^{l+1}}. By \eqref{eq:isom-p_G}, we see 
that $L=F_a$ satisfies the properties {\rm (i)} and {\rm (ii)} 
of Lemma \ref{lemma:Saito(1.15)} {\rm (2)}. 
Thus, the results of the lemma can be applied for such an $(R,G)$. Recall 
the split exact sequence \eqref{eq:exact-Weyl}:
\begin{equation*}
1 \longrightarrow T_G(R)\longrightarrow W(R) \xrightarrow{(\pi_G)_*} W(R/G) 
\longrightarrow 1.
\end{equation*}
By Lemma \ref{lemma:countings1} (4) and Lemma
\ref{lemma:Saito(1.15)} {\rm (4)}, the lattice $E_F^{-1}\big(T_G(R)\big)$ is generated by
\[k(\alpha)a\otimes \big(\alpha^\vee\,\mathrm{mod}\,\rad (I_F)\big)
=a\otimes \big(\alpha^\dagger\,\mathrm{mod}\,\rad (I_F)\big)
\quad \text{for}\quad \alpha\in R\cap L,\]
where
\begin{equation}\label{defn:dagger}
\alpha^\dagger:=k(\alpha)\alpha^\vee.\index[notations]{a311@$\alpha^\dagger$}
\end{equation}
The following is known:
\begin{itemize}
\item[(i)] The set $\{\alpha^\dagger\,|\, \alpha\in R\cap F_a\}$
is an affine root system which is denoted by 
$(R,G)_a$\index[notations]{r311@$(R,G)_a$}, 
and its ordered subset 
$\{\alpha_i^\dagger\}_{0\leq i\leq l}$ is a simple system of $(R,G)_a$
(see \cite{SaitoTakebayashi1997}),
\end{itemize}
{and the next formula follows easily by the definition:
\begin{itemize}
\item[(ii)] $Q((R,G)_a)\equiv \bigoplus_{i=0}^l\Z\alpha_i^\dagger\,
(\mathrm{mod}\,\rad(I_F))$.
\end{itemize}}
Define a lattice $Q_{F/\mathrm{rad}(I_F),a}\subset F/\rad (I_F)$ of rank $l$ by
\begin{equation}\label{defn:Q_{f,a}}
Q_{{F/\mathrm{rad}(I_F),a}}=\bigoplus_{i=0}^l\Z\alpha_i^\dagger\, (\mathrm{mod}\,\rad(I_F)).
\index[notations]{q311@$Q_{F/\mathrm{rad}(I_F),a}$}
\end{equation}
Combining the above results, we see that
\begin{equation}
{T_G(R)}=E_F(a\otimes Q_{F/\mathrm{rad}(I_F),a}),
\end{equation}
and there exists a semidirect product expression
\begin{equation}
W(R)\cong E_F(a\otimes Q_{F/\mathrm{rad}(I_F),a})\rtimes W(R\cap F_a).
\end{equation} 
The explicit form of $Q_{F/\mathrm{rad}(I_F),a}$ is given in Table 1 below.
\vskip 3mm
\noindent
{\bf Step 3}. Third, study the group $W(R/G)\cong W(R\cap F_a)$. 
Since $R/G$ is a reduced affine root system
belonging to $(F/G,I_{F/G})$, the structure of $W(R/G)$ is well-known. There 
exists a lattice $Q_{F/\mathrm{rad}(I_F),b}\subset (F/G)/\rad (I_{F/G})$ 
\index[notations]{q312@$Q_{F/\mathrm{rad}(I_F),b}$}
of rank $l$ so that
\begin{equation}
W(R/G)\cong 
E_{F/G}\big(\overline{b}\otimes Q_{F/\mathrm{rad}(I_F),b}\big)\rtimes 
W\big((R/G)/\rad (I_{F/G})\big),
\end{equation}
{where $E_{F/G}:(F/G)\otimes (F/G)/\mathrm{rad}(I_{F/G})\longrightarrow 
\mathrm{End}(F/G)$ is the Eichler-Siegel map for $(F/G,I_{F/G})$. 
\index[notations]{e312@$E_{F/G}$}}
By the canonical isomorphism $(F/G)/\rad (I_{F/G})\cong F/\rad (I_F)$, the lattice
$Q_{F/\mathrm{rad}(I_F),b}$ is regarded as a sublattice in 
$F/\rad(I_F)$, and the root system 
$(R/G)/\rad (I_{F/G})$ is identified with $R/\rad (I_F)$. 
Note that $R/ \rad (I_F)$ is a finite root system belonging to 
$(F/\mathrm{rad}(I_F),I_{F/\mathrm{rad}(I_F)})${\color{red},} which is not 
{necessarily} reduced.

Hence, we have
\begin{equation}
W(R/G)\cong 
E_{F/G}\big(\overline{b}\otimes Q_{F/\mathrm{rad}(I_F),b}\big)\rtimes W(R/\rad(I_F)).
\end{equation}
The explicit form of $Q_{F/\mathrm{rad}(I_F),b}$ is given in Table 1 below.
\vskip 3mm
\noindent
{\bf Step 4}. 
Finally, applying Lemma \ref{lemma:Saito(1.15)} for
the pair $(R,\rad (I_F))$, we have the following proposition. 

\begin{prop}\label{prop:SaitoSaito1}
\begin{enumerate}
\item
Define a subspace $F_f$ of $F$ by
\begin{equation}\label{eq:F_f}
F_f=\bigoplus_{i=1}^l\R \alpha_i.\index[notations]{f311@$F_f$}
\end{equation}
Then,  it satisfies properties {\rm (i)} and {\rm (ii)} of Lemma 
\ref{lemma:Saito(1.15)} {\rm (2)} for the pair $(R,\mathrm{rad}(I_F))$. 
That is, the results of the lemma can be applied for such an $(R,\mathrm{rad}(I_F))$. 
Especially, 
there exists a split exact sequence
\begin{equation}\label{eq:exact-Weyl2}
1 \longrightarrow T_{\mathrm{rad}(I_F)}(R) \longrightarrow W(R) 
\xrightarrow{(\pi_{\rad (I_F)})_*} W(R/\mathrm{rad}(I_F)) 
\longrightarrow 1.
\end{equation}
\vskip 1mm
\noindent
\item Recall the lattices $Q_{F/\rad(I_F),a},Q_{F/\rad(I_F),b}\subset  F/\mathrm{rad}(I_F)$ of 
rank $l$ introduced above. We have 
\begin{equation}
T_{\mathrm{rad}(I_F)}(R)=E_F\big(a\otimes Q_{F/\rad (I_F),a}+b\otimes Q_{F/\rad (I_F),b}\big).
\end{equation} 
Therefore, we get a semidirect product expression
\[W(R)\cong E_F\big(a\otimes Q_{F/\rad(I_F),a}+b\otimes Q_{F/\rad(I_F),b}\big)\rtimes 
W(R/\rad (I_F)). \]
\end{enumerate}
\end{prop}
\begin{center}
\begin{table}[h]
\caption{An explicit description of $Q_{F/\rad (I_F),\sharp}\ (\sharp=a,b)$}
\label{table2-auto}
\renewcommand{\arraystretch}{1.4}
\begin{tabular}{|c|c|c|}
\noalign{\hrule height0.8pt}
$(R,G)$&  $Q_{F/\rad (I_F),a}$ & $Q_{F/\rad (I_F),b}$ \\
\noalign{\hrule height0.8pt}
\hline
$X_l^{(1,1)}$ & $Q\big((R/\rad(I_F))^{\vee}\big)$ & $Q\big((R/\rad(I_F))^{\vee}\big)$ \\
\hline
$X_l^{(1,t)}$ & 
$\widetilde{Q}\big((R/\rad (I_F))^\vee\big)$
& $Q\big((R/\rad (I_F))^\vee\big)$ \\
\hline
$X_l^{(t,1)}$ & $Q\big((R/\rad (I_F))^\vee\big)$ &
$\widetilde{Q}\big((R/\rad (I_F))^\vee\big)$ \\
\hline
$X_l^{(t,t)}$  & 
$\widetilde{Q}\big((R/\rad (I_F))^\vee\big)$ &
$\widetilde{Q}\big((R/\rad (I_F))^\vee\big)$ \\
\hline
\hline
$BC_l^{(2,1)}$ & 
$Q\big((R/\rad (I_F))^\vee\big)$ & $Q\big((R/\rad (I_F))^\vee\big)$ \\
\hline
$BC_l^{(2,4)}$ & 
$2Q\big((R/\rad (I_F))^\vee\big)$ & $Q\big((R/\rad (I_F))^\vee\big)$ \\
\hline
$BC_l^{(2,2)}(1)$  & 
$Q\big((R/\rad(I_F))^{\dc})^\vee\big)$
& $Q\big((R/\rad (I_F))^\vee\big)$
\\
\hline
$BC_l^{(2,2)}(2)$ & 
$\widetilde{Q}\big((R/\rad (I_F))^\vee\big)$ & $Q\big((R/\rad (I_F))^\vee\big)$\\
\hline
\hline
$A_1^{(1,1)\ast}$  & $Q\big((R/\rad (I_F))^\vee\big)$ & $Q\big((R/\rad (I_F))^\vee\big)$\\
\hline
$B_l^{(2,2)\ast}$  & 
$\widetilde{Q}\big((R/\rad (I_F))^\vee\big)$ & 
$\widetilde{Q}\big((R/\rad (I_F))^\vee\big)$ \\
\hline
$C_l^{(1,1)\ast}$  & 
$Q\big((R/\rad (I_F))^\vee\big)$ & $Q\big((R/\rad (I_F))^\vee\big)$ \\
\noalign{\hrule height0.8pt}
\end{tabular}
\end{table}
\end{center}
Here $X=A,B,C,D,E,F$ or $G$, and $t=2,3$. In addition, 
$\widetilde{Q}\big((R/\rad (I_F))^\vee\big)$ is the sub-lattice of $Q\big((R/\rad (I_F))^\vee\big)$
defined by 
\[\widetilde{Q}\big((R/\rad (I_F))^\vee\big)=\sum_{\gamma\in R/\rad(I_F)}
\dfrac{I_{F/\rad(I_F)}(\gamma,\gamma)}{I_{F/\rad(I_F)}(\gamma_s,\gamma_s)}
\mathbb{Z}\gamma^\vee
\index[notations]{q313@$\widetilde{Q}\big((R/\rad (I_F))^\vee\big)$}
\]
where $\gamma_s\in R/\rad (I_F)$ is a short root.
\medskip
\subsection{Modular part}\label{sect_Auto-I}

{
In the following three subsections, we study the structure of $\mathrm{Aut}(R)$.
We use these results in the weak classification of mERSs with non-reduced
affine quotients (see $\S$ \ref{sect:Classification-weak} and Section \ref{sect_mers-II}
below).}

\begin{rem}
In \cite{Satake1995}, Satake also gave a description of the group $\mathrm{Aut}(R)$. 
However, he excluded the 
mERSs $(R,G)$ of types $BC_l^{(2,1)}$, 
$BC_l^{(2,4)}$, $BC_l^{(2,2)}(1)$, $BC_l^{(2,2)}(2)$, $A_1^{(1,1)\ast}$, 
$B_l^{(2,2)\ast}$ and $C_l^{(1,1)\ast}$ as ``exceptions''. 
Some of these exceptions are necessary for the weak 
classification of mERSs with non-reduced affine quotients in 
$\S$ \ref{sect:Classification-weak} and Section \ref{sect_mers-II}.
\end{rem}

By Lemma \ref{lemma:Aut(R)}, 
an element $\psi\in \mathrm{Aut}(R)$ preserves
$\mathrm{rad}_{\Z}(I_F)=\rad (I_F)\cap Q(R)$. In other words, the restriction
$\psi|_{\mathrm{rad}(I_F)}$ belongs to $GL(\mathrm{rad}_{\Z}(I_F))$,
where
\[GL(\mathrm{rad}_{\Z}(I_F)):=\big\{g\in GL(\mathrm{rad}(I_F))\,\big|\,
g(\mathrm{rad}_{\Z}(I_F))=\mathrm{rad}_{\Z}(I_F)\big\}.\]
The map defined by $\psi\mapsto \psi|_{\mbox{\scriptsize rad}(I_F)}$ is 
denoted by $\resrad:\mathrm{Aut}(R)\to GL(\mbox{rad}_{\Z}(I_F))$.
Its image
\[\Gamma(R):=\resrad (\mathrm{Aut}(R))\index[notations]{g311@$\Gamma(R)$}\]
is called the {\bf modular part}\index[index]{modular part} of the automorphism group. %
Define the {\bf elliptic modular group}\index[index]{elliptic modular group} 
$\Gamma^+(R)$ attached to $R$ by
\[\Gamma^+(R)=\Gamma(R)\cap SL(\mbox{rad}_{\Z}(I_F)),
\index[notations]{g312@$\Gamma^+(R)$}\]
where 
\[SL(\mbox{rad}_{\Z}(I_F))
:=\big\{g\in GL(\mbox{rad}_{\Z}(I_F))\,\big|\,\det g=1\big\}.\]
Consider the integral basis $(b,a)$ of $\rad_\Z(I_F)$.  There  is
an isomorphism 
$SL(\mbox{rad}_{\Z}(I_F))\xrightarrow{\sim}SL_2(\Z)$
defined by 
\begin{equation}\label{eqn:SL2}
g\longmapsto \begin{pmatrix} p & q \\ r & s\end{pmatrix}\quad \text{where}
\quad
(g(b),g(a))=(b,a)
\begin{pmatrix} p & q \\ r & s\end{pmatrix}.
\end{equation}

Define the subgroup $\mathrm{Aut}^+(R)$
of orientation preserving automorphisms by
\[\mathrm{Aut}^+(R)
=\bigl\{\psi\in \mathrm{Aut}(R)\,\bigl|\,\psi|_{\mbox{\scriptsize rad}(I_F)}\in
SL(\mbox{rad}_{\Z}(I_F))\bigr\}.\index[notations]{aut311@$\mathrm{Aut}^+(R)$}\]
Hence, we have the following commutative diagram
{\small
\begin{equation}\label{comm-diagram-Autom}
\begin{array}{l}\xymatrix@C=35pt
{& & 1 \ar[d] & 1 \ar[d]\\
1 \ar[r] & \mathrm{Aut}(R,\rad (I_F)) \ar[r] \ar@{=}[d] & \mathrm{Aut}^+(R)
\ar[r]^-{\resrad} \ar[d] & \Gamma^+(R) \ar[r] \ar[d] & 1\\
1\ar[r] & \mathrm{Aut}(R,\rad (I_F)) \ar[r] & \mathrm{Aut}(R) \ar[r]^-{\resrad} 
\ar[d] & \Gamma(R) \ar[r] \ar[d] & 1\\
& & \Z/2\Z \ar[r]^{\sim} \ar[d] & \Z/2\Z  \ar[d]& \\
& &1& 1&
}\end{array},
\end{equation}}
where
\[\mathrm{Aut}(R,\rad (I_F)):=\big\{\psi\in\mathrm{Aut}(R)\,\big|\,\psi|_{\rad (I_F)}
=\mathrm{id}_{\rad (I_F)}\big\}.
\index[notations]{aut313@$\mathrm{Aut}(R,\rad (I_F))$}\]

Both the left and right vertical exact sequences in 
\eqref{comm-diagram-Autom}
split. Indeed, 
recalling the subspace $F_f$ of $F$ {introduced in \eqref{eq:F_f}}, 
we have a decomposition of $F$:
\begin{equation}\label{eq:DecompositionF}
F = F_f \oplus \rad(I_F) = F_f \oplus (\R b \oplus \R a).
\end{equation}
Define an involutive automorphism $\iota_{F_f} \in \mathrm{Aut}(R)$ by
\begin{equation}\label{eq:iota_Ff}
\iota_{F_f}|_{F_f} =\mathrm{id}_{F_f},\quad \iota_{F_f} (b) = b,\quad
\iota_{F_f} (a) = -a.
\end{equation}
By the construction, we immediately have the
semidirect product expression
\[\mathrm{Aut}(R) =\mathrm{Aut}^+(R) \rtimes \langle\iota_{F_f}\rangle,\]
 where 
$\langle\iota_{F_f}\rangle\cong \Z/2\Z$ is the group generated by 
$\iota_{F_f} \in \mathrm{Aut(R)}$, and this induces a semidirect 
expression
$\Gamma(R) = \Gamma^+(R) \rtimes \iota_{F_f} |_{\rad_\Z (I_F)} \cong
\Gamma^+(R) \rtimes \Z/2\Z$. Therefore, the study of $\mathrm{Aut}(R)$
is reduced to one of $\mathrm{Aut}^+(R)$. \\

{
Let us study the group $\Gamma^+(R)$. 
Recall some notations from the theory of congruence subgroups 
(see \cite{Koblitz1993}, for example). For
a positive integer $N$, set
\begin{align*}
\Gamma_0(N)=
&\left\{\left.\begin{pmatrix} p & q\\ r & s\end{pmatrix}
\in SL_2(\Z)\,\right|\,
r\equiv 0\, (\mathrm{mod}\, N)
\right\},\\
\Gamma^0(N)=
&\left\{\left.\begin{pmatrix} p & q\\ r & s\end{pmatrix}
\in SL_2(\Z)\,\right|\,
q\equiv 0\, (\mathrm{mod}\, N)
\right\}.
\end{align*}
Define a subgroup $\Gamma_0^+(R)$ of $SL_2(\Z)$ by
\[\Gamma_0^+(R)=\begin{cases}
SL_2(\Z) & \text{for}\ X_l^{(t,t)}\ \text{and}\ X_l^{(t,t)\ast} \, (t=1,2,3),\\
\Gamma_0(2) & \text{for}\ X_l^{(1,2)},\ BC_l^{(2,4)},\ 
BC_l^{(2,2)}(1)\ \text{and}\ BC_l^{(2,2)}(2),\\
\Gamma^0(2)& \text{for}\ X_l^{(2,1)}\ \text{and}\ BC_l^{(2,1)},\\
\Gamma_0(3)& \text{for}\ G_2^{(1,3)},\\
\Gamma^0(3)& \text{for}\ G_2^{(3,1)}.
\end{cases}\]

The first step in studying the elliptic modular group $\Gamma^+(R)$
is the following lemma.

\begin{lemma}\label{lemma:Gamma0+}
Under the identification \eqref{eqn:SL2}, we regard $\Gamma^+(R)$ as 
a subgroup of $SL_2(\Z)$. Then, $\Gamma_0^+(R)$ is a subgroup of 
$\Gamma^+(R)$. 
\end{lemma}

Before proving the above lemma, we need certain preparations.

\vskip 3mm
\noindent
{\bf 1}. First, let us introduce an element $\rho\in F_f$ by
\begin{equation*}\label{eqn:rho}
\rho=\begin{cases}
\tfrac{1}{2}(\varepsilon_1+\cdots+\varepsilon_l) & 
\mbox{for $BC_l^{(2,4)}$ and $C_l^{(1,1)\ast}$},\\
\tfrac12 \varepsilon & \mbox{for $A_1^{(1,1)\ast}$},\\ 
\varepsilon_1+\cdots+\varepsilon_l& 
\mbox{for $B_l^{(2,2)\ast}$},\\
0 & \mbox{otherwise}.
\end{cases}
\end{equation*}
Here, $\varepsilon_1,\ldots, \varepsilon_l$ is an orthonormal basis of $F_f$ 
(see \eqref{def_base-F}), and
$\varepsilon$ is an element of $F_f$ such that $I_F(\varepsilon,\varepsilon) = 2$
(see Appendix \ref{sect_clasfficiation-red-quot}).
\begin{rem}
For later convenience, we denote $\varepsilon_1=\varepsilon$ in the case of 
$A_1^{(1,1)\ast}$.  
\end{rem}
Define an element $\mathcal{T}=\mathcal{T}_{(R,G)}\in 
O(F,I_F)$ by
\begin{equation*}\label{eqn:T}
\mathcal{T}=\begin{cases}
E_F(b\otimes \rho) & \mbox{if $(R,G)$ is of type $BC_l^{(2,4)}$},\\
E_F((b+a)\otimes \rho) & 
\!\!\! \hspace*{0.2mm}\begin{array}{l}\mbox{if $(R,G)$ is either of type}\\
\qquad \mbox{$A_1^{(1,1)\ast}$, $B_l^{(2,2)\ast}$ or $C_l^{(1,1)\ast}$,}
\end{array}\\
\mathrm{id}_F & \mbox{otherwise}.
\end{cases}
\end{equation*}
Note that the operator $\mathcal{T}$ does not belong to $\mathrm{Aut}(R)$ in general. 
\\

Set $F_f^{sp}=\mathcal{T}(F_f)$. Then, we have a decomposition 
\begin{equation}\label{eqn:Lf'}
F=F_f^{sp}\oplus\mathrm{rad}(I_F).
\end{equation}
(A) If $(R,G)$ is an mERS except for of type $BC_l^{(2,4)}$, $A_1^{(1,1)\ast}$, 
$B_l^{(2,2)\ast}$ and $C_l^{(1,1)\ast}$, the operator $\mathcal{T}$ is $\mathrm{id}_F$.
Therefore, we have $F_f^{sp}=F_f$ and the direct sum decomposition \eqref{eqn:Lf'} 
of $F$ coincides with \eqref{eq:DecompositionF}.
\vskip 3mm
\noindent 
(B) Assume $(R,G)$ is of type $BC_l^{(2,4)}$, $A_1^{(1,1)\ast}$, 
$B_l^{(2,2)\ast}$ or $C_l^{(1,1)\ast}$. Recalling the explicit description of 
$\alpha_i\ (1\leq i\leq l)$ in Appendix \ref{sect_clasfficiation-red-quot}, we have
\begin{equation}\label{eq:Ff}
F_f=\bigoplus_{i=1}^l \R \varepsilon_i
\end{equation}
for such an $(R,G)$.  
Set $\varepsilon_i^{sp}=-\mathcal{T}(\varepsilon_i)\ (1\leq i\leq l)$.
The explicit form of $\varepsilon_i^{sp}$ is given as
\begin{equation}\label{eq:eps-sp}
\varepsilon_i^{sp}=
\left\{\begin{array}{ll}
-\varepsilon_i+\tfrac12 b & \mbox{if } (R,G)=BC_l^{(2,4)},\\
-\varepsilon_i+b+a & \mbox{if } (R,G)=A_1^{(1,1)\ast}, B_l^{(2,2)\ast},\\
-\varepsilon_i+\tfrac12 b+\tfrac12 a & \mbox{if } (R,G)=C_l^{(1,1)\ast}.
\end{array}\right.
\end{equation}
By \eqref{eq:Ff}, it follows immediately that
\[F_f^{sp}=\bigoplus_{i=1}^l\R\varepsilon_i^{sp}.\]
Note that $F_f^{sp}\ne F_f$ in these cases.  
\vskip 5mm
\noindent
{\bf 2}. Second, 
let $(R,G)$ be a mERS of type  $BC_l^{(2,4)}$, $A_1^{(1,1)\ast}$, 
$B_l^{(2,2)\ast}$ or $C_l^{(1,1)\ast}$. By the explicit description of 
the set $R$ in Appendix \ref{sect_clasfficiation-red-quot} for such a mERS and 
the formula \eqref{eq:eps-sp}, we have a description of $R$ in terms of a basis
$\varepsilon_1^{sp},\ldots,\varepsilon_l^{sp},b,a$ of $F$ as follows:
\begin{equation}\label{eqn:R-*}
\begin{array}{lll}
R&\!\!=\left\{\!\!\begin{array}{ll}
\pm \varepsilon_i^{sp}+\left(r+\tfrac12\right)b+sa & (1\leq i\leq l,\, r,s\in\mathbb{Z}),\\
\pm 2\varepsilon_i^{sp}+2rb+4sa& (1\leq i\leq l,\, r,s\in\mathbb{Z}),\\
\pm\varepsilon_i^{sp}\pm\check{\varepsilon}_j+rb+2sa & 
(1\leq i<j\leq l,\, r,s\in\mathbb{Z})
\end{array}\right\} & \!\!\mbox{for }BC_l^{(2,4)},\\[20pt]
R&\!\! =\left\{\!\!\begin{array}{ll}
\pm\varepsilon_1^{sp}+rb+sa \ \ 
\left(\!\!\begin{array}{c}1\leq i\leq l,\, r,s\in\mathbb{Z}
\mbox{ such that}\\
(r-1)(s-1)\equiv 0\mbox{ mod }2
\end{array}\!\!\right)
\end{array}\!\!\right\} & \!\!\mbox{for }A_1^{(1,1)\ast},\\[12pt]
R&\!\! =\left\{\!\! \begin{array}{ll}
\pm \varepsilon_i^{sp}+rb+sa \ \  \left(
\!\!\begin{array}{c}1\leq i\leq l,\, r,s\in\mathbb{Z}
\mbox{ such that}\\
(r-1)(s-1)\equiv 0\mbox{ mod }2
\end{array}\!\!\right),\\[10pt]
\pm\varepsilon_i^{sp}\pm\varepsilon_j^{sp}+2rb+2sa \ \   
(1\leq i<j\leq l,\, r,s\in\mathbb{Z})
\end{array}\!\!\right\} & \!\!\mbox{for }B_l^{(2,2)\ast},\\[20pt]
R&\!\! =\left\{\!\! \begin{array}{ll}
\pm 2\varepsilon_i^{sp}+rb+sa \ \ 
\left(\!\!\begin{array}{c}1\leq i\leq l,\, r,s\in\mathbb{Z}
\mbox{ such that}\\
(r-1)(s-1)\equiv 0\mbox{ mod }2
\end{array}\!\!\right),\\[10pt]
\pm\varepsilon_i^{sp}\pm\varepsilon_j^{sp}+rb+sa \ \  
(1\leq i<j\leq l,\, r,s\in\mathbb{Z})
\end{array}\!\!\right\} & \!\!\mbox{for }C_l^{(1,1)\ast}.
\end{array}\end{equation}
These descriptions are used in the proof of Lemma \ref{lemma:Gamma0+}.  
\vskip 5mm
\noindent
{\bf 3}. Third, introduce two subgroups $O_{F_f}(F,I_F)$ 
and $O_{F_f^{sp}}(F,I_F)$ of $O(F,I_F)$ by 
\[
O_{F_f^{\natural}}(F,I_F) = \{ \psi \in O(F,I_F) \ | \ 
\psi(F_f^{\natural})\subset F_f^{\natural}\}
\quad (\natural =\emptyset\, \mbox{or}\, sp).
\]
Then, we have a group isomorphism
\begin{equation*}\label{eqn:AdT2}
Ad_{\mathcal{T}}\ :\ O_{F_f}(F,I_F)\xrightarrow{\sim} 
O_{F_f^{sp}}(F,I_F),\quad 
\psi\ \longmapsto \ \mathcal{T}\psi\mathcal{T}^{-1}.
\end{equation*}

Consider an injective group homomorphism
\begin{equation}\label{eq:split}
\chi_{F_f} \ : \ 
GL(\mathrm{rad}_{\mathbb{Z}}(I_F)) \longrightarrow O_{F_f}(F,I_F), \quad 
g\ \longmapsto \ \chi_{F_f}(g)
\end{equation}
where
$\chi_{F_f}(g)|_{L_f}=\mathrm{id}_{F_f}, \chi_{F_f}(g)|_{\mathrm{rad}(I_F)}=g$, a
nd define
\begin{equation*}\label{eq:split'}
\chi_{F_f^{sp}} \ : \ 
GL(\mathrm{rad}_{\mathbb{Z}}(I_F)) \longrightarrow O_{F_f^{sp}}(F,I_F), \quad 
g\ \longmapsto \ Ad_{\mathcal{T}}\big(\chi_{F_f}(g)\big).
\end{equation*}
Since the restriction $\mathcal{T}|_{\mathrm{rad}(I_F)}$ is equal to 
$\mathrm{id}_{\mathrm{rad}(I_F)}$, we have
\begin{equation*}\label{eqn:section'}
\chi_{F_f^{sp}}(g)|_{\mathrm{rad}(I_F)}=g
\end{equation*}
for every $g\in GL_2(\mathbb{Z})$. 

\begin{rem}\label{rem:chi_Lf}
{\rm (1) Restricting these homomorphisms, we have group homomorphisms
\[\chi_{F_f^{\natural}}|_{SL(\mathrm{rad}_{\mathbb{Z}}(I_F))} \ : \ 
SL(\mathrm{rad}_{\mathbb{Z}}(I_F)) \longrightarrow O_{F_f^{\natural}}(F,I_F)
\quad\mbox{for }\natural=\emptyset\ \mbox{or}\ sp. \]
Unless there is a risk of confusion, we denote them by 
$\chi_{F_f^{\natural}}=\chi_{F_f^{\natural}}|_{SL(\mathrm{rad}_{\mathbb{Z}}(I_F))}$, for
simplicity. 
\vskip 1mm
\noindent
(2) By the definition, the operator $\iota_{F_f}$ introduced in \eqref{eq:iota_Ff}
coincides with $\chi_{F_f}\left(\begin{pmatrix} 1 & 0\\ 0 & -1\end{pmatrix}\right)$.
}\end{rem}
\vskip 3mm
Now, let us prove Lemma \ref{lemma:Gamma0+}. It is enough to prove the following
claim. 
\vskip 3mm
\noindent
{\bf Claim}. 
{\it For every $g\in \Gamma_0^+(R)$, the element $\chi_{L_f^{sp}}(g)$ belongs to 
$\mathrm{Aut}^+(R)$. That is, $\Gamma_0^+(R)$ is a subgroup of 
$\Gamma^+(R)$. }

\begin{proof}
Assume $F_f^{sp}=F_f$. Then, the statement follows from case-by-case 
brute-force computation by using the explicit description of the set $R$ in 
Appendix \ref{sect_clasfficiation-red-quot}. Since it is straightforward,  
we leave the proof to the reader. 

If $F_f^{sp}\ne F_f$, we use the description \eqref{eqn:R-*} of $R$. In the following, 
we give a proof only for of type $B_l^{(2,2)\ast}$. For the other case, the lemma is 
obtained by similar method. 

Recalling the description \eqref{eqn:R-*} of the set $R$ for $B_l^{(2,2)\ast}$, 
it is divided into two pieces $R=R_s\sqcup R_l$ where
\begin{equation*}
\begin{array}{ll}
R_s:=\left\{\pm \varepsilon_i^{sp}+rb+sa\,\left|\begin{array}{c}
1\leq i\leq l,\, r,s\in\mathbb{Z}\mbox{ such that}\\
(r-1)(s-1)\equiv 0\,(\mathrm{mod}\,2)
\end{array}\right.\!\! \right\},\\[10pt]
R_l:=\{\pm\varepsilon_i^{sp}\pm\varepsilon_j^{sp}+2rb+2sa\,|\,
1\leq i<j\leq l,\, r,s\in\mathbb{Z}\}.
\end{array}
\end{equation*}

Since $R_l$ is obviously preserved by $\chi_{F_f^{sp}}(g)$ for every 
$g\in {SL}_2(\mathbb{Z})$, it is enough to show
\begin{equation}\label{eqn:stablility}
\chi_{F_f^{sp}}(g)(R_s)\subset R_s\quad \mbox{for every }g=
\begin{pmatrix} x & y \\ z & w\end{pmatrix}
\in {SL}_2(\mathbb{Z}).
\end{equation}

For $\alpha_s=\pm\varepsilon_i^{sp}+rb+sa\in R_s$, we have
\begin{align*}
\chi_{F_f^{sp}}(g)(\alpha_s)=\pm \varepsilon_i^{sp}+(rx+sy)b+(rz+sw)a.
\end{align*}
By congruent expressions
$xw+yz\equiv 1\ \mathrm{mod}\,2$, $rs+1\equiv r+s\ \mathrm{mod}\,2$, 
$r^2\equiv r\ \mathrm{mod}\,2$ and $s^2\equiv s\ \mathrm{mod}\,2$ which
are verified in this setting, we have
\begin{align*}
(rx+sy-1)(rz+sw-1)&\equiv r(x-1)(z-1)+s(y-1)(w-1)\ \mathrm{mod}\ 2.
\end{align*}
Since $xw-yz=1$, at least one of $x-1$ and $z-1$ ({\it resp}. $y-1$ and $w-1$) is 
an even integer. Therefore,  we have
\[\mbox{the right hand side} \equiv 0 \ \mathrm{mod}\ 2.\] 
Thus, \eqref{eqn:stablility} is obtained, and we have 
$\chi_{F_f^{sp}}(g)\in \mathrm{Aut}^+(R)$.
\end{proof}
\begin{prop}\label{prop:Gamma}
\begin{equation}
\Gamma_0^+(R)=\Gamma^+(R).
\end{equation}
\end{prop}
\begin{proof}
By the well-known fact in the theory of congruence subgroups 
(see \cite{Koblitz1993}, for example)
, we have
\[[{SL}_2(\mathbb{Z}):\Gamma_0^+(R)]=\begin{cases}
1 & \mbox{if }\,\Gamma_0^+(R)={SL}_2(\mathbb{Z}) ,\\
3 & \mbox{if }\,\Gamma_0^+(R)=\Gamma_0(2), \Gamma^0(2),\\
4 & \mbox{if }\,\Gamma_0^+(R)=\Gamma_0(3), \Gamma^0(3).
\end{cases}\]
1) If $\Gamma_0^+(R)={SL}_2(\mathbb{Z})$, 
the assertion is obviously obtained by Lemma \ref{lemma:Gamma0+}.  \\
\\
2) Assume $\Gamma_{0}^+(R)$ is either $\Gamma_0(2)$ or $\Gamma^0(2)$. 
Namely, $(R,G)$ is assumed to be of type $X_l^{(1,2)},X_l^{(2,1)},
BC_l^{(2,4)},BC_l^{(2,2)}(1)$ or $BC_l^{(2,2)}(2)$.
Since the index $[{SL}_2(\mathbb{Z}):\Gamma_{0}^+(R)]=3$ is a prime number
in this case, the group $\Gamma^+(R)$ is equal to either $\Gamma_{0}^+(R)$ or 
${SL}_2(\mathbb{Z})$. Therefore, for proving the statement, 
it is enough to show that $\Gamma^+(R)\ne {SL}_2(\mathbb{Z})$.  
More precisely, since the element 
\[S:=\begin{pmatrix} 0 & -1 \\ 1 & 0\end{pmatrix}\in {SL}_2(\mathbb{Z})\] 
belongs to neither $\Gamma_0(2)$ nor $\Gamma^0(2)$, it is enough to prove that 
there is no element $\psi\in \mathrm{Aut}^+(R)$ such that $\mathrm{res}_{\mathrm{rad}(I),F}(\psi)=S$.

In the following, we will show the assertion only for of type $BC_l^{(2,1)}$. 
For the other cases, the statement is obtained by a similar method. 
Assume there exists an element $\psi\in \mathrm{Aut}^+(R)$ such that 
$\mathrm{res}_{\mathrm{rad}(I_F),F}(\psi)=S$. That is, we assume
\[\psi(b)=a\quad \mbox{and}\quad \psi(a)=-b.\]
Recall the explicit description of the sets $R_s$ of short roots, and $R_l$ of long roots 
for of type $BC_l^{(2,1)}$:
\[\begin{aligned}
R_s&=\{\pm\varepsilon_i+rb+sa\, |\,1\leq i\leq l,\,r,s\in\mathbb{Z}\},\\[3pt]
R_l&=\{\pm2\varepsilon_i+(2r+1)b+sa\, |\,1\leq i\leq l,\,r,s\in\mathbb{Z}\},
\end{aligned}\]
in Appendix \ref{sect_clasfficiation-red-quot} . 
Since $\varepsilon_i\in R_s$ and $\psi$ preserves the 
set $R_s$, the element $\psi(\varepsilon_i)\in R_s$ has the following form:
\[\psi(\varepsilon_i)=\kappa_i\varepsilon_i+r_ib+s_ia\quad 
(\kappa_i\in\{\pm 1\},\ r_i,s_i\in\mathbb{Z}).\]
Therefore, we have
\[\begin{aligned}
\psi(2\varepsilon_i+b)&=2\psi(\varepsilon_i)+\psi(b)
=2\kappa_i\varepsilon_i+2r_i b+(2s_i+1) a.
\end{aligned}\]
On the other hand , since $2\varepsilon_i+b$ is a long root of type $BC_l^{(2,1)}$, 
the image $\psi(2\varepsilon_i+b)$ should belong to $R_l$. Therefore, it has the  
following form:
\[\psi(2\varepsilon_i+b)=2\kappa'_i\varepsilon_i+(2r+1)b+sa\quad
\mbox{for some }\kappa'_i\in \{\pm 1\},\, r,s\in\mathbb{Z}.\]
Therefore, we have
\[2r_i=2r+1.\]
But, this is a contradiction. Thus, the group
$\Gamma^+(R)$ does not coincide with ${SL}_2(\mathbb{Z})$, and it should be 
equal to $\Gamma_0^+(R)=\Gamma_0(2)$, as desired. \\
\\
3) Assume $(R,G)$ is of type $G_2^{(1,3)}$.  In this case, the group 
$\Gamma_0^+(R)$ coinsides with $\Gamma_0(3)$, and 
$[{SL}_2(\mathbb{Z}),\Gamma_0^+(R)]=4$. A complete set of coset 
representatives of ${SL}_2(\mathbb{Z})/\Gamma_0^+(R)$ is given by
\[\left\{E:=\begin{pmatrix}1 & 0\\ 0 & 1\end{pmatrix},\, S,\, 
U:=\begin{pmatrix} 1 & 0 \\ -1 & 1\end{pmatrix},\, U^2=
\begin{pmatrix} 1 & 0 \\ -2 & 1\end{pmatrix}\right\}.\]
For each $g=S,U,U^2$, one can check that there is no element 
$\psi\in\mathrm{Aut}^+(R)$ such that $\mathrm{res}_{\mathrm{rad}(I_F),F}(\psi)=g$ 
by a similar method as in the previous case. Thus, we have 
\[\Gamma^+(R)=\Gamma_0(3)=\Gamma^+_0(R).\]
In the case of $(R,G)$ being of type $G_2^{(3,1)}$, the statement can be proved by 
a similar method. Thus, the proposition is obtained. 
\end{proof}
\noindent
\begin{thm}\label{thm:1}
Both the upper and lower horizontal short exact sequences in 
\eqref{comm-diagram-Autom} split.  That is, there are semidirect product 
expressions
\begin{equation}\label{eq:horizontal1}
\mathrm{Aut}^+(R)\cong\mathrm{Aut}(R,\rad (I_F))\rtimes \Gamma^+(R),
\end{equation}
\begin{equation}
\mathrm{Aut}(R)\cong \label{eq:horizontal2}
\mathrm{Aut}(R,\rad (I_F))\rtimes \big(\Gamma^+(R)\rtimes \Z/2\Z\big).
\end{equation}
\end{thm}

\begin{proof}
First, we prove the upper horizontal short exact sequence
\[1 \longrightarrow  \mathrm{Aut}(R,\mathrm{rad}(I_F)) \longrightarrow 
\mathrm{Aut}^+(R) \ \xrightarrow{\mathrm{res}_{\mathrm{rad}(I_F),F}} \ \Gamma^+(R) 
\longrightarrow  1\]
in \eqref{comm-diagram-Autom} splits. By Lemma \ref{lemma:Gamma0+} and
Proposition \ref{prop:Gamma}, it follows that the restriction 
$\chi_{F_f^{sp}}|_{\Gamma^+(R)}$ induces an injective 
homomorphism{\rm :} $\Gamma^+(R)\to \mathrm{Aut}^+(R)$ such that 
$\mathrm{res}_{\mathrm{rad}(I_F),F}\circ\chi_{F_f^{sp}}=\mathrm{id}_{\Gamma^+(R)}$,
as desired.

Next, let us prove the lower horizontal short exact sequence
\[1 \longrightarrow  \mathrm{Aut}(R,\mathrm{rad}(I_F)) \longrightarrow 
\mathrm{Aut}(R) \ \xrightarrow{\mathrm{res}_{\mathrm{rad}(I_F),F}} \ \Gamma(R) 
\longrightarrow  1\]
in \eqref{comm-diagram-Autom} splits. Recall the involutive automorphism 
$\iota_{L_f}=\chi_{L_f}\left(\begin{pmatrix}
1 & 0 \\ 0 & -1\end{pmatrix}\right)\in \mathrm{Aut}(R)$
introduced in \eqref{eq:iota_Ff} (see Remark \ref{rem:chi_Lf} (2), also), and set
\begin{equation}
\iota_{F_f^{sp}}=Ad_{\mathcal{T}}(\iota_{F_f}).
\end{equation}

\begin{lemma}
The involution $\iota_{F_f^{sp}}$ belongs to $\mathrm{Aut}(R)$. 
\end{lemma}
\vskip 1mm
\noindent
{\it Proof of the lemma}. 
If $F_f^{sp}=F_f$, there is nothing to prove. Assume $F_f^{sp}\ne F_f$. That is, 
$(R,G)$ is assumed to be of type $BC_l^{(2,4)}$, $A_1^{(1,1)\ast}$, $B_l^{(2,2)\ast}$ 
or $C_l^{(1,1)\ast}$. By the definition \eqref{eqn:T} of $\mathcal{T}$, we have
\[\iota_{F_f^{sp}}\ :\ \left\{
\begin{array}{lllllll}
\varepsilon_i^{sp} &\!\!\!\! \longmapsto &\!\!\! 
\varepsilon_i^{sp} & \mbox{for }1\leq i\leq l,\\
b &\!\!\!\! \longmapsto &\!\!\! b,\\
a &\!\!\!\! \longmapsto &\!\!\! -a.
\end{array}\right.\]
By the explicit description \eqref{eqn:R-*} of $R$, it is easy to see that 
the set $R$ is preserved by $\iota_{F_f^{sp}}$. Thus, we have the lemma. 
\hfill $\square$
\vskip 3mm
Consider the semidirect product $\Gamma^+(R)\rtimes \langle \iota_{F_f^{sp}}\rangle$
in $GL(F)$. By the above lemma, it follows that this group coincides with $\Gamma(R)$ 
and the restriction $\chi_{F_f^{sp}}|_{\Gamma(R)}$ induces an injective 
homomorphism{\rm :} $\Gamma(R)\to \mathrm{Aut}(R)$ such that 
$\mathrm{res}_{\mathrm{rad}(I_F),F}\circ\chi_{F_f^{sp}}=\mathrm{id}_{\Gamma(R)}$.
Thus, the theorem is proved.
\end{proof}
}
\subsection{The group $\mathrm{Aut}(R,\rad (I_F))$}

Let us study the group $\mathrm{Aut}(R,\rad (I_F))$. 
The canonical projection $\pi_{\rad(I_F)}:F\to F/\mathrm{rad}(I_F)$ 
induces a 
morphism
\[(\pi_{\rad(I_F)})_*:\mathrm{Aut}(R,\mathrm{rad}(I_F))\to 
\mathrm{Aut}\big(R/\rad (I_F)\big), \]
where
\begin{align*}
&\mbox{Aut}\big(R/\rad (I_F)\big)\\
&\quad :=\bigl\{\phi\in O\bigl(F/\mathrm{rad}(I_F),I_{F/\mathrm{rad}(I_F)}\bigr)\,
\bigl|\,\phi\big(R/\rad(I_F)\big)=R/\rad (I_F)\bigr\}.
\end{align*}

\begin{prop}\label{prop:2} \hspace{0.1in} 
\begin{enumerate}
\item The morphism $(\pi_{\rad(I_F)})_*:\mathrm{Aut}(R,\mathrm{rad}(I_F))\to 
\mathrm{Aut}\big(R/\rad(I_F)\big)$ is surjective. Therefore, there exists the canonical exact 
sequence
\begin{equation}\label{eqn:exact-seq-Aut}
1\to \mathrm{Ker}\bigl((\pi_{\rad(I_F)})_* \bigr)\to 
\mathrm{Aut}(R,\mathrm{rad}(I_F))
\xrightarrow{(\pi_{\rad(I_F)})_*} \mathrm{Aut}\big(R/\rad(I_F)\big)\to 1.
\end{equation}
\item The above exact sequence splits. The group 
$\mathrm{Aut}(R,\mathrm{rad}(I_F))$ is isomorphic to the
semidirect product
$\mathrm{Ker}\bigl((\pi_{\rad(I_F)})_* \bigr)\rtimes \mathrm{Aut}\big(R/\rad (I_F)\big)$.
\end{enumerate}
\end{prop}
{
\begin{proof}
(1) Let $\phi\in \mathrm{Aut}\big(R/\mathrm{rad}(I_F)\big)$. 
For $1\leq i\leq l$, there exist integers 
$m_{j,i}\in\mathbb{Z}$ such that
\[\phi(\pi_{\mathrm{rad}(I_F)}(\alpha_i))
=\sum_{j=1}^lm_{j,i}\pi_{\mathrm{rad}(I_F)}(\alpha_j).\]  
Define $\psi_{\phi}\in {GL}(F)$ by
\[\psi_{\phi}(\alpha_i)=\sum_{j=1}^lm_{j,i}\alpha_j \quad (1\leq i\leq l),\qquad
\psi_{\phi}(\delta_{\sharp}):=\delta_{\sharp} \quad (\sharp=b,a). \]
Then, it is easy to see that the map $\psi_{\phi}$ belongs to 
$\mbox{\rm Aut}(R,\mathrm{rad}(I_F))$ and 
$(\pi_{\mathrm{rad}(I_F)})_*(\psi_{\phi})=\phi$. 
\vskip 2mm
\noindent
(2) It is immediate to see that the correspondence 
$\phi\mapsto \psi_{\phi}$ defines a group homomorphism 
$\zeta:\mathrm{Aut}\big(R/\mathrm{rad}(I_F)\big)\to 
\mathrm{Aut}(R,\mathrm{rad}(I_F))$
such that the condition
$(\pi_{\mathrm{rad}(I_F)})_*\circ \zeta=\mathrm{id}$
is satisfied. That is, it gives a section for the exact sequence \eqref{eqn:exact-seq-Aut},
and the assertion (2) is verified. 
\end{proof}
}

Let $\Gamma\big(R/\rad (I_F)\big)$ be the Dynkin diagram of the finite root system 
$R/\rad(I_F)$ and denote the group of diagram automorphisms of 
$\Gamma\big(R/\rad(I_F)\big)$  by 
$\mathrm{Aut}\big(\Gamma\big(R/\rad (I_F)\big)\big)$. 
It is known that the group $\mathrm{Aut}\big(R/\rad (I_F)\big)$ of automorphisms of 
the finite root system $R/\rad (I_F)$ is isomorphic to
the semidirect product $W\big(R/\rad(I_F)\big)\rtimes 
\mathrm{Aut}\big(\Gamma\big(R/\rad (I_F)\big)\big)$. \\

Combining these results, we have the following corollary:
\begin{cor}\label{cor:normal}
\begin{equation*}\label{eqn:normal} 
\begin{split}
&\mathrm{Aut}^+(R)\\
\cong 
&\bigg(\mathrm{Ker}\bigl((\pi_{\rad(I_F)})_* \bigr)\rtimes 
\Big(W\big(R/\rad(I_F)\big)\rtimes \mathrm{Aut}\big(\Gamma\big(R/\rad (I_F)\big)\big)\Big)
\bigg)\rtimes \Gamma^+(R).
\end{split}
\end{equation*}
\end{cor}
{
Note that the structure of the finite Weyl group $W\big(R/\rad (I_F)\big)$ and the group 
$\mathrm{Aut}\big(\Gamma\big(R/\rad (I_F)\big)\big)$ of diagram automorphisms are 
known. In addition, the explicit form of the elliptic modular group $\Gamma^+(R)$ of 
$R$ is already studied in the previous subsection. Therefore, 
the study of $\mathrm{Aut}^+(R)$ reduces to one of the kernel 
$\mathrm{Ker}\bigl((\pi_{\rad(I_F)})_* \bigr)$.  }\\

The dual weight lattice $P\big((R/\rad (I_F))^{\vee}\big)$
of the finite quotient root system $R/\rad(I_F)$} is defined by
\begin{align*}
& P\big((R/\rad (I_F))^{\vee}\big) \\
=
&\big\{\mu\in F/\mathrm{rad}(I_F)\,\big|\, I_{F/\rad(I_F)}(\mu,\gamma)\in \Z 
\mbox{ for every }\gamma\in R/\rad(I_F)\big\}.
\end{align*}

Let $\Pi=\{\alpha_i\}_{0\leq i\leq l}$ be a simple system of $(R,G)$ in the sense of 
Definition \ref{defn:basis}. Set
\[\Pi_f=\Pi\setminus \{\alpha_0\}.\]

\begin{rem}\label{rem:simple-system-Pif}
{\rm The image $\pi_{\rad(I_F)}(\Pi_f)=
\big\{\pi_{\rad(I_F)}(\alpha_i)\big\}_{1\leq i\leq l}$ is a simple system of the finite 
quotient root system $R/\rad(I_F)$, even when $R/\rad(I_F)$ is non-reduced. 
}\end{rem}

Since the restriction $I_{F_f}$ $:=I\big|_{F_f\times F_f}$  is non-degenerate, one can take the
dual basis $(\varpi_1^\vee,\ldots,\varpi_l^\vee)$ of $(\alpha_1,\ldots,\alpha_l)$
with respect to $I_{F_f}$. That is, $\{\varpi_i^\vee\}_{1\leq i\leq l}$  is a subset of $F_f$
so that 
\[I_{F_f}(\varpi_i^\vee,\alpha_j)=\delta_{i,j}\quad \text{for every }1\leq i,j\leq l.\]
Then  the image $\big(\pi_{\rad(I_{F_f})}(\varpi_1^\vee),\ldots,
\pi_{\rad(I_{F_f})}(\varpi_l^\vee)\big)$ is a $\Z$-basis of the lattice 
$P\big((R/\rad (I_F))^{\vee}\big)$:
\[P\big((R/\rad (I_F))^{\vee}\big)=\bigoplus_{i=1}^l \Z \pi_{\rad(I_{F_f})}(\varpi_i^\vee).\]

{
Set
\[P\big((R/\mathrm{rad}(I_F))^{\vee}\big)_{\sharp}
={\sharp}\otimes P\big((R/\mathrm{rad}(I_F))^{\vee}\big)\quad\mbox{for }
\sharp=b,a.\] 
The following lemma describes the structure of the group
$\mbox{\rm Ker}\bigl((\pi_{\mathrm{rad}(I_F)})_*\bigr)$. 
\begin{lemma}\label{lemma:Ker1} 
We have
\begin{equation}\label{eqn:Ker1}
\begin{aligned}
&\mbox{\rm Ker}\bigl((\pi_{\mathrm{rad}(I_F)})_*\bigr)\\
&=
\bigl\{\psi\in E_F\bigl(P\big((R/\mathrm{rad}(I_F))^{\vee}\big)_b\oplus 
P\big((R/\mathrm{rad}(I_F))^{\vee}\big)_a\bigr)\,\bigl|\,
\psi(R)=R\bigr\}.
\end{aligned}
\end{equation}
\end{lemma}
\begin{proof}
Assume that $\psi$ is an element of the right hand side of $(\ref{eqn:Ker1})$.
Then, by direct computation, one can easily check that
$\psi\in \mbox{\rm Ker}\bigl((\pi_{\mathrm{rad}(I_F)})_*\bigr)$.

Conversely, let $\psi\in \mbox{\rm Ker}\bigl((\pi_{\mathrm{rad}(I_F)})_*\bigr)$. Then, 
by the definition, we have
$$\psi|_{\mathrm{rad}(I_F)}=\mbox{id},\quad \psi(\alpha_i)\equiv 
\alpha_i \mbox{ mod rad}(I_F)\quad \mbox{for every }1\leq i\leq l.$$
Since $\psi$ preserves the root lattice 
$Q(R)=\bigl(\oplus_{i=1}^l\mathbb{Z}\alpha_i\bigr)\oplus\mathbb{Z}b
\oplus\mathbb{Z}a$, there exist $\mu_b,\mu_a\in 
\mbox{Hom}_{\mathbb{Z}}\big(Q(R/\mathrm{rad}(I_F)),\mathbb{Z}\big)$ such that 
\[\psi(\beta)=\beta-\mu_b(\pi_{\mathrm{rad}(I_F)}(\beta))b
-\mu_a(\pi_{\mathrm{rad}(I_F)}(\beta))a\quad\mbox{for every }
\beta\in Q(R).\]
As $I_{\mathrm{rad}(I_F)}$ is non-degenerate, there exist 
$u_b,u_a\in P\big((R/\mathrm{rad}(I_F))^\vee\big)$ such that
\[\mu_{\sharp}(\pi_{\mathrm{rad}(I_F)}(\beta))
=I_{\mathrm{rad}(I_F)}(\pi_{\mathrm{rad}(I_F)}(\beta),u_\sharp)\quad 
\mbox{for }\sharp=b,a.\]
Therefore, we have
\[\psi=E_{F}(b\otimes u_b+a\otimes u_a)\in
E_F\big(P\big((R/\mathrm{rad}(I_F))^\vee\big)_b\oplus 
P\big((R/\mathrm{rad}(I_F))^\vee\big)_a\big).\]
Thus, we have the lemma.
\end{proof}
For $\psi=E_F(b\otimes u_b+a\otimes u_a)$, set 
\begin{equation*}\label{eqn:psi_sharp}
\psi_{\sharp}:=E_F({\sharp}\otimes u_{\sharp})
\in E_F\bigr(P\big((R/\mathrm{rad}(I_F))^\vee\big)_{\sharp}\bigr)\quad
\mbox{for }\sharp=b,a.
\end{equation*}
Then, by the definition of the Eichler-Siegel map $E_F$, it is easy to see 
\begin{equation}\label{eqn:comm-psi}
\psi=\psi_a\psi_b=\psi_b\psi_a.
\end{equation}
In other words, the group $E_F\bigl(P\big((R/\mathrm{rad}(I_F))^\vee\big)_b\oplus 
P\big((R/\mathrm{rad}(I_F))^\vee\big)_a\bigr)$
decomposes into the (internal) direct sum of 
$E_F\bigl(P\big((R/\mathrm{rad}(I_F))^\vee\big)_\sharp\bigr)$ $(\sharp=b,a)$:
\begin{equation}
\begin{aligned}
&E_F\bigl(P\big((R/\mathrm{rad}(I_F))^\vee\big)_b\oplus 
P\big((R/\mathrm{rad}(I_F))^\vee\big)_a\bigr)\\
&=E_F\bigl(P\big((R/\mathrm{rad}(I_F))^\vee\big)_b\bigr)\oplus 
E_F\bigl(P\big((R/\mathrm{rad}(I_F))^\vee\big)_a\bigr).
\end{aligned}
\end{equation}

\begin{lemma}\label{lemma:ker2}
For $\psi=E_F(b\otimes u_b+a\otimes u_a)$, the following conditions are 
equivalent{\rm :}
\begin{itemize}
\item[(a)] It belongs to $\mbox{\rm Ker}\bigl((\pi_{\mathrm{rad}(I_F)})_*\bigr).$
\vskip 1mm
\item[(b)] For every $\sharp=b,a$, $\psi_\sharp(R)=R.$
\end{itemize}
\end{lemma}

\begin{proof}
Note that, by Lemma \ref{lemma:Ker1},  the condition (a) is equivalent to $\psi(R)=R$.
\vskip 1mm
\noindent
(b) $\Rightarrow$ (a):\, This part follows obviously by (\ref{eqn:comm-psi}).
\vskip 1mm
\noindent
(a) $\Rightarrow$ (b):\,  Using an explicit description for the set $R$ in Appendix 
\ref{sect_clasfficiation-red-quot},
one can show the statement by case by case analysis. In the following, we give a 
explicit proof only for the case of $B_l^{(2,2)\ast}$. For other cases, the statement
is  proved by a similar way to this case with an easier argument. 

Assume $(R,G)$ is of type $B_l^{(2,2)\ast}$. Recall the explicit description
of the set $R=R_s\sqcup R_l$ in Appendix \ref{sect_clasfficiation-red-quot}:
\begin{equation}\label{eqn:R-B_l^{(2,2)star}}
\begin{array}{ll}
R_s=\{\pm \varepsilon_i+rb+sa\,|\, 1\leq i\leq l,\, r,s\in\mathbb{Z}
\mbox{ such that }rs\equiv 0\,(\mathrm{mod}\,2)\},\\[3pt]
R_l=\{\pm\varepsilon_i\pm\varepsilon_j+2rb+2sa\,|\,1\leq i<j\leq l,\, r,s\in\mathbb{Z}\}.
\end{array}.
\end{equation}
Take $\psi=E_F(b\otimes u_b+a\otimes u_a)\in 
\mbox{\rm Ker}\bigl((\pi_{\mathrm{rad}(I_F)})_*\bigr)$. Since $\psi(R_s)=R_s$, 
an element
\begin{align*}
\psi(\pm \varepsilon_i+rb+sa)&=
\pm\varepsilon_i+(r\mp I_F(u_b,\varepsilon_i))b+(s\mp I_F(u_a,\varepsilon_i))a
\end{align*} 
should belong to $R_s$ for every pair $r,s$ such that
$rs\equiv 0\, (\mathrm{mod}\, 2)$. That is,
\begin{align*}
(r\mp I_F(u_b,\varepsilon_i))(s\mp I_F(u_a,\varepsilon_i))\equiv 0\ (\mathrm{mod}\, 2)
.
\end{align*} 
Therefore, we have
\begin{equation}\label{eqn:psib-psia}
I_F(u_b,\varepsilon_i)\in 2\mathbb{Z}\quad \mbox{and}\quad I_F(u_a,\varepsilon_i)
\in 2\mathbb{Z}.
\end{equation}
By the condition \eqref{eqn:psib-psia}, it follows that
\begin{align*}
\psi_b(\pm \varepsilon_i+rb+sa)&=
\pm\varepsilon_i+(r\mp I_F(u_b,\varepsilon_i))b+sa
\end{align*} 
and 
\[(r\mp I_F(u_b,\varepsilon_i))s\equiv 0\ (\mathrm{mod}\, 2).\]
Therefore, $\psi_b$ preserves the set $R_s$ of short roots. 
 
On the other hand, for every $r,s\in\mathbb{Z}$, an element
\begin{align*}
\psi_b(\pm\varepsilon_i\pm\varepsilon_j+2rb+2sa)&=
\pm \varepsilon_i\pm\varepsilon_j
+(\mp I_F(u_b,\varepsilon_i)\mp I_F(u_b,\varepsilon_j)+2r)b+2sa
\end{align*}
belongs to $R_l$ under the condition \eqref{eqn:psib-psia}. That is,  
$\psi_b$ preserves the set $R_l$ of long roots. Combining these two results, 
we have $\psi_b(R)=R$.  

By a similar way, we have $\psi_a(R)=R$. Thus, the condition (b) is satisfied, and 
the statement is proved in this case. 
\end{proof}

For $\sharp=b,a$, define
\[\mbox{\rm Ker}\bigl((\pi_{\mathrm{rad}(I_F)})_*\bigr)_\sharp
=\mbox{\rm Ker}\bigl((\pi_{\mathrm{rad}(I_F)})_*\bigr)\cap
E_F\bigl(P\big((R/\mathrm{rad}(I_F))^\vee\big)_{\sharp}\bigr).\]
The following corollary is immediately verified. 

\begin{cor}\label{cor:ker3}
{\rm (1)} For each $\sharp=b,a$, we have
\begin{equation*}\label{eqn:(B.2.4)}
\mbox{\rm Ker}\bigl((\pi_{\mathrm{rad}(I_F)})_*\bigr)_\sharp=
\bigl\{\psi_\sharp\in E_F\bigl(P\big((R/\mathrm{rad}(I_F))^\vee\big)_{\sharp}\bigr)\,\bigl|\,
\psi_\sharp(R)=R\bigr\}.
\end{equation*}
{\rm (2)} There is an induced direct sum decomposition of 
$\mbox{\rm Ker}\bigl((\pi_{\mathrm{rad}(I_F)})_*\bigr)${\rm :}
$$\mbox{\rm Ker}\bigl((\pi_{\mathrm{rad}(I_F)})_*\bigr)
=\mbox{\rm Ker}\bigl((\pi_{\mathrm{rad}(I_F)})_*\bigr)_b\bigoplus
\mbox{\rm Ker}\bigl((\pi_{\mathrm{rad}(I_F)})_*\bigr)_a.$$
\end{cor}
\vskip 3mm

Since $E_F\big(P\big((R/\mathrm{rad}(I_F))^\vee\big)_{\sharp}\big)
=E_F\big(\sharp\otimes P\big((R/\mathrm{rad}(I_F))^\vee\big)\big)$ is a free abelian 
group and $\mbox{\rm Ker}\bigl((\pi_{\mathrm{rad}(I_F)})_*\bigr)_\sharp$ is a 
subgroup of it for each $\sharp=b,a$, there exist a lattice 
$M_{F/\mathrm{rad}(I_F),\sharp}\subset P\big((R/\mathrm{rad}(I_F))^\vee\big)$ 
such that
\begin{equation*}\label{eqn:def-M}
\mbox{\rm Ker}\bigl((\pi_{\mathrm{rad}(I_F)})_*\bigr)_\sharp
=E_F(\delta_{\sharp}\otimes
M_{F/\mathrm{rad}(I_F),\sharp})\quad\mbox{for }\sharp=b,a.
\end{equation*}
Thus, we have the following proposition:

\begin{prop}\label{prop:ker}
Let $(R,G)$ be a mERS in the K. Saito's list.  There
exist sublattices $M_{F/\rad(I_F),a},M_{F/\rad(I_F),b}\subset 
P\big((R/\rad(I_F))^{\vee}\big)$ 
such that
\[\mathrm{Ker}\bigl((\pi_{\mathrm{rad}(I_F)})_* \bigr)=
E_F\big(a\otimes M_{F/\rad(I_F),a}\big)
\bigoplus E_F\big(b\otimes M_{F/\rad(I_F),b}\bigr).\]
\end{prop}

Let us determine the lattice $M_{F/\mathrm{rad}(I_F),\sharp}\ 
(\sharp=b,a)$ in explicit way. Since the results depend on specific calculations for each 
case, we only describe the strategy here. 

Recall the dual basis 
$(\varpi_1^{\vee},\ldots,\varpi_l^{\vee})$ of $(\alpha_1,\ldots,\alpha_l)$ with respect
to $I_F$ which was introduced after Remark \ref{rem:simple-system-Pif}.
Under the identification 
$\pi_{\mathrm{rad}(I_F)}:F_f\overset{\sim}{\to} F/\mathrm{rad}(I_F)$, 
they are regarded as the fundamental coweights of $R/\mathrm{rad}(I_F)$, and 
the coweight lattice $P\big(R/\mathrm{rad}(I_F))^{\vee}\big)$ is identified with 
$\oplus_{i=1}^l\mathbb{Z} \varpi_i^{\vee}$. 
Recall $\psi_{\sharp}=E_F(\delta_{\sharp}\otimes u_{\sharp})$ with
$u_{\sharp}\in P(R_f^{\vee})$ for $\sharp=b,a$ (see (\ref{eqn:psi_sharp})). By  using
an $\mathbb{Z}$-basis $\varpi_i^{\vee}\ (1\leq i\leq l)$, the condition (b) of Lemma
\ref{lemma:ker2} is described in an explicit way. }\\

We give an explicit description of the lattice 
$M_{F/\rad(I_F),\sharp}\ (\sharp=a,b)$ for each $(R,G)$ in the following table. 
\newpage
\begin{center}
\begin{table}[h]
\caption{An explicit description of $M_{F/\rad(I_F),\sharp}$}\label{table1-auto}
\renewcommand{\arraystretch}{1.4}
\begin{tabular}{|c|c|c|}
\noalign{\hrule height0.8pt}
$(R,G)$&  $M_{F/\rad(I_F),a}$ & $M_{F/\rad(I_F),b}$\\
\noalign{\hrule height0.8pt}
\hline
$X_l^{(1,1)}$ & $P\big((R/\rad(I_F))^{\vee}\big)$ & $P\big((R/\rad(I_F))^{\vee}\big)$ \\
\hline
$X_l^{(1,t)}$ & $\widetilde{P}\big((R/\rad(I_F))^{\vee}\big)$ & 
$P\big((R/\rad(I_F))^{\vee}\big)$ \\
\hline
$X_l^{(t,1)}$ & $P\big((R/\rad(I_F))^{\vee}\big)$
& $\widetilde{P}\big((R/\rad(I_F))^{\vee}\big)$\\
\hline
$X_l^{(t,t)}$ & $\widetilde{P}\big((R/\rad(I_F))^{\vee}\big)$ &
$\widetilde{P}\big((R/\rad(I_F))^{\vee}\big)$ \\
\hline
\hline
$BC_l^{(2,1)}$ & 
$P\big((R/\rad(I_F))^{\vee}\big)$ & $P\big((R/\rad(I_F))^{\vee}\big)$ \\
\hline
$BC_l^{(2,4)}$ & 
$2P\big((R/\rad(I_F))^{\vee}\big)$ & $P\big((R/\rad(I_F))^{\vee}\big)$ \\
\hline
$BC_l^{(2,2)}(1)$  & 
$P\big(((R/\rad(I_F))^{\dc})^{\vee}\big)$ & 
$P\big((R/\rad(I_F))^{\vee}\big)$ \\
\hline
$
BC_l^{(2,2)}(2)$ & 
$\widetilde{P}\big((R/\rad(I_F))^{\vee}\big)$ & $P\big((R/\rad(I_F))^{\vee}\big)$\\
\hline
\hline
$A_1^{(1,1)\ast}$  & $2P\big((R/\rad(I_F))^{\vee}\big)$ & $2P\big((R/\rad(I_F))^{\vee}\big)$ \\
\hline
$B_l^{(2,2)\ast}\, (l\geq 2)$  & $2P\big((R/\rad(I_F))^{\vee}\big)$ & 
$2P\big((R/\rad(I_F))^{\vee}\big)$ \\
\hline
$C_l^{(1,1)\ast}\, (l\geq 2)$  & 
$\widetilde{P}\big((R/\rad(I_F))^{\vee}\big)$ & 
$\widetilde{P}\big((R/\rad(I_F))^{\vee}\big)$ \\
\noalign{\hrule height0.8pt}
\end{tabular}
\end{table}
\end{center}
Here, $\widetilde{P}\big((R/\rad(I_F))^\vee\big)$ is the sub-lattice of 
$P\big((R/\rad(I_F))^\vee\big)$ defined by
\[ \widetilde{P}\big((R/\rad(I_F))^\vee\big)=
\bigoplus_{i=1}^l \frac{I_{F/\rad(I_F)}\big(\pi_{\rad(I_F)}(\alpha_i),\pi_{\rad(I_F)}(\alpha_i)\big)}
{I_{F/\rad(I_F)}(\gamma_s,\gamma_s)}\Z 
\pi_{\rad(I_F)}(\varpi_i^\vee), \]
where $\gamma_s \in R/\rad (I_F)$ is a short root. 

\subsection{The outer automorphisms}

It is easy to see that the Weyl group $W(R)$ is a normal subgroup of 
$\mathrm{Aut}^+(R)$. Set 
\[\mathrm{Out}^+(R)=\mathrm{Aut}^+(R)/W(R). \]
Call an element of $\mathrm{Out}^+(R)$ (orientation preserving) {\it outer
automorphism} of $R$. Since an element $w\in W(R)$ acts on the radical
$\mathrm{rad}(I_F)$ trivially, $W(R)$ is a (normal) subgroup of  
$\mathrm{Aut}(R,\mathrm{rad}(I_F))$. Dividing the upper horizontal exact sequence 
in \eqref{comm-diagram-Autom} by $W(R)$, we have a new exact sequence
\begin{equation}\label{eq:outer}
1  \longrightarrow  \Omega(R)\longrightarrow \mathrm{Out}^+(R) \ 
\longrightarrow \Gamma^+(R) \longrightarrow  1 ,
\end{equation}
where
\[\Omega(R):=\mathrm{Aut}(R,\mathrm{rad}({I_F}))/W(R),\]
and call it the {\bf $\Omega$-part} of $\mathrm{Out}^+(R)$.  
Furthermore, as an immediate 
consequence of Theorem \ref{thm:1}, we get the following corollary.
\begin{cor}
The exact sequence \eqref{eq:outer} also splits. That is, the group 
$\mathrm{Out}^+(R)$ is isomorphic to the
semidirect product $\Omega(R)\rtimes\Gamma^+(R)$. 
\end{cor}

{
In the rest of this subsection, we study the group $\Omega(R)$. 
By Proposition \ref{prop:SaitoSaito1} and Proposition \ref{prop:ker}, we have
\[ \mathrm{Ker}\bigl((\pi_{\mathrm{rad}(I_F)})_*\bigr)\cap W(R)
=E_F(T_{\rad (I_F)}(R)).\]
Therefore, by exact sequences \eqref{eq:exact-Weyl2} and \eqref{eqn:exact-seq-Aut},
we have the following commutative diagram in which all vertical and horizontal 
sequences are exact:
{\small
\begin{equation}\label{eqn:comm-diagram}
\begin{array}{l}\xymatrix @C=18pt{
&1 \ar[d] & 1 \ar[d] & 1 \ar[d]\\
1 \ar[r] & E_F(T_{\rad (I_F)}(R)) \ar[r] \ar[d] & W(R) \ar[r]^-{(\pi_{\mathrm{rad}(I_F)})_*} 
\ar[d] & W\big(R/\rad(I_F)\big) \ar[r] \ar[d]& 1\\
1\ar[r] & \mathrm{Ker}\bigl((\pi_{\mathrm{rad}(I_F)})_*\bigr) \ar[r] \ar[d] 
& \mbox{\rm Aut}(R,\mbox{\rm rad}(I_F)) \ar[r]^-{(\pi_{\mathrm{rad}(I_F)})_*} \ar[d] 
& \mathrm{Aut}\big(R/\rad(I_F)\big) \ar[r] \ar[d]& 1\\
1 \ar[r] & \dfrac{\mathrm{Ker}\bigl((\pi_{\mathrm{rad}(I_F)})_*\bigr)}
{E_F(T_{\rad (I_F)}(R))} 
\ar[r] \ar[d] & \Omega(R)
 \ar[r]^-{\overline{(\pi_{\mathrm{rad}(I_F)})_*}} \ar[d] 
& \mathrm{Aut}\big(\Gamma\big(R/\rad(I_F)\big)\big)\ar[r] 
 \ar[d]& 1\\
&1&1& 1&
}\end{array},
\end{equation}}
where $\overline{(\pi_{\mathrm{rad}(I_F)})_*}:\Omega(R)\longrightarrow
\mathrm{Aut}\big(\Gamma(R/\rad(I_F))\big)$ is the induced group homomorphism.  
As a corollary of Proposition \ref{prop:2}, we have the following result.

\begin{cor}\label{cor:H-red}
The lower horizontal exact sequence
\begin{equation*}\label{eqn:lower-splits}
1 \longrightarrow \dfrac{\mathrm{Ker}\bigl((\pi_{\mathrm{rad}(I_F)})_*\bigr)}
{E_F(T_{\rad (I_F)}(R))} 
\longrightarrow \Omega(R)\xrightarrow{\overline{(\pi_{\mathrm{rad}(I_F)})}_*}
\mathrm{Aut}\big(\Gamma\big(R/\rad(I_F)\big)\big)\longrightarrow 1
\end{equation*}
of {\rm (\ref{eqn:comm-diagram})} splits. In other words, the group $\Omega(R)$ is
isomorphic to a semidirect product 
$\dfrac{\mathrm{Ker}\bigl((\pi_{\mathrm{rad}(I_F)})_*\bigr)}{E_F(T_{\rad (I_F)}(R))}
\rtimes\mathrm{Aut}\big(\Gamma\big(R/\rad(I_F)\big)\big)$. 
\end{cor}

\begin{proof}
Recall the section $\zeta:\mathrm{Aut}(R/\mathrm{rad}(I_F))\longrightarrow 
\mathrm{Aut}(R,\mathrm{rad}(I_F))$
of the exact sequence (\ref{eqn:exact-seq-Aut}) which is constructed in the proof of
Proposition \ref{prop:2} (2). Consider a composition
\[\mathrm{Aut}(R/\mathrm{rad}(I_F))\overset{\zeta}{\longrightarrow} 
\mathrm{Aut}(R,\mathrm{rad}(I_F)) \longrightarrow 
\dfrac{\mathrm{Aut}(R,\mathrm{rad}(I_F))}{W(R)}\ (=\Omega(R)).\]
Since $\zeta\big(W(R/\mathrm{rad}(I_F))\big)\subset W(R)$, a group homomorphism
\[\overline{\zeta}:\dfrac{\mathrm{Aut}(R/\mathrm{rad}(I_F))}{W(R/\mathrm{rad}(I_F))}
\longrightarrow \Omega(R)\]
is naturally induced. Furthermore, by the definition, the condition
\[\overline{(\pi_{\mathrm{rad}(I_F)})}_*\circ \overline{\zeta}=\mathrm{id}\]
is satisfied. That is, the lemma is obtained.   
\end{proof}

Combining the above results, we have
\begin{equation*}\label{eq:Omega2}
\begin{split}
\Omega(R) 
\cong 
&\dfrac{\mathrm{Ker}\bigl((\pi_{\mathrm{rad}(I_F)})_*\bigr)}{E_F(T_{\rad (I_F)}(R))}\rtimes 
\mathrm{Aut}\big(\Gamma\big(R/\rad(I_F)\big)\big) \\
\cong
&
\left(\dfrac{M_{F/\rad(I_F),a}}{Q_{F/\rad(I_F),a}}\bigoplus \dfrac{M_{F/\rad(I_F),b}}{Q_{F/\rad(I_F),b}}\right)
\rtimes\mathrm{Aut}\big(\Gamma\big(R/\rad(I_F)\big)\big).
\end{split}
\end{equation*}
}
In the following table, we give an explicit description of the group $\Omega(R)$.

\begin{center}

\begin{table}[h]
\caption{{An explicit description of} $\Omega(R)$}\label{table3-auto}
\scalebox{0.83}{
\renewcommand{\arraystretch}{1.4}
\begin{tabular}{|c|c|c|c|c|}
\noalign{\hrule height0.8pt}
$(R,G)$& $\frac{M_{F/\rad(I_F),a}}{Q_{F/\rad(I_F),a}}$ & $\frac{M_{F/\rad(I_F),b}}{Q_{F/\rad(I_F),b}}$ & 
$\mathrm{Aut}\big(\Gamma\big(R/\rad(I_F)\big)\big)$ & $\Omega(R)$ 
\rule[-3mm]{0mm}{8mm}
\\
\noalign{\hrule height0.8pt}
$A_1^{(1,1)}$ & $\Z/2\Z$ & $\Z/2\Z$ & $\mathrm{id}$ & $(\Z/2\Z)^2$\\
\hline
$A_l^{(1,1)}\,(l\geq 2)$ & $\Z/(l+1)\Z$ & $\Z/(l+1)\Z$ & $\Z/ 2\Z$ &
$(\Z/(l+1)\Z)^{2}\rtimes (\Z/2\Z)$\\
\hline
$D_4^{(1,1)}$ & $(\Z/ 2\Z)^2$ & $(\Z/ 2\Z)^2$ & $\mathfrak{S}_3$ &
$(\Z/2\Z)^4\rtimes \mathfrak{S}_3$\\
\hline
$D_l^{(1,1)}\, \bigg(\!\! \begin{array}{lll}l\geq 5,\\[-5pt] l:\mathrm{odd}\end{array}
\!\! \bigg)$ 
& $\Z/4\Z$  & $\Z/4\Z$ & $\Z/ 2\Z$ &
$(\Z/4\Z)^{2}\rtimes (\Z/2\Z)$\\
\hline 
$D_l^{(1,1)}\, \bigg(\!\! \begin{array}{lll}l\geq 6,\\[-5pt] l:\mathrm{even}\end{array}
\!\! \bigg)$
& $(\Z/2\Z)^2$ & $(\Z/2\Z)^2$ & 
$\Z/ 2\Z$ & $(\Z/2\Z)^4\rtimes (\Z/2\Z)$\\
\hline 
$E_6^{(1,1)}$ & $\Z/3\Z$ & $\Z/3\Z$ & $\Z/2\Z$ & $(\Z/3\Z)^2\rtimes (\Z/2\Z)$\\
\hline
$E_7^{(1,1)}$ & $\Z/2\Z$ & $\Z/2\Z$ & $\mathrm{id}$ & $(\Z/2\Z)^2$\\
\hline
$E_8^{(1,1)}$ & $\mathrm{id}$ & $\mathrm{id}$ & $\mathrm{id}$ & $\mathrm{id}$ \\
\hline 
\hline
$B_l^{(t_1,t_2)}$ & $\Z/2\Z$ & $\Z/2\Z$ & $\mathrm{id}$ &$(\Z/2\Z)^2$\\
\hline
$C_l^{(t_1,t_2)}$ & $\Z/2\Z$ & $\Z/2\Z$ & $\mathrm{id}$ & $(\Z/2\Z)^2$\\
\hline
$F_4^{(t_1,t_2)}$ & $\mathrm{id}$ & $\mathrm{id}$ & $\mathrm{id}$ & $\mathrm{id}$ \\
\hline
$G_2^{(t_1,t_2)}$ & $\mathrm{id}$ & $\mathrm{id}$ & $\mathrm{id}$ & $\mathrm{id}$ \\
\hline
\hline
$BC_l^{(2,1)}$ & $\mathrm{id}$ & $\mathrm{id}$ & $\mathrm{id}$ & $\mathrm{id}$\\
\hline
$BC_l^{(2,4)}$ & $\mathrm{id}$ & $\mathrm{id}$ & $\mathrm{id}$& $\mathrm{id}$\\
\hline
$BC_l^{(2,2)}(1)$  & $\Z/2\Z$ & $\mathrm{id}$  & $\mathrm{id}$ & $\Z/2\Z$\\
\hline
$BC_l^{(2,2)}(2)$ & $\Z/2\Z$ & $\mathrm{id}$ & $\mathrm{id}$& $\Z/2\Z$ \\
\hline
\hline
$A_1^{(1,1)\ast}$  & $\mathrm{id}$ & $\mathrm{id}$ & $\mathrm{id}$ & 
$\mathrm{id}$\\
\hline
$B_l^{(2,2)\ast}$  & $\mathrm{id}$ & $\mathrm{id}$ & $\mathrm{id}$ & 
$\mathrm{id}$\\
\hline
$C_l^{(1,1)\ast}$  & $\mathrm{id}$ & $\mathrm{id}$ & $\mathrm{id}$ &
$\mathrm{id}$\\
\noalign{\hrule height0.8pt}
\end{tabular}
}
\end{table}
\end{center}

For simplicity, in the rest of this article, 
we extend the range of $l$ in the definition of the mERS of type $B_l^{(1,1)}, \, B_l^{(1,2)}, \, C_l^{(2,1)}, \, C_l^{(2,2)}$ and $D_l^{(1,1)}$ to $l\geq 2$. It turns out that the structure of the automorphism groups described in the previous subsubsection 
is still valid; for type $D_l^{(1,1)}$, one should read ($l\geq 3$, $l$: odd) and ($l\geq 2$, $l$: even $\neq 4$), in their appropriate places.


\part{New marked elliptic root systems}\label{chapter:new-mERS}

In this chapter, we construct all possible marked elliptic root systems $(R,G)$ with non-reduced affine quotient $R/G$. 
Let $R_s, R_m$ and $R_l$ be the subsets of the short (resp. middle length and long) roots in $R$. Since the subsets
$R_s \cup R_m$ and $R_m \cup R_l$ are reduced elliptic root systems, this can be achieved
by ``gluing''  two reduced elliptic root systems. \\

In Section \ref{sect:MERS-non-red-quot}, we give all possible marked elliptic root systems explicitly and state the weak classification (Theorems \ref{thm_classification-reduced} and \ref{thm_classification-non-reduced}); we do \textit{not} prove that two marked elliptic root systems in the above list are non-isomorphic as marked root systems, in this chapter.  Section \ref{sect_mers-II} is devoted to the proof of the last two theorems. Here, after explaining what we mean by ``gluing root systems'', we prove these two theorems by performing gluing all possible pairs of elliptic root systems of type $B_l$ and of type $C_l$. In Section \ref{sect:more-mERS}, we 
study some properties of marked elliptic root systems $(R,G)$ with non-reduced affine quotient $R/G$. In $\S$ \ref{sect:isom-root-sys}, we classify the isomorphism classes of the marked elliptic root systems listed above, as root systems (Theorems \ref{thm:isom-red-mERS-nred_root-sys} and \ref{thm:isom-nred-mERS-nred_root-sys}). In $\S$ \ref{sect_two-affine-quot}, we study affine quotients of a given elliptic root system. 


\section{MERSs with non-reduced quotient}\label{sect:MERS-non-red-quot}
Let $(F,I)$ be an $\R$-vector space with a quadratic form whose signature is $(l,2,0)$.
Here we introduce several mERSs $(R,G)$ with non-reduced affine quotient $R/G$. It should be noticed that the finite quotient of any such root system is of type $BC_l$. In this section, we fix the marking $G=\R a$. 
\subsection{Reduced mERSs}\label{sect:intro_red-non-red-quot}
First consider the reduced $R$'s. Since there are only 6 such cases, we call each type of the mERS by making reference to $BC_l$. 
\subsubsection{Types $BC_l^{(1,2)}, BC_l^{(4,2)}, BC_l^{(2,2)\sigma}(1), BC_l^{(2,2)\sigma}(2)$}
\label{sect:intro_red-classic}

Each root system $R$ of the above type is defined by
\[ R=(R(X_l)_s+\Z a) \cup (R(X_l)_m+\Z a) \cup (R(X_l)_l+(1+2\Z)a), \]
where $R(X_l)$, for each $R$, is the affine root system given by the following table:
\begin{center}
\begin{tabular}{|c||c|c|c|c|} \hline
Type of $R$ & $BC^{(1,2)}_l$ & $BC_l^{(4,2)}$ & $BC_l^{(2,2)\sigma}(1)$ & $BC_l^{(2,2)\sigma}(2)$ 
 \\ \hline
$X_l$ & $BCC_l$ & $C^\vee BC_l $ & $BB_l^\vee$ & $C^\vee C_l$  \\ \hline
\end{tabular}
\end{center}
Recall that these $4$ infinite series of non-reduced affine root systems are defined as follows (cf. \cite{Macdonald1972}):
\begin{align*}
R(BCC_l)=
&R(BC_l)+\Z b, \\
R(BB_l^\vee)=
&(R(BC_l)_s+\Z b) \cup (R(BC_l)_m+\Z b) \cup (R(BC_l)_l+2\Z b), \\
R(C^\vee C_l)=
&(R(BC_l)_s+\Z b) \cup (R(BC_l)_m+2\Z b) \cup (R(BC_l)_l+2\Z b), \\
R(C^\vee BC_l)=
&(R(BC_l)_s+\Z b) \cup (R(BC_l)_m+2\Z b) \cup (R(BC_l)_l+2\Z b),
\end{align*}
and the root systems of type $BC_l$ are defined by 
\[ R(BC_l)=\{ \, \pm \vep_i\, \}_{1\leq i\leq l} \cup \{ \, \pm (\vep_i\pm \vep_j)\, \vert \, 1\leq i<j\leq l\, \}
\cup \{ \, \pm 2\vep_i\, \}_{1\leq i\leq l}, \]
where $\{\vep_i\}_{1\leq i \leq l}$ satisfy $I(\vep_i,\vep_j)=\delta_{i,j}$, the Kronecker delta. 

For detailed information about these $4$ infinite series of root systems, see the following sections: \\
\begin{center}
\begin{tabular}{|c||c|c|c|c|} \hline
Type of $R$ & $BC^{(1,2)}_l$ & $BC_l^{(4,2)}$ & $BC_l^{(2,2)\sigma}(1)$ & $BC_l^{(2,2)\sigma}(2)$ 
 \\ \hline
 & \S \ref{data:BC-1} & \S \ref{data:BC-3} & \S \ref{data:BC-5} & \S \ref{data:BC-6}  \\ \hline
\end{tabular}
\end{center}

\subsubsection{Reduced $\ast$-type}\label{section:intro_red-star}
There are two mERS of $\ast$-type, and they are defined by
\begin{align*}
R(BC_l^{(1,1)\ast})=
&(R(BCC_l)_s+\Z a) \cup (R(BCC_l)_m+\Z a) \cup (R(BC_l)_l+L_{1,1}), \\
R(BC_l^{(4,4)\ast})=
&(R(BC_l)_s+L_{1,1}) \cup (R(C^\vee BC_l)_m+2\Z a) \cup (R(C^\vee BC_l)_l+4\Z a),
\end{align*}
where we set 
\[ L_{1,1}=\{ \, ma+nb\, \vert \, (m-1)(n -1) \equiv 0 [2]\, \}. \]

For detailed information about $BC_l^{(1,1)\ast}$ (resp. $BC_l^{(4,4)\ast}$), see \S \ref{data:BC-2} and \ref{data:BC-4}, respectively.

Notice that they had been discovered by S. Azam \cite{Azam2002} in 2002. 
Hence, there is essentially no newly found reduced mERS with non-reduced affine quotient.

\subsection{Non-reduced mERSs}
Now consider the non-reduced $(R,G)$'s. Since there are many such mERSs, we call each type of the mERS by making reference to its affine quotient $R(X_l)=R/G$,  which is one of the types $BCC_l, C^\vee BC_l, BB_l^\vee$ and $C^\vee C_l$. See Appendix \ref{sect_affine-root-system} for detailed information on non-reduced affine root systems.

\subsubsection{Classical type}\label{sect:intro_non-red-classic}
Introduce now $4$ infinite series of the mERSs: 
$X_l^{(1)}$, \, $X_l^{(2)}(1)$, \,$X_l^{(2)}(2)$ and 
$X_l^{(4)}$. They are defined by
\begin{align*}
R(X_l^{(1)})=
&R(X_l)+\Z a, \\
R(X_l^{(2)}(i))=
&(R(X_l)_s+\Z a) \cup (R(X_l)_m+i \Z a) \cup (R(X_l)_l+2\Z a) \quad i=1,2, \\
R(X_l^{(4)})=
&(R(X_l)_s+\Z a) \cup (R(X_l)_m+2\Z a) \cup (R(X_l)_l+4\Z a). 
\end{align*}
Here, $R(X_l)$ is one of the $4$ types: $BCC_l, C^\vee BC_l, BB_l^\vee$ and $C^\vee C_l$.
\smallskip

For detailed information about these $16$ infinite series of mERSs, see the following sections: \\
\begin{center}
\begin{tabular}{|c||c|c|c|c|} \hline
$X_l$ &  $X_l^{(1)}$ & $X_l^{(2)}(1)$ & $X_l^{(2)}(2)$  & $X_l^{(4)}$ \\ \hline
$BCC_l$ & \S \ref{data:BCC-1} & \S \ref{data:BCC-4} & \S \ref{data:BCC-5} & \S \ref{data:BCC-8}
\\ \hline
$C^\vee BC_l $ & \S \ref{data:CBC-1} & \S \ref{data:CBC-2} & \S \ref{data:CBC-3} &\S \ref{data:CBC-6} \\ \hline
$BB_l^\vee$ & \S \ref{data:BB-1} &\S \ref{data:BB-2} & \S \ref{data:BB-3}& \S \ref{data:BB-5}\\ \hline
$C^\vee C_l$  & \S \ref{data:CC-1} & \S \ref{data:CC-4}& \S \ref{data:CC-5}& \S \ref{data:CC-12}\\ \hline
\end{tabular}
\end{center}
\subsubsection{$\ast$-type}\label{sect:intro_non-red-star}
For a non-reduced mERS $(R,G)$, there can be several $\ast$-types : $\ast_i, \, \ast_{i'}$ ($i=0,1$), $\ast_s$ and $\ast_l$. We introduce them gradually. For this purpose, let us  introduce the subsets $L_{i,j}, L_{i,j}^{s_1,s_2}$ ($i,j=0,1$) and ($s_1, s_2 \in \Z_{>0}$) of the lattice $\rad_{\Z}(I)=\Z a\oplus \Z b$ as follows ($L_{1,1}$ has already been introduced):
\begin{align*}
L_{i,j}=
&\{\, ma+nb\, \vert \, (m-i)(n-j)\equiv 0 [2]\, \}, \\
L_{i,j}^{s_1,s_2}=
&\{\, s_2ma+s_1nb\, \vert \, (m-i)(n-j)\equiv 0 [2]\, \}. 
\end{align*}
Let $t_1$ be an integer defined by the following table: \\
\begin{center}
\begin{tabular}{|c||c|c|c|} \hline
$X_l$ & $BCC_l$ & $C^\vee BC_l $ & $C^\vee C_l$  \\ \hline
$t_1$ & $1$ & $4$ & $2$  \\ \hline
\end{tabular}
\end{center}
\vskip 0.1in

and $t_2$  an integer such that $(t_1, t_2) \in \{1,2,4\}^2 \setminus \{ (1,4), (4,1)\}$.  \\

The set $R_m$ of middle length roots is always of the form
\[ R_m=R(X_l)_m+\min\{2,t_2\}\Z a, \]
and the set $R_s$ (resp. $R_l$) of short (resp. long) roots are described as follows: \\
\vskip 0.2in
\noindent{\fbox{\textbf{$\ast_i$-type ($i=0,1$)}}} \hskip 0.1in 
For each such pair $(t_1,t_2)$, the subsets $R_s$ and $R_l$  are given by
\begin{enumerate}
\item for $i=0$ and
\begin{enumerate}
\item[i)] $t_1, t_2 \in \{1,2\}$ and $(t_1, t_2) \neq (2,2)$, 
\[ R_s=R(X_l)_s+\Z a \qquad  \text{and} \qquad R_l=R(BC_l)_l+L_{0,0}^{t_1,t_2}, \]
\item[ii)] $t_1=t_2=2$, 
\[ R_s=R(BC_l)_s+L_{0,0} \qquad \text{and} \qquad R_l=R(BC_l)_l+L_{0,0}^{2,2}, \]
\item[ii)] $t_1,t_2 \in \{2,4\}$ and $(t_1, t_2) \neq (2,2)$, 
\[ R_s=R(BC_l)_s+L_{0,0} \qquad \text{and} \qquad R_l=R(X_l)_l+t_2\Z a,\]
\end{enumerate}
\item for $i=1$ and 
\begin{enumerate}
\item[i)] $t_1, t_2 \in \{1,2\}$ and $(t_1, t_2) \neq (2,2)$, 
\[ R_s=R(X_l)_s+\Z a \qquad \text{and} \qquad R_l=R(BC_l)_l+L_{1,1}^{t_1,t_2}, \]
\item[ii)] $t_1=t_2=2$,
\[ R_s=R(BC_l)_s+L_{0,0} \qquad \text{and} \qquad R_l=R(BC_l)_l+L_{1,1}^{2,2}, \]
\item[iii)] $t_1,t_2 \in \{2,4\}$ and $(t_1, t_2) \neq (2,2)$, 
\[ R_s=R(BC_l)_s+L_{1,1} \qquad \text{and} \qquad R_l=R(X_l)_l+t_2\Z a. \]
\end{enumerate}
\end{enumerate}
For $i=0,1$, the nomenclature of each mERS given by the above formulas, together with the sections with their root data, is as follows:
{\small
\begin{center}
\scalebox{0.8}{
\begin{tabular}{|c||c|c|c|c|c|c|c|} \hline
$(t_1,t_2)$ & $(1,1)$ & $(1,2)$ & $(2,1)$ & $(2,2)$ & $(2,4)$ & $(4,2)$ & $(4,4)$\\ \hline
Type of $(R,G)$ & $BCC_l^{(1)\ast_i}$ & $BCC_l^{(2)\ast_i}$ & $C^\vee C_l^{(1)\ast_i}$  
& $C^\vee C_l^{(2)\ast_i}$ & $C^\vee C_l^{(4)\ast_i}$ & $C^\vee BC_l^{(2)\ast_i}$ & $C^\vee BC_l^{(4)\ast_i}$\\ \hline
$i=0$ & \S \ref{data:BCC-2} & \S \ref{data:BCC-6} & \S \ref{data:CC-2} & \S \ref{data:CC-8}
& \S \ref{data:CC-13} & \S \ref{data:CBC-4} & \S \ref{data:CBC-7} \\ \hline
$i=1$ &  & \S \ref{data:BCC-7} & \S \ref{data:CC-3} & \S \ref{data:CC-10} 
& \S \ref{data:CC-14} & \S \ref{data:CBC-5} &  \\ \hline
\end{tabular}}
\end{center}}
Remark that the ERSs
$R(BCC_l^{(1)\ast_1})$ and $R(C^\vee BC_l^{(4)\ast_1})$ are reduced and they are exactly the ERSs $R(BC_l^{(1,1)\ast})$ and $R(BC_l^{(4,4)\ast})$, respectively. 
\vskip 0.2in

\noindent{\fbox{\textbf{$\ast_{i'}$-type ($i=0,1$)}}} \hskip 0.1in 
We have $t=t_1=t_2 \in \{1,2,4\}$, and the subsets $R_s$ and $R_l$  are given by
\begin{enumerate}
\item for $t=1$,
\[ R(BCC_l^{(1)\ast_{0'}})_s=R(BCC_l)_s+\Z a \qquad \text{and} \qquad R(BCC_l^{(1)\ast_{0'}})_l=R(BC_l)_l+L_{0,1}, \]
\item for $t=2$,
\[ R(C^\vee C_l^{(2)\ast_{1'}})_s=R(BC_l)_s+L_{0,0} \qquad \text{and} \qquad 
R(C^\vee C_l^{(2)\ast_{1'}})_l=R(BC_l)_l+L_{1,0}^{2,2} \]
\item for $t=4$, 
\[ R(C^\vee BC_l^{(4)\ast_{0'}})_s=R(BC_l)_s+L_{0,1} \qquad \text{and} \qquad
   R(C^\vee BC_l^{(4)\ast_{0'}})_l=R(C^\vee BC_l)_l+4\Z a.  \]
\end{enumerate}
For detailed information about these $3$ root systems, see the following sections: \\
\begin{center}
\begin{tabular}{|c||c|c|c|} \hline
$X_l$ &  $BCC_l$ & $C^\vee BC_l$ & $C^\vee C_l$  \\ \hline
$\ast_{0'}$ &\S \ref{data:BCC-3} & & \S \ref{data:CBC-8}\\ \hline
$\ast_{1'}$ & &\S \ref{data:CC-9} & \\ \hline
\end{tabular}
\end{center}

\vskip 0.2in 

\noindent{\fbox{\textbf{$\ast_{\natural}$-type ($\natural \in \{s,l\}$)}}} \hskip 0.1in 
When $t_1=t_2=2$,  the subsets $R_s$ and $R_l$  are given by
\begin{enumerate}
\item for $\natural=s$,
\[ R(C^\vee C_l^{(2)\ast_{s}})_s=R(BC_l)_s+L_{0,0}  \qquad \text{and} \qquad
   R(C^\vee C_l^{(2)\ast_{s}})_l=R(C^\vee C_l)_l+2\Z a, \]
\item for $\natural=l$,
\[  R(C^\vee C_l^{(2)\ast_{l}})_s=R(C^\vee C_l)_s+\Z a \qquad \text{and} \qquad 
    R(C^\vee C_l^{(2)\ast_{l}})_l=R(BC_l)_l+L_{0,0}^{2,2}. \]
\end{enumerate}
\begin{center}
\begin{tabular}{|c||c|c|} \hline
$\ast_\natural$ ($\natural \in \{s,l\}$) &  $\ast_s$ & $\ast_l$  \\ \hline
$C^\vee C_l^{(2)\ast_\natural}$ &\S \ref{data:CC-6} &  \S \ref{data:CC-7}\\ \hline
\end{tabular}
\end{center}
\vskip 0.2in
Remark that the numbers $t_i$ ($i=1,2$) are the so-called first (resp. second) \textbf{tier numbers} (see I.3.2). Finally, there is one isolated case (cf. \S \ref{data:BB-4}):
\[ R(BB_2^{\vee (2)\ast})=(R(BC_2)_s+\Z a+\Z b) \cup (R(BC_2)_m+L_{0,0}) \cup (R(BC_2)_l+2\Z a+2\Z b). \]

\subsubsection{$\diamond$-type}\label{sect:intro_non-red-diamond}
There is only one such root system which we may call of type $C^\vee C_l^{(2)\diamond}$. This root system $R$ is defined as follows:
\[ R=(R(C^\vee C_l)_s+\Z a) \cup (R(C^\vee C_l)_m+\Z a) \cup (R(BC_l)_l+\{ \, ma+2nb\, \vert\,  m-n\equiv 0 [2]\, \}). \]
For detailed information about this root system, see \S \ref{data:CC-11}.
\medskip

\subsection{Weak classification}\label{sect:Classification-weak}

\begin{thm}\label{thm_classification-reduced} 
Let $(R,G)$ be a reduced marked elliptic root system belonging to the real vector space
with symmetric bilinear form of signature $(l,2,0)$. Assume that the quotient root system $R/G$ is non-reduced. Then it is isomorphic to one of the following:
\[ BC_l^{(1,2)}, \qquad BC_l^{(1,1)\ast}, \qquad BC_l^{(4,2)}, \qquad BC_l^{(4,4)\ast},  \qquad  BC_l^{(2,2)\sigma}(2) \]
for $l\geq 1$ and for $l\geq 2$, there is also 
\[ BC_l^{(2,2)\sigma}(1). \] 
\end{thm}
Here, by the same way as we have stated after Theorem \ref{thm:classification-red}, we see that all of these mERSs are non-isomorphic. Thus, this classification gives the classification of the isomorphism classes of reduced mERSs with non-reduced affine quotient. 

\begin{thm}\label{thm_classification-non-reduced}
Let $(R,G)$ be a non-reduced marked elliptic root system belonging to the real vector space
with symmetric bilinear form of signature $(l,2,0)$.  Then, 
as marked elliptic root system, it is isomorphic to one of the following:

\item 1. $R/G$ of type $BCC_l$ $(l\geq 1)$:
\begin{align*}
& BCC_l^{(1)}, \; BCC_l^{(1)\ast_0},\; BCC_l^{(1)\ast_{0'}}, \\
& BCC_l^{(2)}(1)\,  (l>1), \; BCC_l^{(2)}(2), \; BCC_l^{(2)\ast_p} \, (p \in \{0,1\}), \\
& BCC_l^{(4)}.
\end{align*}
\item 2. $R/G$ of type $C^\vee BC_l$ $(l\geq 1)$:
\begin{align*}
& C^\vee BC_l^{(1)}, \\
& C^\vee BC_l^{(2)}(1)\,  (l>1), \; C^\vee BC_l^{(2)}(2), \; C^\vee BC_l^{(2)\ast_p} \, (p \in \{0,1\}), 
\\
& C^\vee BC_l^{(4)}, \; C^\vee BC_l^{(4)\ast_0},\; C^\vee BC_l^{(4)\ast_{0'}}.
\end{align*}
\item 3. $R/G$ of type $BB_l^\vee$ $(l\geq 2)$:
\begin{align*}
& BB_l^{\vee \, (1)}, \; BB_l^{\vee \, (2)}(1)\,  (l>1), \;  BB_l^{\vee \, (2)}(2), \; BB_l^{\vee \, (4)}, \\
& BB_2^{\vee (2) \ast}.
\end{align*}
\item 4. $R/G$ of type $C^\vee C_l$ $(l\geq 1)$:
\begin{align*}
&C^\vee C_l^{(1)}, \; C^\vee C_l^{(1)\ast_p} \, (p \in \{0,1\}), \\
&C^\vee C_l^{(2)}(1)\,  (l>1), \; C^\vee C_l^{(2)}(2), \; C^\vee C_l^{(2)\ast_s}, \; C^\vee C_l^{(2)\ast_l}, \\
&C^\vee C_l^{(2)\ast_{0}}, \; C^\vee C_l^{(2)\ast_1}, \; 
C^\vee C_l^{(2)\ast_{1'}}, \; C^\vee C_l^{(2)\diamond}, \\
&C^\vee C_l^{(4)}, \; C^\vee C_l^{(4)\ast_p} \, (p \in \{0,1\}).
\end{align*}
\end{thm}

\section[Proof of weak classification theorems]{Proof of Theorems \ref{thm_classification-reduced} and \ref{thm_classification-non-reduced} }
\label{sect_mers-II}
Even if the obtained root systems are always reduced elliptic root systems, in view of the reduced mERSs, but still, Saito's classification is not complete. 

We classify now the reduced mERSs $(R,G)$ where the quotient $R/G$ is non-reduced. 

\subsection{Preliminary}
Here, we fix the following convention which will be used throughout this section. \\
Let $(F,I)$ be the real vector space $F=\R^{l+2}$ equipped with a positive semi-definite symmetric bilinear form $I$ with $\dim \rad(I)=2$. Fix a basis $(\vep_1, \vep_2, \cdots, \vep_l, b,a)$ of $F$ (cf. \eqref{eqn:SL2}) satisfying \eqref{def_base-F}, i.e., 
\[ I(\vep_i, \vep_j)=\delta_{i,j}, \qquad \rad(I)=\R a\oplus  \R b. \]
We realize the root systems of type $B_l, C_l$ and $D_l$ in $F_f:=\bigoplus_{1\leq i\leq l} \R \vep_i$ as in Appendix \ref{sect_finite-root-system}, i.e., 
\begin{align*}
& R(D_l)=\{ \, \pm (\vep_i \pm \vep_j)\, \vert\, 1\leq i<j\leq l\, \}, \\
& R(B_l)=R(D_l) \cup \{ \, \pm \vep_i\,\}_{1\leq i\leq l}, \quad R(C_l)=R(D_l) \cup \{\, \pm 2\vep_i\ \}_{1\leq i\leq l}. 
\end{align*}
As the group of diagram automorphisms of type  $B_l$ and $C_l$ are trivial, it follows that
\[ \Aut(R_f)=W(R_f), \qquad R_f=R(B_l)\; \text{or} \; R(C_l). \]
For type $D_l$, the automorphism group of its Dynkin diagram is 
\[ \Aut(\Gamma(D_l)) \cong 
\begin{cases} \; \Z/2\Z \quad & l\geq 2\; \text{s.t.}\; l\neq 4, \\
                     \;\fS_3 \quad &l=4. \end{cases}
\]
An outer automorphism of $R(D_l)$ is given by
\[  \sigma: \; F_f \longrightarrow F_f; \quad \vep_i \mapsto  
                   \begin{cases} \; \vep_i \; & 1\leq i<l  \\ \;  -\vep_l \; & i=l \end{cases}, 
\]
and this is realized as
\[ r_{\vep_l}=r_{2\vep_l} \in W(B_l)=W(C_l). \]
In particular, 
\[ W(B_l)=W(C_l) \cong W(D_l) \rtimes \langle \sigma \rangle. \]
For later use, set
\[ \Aut'(R(D_l))=W(D_l) \rtimes \langle \sigma \rangle. \]
Remark that, with this notation,
\[ \Aut'(R(D_l))=\Aut(R(B_l))=\Aut(R(C_l)) \cong (\Z/2\Z)^l \rtimes \fS_l{\color{red},} \]
and gives exactly the automorphism group $\Aut(R(D_l))$ for $l\neq 4$ (including the cases $l=2,3$). Nevertheless,
$\Aut'(R(D_4))$ is not even a normal subgroup of $\Aut(R(D_4)) =W(D_4) \rtimes \fS_3 \cong W(F_4)$. 


\subsection{Key idea}\label{sect_3.1}
Let $(R,G)$ be a mERS whose quotient $R_{\af}:=R/G$ is non-reduced, $R_l,\, R_m$ and $R_s$ be the subset of $R$ consisting of the long, middle length and short roots, respectively. Both of the subsets $R_+:=R_l \cup R_m$ and $R_-:=R_m \cup R_s$ are reduced elliptic root systems. In particular, as we assume that $R_{\af}$ is a non-reduced affine root system, their quotients 
by the full radical $\rad(I)$ are of type $C_l$ and $B_l$, respectively. By the classification of mERSs $(R,G)$ whose quotient $R/G$ is reduced due to K. Saito \cite{Saito1985} (cf. Appendix \ref{sect_clasfficiation-red-quot}), the possible type of the mERSs $(R_{-},G)$ and $(R_{+},G)$ are given in the next table:

\begin{table}[h]
\resizebox{.98\textwidth}{!}{%
\begin{tabular}{|c||c|c|c|c|c|c|c|c|c|c|} \hline
Type of $(R_{-},G)$ 
& $B_l^{(1,1)}$ & $B_l^{(1,2)}$ & $C_2^{(1,1)}$ & $C_2^{(1,2)}$ & $B_l^{(2,1)}$ & $B_l^{(2,2)}$ & $B_l^{(2,2)\ast}$ & $C_2^{(1,1)\ast}$ & $A_1^{(1,1)}$ \\ 
\cline{1-9}
Type of $(R_{+},G)$ 
& $C_l^{(2,2)}$ & $C_l^{(2,1)}$ & $B_2^{(2,2)}$ & $B_2^{(2,1)}$ &$C_l^{(1,2)}$ &  $C_l^{(1,1)}$ & $C_l^{(1,1)\ast}$ & $B_2^{(2,2)\ast}$ & $A_1^{(1,1)\ast}$ \\ \hline
Range of $l$ &  \multicolumn{2}{c|}{$l> 2$} & \multicolumn{2}{c|}{$l=2$} & \multicolumn{3}{c|}{$l\geq 2$}  & $l=2$ & $l=1$\\ \hline
\end{tabular}}
\end{table}
For $t \in \{1,2\}$, set 
\begin{align*}
R(B_2^{(1,t)})=
&R(B_2)+t\Z a+\Z b, \\
R(C_2^{(2,t)})=
&(R(C_2)_s+\Z a+\Z b) \cup (R(C_2)_l+t\Z a+2\Z b),
\end{align*}
to simplify further discussion.

Thus, to classify the mERSs $(R,G)$ of rank $>1$, we just have to glue two mERSs, the one $(R_{-},G)$  in the first line and the other $(R_{+},G)$ in the second line, in the above table. Here, by the term ``glue'', we mean the pushout
\[ (R_+,G) \bigcup_{(R_m,G)} (R_-,G) \]
where  $R_m=R_+ \cap R_-$ is isomorphic to the elliptic root system of type $D_l^{(1,1)}$, except for the case where $R_\pm$ are of type $C_2^{(1,1)\ast}$ and $B_2^{(1,1)\ast}$, respectively. In the latter case, $R_m$ is given by $R(D_2^{(1,1)\ast}):=R(D_2)+\{\, ma+nb\, \vert\, mn \equiv 0\, [2]\, \}$.

\smallskip

The \textbf{gluing procedure} can be understood as follows. Let $R_+$ and $R_-$ be elliptic root systems realised in $(F,I)$ as in Appendix \ref{sect_clasfficiation-red-quot}, whose full quotients are of type $C_l$ and $B_l$ respectively, such that $R_m:=R_+ \cap R_-=(R_+)_s = (R_-)_l$ is either of type $D_l^{(1,1)}$ or $D_2^{(1,1)\ast}$. As any automorphism $\varphi$ of $(R_m,G)$ is the restriction of a linear isomorphism of $(F,I)$, we can define a subset $R$ of $(F,I)$ by
\[ R=R_- \amalg_\varphi R_+:=R_- \cup \varphi(R_+). 
\]

\begin{rem} {\rm One can also consider the opposite gluing, i.e., $R_+ \amalg_\varphi R_-=R_+ \cup \varphi(R_-)$. 
As the group of automorphisms of $C_l$-type elliptic root systems realized in 
\[ \Aut'(R(D_l^{(1,1)}),G):=\{ \, \varphi \in \Aut'(R(D_l^{(1,1)}))\, \vert \, \varphi\vert_G=\id_G\} 
\] 
has a smaller index ($1, 2$ or $4$) than that of $B_l$-type, we have chosen the gluing procedure as above for rank $>2$. For rank $2$ cases, a convenient one will be chosen. }
\end{rem}

Since the full quotient $R/\rad(I)$ is a root system of type $BC_l$, we may consider only those $\varphi$ which belong to the subgroup $\Aut'(R(D_l^{(1,1)}),G)$ of $\Aut^+(R(D_l^{(1,1)}),G)$,  isomorphic to
\[ \left(E_F(a \otimes P(D_l^\vee) + b \otimes P(D_l^\vee))\rtimes \Aut'(D_l)\right) \rtimes U_-  \]
under the isomorphism in Corollary \ref{cor:normal}, where we set $U_-$ being the subgroup of $SL(2,\Z)$ generated by ${}^tT$. As we are interested in the classification of mERSs, we shall concentrate on $\Gamma^+(R) \cap U_-$.  For a positive integer $n>0$, let $U_-(n)$ be the subgroup of $U_-$ generated by $({}^tT)^n$. 
By Proposition \ref{prop:Gamma}, we have the next list:
\[ \Gamma^+(R) \cap U_-=\begin{cases} \; U_-\; 
                  & R\, \text{of type}\; X_l^{(t,t)}, \, X_l^{(t,t)\ast}, \, X_l^{(2,1)}, \, BC_l^{(2,1)}, \, G_2^{(3,1)}, \\
                                                                \; U_-(2)\; 
                  & R\, \text{of type}\; X_l^{(1,2)}, \, BC_l^{(2,4)}, \, BC_l^{(2,2)}(1), \, BC_l^{(2,2)}(2), \\
                                                               \; U_-(3)\; 
                  & R\, \text{of type}\; G_2^{(1,3)}.
                                          \end{cases}
\]
\begin{lemma}
The set $R=R_- \amalg_\varphi R_+$ is a root system, i.e., it is stable under the action of the reflection group of $R$. 
\end{lemma}
\begin{proof}
It is sufficient to show that 
\begin{enumerate}
\item
$R_s$ is stable by $r_\alpha$ for $\alpha \in R_l\cong R(BC_l)_l+\Z a+\Z b$, and
\item
$\varphi(R_l)$ is stable by $r_\alpha$ for $\alpha \in R_s \cong R(BC_l)_s+\Z a+\Z b$,
\end{enumerate}
since any such set is a subset of $R_- \amalg_{\id} R(C_l^{(1,1)})$ and $R(B_l^{(r,r)}) \amalg_{\varphi} R_+$ with $r \in \{1,2\}$. 
\end{proof}

Let  $\Aut^+(R_\pm,G)= \Aut^+(R_\pm) \cap \Aut'(R_m,G)$. For $\varphi_{\pm} \in \Aut^+(R_\pm,G)$ and $\varphi_0 \in \Aut'(R_m,G)$, setting  $\varphi=\varphi_-\circ \varphi_0 \circ \varphi_+$, it follows that
\begin{align*}
R_- \amalg_{\varphi} R_+=
&\varphi_-(R_-) \cup (\varphi_-\circ \varphi_0)(\varphi_+(R_+)) 
=\varphi_-(R_- \cup \varphi_0(R_+)) \\
=
&\varphi_-(R_- \amalg_{\varphi_0} R_+),
\end{align*}
namely, $R_- \amalg_\varphi R_+$ is isomorphic to $R_- \amalg_{\varphi_0} R_+$. Hence, 
it is enough to analyse $R_- \amalg_\varphi R_+$ for each coset in $\Aut^+(R_-,G) \backslash \Aut'(R_m,G) \slash \Aut^+(R_+,G)$. 

For later use, for $\delta \in \rad(I)$, set 
\begin{equation}\label{def_translation}
\tau_\delta=E_F\left((-\delta)\otimes (\vep_1+\vep_2+\cdots +\vep_l)\right).
\end{equation}

First classify non-reduced mERSs for $l\geq 2$. The cases when $l=1$ will be treated separately. \\

Some cases require detailed analysis, and for each type of the subtleties we will indicate a way to handle them when it first appears; \S \ref{subsubsect_BCC-l}.(2)(I) and \ref{subsect_CC-l}.(2)(I).



\subsection {${R/G}$ of type $BCC_l\; (l> 1)$}\label{subsubsect_BCC-l}
In this case, since 
\[ R_+/G \cong R(C_l^{(1)}) \qquad \text{and}\qquad R_-/G \cong \begin{cases} \; R(B_l^{(1)}) \; &l>2, \\ \; R(C_2^{(1)})\; &l=2, \end{cases} 
\]
the mERS $(R_+,G)$ is either of type $C_l^{(1,1)}, \, C_l^{(1,2)}$ or $C_l^{(1,1)\ast}$, and
$(R_-,G)$ is either of type
\begin{enumerate}
\item $B_l^{(1,1)}$ or $B_l^{(1,2)}$ for $l>2$ and
\item $C_2^{(1,1)}$ or $C_2^{(1,2)}$ for $l=2$.
\end{enumerate}
\vskip 0.1in

\noindent{(1).\; \fbox{The case $R_+=R(C_l^{(1,1)})$\; ($l>1$)}} \quad \\

In this case, $\Aut^+(R_+,G)=\Aut'(R_m,G)$ and it is enough to consider the gluing via trivial automorphism.
\begin{enumerate}
\item[(I)] $R_-=\begin{cases} \; R(B_l^{(1,1)})\; & l>2, \\ \;R(C_2^{(1,1)})\; &l=2. \end{cases}$ \quad Here, we obtain the root system
\begin{align*}
R=
&(R(BC_l)_s +\Z a+\Z b) \cup (R(BC_l)_m+\Z a+\Z b) \cup (R(BC_l)_l+\Z a+\Z b) \\
=
&R(\mathbf{BCC_l^{(1)}}).
\end{align*}
\item[(II)] $R_-=\begin{cases}\; R(B_l^{(1,2)})\; &l>2, \\ \; R(C_2^{(1,2)})\; &l=2. \end{cases}$ \quad Here, we obtain the root system
\begin{align*}
R=
&(R(BC_l)_s+\Z a+\Z b)\cup (R(BC_l)_m+2\Z a+\Z b) \cup (R(BC_l)_l+2\Z a+\Z b) \\ =
&R(\mathbf{BCC_l^{(2)}(2)}). 
\end{align*}
\end{enumerate}

\vskip+0.1in


\noindent{(2).\; \fbox{The case $R_+=R(C_l^{(1,2)})$\; ($l>1$)}} \quad \\

In this case, the index $[\Aut'(R_m,G):\Aut^+(R(C_l^{(1,2)}),G)]$ is $4$. Notice that the set 
$U_-/U_-(2)$ is represented by $\{ 1, {}^tT\}$.
\begin{enumerate}
\item[(I)] $R_-=\begin{cases}\; R(B_l^{(1,1)})\; &l>2, \\ \; R(C_2^{(1,1)})\; &l=2.\end{cases}$ \\ Here, the coset $\Aut'(R_m,G)/\Aut^+(R(C_l^{(1,2)}),G)$ is represented by 
$\{1, {}^tT \} \circ \{ \id, \tau_{\frac{1}{2}a} \}$.\footnote{Set $M=E_F( a\otimes \tilde{P}(C_l^\vee)+b\otimes P(C_l^\vee))$. Then, since $M \cup \tau_{\frac{1}{2}a}M=E_F(P(D_l^\vee)\otimes a+P(D_l^\vee)\otimes b)$, 
\begin{align*}
\Aut'(R_m)/\Aut'(D_l)=
&(M \cup \tau_{\frac{1}{2}a}M) \cdot (\chi(1)U_-(2) \cup \chi({}^tT)U_-(2)) \\
=
&\chi(1)(M\cup \tau_{\frac{1}{2}a}M)U_-(2)\cup \chi({}^tT)(M\cup \tau_{\frac{1}{2}a}M)U_-(2),
\end{align*}
whence $\Aut'(R_m)/\Aut^+(R(C_l^{(1,2)}))$ is represented by $\{1, {}^tT\}\circ \{\id, \tau_{\frac{1}{2}a}\}$ as $\Aut^+(R(C_l^{(1,2)}))/\Aut'(D_l)\cong MU_-(2)$.}
Only $\tau_0=\id$ and $\chi({}^tT)$ extend to automorphisms of $R_-$. 
\begin{enumerate}
\item[(i)] For $\varphi=\id$, we obtain the root system
\begin{align*}
R=
& (R(BC_l)_s+\Z a+\Z b)\cup (R(BC_l)_m+\Z a+\Z b) \cup (R(BC_l)_l+2\Z a+\Z b) \\ =
&R(\mathbf{BCC_l^{(2)}(1)}). 
\end{align*}
\item[(ii)] For $\varphi=\tau_{\frac{1}{2}a}$, we obtain the root system
\begin{align*}
R=
& (R(BC_l)_s+\Z a+\Z b)\cup (R(BC_l)_m+\Z a+\Z b) \\
\cup 
&(R(BC_l)_l+(1+2\Z) a+\Z b) 
= R(\mathbf{BC_l^{(1,2)}}).
\end{align*}
\end{enumerate}
\item[(II)] $R_-=\begin{cases}\; R(B_l^{(1,2)})\; &l>2, \\ \; R(C_2^{(1,2)})\; &l=2.\end{cases}$ \\
Here, the coset $\Aut'(R_m,G)/\Aut^+(R(C_l^{(1,2)}),G)$ is represented by $\{ \id, \tau_{a} \} \circ \{1, {}^tT \}$. Only $\tau_0=\id$ and $\tau_a$ extend to automorphisms of $R_-$.
\begin{enumerate}
\item[(i)] For $\varphi=\id$, we obtain the root system
\begin{align*}
R=
&(R(BC_l)_s+\Z a+\Z b)\cup (R(BC_l)_m+2\Z a+\Z b) \cup (R(BC_l)_l+4\Z a+\Z b) \\
=
&R(\mathbf{BCC_l^{(4)}}).
\end{align*}
\item[(ii)] For $\varphi=\chi({}^tT)$, we have $R\cong R(BCC_l^{(4)})$ via $\chi(({}^tT)^{-2})$. 
\end{enumerate}
\end{enumerate}

\vskip+0.1in

\medskip

\noindent{(3).\; \fbox{The case $R_+=R(C_l^{(1,1)\ast})$\; ($l>1$)}} \quad \\

In this case, the index $[\Aut'(R_m,G):\Aut^+(R(C_l^{(1,1)\ast})),G]$ is $4$.  

\begin{enumerate}
\item[(I)] $R_-=\begin{cases}\; R(B_l^{(1,1)})\; &l>2, \\ \; R(C_2^{(1,1)})\; &l=2.\end{cases}$ \\ Here, the coset $\Aut'(R_m,G)/\Aut^+(R(C_l^{(1,1)\ast}),G)$ is represented by the next $4$ elements: $\tau_\delta$ with $\delta \in \{0, \frac{1}{2}a, \frac{1}{2}b, \frac{1}{2}(a+b)\}$. Except for $\tau_0=\id$, they don't extend to automorphisms of $R_-$.
\begin{enumerate}
\item[(i)] For $\varphi=\id$, we obtain the root system
\begin{align*}
R=
&(R(BC_l)_s+\Z a+\Z b) \cup (R(BC_l)_m+\Z a+\Z b) \\
\cup
&(R(BC_l)_l+ \{\, ma+nb\, \vert\, m,n \in \Z\, \text{s.t.}\, mn \equiv 0\, [2]\, \}) 
= 
R(\mathbf{BCC_l^{(1)\ast_0}}).
\end{align*}
\item[(ii)] For $\varphi=\tau_{\frac{1}{2}a}$, we have  $R\cong R(BCC_l^{(1)\ast_0})$ via $\chi(({}^tT)^{-1})$.
\item[(iii)] For $\varphi=\tau_{\frac{1}{2}b}$, we obtain the root system
\begin{align*}
R=
&(R(BC_l)_s+\Z a+\Z b) \cup (R(BC_l)_m+\Z a+\Z b) \\
\cup
&(R(BC_l)_l+ \{\, ma+nb\, \vert\, m,n \in \Z\, \text{s.t.}\, m(n-1) \equiv 0\, [2]\, \}) \\
=
&R(\mathbf{BCC_l^{(1)\ast_{0'}}}).
\end{align*}
\item[(iv)] For $\varphi=\tau_{\frac{1}{2}(a+b)}$, we obtain the root system
\begin{align*}
R=
&(R(BC_l)_s+\Z a+\Z b) \cup (R(BC_l)_m+\Z a+\Z b) \\
\cup
&(R(BC_l)_l+ \{\, ma+nb\, \vert\, m,n \in \Z\, \text{s.t.}\, (m-1)(n-1) \equiv 0\, [2]\, \}) \\
=
&R(\mathbf{BC_l^{(1,1)\ast}}).
\end{align*}
\end{enumerate}
\item[(II)] $R_-=\begin{cases}\; R(B_l^{(1,2)})\; &l>2, \\ \; R(C_2^{(1,2)})\; &l=2.\end{cases}$ \\ Now, the coset $\Aut'(R_m,G)/\Aut^+(R(C_l^{(1,1)\ast}),G)$ is represented by the next $4$ elements: $\tau_\delta$ with $\delta \in \{0, a, \frac{1}{2}b, a+\frac{1}{2}b\}$. The only $\tau_0=\id$ and $\tau_a$ extend to automorphisms of $R_-$. 
\begin{enumerate}
\item[(i)] For $\varphi=\id$, we obtain the root system
\begin{align*}
R=
&(R(BC_l)_s+\Z a+\Z b) \cup (R(BC_l)_m+2\Z a+\Z b) \\
\cup
&(R(BC_l)_l+ \{\, 2ma+nb\, \vert\, m,n \in \Z\, \text{s.t.}\, mn \equiv 0\, [2]\, \})
=
R(\mathbf{BCC_l^{(2)\ast_0}}).
\end{align*}
\item[(ii)] For $\varphi=\tau_{a+\frac{1}{2}b}$, we get the root system
\begin{align*}
R=
&(R(BC_l)_s+\Z a+\Z b) \cup (R(BC_l)_m+2\Z a+\Z b) \\
\cup
&(R(BC_l)_l+ \{\, 2ma+nb\, \vert\, m,n \in \Z\, \text{s.t.}\, (m-1)(n-1) \equiv 0\, [2]\, \}) \\
=
&R(\mathbf{BCC_l^{(2)\ast_1}}).
\end{align*}
\end{enumerate}
\end{enumerate}

\subsection{$R/G$ of type $C^\vee BC_l (l> 1)$}  
In this case, since 
\[ R_+/G \cong \begin{cases} \; R(C_l^{(2)}) \; &l>2, \\ \; R(B_2^{(2)})\; &l=2, \end{cases} \qquad 
\text{and} \qquad
  R_-/G \cong R(B_l^{(2)}), 
\]
the mERS $(R_+,G)$ is either of type
\begin{enumerate}
\item $C_l^{(2,1)}$ or $C_l^{(2,2)}$ for $l>2$ and
\item $B_2^{(2,1)}$ or $B_2^{(2,2)}$ for $l=2$,
\end{enumerate}
and $(R_-,G)$ is either of type $B_l^{(2,1)}, \, B_l^{(2,2)}$ or $B_l^{(2,2)\ast}$. 
\vskip 0.1in

\noindent{(1).\; \fbox{The case $R_+=\begin{cases}\; R(C_l^{(2,1)})\; &l>2, \\ \; R(B_2^{(2,1)})\; &l=2.\end{cases}$}} \quad \\

In this case, the index $[\Aut'(R_m):\Aut^+(R_+)]$ is $2$. 
\begin{enumerate}
\item[(I)] $R_-=R(B_l^{(2,1)})$\; ($l>1$). The coset $\Aut'(R_m,G)/\Aut^+(R_+,G)$ is represented by $\{ \id, \tau_{b} \}$. All of them extend to automorphisms of $R(B_l^{(2,1)})$ and we obtain the root system
\begin{align*} 
R=
&(R(BC_l)_s+\Z a+\Z b) \cup (R(BC_l)_m+\Z a+2\Z b) \cup (R(BC_l)_l+\Z a+4\Z b) \\
=
& R(\mathbf{C^\vee BC_l^{(1)}}).
\end{align*}
\item[(II)] $R_-=R(B_l^{(2,2)})$\; ($l>1$). Now, the coset $\Aut'(R_m,G)/\Aut^+(R_+,G)$ is represented by $\{ \id, \tau_{b} \}$.  All of them extend to automorphisms of $R(B_l^{(2,2)})$
and we obtain the root system
\begin{align*}
R=
&(R(BC_l)_s+\Z a+\Z b) \cup (R(BC_l)_m+2\Z a+2\Z b) \cup (R(BC_l)_l+2\Z a+4\Z b) \\
=
&R(\mathbf{C^\vee BC_l^{(2)}(2)}).
\end{align*}
\item[(III)] $R_-=R(B_l^{(2,2)\ast})$\; ($l>1$). Now, the coset $\Aut'(R_m,G)/\Aut^+(R_+,G)$ is represented by $\{ \id, \tau_{b} \}$.   The map $\tau_b$ does not extend to automorphisms of $R(B_l^{(2,2)\ast})$. 
\begin{enumerate}
\item[(i)] For $\varphi=\id$, we obtain the root system
\begin{align*}
R=
&(R(BC_l)_s+\{\,ma+nb\, \vert\, m,n \in \Z\, \text{s.t.}\, mn \equiv 0\, [2]\, \}) \\
\cup
&(R(BC_l)_m+2\Z a+2\Z b) \cup (R(BC_l)_l+2\Z a+4\Z b)= R(\mathbf{C^\vee BC_l^{(2)\ast_0}}).
\end{align*}
\item[(ii)] For $\varphi=\tau_{b}$, we have $R \cong R(C^\vee BC_l^{(2)\ast_1})$ via $\tau_{a+b}$, where 
\begin{align*}
R(\mathbf{C^\vee BC_l^{(2)\ast_1}})=
&(R(BC_l)_s +\{\, ma+nb\, \vert\, m,n \in \Z\, \text{s.t.}\, (m-1)(n-1)\equiv 0\, [2]\, \}) \\
\cup 
&(R(BC_l)_m+2\Z a +2\Z b) \cup (R(BC_l)_l+2\Z a+4\Z b).
\end{align*}
\end{enumerate}
\end{enumerate}

\vskip+0.1in

\medskip

\noindent{(2).\; \fbox{The case $R_+=\begin{cases}\; R(C_l^{(2,2)})\; &l>2, \\ \; R(B_2^{(2,2)})\; &l=2. \end{cases}$}} \quad \\

In this case, the index $[\Aut'(R_m,G):\Aut^+(R_+,G)]$ is $4$.   

\begin{enumerate}
\item[(I)] $R_-=R(B_l^{(2,1)})$\; ($l>1$). The coset $\Aut'(R_m,G)/\Aut^+(R_+,G)$ is represented by the next $4$ elements: $\tau_\delta$ with $\delta \in \{0, \frac{1}{2}a, b, \frac{1}{2}a+b\}$. Only $\tau_0=\id$ and $\tau_b$ extend to automorphisms of $R(B_l^{(2,1)})$.
\begin{enumerate}
\item[(i)] For $\varphi=\id$, we obtain the root system
\begin{align*}
 R=
 &(R(BC_l)_s+\Z a+\Z b) \cup (R(BC_l)_m+\Z a+2\Z b) \cup (R(BC_l)_l+2\Z a+4\Z b) \\
 =
 &R(\mathbf{C^\vee BC_l^{(2)}(1)}).
\end{align*}
\item[(ii)] For $\varphi=\tau_{\frac{1}{2}a}$, we obtain the root system
\begin{align*}
R=
&(R(BC_l)_s+\Z a+\Z b) \cup (R(BC_l)_m+\Z a+2\Z b)  \\
\cup
&(R(BC_l)_l+(1+2\Z) a+4\Z b)= R(\mathbf{BC_l^{(4,2)}}).
\end{align*}
\end{enumerate}
\item[(II)] $R_-=R(B_l^{(2,2)})$\; ($l>1$). Now, the coset $\Aut'(R_m,G)/\Aut^+(R_+,G)$ is represented by the next $4$ elements: $\tau_\delta$ with $\delta \in \{0, a, b, a+b)\}$. All of them extend to automorphisms of $R(B_l^{(2,2)})$ and we obtain the root system
\begin{align*}
R=
&(R(BC_l)_s+\Z a+\Z b) \cup (R(BC_l)_m+2\Z a+2\Z b) \cup (R(BC_l)_l+4\Z a+4\Z b) \\
=
& R(\mathbf{C^\vee BC_l^{(4)}}).
\end{align*}
\item[(III)] $R_-=R(B_l^{(2,2)\ast})$\; ($l>1$). Here, the coset $\Aut'(R_m,G)/\Aut^+(R_+,G)$ is represented by the next $4$ elements: $\tau_\delta$ with $\delta \in \{0, a, b, a+b\}$. Except for $\tau_0=\id$, they don't extend to automorphisms of $R(B_l^{(2,2)\ast})$.
\begin{enumerate}
\item[(i)] For $\varphi=\id$, we obtain the root system
\begin{align*}
R=
&(R(BC_l)_s+\{\, ma+nb\, \vert\, m,n \in \Z\, \text{s.t.}\, mn\equiv 0\, [2]\, \}) \\
\cup
&(R(BC_l)_m+2\Z a+2\Z b)  \cup (R(BC_l)_l+4\Z a+4\Z b) =R(\mathbf{C^\vee BC_l^{(4)\ast_0}}).
\end{align*}
\item[(ii)] For $\varphi=\tau_a$, we have $R\cong (R(C^\vee BC_l^{(4)\ast_0}))$ via $\chi(({}^tT)^{-1}) \circ \tau_a$. 
\item[(iii)] For $\varphi=\tau_b$, we have $R \cong R(C^\vee BC_l^{(4)\ast_{0'}})$ via $\tau_b$, where
\begin{align*}
R(\mathbf{C^\vee BC_l^{(4)\ast_{0'}}})=
&(R(BC_l)_s+\{\, ma+nb\, \vert\, m,n \in \Z\, \text{s.t.}\, m(n-1)\equiv 0\, [2]\, \}) \\
\cup
&(R(BC_l)_m+2\Z a+2\Z b)  \cup (R(BC_l)_l+4\Z a+4\Z b).
\end{align*}
\item[(iv)] For $\varphi=\tau_{a+b}$, we have $R \cong R(BC_l^{(4,4)\ast})$ via $\tau_{a+b}$, where
\begin{align*}
R(\mathbf{BC_l^{(4,4)\ast}})=
&(R(BC_l)_s+\{\, ma+nb\, \vert\, m,n \in \Z\, \text{s.t.}\, (m-1)(n-1)\equiv 0\, [2]\, \}) \\
\cup
&(R(BC_l)_m+2\Z a+2\Z b)  \cup (R(BC_l)_l+4\Z a+4\Z b).
\end{align*}
\end{enumerate}
\end{enumerate}

\subsection{$R/G$ of type $BB_l^\vee \; (l\geq 2)$}\label{subsubsect_BB-l}
In this case, since 
\[ R_+/G \cong \begin{cases} \; R(C_l^{(2)}) \; &l>2, \\ \; R(B_2^{(2)})\; &l=2, \end{cases} \qquad 
\text{and} \qquad
  R_-/G \cong \begin{cases} \; R(B_l^{(1)}) \; &l>2, \\ \; R(C_2^{(1)})\; &l=2, \end{cases}
\]
the mERS $(R_+,G)$ is either of type
\begin{enumerate}
\item $C_l^{(2,1)}$ or $C_l^{(2,2)}$ for $l>2$ and
\item $B_2^{(2,1)}$, $B_2^{(2,2)}$ or $B_2^{(2,2)\ast}$ for $l=2$,
\end{enumerate}
and the mERS $(R_-,G)$ is either of type 
\begin{enumerate}
\item $B_l^{(1,1)}$ or $B_l^{(1,2)}$ for $l>2$ and
\item $C_2^{(1,1)}$, $C_2^{(1,2)}$ or $C_2^{(1,1)\ast}$ for $l=2$. 
\end{enumerate}
\vskip 0.1in


\noindent{(1).\; \fbox{The case $R_+ = R(C_l^{(2,1)})$\; ($l>2)$}} \quad \\

Here, the index $[\Aut'(R_m,G):\Aut^+(R(C_l^{(2,1)}),G)]$ is $2$. 
\begin{enumerate}
\item[(I)] $R_-=R(B_l^{(1,1)})$. The coset $\Aut'(R_m,G)/\Aut^+(R(C_l^{(2,1)}),G)$ is represented by $ \{ \id, \tau_{\frac{1}{2}b} \}$.  $\tau_{\frac{1}{2}b}$ does not extend to automorphisms of $R(B_l^{(1,1)})$. 
\begin{enumerate}
\item[(i)] For $\varphi=\id$, we obtain the root system
\begin{align*}
R=
&(R(BC_l)_s+\Z a+\Z b) \cup (R(BC_l)_m+\Z a+\Z b) \cup (R(BC_l)_l+\Z a+2\Z b) \\
=
&R(\mathbf{BB_l^{\vee (1)}}).
\end{align*}
\item[(ii)] For $\varphi=\tau_{\frac{1}{2}b}$, we obtain the root system of type $BC_l^{(2,1)}$. 
\end{enumerate}
\item[(II)] $R_-=R(B_l^{(1,2)})$. Now, the coset $\Aut'(R_m,G)/\Aut^+(R(C_l^{(2,1)}),G)$ is represented by $\{ \id, \tau_{\frac{1}{2}b} \}$.  The map $\tau_{\frac{1}{2}b}=\id$ does not extend to an automorphism of $R(B_l^{(1,2)})$. 
\begin{enumerate}
\item[(i)] For $\varphi=\id$, we obtain the root system
\begin{align*}
R=
&(R(BC_l)_s+\Z a+\Z b) \cup (R(BC_l)_m+2\Z a+\Z b) \cup (R(BC_l)_l+2\Z a+2\Z b) \\
=
&R(\mathbf{BB_l^{\vee (2)}(2)}).
\end{align*}
\item[(ii)] For $\varphi=\tau_{\frac{1}{2}b}$, we get the root system of type $BC_l^{(2,2)}(2)$.
\end{enumerate}
\end{enumerate}

\vskip+0.1in

\medskip

\noindent{(2).\; \fbox{The case $R_+=R(C_l^{(2,2)})$\; ($l>2$)}} \quad \\

In this case, the index $[\Aut'(R_m,G):\Aut^+(R(C_l^{(2,2)}),G)]$ is $4$. 

\begin{enumerate}
\item[(I)] $R_-=R(B_l^{(1,1)})$. Now, the coset $\Aut'(R_m,G)/\Aut^+(R(C_l^{(2,2)}),G)$ is represented by the next $4$ elements: $\tau_\delta$ with $\delta \in \{0, \frac{1}{2}a, \frac{1}{2}b, \frac{1}{2}(a+b)\}$. Except for $\tau_0=\id$, they don't extend to automorphisms of $R(B_l^{(1,1)})$.
\begin{enumerate}
\item[(i)] For $\varphi=\id$, we obtain the root system
\begin{align*}
R=
&(R(BC_l)_s+\Z a+\Z b) \cup (R(BC_l)_m+\Z a+\Z b) \cup (R(BC_l)_l+2\Z a+2\Z b) \\
=
&R(\mathbf{BB_l^{\vee (2)}(1)}).
\end{align*}
\item[(ii)] For $\varphi=\tau_{\frac{1}{2}a}$, we get the root system
\begin{align*}
 R=
 &(R(BC_l)_s+\Z a+\Z b) \cup (R(BC_l)_m+\Z a+\Z b)\\
 \cup
&(R(BC_l)_l+(1+2\Z) a+2\Z b)= R(\mathbf{BC_l^{(2,2)\sigma}(1)}).
\end{align*}
\item[(iii)] For $\varphi=\tau_{\frac{1}{2}b}$, we obtain the root system of type $BC_l^{(2,2)}(1)$.
\item[(iv)] For $\varphi=\tau_{\frac{1}{2}(a+b)}$, we have $R\cong R(BC_l^{(2,2)}(1))$ via $\chi(({}^tT)^{-1})$. 
\end{enumerate}
\item[(II)] $R_-=R(B_l^{(1,2)})$. Here, the coset $\Aut'(R_m,G)/\Aut^+(R(C_l^{(2,2)}),G)$ is represented by the next $4$ elements: $\tau_\delta$ with $\delta \in \{0, a, \frac{1}{2}b, a+\frac{1}{2}b\}$. Only $\tau_0=\id$ and $\tau_a$ extend to automorphisms of $R(B_l^{(1,2)})$.
\begin{enumerate}
\item[(i)] For $\varphi=\id$, we obtain the root system
\begin{align*}
R=
&(R(BC_l)_s+\Z a+\Z b) \cup (R(BC_l)_m+2\Z a+\Z b) \cup (R(BC_l)_l+4\Z a+2\Z b) \\
=
& R(\mathbf{BB_l^{\vee (4)}}).
\end{align*}
\item[(ii)] For $\varphi=\tau_{\frac{1}{2}b}$, we get the root system of type $BC_l^{(2,4)}$.
\end{enumerate}
\end{enumerate}

\vskip+0.1in

\medskip

\noindent{(3).\; \fbox{The case $R_-=R(C_2^{(1,1)})$}} \quad \\

In this case, the index $[\Aut'(R_m,G):\Aut^+(R_-,G)]$ is $4$. 

\begin{enumerate}
\item[(I)] $R_+=R(B_2^{(2,1)})$. In this case, the coset $\Aut'(R_m,G)/\Aut^+(R_-,G)$ is represented by the next $4$ elements: $\tau_\delta$ with $\delta \in \{0, \frac{1}{2}a, \frac{1}{2}b, \frac{1}{2}(a+b)\}$. Only $\tau_0=\id$ and $\tau_{\frac{1}{2}a}$ extend to automorphisms of $R(B_2^{(2,1)})$.
\begin{enumerate}
\item[(i)] For $\varphi=\id$, we obtain the root system
\begin{align*}
R=&(R(BC_2)_s+\Z a+\Z b) \cup (R(BC_2)_m+\Z a+\Z b) \cup (R(BC_2)_l+\Z a+2\Z b)\\
=&R(\mathbf{BB_2^{\vee (1)}}).
\end{align*}
\item[(ii)] For $\varphi=\tau_{\frac{1}{2}b}$, we have $R \cong  R(BC_2^{(2,1)})$ via $\tau_{\frac{1}{2}b}$.
\end{enumerate}
\item[(II)] $R_+=R(B_2^{(2,2)})$. In this case, the coset $\Aut'(R_m,G)/\Aut^+(R_-,G)$ is represented by the next $4$ elements: $\tau_\delta$ with $\delta \in \{0, \frac{1}{2}a, \frac{1}{2}b, \frac{1}{2}(a+b)\}$. Only $\tau_0=\id$ extends to an automorphism of $R(B_2^{(2,2)})$.
\begin{enumerate}
\item[(i)] For $\varphi=\id$, we obtain the root system
\begin{align*}
R=&(R(BC_2)_s+\Z a+\Z b) \cup (R(BC_2)_m+\Z a+\Z b) \cup (R(BC_2)_l+2\Z a+2\Z b)\\
=&R(\mathbf{BB_2^{\vee (2)}(1)}).
\end{align*}
\item[(ii)] For $\varphi=\tau_{\frac{1}{2}a}$, we get $R \cong R(BC_2^{(2,2)\sigma}(1))$ via $\tau_{\frac{1}{2}a}$, where
\begin{align*}
R(\mathbf{BC_2^{(2,2)\sigma}(1)})=
&(R(BC_2)_s+\Z a+\Z b) \cup (R(BC_2)_m+\Z a+\Z b) \\
\cup 
&(R(BC_2)_l+(1+2\Z)a+2\Z b).
\end{align*} 
\item[(iii)] For $\varphi=\tau_{\frac{1}{2}b}$, we have $R \cong R(BC_2^{(2,2)}(1))$ via $\tau_{\frac{1}{2}b}$. 
\item[(iv)] For $\varphi=\tau_{\frac{1}{2}(a+b)}$, we have $R \cong R(BC_2^{(2,2)}(1))$ via $\chi(({}^tT)^{-1})\circ \tau_{\frac{1}{2}(a+b)}$. 
\end{enumerate}
\end{enumerate}

\vskip+0.1in

\medskip

\noindent{(4).\; \fbox{The case $R_-=R(C_2^{(1,1)\ast})$}} \quad \\

In this case, the index $[\Aut'(R_m,G):\Aut^+(R(C_2^{(1,1)\ast}),G)]$ is $4$. 

\begin{enumerate}
\item[(I)] $R_+=R(B_2^{(2,2)\ast})$. \\
Now, the coset $\Aut'(R_m,G)/\Aut^+(R(C_2^{(1,1)\ast}),G)$ is represented by the next $4$ elements: $\tau_\delta$ with $\delta \in \{ 0, \frac{1}{2}a, \frac{1}{2}b, \frac{1}{2}(a+b)\}$. All of them extend to automorphisms of $R(B_2^{(2,2)\ast})$, thus we obtain
\begin{align*}
R=&(R(BC_2)_s+\Z a+\Z b) \cup (R(BC_2)_m+\{\, ma+nb\, \vert\, mn \equiv 0\, [2]\, \}) \\
\cup &(R(BC_2)_l+2\Z a+2\Z b) = R(\mathbf{BB_2^{\vee (2)\ast}}).
\end{align*}
\end{enumerate}

\vskip+0.1in

\medskip

\noindent{(5).\; \fbox{The case $R_-=R(C_2^{(1,2)})$}} \quad \\

In this case, the index $[\Aut'(R_m,G):\Aut^+(R(C_2^{(1,2)}),G)]$ is $4$.  
Notice that the set $U_-/U_-(2)$ is represented by $\{ 1,{}^tT \}$.
\begin{enumerate}
\item[(I)] $R_+=R(B_2^{(2,1)})$. Here, the coset $\Aut'(R_m,G)/\Aut^+(R(C_2^{(1,2)}),G)$ is represented by $\{1, {}^tT\} \circ \{ \id, \tau_{\frac{1}{2}b}\}$.
Only $\tau_0=\id$ and $\chi({}^tT)$ extend to automorphisms of $R(B_2^{(2,1)})$. 
\begin{enumerate} 
\item[(i)] For $\varphi=\id$, we obtain the root system 
\begin{align*}
R=&(R(BC_2)_s+\Z a+\Z b) \cup (R(BC_2)_m+2\Z a+\Z b) \cup (R(BC_2)_l+2\Z a+2\Z b)\\
=&R(\mathbf{BB_2^{\vee(2)}(2)}).
\end{align*}
\item[(ii)] For $\varphi=\tau_{\frac{1}{2}b}$, we have  $R \cong R(BC_2^{(2,2)}(2))$ via $\tau_{\frac{1}{2}b}$. 
\end{enumerate}
\item[(II)] $R_+=R(B_2^{(2,2)})$. Now, the coset $\Aut'(R_m,G)/\Aut^+(R(C_2^{(1,2)}),G)$ is represented by $\{ \id, \tau_{b}\} \circ \{1,{}^tT\}$. 
All of them extend to automorphisms of $R(B_2^{(2,2)})$. Thus, we obtain
\begin{align*}
R=&(R(BC_2)_s+\Z a+\Z b) \cup (R(BC_2)_m+2\Z a+\Z b) \cup (R(BC_2)_l+4\Z a+2\Z b)\\
=&R(\mathbf{BB_2^{\vee(4)}}).
\end{align*}
\end{enumerate}





\subsection{$R/G$ of type $C^\vee C_l\; (l> 1)$}\label{subsect_CC-l}
In this case, since 
\[ R_+/G=R(C_l^{(1)}) \qquad \text{and} \qquad  R_-/G=R(B_l^{(2)}), \] 
the mERS $(R_+,G)$ is either of type \\
\centerline{$C_l^{(1,1)}, \, C_l^{(1,2)}$ or $C_l^{(1,1)\ast}$,}
and the mERS $(R_-,G)$ is either of type \\
\centerline{$B_l^{(2,1)}, \; B_l^{(2,2)}$ or $B_l^{(2,2)\ast}$. }


\noindent{(1).\; \fbox{The case $R_+=R(C_l^{(1,1)})$\; ($l>1$)}} \quad \\

In this case, $\Aut^+(R_+,G)=\Aut'(R_m,G)$ and it is sufficient to consider the gluing via trivial automorphism.
\begin{enumerate}
\item[(I)] $R_-=R(B_l^{(2,1)})$. Now, we obtain the root system
\begin{align*}
R=
&(R(BC_l)_s+\Z a+\Z b) \cup (R(BC_l)_m+\Z a+2\Z b) \cup (R(BC_l)_l+\Z a+2\Z b) \\
=
&R(\mathbf{C^\vee C_l^{(1)}}).
\end{align*}
\item[(II)] $R_-=R(B_l^{(2,2)})$. Here, we get the root system
\begin{align*}
R=
&(R(BC_l)_s+\Z a+\Z b) \cup (R(BC_l)_m+2\Z a+2\Z b) \cup (R(BC_l)_l+2\Z a+2\Z b) \\
=
&R(\mathbf{C^\vee C_l^{(2)}(2)}).
\end{align*}
\item[(III)] $R_-=R(B_l^{(2,2)\ast})$. We obtain the root system
\begin{align*}
R=
&(R(BC_l)_s+\{\, ma+nb\, \vert\, m,n \in \Z\, \text{s.t.}\, mn\equiv 0\, [2]\, \}) \\
\cup 
&(R(BC_l)_m+2\Z a+2\Z b) \cup (R(BC_l)_l+2\Z a+2\Z b)= R(\mathbf{C^\vee C_l^{(2)\ast_s}}).\end{align*}
\end{enumerate}

\vskip+0.1in

\medskip

\noindent{(2).\; \fbox{The case $R_+ =R(C_l^{(1,2)})$\; ($l>1$)}} \quad \\

In this case, the index $[\Aut'(R_m,G):\Aut^+(R(C_l^{(1,2)}),G)]$ is $4$. 
Notice that the set $U_-/U_-(2)$ is represented by $\{ 1, {}^tT\}$.  
\begin{enumerate}
\item[(I)] $R_-=R(B_l^{(2,1)})$. Here, the coset $\Aut'(R_m,G)/\Aut^+(R(C_l^{(1,2)}),G)$ is represented by $\{1, {}^tT \} \circ \{ \id, \tau_{\frac{1}{2}a} \}$. Except for $\tau_0=\id$,  they don't extend to automorphisms of $R(B_l^{(2,1)})$.\footnote{As $R_-=(R(BC_l)_s+\Z a+\Z b) \cup (R(BC_l)_m+\Z a+2\Z b)$, $R_+$ is realized as $(R(BC_l)_m+\Z a+\Z(2b)) \cup (R(BC_l)_l+2\Z a+\Z(2b))$. Since the matrix expression of $\chi({}^tT)$ with respect to the basis $(2b,a)$ is $\begin{pmatrix} 1 & 0 \\ 1 & 1\end{pmatrix}$, it follows that $\chi({}^tT)$ is not extendable to $\Aut^+(R_-,G)$.}
\begin{enumerate}
\item[(i)] For $\varphi=\id$, we obtain the root system
\begin{align*}
R=
&(R(BC_l)_s+\Z a+\Z b) \cup (R(BC_l)_m+\Z a+2\Z b) \cup (R(BC_l)_l+2\Z a+2\Z b) \\
=
&R(\mathbf{C^\vee C_l^{(2)}(1)}).
\end{align*}
\item[(ii)] For $\varphi=\tau_{\frac{1}{2}a}$, we get the root system
\begin{align*}
R=
&(R(BC_l)_s+\Z a+\Z b) \cup (R(BC_l)_m+\Z a+2\Z b) \\
\cup 
&(R(BC_l)_l+(1+2\Z) a+2\Z b)= R(\mathbf{BC_l^{(2,2)\sigma}(2)}).
\end{align*}
\item[(iii)] For $\varphi=\chi({}^tT)$, we obtain the root system
\begin{align*}
R=
&(R(BC_l)_s+\Z a+\Z b) \cup (R(BC_l)_m+\Z a+2\Z b) \\
\cup 
&(R(BC_l)_l+\{\, ma+2nb\, \vert\, m-n \equiv 0\, [2]\,\})=R(\mathbf{C^\vee C_l^{(2)\diamond}}).
\end{align*}
\item[(iv)] For $\varphi=\chi({}^tT)\circ \tau_{\frac{1}{2}a}$, we have $R \cong R(C^\vee C_l^{(2)\diamond})$ via $\tau_b$.
\end{enumerate}
\item[(II)] $R_-=R(B_l^{(2,2)})$.  Now, the coset $\Aut'(R_m,G)/\Aut^+(R(C_l^{(1,2)}),G)$ is represented by $\{ \id, \tau_{a} \} \circ \{1,  {}^tT \}$.  All of them extend to automorphisms of $R(B_l^{(2,2)})$. Thus we have
\begin{align*}
R=
&(R(BC_l)_s+\Z a+\Z b) \cup (R(BC_l)_m+2\Z a+2\Z b) \cup (R(BC_l)_l+4\Z a+2\Z b) \\
=
&R(\mathbf{C^\vee C_l^{(4)}}).
\end{align*}
\item[(III)] $R_-=R(B_l^{(2,2)\ast})$.  The coset $\Aut'(R_m,G)/\Aut^+(R(C_l^{(1,2)}),G)$ is represented by $ \{1, {}^tT \}\circ \{ \id, \tau_{a} \}$.  Only $\tau_0=\id$ and $\chi({}^tT)$ extend to automorphisms of $R(B_l^{(2,2)\ast})$. 
\begin{enumerate}
\item[(i)] For $\varphi=\id$, we obtain the root system
\begin{align*}
R=
&(R(BC_l)_s+\{\, ma+nb\, \vert\, m,n \in \Z\, \text{s.t.}\, mn\equiv 0\, [2]\, \}) \\
\cup
&(R(BC_l)_m+2\Z a+2\Z b) \cup (R(BC_l)_l+4\Z a+2\Z b)= R(\mathbf{C^\vee C_l^{(4)\ast_0}}).
\end{align*}
\item[(ii)] For $\varphi=\tau_{a}$, we have $R\cong R(C^\vee C_l^{(4)\ast_1})$ via $\tau_{a+b}$, where 
\begin{align*}
R(\mathbf{C^\vee C_l^{(4)\ast_1}})=
&(R(BC_l)_s+\{\, ma+nb\, \vert\, m,n \in \Z\, \text{s.t.}\, (m-1)(n-1)\equiv 0\, [2]\, \}) \\
\cup
&(R(BC_l)_m+2\Z a+2\Z b) \cup (R(BC_l)_l+4\Z a+2\Z b).
\end{align*}
\end{enumerate}
\end{enumerate}

\vskip+0.1in


\noindent{(3).\; \fbox{The case $R_+ =R(C_l^{(1,1)\ast})$\; ($l>1$)}} \quad \\

In this case, the index $[\Aut'(R_m,G):\Aut^+(R(C_l^{(1,1)\ast}),G)]$ is $4$. 
\begin{enumerate}
\item[(I)] $R_-=R(B_l^{(2,1)})$. Here, the coset $\Aut'(R_m,G)/\Aut^+(R(C_l^{(1,1)\ast}),G)$ is represented by the next $4$ elements: $\tau_\delta$ with $\delta \in \{0, \frac{1}{2}a, b, \frac{1}{2}a+b\}$. Only $\tau_0=\id$ and $\tau_b$ extend to automorphisms of $R(B_l^{(2,1)})$. 
\begin{enumerate}
\item[(i)] For $\varphi=\id$, we obtain the root system
\begin{align*}
R=
&(R(BC_l)_s+\Z a+\Z b) \cup (R(BC_l)_m+\Z a+2\Z b) \\
\cup
&(R(BC_l)_l+\{\, ma+2nb\, \vert\, m,n \in \Z\, \text{s.t.}\, mn \equiv 0\, [2]\, \})
= R(\mathbf{C^\vee C_l^{(1)\ast_0}}).
\end{align*}
\item[(ii)] For $\varphi=\tau_{\frac{1}{2}a+b}$, we get the root system
\begin{align*}
R=
&(R(BC_l)_s+\Z a+\Z b) \cup (R(BC_l)_m+\Z a+2\Z b) \\
\cup
&(R(BC_l)_l+\{\, ma+2nb\, \vert\, m,n \in \Z\, \text{s.t.}\, (m-1)(n-1) \equiv 0\, [2]\, \}) \\
= 
&R(\mathbf{C^\vee C_l^{(1)\ast_1}}).
\end{align*}
\end{enumerate}
\item[(II)] $R_-=R(B_l^{(2,2)})$. Now, the coset $\Aut'(R_m,G)/\Aut^+(R(C_l^{(1,1)\ast}),G)$ is represented by the next $4$ elements: $\tau_\delta$ with $\delta \in \{0, a, b, a+b\}$. All of them extend to automorphisms of $R(B_l^{(2,2)})$. Thus we obtain
\begin{align*}
R=
&(R(BC_l)_s+\Z a+\Z b) \cup (R(BC_l)_m+2\Z a+2\Z b) \\
\cup
&(R(BC_l)_l+\{\, 2ma+2nb\, \vert\, m,n \in \Z\, \text{s.t.}\, mn\equiv 0\, [2]\, \})
= R(\mathbf{C^\vee C_l^{(2)\ast_l}}).
\end{align*}
\item[(III)] $R_-=R(B_l^{(2,2)\ast})$.  The coset $\Aut'(R_m,G)/\Aut^+(R(C_l^{(1,1)\ast}),G)$ is represented by the next $4$ elements: $\tau_\delta$ with $\delta \in \{0, a, b, a+b\}$. Except for $\tau_0=\id$, they don't extend to automorphisms of $R(B_l^{(2,2)\ast})$.   
\begin{enumerate}
\item[(i)] For $\varphi=\id$, we get the root system
\begin{align*}
R=
&(R(BC_l)_s+\{\, ma+nb\, \vert\, m,n \in \Z\, \text{s.t.}\, mn\equiv 0\, [2]\, \}) \\
\cup
&(R(BC_l)_m+2\Z a+2\Z b) \\
\cup
&(R(BC_l)_l+\{\, 2ma+2nb\, \vert\, m,n \in \Z\, \text{s.t.}\, mn\equiv 0\, [2]\, \})
= R(\mathbf{C^\vee C_l^{(2)\ast_0}}).
\end{align*}
\item[(ii)] For $\varphi=\tau_a$, we obtain the root system
\begin{align*}
R=
&(R(BC_l)_s+\{\, ma+nb\, \vert\, m,n \in \Z\, \text{s.t.}\, mn\equiv 0\, [2]\, \}) \\
\cup
&(R(BC_l)_m+2\Z a+2\Z b) \\
\cup
&(R(BC_l)_l+\{\, 2ma+2nb\, \vert\, m,n \in \Z\, \text{s.t.}\, (m-1)n\equiv 0\, [2]\, \})
= R(\mathbf{C^\vee C_l^{(2)\ast_{1'}}}).
\end{align*}
\item[(iii)] For $\varphi=\tau_{a+b}$, we get  the root system
\begin{align*}
R=
&(R(BC_l)_s+\{\, ma+nb\, \vert\, m,n \in \Z\, \text{s.t.}\, mn\equiv 0\, [2]\, \}) \\
\cup
&(R(BC_l)_m+2\Z a+2\Z b) \\
\cup
&(R(BC_l)_l+\{\, 2ma+2nb\, \vert\, m,n \in \Z\, \text{s.t.}\, (m-1)(n-1)\equiv 0\, [2]\, \}) \\
= 
&R(\mathbf{C^\vee C_l^{(2)\ast_1}}).
\end{align*}
\item[(iv)] For $\varphi=\tau_b$, we have $R\cong R(C^\vee C_l^{(2)\ast_1})$ via $\tau_a \circ \chi(({}^tT)^{-1})$. 
\end{enumerate}
\end{enumerate}
\vskip 0.2in

\begin{rem}\label{rem_CC2}
\begin{enumerate}
\item
The isomorphism classes of the root systems
\begin{align*}
R =
&\big( R(BC_l)_s+\{ \, ma+nb\, \vert\, m,n \in \Z, \; (m-p')(n-q')\equiv 0\, [2]\, \}\big) \\
\cup
&\big( R(BC_l)_m+2\Z a+2\Z b \big)  \\
\cup
&\big( R(BC_l)_l+ \{ \, 2ma+2nb\, \vert\, 1\leq i\leq l, \; m,n \in \Z, \\
& \phantom{\big( R(BC_l)_l+ \{ \, 2ma+2nb\, \vert\,}(m-p-p')(n-q-q')\equiv 0\, [2]\, \}\big)
\end{align*}
for any $p,q, p', q' \in \{0,1\}$
are independent of $p'$ and $q'$; we shall denote such root systems by $\mathbf{C^\vee C_l^{(2)\ast_{p,q}}}$. 
\item The name in the above list corresponds as follows:
\begin{center}
\begin{tabular}{|c|c|c|} \hline
$C^\vee C^{(2)\ast_0}$  & 
$C^\vee C^{(2)\ast_1}$  & 
$C^\vee C^{(2)\ast_{1'}}$ \\ \hline
$C^\vee C^{(2)\ast_{0,0}}$ & 
$C^\vee C^{(2)\ast_{1,1}}$ & 
$C^\vee C^{(2)\ast_{1,0}}$ \\ \hline
\end{tabular}
\end{center}
\item The mERSs $(C^\vee C^{(2)\ast_{0,1}}, \R a)$ and $(C^\vee C^{(2)\ast_{1,1}}, \R a)$ are isomorphic. As a corollary, one sees that the three root systems of type $C^\vee C_l^{(2)\ast_{p,q}}$ with $(p,q) \neq (0,0)$ are all isomorphic, since $C^\vee C_l^{(2)\ast_{1,0}}$ and $C^\vee C_l^{(2)\ast_{0,1}}$ are isomorphic via $a \leftrightarrow b$.
\end{enumerate}
\end{rem}

\subsection{Rank $1$ cases}\label{sect_rank-1} 
Here, we briefly explain how one can classify the mERSs $(R,G)$ of rank $1$ whose quotient $R/G$ are non-reduced. We shall realise such root systems in the $\R$-vector space $F=\R \vep \oplus \R a \oplus \R b$ equipped with the symmetric bilinear form $I$ that satisfies $I(\vep, \vep)=1$ and $\R a \oplus \R b$ coincides with the radical of $I$. 

A reduced mERS of rank $1$ is either of type $A_1^{(1,1)}$ or $A_1^{(1,1)\ast}$.  Contrary to higher rank cases, we cannot proceed the classification
of the mERSs of rank $1$ by ``gluing'' two root systems, since there is no middle length root in this case. Nevertheless, we shall classify these root systems with an analogous idea. First fix a realisation of $R_s$ in $(F,I)$ as follows:
$R_s=\{ \pm \vep\}+L$ where
\[ L:=\begin{cases} \; \Z a\oplus \Z b \qquad & R_s:\,\text{of type}\, A_{1}^{(1,1)} \\
          \; \{\, ma+nb\, \vert\, m,n \in \Z\, \text{s.t.}\, mn\equiv 0\, [2]\, \} \qquad 
          & R_s:\, \text{of type}\, A_1^{(1,1)\ast}. \end{cases}
\]
Remark that the lattice $\langle L\rangle$, generated by the elements of $L$ by definition, is always $\rad(I)_{\Z}=\Z a\oplus \Z b$. Next, realize $R_l$ in the form
\[ R_l=\{ \pm (2\vep+m_0)\}+M, \]
where $m_0 \in \rad(I)$ and $M \subset \rad(I)$ is a subset containing $0$, assuming that $R_l$ is a root system of type $A_1^{(1,1)}$ or $A_1^{(1,1)\ast}$. Set $\gamma=2\vep+m_0$. Now, we give several properties of these data. 
\begin{enumerate}
\item For $m, m' \in M$, we have
\[ r_{\gamma+m'}(\gamma+m)=-\gamma+m-2m' \]
which implies that $M+\langle 2M\rangle \subset M$. In particular, 
$\langle 2M \rangle \subset M$ as $M$ contains $0$ and $M=-M$ by setting $m'=m$ in the above formula. 
\item Assume that $R=R_s \cup R_l$ is a root system. By assumption, 
\[ r_\vep(\gamma)=\gamma-4\vep=-\gamma+2m_0 \in -\gamma+M \]
which implies that $2m_0 \in M$.  Now, for any $m \in M$ and $l \in L$, 
\[ r_{\gamma+m}(-\vep+l)=\vep+m_0+m+l \in \vep+L, \]
that is, $m_0+M+L \subset L$. As both $M$ and $L$ contains $0$,  $m_0 \in L$. In particular, it follows that $M+L_0 \subset L_0$ as $L_0=\langle L \rangle$ which implies  $M \subset L_0$. 
\item Now, for any $l \in L$ and $m \in M$, 
\[ r_{\vep+l}\circ r_\vep (\gamma+m)=\gamma+m+4l, \]
which implies that $M+4L \subset M$, i.e., $M+4L_0 \subset M$. In particular, as $M$ contains $0$, it follows that $M \supset 4L_0$.
\end{enumerate}
{Summarising} the above arguments, we have
\begin{lemma}\label{lemma_rank-1}
\begin{enumerate}
\item $M$ is a subset of $L_0$ satisfying
\begin{enumerate}
\item[(i)] $M+4L_0 \subset M$, in particular, $4L_0 \subset M$, and 
\item[(ii)] $M+\langle 2M \rangle \subset M$.
\end{enumerate}
\item $m_0 \in L$ satisfies
\begin{enumerate}
\item[(i)] $2m_0 \in M$ and
\item[(ii)] $m_0+M+L \subset L$. 
\end{enumerate}
\end{enumerate}
\end{lemma}

With this technical lemma, we shall analyse each case separately. 
\subsubsection{$\mathbf{R_s}$ $\&$ $\mathbf{R_l}$ of type $\mathbf{A_1^{(1,1)}}$}
In this case, $M$ is a sub-lattice of $L=L_0$ of full rank containing $4L$. 
By direct computation, such lattices can be classified as follows:

\begin{enumerate}
\item The case $[L:M]=1$.  In this case, $M_1=L$.  
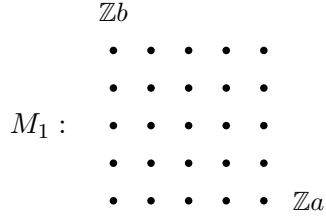
\begin{center}
\begin{tikzpicture}[scale=0.5]
\draw (4.5,0) node[right] {\small $\Z a$}; \draw (0,4.5) node[above] {\small $\Z b$};

\foreach \x in {0,1,...,4} \fill (\x,0) circle (0.1) ; 
\foreach \x in {0,1,...,4} \fill (\x,1) circle (0.1) ; 
\foreach \x in {0,1,...,4} \fill (\x,2) circle (0.1) ; 
\foreach \x in {0,1,...,4} \fill (\x,3) circle (0.1) ; 
\foreach \x in {0,1,...,4} \fill (\x,4) circle (0.1) ; 

\draw (-2,2) node[]  {$M_1:$};

\end{tikzpicture}
\end{center}
\item The case $[L:M]=2$.  In this case, 
\begin{enumerate}
\item[i)] $M_{2,1}=2\Z a\oplus \Z b$, 
\item[ii)] $M_{2,2}=\Z a \oplus 2\Z b$, 
\item[iii)] $M_{2,3}=\Z(a+b)\oplus \Z(a-b)=\Z (a+b)\oplus 2\Z b=2\Z a \oplus \Z(b+a)$.
\end{enumerate}
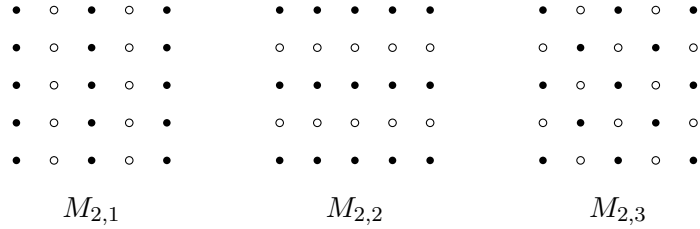
\begin{center}
\vskip 0.05in
\begin{tikzpicture}[scale=0.5]

\foreach \y in {0,1,...,4} \fill (0,\y) circle (0.1) ; 
\foreach \y in {0,1,...,4} \draw (1,\y) circle (0.1) ; 
\foreach \y in {0,1,...,4} \fill (2,\y) circle (0.1) ; 
\foreach \y in {0,1,...,4} \draw (3,\y) circle (0.1) ; 
\foreach \y in {0,1,...,4} \fill (4,\y) circle (0.1) ; 

\draw (2,-2) node[above]  {$M_{2,1}$};

\foreach \x in {7,8,...,11} \fill (\x,0) circle (0.1) ; 
\foreach \x in {7,8,...,11} \draw (\x,1) circle (0.1) ; 
\foreach \x in {7,8,...,11} \fill (\x,2) circle (0.1) ; 
\foreach \x in {7,8,...,11} \draw (\x,3) circle (0.1) ; 
\foreach \x in {7,8,...,11} \fill (\x,4) circle (0.1) ; 

\draw (9,-2) node[above]  {$M_{2,2}$};

\foreach \x in {14,16,18} \fill (\x,0) circle (0.1) ; 
\foreach \x in {15,17} \draw (\x,0) circle (0.1) ; 
\foreach \x in {14,16,18} \draw (\x,1) circle (0.1) ; 
\foreach \x in {15,17} \fill (\x,1) circle (0.1) ; 
\foreach \x in {14,16,18} \fill (\x,2) circle (0.1) ; 
\foreach \x in {15,17} \draw (\x,2) circle (0.1) ; 
\foreach \x in {14,16,18} \draw (\x,3) circle (0.1) ; 
\foreach \x in {15,17} \fill (\x,3) circle (0.1) ; 
\foreach \x in {14,16,18} \fill (\x,4) circle (0.1) ; 
\foreach \x in {15,17} \draw (\x,4) circle (0.1) ; 

\draw (16,-2) node[above]  {$M_{2,3}$};
\end{tikzpicture}
\end{center}

\item The case $[L:M]=4$. In this case, 
\begin{enumerate}
\item[i)] $M_{4,1}=4\Z a \oplus \Z b$,
\item[ii)] $M_{4,2}=\Z a \oplus 4\Z b$,
\item[iii)] $M_{4,3}=2\Z a \oplus 2\Z b=2L$, 
\item[iv)] $M_{4,4}=\Z(a+b)\oplus 4\Z b=4\Z a\oplus \Z(b+a)$,
\item[v)] $M_{4,5}=\Z(-a+b)\oplus 4\Z b=4\Z a \oplus \Z(-a+b)$,
\item[vi)] $M_{4,6}=\Z(2a+b)\oplus 2\Z b=4\Z a+\Z (b+2a)$,
\item[vii)] $M_{4,7}=\Z(a+2b) \oplus 4\Z b=2\Z a \oplus \Z(2b+a) $.
\end{enumerate}
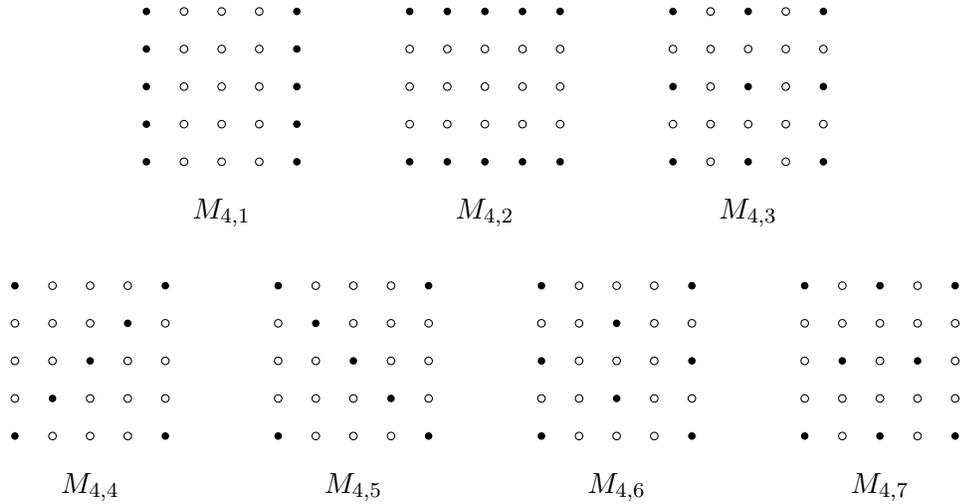
\begin{center}
\vskip 0.05in
\begin{tikzpicture}[scale=0.37]

\foreach \y in {0,1,...,4} \fill (0,\y) circle (0.1) ; 
\foreach \y in {0,1,...,4} \draw (1,\y) circle (0.1) ; 
\foreach \y in {0,1,...,4} \draw (2,\y) circle (0.1) ; 
\foreach \y in {0,1,...,4} \draw (3,\y) circle (0.1) ; 
\foreach \y in {0,1,...,4} \fill (4,\y) circle (0.1) ; 

\draw (2,-2) node[above]  {$M_{4,1}$};

\foreach \x in {7,8,...,11} \fill (\x,0) circle (0.1) ; 
\foreach \x in {7,8,...,11} \draw (\x,1) circle (0.1) ; 
\foreach \x in {7,8,...,11} \draw (\x,2) circle (0.1) ; 
\foreach \x in {7,8,...,11} \draw (\x,3) circle (0.1) ; 
\foreach \x in {7,8,...,11} \fill (\x,4) circle (0.1) ; 

\draw (9,-2) node[above]  {$M_{4,2}$};

\foreach \x in {14,16,18} \fill (\x,0) circle (0.1) ; 
\foreach \x in {15,17} \draw (\x,0) circle (0.1) ; 
\foreach \x in {14,...,18} \draw (\x,1) circle (0.1) ; 
\foreach \x in {14,16,18} \fill (\x,2) circle (0.1) ; 
\foreach \x in {15,17} \draw (\x,2) circle (0.1) ; 
\foreach \x in {14,...,18} \draw (\x,3) circle (0.1) ; 
\foreach \x in {14,16,18} \fill (\x,4) circle (0.1) ; 
\foreach \x in {15,17} \draw (\x,4) circle (0.1) ; 

\draw (16,-2) node[above]  {$M_{4,3}$};
\end{tikzpicture}
\end{center}

\begin{center}
\begin{tikzpicture}[scale=0.37]

\foreach \y in {0,1,...,4} \fill (\y,\y) circle (0.1) ; 
\foreach \y in {0,2,3,4} \draw (1,\y) circle (0.1) ; 
\foreach \y in {0,1,3,4} \draw (2,\y) circle (0.1) ; 
\foreach \y in {0,1,2,4} \draw (3,\y) circle (0.1) ; 
\foreach \y in {1,2,3} \draw (0,\y) circle (0.1);
\foreach \y in {1,2,3} \draw (4,\y) circle (0.1);
\fill (0,4) circle (0.1) ; \fill (4,0) circle (0.1) ; 

\draw (2,-2) node[above]  {$M_{4,4}$};

\foreach \x in {7,8,...,11} \fill (\x,11-\x) circle (0.1) ; 
\foreach \x in {7,8,9,11} \draw (\x,1) circle (0.1) ; 
\foreach \x in {7,8,10,11} \draw (\x,2) circle (0.1) ; 
\foreach \x in {7,9,10,11} \draw (\x,3) circle (0.1) ; 
\foreach \x in {8,9,10} \draw (\x,4) circle (0.1) ; 
\foreach \x in {8,9,10} \draw (\x,0) circle (0.1);
\fill (7,0) circle (0.1); \fill (11,4) circle (0.1);

\draw (9,-2) node[above]  {$M_{4,5}$};

\foreach \x in {14,18} \fill (\x,0) circle (0.1) ; 
\foreach \x in {15,16,17} \draw (\x,0) circle (0.1) ; 
\foreach \x in {14,15,17,18} \draw (\x,1) circle (0.1) ; 
\foreach \x in {14,18} \fill (\x,2) circle (0.1) ; 
\foreach \x in {15,16,17} \draw (\x,2) circle (0.1) ; 
\foreach \x in {14,15,17,18} \draw (\x,3) circle (0.1) ; 
\foreach \x in {14,18} \fill (\x,4) circle (0.1) ; 
\foreach \x in {15,16,17} \draw (\x,4) circle (0.1) ; 
\foreach \y in {1,3} \fill (16,\y) circle (0.1);

\draw (16,-2) node[above]  {$M_{4,6}$};

\foreach \x in {21,23,25} \fill (\x,0) circle (0.1) ; 
\foreach \x in {22,24} \draw (\x,0) circle (0.1) ; 
\foreach \x in {21,...,25} \draw (\x,1) circle (0.1) ; 
\foreach \x in {21,23,25} \draw (\x,2) circle (0.1) ; 
\foreach \x in {22,24} \fill (\x,2) circle (0.1) ; 
\foreach \x in {21,...,25} \draw (\x,3) circle (0.1) ; 
\foreach \x in {21,23,25} \fill (\x,4) circle (0.1) ; 
\foreach \x in {22,24} \draw (\x,4) circle (0.1) ; 

\draw (23,-2) node[above]  {$M_{4,7}$};

\end{tikzpicture}
\end{center}

\item The case $[L:M]=8$.
In this case, 
\begin{enumerate}
\item[i)] $M_{8,1}=4\Z a\oplus 2\Z b$,
\item[ii)] $M_{8,2}=2\Z a \oplus 4\Z b$,
\item[iii)] $M_{8,3}=2\Z(a+b)\oplus 2\Z(a-b)=2 \Z (a+b)\oplus 4\Z b=4\Z a \oplus 2\Z(b+a)$.
\end{enumerate}
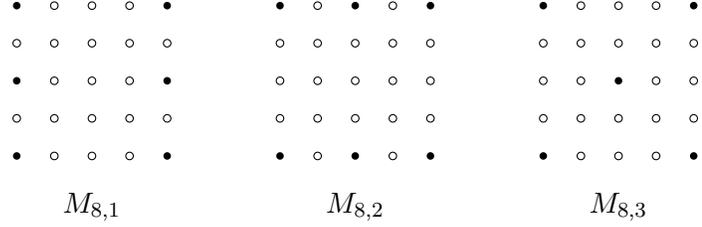
\begin{center}
\vskip 0.05in
\begin{tikzpicture}[scale=0.5]

\foreach \y in {0,2,4} \fill (0,\y) circle (0.1) ; 
\foreach \y in {1,3} \draw (0,\y) circle (0.1) ; 
\foreach \y in {0,1,...,4} \draw (1,\y) circle (0.1) ; 
\foreach \y in {0,1,...,4} \draw (2,\y) circle (0.1) ; 
\foreach \y in {0,1,...,4} \draw (3,\y) circle (0.1) ; 
\foreach \y in {0,2,4} \fill (4,\y) circle (0.1) ; 
\foreach \y in {1,3} \draw (4,\y) circle (0.1) ; 

\draw (2,-2) node[above]  {$M_{8,1}$};

\foreach \x in {7,9,11} \fill (\x,0) circle (0.1) ; 
\foreach \x in {8,10} \draw (\x,0) circle (0.1) ; 
\foreach \x in {7,8,...,11} \draw (\x,1) circle (0.1) ; 
\foreach \x in {7,8,...,11} \draw (\x,2) circle (0.1) ; 
\foreach \x in {7,8,...,11} \draw (\x,3) circle (0.1) ; 
\foreach \x in {7,9,11} \fill (\x,4) circle (0.1) ; 
\foreach \x in {8,10} \draw (\x,4) circle (0.1) ;

\draw (9,-2) node[above]  {$M_{8,2}$};

\foreach \x in {14,18} \fill (\x,0) circle (0.1) ; 
\foreach \x in {15,16,17} \draw (\x,0) circle (0.1) ; 
\foreach \x in {14,15,...,18} \draw (\x,1) circle (0.1) ; 
\fill (16,2) circle (0.1) ; 
\foreach \x in {14,15,17,18} \draw (\x,2) circle (0.1) ; 
\foreach \x in {14,15,...,18} \draw (\x,3) circle (0.1) ; 
\foreach \x in {14,18} \fill (\x,4) circle (0.1) ; 
\foreach \x in {15,16,17} \draw (\x,4) circle (0.1) ; 

\draw (16,-2) node[above]  {$M_{8,3}$};
\end{tikzpicture}
\end{center}

\item The case $[L:M]=16$. In this case, we have $M_{16}=4L$. 
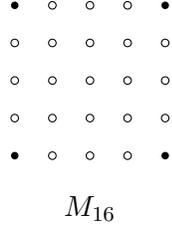
\begin{center}
\vskip 0.05in
\begin{tikzpicture}[scale=0.5]

\foreach \x in {0,4} \fill (\x,0) circle (0.1) ; 
\foreach \x in {1,2,3} \draw (\x,0) circle (0.1);
\foreach \x in {0,1,...,4} \draw(\x, 1) circle (0.1) ; 
\foreach \x in {0,1,...,4} \draw (\x,2) circle (0.1) ; 
\foreach \x in {0,1,...,4} \draw (\x,3) circle (0.1) ; 
\foreach \x in {0,4} \fill (\x,4) circle (0.1) ; 
\foreach \x in {1,2,3} \draw (\x,4) circle (0.1);

\draw (2,-2) node[above]  {$M_{16}$};
\end{tikzpicture}
\end{center}

\end{enumerate}
Here, each figure describes the set $\{\, ma+nb\, \vert\, m,n \in \{0,1,2,3,4\}\, \}$ with the rules: $\white \in L \setminus M$ and $\black \in M$, the leftmost and lowest circle indicates $0 \in L$, the next right circle to $0$ indicates $a$ and the next up circle to $0$ indicates $b$.

Let us now classify the possible lattices $M$, together with possible $m_0$'s, up to $SL(\rad_{\Z}(I))$-action on $L$ fixing $\Z a$. Fix a basis $(b,a)$ of $\rad_{\Z}(I)$ to identify $SL(\rad_{\Z}(I))$ with $SL_2(\Z)$. \\

It is clear that $M_{2,3}=\chi({}^tT)(M_{2,1})$, $M_{4,4}=\chi({}^tT)(M_{4,1})$, $M_{4,5}=\chi(({}^tT)^{-1})(M_{4,1})$, $M_{4,6}=\chi(({}^tT)^2)(M_{4,1})$, and $M_{8,3}=\chi({}^tT)(M_{8,1})$.  This implies the following classification: 

\begin{enumerate}
\item The case $[L:M]=1$.   In this case, $M=M_1=L$, and 
we obtain the root system 
\[ R=(\{ \pm\vep \}+\Z a+\Z b)\cup (\{ \pm 2\vep \}+\Z a+\Z b)=R(\mathbf{BCC_1^{(1)}}).
\]
\item The case $[L:M]=2$.  In this case, we have
\begin{enumerate}
\item[(I)] $M=M_{2,1}=2\Z a\oplus \Z b$.
\begin{enumerate}
\item[(i)] For $m_0=0$,  
\[ R=(\{ \pm\vep \}+\Z a+\Z b) \cup (\{ \pm 2\vep \}+2\Z a+\Z b) = R(\mathbf{BCC_1^{(2)}(2)}).
\]
\item[(ii)] For $m_0=a$,  
\[ R=(\{ \pm\vep \}+\Z a+\Z b) \cup (\{ \pm 2\vep \}+(1+2\Z) a+\Z b) =R(\mathbf{BC_1^{(1,2)}}).
\]
\end{enumerate}
\item[(II))] $M=M_{2,2}=\Z a \oplus 2\Z b$. 
\begin{enumerate}
\item[(i)] For $m_0=0$,  
\[ R=(\{ \pm\vep \}+\Z a+\Z b) \cup (\{ \pm 2\vep \}+\Z a+2\Z b) =R(\mathbf{C^\vee C_1^{(1)}}).
\]
\item[(ii)] For $m_0=b$,  we obtain the root system of type $\mathbf{BC_1^{(2,1)}}$.
\end{enumerate}
\end{enumerate}
\item The case $[L:M]=4$. In this case, we have
\begin{enumerate}
\item[(I)] $M=M_{4,1}=4\Z a \oplus \Z b$. In this case, 
\[ R=(\{ \pm\vep \}+\Z a+\Z b) \cup (\{ \pm 2\vep \}+4\Z a+\Z b) =R(\mathbf{BCC_1^{(4)}}).
\]
\item[(II)] $M=M_{4,2}=\Z a \oplus 4\Z b$. In this case,  
\[ R=(\{ \pm\vep \}+\Z a+\Z b) \cup (\{ \pm 2\vep \}+\Z a+4\Z b) =R(\mathbf{C^\vee BC_1^{(1)}}).\]
\item[(III)] $M=M_{4,3}=2\Z a \oplus 2\Z b=2L$.
\begin{enumerate}
\item[(i)] For $m_0=0$, 
\[ R=(\{ \pm\vep \}+\Z a+\Z b) \cup (\{ \pm 2\vep \}+2\Z a+2\Z b) =R(\mathbf{C^\vee C_1^{(2)}(2)}).
\]
\item[(ii)] For $m_0=a$,  
\[ R=(\{ \pm\vep \}+\Z a+\Z b) \cup (\{ \pm 2\vep \}+(1+2\Z) a+2\Z b) 
=R(\mathbf{BC_1^{(2,2)\sigma}(2)}).
\]
\item[(iii)] For $m_0=b$, we obtain the root system of type
$\mathbf{BC_1^{(2,2)}(2)}$.
\item[(iv)] For $m_0=b+a$, we have
$\chi({}^tT)(R(BC_1^{(2,2)}(2)))$. 
\end{enumerate}
\item[(IV)] $M=M_{4,7}=\Z(a+2b) \oplus 4\Z b=2\Z a \oplus \Z(2b+a) $. 
\begin{enumerate}
\item[(i)] For $m_0=0$, 
\[ R=(\{ \pm\vep \}+\Z a+\Z b) \cup (\{ \pm 2\vep \}+2\Z a+\Z (2b+a)) 
=R(\mathbf{C^\vee C_1^{(2)\diamond}}).
\]
\item[(ii)] For $m_0=a$, we have 
$\tau_b(R(C^\vee C_1^{(2)\diamond}))$.
\end{enumerate}
\end{enumerate}
\item The case $[L:M]=8$.
In this case, we have
\begin{enumerate}
\item[(I)] $M=M_{8,1}=4\Z a\oplus 2\Z b$.
\begin{enumerate}
\item[(i)] For $m_0=0$,  
\[ R=(\{ \pm\vep \}+\Z a+\Z b) \cup (\{ \pm 2\vep \}+4\Z a+2\Z b) =R(\mathbf{C^\vee C_1^{(4)}}).
\]
\item[(ii)] For $m_0=b$, we obtain the root system of type $\mathbf{BC_1^{(2,4)}}$.
\end{enumerate}
\item[(II)] $M=M_{8,2}=2\Z a \oplus 4\Z b$.
\begin{enumerate}
\item[(i)] For $m_0=0$, 
\[ R=(\{ \pm\vep \}+\Z a+\Z b) \cup (\{ \pm 2\vep \}+2\Z a+4\Z b) =R(\mathbf{C^\vee BC_1^{(2)}(2)}).
\]
\item[(ii)] For $m_0=a$,  
\[ R=(\{ \pm\vep \}+\Z a+\Z b) \cup (\{ \pm 2\vep \}+(1+2\Z) a+4\Z b) =R(\mathbf{BC_1^{(4,2)}}). 
\]
\end{enumerate}
\end{enumerate}
\item The case $[L:M]=16$. In this case, we have $M=M_{16}=4L$ and 
\[ R=(\{ \pm\vep \}+\Z a+\Z b) \cup (\{ \pm 2\vep \}+4\Z a+4\Z b) =R(\mathbf{C^\vee BC_1^{(4)}}).
\]
\end{enumerate}

\begin{rem}\label{rem_m_0} If $m_0 \in 2L$, by the translation $\tau_{\frac{1}{2}m_0}$ $($cf. \eqref{def_translation}$)$, we see that it is isomorphic to the case $m_0=0$. Hence, the above choice of $m_0$ is given up  to the translation by $2L$. 
\end{rem}

\subsubsection{$\mathbf{R_s}$ of type $\mathbf{A_1^{(1,1)}}$ and $\mathbf{R_l}$ of type $\mathbf{A_1^{(1,1)\ast}}$}
In this case, as $M$ is contained in $L_0$, it is clear that $\langle 2M \rangle \subset 2L_0$. By Lemma \ref{lemma_rank-1}, it follows that the set $M$ is stable under the translation of $\langle 2M \rangle +4L_0$ and the latter is a lattice satisfying
\[ 2L_0 \supset \langle 2M \rangle +4L_0 \supset 4L_0. \]
Hence, the index of the sub-lattice $M':=\langle 2M \rangle +4L_0$ of $2L_0$ is either $1,2$ or $4$. Thus, up to the $SL_2(\Z)$-action on $L$ fixing $\Z a$, it is equivalent either to
$2L_0, 2 \Z a+4\Z b, 4 \Z a+2\Z b$ or $4L_0$. We classify $R$ in each case:
\begin{enumerate}
\item The case $M'=2L_0$. We have
\[ M=\{\, ma+nb\, \vert\, m,n \in \Z \, \text{s.t.}\, mn\equiv 0\, [2]\, \}. \]
\begin{enumerate}
\item[(I)] For $m_0=0$, 
\begin{align*}
R=
&(\{ \pm\vep \}+\Z a+\Z b) \\
\cup
&(\{ \pm 2\vep \}+\{\, ma+nb\, \vert\, m,n \in \Z\, \text{s.t.}\, mn \equiv 0\, [2]\, \})
= R(\mathbf{BCC_1^{(1)\ast_0}}).
\end{align*}
\item[(II)] For $m_0=a$, we have $\chi({}^tT)(R(BCC_1^{(1)\ast_0}))$.
\item[(III)] For $m_0=b$, 
\begin{align*}
R=
&(\{ \pm\vep \}+\Z a+\Z b) \\
\cup 
&(\{ \pm 2\vep \}+\{\, ma+nb\, \vert\, m,n \in \Z\, \text{s.t.}\, m(n-1) \equiv 0\, [2]\, \})
= R(\mathbf{BCC_1^{(1)\ast_{0'}}}).
\end{align*}
\item[(IV)] For $m_0=a+b$, 
\begin{align*}
R=
&(\{ \pm\vep \}+\Z a+\Z b) \\
\cup 
&(\{ \pm 2\vep \}+\{\, ma+nb\, \vert\, m,n \in \Z\, \text{s.t.}\, (m-1)(n-1) \equiv 0\, [2]\, \}) \\
= 
&R(\mathbf{BC_1^{(1,1)\ast}}).
\end{align*}
\end{enumerate}
\item The case $M'=2\Z a+4\Z b$. In this case, we have
\[ M=\{\, ma+2nb\, \vert\, m,n \in \Z \, \text{s.t.}\, mn\equiv 0\, [2]\, \}. \]
\begin{enumerate}
\item[(I)] For $m_0=0$, 
\begin{align*}
R=
&(\{ \pm\vep \}+\Z a+\Z b) \\
\cup 
&(\{ \pm 2\vep \}+\{\, ma+2nb\, \vert\, m,n \in \Z\, \text{s.t.}\, mn \equiv 0\, [2]\, \})
= R(\mathbf{C^\vee C_1^{(1)\ast_0}}).
\end{align*}
\item[(II)] For $m_0=a+2b$, 
\begin{align*}
R=
&(\{ \pm\vep \}+\Z a+\Z b) \\
\cup 
&(\{ \pm 2\vep \}+\{\, ma+2nb\, \vert\, m,n \in \Z\, \text{s.t.}\, (m-1)(n-1) \equiv 0\, [2]\, \}) \\
= 
&R(\mathbf{C^\vee C_1^{(1)\ast_1}}).
\end{align*}
\end{enumerate}
\item The case $M'=4\Z a+2\Z b$. In this case, we have
\[ M=\{\, 2ma+nb\, \vert\, m,n \in \Z \, \text{s.t.}\, mn\equiv 0\, [2]\, \}. \]
\begin{enumerate}
\item[(I)] For $m_0=0$, 
\begin{align*}
R=
&(\{ \pm\vep \}+\Z a+\Z b) \\
\cup 
&(\{ \pm 2\vep \}+\{\, 2ma+nb\, \vert\, m,n \in \Z\, \text{s.t.}\, mn \equiv 0\, [2]\, \})
= R(\mathbf{BCC_1^{(2)\ast_0}}).
\end{align*}
\item[(II)] For $m_0=2a+b$, 
\begin{align*}
R=
&(\{ \pm\vep \}+\Z a+\Z b) \\
\cup 
&(\{ \pm 2\vep \}+\{\, 2ma+nb\, \vert\, m,n \in \Z\, \text{s.t.}\, (m-1)(n-1) \equiv 0\, [2]\, \}) \\
= 
&R(\mathbf{BCC_1^{(2)\ast_1}}).
\end{align*}
\end{enumerate}
\item The case $M'=4L_0$. In this case, we have
\[ M=\{\, 2ma+2nb\, \vert\, m,n \in \Z \, \text{s.t.}\, mn\equiv 0\, [2]\, \} \]
and  
\begin{align*}
R=
&(\{ \pm\vep \}+\Z a+\Z b) \\
\cup 
&(\{ \pm 2\vep \}+\{\, 2ma+2nb\, \vert\, m,n \in \Z\, \text{s.t.}\, mn \equiv 0\, [2]\, \})
= R(\mathbf{C^\vee C_1^{(2)\ast_l}}).
\end{align*}
\end{enumerate}


\subsubsection{$\mathbf{R_s}$ of type $\mathbf{A_1^{(1,1)\ast}}$ and $\mathbf{R_l}$  of type $\mathbf{A_1^{(1,1)}}$}
This case is dual to the previous section, so let us just compute their duals and give the results. We write only those cases where we obtain new root systems, and keep the same numbering system as in the previous section. 

\begin{enumerate}
\item 
\begin{enumerate}
\item[(I)] We get the root system 
\begin{align*}
R=
&(\{ \pm\vep \}+\{\, ma+nb\, \vert\, m,n \in \Z\, \text{s.t.}\, mn \equiv 0\, [2]\, \})\\
\cup 
&(\{ \pm 2\vep \}+4\Z a+4\Z b) = R(\mathbf{C^\vee BC_1^{(4)\ast_0}}).
\end{align*}
\item[(III)] We obtain the root system 
\begin{align*}
R=
 &(\{ \pm\vep \}+\{\, ma+nb\, \vert\, m,n \in \Z\, \text{s.t.}\, m(n-1) \equiv 0\, [2]\, \})\\
\cup 
&(\{ \pm 2\vep \}+4\Z a+4\Z b) = R(\mathbf{C^\vee BC_1^{(4)\ast_{0'}}}).
\end{align*}
\item[(IV)] For $m_0=a+b$, we  get the root system 
\begin{align*}
R=
 &(\{ \pm\vep \}+\{\, ma+nb\, \vert\, m,n \in \Z\, \text{s.t.}\, (m-1)(n-1) \equiv 0\, [2]\, \})\\
\cup 
&(\{ \pm 2\vep \}+4\Z a+4\Z b) = R(\mathbf{BC_1^{(4,4)\ast}}).
\end{align*}
\end{enumerate}
\item 
\begin{enumerate}
\item[(I)] We obtain the root system 
\begin{align*}
R=
 &(\{ \pm\vep \}+\{\, ma+nb\, \vert\, m,n \in \Z\, \text{s.t.}\, mn \equiv 0\, [2]\, \})\\
\cup 
&(\{ \pm 2\vep \}+4\Z a+2\Z b) = R(\mathbf{C^\vee C_1^{(4)\ast_0}}).
\end{align*}
\item[(II)] We get the root system 
\begin{align*}
R=
&(\{ \pm\vep \}+\{\, ma+nb\, \vert\, m,n \in \Z\, \text{s.t.}\, (m-1)(n-1) \equiv 0\, [2]\, \})\\
\cup 
&(\{ \pm 2\vep \}+4\Z a+2\Z b) = R(\mathbf{C^\vee C_1^{(4)\ast_1}}).
\end{align*}
\end{enumerate}
\item 
\begin{enumerate}
\item[(I)] We obtain the root system 
\begin{align*}
R=
&(\{ \pm\vep \}+\{\, ma+nb\, \vert\, m,n \in \Z\, \text{s.t.}\, mn \equiv 0\, [2]\, \})\\
\cup 
&(\{ \pm 2\vep \}+2 \Z a+4\Z b) = R(\mathbf{C^\vee BC_1^{(2)\ast_0}}).
\end{align*}
\item[(II)] We get the root system 
\begin{align*}
R=
&(\{ \pm\vep \}+\{\, ma+nb\, \vert\, m,n \in \Z\, \text{s.t.}\, (m-1)(n-1) \equiv 0\, [2]\, \})\\
\cup 
&(\{ \pm 2\vep \}+2\Z a+4\Z b) = R(\mathbf{C^\vee BC_1^{(2)\ast_1}}).
\end{align*}
\end{enumerate}
\item We obtain the root system 
\begin{align*}
R=
&(\{ \pm\vep \}+\{\, ma+nb\, \vert\, m,n \in \Z\, \text{s.t.}\, mn \equiv 0\, [2]\, \}) \\
\cup 
&(\{ \pm 2\vep \}+2\Z a+2\Z b) = R(\mathbf{C^\vee C_1^{(2)\ast_s}}).
\end{align*}
\end{enumerate}

\subsubsection{$\mathbf{R_s}$ $\&$ $\mathbf{R_l}$ of type $\mathbf{A_1^{(1,1)\ast}}$}
In this case, fix 
\[ L=\{\, ma+nb\, \vert\, m,n \in \Z\, \text{s.t.} \, mn \equiv 0\, [2]\, \} \]
as in the head of \S \ref{sect_rank-1}. By Lemma \ref{lemma_rank-1} (2) (ii), both $m_0$ and $M$ are contained in $2L_0$. Hence the only possibility for $M$ is $2L$, that is,
\[ M=\{\, 2ma+2nb\, \vert\, m,n \in \Z\, \text{s.t.}\, mn \equiv 0\, [2]\, \}. \]
Thus, we have the next classification:
\begin{enumerate}
\item For $m_0=0$, we have the root system
\begin{align*}
R=
&(\{ \pm\vep \}+\{\, ma+nb\, \vert\, m,n \in \Z\, \text{s.t.} \, mn \equiv 0\, [2]\, \})  \\
\cup
&(\{ \pm 2\vep \}+\{\, 2ma+2nb\, \vert\, m,n \in \Z\, \text{s.t.}\, mn \equiv 0\, [2]\, \})
= R(\mathbf{C^\vee C_1^{(2)\ast_0}}).
\end{align*}
\item For $m_0=2a$, 
\begin{align*}
R=
&(\{ \pm\vep \}+\{\, ma+nb\, \vert\, m,n \in \Z\, \text{s.t.} \, mn \equiv 0\, [2]\, \})  \\
\cup
&(\{ \pm 2\vep \}+\{\, 2ma+2nb\, \vert\, m,n \in \Z\, \text{s.t.}\, (m-1)n \equiv 0\, [2]\, \})
= R(\mathbf{C^\vee C_1^{(2)\ast_{1'}}}).
\end{align*}
\item For $m_0=2b$, we get $R=\tau_a\circ \chi({}^tT)((R(C^\vee C_1^{(2)\ast_1})))$. 
\item For $m_0=2(a+b)$, we obtain
\begin{align*}
R=
&(\{ \pm\vep \}+\{\, ma+nb\, \vert\, m,n \in \Z\, \text{s.t.} \, mn \equiv 0\, [2]\, \})  \\
\cup
&(\{ \pm 2\vep \}+\{\, 2ma+2nb\, \vert\, m,n \in \Z\, \text{s.t.}\, (m-1)(n-1) \equiv 0\, [2]\, \}) \\
= 
&R(\mathbf{C^\vee C_1^{(2)\ast_1}}).
\end{align*}
\end{enumerate}



\section{More on new mERSs}\label{sect:more-mERS}

\subsection{Isomorphic root systems}\label{sect:isom-root-sys}
Among mERSs, there exist several pairs, say $(R_1,G_1)$ and $(R_2,G_2)$ who are not isomorphic as mERS but are isomorphic as ERS. In this subsection, we provide those  where at least one of the two mERSs has a non-reduced affine quotient. 

\begin{thm} \label{thm:isom-red-mERS-nred_root-sys} Among the reduced $BC_l$-type marked elliptic root systems with non-reduced affine quotient, we have the following isomorphisms
$($as root systems$):$
\begin{align*}
&R(BC_l^{(1,2)}) \cong R(BC_l^{(2,1)}), \qquad R(BC_l^{(4,2)}) \cong R(BC_l^{(2,4)}), \\
&R(BC_l^{(2,2)\sigma}(i)) \cong R(BC_l^{(2,2)}(i)) \qquad i \in \{1,2\}.
\end{align*}

\end{thm}
\medskip

\begin{proof}They follow from the isomorphism $a \leftrightarrow b$.
\end{proof}

Therefore, this theorem implies that there exist $6$ infinite series of non-isomorphic reduced mERSs whose full quotient is of type $BC_l$: 
$R(BC_l^{(1,2)}) \cong R(BC_l^{(2,1)})$, 
$R(BC_l^{(2,2)\sigma}(1)) \cong R(BC_l^{(2,2)}(1))$, $R(BC_l^{(2,2)\sigma}(2)) \cong R(BC_l^{(2,2)}(2))$, $R(BC_l^{(4,2)}) \cong R(BC_l^{(2,4)})$, $R(BC_l^{(1,1)\ast})$ and $R(BC_l^{(4,4)\ast})$.\\

\smallskip

\begin{thm} \label{thm:isom-nred-mERS-nred_root-sys}Among the above $34$ infinite series and $1$ exceptional type of non-reduced marked elliptic root systems, we have the next isomorphisms $($as root systems$):$
\begin{enumerate}
\item Via the isomorphism $a \leftrightarrow b:$
\begin{align*}
&R(BCC_l^{(2)}(1)) \cong R(BB_l^{\vee (1)}), \qquad 
  R(C^\vee BC_l^{(2)}(1)) \cong R(BB_l^{\vee (4)}), \\
&R(BCC_l^{(2)}(2)) \cong R(C^\vee C_l^{(1)}), \qquad 
  R(C^\vee BC_l^{(2)}(2)) \cong R(C^\vee C_l^{(4)}), \\
&R(BCC_l^{(2)\ast_p}) \cong R(C^\vee C_l^{(1)\ast_p}), \quad
  R(C^\vee BC_l^{(2)\ast_p}) \cong R(C^\vee C_l^{(4)\ast_p}) \quad (p \in \{0,1\}), \\
&\phantom{ABCDEFGHI} R(BB_l^{\vee (2)}(2)) \cong R(C^\vee C_l^{(2)}(1)).
\end{align*}
\item Via the isomorphism $a \mapsto a+b, \; b \mapsto b:$
\begin{align*}
&R(BCC_l^{(1)\ast_{0'}}) \cong R(BCC_l^{(1)\ast_0}), \qquad
  R(C^\vee BC_l^{(4)\ast_{0'}}) \cong R(C^\vee BC_l^{(4)\ast_0}), \\
&\phantom{ABCDEFGHI} R(C^\vee C_l^{(2)\ast_{1'}}) \cong R(C^\vee C_l^{(2)\ast_1}).
\end{align*}
\item Via an exotic isomorphism $\&$ $a \leftrightarrow b:$
\[ R(BCC_l^{(4)}) \cong R(C^\vee BC_l^{(1)}) \cong R(C^\vee C_l^{(2)\diamond}). \]
\end{enumerate}
\end{thm}
As a consequence, there are only $20$ infinite series and $1$ exceptional type of non-isomorphic elliptic root systems.
\vskip 0.2in

\begin{proof}

To see the isomorphism of elliptic root systems induced from the switching $a \leftrightarrow b$, it is convenient to look at their gluing data. 
Let $(R,G)$ be a mERS with non-reduced affine quotient $R/G$. As we see, the root system $R$ can be expressed as $R_- \amalg_\varphi R_+$. Denote the linear isomorphism of $F=\bigoplus_{i=1}^l \R \vep_i \oplus \R a \oplus \R b$ defined by $\vep_i \mapsto \vep_i$ ($1\leq i\leq l$) and $a \mapsto b, \; b \mapsto a$ by $\Phi$. By definition, it follows that $\Phi(R)=\Phi(R_-) \amalg_{\Phi \circ \varphi \circ \Phi^{-1}} \Phi(R_+)$. Here are the tables of such data: 
\begin{enumerate}
\item The case $\varphi=\id$:

\begin{center}
\scalebox{0.87}{
\begin{tabular}{|c||c|c|c|c|c|} \hline 
& $C_l^{(2,2)}$ & $C_l^{(2,1)}$ & $C_l^{(1,2)}$ & $C_l^{(1,1)}$ & $C_l^{(1,1)\ast}$ \\ \hline \hline
$B_l^{(1,1)}$ &
$BB_l^{\vee (2)}(1)$ & $BB_l^{\vee (1)}$ & $BCC_l^{(2)}(1)$ & $BCC_l^{(1)}$
& $BCC_l^{(1)\ast_0}$ \\ \hline
$B_l^{(1,2)}$ &
$BB_l^{\vee (4)}$ & $BB_l^{\vee (2)}(2)$ & $BCC_l^{(4)}$ & $BCC_l^{(2)}(2)$
& $BCC_l^{(2)\ast_0}$ \\ \hline
$B_l^{(2,1)}$ & $C^\vee BC_l^{(2)}(1)$ & $C^\vee BC_l^{(1)}$ & $C^\vee C_l^{(2)}(1)$ 
& $C^\vee C_l^{(1)}$ & $C^\vee C_l^{(1)\ast_0}$\\ \hline
$B_l^{(2,2)}$ &
$C^\vee BC_l^{(4)}$ & $C^\vee BC_l^{(2)}(2)$ & $C^\vee C_l^{(4)}$ & $C^\vee C_l^{(2)}(2)$ 
& $C^\vee C_l^{(2)\ast_l}$ \\ \hline
$B_l^{(2,2)\ast}$ &
$C^\vee BC_l^{(4)\ast_0}$ & $C^\vee BC_l^{(2)\ast_0}$ & $C^\vee C_l^{(4)\ast_0}$ & $C^\vee C_l^{(2)\ast_s}$ & $C^\vee C_l^{(2)\ast_0}$ \\ \hline
\end{tabular}%
}
\end{center}
\vskip 0.2in

\item The case $\varphi$ is of the form $\tau_\delta$ for some $\delta \in \rad(I)$:

\begin{center}

\resizebox{0.85\textwidth}{!}{%
\begin{tabular}{|c||c|c|c|c|} \hline 
&  $C_l^{(2,2)}$ & $C_l^{(2,1)}$ & $C_l^{(1,2)}$ & $C_l^{(1,1)\ast}$ \\ \hline \hline
$B_l^{(1,1)}$ &
$\begin{matrix} BC_l^{(2,2)}(1) \\ BC_l^{(2,2)\sigma}(1) \end{matrix}$
& $BC_l^{(2,1)}$  & $BC_l^{(1,2)}$ 
& $\begin{matrix} BCC_l^{(1)\ast_{0'}} \\ BC_l^{(1,1)\ast} \end{matrix}$ 
\\ \hline
$B_l^{(1,2)}$ &
$BC_l^{(2,4)}$  & $BC_l^{(2,2)}(2)$ & & $BCC_l^{(2)\ast_1}$
\\ \hline
$B_l^{(2,1)}$ &
$BC_l^{(4,2)}$ & 
& $BC_l^{(2,2)\sigma}(2)$ 
& $C^\vee C_l^{(1)\ast_1}$ 
\\ \hline
$B_l^{(2,2)\ast}$ 
& $\begin{matrix} C^\vee BC_l^{(4)\ast_{0'}} \\ BC_l^{(4,4)\ast} \end{matrix}$  
& $C^\vee BC_l^{(2)\ast_1}$ & $C^\vee C_l^{(4)\ast_1}$
& $\begin{matrix} C^\vee C_l^{(2)\ast_{1'}} \\ C^\vee C_l^{(2)\ast_1} \end{matrix}$ 
\\ \hline
\end{tabular}%
}
\end{center}
\item The case $\varphi \in SL_2(\Z)$. Here we only have $C^\vee C_l^{(2)\diamond}$.
\end{enumerate}

\begin{rem} In the above tables, the duality $R \leftrightarrow R^\vee$ of root systems can be seen via the symmetry with respect to the diagonal. 
\end{rem}
From the above tables, we obtain the isomorphisms for reduced root systems as follows : 
\begin{align*}
 &R(BC_l^{(1,2}) \cong R(BC_l^{(2,1)}), \qquad R(BC_l^{(4,2)}) \cong R(BC_l^{(2,4)}), \\
 &R(BC_l^{(2,2)\sigma}(1)) \cong R(BC_l^{(2,2)}(1)), \qquad 
    R(BC_l^{(2,2)\sigma}(2)) \cong R(BC_l^{(2,2)}(2)). 
\end{align*}
The quotient root systems of the left hand side give $BCC_l, C^\vee BC_l, BB_l^\vee$ and $C^\vee C_l$ respectively, whereas those of the right hand side are always $BC_l^{(2)}$, which is a reduced affine root system.  
Similarly, for non-reduced root systems, we obtain
\begin{align*}
& R(BCC_l^{(2)}(1)) \cong R(BB_l^{\vee\, (1)}), \qquad 
    R(C^\vee BC_l^{(2)}(1)) \cong R(BB_l^{\vee (4)}), \\
& R(BCC_l^{(2)}(2)) \cong R(C^\vee C_l^{(1)}),  \qquad
   R(C^\vee BC_l^{(2)}(2)) \cong R(C^\vee C_l^{(4)}), \\
& R(BCC_l^{(2)\ast_p}) \cong R(C^\vee C_l^{(1)\ast_p}), \qquad 
   R(C^\vee BC_l^{(2)\ast_p}) \cong R(C^\vee C_l^{(4)\ast_p}), \\
& R(BCC_l^{(4)}) \cong R(C^\vee BC_l^{(1)}), \qquad
    R(BB_l^{\vee\, (2)}(2)) \cong R(C^\vee C_l^{(2)}(1)).
\end{align*}
\begin{rem} Strictly speaking, the above argument using the construction of a root system via gluing procedure cannot be applied to the rank $1$ case, but we can show that any of the above isomorphisms are valid for the rank $1$ root systems by case by case checking.
\end{rem}
The next isomorphisms explain the symbol $\ast_{i'}$ ($i=0,1$):
\begin{align*}
& R(BCC_l^{(1)\ast_{0'}}) \cong R(BCC_l^{(1)\ast_0}), \qquad 
R(C^\vee BC_l^{(4)\ast_{0'}}) \cong R(C^\vee BC_l^{(4)\ast_0}), \\
& R(C^\vee C_l^{(2)\ast_{1'}}) \cong R(C^\vee C_l^{(2)\ast_{1}}).
\end{align*}
Indeed, these isomorphisms are essentially given by $a \mapsto a+b, \; b \mapsto b$. 
Let us show that the root systems $R(BCC_l^{(2)\ast_0})$ and $R(BCC_l^{(2)\ast_1})$ are not isomorphic. If they were, there would exist an isometry $\Phi$ of $(F,I)$, to which they belong, such that $\Phi(R(BCC_l^{(2)\ast_0}))=R(BCC_l^{(2)\ast_1}))$. The restriction of the linear map $\Phi$ to their subsets of short roots and long roots impose an incompatible condition which shows that such a linear map cannot exist. Likewise, one can show that  the root systems $R(C^\vee BC_l^{(2)\ast_0})$ and $R(C^\vee BC_l^{(2)\ast_1})$ are not isomorphic.  Details are left to the reader.

Finally, there is one exotic case:
 \begin{align*}
 R(C^\vee C_l^{(2)\diamond})=
 &\left(\{ \pm \vep_i  \pm \vep_j\, \vert\, 1\leq i<j\leq l\, \}+\Z(a+2b)+ 2\Z b\right) \\
 & \cup \left( \{ \pm \vep_i\, \vert\, 1\leq i\leq l\, \}+\Z(a+2b)+ \Z b\right) \\
 & \cup \left( \{ \pm 2\vep_i\, \vert\, 1\leq i\leq l\, \}+\Z(a+2b)+4\Z b\right) \\
 \cong 
& R(C^\vee BC_l^{(1)}) \cong R(BCC_l^{(4)}),
 \end{align*}
thus, we have the next isomorphisms (as root systems):
\begin{align*}
R(C^\vee C_l^{(2)\diamond}) \cong R(C^\vee BC_l^{(1)}) \cong R(BCC_l^{(4)}).
\end{align*}
\end{proof}

\begin{rem}\label{rem-comparison-Azam-etc}
S. Azam et al. $($\cite{Azam2002} and \cite{AzamKhaliliYousofzadeh2005}$)$ obtained a list of possible ERSs via combinatorial method. 
\begin{enumerate}
\item The next table shows the {correspondence} between the two descriptions$:$ \\
\begin{center}
\resizebox{.85\textwidth}{!}{%
\begin{tabular}{|c||c|} \hline
$BC$-triple & mERS $(l\geq 2)$ \\ \hline \hline
$(\Lambda, \Lambda, \Lambda^{(1)}+\sigma_1)$ & $BC_l^{(2,1)}$ and $BC_l^{(1,2)}$  \\ \hline
$(\Lambda, \Lambda, S(0,\tilde{\sigma_1}, \tilde{\sigma_2})+\sigma_1+\sigma_2)$ & $BC_l^{(1,1)\ast}$ \\ \hline
$(\Lambda, \Lambda^{(1)}, 2\Lambda^{(1)}+\sigma_2)$ & $BC_l^{(4,2)}$ and $BC_l^{(2,4)}$ \\ \hline
$(S(0,\tilde{\sigma_1}, \tilde{\sigma_2}),2\Lambda, 4\Lambda+2\sigma_1+2\sigma_2)$
& $BC_l^{(4,4)\ast}$ \\ \hline
$(\Lambda, \Lambda, 2\Lambda+\sigma_1)$ & $BC_l^{(2,2)}(1)$ and $BC_l^{(2,2)\sigma}(1)$ \\ \hline
$(\Lambda, \Lambda^{(1)}, 2\Lambda+\sigma_2)$ & $BC_l^{(2,2)}(2)$ and $BC_l^{(2,2)\sigma}(2)$ 
\\ \hline 
\end{tabular}}
\end{center}

\begin{center}
\resizebox{.85\textwidth}{!}{%
\begin{tabular}{|c||c|} \hline
$BC$-triple & mERS $(l\geq 2)$ \\
\hline
$(\Lambda, \Lambda, \Lambda)$ & $BCC_l^{(1)}$ \\ \hline
$(\Lambda, \Lambda,S(0,\tilde{\sigma_1},\tilde{\sigma_2}))$ &
$BCC_l^{(1)\ast_0}$ and $BCC_l^{(1)\ast_{0'}}$ \\ \hline
$(\Lambda, \Lambda, \Lambda^{(1)})$ & $BCC_l^{(2)}(1)$ and $BB_l^{\vee (1)}$ \\ \hline
$(\Lambda, \Lambda^{(1)}, \Lambda^{(1)})$ & $BCC_l^{(2)}(2)$ and $C^\vee C_l^{(1)}$ \\ \hline
$(\Lambda, \Lambda^{(1)}, S(0,2\tilde{\sigma_1}, \tilde{\sigma_2}))$
& $BCC_l^{(2)\ast_0}$ and $C^\vee C_l^{(1)\ast_0}$ \\ \hline
$(\Lambda, \Lambda^{(1)}, S(0,2\tilde{\sigma_1}, \tilde{\sigma_2})+\sigma_2)$
& $BCC_l^{(2)\ast_1}$ and $C^\vee C_l^{(1)\ast_1}$ \\ \hline
$(\Lambda, \Lambda^{(1)}, 4\Z \sigma_1\oplus \Z \sigma_2)$ &
$BCC_l^{(4)}, \, C^\vee BC_l^{(1)}$ and $C^\vee C_l^{(2)\diamond}$ \\ \hline
$(\Lambda, \Lambda^{(1)}, 2\Lambda^{(1)})$ & $C^\vee BC_l^{(2)}(1)$ and $BB_l^{\vee (4)}$ \\ \hline
$(\Lambda, 2\Lambda, 2\Lambda^{(1)})$ & $C^\vee BC_l^{(2)}(2)$ and $C^\vee C_l^{(4)}$ \\ \hline
$(S(0,\tilde{\sigma_1}, \tilde{\sigma_2}), 2\Lambda, 2\Lambda^{(1)})$ & $C^\vee BC_l^{(2)\ast_0}$ and $C^\vee C_l^{(4)\ast_0}$ \\ \hline
$(S(0,\tilde{\sigma_1}, \tilde{\sigma_2}), 2\Lambda, 2\Lambda^{(1)}+2\sigma_1)$ & $C^\vee BC_l^{(2)\ast_1}$ and $C^\vee C_l^{(4)\ast_1}$ \\ \hline
$(\Lambda, 2\Lambda, 4\Lambda)$ & $C^\vee BC_l^{(4)}$ \\ \hline
$(S(0,\tilde{\sigma_1}, \tilde{\sigma_2}), 2\Lambda, 4\Lambda)$ & $C^\vee BC_l^{(4)\ast_0}$ and $C^\vee BC_l^{(4)\ast_{0'}}$\\ \hline
$(\Lambda, \Lambda, 2\Lambda)$ & $BB_l^{\vee (2)}(1)$ \\ \hline
$(\Lambda, \Lambda^{(1)}, 2\Lambda)$ & $BB_l^{\vee (2)}(2)$ and $C^\vee C_l^{(2)}(1)$\\ \hline
$(\Lambda, 2\Lambda, 2\Lambda)$ & $C^\vee C_l^{(2)}(2)$\\ \hline
$(S(0,\tilde{\sigma_1}, \tilde{\sigma_2}), 2\Lambda, 2\Lambda)$ & $C^\vee C_l^{(2)\ast_s}$\\ \hline
$(\Lambda, 2\Lambda, 2S(0,\tilde{\sigma_1}, \tilde{\sigma_2}))$ & $C^\vee C_l^{(2)\ast_l}$\\ \hline
$(S(0,\tilde{\sigma_1}, \tilde{\sigma_2}), 2\Lambda, 2S(0,\tilde{\sigma_1}, \tilde{\sigma_2}))$ & $C^\vee C_l^{(2)\ast_0}$\\ \hline
$(S(0,\tilde{\sigma_1}, \tilde{\sigma_2}), 2\Lambda, 2S(0,\tilde{\sigma_1}, \tilde{\sigma_2})+2\sigma_1)$ & $C^\vee C_l^{(2)\ast_1}$ and $C^\vee C_l^{(2)\ast_{1'}}$\\ \hline
$(\Lambda, S(0,\tilde{\sigma_1}, \tilde{\sigma_2}),2\Lambda)$ & $BB_2^{\vee (2)\ast}$ \\ \hline
\end{tabular}}
\end{center}
\vskip 0.1in
Here, a $BC$-triple is a triple $(S,L,E)$ consisting of subsets of $\Lambda:=\rad_{\Z}(I)=\Z a\oplus \Z b$ which corresponds to the root system $R=(R(BC_l)_s+S) \cup (R(BC_l)_m+L) \cup (R(BC_l)_l+E)$.\footnote{For these admitted rank $1$ case, it suffices to think of the pair $(S,E)$, called $BC$-pair.}  \\
\item It is clear that each ERS corresponds to a $BC$-triple, but it is not evident whether any $BC$-triple gives an ERS. Indeed, in \cite{AzamKhaliliYousofzadeh2005}, the last $BC$-triple in the above table was defined for any finite rank $l \geq 2$ whereas the corresponding root system exists only if $l=2$. 
\end{enumerate}
\end{rem}
\newpage
\subsection{On affine quotient root systems}\label{sect_two-affine-quot}

Set $G_a=\R a$ and $G_b=\R b$. Let $R$ be a reduced elliptic root system. It is known that the pair $(R/G_b, R/G_a)$ determines uniquely the type of the root system $R$. 
The next table shows that this is not the case for non-reduced elliptic root systems
with given two affine root systems $R/G_b$ and $R/G_a$: \\

\begin{center}
\resizebox{.985\textwidth}{!}{%
\begin{tabular}{|c||c|c|c|c|c|} \hline 
 & $BCC_l$ & $C^\vee BC_l$ & $BB_l^\vee$ & $C^\vee C_l$ & $BC_l^{(2)}$ \\ \hline \hline
$BCC_l$ &
{\small $\begin{matrix} 
BCC_l^{(1)}  \\ 
BCC_l^{(1)\ast_i}\; i \in \{0, 0'\} \\
BC_l^{(1,1)\ast} 
\end{matrix}$}
&
$C^\vee BC_l^{(1)}$ 
&
$BB_l^{\vee (1)}$
&
{\small $\begin{matrix}
C^\vee C_l^{(1)}  \\
C^\vee C_l^{(1)\ast_i} \; i \in \{0,1\} \\
C^\vee C_l^{(2)\diamond} 
\end{matrix}$}

&
$BC_l^{(2,1)}$ \\ \hline
$C^\vee BC_l$ &
$BCC_l^{(4)}$
&
{\small $\begin{matrix}
C^\vee BC_l^{(4)}  \\
C^\vee BC_l^{(4)\ast_i} \; i \in \{0,0'\} \\
BC_l^{(4,4)\ast} 
\end{matrix}$}
&
$BB_l^{\vee (4)}$
&
{\small $\begin{matrix}
C^\vee C_l^{(4)}  \\
C^\vee C_l^{(4)\ast_i} \; i \in \{0,1\} \\
(C^\vee C_l^{(2)\diamond})^\vee \; (\sharp)
\end{matrix}$}

&
$BC_l^{(2,4)}$ \\ \hline
$BB_l^\vee$ &
$BCC_l^{(2)}(1)$
&
$C^\vee BC_l^{(2)}(1)$
&
{\small $\begin{matrix} 
BB_l^{\vee (2)}(1) \\ 
BB_2^{\vee (2)\ast} 
\end{matrix}$}

&
$C^\vee C_l^{(2)}(1)$
&
$BC_l^{(2,2)}(1)$ \\ \hline
$C^\vee C_l$ &
{\small $\begin{matrix} 
BCC_l^{(2)}(2)  \\
BCC_l^{(2)\ast_i} \; i\in \{0,1\}
\end{matrix}$} 

&
{\small $\begin{matrix}
C^\vee BC_l^{(2)}(2)  \\
C^\vee BC_l^{(2)\ast_i} \; i \in \{0,1\}
\end{matrix}$}

&
$BB_l^{\vee (2)}(2)$
&
{\small $\begin{matrix}
C^\vee C_l^{(2)}(2) & \\
C^\vee C_l^{(2)\ast_s} & C^\vee C_l^{(2)\ast_l} \\
C^\vee C_l^{(2)\ast_i} & i \in \{0,1,1'\} 
\end{matrix}$}
&
$BC_l^{(2,2)}(2)$ \\ \hline
$BC_l^{(2)}$ &
$BC_l^{(1,2)}$
&
$BC_l^{(4,2)}$ 
&
$BC_l^{(2,2)\sigma}(1)$
&
$BC_l^{(2,2)\sigma}(2)$ 
&
($\flat$)
\\ \hline
\end{tabular}%
}
\end{center}
\vskip 0.2in

Here $(\flat)$ is the root system defined by
\begin{align*}
R=
&(R(BC_l)_s+\Z a+\Z b) 
\cup 
(R(BC_l)_m+\Z a+\Z b) \\
\cup
&(R(BC_l)_l+(1+2\Z)a+(1+2\Z)b).
\end{align*}
This root system is isomorphic to $R(BC_l^{(2,2)}(1))$ since
\[
 (1+2\Z)a+(1+2\Z)b 
=
2\Z a+(1+2\Z)(b+a) 
\]
which implies that 
\[ R=\chi(T)(R(BC_l^{(2,2)}(1))). \]
Nevertheless, both $R/G_a$ and $R/G_b$ are isomorphic to $BC_l^{(2)}$. \\

Let us now describe the dual root system $(\sharp)$ of $R(C^\vee C_l^{(2)\diamond})$. By definition, it is given by
\begin{align*}
R((C^\vee C_l^{(2)\diamond})^\vee):=R(C^\vee C_l^{(2)\diamond})^\vee=
&(R(BC_l)_s+\{\, ma'+nb'\, \vert\, m-n\equiv 0\, [2]\, \}) \\
\cup
&(R(BC_l)_m+2\Z a'+2\Z b') \\
\cup 
&(R(BC_l)_l+4\Z a'+2\Z b'),
\end{align*}
where we set $a'=\frac{1}{2}a$ and $b'=b$. Note that $R((C^\vee C_l^{(2)\diamond})^\vee)/\R b' \cong R(C^\vee BC_l)$. 
Since we have
\begin{align*}
\{\, ma'+nb'\, \vert\, m-n\equiv 0\, [2]\, \}=
&2\Z a'+\Z (b'+a'), \\
2\Z a'+2\Z b'=
&2\Z a'+2\Z (b'+a'), 
\end{align*}
and
\begin{align*}
4\Z a'+2\Z b'=
&\{ \, 2(2m-n)a'+2n(b'+a')\, \vert\, m,n \in \Z\, \} \\
=
&\{ \, m' (2a')+2n'(b'+a')\, \vert\, m'-n'\equiv 0\, [2]\, \},
\end{align*}
it follows that the marked root system $C^\vee C_l^{(2)\diamond}$ is self-dual. \\

\begin{rem} Given an elliptic root system $R$, one may ask what can be the possible affine quotients. 
It  can be shown that 
\begin{enumerate}
\item if $R$ is isomorphic to $R(BCC_l^{(4)}) \cong R(C^\vee BC_l^{(1)}) \cong R(C^\vee C_l^{((2) \diamond})$, there are $3$ non-isomorphic affine quotients, 
\item otherwise there are at most $2$ affine quotients.
\end{enumerate}
These facts can be shown using the automorphism groups of $R$, which will be treated in our forthcoming paper. 
\end{rem}

Remark that the table also reflects nicely the duality $R \leftrightarrow R^\vee$, as $R^\vee/G_a$ and $R^\vee /G_b$ are isomorphic to $(R/G_a)^\vee$ and $(R/G_b)^\vee$, respectively.  For the reader's convenience, here is the duality of non-reduced affine root systems (cf. \S \ref{sect_non-red-affine}): \\

\begin{center}
\begin{tabular}{|c||c|c|c|c|} \hline
Type of $R$ & $BCC_l$ & $C^\vee BC_l$ & $BB_l^\vee$ & $C^\vee C_l$ \\ \hline
Type of $R^\vee$ & $C^\vee BC_l$ & $BCC_l$ & $BB_l^\vee$ & $C^\vee C_l$  \\ \hline
\end{tabular}
\end{center}

\vskip 0.1in
\subsection{Reduced pairs}\label{sect_red-pair}

For later use, we provide the list of reduced pairs $(R^{\dc}, R^{\mc})$ for each 
non-reduced mERS $(R,G)$:


\begin{enumerate}
\item $R/G$: type $BCC_l$ \quad $(l\geq 1)$\\
\begin{center}
\resizebox{.93\textwidth}{!}{%
\begin{tabular}{|c|c|c||c|c|c|} \hline
Type of $R$ & Type of $R^{\dc}$ & Type of $R^{\mc}$ & Type of $R$ & Type of $R^{\dc}$ & Type of $R^{\mc}$
\\ \hline
$BCC_l^{(1)}$ & $BC_l^{(1,1)\ast}$ & $C_l^{(1,1)}$ &
$\begin{matrix} BCC_l^{(2)}(1)\\ (l\geq 2) \end{matrix}$& $BC_l^{(2,2)}(1)$ & $C_l^{(1,2)}$  
\\ \hline
$BCC_l^{(1)\ast_0}$ & $BC_l^{(1,2)}$ & $C_l^{(1,1)\ast}$ &
$BCC_l^{(2)}(2)$ & $BC_l^{(2,2)}(2)$ & $C_l^{(1,1)}$
\\ \hline
$BCC_l^{(1)\ast_{0'}}$ & $BC_l^{(2,1)}$ & $C_l^{(1,1)\ast}$ &
$BCC_l^{(2)\ast_0}$ & $BC_l^{(2,4)}$ & $C_l^{(1,1)\ast}$ 
\\ \hline
$BCC_l^{(4)}$ & $BC_l^{(2,4)}$ & $BC_l^{(1,2)}$ &
$BCC_l^{(2)\ast_1}$ & $BC_l^{(2,2)}(2)$ & $BC_l^{(1,1)\ast}$ 
\\ \hline
\end{tabular}}

\end{center}
\vskip 0.2in

\item $R/G$: type $C^\vee BC_l$ \quad $(l\geq 1)$ \\
\begin{center}
\resizebox{.93\textwidth}{!}{%
\begin{tabular}{|c|c|c||c|c|c|} \hline
Type of $R$ & Type of $R^{\dc}$ & Type of $R^{\mc}$ & Type of $R$ & Type of $R^{\dc}$ & Type of $R^{\mc}$
\\ \hline
$C^\vee BC_l^{(1)}$ & $BC_l^{(4,2)}$ & $BC_l^{(2,1)}$ &
$\begin{matrix} C^\vee BC_l^{(2)}(1) \\ (l\geq 2) \end{matrix}$ & $B_l^{(2,1)}$ & $BC_l^{(2,2)}(1)$  
\\ \hline
$C^\vee BC_l^{(4)}$ & $B_l^{(2,2)}$ & $BC_l^{(4,4)\ast}$ &
$C^\vee BC_l^{(2)}(2)$ & $B_l^{(2,2)}$ & $BC_l^{(2,2)}(2)$
\\ \hline
$C^\vee BC_l^{(4)\ast_0}$ & $B_l^{(2,2)\ast}$ & $BC_l^{(4,2)}$ &
$C^\vee BC_l^{(2)\ast_0}$ & $B_l^{(2,2)\ast}$  & $BC_l^{(2,1)}$
\\ \hline
$C^\vee BC_l^{(4)\ast_{0'}}$ & $B_l^{(2,2)\ast}$ & $BC_l^{(2,4)}$ &
$C^\vee BC_l^{(2)\ast_1}$ & $BC_l^{(4,4)\ast}$ & $BC_l^{(2,2)}(2)$ 
\\ \hline
\end{tabular}}

\end{center}
\vskip 0.2in

\item $R/G$: type $BB_l^\vee$ \quad $(l\geq 2)$\\
\begin{center}
\resizebox{.93\textwidth}{!}{%
\begin{tabular}{|c|c|c||c|c|c|} \hline
Type of $R$ & Type of $R^{\dc}$ & Type of $R^{\mc}$ & Type of $R$ & Type of $R^{\dc}$ & Type of $R^{\mc}$
\\ \hline
$BB_l^{\vee (1)}$ & $BC_l^{(2,2)\sigma}(1)$ & $C_l^{(2,1)}$ &
$BB_l^{\vee (2)}(1)$ & $B_l^{(1,1)}$ & $C_l^{(2,2)}$
\\ \hline
$BB_l^{\vee (4)}$ & $B_l^{(1,2)}$ & $BC_l^{(2,2)\sigma}(1)$ &
$BB_l^{\vee (2)}(2)$ & $B_l^{(1,2)}$ & $C_l^{(2,1)}$
\\ \hline
 & &  &
$BB_2^{\vee (2)\ast}$ & $C_2^{(1,1)\ast}$ & $B_2^{(2,2)\ast}$
\\ \hline
\end{tabular}}

\end{center}
\vskip 0.2in
\item $R/G$: type $C^\vee C_l$ \quad $(l\geq 1)$\\
\begin{center}
\resizebox{.93\textwidth}{!}{%
\begin{tabular}{|c|c|c||c|c|c|} \hline
Type of $R$ & Type of $R^{\dc}$ & Type of $R^{\mc}$ & Type of $R$ & Type of $R^{\dc}$ & Type of $R^{\mc}$
\\ \hline
$C^\vee C_l^{(1)}$ & $BC_l^{(2,2)\sigma}(2)$ & $C_l^{(1,1)}$ &
$C^\vee C_l^{(4)}$ & $B_l^{(2,2)}$ & $BC_l^{(2,2)\sigma}(2)$
\\ \hline
$C^\vee C_l^{(1)\ast_0}$ & $BC_l^{(4,2)}$ & $C_l^{(1,1)\ast}$ &
$C^\vee C_l^{(4)\ast_0}$ & $B_l^{(2,2)\ast}$ & $BC_l^{(1,2)}$
\\ \hline
$C^\vee C_l^{(1)\ast_1}$ & $BC_l^{(2,2)\sigma}(2)$ & $BC_l^{(1,1)\ast}$ &
$C^\vee C_l^{(4)\ast_1}$ & $BC_l^{(4,4)\ast}$ & $BC_l^{(2,2)\sigma}(2)$
\\ \hline
$\begin{matrix}C^\vee C_l^{(2)}(1) \\ (l\geq 2) \end{matrix}$ & $B_l^{(2,1)}$ & $C_l^{(1,2)}$ &
$C^\vee C_l^{(2)}(2)$ & $B_l^{(2,2)}$ & $C_l^{(1,1)}$
\\ \hline
$C^\vee C_l^{(2)\ast_s}$ & $BC_l^{(4,4)\ast}$ & $C_l^{(1,1)}$ &
$C^\vee C_l^{(2)\ast_l}$ & $B_l^{(2,2)}$ & $BC_l^{(1,1)\ast}$
\\ \hline
$C^\vee C_l^{(2)\ast_0}$ & $B_l^{(2,2)\ast}$ & $C_l^{(1,1)\ast}$ &
$C^\vee C_l^{(2)\ast_1}$ & $BC_l^{(4,4)\ast}$ & $BC_l^{(1,1)\ast}$
\\ \hline
$C^\vee C_l^{(2) \diamond}$ & $BC_l^{(4,2)}$ & $BC_l^{(1,2)}$ &
$C^\vee C_l^{(2)\ast_{1'}}$ & $BC_l^{(4,4)\ast}$ & $BC_l^{(1,1)\ast}$
\\ \hline
\end{tabular}}
\vskip 0.2in
\end{center}
\end{enumerate}
Here we set 
\begin{align*}
&C_1^{(1,1)}=B_1^{(2,2)}=A_1^{(1,1)}, \qquad C_1^{(1,1)\ast}=B_1^{(2,2)\ast}=A_1^{(1,1)\ast}, \\
&B_2^{(1,t)}=C_2^{(1,t)} \quad \text{and} \quad C_2^{(2,t)}=B_2^{(2,t)} \qquad (t=1,2),
\end{align*}
for simplicity.\\

We close this chapter summarizing what is achieved so far and what is still necessary to work out to complete the classification of mERSs with non-reduced 
affine quotient. 
At this stage, one can not tell whether two given non-reduced mERSs in the above theorems are non-isomorphic. 
In Section \ref{sect:main}, we discuss this question.
By Proposition \ref{prop:indendence-Gamma}, Theorem \ref{thm:classification-thm} and Corollary \ref{cor:classification-thm}, 
it turns out that the above theorem gives a complete list of isomorphism classes of non-reduced mERSs (Theorem \ref{thm:main}).

\part{Classification}\label{chapter:classification}

 {Let $(R,G)$ be a mERS with the non-reduced affine quotient $R/G$.  
In $\S$ \ref{sect:Classification-weak},  we already showed that $(R,G)$ is isomorphic to 
one of the 6 infinite series 
if $R$ is reduced 
(Theorem \ref{thm_classification-reduced}), and
is isomorphic to one of the 34 infinite series 
or 1 exceptional type if $R$ is 
non-reduced (Theorem 
\ref{thm_classification-non-reduced}). However, we have not {discussed}
the following fundamental question:
\begin{center}
{\it Are these 
mERSs really non-isomorphic to each other?}
\end{center}
In this section, we give an affirmative
answer to
this question, by using the 
``elliptic diagram'' $\Gamma(R,G)$ of a mERS $(R,G)$ with the non-reduced 
affine quotient. The elliptic diagram $\Gamma(R,G)$ of $(R,G)$ is a 
colored graph consisting of finite number of nodes, and edges connecting them, which
depends only on an isomorphism class of mERSs. One has that 
\begin{itemize}
\item if two mERSs $(R_1,G_1)$ and $(R_2,G_2)$ are isomorphic, their elliptic diagrams 
$\Gamma(R_1,G_1)$ and $\Gamma(R_2,G_2)$ are isomorphic to each other
(Proposition \ref{prop:indendence-Gamma}), and 
\vskip 1mm
\item if $(R_1,G_1)$ and $(R_2,G_2)$ are two different mERSs in the sense of 
Theorem \ref{thm_classification-reduced} and Theorem 
\ref{thm_classification-non-reduced}), then $\Gamma(R_1,G_1)$
is not isomorphic to $\Gamma(R_2,G_2)$ (see Theorem \ref{thm:classification-thm}
and $\S$ \ref{sect_Remarks-classification} (i)).
\end{itemize}
That is, the results in $\S$ \ref{sect:Classification-weak} give the complete list of the 
isomorphism classes of mERSs with non-reduced affine quotients.
\vskip 3mm

Let us briefly summarize the results of this section.  
In $\S$ \ref{sect:ans-Q1}, we give an answer for
Question 1 in $\S$ \ref{sect_marking}, that is, 
we describe some relationships between the six reduced affine quotient root 
systems: $(R/G)^{\dc}$, $(R/G)^{\mc}$, $(R^{\dc}/G)^{\dc}$, $(R^{\dc}/G)^{\mc}$, 
$(R^{\mc}/G)^{\dc}$ and $(R^{\mc}/G)^{\mc}$, appeared in $\S$ \ref{sect_marking}. 
\S \ref{sect_PRB} to \ref{sect:unique-basis} are devoted to 
preparations for defining elliptic diagrams. More precisely, we introduce the technical 
tools: the {\it prime map} (see $\S$ {\ref{sect_PRB}), 
 {\it paired reduced simple systems} (see $\S$ \ref{sect_PRB}), 
 {\it non-reduced counting numbers} (see $\S$ \ref{sect:NRCN}), 
 and study their
basic properties. The goals in these preparations are Proposition 
\ref{prop:description-R} in $\S$ \ref{sect:description-R} and 
Proposition \ref{prop:nr-counting-Autom1} in $\S$ \ref{sect:unique-basis}. 
The former expresses the root system $R$ in terms of its affine quotient $R/G$ and 
the countings. {This} is  a ``non-reduced analogue'' of
the formula \eqref{eq:description-redR} for mERSs with reduced affine quotients. 
The later says that a paired simple system is uniquely determined (up to 
isomorphism) for a given $(R,G)$. This result plays one of the central roles for classifying 
isomorphism classes of mERSs with non-reduced affine quotients in Section
\ref{sect:main}. 
 
In $\S$ \ref{sect_elliptic-diagram-II}, we define the elliptic diagram $\Gamma(R,G)$
for a mERS $(R,G)$ with the non-reduced affine quotient (Definition 
\ref{defn_elliptic-diagram-II}), and show that this diagram depends only on the 
isomorphism classes of mERSs (Proposition \ref{prop:indendence-Gamma}).  
Section \ref{sect:main} is devoted to proving Theorem 
\ref{thm:classification-thm} which we call the classification theorem of mERSs with
non-reduced affine quotient. In the proof of this theorem, we need case-by-case detailed
analysis. We note that this theorem does not give us a complete list of such mERSs. 
Indeed, it only lists all possibilities, but does not state 
whether a root 
system with such a diagram really exists. In other words, we need to solve the 
``existence problem'' of 
a mERS with a given elliptic diagram. However, comparing
the results of Chapter 2, we show that a mERS really exists for every possible elliptic
diagrams in Theorem \ref{thm:classification-thm}. Thus, we have a complete list of 
isomorphism classes of mERS with non-reduced affine quotients (see $\S$ 
\ref{sect_Remarks-classification}, in detail).   
In Section \ref{sect_str-red-pair}, we give detailed description of $R^{\dc}$ and 
$R^{\mc}$. 
}


\section[Reduced subsystems of affine quotients]{Reduced subsystems of affine quotients}
\label{sect:elliptic-diagram}

We define the elliptic diagram $\Gamma(R,G)$ for a mERS 
$(R,G)$ with non-reduced affine quotient and show that the isomorphism classes of 
such mERSs are classified by these diagrams. These are generalizations of 
the results in \cite{Saito1985} for mERSs with reduced quotients. 
In this section, we simplify the symbols as far as possible, e.g., omitting $(R,G)$ in the notations of the elliptic diagrams, the subspace $F$ of $I_F$, etc. 
 
\subsection{On reduced affine subquotients}\label{sect:ans-Q1} 
First of all, we give an answer to
Question 1 in $\S$ \ref{sect_marking}. 
Here and after, we assume that $(R,G)$ is an arbitrary mERS, unless 
otherwise specified. Recall the canonical projection $\pi_G:F\to F/G$. 
For $\alpha\in R$, we denote 
$\overline{\alpha}:=\pi_G(\alpha)$ for simplicity. 
\\

Let $(R^{\dc},R^{\mc})$ be a reduced pair of $R$.
For a reduced $R$, we regard $R^{\dc}=R=R^{\mc}$. 
Similarly, for the quotient affine root system $R/G$, let 
$((R/G)^{\dc},(R/G)^{\mc})$ be its reduced pair.

\begin{lemma}\label{lemma:reduced_pair1}
Let $\alpha\in R$.
\vskip 1mm
\noindent
{\rm (1)} If $\overline{\alpha}\in (R/G)^{\dc}$, $\alpha$ belongs to $R^{\dc}$.
\vskip 1mm
\noindent
{\rm (2)} If $\overline{\alpha}\in (R/G)^{\mc}$, $\alpha$ belongs to $R^{\mc}$.
\end{lemma}
\begin{proof}
Let us prove statement (1). Assume $\alpha\not\in R^{\dc}$. 
By definition, it means 
$\alpha/2 \in R$.  Therefore, we have 
$\overline{\alpha/2}=\overline{\alpha}/2 \in R/G$, i.e., 
$\overline{\alpha}$ does not belongs to $(R/G)^{\dc}$. 
 Statement (2) is obtained by a similar method.   
\end{proof}

Let us recall the setting of Question 1. Assume that $R/G$ is non-reduced.
We note that both of $(R^{\dc},G)$ and $(R^{\mc},G)$ are reduced mERSs. Since the 
corresponding affine quotients $R^{\dc}/G$ and $R^{\mc}/G$ are not necessarily 
reduced, four reduced affine root systems $(R^{\dc}/G)^{\dc}$, $(R^{\dc}/G)^{\mc}$, 
$(R^{\mc}/G)^{\dc}$, $(R^{\mc}/G)^{\mc}$
may be distinct. In addition, there are two affine 
root systems $(R/G)^{\dc}$ and $(R/G)^{\mc}$. Question 1 asks 
the relationships between these
six root systems. The following lemma gives a (partial) answer for this question. 

\begin{lemma}\label{lemma:reduced_pair2}
{\rm (1)} $(R/G)^{\dc}=(R^{\dc}/G)^{\dc}$.
\vskip 1mm
\noindent
{\rm (2)} $(R/G)^{\mc}=(R^{\mc}/G)^{\mc}$.
\end{lemma}

\begin{proof}
Since  statement (2)  can be proved in a similar way as statement (1), we only prove (1).

First, we prove the inclusion
\begin{equation}\label{eq:step1-1}
(R/G)^{\dc}\subset (R^{\dc}/G)^{\dc}.
\end{equation}
By Lemma \ref{lemma:reduced_pair1} (1), we have
\[(R/G)^{\dc}\subset R^{\dc}/G.\]
Let $\overline{\alpha}$ be an element of $(R/G)^{\dc}$, i.e., $\overline{\alpha}/2 \not\in R/G$.
It follows from the inclusion $R^{\dc}/G \subset R/G$ that $\overline{\alpha}/2 \not\in R^{\dc}/G$, 
i.e., $\overline{\alpha} \in (R^{\dc}/G)^{\dc}$, which implies \eqref{eq:step1-1}.

Second, 
assume that the inclusion \eqref{eq:step1-1} is strict, i.e., there exists $\overline{\alpha}\in (R^{\dc}/G)^{\dc}\setminus (R/G)^{\dc}$. 
By the condition $\overline{\alpha}\not\in (R/G)^{\dc}$, we have 
$\overline{\alpha}/2\in R/G$. 
In other words, there exists $\beta\in R$ such that 
$\overline{\beta}=\overline{\alpha}/2$. 
In particular, $\beta$ is a short root. 
On the other hand, by the assumption $\overline{\alpha}\in (R^{\dc}/G)^{\dc}$, 
one has $2\overline{\beta}\in (R^{\dc}/G)^{\dc}$. Therefore, we have 
$\overline{\beta}\not\in R^{\dc}/G$. Since $R^{\dc}$ contains the set $R_s$ of short roots 
of $R$, the condition $\overline{\beta}\not\in R^{\dc}/G$  
implies that $\beta$ is not 
a short root. This is a contradiction. Thus, we have statement (1).
\end{proof}

\begin{rem}
In the following, we use only $(R/G)^{\dc}=(R^{\dc}/G)^{\dc}$ and $(R/G)^{\mc}=(R^{\mc}/G)^{\mc}$, and
neither $(R^{\dc}/G)^{\mc}$ nor $(R^{\mc}/G)^{\dc}$ will be used. So, we omit to give explicit descriptions
of them.
\end{rem}
\vskip 5mm
\subsection{Paired simple system}\label{sect_PRB}
Let $(R^{\dc},R^{\mc})$ be the reduced pair of a mERS $(R,G)$ and 
$\Pi^{\dc}=\Pi_{R^{\dc}}$ be a simple system of 
the reduced mERS $(R^{\dc},G)$ in the sense of Definition \ref{defn:basis}. It is immediate
to see that $\Pi^{\dc}=\Pi_{R^{\dc}}$ also gives a simple system of $(R,G)$. 
However, because of 
non-reducedness, the information on $\Pi^{\dc}=\Pi_{R^{\dc}}$ is {\it not} sufficient for describing 
the structure of the set $R$ of all roots, that is,  Lemma 
\ref{lemma:countings1} {\rm (4)} {holds} in general, 
which plays an important role 
in the detailed study for a mERS $(R,G)$ with reduced $R/G$.

To fill the missing information, 
we introduce the notion of a ``paired simple system'' $(\Pi^{\dc},\Pi^{\mc})$ for
the reduced pair $(R^{\dc},R^{\mc})$ of a mERS $(R,G)$ (see Definition
\ref{defn:PRB} below). Here, $\Pi^{\mc}$ 
is a linearly independent 
subset of $R^{\mc}$, which satisfies certain properties.
With this additional information, we generalize  the description 
\eqref{eq:description-redR} of $R$ for an arbitrary case in Proposition 
\ref{prop:description-R} below. \\

Let $\Pi^{\dc}=\Pi_{R^{\dc}}$ be a simple system of the reduced mERS $(R^{\dc},G)$. 
As we mentioned above, it gives a simple system of $(R,G)$ also. 
Furthermore, its image 
$\pi_G(\Pi^{\dc})=\pi_G(\Pi_{R^{\dc}})$ is a simple system of 
$(R^{\dc}/G)^{\dc}=(R/G)^{\dc}$ (see Lemma \ref{lemma:reduced_pair2} (1)). 
Motivated by Lemma \ref{lemma_base}, we want to 
take a simple system $\Pi_{R^{\mc}}$ of $(R^{\mc},G)$ such that 
\begin{enumerate}
\item[i)] $\pi_G(\Pi_{R^{\mc}})$ is a simple system of $(R/G)^{\mc}$, and 
\item[ii)] $\pi_G(\Pi_{R^{\mc}})\setminus \pi_G(\Pi_{R^{\dc}})$ contains only 1 root for 
$R/G$ being of type $BCC_l$, $C^\vee BC_l$ and $BB^\vee_l$, and 2 roots 
of type $C^\vee C_l$. 
\end{enumerate}
But it is {\it impossible}. 
Indeed, let $\Pi_{R^{\mc}}$ be a simple system of $(R^{\mc},G)$. 
Then, its image $\pi_G(\Pi_{R^{\mc}})$ is a simple system of the affine root system
$(R^{\mc}/G)^{\dc}$. On the other hand, by Lemma \ref{lemma:reduced_pair2} (2), 
we already know $(R/G)^{\mc}=(R^{\mc}/G)^{\mc}$. Therefore, the set 
$\pi_G(\Pi_{R^{\mc}})$ is not necessary included in $(R/G)^{\mc}$, and there is {\it no}
simple system $\Pi_{R^{\mc}}$ of $(R^{\mc},G)$ which satisfies condition i) in general.

\begin{ex}{\rm
Let $(R,G)$ be a mERS of type $C^\vee C_l^{(4)}$ with $l\geq 2$
(see $\S$ \ref{subsect_CC-l} (2) (II) or $\S$ \ref{data:CC-12}).
As in the list in $\S$ \ref{sect_red-pair} (4), 
the reduced mERS $(R^{\mc},G)$ is of type 
$BC_l^{(2,2)\sigma}(2)$. Furthermore, as in the list in $\S$ \ref{sect_two-affine-quot},
$R^{\mc}/G$ is a non-reduced affine root system of type $C^\vee C_l$. 
In other words, $(R^{\mc}/G)^{\dc}$ is of type $B_l^{(2)}$ and 
$(R^{\mc}/G)^{\mc}=(R/G)^{\mc}$ is of type $C_l^{(1)}$, respectively.  Especially, 
$(R^{\mc}/G)^{\dc}$ (resp. $(R^{\mc}/G)^{\mc}$) 
consists of short and middle (resp. middle and long) roots. Therefore, any simple 
system $\Pi_{R^{\mc}}$ of $(R^{\mc},G)$ consists of short and middle roots also, and 
the image $\pi_G(\Pi_{R^{\mc}})$ is not a subset of $(R/G)^{\mc}$.
}\end{ex}

Instead, we take a linearly independent subset $\Pi^{\mc}$ of $R^{\mc}$ such that both 
the conditions i) and ii) are satisfied. However, if we take it unconditionally, we will 
lose control. Therefore, we have to select a reasonable $\Pi^{\mc}$ 
in accordance with the first 
choice of $\Pi^{\dc}=\Pi_{R^{\dc}}$. For this purpose, we introduce  the ``prime map'' 
$(\ \cdot\ )^{pr}:R\to R^{\mc}$. \\

For each $\alpha \in R$, we introduce a root 
$\alpha'\in R$\index[notations]{a711@$\alpha'$} by 
the next rules: 
\begin{itemize}
\item[i)] If $2\overline{\alpha}\not\in R/G$, set
\[\alpha'=\alpha.\]
It is obvious that $\alpha'\in R$ and $2\alpha'\not\in R$.
\vskip 1mm
\item[ii)]  
Otherwise, one can take a
representative $\beta\in \pi_G^{-1}(2\overline{\alpha})\cap R$ of 
$2\overline{\alpha}$. Since 
$\overline{\beta}=2\overline{\alpha}$, we have
$\beta-2\alpha\in Q(R)\cap G=\Z a$. 
Hence,
\[\pi_G^{-1}(2\overline{\alpha})\cap R\subset 2\alpha+\Z a.\]
Let $\alpha'\in 
\pi_G^{-1}(2\overline{\alpha})\cap R$ be a unique element satisfying
\[\alpha'\in 2\alpha+\Z_{\leq 0}a\quad\text{and}\quad
(\alpha')^*=\alpha'+k(\alpha')a\in 2\alpha+\Z_{>0}a.\]
Since $\alpha'$ is a long root, we have $2\alpha'\not\in R$, 
as in the previous case.
\end{itemize} 
This is the definition of $\alpha'\in R$ for $\alpha\in R$. We note that the 
correspondence $\alpha \mapsto \alpha'$ defines a map 
$(\,\cdot\,)^{pr}:R\to R^{\mc}$ which we call the 
\textbf{prime map}\index[index]{prime map} 
for $R$. 

\begin{rem}\label{rem:a}
The definition of the map $(\,\cdot\,)^{pr}:R\to R^{\mc}$ depends on the choice of a 
generator of the lattice $Q(R)\cap G$ of rank 1.
\end{rem}

Define a subgroup $\Aut(R,G)$ of $\Aut(R)$ by
\[ \Aut(R,G)=\{\varphi\in \Aut(R)\,|\, \varphi|_G=\mathrm{id}_G\}.\]

\begin{lemma}\label{lemma:commutativity-prime}
Let $\varphi\in  \Aut(R,G)$. We have 
\begin{equation}\label{commutativity-prime}
(\varphi(\alpha))'=\varphi(\alpha')\quad  \text{for every $\alpha\in R$}.
\end{equation}
That is, every automorphism $\varphi\in \Aut(R,G)$ commutes with 
the map $(\,\cdot\, )^{pr}:R\to R^{\mc}$. 
\end{lemma}
\begin{proof}
Note that, since $\varphi(a)=a$, the induced automorphism 
$\overline{\varphi}\in\mathrm{Aut}(R/G)$ is well-defined. 
Since the statement is trivial if $2\overline{\alpha}\not\in R/G$, we may 
assume $2\overline{\alpha}\in R/G$. 
In this case, we have $2\overline{\varphi(\alpha)}\in R/G$. 
By definition, we have
$\pi_G^{-1}(2\overline{\varphi(\alpha)})\cap R\subset 2\varphi(\alpha)+\Z a$. 
Hence, $(\varphi(\alpha))'\in R$ is determined uniquely by the following conditions:
\begin{align*}
(\varphi(\alpha))'\in 
&\; 2\varphi(\alpha)+\Z_{\leq 0}a, \\
((\varphi(\alpha))')^*=
(\varphi(\alpha))'+k((\varphi(\alpha))')a\in 
&\; 2\varphi(\alpha)+\Z_{> 0}a.
\end{align*}

Denote $\alpha'=2\alpha+pa$ with $p\in\Z_{\leq 0}$. Recall that 
$p+k(\alpha')>0$. We have
\[\begin{aligned}
\varphi(\alpha')&=\varphi(2\alpha+pa)\in 2\varphi(\alpha)+\Z_{\leq 0}a.
\end{aligned}\]
On the other hand, we have
\[\begin{aligned}
(\varphi(\alpha'))^*&=\varphi(\alpha')+k(\varphi(\alpha'))a\\
&=\varphi(\alpha')+k(\alpha')a \quad (\text{by Lemma \ref{lemma:countings1} (3)})\\
&=2\varphi(\alpha)+(p+k(\alpha'))a\\
&\in 2\varphi(\alpha)+\Z_{>0}a\quad (\because\ p+k(\alpha')>0).
\end{aligned}\]
Thus, the element $\varphi(\alpha')$ satisfies the defining
condition of $(\varphi(\alpha))'$, and we have \eqref{commutativity-prime}.
\end{proof}
\medskip 
Let us introduce a ``paired simple system'' $(\Pi^{\dc},\Pi^{\mc})$ for a mERS $(R,G)$ 
as follows. 

First, take a simple system 
$\Pi^{\dc}=\{\alpha_i\}_{0\leq i\leq l}$\index[notations]{p711@$\Pi^{\dc}$} 
of the reduced 
mERS $(R^{\dc},G)$ in the sense of Definition \ref{defn:basis}. 
Then, its image $(\Pi^{\dc})_G:=\pi_G(\Pi^{\dc})=\{\overline{\alpha_i}\}_{0\leq i\leq l}$
is a simple system of the affine root system $(R^{\dc}/G)^{\dc}=(R/G)^{\dc}$.  
For later use, in addition, we assume that
\begin{equation}\label{condition-Pi'}
\text{if $\big(\pi_G^{-1}(\overline{\alpha})\cap R\big)
\setminus \frac12 R_l\ne\emptyset$ 
for $\alpha\in \Pi^{\dc}$, $\alpha$ belongs to 
$\big(\pi_G^{-1}(\overline{\alpha})\cap R\big) \setminus \frac12 R_l$.}
\end{equation}

Second, define a linearly independent subset $\Pi^{\mc}\subset R^{\mc}$ by
\[\Pi^{\mc}=(\Pi^{\dc})^{pr}=\{\alpha_i'\}_{0\leq i\leq l}\index[notations]{p712@
$\Pi^{\mc}$},\]
where $\alpha_i'$ is the image of $\alpha_i\in\Pi^{\dc}$ under the prime map 
$(\,\cdot\,)^{pr}:R\to R^{\mc}$. 
By definition, the image $\overline{\alpha'_i}$ of $\alpha_i$ ($0\leq i\leq l$) under 
the projection $\pi_G: F \rightarrow F/G$ satisfies 
\begin{equation}\label{eq:paired-basis}
\overline{\alpha'_i}=\begin{cases}
\overline{\alpha_i} & \text{if }2\overline{\alpha_i}\not \in R/G,\\
2\overline{\alpha_i} & \text{otherwise}.
\end{cases}
\end{equation}
Therefore, as we already claimed in the proof of Lemma \ref{lemma_base}, 
$(\Pi^{\mc})_G:=\pi_G(\Pi^{\mc})$ is a simple system of $(R/G)^{\mc}$. 
\begin{defn}\label{defn:PRB}
Such a pair $(\Pi^{\dc},\Pi^{\mc})$ is called a 
\textbf{paired simple system}\index[index]{paired simple system} of 
$(R,G)$. 
\end{defn}

By the above discussion, we may assume that the pair $((\Pi^{\dc})_G, (\Pi^{\mc})_G)$ 
is a paired simple system of $R/G$ satisfying Lemma \ref{lemma_base}. For 
this reason, we will write $\Pi^{\dc}_G$\index[notations]{p713@$\Pi^{\dc}_G$}, 
$\Pi^{\mc}_G$\index[notations]{p714@$\Pi^{\mc}_G$} instead of 
$(\Pi^{\dc})_G, (\Pi^{\mc})_G$, respectively, to simplify the notations.
\medskip 
\subsection{Descriptions of $R$ attached to $(\Pi^{\dc}, \Pi^{\mc})$}
In this subsection, we introduce two descriptions of the set $R$
of roots attached to a paired simple system $(\Pi^{\dc}, \Pi^{\mc})$
(see \eqref{description-R0} in Corollary \ref{cor:nr-countings0} and 
\eqref{description-R1} in Corollary \ref{cor:description-R}, below). 
These are ``non-reduced analogues'' of  the description \eqref{eq:description-redR} 
which is one of the key formulas for analyzing detailed structures
of mERSs with reduced affine quotients. Similar to the previous cases, they play 
central roles for studying mERSs with non-reduced affine quotients.

\begin{rem}
In the rest of the article, we mainly use the description \eqref{description-R1}.
The reason why we choose \eqref{description-R1} is explained in the next subsection. 
\end{rem}
\medskip

Define two subgroups $\W[\Pi^{\dc}]$ and $\W[\Pi^{\mc}]$ of $W(R)$ by 
\[\W[\Pi^{\dc}]=\langle r_\alpha\, |\,\alpha\in \Pi^{\dc}\rangle
\index[notations]{w711@$\W[\Pi^{\dc}]$}
\quad\text{and}\quad
\W[\Pi^{\mc}]=\langle r_{\alpha'}\, |\,\alpha'\in \Pi^{\mc}\rangle.
\index[notations]{w712@$\W[\Pi^{\mc}]$}\]
Since every element in the Weyl group
$W(R)$ stabilizes the marking $G$, the canonical projection $\pi_G:F\to F/G$
induces a surjective group homomorphism 
\[(\pi_G)_*:W(R)\to W(R/G).\]
For each $\alpha\in R$, one has
$(\pi_G)_*(r_{\alpha})=\overline{r}_{\overline{\alpha}}$, where 
$\overline{r}_{\overline{\alpha}}$ is the reflection
in $W(R/G)$ attached to the root $\overline{\alpha}\in R/G$. 
Let $\overline{\W[\Pi^{\dc}_G]}$ ({\it resp}. $\overline{\W[\Pi^{\mc}_G]}$) be the group 
generated by $\overline{r}_{\overline{\alpha}}\, (\overline{\alpha}\in \Pi_G^{\dc})$ 
({\it resp}. $\overline{r}_{\overline{\alpha}}\, (\overline{\alpha}\in \Pi_G^{\mc}))$. 

\begin{lemma}\label{lemma:paired-reduced-basis}
{\rm (1)} $\overline{\W[\Pi^{\dc}_G]}=W(R/G)=\overline{\W[\Pi^{\mc}_G]}$.
\vskip 1mm
\noindent
{\rm (2)} $\pi_G(\W[\Pi^{\dc}].\Pi^{\dc})=(R/G)^{\dc}$ and 
$\pi_G(\W[\Pi^{\mc}].\Pi^{\mc})=(R/G)^{\mc}$.
\end{lemma}
\begin{proof}
(1) Since $\Pi_G^{\dc}$ is a simple system of the reduced affine root system of 
$(R/G)^{\dc}$, we have 
$W((R/G)^{\dc})=\overline{\W[\Pi^{\dc}_G]}$. In addition, 
since $W(R/G)=W((R/G)^{\dc})$,
we have the first equality.  The second equality 
follows form the definition of $\Pi^{\mc}_G$.
\vskip 1mm
\noindent
(2) As the second equality is obtained by a similar way to the first one,
we only give a proof of the {first} one. 
Since $\overline{\W[\Pi^{\dc}_G]}=W((R/G)^{\dc})$, we have
\[\pi_G(\W[\Pi^{\dc}].\Pi^{\dc})=\overline{\W[\Pi^{\dc}_G]}.\Pi^{\dc}_G
=W((R/G)^{\dc}).\Pi^{\dc}_G=(R/G)^{\dc}.\]
as desired. 
\end{proof}

Define vector subspaces $L_{\Pi^{\dc}}$ and 
$L_{\Pi^{\mc}}$ of $F$ by 
\begin{equation}\label{defn:L}
L_{\Pi^{\dc}}= \bigoplus_{i=0}^l \R \alpha_i\index[notations]{l711@$L_{\Pi^{\dc}}$}
\quad\text{and}\quad
L_{\Pi^{\mc}}=\bigoplus_{i=0}^l \R \alpha'_i,\index[notations]{l712@$L_{\Pi^{\mc}}$}
\end{equation}
respectively. Then, we have two decompositions 
\begin{equation}\label{eq:decomp-L'-L''}
F=L_{\Pi^{\dc}}\bigoplus G=L_{\Pi^{\mc}}\bigoplus G
\end{equation}
of the vector space $F$, and they induce isomorphisms of vector spaces
\begin{equation}\label{affine-lift}
\pi_G|_{L_{\Pi^{\dc}}}:L_{\Pi^{\dc}}\xrightarrow{\sim} F/G\quad \text{and}\quad
\pi_G|_{L_{\Pi^{\mc}}}:L_{\Pi^{\mc}}\xrightarrow{\sim} F/G,
\end{equation}
respectively. \\

It is easy to see that the subset $\W[\Pi^{\dc}].\Pi^{\dc}$ 
({\it resp}. $\W[\Pi^{\mc}].\Pi^{\mc}$) 
is a root system belonging to 
$(L_{\Pi^{\dc}},I_{L_{\Pi^{\dc}}})$ ({\it resp}. $(L_{\Pi^{\mc}},I_{L_{\Pi^{\mc}}})$).
Furthermore, $\Pi^{\dc}=\{\alpha_0,\ldots,\alpha_l\}$ ({\it resp}.
$\Pi^{\mc}=\{\alpha_0',\ldots,\alpha_l'\}$) is a simple system of $\W[\Pi^{\dc}].\Pi^{\dc}$ 
({\it resp}. $\W[\Pi^{\mc}].\Pi^{\mc}$).
Especially, we have
\begin{equation}\label{affine_root_lattices}
Q(\W[\Pi^{\dc}].\Pi^{\dc})=\bigoplus_{i=0}^l \Z \alpha_i \quad \text{and}\quad 
Q(\W[\Pi^{\mc}].\Pi^{\mc})=\bigoplus_{i=0}^l \Z \alpha'_i,
\end{equation} 
where $Q(\W[\Pi^{\dc}].\Pi^{\dc})$ ({\it resp}. $Q(\W[\Pi^{\mc}].\Pi^{\mc})$) is the root lattice of 
$\W[\Pi^{\dc}].\Pi^{\dc}$ ({\it resp}. $\W[\Pi^{\mc}].\Pi^{\mc}$).
Furthermore, the isomorphisms in \eqref{affine-lift} induce bijections
\begin{equation}\label{(R/G)'-(R/G)''}
\W[\Pi^{\dc}].\Pi^{\dc} \xrightarrow{\sim} (R/G)^{\dc}\quad \text{and} \quad 
\W[\Pi^{\mc}].\Pi^{\mc} \xrightarrow{\sim} (R/G)^{\mc}.
\end{equation}

\begin{prop}\label{prop:description-R0}
{\rm (1)} The sets $R^{\dc}\cap L_{\Pi^{\dc}}$ and $R^{\mc}\cap L_{\Pi^{\mc}}$ are reduced affine root 
systems belonging to $(L_{\Pi^{\dc}},I_{L_{\Pi^{\dc}}})$ and $(L_{\Pi^{\mc}},I_{L_{\Pi^{\mc}}})$, respectively.
\vskip 1mm
\noindent
{\rm (2)} One has $R^{\dc}\cap L_{\Pi^{\dc}}=\W[\Pi^{\dc}].\Pi^{\dc}$ and 
$R^{\mc}\cap L_{\Pi^{\mc}}=\W[\Pi^{\mc}].\Pi^{\mc}$.
\end{prop}

\begin{proof}
(1) Since the statements for $R^{\mc}$ are obtained  {in a} similar way to those for 
$R^{\dc}$, we only prove the later case. First, we prove $R^{\dc}\cap L_{\Pi^{\dc}}$ 
is a root system belonging to $(L_{\Pi^{\dc}},I_{L_{\Pi^{\dc}}})$. 
Recall a result due to K. Saito \cite{Saito1985}. 
Let $R$ be a root system belonging to $(F,I)$ in the sense of Definition 
\ref{defn_root-sys}, $G$ be a subspace of $\rad{I}$ defined over $\Z$ and 
$\pi_G:F\longrightarrow F/G$ be the canonical projection. Denote $W(R)$ 
(resp. $W(R/G)$) the Weyl group of $R$ (resp. the quotient root system $R/G$).

\begin{lemma}[\cite{Saito1985} (1.15) Note]
Let $L$ be a subspace of $F$ such that
\begin{itemize}
\item[(a)] $F=L\bigoplus G$,
\vskip 1mm
\item[(b)] $(\pi_G)_*|_{W(R\cap L)}:W(R\cap L)\longrightarrow W(R/G)$ is surjective,
\end{itemize}
where $(\pi_G)_*:W(R)\longrightarrow W(R/G)$ is the group homomorphism induced
form $\pi_G:F\longrightarrow F/G$, and $W(R\cap L)$ is the subgroup of $W(R)$
generated by reflections $r_\alpha\ (\alpha\in R\cap L)$. Then, we have
\begin{itemize}
\item[(i)] $R\cap L$ is a root system belonging to $(L,I|_L)$,
\vskip 1mm
\item[(ii)] the linear isomorphism $\pi_G|_L:L\xrightarrow{\sim} F/G$ induces an injective
map $R\cap L\longrightarrow R/G$ which induces an isomorphism $W(R\cap L)
\xrightarrow{\sim} W(R/G)$. 
\end{itemize}
\end{lemma}

Apply this lemma for the case  $R=R^{\dc}$ and  $L=L_{\Pi^{\dc}}$. 
Condition (a) is satisfied by \eqref{eq:decomp-L'-L''}. Since $\Pi^{\dc}\subset
R^{\dc}\cap L_{\Pi^{\dc}}$ and Lemma \ref{lemma:paired-reduced-basis} (1), 
condition (b) is also satisfied. 
Therefore, $R^{\dc}\cap L_{\Pi^{\dc}}$ 
is a root system belonging to $(L_{\Pi^{\dc}},I_{L_{\Pi^{\dc}}})$ by the above lemma.

Second, let us prove $R^{\dc}\cap L_{\Pi^{\dc}}$ is reduced. However, this is obvious by 
the reducibility of $R^{\dc}$. 
\vskip 1mm
\noindent
(2) Let us prove the first equality. 
If $R^{\dc}/G$ is reduced, there is nothing to prove. 
Otherwise, it is of type $BCC_l$, $C^\vee BC_l$, 
$BB^\vee_l$ or $C^\vee C_l$. The reduced root system $R^{\dc}\cap L_{\Pi^{\dc}}$ 
is embedded into $R^{\dc}/G$ by the isomorphism 
$\pi_G|_{L_{\Pi^{\dc}}}:L_{\Pi^{\dc}}\xrightarrow{\sim} F/G$.
{Furthermore}, since $\W[\Pi^{\dc}].\Pi^{\dc}$ is a subset of $R\cap L_{\Pi^{\dc}}$, the 
image $\pi_G(R^{\dc}\cap L_{\Pi^{\dc}})$ is a reduced subsystem of $R^{\dc}/G$ which 
contains $\pi_G(\W[\Pi^{\dc}]\cdot \Pi^{\dc})=(R^{\dc}/G)^{\dc}$ 
by Lemma \ref{lemma:paired-reduced-basis} (2).
That is, we have a sequence of root systems
\[(R^{\dc}/G)^{\dc}\subset \pi_G(R^{\dc}\cap L_{\Pi^{\dc}})\subset R^{\dc}/G\]
such that both $(R^{\dc}/G)^{\dc}$ and $\pi_G(R^{\dc}\cap L_{\Pi^{\dc}})$ are reduced. 

Assume $(R^{\dc}/G)^{\dc}\subsetneq \pi_G(R^{\dc}\cap L_{\Pi^{\dc}})$. Since both 
$(R^{\dc}/G)^{\dc}$ and $\pi_G(R^{\dc}\cap L_{\Pi^{\dc}})$ are reduced, the root lattice 
$Q((R^{\dc}/G)^{\dc})$ is a proper sublattice of $Q(\pi_G(R^{\dc}\cap L_{\Pi^{\dc}}))$. 
On the other hand, by Lemma \ref{lemma:reduced_pair1} (1), we have 
$Q((R^{\dc}/G)^{\dc})=Q(R^{\dc}/G)$. 
This is a contradiction. Thus
\[ \pi_G(R^{\dc}\cap L_{\Pi^{\dc}})=(R^{\dc}/G)^{\dc}=\pi_G(\W[\Pi^{\dc}].\Pi^{\dc}). \]
Since $\pi_G|_{L_{\Pi^{\dc}}}:L_{\Pi^{\dc}}\xrightarrow{\sim} F/G$ is an isomorphism, 
we have $R^{\dc}\cap L_{\Pi^{\dc}}=\W[\Pi^{\dc}].\Pi^{\dc}$ as desired.  \\

For the second equality, we have
\[(R^{\mc}/G)^{\mc}\subset \pi_G(R^{\mc}\cap L_{\Pi^{\mc}})\subset R^{\mc}/G\]
by a similar way to the first case. In addition, we note that both 
$(R^{\mc}/G)^{\mc}=\pi_G(\W[\Pi^{\mc}].\Pi^{\mc})$ and 
$\pi_G(R^{\mc}\cap L_{\Pi^{\mc}})$ are reduced. 
Hence, the statement is reduced to the following claim.
\vskip 3mm
\noindent
{\bf Claim}. {\it  There is no reduced subsystem $R_1$ such that 
$(R^{\mc}/G)^{\mc}\subsetneq R_1\subset R^{\mc}/G$. }
\begin{proof}
This claim follows from case-by-case detailed analysis. In the following, we give 
a proof for the case when $R^{\mc}/G$ is of type $C^\vee C_l$, for example. Since 
the other cases are obtained by a similar way,  we leave the detailed proof to 
the reader.

Recall an explicit description of the set $R^{\mc}/G$ of roots of type $C^\vee C_l$ 
(see $\S$ \ref{sect:CveeC_l} in Appendix \ref{sect_affine-root-system}):
\[R^{\mc}/G=\left\{\begin{array}{lll}
\pm \vep_i+nb\ (1\leq i\leq l,\, n\in\Z),\\
\pm\vep_i\pm\vep_j+2nb\ (1\leq i<j\leq l,\, n\in\Z),\\
\pm 2\vep_i+2nb\ (1\leq i\leq l,\, n\in\Z)
\end{array}\right\}.\]
Here, we normalize a symmetric bilinear form $I_{F/G}$ on $F/G$ by 
$I_{F/G}(\vep_i,\vep_j)=\delta_{i,j}$, 
and $b$ is a fixed generator of the lattice $\rad(I_{F/G})\cap Q(R^{\mc}/G)$ of rank $1$. 
Note that in this case $(R^{\mc}/G)^{\mc}$ consists of all middle and long roots of $R^{\mc}/G$. 

Assume  there exists a reduced subsystem $R_1$ of $R^{\mc}/G$ such that 
$(R^{\mc}/G)^{\mc}\\ \subsetneq R_1$. Then $R_1$ must contain a short root $\alpha$. Since $R_1$
is a root system, $-\alpha$ is an element of $R_1$. So, we may assume 
$\alpha=\vep_i+mb$ for some $1\leq i\leq l$ and $m\in\Z$. Since $(R^{\mc}/G)^{\mc}$ is a 
subsystem of $R_1$, the root system $R_1$ contains all middle and long roots of 
$R^{\mc}/G$. Therefore, $2\alpha=2\vep_i+2mb$ is an element of $R_1$. 
This contradicts the reducedness of $R_1$. Thus, there is no such an $R_1$.
\end{proof}

By the above claim, $\pi_G(R^{\mc}\cap L_{\Pi^{\mc}})$ should be equal to 
$\pi_G(\W[\Pi^{\mc}]\cdot \Pi^{\mc})=(R^{\mc}/G)^{\mc}$. By the same reason as in the first case, we have
$R^{\mc}\cap L_{\Pi^{\mc}}=\W[\Pi^{\mc}].\Pi^{\mc}$, which completes the proof of statement (2).
\end{proof}

\begin{cor}\label{cor:nr-countings0}
We have the following description of the set $R$ of roots{\rm :}
\begin{equation}\label{description-R0}
R=\Big(\bigsqcup_{\gamma\in R^{\dc}\cap L_{\Pi^{\dc}}}(\gamma+\Z k(\gamma)a)\Big)
\bigcup
\Big(\bigsqcup_{\gamma\in R^{\mc}\cap L_{\Pi^{\mc}}}
(\gamma+\Z k(\gamma)a)\Big).
\end{equation}
\end{cor}
\begin{proof}
By Lemma \ref{lemma:paired-reduced-basis} (2) and Proposition 
\ref{prop:description-R0} (2), it follows that
\[\begin{aligned}
&\pi_G(R^{\dc}\cap L_{\Pi^{\dc}}) \cup \pi_G(R^{\mc}\cap L_{\Pi^{\mc}}) \\
=
&\pi_G(\W[\Pi^{\dc}].\Pi^{\dc})\cup \pi_G(\W[\Pi^{\mc}].\Pi^{\mc}) \\
=
&(R/G)^{\dc}\cup (R/G)^{\mc}=R/G.
\end{aligned}\]
Hence, the statement is an immediate consequence of Lemma 
\ref{lemma:countings1} (1).
\end{proof}

On the other hand, by replacing $\W[\Pi^{\mc}].\Pi^{\mc}=R^{\mc}\cap L_{\Pi^{\mc}}$ with $\W[\Pi^{\dc}].\Pi^{\mc}$, 
we have another description of $R$. 

\begin{cor}\label{cor:description-R}
{\rm (1)} The canonical projection $\pi_G:F\to F/G$ induces a bijection 
$\pi_G|_{\W[\Pi^{\dc}].\Pi^{\mc}}:\W[\Pi^{\dc}].\Pi^{\mc}
\xrightarrow{\sim} (R/G)^{\mc}$.
\vskip 1mm
\noindent
{\rm (2)} There is another description of the set $R$ of roots{\rm :}
\begin{equation}\label{description-R1}
R=\Big(\bigsqcup_{\gamma\in \W[\Pi^{\dc}].\Pi^{\dc}}(\gamma+\Z k(\gamma)a)\Big)
\bigcup\Big(\bigsqcup_{\gamma\in \W[\Pi^{\dc}].\Pi^{\mc}}
(\gamma+\Z k(\gamma)a)\Big).
\end{equation}
\end{cor}
\begin{proof}
(1) It is easy to see that the map $\pi_G|_{\W[\Pi^{\dc}].\Pi^{\mc}}$ is surjective. Indeed, by Lemma 
\ref{lemma:paired-reduced-basis} (1), (2), we have 
$\pi_G(\W[\Pi^{\dc}].\Pi^{\mc})=\pi_G(\W[\Pi^{\mc}].\Pi^{\mc})\\=(R^{\mc}/G)^{\mc}$. Therefore, it is enough {to} show that the map is injective.

Recall the definition $\Pi^{\mc}:=(\Pi^{\dc})^{pr}$. More precisely, for $0\leq i\leq l$, 
$\alpha_i'\in\Pi^{\mc}$ is defined by
\[\alpha_i':=\begin{cases}
\alpha_i & \text{if $2\overline{\alpha_i}\not\in R/G$},\\
2\alpha_i+p_i a \text{ for some $p_i \in\Z_{\leq 0}$} & 
\text{if $2\overline{\alpha_i}\in R/G$}.
\end{cases}\]
Recall  also that the set $\Pi_G^{\dc}:=\{\overline{\alpha_i}\}_{0\leq i\leq l}$ \big({\it resp}.
$\Pi_G^{\mc}:=\big\{\overline{\alpha_i'}\big\}_{0\leq i\leq l}$\big) is a simple 
system of the affine root system $(R^{\dc}/G)^{\dc}=(R/G)^{\dc}$({\it resp}. $(R^{\mc}/G)^{\mc}=(R/G)^{\mc}$). 

Let $\gamma_1=w_1(\alpha_{j_1}'),
\gamma_2=w_2(\alpha_{j_2}')$
be elements in $\W[\Pi^{\dc}].\Pi^{\mc}$ ($w_1,w_2\in \W[\Pi^{\dc}]$)
such that 
\begin{equation}\label{gamma1=gamma2}
\overline{\gamma_1}=\overline{\gamma_2}\quad \Longleftrightarrow\quad
\overline{w_1}(\overline{\alpha_{j_1}'})=\overline{w_2}(\overline{\alpha_{j_2}'}).
\end{equation}
Here, $\overline{w_q}\in \overline{\W[\Pi^{\dc}]}=W(R/G)$ for each $q=1,2$. 
Our goal is to show $\gamma_1=\gamma_2$. 

Note that \eqref{gamma1=gamma2} is equivalent to the condition that
$\overline{\alpha_{i_1}'}$ and $\overline{\alpha_{i_2}'}$ are included in the same 
$W(R/G)$-orbit. Let us partition
the set $\Pi^{\mc}$ into the following two pieces:
\[\Pi^{\mc}=(\Pi^{\dc}\cap\Pi^{\mc})\amalg (\Pi^{\mc}\setminus \Pi^{\dc}).\]
Taking the image under $\pi_G$, we have
\[\Pi^{\mc}_G=(\Pi^{\dc}_G\cap\Pi^{\mc}_G)\amalg (\Pi^{\mc}_G\setminus \Pi^{\dc}_G).\] 
Assume $\overline{\alpha_{j_1}'}\in \Pi^{\mc}_G\setminus \Pi^{\dc}_G$. Then, 
in the theory of non-reduced affine root systems, it is known that the $W(R/G)$-orbit 
passing through $\overline{\alpha_{j_1}'}$ does not intersect with other 
$W(R/G)$-orbits. Therefore, it is enough to consider the following two cases: 
\begin{enumerate}
\item[(a)] $\overline{\alpha_{j_1}'}
=\overline{\alpha_{j_2}'}\in \Pi^{\mc}_G\setminus \Pi^{\dc}_G$; 
\item[(b)] both $\overline{\alpha_{j_1}'}$ and $\overline{\alpha_{j_2}'}$ belong to 
$\Pi^{\dc}_G\cap \Pi^{\mc}_G$.
\end{enumerate}
\vskip 1mm
\noindent
(a) The condition $\overline{\alpha_{j_1}'}=\overline{\alpha_{j_2}'}$ implies 
$\alpha_{j_1}'=\alpha_{j_2}'$. 
Therefore,we have $j:=j_1=j_2$, and
denote the element $\alpha_{j_1}'=\alpha_{j_2}'$ by $\alpha_j'=2\alpha_j+p_ja$. 
Hence, we have
\begin{equation}\label{split-gamma}
\gamma_q=2w_q(\alpha_j)+p_ja\quad \text{ for }q=1,2.
\end{equation}
Since $2w_q(\alpha_j)\in L_{\Pi^{\dc}}$ and 
$\pi_G|_{L_{\Pi}^{\dc}}:L_{\Pi^{\dc}}\xrightarrow{\sim} F/G$ is an isomorphism, condition
\eqref{gamma1=gamma2} implies $2w_1(\alpha_j)=2w_2(\alpha_j)$. Thus, we 
have
\[\gamma_1=2w_1(\alpha_j)+p_ja=2w_2(\alpha_j)+p_ja=\gamma_2.\]
(b) In this case, $\alpha_{j_q}'=\alpha_{j_q}$ and 
$\gamma_p=w_p(\alpha_{j_q})\in L_{\Pi^{\dc}}$ for $q=1,2$. Furthermore, 
 condition \eqref{gamma1=gamma2} is equivalent to
$\overline{w_1}(\overline{\alpha_{j_1}})=\overline{w_2}(\overline{\alpha_{j_2}})$. 
Hence, by using the isomorphism $\pi_G|_{L_{\Pi^{\dc}}}:L_{\Pi^{\dc}}\xrightarrow{\sim} F/G$, 
the condition 
$\overline{w_1}(\overline{\alpha_{j_1}})=\overline{w_2}(\overline{\alpha_{j_2}})$
implies 
\[\gamma_1=w_1(\alpha_{j_1})=w_2(\alpha_{j_2})=\gamma_2.\]

Thus, we complete the proof of injectivity of the map $\pi_G|_{\W[\Pi^{\dc}].\Pi^{\mc}}$, and
statement (1) is proved.
\vskip 1mm
\noindent
(2) By Lemma \ref{lemma:paired-reduced-basis} (1), we have 
$\pi_G(\W[\Pi^{\dc}].\Pi^{\mc})=\pi_G(\W[\Pi^{\mc}].\Pi^{\mc})$. Hence, the statement 
follows from the same reason as the previous corollary. 
\end{proof}

\begin{rem}
Since {the} two groups $W(\Pi^{\dc})$ and $W(\Pi^{\mc})$ are different 
in general, 
the descriptions \eqref{description-R0} are \eqref{description-R1} are not same.
\end{rem}

\subsection{Choice of the splitting $F=L_{\Pi^{\dc}}\oplus G$}\label{sect:radZ}
By using the method in the proof of the above corollary, one can rewrite the
description \eqref{description-R1} by
\begin{equation}\label{description-R1-1}
R=\Big(\bigsqcup_{\gamma\in \W[\Pi^{\dc}].\Pi^{\dc}}(\gamma+\Z k(\gamma)a)\Big)
\bigcup\Big(\bigsqcup_{\gamma\in \W[\Pi^{\dc}].(\Pi^{\mc} \setminus \Pi^{\dc})}
(\gamma+\Z k(\gamma)a)\Big).
\end{equation}

As we already mentioned above, 
in the rest of the article, 
we mainly use the expression
\eqref{description-R1-1} of $R$ 
(or its modified version \eqref{description-R-2}).
On the other hand, 
the expression \eqref{description-R0} is used only in the proof of Lemma 
\ref{lemma:delta_b} below.\\

Let us briefly explain why we choose  the expression  \eqref{description-R1-1}.
Roughly speaking, the reason is that this expression  is compatible with 
the decomposition $F=L_{\Pi^{\dc}}\oplus G$ of the vector space $F$. 
Indeed, since $\W[\Pi^{\dc}].\Pi^{\dc}\subset L_{\Pi^{\dc}}$, the compatibility of the first term
$\bigsqcup_{\gamma\in \W[\Pi^{\dc}].\Pi^{\dc}}(\gamma+\Z k(\gamma)a)$ of 
\eqref{description-R1-1} is obvious. Let us discuss the second term. By using \eqref{split-gamma}, we have
\begin{align*}
&\bigsqcup_{\gamma\in \W[\Pi^{\dc}].(\Pi^{\mc}\setminus \Pi^{\dc})}
(\gamma+\Z k(\gamma)a) \\
=
&\bigsqcup_{\begin{subarray}{c}
\gamma=w(\alpha_j'):\\
w\in \W[\Pi^{\dc}],\, \alpha_j'\in \Pi^{\mc}\setminus \Pi^{\dc}
\end{subarray}}
\big(2w(\alpha_j)+\big(p_j+\Z k(\gamma)\big)a\big).
\end{align*}
Since $2w(\alpha_j)\in L_{\Pi^{\dc}}$, the second term 
of \eqref{description-R1-1} is also 
compatible with the decomposition $F=L_{\Pi^{\dc}}\oplus G$. 
In other words, in the expression  \eqref{description-R1-1}, the only 
hyperplane
$L_{\Pi^{\dc}}$ is used. This fact is suitable for studying the detailed structure of the set
$R$ of roots.  \\

For comparison, let us recall the expression \eqref{description-R0}:
\[R=\Big(\bigsqcup_{\gamma\in R^{\dc}\cap L_{\Pi^{\dc}}}(\gamma+\Z k(\gamma)a)\Big)
\bigcup
\Big(\bigsqcup_{\gamma\in R^{\mc}\cap L_{\Pi^{\mc}}}
(\gamma+\Z k(\gamma)a)\Big).\]
The first term of the right hand side is compatible
with the decomposition $F=L_{\Pi^{\dc}}\oplus G$. On the other hand, the second term
is compatible with the decomposition $F=L_{\Pi^{\mc}}\oplus G$, but not with 
$F=L_{\Pi^{\dc}}\oplus G$. That is, two distinct hyperplanes $L_{\Pi^{\dc}}$
and $L_{\Pi^{\mc}}$ are used in this description. There is no problem when studying the first
term and the second term individually, but it is complicated to analyze 
the whole structure of $R$.  This is the reason why the description 
\eqref{description-R0} is {\it not} suitable for studying the structure of $R$. \\

Recall that $\Pi_G^{\dc}=\pi_G(\Pi^{\dc})$ is a simple system of $R/G$. 
As is well-known, 
there exists a unique set
$\{n_{\alpha}\}_{\alpha\in\Pi^{\dc}}$ of positive integers such that 
\begin{itemize}
\item[(a)] $\delta:=
\sum_{\alpha\in \Pi^{\dc}} n_\alpha \overline{\alpha}$
is a generator of the lattice $Q(R/G)\cap \mathrm{rad}(I_{F/G})$ of 
rank $1$;
\vskip 1mm
\item[(b)] there exists a node $\alpha_0\in \Pi^{\dc}$ such that $n_{\alpha_0}=1$,
\end{itemize}
(see \cite{Kac1990}).
Define an element $\delta_b\in Q(\W[\Pi^{\dc}].\Pi^{\dc})\cap \mathrm{rad}(I)$ by
\begin{equation}\label{defn:delta_b}
\delta_b:=\sum_{\alpha\in \Pi^{\dc}}n_{\alpha} \alpha.
\end{equation}
By definition, it is immediate to see that 
\begin{itemize}
\item[(c)] $\delta_b$ is a generator 
of the lattice $Q(\W[\Pi^{\dc}].\Pi^{\dc})\cap \mathrm{rad}(I)$ of rank $1$.
\end{itemize}

\begin{lemma}\label{lemma:delta_b}
The lattice $\mathrm{rad}_\Z(I)=Q(R)\cap \mathrm{rad}(I)$ is
generated by $\delta_b$ and $a$. That is, 
\begin{equation}
\mathrm{rad}_\Z(I)=\Z \delta_b \bigoplus \Z a. 
\end{equation}
\end{lemma}
\begin{proof}
Since both $\delta_b$ and $a$ belong to $\mathrm{rad}_\Z (I)$, it is enough to
show $\mathrm{rad}_\Z(I)\subset \Z \delta_b \bigoplus \Z a$. 
By Proposition \ref{prop:description-R0} (2), we have
\begin{align*}
R&=\Big(\bigsqcup_{\gamma\in \W[\Pi^{\dc}].\Pi^{\dc}}(\gamma+\Z k(\gamma)a)\Big)
\bigcup \Big(\bigsqcup_{\gamma\in \W[\Pi^{\mc}].\Pi^{\mc}}(\gamma+\Z k(\gamma)a)\Big)\\
&\subset \Big(Q(\W[\Pi^{\dc}].\Pi^{\dc})\bigoplus\Z a\Big)\bigcup 
\Big(Q(\W[\Pi^{\mc}].\Pi^{\mc})\bigoplus\Z a\Big).
\end{align*}
It follows from \eqref{affine_root_lattices} that
\begin{align*}
&Q(\W[\Pi^{\dc}].\Pi^{\dc})\bigoplus\Z a=\Big(\bigoplus_{i=0}^l \Z\alpha_i\Big)\bigoplus\Z a, \\
&Q(\W[\Pi^{\mc}].\Pi^{\mc})\bigoplus\Z a=\Big(\bigoplus_{i=0}^l \Z\alpha_i'\Big)\bigoplus\Z a.
\end{align*}
By the definition of the prime map, we have
\[\alpha_i'\in \Z \alpha_i\bigoplus\Z a\quad \text{for every }0\leq i\leq l.\] 
Therefore, $Q(\W[\Pi^{\mc}].\Pi^{\mc})\bigoplus\Z a$ is a sublattice of 
$Q(\W[\Pi^{\dc}].\Pi^{\dc})\bigoplus\Z a$, and
\[\mathrm{rad}_{\Z}(I)=Q(R)\cap \mathrm{rad}(I)\subset 
\Big(Q(\W[\Pi^{\dc}].\Pi^{\dc})\bigoplus\Z a \Big)\cap \mathrm{rad}(I).\]
Since the right hand side $=\Z\delta_b\bigoplus \Z a$ by (c), we have 
$\mathrm{rad}_\Z(I)\subset \Z \delta_b\bigoplus \Z a$, as desired.
\end{proof}
\vskip 5mm

\subsection{Non-reduced counting numbers}\label{sect:NRCN}
In this subsection, we introduce a positive number $k^{\nr}(\alpha)$ for 
$\alpha\in\Pi^{\dc}\cup (\Pi^{\mc})^*$, called the ``{\it non-reduced counting number}'' of 
$\alpha$, and study its basic properties. This terminology plays a central role for 
defining the elliptic diagram for an arbitrary mERS 
(see $\S$ \ref{sect_elliptic-diagram-II}). Furthermore, 
we have an explicit formula for the counting number of $(\alpha')^\ast\in (\Pi^{\mc})^\ast$ 
in terms of the counting number 
of $\alpha\in \Pi^{\dc}$ (see Proposition \ref{prop:countings-prime}). 
This formula is used for describing the detailed structure of $R$ in the next 
subsection.     \\
 
Let us define the non-reduced counting numbers.
\begin{defn}
For an element $\alpha \in \Pi^{\dc}$, define a positive integer
$k^{\nr}((\alpha')^\ast)\in \Z_{>0}$\index[notations]{k711@$k^{\nr}((\alpha')^\ast)$} and 
a positive half integer 
$k^{\nr}(\alpha)\in\frac12 \Z_{>0}$ by the next equalities{\rm :}
\begin{align} 
(\alpha')^\ast=&((\alpha')^\ast:\alpha)_G\,\alpha + k^{\nr}((\alpha')^\ast)a, 
\label{defn:nr-countings_alpha'-ast}
\\
\alpha=&(\alpha:(\alpha')^\ast)_G\,(\alpha')^\ast-k^{\nr}(\alpha)a,
\label{defn:nr-countings_alpha}
\end{align}
where 
\begin{equation*}
(x:y)_G:=(\overline{x}:\overline{y}).
\end{equation*}
We call $k^{\nr}(\alpha)$\index[notations]{k712@$k^{\nr}(\alpha)$} the 
\textbf{non-reduced counting number}\index[index]{counting number!non-reduced 
counting number@non-reduced --}  
of an element $\alpha\in \Pi^{\dc}\cup (\Pi^{\mc})^\ast$.
\end{defn}

By definition, one has
\[((\alpha')^\ast:\alpha)_G \in \{1,2\}\quad \text{and}\quad 
((\alpha')^\ast:\alpha)_G\, (\alpha:(\alpha')^\ast)_G=1.\] 
\begin{lemma}\label{lemma:nr-countings1} 
Let $\alpha$ be an element of $\Pi^{\dc}$. 
\vskip 1mm
\noindent
{\rm (1)} $k^{\mathrm{nr}}((\alpha')^\ast)=
((\alpha')^\ast:\alpha)_G\, k^{\mathrm{nr}}(\alpha)$.
\vskip 1mm
\noindent
{\rm (2)}
$k^{\mathrm{nr}}((\alpha')^\ast)=\begin{cases}
\frac12 k((\alpha')^\ast) & \text{if }\alpha'\ne 2\alpha\text{ and }
\overline{\alpha'}=2\overline{\alpha},\\
k((\alpha')^\ast) & \text{otherwise}.
\end{cases}$
\end{lemma}

\begin{proof}
(1) By definition, we have
\[\begin{aligned}
(\alpha')^\ast&=
((\alpha')^\ast:\alpha)_G\,
\big\{(\alpha:(\alpha')^\ast)_G\,(\alpha')^\ast-k^{\nr}(\alpha)a\big\} 
+ k^{\nr}((\alpha')^\ast)a\\
&=(\alpha')^\ast+\big\{k^{\nr}((\alpha')^\ast)
-((\alpha')^\ast:\alpha)_G \, k^{\nr}(\alpha)\big\}a.
\end{aligned}\] 
Thus, we have the statement.
\vskip 1mm
\noindent
(2) Assume $((\alpha')^\ast:\alpha)_G=1$. Since 
$((\alpha')^\ast:\alpha)_G=(\alpha':\alpha)_G$, this assumption is equivalent to
$\overline{\alpha'}=\overline{\alpha}$. In this case, $\alpha'=\alpha$. Therefore, 
the statement is obviously verified. Otherwise, we have 
$((\alpha')^\ast:\alpha)_G=2$, and 
$\overline{\alpha'}=2\overline{\alpha}$. 
By the definition of $\alpha'$, 
there exists a unique non-positive integer $p\in\Z_{\leq 0}$ such that
$\alpha'=2\alpha+pa$ and 
\begin{equation*}
(\alpha')^\ast=\alpha'+k(\alpha')a=2\alpha+(k(\alpha')+p)a\quad
\text{with }k(\alpha')+p>0.
\end{equation*}
Comparing 
this equality with the definition $k^{\mathrm{nr}}((\alpha')^\ast)$, we have
\begin{equation}\label{nr-countings41}
k^{\mathrm{nr}}((\alpha')^\ast)=k(\alpha')+p.
\end{equation}
If $p=0\ (\Longleftrightarrow\, \alpha'=2\alpha)$, we have
\[k^{\mathrm{nr}}((\alpha')^\ast)=k(\alpha')=k((\alpha')^\ast)\]
by  \eqref{k=kast}, as desired. Otherwise, consider the image of the roots $\alpha'$ 
and $(\alpha')^\ast$ by the automorphism $(-\mathrm{Id}) \circ r_{\alpha}$, where 
$r_{\alpha}$ signifies the reflection with respect to the root $\alpha$. 
By direct computation, we see that both
$2\alpha-k^{nr}((\alpha')^\ast)a$ and $2\alpha-pa$ are roots. 
Thus, the only possible case is $p=-k^{\mathrm{nr}}((\alpha')^\ast)$. 
Substituting this result to \eqref{nr-countings41}, we have
\[2k^{\mathrm{nr}}((\alpha')^\ast)=k(\alpha')=k((\alpha')^*),\]
as desired.  
\end{proof} 

By Lemma  \ref{lemma:nr-countings1} and \eqref{k=kast}, 
the following formula is verified immediately: for $\alpha\in \Pi^{\dc}$,  we have
\begin{equation}\label{nr-countings_alpha1}
k^{\mathrm{nr}}(\alpha)=(\alpha\, :\, (\alpha')^\ast)_G\, k^\nr((\alpha')^*)=
\begin{cases}
\frac14 k(\alpha') & \text{if }\alpha'\ne 2\alpha\text{ and }
\overline{\alpha'}=2\overline{\alpha},\\
\frac12 k(\alpha') & \text{if }\alpha'=2\alpha
,\\
k(\alpha') & \text{otherwise}.
\end{cases}
\end{equation}
On the other hand, there is another description of the non-reduced counting
number $k^\nr(\alpha)$ for $\alpha\in \Pi^{\dc}$. 

\begin{prop}\label{prop:nr-countings1}
For every $\alpha\in \Pi^{\dc}$, the next formula is valid{\rm :}
\begin{equation}\label{nr-countings_alpha2}
k^{\mathrm{nr}}(\alpha)=
\begin{cases}
\frac12 k(\alpha) & \text{if } (\alpha')^\ast\ne 2\alpha^\ast
\text{ and }
\overline{\alpha'}= 2\overline{\alpha},\\
k(\alpha) & \text{otherwise}.
\end{cases}
\end{equation}
\end{prop}

\begin{proof}
If $\overline{\alpha'}\ne 2\overline{\alpha}$, one has $\alpha'=\alpha$ as in the 
proof of the statement (2) of Lemma \ref{lemma:nr-countings1}.
Therefore, the statement is a direct consequence of the formula 
\eqref{nr-countings_alpha1}, and it is
enough to show \eqref{nr-countings_alpha2} under the assumption that 
$\overline{\alpha'}= 2\overline{\alpha}$.\\
 
Assume $(\alpha')^*=2\alpha^*$. Then, $2\alpha^*=2\alpha+2k(\alpha)a$
is a root. By the definition \eqref{defn:nr-countings_alpha'-ast}
of $k^{\nr}((\alpha')^\ast)$, we have
$k^{\nr}((\alpha')^\ast)=2k(\alpha)$.
By the statement (1) of Lemma \ref{lemma:nr-countings1}, 
we have 
\begin{equation*}
k^{\nr}(\alpha)=\dfrac12 k^{\nr}((\alpha')^\ast)=k(\alpha),
\end{equation*}
as desired. \\

Otherwise, we have $(\alpha')^*\ne 2\alpha^\ast$. By the
assumption $\overline{\alpha'}= 2\overline{\alpha}$, one has
\begin{equation}\label{nr-countings0}
\alpha'=2\alpha+pa\quad (p\in \Z_{\leq 0})\quad\text{and}\quad
(\alpha')^\ast=\alpha'+k(\alpha')a=2\alpha+(k(\alpha')+p)a.
\end{equation}
Let us consider the following two cases separately.
\vskip 1mm
{\bf Case (i)}\, :\, $\big(\pi_G^{-1}(\overline{\alpha})\cap R\big) \setminus \frac12 R_l=\emptyset$;
\vskip 1mm
{\bf Case (ii)}\, :\, $\big(\pi_G^{-1}(\overline{\alpha})\cap R\big) \setminus \frac12 R_l\ne \emptyset$.
 
\vskip 3mm
\noindent
{\bf Case (i)}. 
Assume $\big(\pi_G^{-1}(\overline{\alpha})\cap R\big)\setminus \frac12 R_l=\emptyset$. It means
that 
\begin{equation}\label{eq:case(i)-1}
\pi_G^{-1}(\overline{\alpha})\cap R=\alpha+\Z k(\alpha)a\subset \frac12 R_l. 
\end{equation}
Especially, we have $2\alpha\in R$. This means that $p=0$ in \eqref{nr-countings0},
and $\alpha'=2\alpha$. Thus, we have
\begin{equation}\label{nr-countings10}
(\alpha')^\ast=(2\alpha)^{\ast}=2\alpha+k(2\alpha)a\in R.
\end{equation}
Moreover, it follows from \eqref{eq:case(i)-1} that
$\alpha^\ast=\alpha+k(\alpha)a\in \pi_G^{-1}(\overline{\alpha})\cap R
\subset \frac12 R_l$. 
Hence, we have
\[2\alpha^*=2\alpha+2k(\alpha)a\in R.\]
This implies $2k(\alpha) \in k(2\alpha)\Z_{>1}$ by the assumption 
$(\alpha')^\ast \neq 2\alpha^\ast$. As $(-\mathrm{Id})\circ r_{(\alpha')^\ast}$ preserves 
$R$ and $(-\mathrm{Id})\circ r_{(\alpha')^\ast}(\alpha)=\alpha+k(2\alpha)a$, one sees 
that $k(2\alpha) \in k(\alpha)\Z$ since $R \cap (\alpha+\Z a)=\alpha+\Z k(\alpha)a$. 
Hence, we obtain 
\begin{equation}\label{nr-countings11}
k(\alpha)=k(2\alpha).
\end{equation}
As $\alpha'=2\alpha\in R$, the formula \eqref{nr-countings_alpha1} implies
\[k^{\mathrm{nr}}(\alpha)=\frac12 k(\alpha')=\frac12 k(2\alpha)=
\frac12 k(\alpha).\]
Thus, we have shown \eqref{nr-countings_alpha2} in this case.  
\vskip 3mm
\noindent
{\bf Case (ii)}. If $\big(\pi_G^{-1}(\overline{\alpha})\cap R\big) \setminus \frac12 R_l\ne\emptyset$,
the element $\alpha\in \Pi^{\dc}$ belongs to 
$\big(\pi_G^{-1}(\overline{\alpha})\cap R\big) \setminus \frac12 R_l$ by the 
condition \eqref{condition-Pi'}. 
Thus,  we have $2\alpha\not\in R$, and
\begin{equation}\label{(5.5.6)-II}
k^{\nr}(\alpha)=\frac14 k(\alpha')
\end{equation}
by the formula \eqref{nr-countings_alpha1}. 
On the other hand, since $2\alpha\not \in R$, we have $\alpha'\ne 2\alpha$. 
This means
\[\alpha'=2\alpha+pa\quad \text{with}\quad p\in \Z_{< 0}\] 
in \eqref{nr-countings0}.
By a similar argument to the last part 
of the proof Lemma \ref{lemma:nr-countings1} (2), it follows that 
\begin{equation}\label{p=-k(alpha'/2)}
p=-\frac12 k(\alpha').
\end{equation}
Note that
$(-\mathrm{Id})\circ r_{\alpha'}(\alpha)=\alpha+pa$
is a root. By \eqref{p=-k(alpha'/2)}, we have
\begin{equation}\label{k(alpha')-1}
p=-\dfrac12 k(\alpha')\in k(\alpha)\Z\quad \Longleftrightarrow \quad 
k(\alpha')\in 2k(\alpha)\Z. 
\end{equation}
Similarly, since 
\[(-\mathrm{Id})\circ r_{\alpha^\ast}(\alpha')=\alpha'+(-2p+4k(\alpha))a\]
is also a root, we have
\begin{equation}\label{k(alpha')-2}
-2p+4k(\alpha)=k(\alpha')+4k(\alpha)\in k(\alpha')\Z\quad \Longleftrightarrow \quad 
4k(\alpha)\in k(\alpha')\Z. 
\end{equation}
By \eqref{k(alpha')-1} and \eqref{k(alpha')-2}, it 
follows that 
\[k(\alpha')\in \{2k(\alpha),4k(\alpha)\}.\]

Suppose that $k(\alpha')=4k(\alpha)$. In this case, as $\alpha'=2\alpha-2k(\alpha)a$, hence 
$(\alpha')^\ast=2\alpha+2k(\alpha)a=2\alpha^\ast$ which contradicts to our assumption. 
Thus, we have
\begin{equation}\label{nr-countings8}
k(\alpha')=2k(\alpha)
\end{equation} 
and it follows from \eqref{(5.5.6)-II} that 
$k^{\nr}(\alpha) = \frac{1}{2}k(\alpha)$, as desired. 
\end{proof}
\vskip 3mm
Let us rewrite the conditions for $\alpha\in \Pi^{\dc}$ in the statement (2) of 
Lemma \ref{lemma:nr-countings1} and one in Proposition \ref{prop:nr-countings1}
, for later use.
\begin{cor}\label{cor:nr-countings1}
For $\alpha\in \Pi^{\dc}$, we have
\begin{align}
k^{\mathrm{nr}}((\alpha')^\ast)&=
\begin{cases}
\frac12 k((\alpha')^\ast) & \text{if }2\alpha\not\in R\text{ and }
2\overline{\alpha}\in R/G,\\
k((\alpha')^\ast) & \text{otherwise},
\end{cases}\label{nr-countings_alpha3}  
\\
k^{\mathrm{nr}}(\alpha)&=
\begin{cases}
\frac12 k(\alpha) & \text{if }\frac12(\alpha')^\ast\not\in R\text{ and }
\frac12\overline{(\alpha')^\ast}\in R/G,\\
k(\alpha) & \text{otherwise}.
\end{cases}\label{nr-countings_alpha4}
\end{align}
\end{cor}

\begin{proof}
Since \eqref{nr-countings_alpha3} is proved in a similar way to 
\eqref{nr-countings_alpha4} with an easier argument, we only give a proof of
the latter. 
For $\alpha\in\Pi^{\dc}$, note  that the following equivalences 
are verified immediately by the definition of the prime map: 
\begin{equation}\label{nr-countings-cor1}
\overline{\alpha'}=2\overline{\alpha}\quad \Longleftrightarrow\quad 
2\overline{\alpha}\in R/G\quad \Longleftrightarrow\quad 
\frac12 \overline{(\alpha')^\ast}\in R/G.
\end{equation}

It is enough to show the following claim for proving \eqref{nr-countings_alpha4}:
\vskip 1mm
\noindent
{\bf Claim.} {\it Under the assumption that
$\overline{\alpha'}= 2\overline{\alpha}$, 
the next two conditions are equivalent.
\vskip 1mm
\noindent
{\rm (a)} $(\alpha')^\ast=2\alpha^\ast$.
\vskip 1mm
\noindent
{\rm (b)} $\frac12 (\alpha')^\ast \in R$.} 
\vskip 3mm
\noindent
Let us prove this claim. Since (a) $\Longrightarrow$ (b) is trivially verified, it is enough 
to show condition (b) implies (a). For that purpose, assume 
$(\alpha')^\ast\ne 2\alpha^\ast$. As
$\overline{\alpha'}= 2\overline{\alpha}$, we have 
\begin{equation}\label{eq:claim}
\frac12 (\alpha')^*=\alpha+\frac12 k(\alpha)a
\end{equation}
by the formulas in the proof of the last proposition.
Indeed, if $\big(\pi_G^{-1}(\overline{\alpha})\cap R\big) \setminus \frac12 R_l=\emptyset$, it follows from \eqref{nr-countings10} and \eqref{nr-countings11}.
Otherwise, it follows from \eqref{p=-k(alpha'/2)} that
$\alpha'=2\alpha-\frac12 k(\alpha')a$. Therefore, we have
\[(\alpha')^*=\alpha'+k(\alpha')a=2\alpha+\dfrac12 k(\alpha')a\quad \Longleftrightarrow
\quad \frac12 (\alpha')^\ast=\alpha+\frac14k(\alpha')a.\]
Substituting \eqref{nr-countings8} to the right hand side, we have  
\eqref{eq:claim}. 

Hence, we have 
\[\frac12\overline{(\alpha')^\ast}=\overline{\alpha}.\] 
Since $R\cap \pi_G^{-1}(\overline{\alpha})=\alpha+\Z k(\alpha)a$,  the 
element $\frac12 (\alpha')^*$ should not belong to $R$. Thus, we have the
claim, and the proof of the corollary is completed.
\end{proof}

Combining the above results, we have the following proposition.

\begin{prop}\label{prop:countings-prime}
Let $\alpha\in \Pi^{\dc}$. If $2\overline{\alpha}\in R/G\ \left(\Longleftrightarrow\ 
\frac12 \overline{(\alpha')^\ast}\in R/G\right)$, we have 
\begin{equation}\label{countings-alpha-prime}
k((\alpha')^\ast)=\begin{cases}
2k(\alpha) & \text{if } 2\alpha\not\in R\text{ and } \frac12 (\alpha')^\ast\not\in R,\\
4k(\alpha) & \text{if } 2\alpha\not\in R\text{ and }\frac12 (\alpha')^\ast\in R,\\
k(\alpha) & \text{if } 2\alpha\in R\text{ and }\frac12 (\alpha')^\ast\not\in R,\\
2k(\alpha) & \text{if } 2\alpha\in R\text{ and }\frac12 (\alpha')^\ast\in R.
\end{cases}
\end{equation}
Otherwise, 
\begin{equation}
k((\alpha')^\ast)=k(\alpha). 
\end{equation}
\end{prop}

\begin{proof}
The statement follows immediately from Lemma \ref{lemma:nr-countings1} (2) 
and Corollary \ref{cor:nr-countings1}.
\end{proof}
The following lemma is useful for later discussion. 

\begin{lemma}\label{lemma:translation}
For every $\lambda\in F$, we have
\begin{equation}\label{eq:translation}
r_{\alpha}r_{(\alpha')^\ast}(\lambda) = \lambda - I(\lambda, \alpha^\vee)
k^{\nr}(\alpha)a.
\end{equation}
\end{lemma}
\begin{proof}
The statement follows {from} direct computation. 
Indeed, by the definition,  
\begin{align*}
r_{\alpha} r_{(\alpha')^\ast}(\lambda) &= 
r_{\alpha}\left(\lambda - I(\lambda,  ((\alpha')^\ast)^\vee)(\alpha')^\ast \right)\\
&=\lambda - I(\lambda, \alpha^\vee)\alpha - I(\lambda,((\alpha')^\ast)^\vee)
\left((\alpha')^\ast - I((\alpha')^\ast, \alpha^\vee)\alpha\right). 
\end{align*}
Since 
\begin{align*}
I(\lambda,((\alpha')^\ast)^\vee) = 
&(\alpha:(\alpha')^\ast)_G\,
I(\lambda, \alpha^\vee), \\
I((\alpha')^\ast,\alpha^\vee) =
&((\alpha')^\ast:\alpha)_G \,
I(\alpha,\alpha^\vee) = 2((\alpha')^\ast:\alpha)_G, 
\end{align*}
one has
\begin{align*}
r_{\alpha} r_{(\alpha')^\ast}(\lambda)  &= 
\lambda - I(\lambda, \alpha^\vee)\left((\alpha:(\alpha')^\ast)_G\, 
(\alpha')^\ast - \alpha\right) \\
&=\lambda - I(\lambda, \alpha^\vee)k^{\nr}(\alpha)a.
\end{align*}
as desired. 
\end{proof}

\begin{rem}{\rm By Lemma \ref{lemma:nr-countings1} (1), the following formula is
verified:
\[ I(\lambda,((\alpha')^\ast)^\vee)k^{\nr}((\alpha')^\ast)=
  I(\lambda,\alpha^\vee)(\alpha:(\alpha')^\ast)_Gk^{\nr}((\alpha')^\ast)=
  I(\lambda,\alpha^\vee)k^{\nr}(\alpha). \]
}\end{rem}
\vskip 5mm
\subsection{A presentation of $R$}\label{sect:description-R}
In this subsection, we give 
an explicit description of the set $R$ of roots (see \eqref{description-R-2} in Proposition
\ref{prop:description-R} below). In this description, every element of $R$ is written
using information on the affine Weyl group $\W[\Pi^{\dc}]$ and 
counting numbers of the elements of $\Pi^{\dc}$. This fact is essentially used 
to define the ``elliptic diagram'' of a mERS $(R,G)$ in $\S$ \ref{sect_elliptic-diagram-II}.

\begin{prop}\label{prop:description-R}
Let $(\Pi^{\dc},\Pi^{\mc})$ be a paired simple system of a mERS $(R,G)$. 
\vskip 1mm
\noindent
{\rm (1)} Let $\gamma=w((\alpha')^\ast)$ for $\alpha\in\Pi^{\dc}$ and $w\in \W[\Pi^{\dc}]$. 
Then 
\begin{equation}\label{counting-set-expform}
\gamma+\Z k(\gamma)a=\begin{cases}
2w\alpha+(2\Z+1)k(\alpha)a & \text{if $2\alpha\not\in R$, 
$2\overline{\alpha}\in R/G$ and $\frac12 (\alpha')^\ast\not\in R$},\\
2(w\alpha+(2\Z+1)k(\alpha)a) & \text{if $2\alpha\not\in R$, 
$2\overline{\alpha}\in R/G$ and $\frac12 (\alpha')^\ast\in R$},\\
2w\alpha+\Z k(\alpha)a & \text{if $2\alpha\in R$ and 
$\frac12 (\alpha')^\ast\not\in R$},\\
2(w\alpha+\Z k(\alpha)a) & \text{if $2\alpha\in R$, and 
$\frac12 (\alpha')^\ast\in R$},\\
w\alpha+\Z k(\alpha)a & \text{if $2\overline{\alpha}\not\in R/G$}.\\
\end{cases}
\end{equation}
{\rm (2)} There exists the following description of the set $R$ of roots{\rm :}
\begin{equation}\label{description-R-2}
R=\Big(\bigsqcup_{\gamma\in \W[\Pi^{\dc}].\Pi^{\dc}}(\gamma+\Z k(\gamma)a)\Big)
\bigcup
\Big(\bigsqcup_{\gamma\in \W[\Pi^{\dc}].(\Pi^{\mc}\setminus \Pi^{\dc})^\ast}
(\gamma+\Z k(\gamma)a)\Big)
\end{equation}
\end{prop} 
\begin{proof}
\noindent
(1) By the definition \eqref{defn:nr-countings_alpha'-ast} of 
$k^{\nr}((\alpha')^\ast)$ and \eqref{nr-countings_alpha3}, 
we have 
\[\begin{aligned}
&w((\alpha')^\ast)+\Z k(w((\alpha')^\ast))a \\
=
&((\alpha')^\ast:\alpha)_G w(\alpha)+k^{\nr}((\alpha')^\ast)a
+\Z k((\alpha')^\ast)a\\
=
&\begin{cases}
2w(\alpha)+(\Z+\frac12) k((\alpha')^\ast)a
& \text{if $2\alpha\not\in R$ and $2\overline{\alpha}\in R/G$},\\
2w(\alpha)+\Z k((\alpha')^\ast)a
& \text{if $2\alpha\in R$ and $2\overline{\alpha}\in R/G$},\\
w(\alpha)+\Z k((\alpha')^\ast)a& \text{otherwise}.
\end{cases}
\end{aligned}\]
Substituting \eqref{countings-alpha-prime} to the equality above, the statement
is obtained. 
\vskip 1mm
\noindent
{\rm (2)} The description \eqref{description-R-2} is an immediate consequence of 
\eqref{description-R1-1}. Indeed, let $\gamma=w((\alpha')^\ast)\in \W[\Pi^{\dc}].
(\Pi^{\mc}\setminus \Pi^{\dc})^\ast$. For such $\gamma$, we have
\begin{align*}
&\gamma+\Z k(\gamma)a \\
=
&w((\alpha')^\ast)+\Z k\big(w((\alpha')^\ast)\big)a\\
=
&w(\alpha'+k(\alpha')a)+\Z k\big(w(\alpha')\big)a\quad (\because\ \eqref{k=kast})\\
=
&w(\alpha')+k(w(\alpha'))a+\Z k\big(w(\alpha')\big)a\quad (\because\ \W[\Pi^{\dc}]\subset
\mathrm{Aut}(R)\text{ $\&$ Lemma \ref{lemma:countings1} (3)})\\
=
&w(\alpha')+\Z k\big(w(\alpha')\big)a.
\end{align*}
Therefore, we get  
\[\bigsqcup_{\gamma\in \W[\Pi^{\dc}].(\Pi^{\mc}\setminus \Pi^{\dc})^\ast}
(\gamma+\Z k(\gamma)a)=\bigsqcup_{\gamma\in \W[\Pi^{\dc}].(\Pi^{\mc}\setminus \Pi^{\dc})}
(\gamma+\Z k(\gamma)a),\]
and the description \eqref{description-R-2} coincides with that of 
\eqref{description-R1-1}.
\end{proof}

\begin{rem}
By definition, we have 
\begin{equation}\label{description-nw}
(\Pi^{\mc}\setminus \Pi^{\dc})^\ast
=\{(\alpha')^\ast\,|\,\alpha\in \Pi^{\dc},\, 2\overline{\alpha}\in R/G\}.
\end{equation}
\end{rem}

\begin{cor}\label{cor:countings-gcd}
The number $\mathrm{g.c.d.}\{k(\alpha)\, |\, \alpha\in \Pi^{\dc}\}$ is equal to $1$.
\end{cor}
\begin{proof}
It follows {from} Proposition \ref{prop:description-R}
that $\bigcup_{\alpha\in \Pi^{\dc}}\Z k(\alpha)a$ 
generates the lattice $Q(R)\cap G$ of rank $1$. Thus,
the statement is verified immediately. 
\end{proof}

\subsection{Uniqueness of paired simple system} \label{sect:unique-basis}
In  this subsection, we give a solution to the ``uniqueness problem'' of 
a simple system of $(R,G)$ stated in $\S$ \ref{sect_marking} 
(see Proposition \ref{prop:nr-counting-Autom1} 
and Proposition \ref{prop:nr-counting-Autom2} below).

\begin{prop}\label{prop:nr-counting-Autom1}
For each $i=1,2$, let $(R_i,G_i)$ be a mERS, and 
$(\Pi^{\dc}_i,\Pi^{\mc}_i)$ be a paired simple system of $(R_i,G_i)$. 
Fix a generator $a_i$ of the lattice $Q(R_i)\cap G_i$ of rank $1$. 
Assume $(R_1,G_1)$ and
$(R_2,G_2)$ are isomorphic. Then, there exists an isomorphism
$\varphi:(R_1,G_1)\xrightarrow{\sim}(R_2,G_2)$ such that 
$\varphi(a_1)=a_2$, 
$\varphi(\Pi_1^{\dc})=\Pi_2^{\dc}$ and $\varphi(\Pi_1^{\mc})=\Pi_2^{\mc}$. 
\end{prop}

It is enough to show the statement in the case where $(R_1,G_1)=(R_2,G_2)$
and $a_1=a_2$. Denote them by $(R,G)$ and $a$, respectively. 
Assume there exists $\varphi\in \Aut(R,G)$ such that 
$\varphi(\Pi_1^{\dc})=\Pi_2^{\dc}$. 
Then, by Lemma \ref{lemma:commutativity-prime}, we have
\[\varphi(\Pi^{\mc}_1)=\varphi((\Pi_1^{\dc})^{pr})=(\varphi(\Pi_1^{\dc}))^{pr}=(\Pi_2^{\dc})^{pr}=\Pi_2^{\mc}.\]
Therefore, the above proposition is reduced to the following.

\begin{prop}\label{prop:nr-counting-Autom2}
Let $\Pi^{\dc}_1$ and $\Pi_2^{\dc}$ be two simple systems of $(R^{\dc},G)$ which
satisfy condition \eqref{condition-Pi'}. 
Then, there exists an automorphism $\varphi\in \Aut(R,G)$ such that 
$\varphi(\Pi^{\dc}_1)=\Pi_2^{\dc}$. 
\end{prop}

In the case when $R/G$ is reduced, this statement was proved by K. Saito 
(\cite{Saito1985}, (6.2) Corollary). 
Our proof is a generalization of the original one which requires 
some adaptation because of the non-reducibility of $R/G$.\\

Before starting the proof of Proposition \ref{prop:nr-counting-Autom2}, we 
prepare the following two lemmas.
\begin{lemma}\label{lemma:nr-counting-Autom3}
Let $\Pi^{\dc}_1=\{\alpha_0,\ldots,\alpha_l\}$ and $\Pi_2^{\dc}=\{\beta_0,\ldots,\beta_l\}$ 
be two simple systems of $(R^{\dc},G)$ which satisfy condition \eqref{condition-Pi'}. 
Assume
\begin{equation}\label{assumtion:alpha=beta}
\overline{\alpha_i}=\overline{\beta_i}\quad\text{for every}\quad 0\leq i\leq l.
\end{equation}
We have the following.
\vskip 1mm
\noindent
{\rm (1)} An equality $k(\alpha_i)=k(\beta_i)$ holds for every $0\leq i\leq l$. 
\vskip 1mm
\noindent
{\rm (2)} The condition $2\alpha_i\in R$ is equivalent to $2\beta_i\in R$.
\vskip 1mm
\noindent
{\rm (3)} An equality $k((\alpha_i')^\ast)=k((\beta_i')^\ast)$ holds for every 
$0\leq i\leq l$.
\vskip 1mm
\noindent
{\rm (4)} Under the assumption  $2\overline{\alpha_i}=2\overline{\beta_i}\in R/G$, 
the following conditions are equivalent{\rm :}  
\begin{itemize}
\item[(a)] $\frac12 (\alpha_i')^\ast\in R$,
\vskip 1mm
\item[(b)] $\frac12 (\beta_i')^\ast\in R$. 
\end{itemize}
\end{lemma}

\begin{proof}
Note that there exist integers 
$m_i\ (0\leq i\leq l)$ such that
\begin{equation}\label{alpha-beta}
\beta_i=\alpha_i+m_i k(\alpha_i)a
\end{equation}
by condition \eqref{assumtion:alpha=beta}.

 Statement (1) is an easy consequence of  
Lemma \ref{lemma:countings1} (1).
Statement  (2) follows {from} 
\eqref{condition-Pi'} and \eqref{assumtion:alpha=beta}. Indeed, 
since $\pi_G^{-1}(\overline{\alpha_i})=\pi_G^{-1}(\overline{\beta_i})$ by
\eqref{assumtion:alpha=beta}, the
condition $\big(\pi_G^{-1}(\overline{\alpha_i})\cap R\big)\setminus \frac12 R_l\ne \emptyset$ is
equivalent to $\big(\pi_G^{-1}(\overline{\beta_i})\cap R\big)\setminus \frac12 R_l\ne \emptyset$.
Since both $\Pi_1^{\dc}$ and $\Pi_2^{\dc}$ satisfy  condition \eqref{condition-Pi'}, 
we have  statement (2). 

Let us prove statements (3) and (4). 
If $2\overline{\alpha_i}=2\overline{\beta_i}\not\in R/G$, we have 
$\alpha'_i=\alpha_i$ and $\beta_i'=\beta_i$. Thus, statement (3) is a direct
consequence of  statement (1), in this case. 
Statement (4) is not concerned in this case.
 
Assume $2\overline{\alpha_i}=2\overline{\beta_i}\in R/G$. Consider the
following three cases: 
\begin{itemize}
\item[(i)] $2\alpha_i\not\in R$ and $\frac12 (\alpha_i')^\ast \not \in R$, 
\vskip 1mm
\item[(ii)] $2\alpha_i\not\in R$ and $\frac12 (\alpha_i')^\ast \in R$, 
\vskip 1mm
\item[(iii)] $2\alpha_i\in R$. 
\end{itemize}
\vskip 1mm
\noindent
Case (i): Let us prove $\frac12 (\beta_i')^\ast\not\in R$. 
Assume $\frac12 (\beta_i')^\ast\in R$. 
Applying Proposition \ref{prop:countings-prime} to the paired simple system
$(\Pi_2^{\dc},\Pi_2^{\mc})$, we have
\[k(\beta'_i)=k((\beta_i')^\ast)=4k(\beta_i)=4k(\alpha_i).\]
Note that the last equality follows from  statement (1). 
Since $2\beta_i\not\in R$ by statement (2), we have the following description 
of $(\beta_i')^\ast\in R$:
\[(\beta_i')^\ast=2\beta_i+\frac12 k(\beta_i')a=2\alpha_i+2(m_i+1)k(\alpha_i)a.\]
Since both $\alpha_i$ and $(\alpha_i')^\ast$ are roots, the element 
$(r_{\alpha_i}r_{(\alpha_i')^\ast})^p((\beta_i')^\ast)$
is also a root for every $p\in\Z$.
By Lemma \ref{lemma:translation}, one has
\[(r_{\alpha_i}r_{(\alpha_i')^\ast})^p((\beta_i')^\ast)=(\beta_i')^\ast-2pk(\alpha_i)a
=2\alpha_i+2(m_i+1-p)k(\alpha_i)a.\] 
Setting $p=m_i+1$, the equality above indicates that $2\alpha_i$ belongs to $R$. 
This is a contradiction. Therefore, we have $\frac12 (\beta_i')^\ast\not\in R$, 
as desired. 

Furthermore,  applying Proposition \ref{prop:countings-prime} to the paired 
simple systems $(\Pi_1^{\dc},\Pi_1^{\mc})$ and $(\Pi_2^{\dc},\Pi_2^{\mc})$, we have
$k((\alpha_i')^\ast)=2k(\alpha_i)$ and $k((\beta_i')^\ast)=2k(\beta_i)$, respectively.
Hence, statement (1) implies (3), in this case.
\vskip 3mm
\noindent
Case (ii): Assume $\frac12 (\beta_i')^\ast\not\in R$. 
Exchanging the roles of
$\alpha_i$ and $\beta_i$ in the previous case, we have a contradiction. 
Thus, we get $\frac12 (\beta_i')^\ast\in R$, and 
\[k((\beta_i')^\ast)=4k(\beta_i)=4k(\alpha_i)=k((\alpha_i')^\ast).\]
Case (iii): By statement (2), we have $2\beta_i\in R$. Thus, 
\[\alpha'_i=2\alpha_i\in R\quad\text{and}\quad \beta'_i=2\beta_i\in R,\]
in this case.  Therefore, one obtains
\[\overline{\alpha_i'}=2\overline{\alpha_i}=2\overline{\beta_i}
=\overline{\beta'_i}\in R/G\]
and  
\[k((\alpha_i')^\ast)=k(\alpha_i')=k(\beta_i')=k((\beta_i')^\ast).\]
This is nothing but statement (3). 
Assume $\frac12 (\alpha_i')^\ast\in R$. By Proposition \ref{prop:countings-prime} 
for $(\Pi_1^{\dc},\Pi_1^{\mc})$, we have $k(\alpha_i')=2k(\alpha_i)$. 
Hence, by the above results, we get  
\[\frac12 (\beta_i')^\ast=\frac12\left(\beta_i'+k(\beta_i')a\right)
=\beta_i+\frac12k((\beta_i')^\ast)a=\alpha_i+(m_i+1)k(\alpha_i)a \in R.\]

Thus, we showed that the condition $\frac12 (\alpha_i)^\ast\in R$ implies
$\frac12 (\beta_i)^\ast\in R$ for every case. 
Exchanging
 the roles $\alpha_i$ and $\beta_i$, we have the 
converse
statement. Thus, the proof of statement (4) is completed.  
\end{proof}

\begin{lemma}\label{lemma:nr-counting-Autom4}
Under the same setting 
as in
the previous lemma, 
there exists an automorphism $\psi_1\in \Aut(R,G)$ such that 
\begin{equation}\label{condition-Autom-alpha=beta}
\psi_1(\alpha_i)=\beta_i\quad\text{for every}\quad 0\leq i\leq l.
\end{equation} 
\end{lemma}
\begin{proof}
Recall a decomposition 
\begin{equation}\label{decomp_F-Pi_1}
F=L_{\Pi_1^{\dc}}\oplus G
\end{equation} 
of $F$, where $L_{\Pi_1^{\dc}}=\bigoplus_{i=0}^l\R\alpha_i$ is a one-codimensional 
subspace of $F$ spanned by $\Pi_1^{\dc}$. Define a linear automorphism 
$\psi_1\in\mathrm{GL}(F)$ by
\[\psi_1\ :\ \Big(\sum_{i=0}^l \lambda_i\alpha_i\Big)+\mu a\ \longmapsto
\Big(\sum_{i=0}^l \lambda_i\alpha_i\Big)+\Big(\mu+\sum_{i=0}^l \lambda_i m_i
k(\alpha_i)\Big)a.\]
By \eqref{alpha-beta},
condition \eqref{condition-Autom-alpha=beta} is 
satisfied, and $\psi_1(a)=a$. 

The remaining part is to prove that  $\psi_1$ preserves the set $R$. By the description
\eqref{description-R-2} of $R$, it is enough {to} show that
\begin{equation}\label{subset-Autom}
\psi_1\big(\gamma+p k(\gamma)a\big)\in R\quad 
\text{for $p\in\Z$ and $\gamma=w_1(\alpha_i)$ or 
$\gamma=w_1((\alpha_i')^\ast)$},
\end{equation} 
where $w_1\in \W[\Pi_1^{\dc}]$ and $\alpha_i\in \Pi_1^{\dc}$. \\
\\
{\bf Claim}. 
\begin{equation}\label{psi_1(alpha)=beta}
\psi_1((\alpha_i')^\ast)=(\beta_i')^\ast.
\end{equation}
\begin{proof}
This equality follows {from} case-by-case analysis. For example, assume 
\begin{equation}\label{gray-black}
\text{$2\alpha_i\not\in R,\quad 2\overline{\alpha_i}\in R/G$ \quad and \quad
$\frac12 (\alpha_i')^\ast \in R$. } 
\end{equation}
By this assumption and Proposition \ref{prop:countings-prime}, 
we have
\begin{equation}\label{psi_1-alpha}
\begin{aligned}
\psi_1((\alpha_i')^\ast)&=\psi_1\left(2\alpha_i+\frac12 k(\alpha'_i)a\right)
=2\beta_i+2 k(\alpha_i)a\\
&=2\beta_i+2 k(\beta_i)a.
\end{aligned}
\end{equation}
Note that we used  $k(\alpha_i)=k(\beta_i)$ in the last equality.
On the other hand, by Lemma \ref{lemma:nr-counting-Autom3} and 
the assumption \eqref{gray-black}, we have 
$2\beta_i\not\in R,\, 2\overline{\beta_i}\in R/G$ and 
$\frac12 (\beta_i')^\ast \in R$. Therefore, by Proposition 
\ref{prop:countings-prime}, we have
\begin{equation*}\label{psi_1-beta}
(\beta_i')^\ast=2\beta_i+\frac12 k(\beta_i')a=2\beta_i+2k(\beta_i)a.
\end{equation*}
Combining these two equalities, we get $\psi_1((\alpha_i')^\ast)=(\beta_i')^\ast$
as desired. For the other cases, the equality \eqref{psi_1(alpha)=beta} holds by a similar
method.
\end{proof}

Assume $\gamma=w_1((\alpha_1)^\ast)$. 
Since $\psi_1 \circ r_{\alpha_j}\circ \psi_1^{-1}=
r_{\psi_1(\alpha_j)}=r_{\beta_j}\in \W[\Pi_2^{\dc}]$ 
for every $0\leq j\leq l$, $w_2:=\psi_1\circ w_1\circ \psi_1^{-1}$ is an element of
$\W[\Pi_2^{\dc}]\subset W(R)$. Therefore, we have
\[\begin{aligned}
\psi_1\big(\gamma+p k(\gamma)a\big)
&=\psi_1(\gamma)+pk(\gamma)a\\
&=(\psi_1\circ w_1\circ \psi_1^{-1})\big(\psi_1((\alpha_i')^\ast)\big)+
p k(w_1((\alpha_i')^\ast))a\\
&=w_2((\beta_i')^\ast)+p k((\alpha_i')^\ast)a \quad \text{($\because\,  
w_1\in \Aut(R)$)}\\
&=w_2((\beta_i')^\ast)+p k((\beta_i')^\ast)a\quad 
\text{(by Lemma \ref{lemma:nr-counting-Autom3} (3))}\\
&=w_2((\beta_i')^\ast)+p k(w_2((\beta_i')^\ast))a\quad 
\text{($\because\, w_2\in \Aut(R)$)}.
\end{aligned}\]
Since $w_2((\beta_i')^\ast)$ is a root, the right hand side belongs to $R$. 
Thus, \eqref{subset-Autom} is obtained. For the case that 
$\gamma=w_1(\alpha_i)$, we have \eqref{subset-Autom} in a similar way
with an easier argument. Thus, we completed  the proof.
\end{proof}

Now, let us prove Proposition \ref{prop:nr-counting-Autom2}.
\vskip 3mm
\noindent
{\it Proof of Proposition \ref{prop:nr-counting-Autom2}}. 
Denote $\Pi_1^{\dc}=\{\alpha_0,\ldots,\alpha_l\}$ and 
$\Pi_2^{\dc}=\{\beta_0,\ldots,\beta_l\}$. Since both
$\pi_G(\Pi^{\dc}_1)=\{\overline{\alpha_0},\ldots,\overline{\alpha_l}\}$ and 
$\pi_G(\Pi^{\dc}_2)=\{\overline{\beta_0},\ldots,\overline{\beta_l}\}$ are simple systems
of the non-reduced affine root system $R/G$, there exists an element 
$\overline{w}\in W(R/G)$ and a sign $\epsilon\in\{\pm 1\}$ such that
\[\{\epsilon\overline{w}(\overline{\alpha_0}),\ldots,
\epsilon\overline{w}(\overline{\alpha_l})\}=
\{\overline{\beta_0},\ldots,\overline{\beta_l}\}.\]
After suitable change of the numbering of the elements of $\Pi_2^{\dc}$, we may 
assume that $\epsilon\overline{w}(\overline{\alpha_i})=\overline{\beta_i}$ for every 
$0\leq i\leq l$. Furthermore, take an element $w\in W(R)$ so that 
$(\pi_G)_*(w)=\overline{w}$. Then, 
$w\in \Aut(R,G)$, and it is easy to see that $w(\Pi_1^{\dc})$ is a simple system of 
$(R,G)$ such that  condition \eqref{condition-Pi'} is satisfied. 
Thus, it is enough to prove the proposition
under the assumption 
\[\epsilon(\overline{\alpha_i})=\overline{\beta_i}\quad\text{for every }0\leq i\leq l.\]

Recall the decomposition $F=L_{\Pi_1^{\dc}}\oplus G$ of $F$, where 
$L_{\Pi_1^{\dc}}=\bigoplus_{i=0}^l\R\alpha_i$ is a one-codimensional subspace of $F$
spanned by $\Pi_1^{\dc}$. With respect to this decomposition, write
$\lambda=\lambda_{L_{\Pi^{\dc}_1}}+\lambda_G$ for $\lambda\in F$, where 
$\lambda_{L_{\Pi^{\dc}_1}}\in L_{\Pi_1^{\dc}}$ and $\lambda_G\in G$. 
Define a linear map $\psi_2\in \mathrm{GL}(F)$ by
\[\psi_2\ :\ \lambda=\lambda_{L_{\Pi^{\dc}_1}}+\lambda_G\ \longmapsto \ \epsilon
\lambda_{L_{\Pi^{\dc}_1}}+\lambda_G.\]
Let us prove that $\psi_2$ belongs to $\Aut(R,G)$. Since there is nothing 
to prove if $\epsilon=1$, we may assume $\epsilon=-1$. By the definition, 
it is obvious that $\psi_2(a)=a$. Let $\alpha\in R$. 
By the expression \eqref{description-R-2} of $R$ in Proposition
\ref{prop:description-R}, there exist $w\in \W[\Pi_1^{\dc}]$,
$\alpha_i\in \Pi^{\dc}_1$ and $m\in \Z$ such that
\[\alpha=\begin{cases}
2w(\alpha_i)+(2m+1)k(\alpha_i)a & 
\text{if $2\alpha_i\not\in R,\, 2\overline{\alpha_i}\in R/G$ 
and $\frac12 (\alpha_i')^\ast\not\in R$},\\
2w(\alpha_i)+(4m+2)k(\alpha_i)a & 
\text{if $2\alpha_i\not\in R,\, 2\overline{\alpha_i}\in R/G$ 
and $\frac12 (\alpha_i')^\ast\not\in R$},\\
2w(\alpha_i)+mk(\alpha_i)a & \text{if $2\alpha_i\in R$ and
$\frac12 (\alpha_i')^\ast\not\in R$},\\
2w(\alpha_i)+2mk(\alpha_i)a & \text{if $2\alpha_i\in R$ and
$\frac12 (\alpha_i')^\ast\in R$},\\
w(\alpha_i)+mk(\alpha_i)a & \text{if $2\overline{\alpha_i}\not \in R/G$}.
\end{cases}\]
In any case, as $-w(\alpha_i)=wr_{\alpha_i}(\alpha_i)$ and 
$wr_{\alpha_i}\in W(\Pi^{\dc})$, we see that $\psi_2(\alpha)\in R$. 
Thus, $\psi_2$ is an element of $\Aut(R,G)$.  

Now, it is enough to show the proposition under the assumption  
\begin{equation*}
\overline{\alpha_i}=\overline{\beta_i}\quad\text{for every }0\leq i\leq l.
\end{equation*}
This is exactly treated in Lemma \ref{lemma:nr-counting-Autom4}. 
Thus, we have the proposition.
\hfill $\square$


\section{Main theorem}\label{sect:main}

In this section, we introduce a diagram associated with each mERS $(R,G)$ with non-
reduced affine quotient $R/G$ {which encodes the informations 
obtained in the previous section. This is an extension of the notion of elliptic diagram introduced in $\S$ \ref{sect_elliptic-diagram-I}.} Interpreting the results obtained in the previous section in terms of diagrams, the strong classification theorem of mERS 
with non-reduced affine quotient is given as an application.

\subsection{Elliptic diagrams II: non-reduced affine quotient
}\label{sect_elliptic-diagram-II} 
The aim of this subsection is to define the elliptic diagram 
{$\Gamma(R,G)$} for a 
mERS $(R,G)$ whose quotient system $R/G$ is \textit{non-reduced}.
Fix a paired simple system $(\Pi^{\dc},\Pi^{\mc})$ of $(R,G)$. 
Let $\Gamma^{\downarrow}(R,G)$ and $\Gamma^{\uparrow}(R,G)$ be the Dynkin 
diagrams of the affine root systems $\W[\Pi^{\dc}].\Pi^{\dc}  \cong(R/G)^{\dc}$ and 
$\W[\Pi^{\color{red}\dc}].\Pi^{\mc} \cong(R/G)^{\mc}$, respectively. 
We regard $\alpha \in \Pi^{\dc}$ (resp. $\alpha'\in \Pi^{\mc}$) as a node of the graph 
$\Gamma^{\downarrow}(R,G)$ (resp. $\Gamma^{\uparrow}(R,G)$) and denote the set 
of nodes by 
$\vert \Gamma^{\downarrow}(R,G)\vert$ and $\vert \Gamma^{\uparrow}(R,G)\vert$, 
respectively.  If there is no risk of confusion, we may denote 
$\Gamma^{\downarrow}=\Gamma^{\downarrow}(R,G)$
\index[notations]{g811@$\Gamma^{\downarrow}=\Gamma^{\downarrow}(R,G)$} and 
$\Gamma^{\uparrow}=\Gamma^{\uparrow}(R,G)$
\index[notations]{g812@$\Gamma^{\uparrow}=\Gamma^{\uparrow}(R,G)$} for 
simplicity.

\begin{defn}\label{defn:exponent-non-reduced}
Define the numbers 
$m_{\alpha}$\index[notations]{m811@$m_\alpha$} for 
$\alpha\in |\Gamma^{\downarrow}(R,G)|$ by
\begin{equation}\label{defn:exponent}
m_{\alpha}=\frac{I_R(\alpha, \alpha)}{2k^{nr}(\alpha)}n_{\alpha},
\end{equation}
and call it the \textbf{exponent}\index[index]{exponent} of the root $\alpha$.
\end{defn}

\begin{rem}
In  \eqref{defn:exponent}, we define the exponent $m_\alpha$ only
for $\alpha\in |\Gamma^\downarrow(R,G)|$.  By using 
$\Gamma^\uparrow(R,G)$ instead of $\Gamma^\downarrow(R,G)$, one can define
the ``exponent'' for a node in $|\Gamma^\uparrow(R,G)|$. The explicit relationship
between these two exponents is given in Section \ref{sect_str-red-pair} {\rm (}see
Proposition \ref{prop:Gamma(R'',G)}{\rm )}. 
\end{rem}


Set $m_{\max}=\max \{m_\alpha \, \vert \, \alpha \in |\Gamma^{\downarrow}(R,G)|\, \}$. 
Let $\Gamma_m^{\downarrow}=\Gamma_m^{\downarrow}(R,G)
$\index[notations]{g813@$\Gamma_m^\downarrow=\Gamma_m^{\downarrow}(R,G)$} 
be the subdiagram of 
$\Gamma^{\downarrow}(R,G)$ consisting of nodes 
\[  \vert \Gamma_m^{\downarrow}(R,G) \vert:=\{ \alpha \in \vert \Gamma^{\downarrow}(R,G)\vert \, 
\vert\, m_\alpha=m_{\max}\, \}.
\]
Set
\[ \vert \Gamma_m^{\uparrow}(R,G)^\ast \vert=\{ (\alpha')^\ast\, \vert \, 
\alpha \in \vert \Gamma_m^{\downarrow}(R,G) \vert\, \}.\]

\begin{defn}\label{defn_elliptic-diagram-II}
The \textbf{elliptic diagram}\index[index]{elliptic diagram} $\Gamma(R,G)$ for a marked elliptic root system 
$(R,G)$ is the graph whose set of nodes is 
\begin{equation}\label{eq:splitting_nodes}
\vert \Gamma(R,G)\vert :=\vert \Gamma^{\downarrow}(R,G) \vert \,\amalg\, 
\vert \Gamma_m^{\uparrow}(R,G)^\ast \vert 
\end{equation}
and is represented by
\begin{itemize}
\item[(Node1')]
\begin{tikzpicture} \fill (0,0) circle(0.1); \draw (0,0.1) node[above] {$\alpha$}; 
\end{tikzpicture}
if either $2\alpha  \in R$ or $\frac{1}{2}\alpha \in R$, 
\item[(Node2')]
\begin{tikzpicture} \draw [black,fill=gray!70] (0,0) circle(0.1); \draw (0,0.1) 
node[above] {$\alpha$}; 
\end{tikzpicture}
if $2\alpha, \frac{1}{2}\alpha \not\in R$ but either $2\overline{\alpha}\in R/G$ or 
$\frac{1}{2}\overline{\alpha} \in R/G$, 
\item[(Node3')]
\begin{tikzpicture} \draw (0,0) circle(0.1); \draw (0,0.1) node[above] {$\alpha$}; 
\end{tikzpicture}
otherwise
\end{itemize}
for $\alpha \in \vert \Gamma(R,G) \vert$, and any two nodes 
$\alpha, \beta \in \vert \Gamma(R,G)\vert$ are connected by the rules 
$($Edge0$)$ - $($Edge4$)$, with the additional cases 
\begin{itemize}
\item[(Edge5)] 
\begin{tikzpicture} 
\draw [black,fill=gray!40] (0,0) circle(0.1); \draw (0,0.1) node[above] {$\alpha$};
\draw[dashed, , double distance=2] (0.1,0) -- (0.9,0); 
{\color{red}\draw (0.45,0) -- (0.58,0.13); \draw (0.45,0) -- (0.58,-0.13);}
\draw [black,fill=gray!80] (1,0) circle(0.1); \draw (1,0.1) node[above] {$\beta$};
\end{tikzpicture}
if $I(\alpha^\vee,\beta)=4, I(\alpha,\beta^\vee)=1$ and 
$\overline{\beta}=2\overline{\alpha}$, for $\node{40}, \node{80} \in \{ \white, \gray, \black\}$.
\end{itemize}
{A node
which satisfies the condition {\rm (Node1')} {\rm (}resp. {\rm (Node2'), (Node3')}{\rm )}
is called a black {\rm (}resp. gray, white{\rm )} node.}
\index[index]{node!black@black --}
\index[index]{node!gray@gray --}
\index[index]{node!white@white --}
\end{defn}

For each $i=1,2$, let $(R_i,G_i)$ be a mERS belonging to $(F_i,I_i)$. 
Fix a paired simple system $(\Pi^{\dc}_i,\Pi^{\mc}_i)$ of $(R_i,G_i)$ and let 
$\Gamma(R_i,G_i)$ be the corresponding elliptic diagram. We say 
that two graphs $\Gamma(R_1,G_1)$ and $\Gamma(R_2,G_2)$ are \textbf{isomorphic} 
if there exists a 
bijection $\varphi:|\Gamma(R_1,G_1)|\xrightarrow{\sim} |\Gamma(R_2,G_2)|$
such that
\begin{itemize}
\item[(i)] the bijection $\varphi$ preserves the splitting \eqref{eq:splitting_nodes}
of nodes;
\vskip 1mm
\item[(ii)] $(I_1)_{R_1}(\alpha,\beta^\vee)
=(I_2)_{R_2}(\varphi(\alpha),\varphi(\beta)^\vee)$ for every 
$\alpha,\beta\in |\Gamma(R_1,G_1)|$;
\vskip 1mm
\item[(iii)] $m_\alpha=m_{\varphi(\alpha)}$ for every $\alpha\in\Pi_1^{\dc}$. \\
\end{itemize}

Note that the definition of $\Gamma(R,G)$ depends on 
a choice of a paired simple system $(\Pi^{\dc},\Pi^{\mc})$ of $(R,G)$. 
However, the following proposition holds.

\begin{prop}\label{prop:indendence-Gamma}
Recall the setting in Proposition \ref{prop:nr-counting-Autom1}. That is, 
let $(R_i,G_i)$ be a mERS, and $(\Pi^{\dc}_i,\Pi^{\mc}_i)$ be a paired  
simple system of $(R_i,G_i)$ for $i=1,2$. Fix a generator $a_i$
of the lattice $Q(R_i)\cap G_i$ of rank $1$. Assume that $(R_1,G_1)$ and
$(R_2,G_2)$ are isomorphic. Then, the graphs
$\Gamma(R_1,G_1)$ and $\Gamma(R_2,G_2)$ are isomorphic. Especially, 
the elliptic diagram $\Gamma(R,G)$ does not depend on a choice of a paired
simple system $(\Pi^{\dc},\Pi^{\mc})$ of $(R,G)$.
\end{prop}

\begin{proof}
Denote $\Pi^{\dc}_1=\{\alpha_0,\ldots,\alpha_l\}\subset R_1$ and 
$\Pi_2^{\dc}=\{\beta_0,\ldots,\beta_l\}\subset R_2$. 
By Proposition \ref{prop:nr-counting-Autom1}, there exists an isomorphism
$\varphi:(R_1,G_1)\xrightarrow{\sim}(R_2,G_2)$ such that 
$\varphi(a_1)=a_2$, $\varphi(\Pi_1^{\dc})=\Pi_2^{\dc}$ and 
$\varphi(\Pi_1^{\mc})=\Pi_2^{\mc}$. 
After renumbering suitably the elements of $\Pi_2^{\dc}$, 
we may assume
\[\varphi(\alpha_i)=\beta_i\quad\text{and}\quad 
\varphi((\alpha_i'))=\beta_i'\quad\text{for every }0\leq i\leq l.\]
\begin{lemma}\label{lemma:countings1-isom}
Under the same setting as Proposition \ref{prop:indendence-Gamma}, we have
the following formulae{\rm :}
\begin{enumerate}
\item $k(\varphi(\alpha))=k(\alpha)$ for every $\alpha\in R_1${\rm ;}
\vskip 1mm
\item $\varphi(\alpha)^\ast=\varphi(\alpha^\ast)$ for every  $\alpha\in R_1$.
\end{enumerate}
\end{lemma}
\begin{proof}
Similar to Proposition \ref{prop:nr-counting-Autom1}, it is enough to prove
the statements in the case where $(R_1,G_1)=(R_2,G_2)$  and $\varphi(a_1)=a_2$.
Denote them by $(R,G)$ and $a$, respectively. Then, 
statement (1) is already proved in 
Lemma \ref{lemma:countings1} (3). Statement (2) follows immediately from
statement (1). 
\end{proof}

Let us return to the proof of the proposition. By Lemma \ref{lemma:countings1-isom}, 
we have
\begin{equation}\label{equality-countings1}
k(\alpha_i)=k(\beta_i)\quad \text{and} \quad 
k((\alpha_i')^\ast)=k((\beta_i')^\ast)\quad \text{for every }
0\leq i\leq l.
\end{equation}
In addition, by Lemma \ref{lemma:nr-counting-Autom3} and the definition of
the non-reduced counting numbers, we have
\begin{equation}\label{equality-countings2}
k^{\nr}(\alpha_i)=k^{\nr}(\beta_i)\quad \text{for every }
0\leq i\leq l.
\end{equation}

On the other hand, since $\varphi$ is an isomorphism between two root 
systems $R_1$ and $R_2$, it is known that there exists a non-zero constant 
$c$ such that 
$I_1(\lambda,\mu)=c I_2(\varphi(\lambda),\varphi(\mu))$ for every 
$\lambda,\mu\in F_1$. 
(see (1.4) Lemma in \cite{Saito1985}). After taking normalizations $(I_1)_{R_1}$
$(I_2)_{R_2}$ of $I_1$ and $I_2$ respectively 
(cf. \eqref{def_normalized-form}),  we have
\begin{equation}\label{isometry}
(I_1)_{R_1}(\lambda,\mu)=
(I_2)_{R_2}(\varphi(\lambda),\varphi(\mu))\quad\text{for every}\quad 
\lambda,\mu\in F_1.
\end{equation} 
Especially, it follows that
$(I_1)_{R_1}(\alpha_i,\alpha_j^\vee)
=(I_2)_{R_2}(\beta_i,\beta_j^\vee)$
for every $0\leq i,j\leq l$.
As a by-product of this equality, we get 
\[n_{\alpha_i}=n_{\beta_i}\quad \text{for every}\quad 0\leq i\leq l.\]
Combining the above results, we have
\[m_{\alpha_i}=\frac{(I_1)_{R_1}(\alpha_i,\alpha_i)}{2k^{\nr}(\alpha_i)}n_{\alpha_i}
=\frac{(I_2)_{R_2}(\beta_i,\beta_i)}{2k^{\nr}(\beta_i)}n_{\beta_i}=m_{\beta_i}.\]
Thus, the automorphism $\varphi$ induces a bijection 
$\varphi:|\Gamma(R_1,G_1)|\xrightarrow{\sim}|\Gamma(R_2,G_2)|$ which 
satisfies the conditions (i), (ii), (iii) above. Therefore, the graphs 
$\Gamma(R_1,G_1)$ and $\Gamma(R_2,G_2)$ are isomorphic.
\end{proof}
In the next subsection, we classify the elliptic diagrams of mERSs with non-reduced
affine quotients, up to ``existence''. 
In other words,  we discuss only the possibility of these diagrams.  
The results are obtained by detailed case analysis depending on the type of $R/G$, 
and are exhibited in Table $4\sim 8$ in $\S$ \ref{sect:list_of_diagrams} (see Theorem
\ref{thm:classification-thm} and Corollary \ref{cor:classification-thm} also). 
For the existence of a mERS whose  diagram is a given one, see $\S$ 
\ref{sect_Remarks-classification} (especially Remark \ref{rem:reconstruction} below).

Until the end of the subsection, we use symbols omitting $(R,G)$ for simplicity: 
$\Gamma=\Gamma(R,G)$, $\Pi=\Pi(R,G)$, etc.
\subsection{Higher rank cases}
Throughout this subsection, we assume $(R,G)$ is a mERS 
with the non-reduced affine quotient $R/G$.
It is of type
\[BCC_l\, (l\geq 1),\quad C^\vee BC_l\, (l\geq 1),\quad 
BB^\vee _l\, (l\geq 2)\quad\text{or}\quad  C^\vee C_l\, (l\geq 1)\] 
(see $\S$ \ref{sect_non-red-affine} for explicit descriptions of such root systems). 
In the following, choose a numbering of the set
$\Pi^{\dc}=|\Gamma^{\downarrow}|=\{\alpha_0,\alpha_1,\ldots,\alpha_l\}$ 
so that $\Pi^{\dc}_G=\{\overline{\alpha_0},\overline{\alpha_1},\ldots,\overline{\alpha_l}\}$ 
is enumerated as in the list of $\S$ \ref{sect_non-red-affine}. \\

Let $(\Gamma^\downarrow)^\w$\index[notations]{g814@
$(\Gamma^\downarrow)^\w$} and 
$(\Gamma^\downarrow)^\nw$\index[notations]{g815@
$(\Gamma^\downarrow)^\nw$}
be subdiagrams of $\Gamma^{\downarrow}$ with
the set of nodes given as follows:
\[|(\Gamma^\downarrow)^\w|
:=\{\alpha_i\in|\Gamma^{\downarrow}|\,|\,
2\overline{\alpha_i}\not\in R/G\}
\quad\text{and}\quad
|(\Gamma^\downarrow)^\nw|:=|\Gamma^{\downarrow}|\setminus 
|(\Gamma^\downarrow)^\w|.\]
Set
\[\begin{array}{rlllllll}
|(\Gamma^\downarrow_m)^\w|&\!\!\!\!=
|(\Gamma^\downarrow)^\w|\cap|\Gamma_m^{\downarrow}|
&\!\!\!\! =\big\{\alpha_i\in \vert (\Gamma^\downarrow)^\w\vert \,\big|\,
m_{\alpha_i}=m_{max}\big\},\\[3pt]
|(\Gamma^\downarrow_m)^\nw|
&\!\!\!\!=|(\Gamma^\downarrow)^\nw| \cap|\Gamma_m^{\downarrow}|
&\!\!\!\!=\big\{\alpha_i\in \vert (\Gamma^\downarrow)^\nw \vert\,
\big|\,m_{\alpha_i}=m_{max}\big\}.
\end{array}\]
Then there exists a decomposition of $|\Gamma_m^{\downarrow}|$:
\[|\Gamma_m^{\downarrow}|=
|(\Gamma^\downarrow_m)^\w|\, \amalg\, 
|(\Gamma^\downarrow_m)^\nw|.\]

Similarly, set
\begin{align*}
&|((\Gamma^\uparrow_m)^\ast)^\w|
=\big\{(\alpha_i')^\ast\in|(\Gamma^\uparrow_m)^\ast|\,\big|\,
\tfrac12\overline{(\alpha_i')^\ast}\not\in R/G \big\}, \\
&
|((\Gamma^\uparrow_m)^\ast)^\nw|=|(\Gamma^\uparrow_m)^\ast|\setminus 
|((\Gamma^\uparrow_m)^\ast)^\w|,
\end{align*}
where $(\Gamma^\uparrow_m)^\ast:=\Gamma_m^\uparrow(R,G)^\ast$ 
(see Definition \ref{defn_elliptic-diagram-II}). Then, we have a decomposition of
$|(\Gamma^\uparrow_m)^\ast|$:
\begin{equation}\label{eq:decomp-upstair-1}
|(\Gamma^\uparrow_m)^\ast|=|((\Gamma^\uparrow_m)^\ast)^\w|\amalg
|((\Gamma^\uparrow_m)^\ast)^\nw|.
\end{equation}

\begin{lemma}\label{lemma:color}
{\rm (1)} A node $\alpha_i\in |(\Gamma^\downarrow)^\w|$ is a 
white node in $|\Gamma^\downarrow|$. 
\vskip 1mm
\noindent
{\rm (2)} A node $\alpha_i\in |(\Gamma^\downarrow)^\nw|$ is not a white node.
\vskip 1mm
\noindent
{\rm (3)} A node $(\alpha_i')^\ast\in |((\Gamma_m^\uparrow)^\ast)^\w|$ is 
a white node.
\vskip 1mm
\noindent
{\rm (4)} A node $(\alpha_i')^\ast\in |((\Gamma_m^\uparrow)^\ast)^\nw |$
is not a white node.  
\end{lemma}
\begin{proof}
For statement (1), it is enough to show that $2\alpha_i\not\in R$ and 
$\frac12 \alpha_i\not\in R$.
The first condition is an immediate consequence of the condition
$2\overline{\alpha_i}\not\in R/G$. Since $\Pi^{\dc}=|\Gamma^{\downarrow}|$ is a subset of $R^{\dc}$, 
the second condition is satisfied for every node in $|\Gamma^{\downarrow}|$. 

For $\alpha_i\in |\Gamma^{\downarrow}|\setminus 
|(\Gamma^\downarrow)^\w |$, we have 
$2\overline{\alpha_i}\in R/G$. If $2\alpha_i\in R$, $\alpha_i$ is a black node.
Otherwise, it is a gray node. Thus, we have statement (2). 

We have statement (3) (resp. (4)) by a similar way to the
proof of statement (1) (resp. (2)). Thus, we completed  the proof of the lemma.
\end{proof}

There is another decomposition of the set 
$|(\Gamma^\uparrow_m)^\ast|$. Set
\begin{align*}
&\big|\big(\big((\Gamma_m^\downarrow)^\w\big)^{pr}\big)^\ast\big|=
\big\{(\alpha_i')^\ast\,\big|\,\alpha_i\in 
|(\Gamma^\downarrow_m)^\w| \big\}, \\
&
\big|\big(\big((\Gamma_m^\downarrow)^\nw\big)^{pr}\big)^\ast\big|=
\big\{(\alpha_i')^\ast\,\big|\,\alpha_i\in |(\Gamma_m^\downarrow)^\nw |
\big\},
\end{align*}
respectively.  Then, we have 
\begin{equation}\label{eq:decomp-upstair-2}
|(\Gamma^\uparrow_m)^\ast|=
\big|\big(\big((\Gamma_m^\downarrow)^\w\big)^{pr}\big)^\ast\big| \amalg
\big|\big(\big((\Gamma_m^\downarrow)^\nw\big)^{pr}\big)^\ast\big|.
\end{equation}

\begin{lemma}\label{lemma:color-2}
The decomposition \eqref{eq:decomp-upstair-2} coincides with 
the decomposition \eqref{eq:decomp-upstair-1}. That is, we have
\[\big|\big(\big((\Gamma_m^\downarrow)^\w\big)^{pr}\big)^\ast\big|=
|((\Gamma^\uparrow_m)^\ast)^\w|\quad\text{and}\quad
[\big|\big(\big((\Gamma_m^\downarrow)^\nw\big)^{pr}\big)^\ast\big|=
|((\Gamma^\uparrow_m)^\ast)^\nw|.\]
In other words, for $\alpha_i\in |\Gamma_m^\downarrow|$, 
\[\begin{array}{lllllllll}
\text{$\alpha_i$ is a white node} & \Longleftrightarrow &
(\alpha_i')^\ast\in |(\Gamma_m^\uparrow)^\ast|\text{ is a white node},\\[3pt]
\text{$\alpha_i$ is not a white node} & \Longleftrightarrow &
(\alpha_i')^\ast\in |(\Gamma_m^\uparrow)^\ast|\text{ is not a white node}.
\end{array}\]
\end{lemma}
\begin{proof}
The statement is an immediate consequence of the definition of the map 
$(\,\cdot\,)^{pr}:R\to R^{\mc}$, and the equivalences in 
\eqref{nr-countings-cor1}.
\end{proof}

Now, we have a decomposition of the set 
$|\Gamma(R,G)|$ of nodes of $\Gamma(R,G)$:
\begin{equation}\label{decomp-graph}
|\Gamma(R,G)|
=\big(\, |(\Gamma^\downarrow)^\w |\,\amalg\, 
|((\Gamma_m^\uparrow)^\ast)^\w |\,\big)
\,\amalg\,\big(\,|(\Gamma^\downarrow)^\nw |\,\amalg\, 
|((\Gamma_m^\uparrow)^\ast)^\nw |\,\big).
\end{equation}
The above lemmas imply that
\begin{itemize}
\item $|\Gamma^\w(R,G)|
:=|(\Gamma^\downarrow)^\w|\,\amalg\, |((\Gamma_m^\uparrow)^\ast)^\w |$ 
is the set of all white nodes in 
$|\Gamma(R,G)|$, and
\vskip 1mm
\item $|\Gamma^\nw(R,G)|
:=|(\Gamma^\downarrow)^\nw|\,\amalg\, |((\Gamma_m^\uparrow)^\ast)^\nw |$ 
is the set of all colored nodes in $|\Gamma(R,G)|$. 
\end{itemize}

{Here} and after, we denote 
\[|\Gamma^\w|=|(\Gamma^\downarrow)^\w|,\quad 
|\Gamma^\nw|=|(\Gamma^\downarrow)^\nw|,\quad 
|\Gamma_m^\w|=|(\Gamma_m^\downarrow)^\w|,\quad 
|\Gamma_m^\nw|=|(\Gamma_m^\downarrow)^\nw|\]
\index[notations]{g816@$\vert \Gamma^\w\vert$}
\index[notations]{g817@$\vert\Gamma^\nw\vert$}
\index[notations]{g818@$\vert\Gamma_m^\w\vert$}
\index[notations]{g819@$\vert\Gamma_m^\nw\vert$}
for simplicity.
\vskip 3mm

Notice that the subset $|\Gamma^\w|$ of $|\Gamma^{\downarrow}|$ is determined by 
the information on the 
affine root system $R/G$.
The explicit form of $|\Gamma^\w|$ is given as follows:
\[|\Gamma^\w|=\begin{cases}
\{\alpha_0,\alpha_1,\ldots,\alpha_{l-1}\} & 
\text{if $R/G$ is of type $BCC_l,\, C^\vee BC_l,\, BB^\vee_l$},\\
\{\alpha_1,\alpha_2,\ldots,\alpha_{l-1}\} & 
\text{if $R/G$ is of type $C^\vee C_l$}.
\end{cases}\]
Note that $|\Gamma^\w|=\emptyset$ if $R/G$ is of type $C^\vee C_1$. \\

To classify the diagrams $\Gamma(R,G)$, 
we divide the problem into the following
two pieces: 
\vskip 1mm
\begin{itemize}
\item[(a)] {\it What is the shape of $\Gamma(R,G)$?}
\vskip 1mm
\item[(b)] {\it What is the color of each node of $\Gamma(R,G)$?}
\end{itemize}
\vskip 1mm
Namely,  problem (a) is equivalent to determining the subsets 
$|\Gamma_m^\w|$ of $|\Gamma^\w|$ and $|\Gamma^\nw_m|$ of 
$|\Gamma^\nw|$, respectively. For problem (b), we already know that
$|\Gamma^\w|
$ ({\it resp}. 
$|\Gamma^\nw|
$)
is the set of all white ({\it resp}. colored) nodes in $|\Gamma^\downarrow|$.
Therefore, the remaining problem is to determine whether the color of the nodes in  
$|\Gamma^\nw|$ is gray or black. 
The goal of this subsection is to solve these problems.\\

For that purpose, we divide the set $|\Gamma^\w|$ of white nodes in 
$|\Gamma^{\downarrow}|$ into the following two pieces.   
A white node $\alpha_i\in |\Gamma^\w|$ is called 
an {\it ordinary white node}\index[index]{node!ordinary white@ordinary white --}  
if $\alpha_i$ is connected to at least two nodes in $|\Gamma^{\downarrow}|$, 
and denote the set of all ordinary white nodes by $|\Gamma^\w|_o$. 
\index[notations]{g820@$\vert\Gamma^\w\vert_o$}
Hence, we have a 
decomposition
\[|\Gamma^\w|=|\Gamma^\w|_o\amalg 
\big(|\Gamma^\w|\setminus |\Gamma^\w|_o\big).\]
A node $\alpha_j\in |\Gamma^\w|\setminus |\Gamma^\w|_o$ is called a {\it boundary
white node}.\index[index]{node!boundary white@boundary white --} 
The explicit form of $|\Gamma^\w|_o$ is given as follows:
\[|\Gamma^\w|_o=\begin{cases}
|\Gamma^\w|\setminus \{\alpha_0\} & 
\text{if $R/G$ is of type $BCC_l,\, C^\vee BC_l$},\\
|\Gamma^\w|\setminus \{\alpha_0,\alpha_1\} & 
\text{if $R/G$ is of type $BB^\vee_l$},\\
|\Gamma^\w| & \text{if $R/G$ is of type $C^\vee C_l$}.
\end{cases}\]
Note that $|\Gamma^\w|_o=\emptyset$ if $R/G$ is of 
type $BB^\vee_2$, and $|\Gamma^\w|=|\Gamma^\w|_o=\emptyset$ 
if $R/G$ is of type $C^\vee C_1$. \\

The following proposition is a first step for the above problem.

\begin{prop}\label{prop:shape-diagram1}
{\rm (1)} Assume both $\alpha_i$ and $\alpha_{i+1}$ are nodes in 
$|\Gamma^\w|_o$. One has $m_{\alpha_i}=m_{\alpha_{i+1}}$.
\vskip 1mm
\noindent
{\rm (2)} Assume $R/G$ is not of type $C^\vee C_l$. 
Let $\alpha_i\in |\Gamma^\w|_o$ and 
$\alpha_j\in |\Gamma^\w|\setminus |\Gamma^\w|_o$. 
One has $m_{\alpha_i}\geq m_{\alpha_j}$. Especially, if $R/G$ is of type 
$BB^\vee_l$, one has $m_{\alpha_i}>m_{\alpha_j}$. 
\vskip 1mm
\noindent
{\rm (3)} 
Let $\alpha_i\in |\Gamma^\w|_o$ and 
$\alpha_j\in |\Gamma^\nw|$. 
One has $m_{\alpha_i}\geq m_{\alpha_j}$.
\vskip 1mm
\noindent
{\rm (4)} Let $\alpha_i\in |\Gamma^\w|_o$ and 
$\alpha_j\in |\Gamma^\nw|$. If 
$m_{\alpha_i}> m_{\alpha_j}$, the element $2\alpha_j$ is a root. 
\end{prop}

\begin{proof}
(1) Since $2\overline{\alpha_i}\not\in R/G$ and 
$2\overline{\alpha_{i+1}}\not\in R/G$,
we have $k^\nr(\alpha_i)=k(\alpha_i)$ and $k^\nr(\alpha_{i+1})=k(\alpha_{i+1})$
by Proposition \ref{prop:nr-countings1}, respectively. Since both $\alpha_i$ and
$\alpha_{i+1}$ are middle roots, we have 
$I(\alpha_i,\alpha_{i+1}^\vee)=I(\alpha_{i+1},\alpha_{i}^\vee)=-1$, and
$I_R(\alpha_i,\alpha_{i})=I_R(\alpha_{i+1},\alpha_{i+1})=4$. It follows from the
first equality that 
\begin{equation}
k(\alpha_i)=k(\alpha_{i+1})
\end{equation}
by Lemma \ref{lemma:countings1} (2). 
In addition, we observe that $n_{\alpha_i}=n_{\alpha_{i+1}}$ for every case 
by explicit data in $\S$ \ref{sect_non-red-affine}. Thus, we  have
\[m_{\alpha_i}=\frac{I_R(\alpha_i,\alpha_i)}{2 k^\nr (\alpha_i)}n_{\alpha_i}
=\frac{I_R(\alpha_{i+1},\alpha_{i+1})}{2 k^\nr (\alpha_{i+1})}n_{\alpha_{i+1}}
=m_{\alpha_{i+1}}.\]
\noindent
(2) By the conditions $|\Gamma^\w|_o\ne\emptyset$ and 
$|\Gamma^\w|\setminus |\Gamma^\w|_o\ne \emptyset$, 
we see that $R/G$ should be of type $BCC_l \,(l\geq 2),\,  
C^\vee BC_l\,(l\geq 2)$ or $BB^\vee_l\,(l\geq 3)$. In each case, 
there exists a unique node $\alpha_{i_0}\in |\Gamma^\w|$ such that it connects 
to $\alpha_j$ by a bond in $\Gamma$. Since $m_{\alpha_i}=m_{\alpha_{i_0}}$ by 
statement (1), it is enough to show the statement under the assumption that
$\alpha_i=\alpha_{i_0}$.

By a similar way as in the proof of statement (1), we  obtain that
$k^\nr(\alpha_i)=k(\alpha_i)$ and $k^\nr(\alpha_j)=k(\alpha_j)$. 
Here, we used  the fact that $\alpha_i$ and $\alpha_j$ belongs to $|\Gamma^\w|$.
In the following table, we give the possible list of 
$\alpha_i$ and $\alpha_j$. In addition, $n_{\alpha_i}$, $n_{\alpha_j}$, 
$I_R(\alpha_i,\alpha_i)$, $I_R(\alpha_j,\alpha_j)$
$I(\alpha_i,\alpha_j^\vee)$ and $I(\alpha_i,\alpha_j^\vee)$ are exhibited also. 
\begin{center}
\begin{table}[h]
\label{table1-alpha_i-alpha_j}
\vskip -2mm
\resizebox{.98\textwidth}{!}{%
\renewcommand{\arraystretch}{1.3}
\begin{tabular}{|c||c|c|c||c|c|c||c|c|c|}
\noalign{\hrule height0.8pt}
$R/G$&  $\alpha_i$ & $n_{\alpha_i}$ & $I_R(\alpha_i,\alpha_i)$ & $\alpha_j$ &
$n_{\alpha_j}$ & $I_R(\alpha_j,\alpha_j)$ & $I(\alpha_i,\alpha_j^\vee)$ & 
$I(\alpha_j,\alpha_i^\vee)$ \\
\noalign{\hrule height0.8pt}
\hline
$BCC_l\, (l\geq 2)$ & $\alpha_1$ & $2$ & $4$ & 
$\alpha_0$ & $1$ & $8$ & $-1$ & $-2$\\
\hline
$C^\vee BC_l\, (l\geq 2)$ & $\alpha_1$ & $1$ & $4$ & 
$\alpha_0$ & $1$ & $2$ & $-2$ & $-1$\\
\hline
$BB^\vee_l\, (l\geq 3)$ & $\alpha_2$ & $2$ & $4$ & $\alpha_0$ & $1$ & 
$4$ & $-1$ & $-1$\\
\cline{5-9}
& & & & $\alpha_1$ & $1$ & $4$ & $-1$ & $-1$\\
\noalign{\hrule height0.8pt}
\end{tabular}}
\end{table}
\end{center}
\begin{rem} 
The dual $\alpha^\vee=\frac{2\alpha}{I(\alpha,\alpha)}$ of $\alpha\in R$
depends on a choice of a bilinear form $I$. On the other hand, since there exists 
a positive constant $c>0$ such that $I=cI_R$ {\rm (}see {\rm \cite{Saito1985} (1.11))}, 
we have
\[I(\alpha,\beta^\vee)=\frac{2I(\alpha,\beta)}{I(\beta,\beta)}=
\frac{2cI_R(\alpha,\beta)}{cI_R(\beta,\beta)}
=I_R(\alpha,\beta^{\vee_R})\quad \text{for every }\alpha,\beta\in R,\]
where $\beta^{\vee_R}:=\frac{2\beta}{I_R(\beta,\beta)}$.
\end{rem}
By direct computation, we have
\[m_{\alpha_i}=\begin{cases}
4/k(\alpha_i) & \text{if $BCC_l\, (l\geq 2)$},\\
2/k(\alpha_i) & \text{if $C^\vee BC_l\, (l\geq 2)$},\\
4/k(\alpha_i) & \text{if $BB^\vee_l\, (l\geq 3)$},
\end{cases}
\quad
m_{\alpha_j}=\begin{cases}
4/k(\alpha_j) & \text{if $BCC_l\, (l\geq 2)$},\\
1/k(\alpha_j) & \text{if $C^\vee BC_l\, (l\geq 2)$},\\
2/k(\alpha_j) & \text{if $BB^\vee_l\, (l\geq 3)$}.
\end{cases}
\]

First, assume $R/G$ is of type $BCC_l\, (l\geq 2)$. 
By Lemma \ref{lemma:countings1} (2), we have 
$k(\alpha_i)|k(\alpha_j)|2k(\alpha_i)$.
Therefore, $k(\alpha_j)=k(\alpha_i)$ or $2k(\alpha_i)$, and 
\[\begin{cases}
m_{\alpha_i}=m_{\alpha_j} & \text{if $k(\alpha_j)=k(\alpha_i)$},\\
m_{\alpha_i}>m_{\alpha_j} & \text{if $k(\alpha_j)=2k(\alpha_i)$}.
\end{cases}\]
Second, in the case when $R/G$ is of type $C^\vee BC_l\, (l\geq 2)$, we have
$k(\alpha_i)=2k(\alpha_j)$ or $k(\alpha_j)$, and 
\[\begin{cases}
m_{\alpha_i}=m_{\alpha_j} & \text{if $2k(\alpha_j)=k(\alpha_i)$},\\
m_{\alpha_i}>m_{\alpha_j} & \text{if $k(\alpha_j)=k(\alpha_i)$},
\end{cases}\]
by a similar method. Finally, if $R/G$ is of type $BB^\vee_l\, (l\geq 3)$, we have
$k(\alpha_i)=k(\alpha_j)$ by Lemma \ref{lemma:countings1} (2). 
Therefore, the inequality $m_{\alpha_i}>m_{\alpha_j}$ is obtained, and 
we get the desired results.
\vskip 1mm
\noindent
(3) Note that the condition $|\Gamma^\nw|\ne \emptyset$
is always satisfied. By the condition $|\Gamma^\w|_o\ne\emptyset$, $R/G$ 
should be of type $BCC_l \,(l\geq 2),\, C^\vee BC_l\,(l\geq 2),\, 
BB^\vee_l\,(l\geq 3)$ or $C^\vee C_l\,(l\geq 2)$. By the same reason as 
 in the previous case, we may assume that $\alpha_i$ connects to 
$\alpha_j$ by an edge in $\Gamma$, and we obtain that $k^\nr(\alpha_i)=k(\alpha_i)$. 
In the following table, we give the possible list of 
$\alpha_i$, $\alpha_j$, and exhibit $n_{\alpha_i}$, $n_{\alpha_j}$, 
$I_R(\alpha_i,\alpha_i)$, $I_R(\alpha_j,\alpha_j)$
$I(\alpha_i,\alpha_j^\vee)$ and $I(\alpha_i,\alpha_j^\vee)$.
\begin{center}
\begin{table}[h]
\label{table2-alpha_i-alpha_j}
\vskip -2mm
\resizebox{.98\textwidth}{!}{%
\renewcommand{\arraystretch}{1.3}
\begin{tabular}{|c||c|c|c||c|c|c||c|c|c|}
\noalign{\hrule height0.8pt}
$R/G$&  $\alpha_i$ & $n_{\alpha_i}$ & $I_R(\alpha_i,\alpha_i)$ & $\alpha_j$ &
$n_{\alpha_j}$ & $I_R(\alpha_j,\alpha_j)$ & $I(\alpha_i,\alpha_j^\vee)$ & 
$I(\alpha_j,\alpha_i^\vee)$ \\
\noalign{\hrule height0.8pt}
\hline
$BCC_l\, (l\geq 2)$ & $\alpha_{l-1}$ & $2$ & $4$ & 
$\alpha_l$ & $2$ & $2$ & $-2$ & $-1$\\
\hline
$C^\vee BC_l\, (l\geq 2)$ & $\alpha_{l-1}$ & $1$ & $4$ & 
$\alpha_l$ & $1$ & $2$ & $-2$ & $-1$\\
\hline
$BB^\vee_l\, (l\geq 3)$ & $\alpha_{l-1}$ & $2$ & $4$ & 
$\alpha_l$ & $2$ & $2$ & $-2$ & $-1$\\
\hline
$C^\vee C_l\, (l\geq 2)$ & $\alpha_1$ & $1$ & $4$ & $\alpha_0$ & $1$ & 
$2$ & $-2$ & $-1$\\
\cline{2-9}
& $\alpha_{l-1}$ & $1$ & $4$ & $\alpha_l$ & $1$ & $2$ & $-2$ & $-1$\\
\noalign{\hrule height0.8pt}
\end{tabular}}
\end{table}
\end{center}
\vskip -5mm
By direct computation, we have
\begin{equation}\label{data-exponent2}
m_{\alpha_i}=\begin{cases}
4/k(\alpha_i) & \text{if $BCC_l\, (l\geq 2)$},\\
2/k(\alpha_i) & \text{if $C^\vee BC_l\, (l\geq 2)$},\\
4/k(\alpha_i) & \text{if $BB^\vee_l\, (l\geq 3)$},\\
2/k(\alpha_i) & \text{if $C^\vee C_l \, (l\geq 2)$},\\
\end{cases}
\quad
m_{\alpha_j}=\begin{cases}
2/k^\nr(\alpha_j) & \text{if $BCC_l\, (l\geq 2)$},\\
1/k^\nr(\alpha_j) & \text{if $C^\vee BC_l\, (l\geq 2)$},\\
2/k^\nr(\alpha_j) & \text{if $BB^\vee_l\, (l\geq 3)$},\\
1/k^\nr(\alpha_j) & \text{if $C^\vee C_l \, (l\geq 2)$},\\
\end{cases}
\end{equation}
Thus, the desired inequality is equivalent to
\begin{equation}\label{inequality-goal}
2k^\nr(\alpha_j)\geq k(\alpha_i).
\end{equation} 

By the definition of the (elliptic) root system, the element 
$r_{\alpha_j}r_{(\alpha_j')^\ast}(\alpha_i)$ is a root.  
By Lemma \ref{lemma:translation}, we have
\[r_{\alpha_j}r_{(\alpha_j')^\ast}(\alpha_i)
=\alpha_i-I(\alpha_i,\alpha_j^\vee)k^\nr(\alpha_j)a=\alpha_i+2k^\nr(\alpha_j)a\in R.\]
By Lemma \ref{lemma:countings1} (1), we have 
$k(\alpha_i)|2k^\nr(\alpha_j)$.
This shows \eqref{inequality-goal}, and we get statement (3).
\vskip 1mm
\noindent
(4) The assumption $m_{\alpha_i}>m_{\alpha_j}$ is equivalent to
\begin{equation*}\label{inequality-collapse1}
2k^\nr(\alpha_j)> k(\alpha_i).
\end{equation*} 
By Lemma \ref{lemma:nr-countings1} and \eqref{nr-countings_alpha3}, we have
\begin{equation}\label{inequality-collapse2}
k(\alpha_i)<2k^\nr(\alpha_j)=k^\nr((\alpha_j')^\ast)=
\begin{cases}
\frac12 k((\alpha'_j)^\ast)=\frac12 k(\alpha'_j) & \text{if $2\alpha_j\not\in R$},\\
k((\alpha'_j)^\ast)=k(\alpha'_j) & \text{if $2\alpha_j\in R$}.
\end{cases}
\end{equation}

On the other hand, since $\alpha_j'$ is a long root, and $\alpha_i$ is a middle
root, $I(\alpha'_j,\alpha_i^\vee)=-2$. 
By Lemma \ref{lemma:countings1} (2), we have
\[k(\alpha'_j)|I(\alpha'_j,\alpha_i^\vee)k(\alpha_i)\quad\Longleftrightarrow\quad
k(\alpha_j')|2k(\alpha_i).\]
That is, there exists  a positive integer $m\in\Z_{>0}$ such that 
$m k(\alpha'_j)=2 k(\alpha_i)$. 

Assume $2\alpha_j\not\in R$. By \eqref{inequality-collapse2} and
the above consequence, we have 
\[mk(\alpha'_j)<k(\alpha'_j)\quad \text{for $m\in \Z_{>0}$}.\]
This is a contradiction. Therefore, the element $2\alpha_j$ should belong to $R$, 
and the proof of the proposition is completed.
\end{proof}

\begin{lemma}\label{lemma:9.5}
Let $\alpha_i\in |\Gamma^\w|_o$ and $\alpha_j\in |\Gamma^\nw|$. 
\vskip 1mm
\noindent
{\rm (1)} If $m_{\alpha_i}=m_{\alpha_j}$, we have 
\begin{equation}
k(\alpha_i)=\begin{cases}
k(\alpha_j) & \text{if $\frac12 (\alpha_i')^\ast\not \in R$},\\
2k(\alpha_j) & \text{if $\frac12 (\alpha_i')^\ast \in R$}.
\end{cases}
\end{equation}
{\rm (2)} If $m_{\alpha_i}>m_{\alpha_j}$,
we have
\begin{equation}\label{collapsed-black}
k(\alpha_i)=k(\alpha_j)
\end{equation}
and $\frac12(\alpha_j')^\ast$ is a root.
\end{lemma}
\begin{proof}
(1) As already shown in the proof of the proposition, the condition
$m_{\alpha_i}=m_{\alpha_j}$ is equivalent to $2k^{nr}(\alpha_j)=k(\alpha_i)$.
By this equality and \eqref{nr-countings_alpha4}, statement (1) follows
immediately. 
\vskip 1mm
\noindent
(2) We already proved  that the condition $m_{\alpha_i}>m_{\alpha_j}$ implies
$2\alpha_j\in R$ in statement (4) of the above proposition.
Furthermore, by a similar argument to the above proof, we have
$mk(\alpha'_j)<2k(\alpha'_j)$ for $m\in \Z_{>0}$.
Therefore, the integer $m$ should be equal to $1$, and we get  
$k(\alpha'_j)=2 k(\alpha_i)$.  
In addition, by Proposition \ref{prop:countings-prime} and 
$2\alpha_j \in R$, we have 
\begin{align*}
2k(\alpha_i)=k(\alpha'_j)
&=\begin{cases}
k(\alpha_j) & \text{if }\frac12 (\alpha_j')^\ast\not\in R,\\
2k(\alpha_j) & \text{if }\frac12 (\alpha_j')^\ast\in R.
\end{cases}
\end{align*}

On the other hand, by {Lemma \ref{lemma:countings1}} (2), we have
$k(\alpha_j)|k(\alpha_i)|2k(\alpha_j)$. From this, we obtain
\[k(\alpha_i)=k(\alpha_j)\,\text{or}\,\, 2k(\alpha_j).\]
Therefore, the only possible case is $k(\alpha_i)=k(\alpha_j)$ and we get 
$\frac12 (\alpha_j' )^\ast\in R$, as desired.
\end{proof}
\subsection{Lowest rank cases}
Let us consider low rank exceptional cases. Namely, $(R,G)$ is assumed to be
a mERS such that the quotient affine system $R/G$ is of type $BCC_1$, 
$C^\vee BC_1$, $BB^\vee_2$ or $C^\vee C_1$. In the following table, we give 
basic data for such affine root systems. 
Let $\alpha_i\in |\Gamma^\w|$ and $\alpha_j\in |\Gamma^\nw|$. 
Note that, in the case of type $C^\vee C_1$, both $\alpha_0$ and $\alpha_1$ 
belong to $|\Gamma^\nw|$ and $|\Gamma^\w|=\emptyset$.
\begin{center}
\begin{table}[h]
\label{table1-alpha-beta}
\vskip -1mm
\resizebox{.98\textwidth}{!}{%
\renewcommand{\arraystretch}{1.3}
\begin{tabular}{|c||c|c|c||c|c|c||c|c|}
\noalign{\hrule height0.8pt}
$R/G$&  $\alpha_i$ & $n_{\alpha_i}$ & $I_R(\alpha_i,\alpha_i)$ & $\alpha_j$ & 
$n_{\alpha_j}$ & $I_R(\alpha_j,\alpha_j)$ & 
$I(\alpha_i,\alpha_j^\vee)$ & 
$I(\alpha_j,\alpha_i^\vee)$
\\
\noalign{\hrule height0.8pt}
\hline
$BCC_1$ & $\alpha_0$ & $1$ & $8$ &$\alpha_1$ & $2$ & $2$ & $-4$ & $-1$\\
\hline
$C^\vee BC_1$ & $\alpha_0$ & $1$ & $2$ & $\alpha_1$ & $1$ & $2$ & 
$-2$ & $-2$\\
\hline
$BB_2^\vee$ & $\alpha_0$ & $1$ & $4$ & $\alpha_2$ & $2$ & $2$ & 
$-2$ & $-1$\\
\cline{2-4}\cline{8-9} & $\alpha_1$ & $1$ & $4$ & & & & $-2$ & $-1$\\
\hline
$C^\vee C_1$ & $\emptyset$ &$\times$ & $\times$ & $\alpha_0$ & $1$ & $2$ & 
$-2$ & $-2$\\
\cline{5-7} && &&$\alpha_1$ & $1$ & $2$ &&\\
\hline
\noalign{\hrule height0.8pt}
\end{tabular}}
\end{table}
\end{center}
\vskip -5mm
By direct computation, one has
\begin{equation}\label{data-exponent3}
m_{\alpha_i}=\begin{cases}
4/k(\alpha_i) & \text{if $BCC_1$},\\
1/k(\alpha_i) & \text{if $C^\vee BC_1$},\\
2/k(\alpha_i) & \text{if $BB^\vee_2$},
\end{cases}
\quad
m_{\alpha_j}=\begin{cases}
2/k^\nr(\alpha_j) & \text{if $BCC_1$},\\
1/k^\nr(\alpha_j)
& \text{if $C^\vee BC_1$},\\
2/k^\nr(\alpha_j) & \text{if $BB^\vee_2$},\\
1/k^\nr(\alpha_j) & \text{if $C^\vee C_1$},\\
\end{cases}
\end{equation}

{
\begin{prop}\label{prop:shape-diagram2}
{\rm (1)} Assume $R/G$ is of type $BCC_1$ or $C^\vee BC_1$.
\begin{itemize}
\item[(a)] If $m_{\alpha_0}=m_{\alpha_1}$, one has 
\[k(\alpha_1)/k(\alpha_0)=-\frac{I(\alpha_1,\alpha_0^\vee)}{t},\quad
\text{where } t=\begin{cases}
1 &\text{if }\tfrac12 (\alpha_1')^\ast \not\in R.\\
2 & \text{if }\tfrac12 (\alpha_1')^\ast \in R.\\
\end{cases}\]
\item[(b)] If $m_{\alpha_0}<m_{\alpha_1}$, then  
\[k(\alpha_1)/k(\alpha_0)=-\frac{I(\alpha_1,\alpha_0^\vee)}{t},\quad
\text{where } t=\begin{cases}
2 &\text{if }\tfrac12 (\alpha_1')^\ast \not\in R.\\
4 & \text{if }\tfrac12 (\alpha_1')^\ast \in R.\\
\end{cases}\]
\item[(c)] If $m_{\alpha_0}>m_{\alpha_1}$, one has 
$2\alpha_1\in R,\ \frac12 (\alpha_1')^\ast\in R$ and
\[k(\alpha_1)/k(\alpha_0)=-I(\alpha_1,\alpha_0^\vee).\]
\end{itemize}
\vskip 1mm
\noindent
{\rm (2)} Assume $R/G$ is of type $BB_2^\vee$. 
\begin{itemize}
\item[(a)] One has $m_{\alpha_2}=m_{max}$.
\vskip 1mm
\item[{\rm (b)}] If $m_{\alpha_i}=m_{max}$ for $i=0$ or $1$, then 
one has
$2\alpha_2\in R,\ \frac12 (\alpha_2')^\ast\in R$ and
\[k(\alpha_i)=k(\alpha_2).\]
Furthermore, if $m_{\alpha_j}<m_{max}$ for $j\in \{0,1\}$ such that $j\ne i$, one has
\[k(\alpha_j)=2k(\alpha_2).\]
\vskip 1mm
\noindent
\item[(c)] 
If both $m_{\alpha_0}$ and $m_{\alpha_1}$ are strictly smaller than 
$m_{max}$, then 
\[
k(\alpha_2)/k(\alpha_i)=\dfrac{1}{t},\quad
\text{where }t=
\begin{cases} 
1 & \text{if }\tfrac12(\alpha_2')^\ast\not\in R,\\
2 & \text{if }\tfrac12(\alpha_2')^\ast\in R,
\end{cases}\quad\text{for }i=0,1.\] 
\end{itemize}
\vskip 1mm
\noindent
{\rm (3)} Assume $R/G$ is of type $C^\vee C_1$ and let $i,j$ be indices such that
$\{i,j\}=\{0,1\}$.
\begin{itemize}
\item[(a)] If $m_{\alpha_i}=m_{\alpha_j}$, one has 
\[k(\alpha_j)/k(\alpha_i)=\dfrac2t,\quad
\text{where } t=\begin{cases}
1 & \text{if }\tfrac12 (\alpha_i')^\ast \in R \text{ and }
\tfrac12 (\alpha_j')^\ast \not\in R,\\
4 &\text{if }\tfrac12 (\alpha_i')^\ast \not\in R \text{ and }
\tfrac12 (\alpha_j')^\ast \in R,\\
2 & \text{otherwise}.
\end{cases}\]
\item[(b)] If $m_{\alpha_i}<m_{\alpha_j}$, then one has
$2\alpha_i\in R,\ \frac12 (\alpha_i')^\ast\in R$ and
\[k(\alpha_j)/k(\alpha_i)=\dfrac2t,\quad 
\text{where } t=\begin{cases}
2 & \text{if }\tfrac12(\alpha_j')^\ast\not\in R\\
4 & \text{if }\tfrac12(\alpha_j')^\ast\in R.
\end{cases}\]
\end{itemize}

\vskip 1mm
\noindent
\end{prop}

\begin{proof}
(1) Since the statements for $C^\vee BC_1$ are obtained by a similar 
way, we give only the proof for $BCC_1$. By \eqref{data-exponent3} and Corollary 
\ref{cor:nr-countings1}, we have
\begin{equation}\label{exponents-BCC1}
m_{\alpha_0}\lesseqqgtr m_{\alpha_1}\quad \Longleftrightarrow\quad
k(\alpha_1)/k(\alpha_0)\lesseqqgtr \begin{cases}
1 & \text{if }\tfrac12 (\alpha_1')^\ast\not\in R,\\
\tfrac12 & \text{if }\tfrac12 (\alpha_1')^\ast\in R.
\end{cases}
\end{equation}
On the other hand, by Lemma \ref{lemma:countings1} (2), we have
\[\begin{array}{lllll}
k(\alpha_0)| I(\alpha_0,\alpha_1^\vee)k(\alpha_1) & \Longleftrightarrow & 
k(\alpha_0) | 4k(\alpha_1),\\[3pt]
k(\alpha_1)| I(\alpha_1,\alpha_0^\vee)k(\alpha_0) & \Longleftrightarrow & 
k(\alpha_1)|k(\alpha_0).
\end{array}\]
Therefore, 
\begin{equation}\label{ratio_countings-BCC1}
\text{the number $k(\alpha_1)/k(\alpha_0)$ should be equal to $1,1/2$ or $1/4$.} 
\end{equation}
\vskip 3mm
\noindent
(a) Assume $m_{\alpha_0}=m_{\alpha_1}$. As $I(\alpha_1,\alpha_0^\vee)=-1$, 
statement (a) is obtained by \eqref{exponents-BCC1}. 
\vskip 3mm
\noindent
(b) Assume $m_{\alpha_0}<m_{\alpha_1}$. 
\vskip 1mm
\noindent
i) If $\tfrac12 (\alpha_1')^\ast\not\in R$, we have $k(\alpha_1)/k(\alpha_0)=1/2$ or 
$1/4$ by \eqref{exponents-BCC1} and \eqref{ratio_countings-BCC1}. 
Assume $k(\alpha_1)/k(\alpha_0)=1/4$. 
By Lemma \ref{lemma:translation}, we have
\begin{align*}
r_{\alpha_1}r_{(\alpha_1')^\ast}(\alpha_0)
&=\alpha_0-I(\alpha_0,\alpha_1^\vee)k^{nr}(\alpha_1)a\\
&=\alpha_0+2k(\alpha_1)a\quad (\because\ \tfrac12(\alpha_1')^\ast\not\in R)\\
&=\alpha_0+\tfrac12 k(\alpha_0)a\quad (\because\ k(\alpha_1)/k(\alpha_0)=1/4).
\end{align*}
Since $r_{\alpha_1}r_{(\alpha_1')^\ast}(\alpha_0)$
is a root, the above formula says that $\tfrac12 k(\alpha_0)a$ belongs to the 
counting set $K_G(\alpha_0)$. This contradicts the minimality of $k(\alpha_0)$.
Thus, the number $k(\alpha_1)/k(\alpha_0)$ should be equal to $1/2$. 
\vskip 1mm
\noindent
ii) If $\tfrac12 (\alpha_1')^\ast\in R$, the number $k(\alpha_1)/k(\alpha_0)$ 
should be equal to $1/4$ by \eqref{exponents-BCC1} and
\eqref{ratio_countings-BCC1}, as desired.
\vskip 3mm
\noindent
(c) Assume $m_{\alpha_0}>m_{\alpha_1}$. By \eqref{ratio_countings-BCC1}, we have
$\tfrac12 (\alpha_1')^\ast\in R$ and $k(\alpha_1)/k(\alpha_0)=1$. By Lemma 
\ref{lemma:translation}, we have
\begin{align*}
r_{\alpha_0}r_{(\alpha_0')^\ast}\big((\alpha_1')^\ast\big)
&=(\alpha_1')^\ast-I\big((\alpha_1')^\ast,\alpha_0^\vee\big)k(\alpha_0)a\\
&=(\alpha_1')^\ast+2k(\alpha_0)a\quad (\because\ 2\overline{\alpha_0}\not\in R/G)\\
&=(\alpha_1')^\ast+2 k(\alpha_1)a\quad (\because\ k(\alpha_1)/k(\alpha_0)=1).
\end{align*}
Assume $2\alpha_1\not\in R$. Then, by Proposition \ref{prop:countings-prime}, we
get  $k\big((\alpha_1')^\ast\big)=4k(\alpha_1)$. Substituting it to the above formula, 
it follows that
\[ r_{\alpha_0}r_{(\alpha_0')^\ast}
\big((\alpha_1')^\ast\big)=(\alpha_1')^\ast+\tfrac12 k\big((\alpha_1')^\ast\big)a\]
is a root. This contradicts to the minimality of the counting number 
$k\big((\alpha_1')^\ast\big)$. Thus, the element $2\alpha_1$ should 
{belong} to $R$, 
and we completed  the proof of statement (1).
\vskip 3mm
\noindent
(2) 
Similar to the previous case, we have
\begin{equation}\label{exponents-BB2}
m_{\alpha_i}\lesseqqgtr m_{\alpha_2}\quad \Longleftrightarrow\quad
k(\alpha_2)/k(\alpha_i)\lesseqqgtr \begin{cases}
2 & \text{if }\tfrac12 (\alpha_2')^\ast\not\in R,\\
1 & \text{if }\tfrac12 (\alpha_2')^\ast\in R,
\end{cases}
\quad \text{for }i=0,1,
\end{equation}
by \eqref{data-exponent3} and Corollary \ref{cor:nr-countings1}, and 
\begin{equation}\label{ratio_countings-BB2}
\text{the number $k(\alpha_2)/k(\alpha_i)$ should be equal to $1$ or $1/2$,} 
\end{equation}
by Lemma \ref{lemma:countings1} (2). 
\vskip 3mm
\noindent
(a) Assume $m_{\alpha_i}>m_{\alpha_2}$. 
By \eqref{exponents-BB2}, we have $k(\alpha_2)>k(\alpha_i)$. However, this 
contradicts \eqref{ratio_countings-BB2}. Thus, it follows that 
$m_{\alpha_i}\leq m_{\alpha_2}$ and statement (a) holds. 
\vskip 3mm
\noindent
(b) Assume $m_{\alpha_i}=m_{max}\ (=m_{\alpha_2})$. 
If $\tfrac12 (\alpha_2')^\ast\not\in R$, 
we have $k(\alpha_2)=2k(\alpha_i)$ by \eqref{exponents-BB2}. Therefore, it follows
from Lemma \ref{lemma:translation} that
\[ r_{\alpha_i}r_{(\alpha_i')^\ast}
(\alpha_2)=\alpha_2-I(\alpha_2,\alpha_i^\vee)k^{nr}(\alpha_i)a=\alpha_2+
\tfrac12 k(\alpha_2)a.\]
This contradicts the minimality of $k(\alpha_2)$. Therefore, 
$\tfrac12 (\alpha_2')^\ast\in R$ and $k(\alpha_2)=k(\alpha_i)$. Hence, we have
\[ r_{\alpha_i}r_{(\alpha_i')^\ast}\big((\alpha_2')^\ast\big)
=(\alpha_2')^\ast-I\big((\alpha_2')^\ast,\alpha_i^\vee\big)k^{nr}(\alpha_i)a
=(\alpha_2')^\ast+ 2k(\alpha_2)a.\]
Assume $2\alpha_2\not\in R$. 
Then, one has $k\big((\alpha_2')^\ast\big)=4k(\alpha_2)$ by 
Proposition \ref{prop:countings-prime}. Substituting this to the above equality, one has
\[ r_{\alpha_i}r_{(\alpha_i')^\ast}
\big((\alpha_2')^\ast\big)=(\alpha_2')^\ast+ \tfrac12 k\big((\alpha_2')^\ast\big)a.\]
This contradicts to the minimality of $k\big((\alpha_2')^\ast\big)$, 
and we have $2\alpha_2\in R$. 

Assume $m_{\alpha_j}<m_{max}$ for $j\in \{0,1\}$ such that $j\ne i$. 
Since $\tfrac12 (\alpha_2')^\ast\in R$, the number $k(\alpha_2)/k(\alpha_j)$ should
be equal to $\tfrac12$ by \eqref{exponents-BB2} and \eqref{ratio_countings-BB2}, 
as desired. 
\vskip 3mm
\noindent
(c) Assume $m_{\alpha_i}<m_{max}$ for $i=0,1$. 
\vskip 1mm
\noindent
i) If $\tfrac12 (\alpha_2')^\ast\not\in R$, we have $k(\alpha_2)/k(\alpha_i)=1/2$ or 
$1$ by \eqref{exponents-BB2} and \eqref{ratio_countings-BB2}. 
Furthermore, it follows from Corollary \ref{cor:nr-countings1} that 
$k^{nr}(\alpha_2)=\frac12 k(\alpha_2)$. 
Assume $k(\alpha_2)/k(\alpha_i)=1/2$. By Lemma \ref{lemma:translation}, we have 
\begin{align*}
r_{\alpha_2}r_{(\alpha_2')^\ast}(\alpha_i)
&=\alpha_i-I(\alpha_i,\alpha_2^\vee)k^{nr}(\alpha_2)a
=\alpha_i+k(\alpha_2)a
=\alpha_i+\frac12 k(\alpha_i)a.
\end{align*}
Since $r_{\alpha_2}r_{(\alpha_2')^\ast}(\alpha_i)\in R$, the above formula 
contradicts 
the minimality of $k(\alpha_i)$. Therefore, we get  $k(\alpha_2)/k(\alpha_i)=1$.
\vskip 1mm
\noindent
ii) If $\tfrac12 (\alpha_2')^\ast\in R$, we have $k(\alpha_2)/k(\alpha_i)=1/2$  by 
\eqref{exponents-BB2} and \eqref{ratio_countings-BB2}. 
\vskip 1mm
\noindent
Thus, we proved the statement.
\vskip 3mm
\noindent
(3) We have
\begin{equation}\label{exponents-CC1}
m_{\alpha_i}\lesseqqgtr m_{\alpha_j}\quad \Longleftrightarrow\quad
k(\alpha_j)/k(\alpha_i)\lesseqqgtr 
\begin{cases}
2 & \text{if }\tfrac12 (\alpha_i')^\ast \in R \text{ and }
\tfrac12 (\alpha_j')^\ast \not\in R,\\
\frac12 &\text{if }\tfrac12 (\alpha_i')^\ast \not\in R \text{ and }
\tfrac12 (\alpha_j')^\ast \in R,\\
 1 & \text{otherwise},
\end{cases}
\end{equation}
by \eqref{data-exponent3} and Corollary \ref{cor:nr-countings1}, and 
\begin{equation}\label{ratio_countings-CC1}
\text{the number $k(\alpha_j)/k(\alpha_i)$ should be equal to $2,1$ or $1/2$,} 
\end{equation}
by Lemma \ref{lemma:countings1} (2). By using \eqref{exponents-CC1} and 
\eqref{ratio_countings-CC1}, one has the statements by a similar argument to
the previous cases. We leave a detailed proof to the reader.
\end{proof}}

\begin{rem} 
Since there is a symmetry of exchanging $\alpha_0$ with $\alpha_1$ when
$R/G$ is of type $BB^\vee_2$ or $C^\vee C_1$, we may assume that 
$m_{\alpha_0}\leq m_{\alpha_1}$ for these cases.  
\end{rem}
\subsection{Classification in terms of diagrams}\label{sect:list_of_diagrams}
Let us translate the results of Proposition \ref{prop:shape-diagram1} and 
Proposition \ref{prop:shape-diagram2} into the 
language of the diagram $\Gamma^\downarrow$.  
We say $\alpha_i\in |\Gamma^\downarrow|$ is a 
\textbf{collapsed node}\index[index]{node!collaped@collapsed --} if 
$(\alpha_i')^\ast\not\in |\Gamma(R,G)|$. 
We note that this condition is equivalent to $m_{\alpha_i}<m_{max}$. 

Let $(R,G)$ be a mERS whose quotient root system $R/G$ is of type 
$BCC_l \,(l\geq 1),\, C^\vee BC_l\,(l\geq 1),\, BB^\vee_l\,(l\geq 2)$ or 
$C^\vee C_l\,(l\geq 1)$. A first observation obtained
from these propositions is 
\begin{itemize}
\item[(a)] {\it for every $\alpha_i\in |\Gamma^\w|_o$, we have 
$m_{\alpha_i}=m_{max}$. That is, $|\Gamma^\w|_o$ is a subset of 
$|\Gamma^\w_m|$.}
\end{itemize}
Since $|\Gamma^\w|_o=|\Gamma^\w|$ for the case when $R/G$ is of type 
$C^\vee C_l\,(l\geq 2)$, we have $|\Gamma^\w|_o=|\Gamma^\w_m|$. 
Furthermore, Proposition \ref{prop:shape-diagram1} (2) says that, if $R/G$ is of
type $BB^\vee_l\,(l\geq 3)$, a node $\alpha_j\in |\Gamma^\w|\setminus 
|\Gamma^\w|_o$ is a collapsed node. Therefore, we get 
$|\Gamma^\w|_o=|\Gamma^\w_m|$ in this case. 

The second observation is
\begin{itemize}
\item[(b)] {\it if $\alpha_j\in |\Gamma^\nw|$ is a collapsed node,
it is a black node.}
\end{itemize}
This is nothing but statement (4) of Proposition \ref{prop:shape-diagram1}, 
 statements (1) (c) and (3) (b) of Proposition \ref{prop:shape-diagram2}.  
 when $R/G$ is of type $BB^\vee_2$ is exceptional. In this case, we
observe that 
\begin{itemize}
\item[(c-1)] {\it a colored node $\alpha_2\in |\Gamma^\nw|$ is 
not a collapsed node for any case, and}
\vskip 1mm
\item[(c-2)] {\it if there is a non-collapsed node $\alpha_i\in |\Gamma^\w|$, every
colored node in $|\Gamma^\downarrow|$ is a black node.}
\end{itemize}
\vskip 3mm

Compiling the above results,  we have a ``possible list'' of 
the elliptic diagrams for mERSs with  non-reduced affine quotients.  
Furthermore, the counting numbers 
$k(\alpha)\ (\alpha\in \Pi^{\dc}
)$
are uniquely determined by the condition 
$\mathrm{g.c.d.}\{k(\alpha)\, |\, \alpha\in \Pi^{\dc}\}=1$ 
(Corollary \ref{cor:countings-gcd}). By Proposition \ref{prop:countings-prime},
$k((\alpha')^\ast)\ ((\alpha')^\ast\in (\Pi^{\mc})^\ast)$ are determined also. 
The results are exhibited in the following tables.
\vskip 5mm
\noindent
{\bf 1. The case when  $R/G$ is of type $BCC_l\, (l\geq 1)$, 
$C^\vee BC_l\ (l\geq 1)$ or $BB_l^\vee\ (l\geq 3)$.}
In this case, 
\[
\begin{aligned}
|\Gamma^\w|_o&=\begin{cases}
\{\alpha_1,\ldots,\alpha_{l-1}\} & \text{if $R/G$ is of type 
$BCC_l\, (l\geq 1)$ or $C^\vee BC_l\ (l\geq 1)$,}\\
\{\alpha_2,\ldots,\alpha_{l-1}\} & \text{if $R/G$ is of type 
$BB_l^\vee\ (l\geq 3)$,}
\end{cases}\\
|\Gamma^\w|&=\{\alpha_0,\ldots,\alpha_{l-1}\}\quad\text{and}\quad
|\Gamma_m^\downarrow|=|\Gamma^\w|_o\amalg \left\{
\begin{array}{c}
\text{non-collapsed nodes}\\
\text{in the following tables}
\end{array}
\right\}.
\end{aligned}\]
In the following tables, we exhibit (a) the color of $\alpha_l${\rm ;}
(b) the color of $(\alpha_l')^\ast${\rm ;}
(c) $k(\alpha_j)$ for $\alpha_j\in |\Gamma^\w|\setminus |\Gamma^\w|_o${\rm ;}
(d) $k(\alpha_i)$ for $\alpha_i\in |\Gamma^\w|_o${\rm ;}
(e) $k(\alpha_l)${\rm ;}
(f) $k((\alpha_l')^\ast)$.

\begin{center}
\begin{table}[h]
\caption{ $BCC_l\ (l\geq 1)$}
\label{table:BCC}
\vskip -2mm
\resizebox{.95\textwidth}{!}{%
\renewcommand{\arraystretch}{1.3}
\begin{tabular}{|c||c|c||c|c||c|c|c||c|}
\noalign{\hrule height0.8pt}
No. & $|\Gamma^\w|\setminus |\Gamma^\w|_o$ &
$|\Gamma^\nw|$ & (a) & (b) &(c) & 
(d) & (e) & (f) \\
\noalign{\hrule height0.8pt}
\hline
1 & $\alpha_0$ (not collapsed) & $\alpha_l$\ (not collapsed) &
g & g & $1$ & $1$ & $1$ & $2$\\
\cline{1-1}\cline{4-9}
2 & 
& & g& b& $2$ & $2$ & $1$ & $4$\\
\cline{1-1}\cline{4-9}
3 & 
& & b & g& $1$ & $1$ & $1$ & $1$\\
\cline{1-1}\cline{4-9}
4 & 
& & b & b & $2$ & $2$ & $1$ & $2$\\
\cline{1-1}\cline{3-9}
5 & & $\alpha_l$\ (collapsed)&b& $\times$ & $1$ & $1$ & $1$ & $2$\\
\hline
6 & $\alpha_0$ (collapsed) & $\alpha_l$\ (not collapsed) & g & g&
$2$ & $1$ & $1$ & $2$
\\
\cline{1-1}\cline{4-9}
7 & 
& & g & b & $4$ & $2$ & $1$ & $4$\\
\cline{1-1}\cline{4-9}
8 & 
& & b & g & $2$ & $1$ & $1$ & $1$\\
\cline{1-1}\cline{4-9}
9 & 
& & b & b & $4$ & $2$ & $1$ & $2$\\
\cline{1-1}\cline{3-9}
10 & & $\alpha_l$\ (collapsed)&b&$\times$& $2$ & $1$ & $1$ & $2$\\
\hline
\noalign{\hrule height0.8pt}
\end{tabular}}
\end{table}
\end{center}

\begin{center}
\begin{table}[h]
\caption{ $C^\vee BC_l\ (l\geq 1)$}
\label{table:CBC}
\vskip -2mm
\resizebox{.95\textwidth}{!}{%
\renewcommand{\arraystretch}{1.3}
\begin{tabular}{|c||c|c||c|c||c|c|c||c|}
\noalign{\hrule height0.8pt}
No. &  $|\Gamma^\w|\setminus |\Gamma^\w|_o$ &
$|\Gamma^\nw|$ & (a) & (b) &(c) & 
(d) & (e) & (f) \\
\noalign{\hrule height0.8pt}
\hline
1 & $\alpha_0$ (not collapsed) & $\alpha_l$\ (not collapsed) &
g & g & $1$ & $2$ & $2$ & $4$\\
\cline{1-1}\cline{4-9}
2 & 
& & g & b & $1$ & $2$ & $1$ & $4$\\
\cline{1-1}\cline{4-9}
3 & 
& & b & g & $1$ & $2$ & $2$ & $2$\\
\cline{1-1}\cline{4-9}
4 & 
& & b & b & $1$ & $2$ & $1$ & $2$\\
\cline{1-1}\cline{3-9}
5 & & $\alpha_l$\ (collapsed)&b& $\times$ & $1$ & $2$ & $2$ & $4$\\
\hline
6 & $\alpha_0$ (collapsed) & $\alpha_l$\ (not collapsed) & g & g&
$1$ & $1$ & $1$ & $2$
\\
\cline{1-1}\cline{4-9}
7 & 
& & g&b& $2$ & $2$ & $1$ & $4$\\
\cline{1-1}\cline{4-9}
8 & 
& & b & g & $1$ & $1$ & $1$ & $1$\\
\cline{1-1}\cline{4-9}
9 & 
& & b & b & $2$ & $2$ & $1$ & $2$\\
\cline{1-1}\cline{3-9}
10 & & $\alpha_l$\ (collapsed)&b&$\times$& $1$ & $1$ & $1$ & $2$\\
\hline
\noalign{\hrule height0.8pt}
\end{tabular}}
\end{table}
\end{center}

\begin{center}
\begin{table}[h]
\caption{ $BB^\vee_l\ (l\geq 3)$}
\label{table:BB-general}
\vskip -2mm
\resizebox{.95\textwidth}{!}{%
\renewcommand{\arraystretch}{1.3}
\begin{tabular}{|c||c|c||c|c||c|c|c||c|}
\noalign{\hrule height0.8pt}
No. & $|\Gamma^\w|\setminus |\Gamma^\w|_o$ &
$|\Gamma^\nw|$ & (a) & (b) &(c) & 
(d) & (e) & (f) \\
\noalign{\hrule height0.8pt}
\hline
1 & $\alpha_0\, \text{(collapsed)}$
 & $\alpha_l$\ (not collapsed) &
g & g & $1$ & $1$ & $1$ & $2$\\
\cline{1-1}\cline{4-9}
2 & $\alpha_1\, \text{(collapsed)}$ & 
& g & b & $2$ & $2$ & $1$ & $4$\\
\cline{1-1}\cline{4-9}
3 & & 
& b & g & $1$ & $1$ & $1$ & $1$\\
\cline{1-1}\cline{4-9}
4 & 
& & b & b & $2$ & $2$ & $1$ & $2$\\
\cline{1-1}\cline{3-9}
5 & & $\alpha_l$\ (collapsed)&b& $\times$ & $1$ & $1$ & $1$ & $2$\\
\hline
\noalign{\hrule height0.8pt}
\end{tabular}}
\end{table}
\end{center}
\newpage

\noindent
{\bf 2. The case when $R/G$ is of type $BB^\vee_2$.}  
In this case, 
\[|\Gamma^\w|_o=\emptyset\quad\text{and}\quad
|\Gamma_m^\downarrow|=\left\{
\begin{array}{c}
\text{non-collapsed nodes}\\
\text{in the following table}
\end{array}
\right\}.\]
Here, we exhibit (a) the color of $\alpha_2${\rm ;}
(b) the color of $(\alpha_2')^\ast${\rm ;}
(c) $k(\alpha_0)${\rm ;}
(d) $k(\alpha_1)${\rm ;}
(e) $k(\alpha_2)${\rm ;}
(f) $k((\alpha_2')^\ast)$.

\begin{center}
\begin{table}[h]
\caption{ $BB^\vee_2$}
\label{table:BB2}
\vskip -2mm
\resizebox{.95\textwidth}{!}{%
\renewcommand{\arraystretch}{1.3}
\begin{tabular}{|c||c|c||c|c||c|c|c||c|}
\noalign{\hrule height0.8pt}
No. & $|\Gamma^\w|$ &
$|\Gamma^\nw|$ & (a) & (b) &(c) & 
(d) & (e) & (f) \\
\noalign{\hrule height0.8pt}
\hline
1 & $\begin{array}{ll}
\alpha_0\, \text{(not collapsed)}\\
\alpha_1\, \text{(not collapsed)}
\end{array}$
& $\alpha_2$\ (not collapsed)
& b & b & $1$ & $1$ & $1$ & $2$\\
\cline{1-2}\cline{4-9}
2 & $\begin{array}{ll}
\alpha_0\, \text{(collapsed)}\\
\alpha_1\, \text{(not collapsed)}
\end{array}$
& 
& b & b & $2$ & $1$ & $1$ & $2$\\
\cline{1-2}\cline{4-9}
3 & $\begin{array}{ll}
\alpha_0\, \text{(collapsed)}\\
\end{array}$
& 
& g & g & $1$ & $1$ & $1$ & $2$\\
\cline{1-1}\cline{4-9}
4 & $\alpha_1$ (collapsed)
& 
& g & b & $2$ & $2$ & $1$ & $4$\\
\cline{1-1}\cline{4-9}
5 & 
& & b & g & $1$ & $1$ & $1$ & $1$\\
\cline{1-1}\cline{4-9}
6 & 
& & b & b & $2$ & $2$ & $1$ & $2$\\
\hline
\noalign{\hrule height0.8pt}
\end{tabular}}
\end{table}
\end{center}
\vskip -5mm
\noindent
{\bf 3. The case when $R/G$ is of type $C^\vee C_l\, (l\geq 1)$.}
In this case 
\[|\Gamma^\w|=|\Gamma^\w|_o=\{\alpha_1,\ldots,\alpha_{l-1}\}\quad\text{and}
\quad |\Gamma_m^\downarrow|=|\Gamma^\w|\amalg\left\{
\begin{array}{c}
\text{non-collapsed nodes}\\
\text{in the following table}
\end{array}
\right\}.\]
Here, we exhibit $\text{(a)}_j$ the color of $\alpha_j${\rm ;}
$\text{(b)}_j$ the color of $(\alpha_j')^\ast$ $(j=0,l)$ {\rm ;}
(c) $k(\alpha_0)${\rm ;}
(d) $k((\alpha_0')^\ast)${\rm ;}
(e) $k(\alpha_i)$ for $\alpha_i\in |\Gamma^\w|${\rm ;}
(f) $k(\alpha_l)${\rm ;}
(g) $k((\alpha_l')^\ast)$.
\vspace{-1mm}
\begin{center}
\begin{table}[h]
\caption{ $C^\vee C_l\ (l\geq 1)$}
\label{table:CC}
\vskip -2mm
\resizebox{.90\textwidth}{!}{%
\renewcommand{\arraystretch}{1.3}
\begin{tabular}{|c||c|c||c|c||c|c||c|c||c||c|c|}
\noalign{\hrule height0.8pt}
No. & $\alpha_0$ &
$\alpha_l$ & $\text{(a)}_0$ & $\text{(b)}_0$ &
$\text{(a)}_l$ & $\text{(b)}_l$ & (c) & 
(d) & (e) & (f) & (g)\\
\noalign{\hrule height0.8pt}
\hline
1 & not collapsed & not collapsed &
g & g & g & g & $1$ & $2$ & $1$ & $1$ & $2$\\
\cline{1-1}\cline{4-12}
2 & &
& g & g & g& b& $2$ & $4$ & $2$ & $1$ & $4$\\
\cline{1-1}\cline{4-12}
3 & 
& & g & g & b& g& $1$ & $2$ & $1$ & $1$ & $1$\\
\cline{1-1}\cline{4-12}
4 & 
& & g & g & b& b& $2$ & $4$ & $2$ & $1$ & $2$\\
\cline{1-1}\cline{4-12}
5 & 
& & g & b & g& b& $1$ & $4$ & $2$ & $1$ & $4$\\
\cline{1-1}\cline{4-12}
6 & 
& & g& b & b & g& $1$ & $4$ & $2$ & $2$ & $2$\\
\cline{1-1}\cline{4-12}
7 & 
& & g & b & b & b & $1$ & $4$ & $2$ & $1$ & $2$\\
\cline{1-1}\cline{4-12}
8 & 
& & b & g & b & g & $1$ & $1$ & $1$ & $1$ & $1$\\
\cline{1-1}\cline{4-12}
9 & 
& & b & g & b & b & $2$ & $2$ & $2$ & $1$ & $2$\\
\cline{1-1}\cline{4-12}
10 & 
& & b & b & b & b & $1$ & $2$ & $2$ & $1$ & $2$\\
\hline
11 & collapsed  & not collapsed &
b & $\times$ &g & g & $1$ & $2$ & $1$ & $1$ & $2$\\
\cline{1-1}\cline{4-12}
12 & 
& & b & $\times$ & g& b& $2$ & $4$  & $2$ & $1$ & $4$\\
\cline{1-1}\cline{4-12}
13 & 
& & b & $\times$ & b& g& $1$ & $2$ & $1$ & $1$ & $1$\\
\cline{1-1}\cline{4-12}
14 & 
& & b & $\times$ & b& b& $2$ & $4$ & $2$ & $1$ & $2$\\
\cline{1-1}\cline{3-12}
15 & & collapsed & b & $\times$ &b& $\times$ & $1$ & $2$ & $1$ & $1$ & 
$2$\\
\hline
\noalign{\hrule height0.8pt}
\end{tabular}}
\end{table}
\end{center}
\vspace{ -6mm}
\begin{rem}
{\rm (1)} For every case, the following equality holds{\rm :}
\begin{equation*}
k((\alpha_i')^\ast)=k(\alpha_i)\quad \text{for }\alpha_i\in |\Gamma^\w|.
\end{equation*}
Therefore, for every element of $\Pi^{\dc}\cup \Pi^{\mc}$, its counting number is given
in the above tables. 
\vskip 1mm
\noindent
{\rm (2)} In $|\Gamma^{\downarrow}|$, a non-collapsed node should exist. By this condition,
we need to add the constraint $l\geq 2$ in the following three cases{\rm :} 
No.10 in Table \ref{table:BCC}, No. 10 in Table \ref{table:CBC} and 
No. 15 in Table \ref{table:CC}. 
\end{rem}

Now, we have one of the main results of this article. 

\begin{thm}\label{thm:classification-thm}
Let $(R,G)$ be a mERS with the non-reduced quotient $R/G$ and 
$(\Pi^{\dc},\Pi^{\mc})$
a paired simple system of $(R,G)$.  
After a suitable choice of a numbering of the set $\Pi^{\dc}=\{\alpha_0,\ldots,\alpha_l\}$,
its elliptic diagram $\Gamma(R,G)$ is isomorphic to one in the above list. 
\end{thm}

We observe that the diagrams appeared in the list (Tables \ref{table:BCC}
$\sim$ \ref{table:CC}) above are not isomorphic to
each other. Furthermore, by Proposition 
\ref{prop:indendence-Gamma}, we have the following corollary.

\begin{cor}\label{cor:classification-thm}
The isomorphism classes of mERSs with non-reduced affine quotients
are classified by their elliptic diagrams.
\end{cor} 

However, we note that Theorem \ref{thm:classification-thm} and 
Corollary \ref{cor:classification-thm} {\it does not}
give a complete solution for the 
classification problem of mERSs with non-reduced affine quotients. Indeed,  
we {\it have not} proved yet the existence of a mERS whose elliptic diagram
is isomorphic to every diagram in the above list. 
For getting the complete list of such mERSs, we need to solve the ``existence
problem''. We will discuss this problem in $\S$ \ref{sect_Remarks-classification}. \\

Finally, we give some comments on the exponents and
counting numbers. 
\begin{lemma}\label{lemma:observation-exponents}
Let $(R,G)$ be a mERSs with non-reduced affine quotient.
\vskip 1mm
\noindent
{\rm (1)} One has $m_{\alpha_i}=
\frac12 m_{max}$ for every $\alpha_i\in |\Gamma^{\downarrow}\setminus\Gamma_m|$.
\vskip 1mm
\noindent
{\rm (2)} One has $g.c.d.\{k(\alpha_i)\,|\,\alpha_i\in |\Gamma_m^\downarrow|\}=1$.
\end{lemma}
This lemma is an easy observation following from  
Tables \ref{table:BCC} $\sim$ \ref{table:CC}. 



\subsection{Remarks on the classification theorem}
\label{sect_Remarks-classification}
As we already mentioned at the end of $\S$ \ref{sect:list_of_diagrams}, 
for getting the complete list of mERS with non-reduced affine quotients, 
we need to solve the following problem.
\begin{prob}\label{prob:existence}
For every diagram in Tables \ref{table:BCC} $\sim$ \ref{table:CC} 
preceding Theorem \ref{thm:classification-thm}, does there 
exist a mERS whose elliptic diagram is isomorphic to it?
\end{prob}
In conclusion, the answer is YES. 
That is, we have the main theorem of this article.
\begin{thm}\label{thm:main}
Tables \ref{table:BCC} $\sim$ \ref{table:CC} preceding Theorem 
\ref{thm:classification-thm}
give the complete list of isomorphism classes of mERSs with non-reduced affine 
quotients.
\end{thm}

There are at least two ways for proving Theorem \ref{thm:main}.
\begin{itemize}
\item[(i)] Method 1: Comparison with the results of Chapter 2;
\vskip 1mm
\item[(ii)] Method 2: Direct computation;
\end{itemize}
\vskip 3mm
\noindent
(i) In Method 1, we use our classification theorem
of such mERSs obtained in Chapter 2 (see Theorem \ref{thm_classification-reduced}
and Theorem \ref{thm_classification-non-reduced}, in detail). By direct computation,
one can check that there exists a unique mERS $(R,G)$ whose elliptic diagram
is isomorphic to one in Tables \ref{table:BCC} $\sim$ \ref{table:CC}. 
The explicit correspondence is given in the following tables.
\vskip 3mm
\noindent
{\bf 1}. The case when $R/G$ is of type $BCC_l\, (l\geq 1)$.\\[-15pt]
\begin{center}
\begin{table}[h]
\label{table:BCC-corresp}
\resizebox{.98\textwidth}{!}{%
\renewcommand{\arraystretch}{1.3}
\begin{tabular}{|c|c||c|c||c|c|c|c||c|}
\noalign{\hrule height0.8pt}
No. & Name of $(R,G)$ &
No.  & Name of $(R,G)$ & No. & Name of $(R,G)$\\
\noalign{\hrule height0.8pt}
\hline
1 & $BC_l^{(1,1)\ast}\, (l\geq 1)$ & 
5 & $BCC_l^{(1)_{{\ast}_{0'}}}\, (l\geq 1)$ & 
8 & $BCC_l^{(1)_{{\ast}_{0}}}\, (l\geq 1)$\\
\hline
2 & $BCC_l^{(2)_{{\ast}_1}}\, (l\geq 1)$ & 
6 & $BC_l^{(1,2)}\, (l\geq 1)$ & 
9 & $BCC_l^{(2)_{{\ast}_{0}}}\, (l\geq 1)$\\
\hline
3 & $BCC_l^{(1)}\, (l\geq 1)$& 
7 & $BCC_l^{(4)}\, (l\geq 1)$ & 
10 & $BCC_l^{(2)}(1)\, (l\geq 2)$\\
\hline
4 & $BCC_l^{(2)}(2)\, (l\geq 1)$\\
\cline{1-2}
\end{tabular}}
\end{table}
\end{center}
\vskip -5mm
\noindent
{\bf 2}. The case when $R/G$ is of type $C^\vee BC_l\, (l\geq 1)$.\\[-15pt]
\begin{center}
\begin{table}[h]
\label{table:CBC-corresp}
\resizebox{.95\textwidth}{!}{%
\renewcommand{\arraystretch}{1.3}
\begin{tabular}{|c|c||c|c||c|c|c|c||c|}
\noalign{\hrule height0.8pt}
No. & Name of $(R,G)$ &
No.  & Name of $(R,G)$ & No. & Name of $(R,G)$\\
\noalign{\hrule height0.8pt}
\hline
1 & $BC_l^{(4,4)\ast}\, (l\geq 1)$ & 
5 & $C^\vee BC_l^{(4)_{{\ast}_{0'}}}\, (l\geq 1)$ & 
8 & $C^\vee BC_l^{(1)}\, (l\geq 1)$\\
\hline
2 & $C^\vee BC_l^{(4)}\, (l\geq 1)$ & 
6 & $BC_l^{(4,2)}\, (l\geq 1)$ & 
9 & $C^\vee BC_l^{(2)_{{\ast}_{0}}}\, (l\geq 1)$\\
\hline
3 & $C^\vee BC_l^{(2)_{\ast_1}}\, (l\geq 1)$& 
7 & $C^\vee BC_l^{(4)_{\ast_0}}\, (l\geq 1)$ & 
10 & $C^\vee BC_l^{(2)}(1)\, (l\geq 2)$\\
\hline
4 & $C^\vee BC_l^{(2)}(2)\, (l\geq 1)$\\
\cline{1-2}
\end{tabular}}
\end{table}
\end{center}
\vskip -5mm
\noindent
{\bf 3}. The case when $R/G$ is of type $BB_l^\vee\, (l\geq 3)$.\\[-15pt]
\begin{center}
\begin{table}[h]
\label{table:BBl-corresp}
\resizebox{.95\textwidth}{!}{%
\renewcommand{\arraystretch}{1.3}
\begin{tabular}{|c|c||c|c||c|c|c|c||c|}
\noalign{\hrule height0.8pt}
No. & Name of $(R,G)$ &
No.  & Name of $(R,G)$ & No. & Name of $(R,G)$\\
\noalign{\hrule height0.8pt}
\hline
1 & $BC_l^{(2,2)\sigma}(1)\, (l\geq 3)$ & 
3 & $BB_l^{\vee (1)}\, (l\geq 3)$ & 
5 & $BB^{\vee (2)}_l(1)\, (l\geq 3)$\\
\hline
2 & $BB_l^{\vee (4)}\, (l\geq 3)$ & 
4 & $BB_l^{\vee (2)}(2)\, (l\geq 3)$ 
\\
\cline{1-4}
\end{tabular}}
\end{table}
\end{center}
\vskip -5mm
\noindent
{\bf 4}. The case when $R/G$ is of type $BB_2^\vee$.\\[-15pt]
\begin{center}
\begin{table}[h]
\label{table:BB2-corresp}
\resizebox{.95\textwidth}{!}{%
\renewcommand{\arraystretch}{1.3}
\begin{tabular}{|c|c||c|c||c|c|c|c||c|}
\noalign{\hrule height0.8pt}
No. & Name of $(R,G)$ &
No.  & Name of $(R,G)$ & No. & Name of $(R,G)$\\
\noalign{\hrule height0.8pt}
\hline
1 & $BB_2^{\vee (2)}(1)$ & 
3 & $BC_2^{(2,2)\sigma}(1)$ &
5 & $BB_2^{\vee (1)}$\\
\hline
2 & $BB_2^{\vee (2)\ast}$ & 
4 & $BB_2^{\vee (4)}$ &
6 & $BB_2^{\vee (2)}(2)$ \\
\hline
\end{tabular}}
\end{table}
\end{center}
\newpage
\vskip -5mm
\noindent
{\bf 5}. The case when $R/G$ is of type $C^\vee C_l\, (l\geq 1)$.\\[-15pt]
\begin{center}
\begin{table}[h]
\label{table:CC-corresp}
\resizebox{.95\textwidth}{!}{%
\renewcommand{\arraystretch}{1.3}
\begin{tabular}{|c|c||c|c||c|c|c|c||c|}
\noalign{\hrule height0.8pt}
No. & Name of $(R,G)$ &
No.  & Name of $(R,G)$ & No. & Name of $(R,G)$\\
\noalign{\hrule height0.8pt}
\hline
1 & $BC_l^{(2,2)\sigma}(2)\, (l\geq 1)$ & 
6 & $C^\vee C_l^{(2)_{{\ast}_{1}}}\, (l\geq 1)$ & 
11 & $C^\vee C_l^{(2)\diamond}\, (l\geq 1)$\\
\hline
2 & $C^\vee C_l^{(4)_{\ast_1}}\, (l\geq 1)$ & 
7 & $C^\vee C_l^{(2)_{\ast_l}}\, (l\geq 1)$ & 
12 & $C^\vee C_l^{(4)_{{\ast}_{0}}}\, (l\geq 1)$\\
\hline
3 & $C^\vee C_l^{(1)_{\ast_1}}\, (l\geq 1)$& 
8 & $C^\vee C_l^{(1)}\, (l\geq 1)$ & 
13 & $C^\vee C_l^{(1)_{\ast_0}}\, (l\geq 1)$\\
\hline
4 & $C^\vee C_l^{(2)_{\ast_{1'}}}\, (l\geq 1)$ &
9 & $C^\vee C_l^{(2)_{\ast_{s}}}\, (l\geq 1)$&
14 & $C^\vee C_l^{(2)_{\ast_0}}\, (l\geq 1)$\\
\hline
5 & $C^\vee C_l^{(4)}\, (l\geq 1)$ &
10 & $C^\vee C_l^{(2)}(2)\, (l\geq 1)$&
15 & $C^\vee C_l^{(2)}(1)\, (l\geq 2)$\\
\hline
\end{tabular}}
\end{table}
\end{center}
\vskip -5mm
\noindent
(ii) In Method 2, we regard \eqref{description-R3} as the ``definition'' of the set $R$
from the data in Tables \ref{table:BCC} $\sim$ \ref{table:CC} by the following way. 
For example, we explain the case of type $BCC_l$ (Table \ref{table:BCC}).
\begin{itemize}
\item[1.] First, we construct a vector space $F$ and a symmetric bilinear form 
$I_F$ on $F$. Set $F_a=\oplus_{i=0}^l \R \alpha_i$, and let
$I_{F_a}:F_a\times F_a\to\R$ be a symmetric bilinear form on $F_a$ such that 
$\Pi^{\dc}=\{\alpha_0,\ldots,\alpha_l\}$ gives
a simple system of type $BC_l^{(2)}$ (in the sense of \cite{Saito1985}). 
Define $F=F_a \oplus \R a$, and let 
$I_F:F\times F\to \R$ be a symmetric bilinear form on $F$ such that
$I_F|_{F_a\times F_a}=I_{F_a}$ and $\R a\subset \mathrm{rad}(I_F)$.
\item[2.] Let $r_{\alpha_i}:F\to F$ be the reflection attached
to $\alpha_i\in\Pi^{\dc}$, and $\W[\Pi^{\dc}]$ be the group generated by the reflections
$r_{\alpha_i}\ (\alpha_i\in\Pi^{\dc})$.    
\item[3.] From the data in Table \ref{table:BCC}, one can construct the 
corresponding diagrams 
(see $\S$ \ref{sect_root-data-non-reduced-1} below). 
Note that there are ten different diagrams for each 
$l\geq 2$, and nine different diagrams for $l=1$. 
\item[4.] {Substituting} the data on counting numbers in Tables \ref{table:BCC}
to the right hand sides of \eqref{counting-set-expform2}, the subsets $R(\alpha_i),
\, R(\alpha_j)$ and $R((\alpha_j')^\ast)$ of $F$ are obtained. Furthermore, 
substituting these results to the right hand side of \eqref{description-R3}, the
subset $R$ of $F$ is obtained.
\end{itemize}
For each ten (or nine) cases in Table \ref{table:BCC}, one can check that this 
$R$ satisfies the axioms (R1) $\sim$ (R5) in Definition \ref{defn_root-sys} by 
direct computation. Hence, Problem \ref{prob:existence} is solved affirmatively
in the case of type $BCC_l$. For the other cases, one can solve the problem
in a similar way.\\

\begin{rem}\label{rem:reconstruction}{\rm
There is the third way for proving Theorem \ref{thm:main}.
\begin{itemize}
\item[(iii)] Method 3: reconstruction theorem.
\end{itemize}
As we explained above, both in Method 1 and in Method 2, 
Theorem \ref{thm:main} is proved
by case-by-case direct computation. 
On the other hand, in Method 3, we can show the theorem 
in a uniformed way. 

In this method, starting with a given diagram $\Gamma=\Gamma(R,G)$ in Tables 
\ref{table:BCC} $\sim$ \ref{table:CC} preceding Theorem \ref{thm:classification-thm},
we reconstruct the mERS $(R,G)$ as follows.
\begin{itemize}
\item[1.] Consider the real vector space $\widehat{F}$ spanned by all the nodes in 
$\Gamma$ equipped with the symmetric bilinear form $\widehat{I}$ on $\widehat{F}$
defined from the information of the edges in $\Gamma$. 
Note that the dimension of $\widehat{F}$ is nothing but the cardinality of the set 
$|\Gamma|$ of all nodes in $\Gamma$. 
Let $\hat{r}_\alpha\in O(\widehat{F},\widehat{I})$ be the reflection attached
to $\alpha\in |\Gamma|$, $\widehat{W}$ be the group generated by 
$\hat{r}_\alpha\ (\alpha\in |\Gamma|)$ and $\widehat{R}:=\widehat{W}.|\Gamma|$. 
\item[2.] Let $\hat{c}$ be a product of all reflections 
$\hat{r}_\alpha\ (\alpha\in |\Gamma|)$ in a suitable order\footnote{An element
$\hat{c}$ is called a {\it pre-Coxeter element}.}.  Take the Jordan decomposition 
$\hat{c}=SU$ of $\hat{c}$, 
where $S$ is the semisimple part of $\hat{c}$ and $U$ is the unipotent part, 
respectively. Set 
$F_{\Gamma}=\widehat{F}/ \mathrm{Im}(U-\mathrm{id})$ and let $I_\Gamma$
be a bilinear form on $F_\Gamma$ induced from $\widehat{I}$. 
Let $\widehat{G}$ be the subspace of $\widehat{F}$ spanned by vectors 
$\alpha_i^\ast-\alpha_i\ (\alpha_i\in |\Gamma_m^\w|)$ and 
$(\alpha_j')^\ast-2\alpha_j\ (\alpha_j\in |\Gamma_m^\nw|)$. Then, 
the image $G_\Gamma:=\pi_{\mathrm{Im}(U-\mathrm{id})}(\widehat{G})$ is 
a one dimensional subspace of $\mathrm{rad}(I_\Gamma)$. 
\item[3.] 
It can be shown that the image
$R_\Gamma$ of $\widehat{R}$ under the canonical projection 
$\pi_{\mathrm{Im}(U-\mathrm{id})}:\widehat{F}\to F_\Gamma$ 
is an elliptic root system belonging to $(F_\Gamma,G_\Gamma)$.
Furthermore, the pair $(R_\Gamma,G_\Gamma)$ is
a mERS whose elliptic diagram is isomorphic to $(R,G)$.
\end{itemize}
Thus, it is shown that, for every diagram $\Gamma$ in Tables \ref{table:BCC} 
$\sim$ \ref{table:CC} preceding Theorem \ref{thm:classification-thm}, a mERS whose 
elliptic diagram is $\Gamma$, exists, and 
Theorem \ref{thm:main} is proved.

Note that the above method (Method 3) is a non-reduced analogue of the result of 
K. Saito for $(R,G)$ with the reduced affine quotient $R/G$ (\cite{Saito1985}, $\S$ 9), 
and the detail for our case will be explained in our forthcoming paper. 
}
\end{rem}
\vskip 5mm
\section{Structure of reduced pairs}\label{sect_str-red-pair}

Let $(R,G)$ be a mERS with the non-reduced affine quotient $R/G$,
and $(R^{\dc},R^{\mc})$ a reduced pair of $R$. The types of the reduced mERSs
$(R^{\dc},G)$ and $(R^{\mc},G)$ are already given in 
$\S$ \ref{sect_red-pair}. 
On the other hand, the isomorphism classes of mERSs are classified in 
Theorems \ref{thm:classification-red} and \ref{thm:main} by their elliptic diagrams. 
Therefore, we arrive at the following problem.
\begin{prob}
Can we extract  the types of $(R^{\dc},G)$ and $(R^{\mc},G)$ from the elliptic
diagram $\Gamma(R,G)$?
\end{prob}
The aim of this subsection is to solve this problem. 
After detailed study on the structures of $(R^{\dc},G)$ and $(R^{\mc},G)$, we will give
an answer for this problem in $\S$ \ref{sect:ans-Problem2}
(see Proposition \ref{prop:red-pair}).
\subsection{Reduced pair and Elliptic diagram}
Let us rewrite the results of Proposition \ref{prop:description-R} in terms of the elliptic 
diagrams.  
Recall the description \eqref{description-R-2} of $R$ and \eqref{description-nw}:
\begin{equation*}
R=\Big(\bigsqcup_{\gamma\in \W[\Pi^{\dc}].\Pi^{\dc}}(\gamma+\Z k(\gamma)a)\Big)
\bigcup
\Big(\bigsqcup_{\gamma\in \W[\Pi^{\dc}].(\Pi^{\mc}\setminus \Pi^{\dc})^\ast}
(\gamma+\Z k(\gamma)a)\Big),
\end{equation*}
\begin{equation*}
(\Pi^{\mc}\setminus \Pi^{\dc})^\ast
=\{(\alpha')^\ast\,|\,\alpha\in \Pi^{\dc},\, 2\overline{\alpha}\in R/G\},
\end{equation*}
respectively. Hence, we have
\[(\Pi^{\mc}\setminus \Pi^{\dc})^\ast=\{(\alpha_j')^\ast\,|\,\alpha_j\in |\Gamma^\nw|\}.\]
Note that, if $\alpha_j\in |\Gamma^\nw|$ is a collapsed node, 
the node $(\alpha_j')^\ast$ does not appear in $\Gamma(R,G)$. 

For $\alpha_i\in \Pi^{\dc}=|\Gamma^{\downarrow}|$, set
\begin{align*}
R(\alpha_i)=&\bigcup_{w\in \W[\Pi^{\dc}]}(w(\alpha_i)+\Z k(\alpha_i)a), \\
R((\alpha_i')^\ast)=
&\bigcup_{w\in \W[\Pi^{\dc}]}(w((\alpha_i')^\ast)+\Z k((\alpha_i')^\ast)a).
\end{align*}
\index[notations]{r911@$R(\alpha_i)$}
\index[notations]{r912@$R((\alpha_i')^\ast)$}
Thus, we can rewrite \eqref{description-R-2} as 
\begin{equation}\label{description-R3}
R=\Big(\bigcup_{\alpha_i\in |\Gamma^\w|}R(\alpha_i)\Big)
\bigcup \Big(\bigcup_{\alpha_j\in |\Gamma^\nw|}R(\alpha_j)\Big)
\bigcup \Big(\bigcup_{\alpha_j\in |\Gamma^\nw|}R((\alpha_j')^\ast)\Big).
\end{equation}
Set
\begin{align*}
|\Gamma^\nw_m|_g^g & =\{\alpha_j\in |\Gamma^\nw_m|\,|\,
\text{$\alpha_j$ is a gray node, $(\alpha_j')^\ast$ is a gray node}\},\\
|\Gamma^\nw_m|_g^b & =\{\alpha_j\in |\Gamma^\nw_m|\,|\,
\text{$\alpha_j$  is a gray node, $(\alpha_j')^\ast$ is a black node}\},\\
|\Gamma^\nw_m|_b^g & =\{\alpha_j\in |\Gamma^\nw_m|\,|\,
\text{$\alpha_j$ is a black node, $(\alpha_j')^\ast$ is a gray node}\},\\
|\Gamma^\nw_m|_b^b & =\{\alpha_j\in |\Gamma^\nw_m|\,|\,
\text{$\alpha_j$ is a black node, $(\alpha_j')^\ast$ is a black node}\}
\end{align*}
\index[notations]{g911@$\vert\Gamma^\nw_m\vert_g^g$}
\index[notations]{g912@$\vert\Gamma^\nw_m\vert_g^b$}
\index[notations]{g913@$\vert\Gamma^\nw_m\vert_b^g$}
\index[notations]{g914@$\vert\Gamma^\nw_m\vert_b^b$}
and
\begin{align*}
|(\Gamma^{\downarrow}\setminus \Gamma_m^\downarrow)^\nw| & 
=|\Gamma^\nw|\setminus |\Gamma_m^\nw|.
\end{align*}
\index[notations]{g915@$\vert(\Gamma^{\downarrow}\setminus 
\Gamma_m^\downarrow)^\nw\vert$}
By the observation (b) in $\S$ \ref{sect:list_of_diagrams}, we know that
$\alpha_j\in |(\Gamma^{\downarrow}\setminus\Gamma_m)^\nw|$ is a black node.
By using these notations, \eqref{counting-set-expform} can be rewritten as 
\begin{equation}\label{counting-set-expform2}
\begin{split}
& R((\alpha_j')^\ast) \\
=
&\begin{cases}
\bigcup_{w\in \W[\Pi^{\dc}]}
(2w(\alpha_j)+(2\Z+1)k(\alpha_j)a) & \text{if $\alpha_j\in |\Gamma^\nw_m|_g^g$},\\
\bigcup_{w\in \W[\Pi^{\dc}]}
(2w(\alpha_j)+(4\Z+2)k(\alpha_j)a) & \text{if $\alpha_j\in |\Gamma^\nw_m|_g^b$},\\
\bigcup_{w\in \W[\Pi^{\dc}]}
(2w(\alpha_j)+\Z k(\alpha_j)a) & \text{if $\alpha_j\in |\Gamma^\nw_m|_b^g$},\\
\bigcup_{w\in \W[\Pi^{\dc}]}
(2w(\alpha_j)+2\Z k(\alpha_j)a) & \text{if $\alpha_j\in |\Gamma^\nw_m|_b^b$ or 
$|(\Gamma^{\downarrow}\setminus\Gamma_m^\downarrow)^\nw|$}.
\end{cases}
\end{split}
\end{equation}
\vskip 3mm
\begin{prop}\label{prop:description-R'-R''}
Let $(R,G)$ be a mERS with the non-reduced quotient $R/G$ and $(R^{\dc},R^{\mc})$ be
the reduced pair of $R$. We have
\begin{equation}\label{description-R'}
R^{\dc}=\Big(\bigcup_{\alpha_i\in |\Gamma^w|}R^{\dc}(\alpha_i)\Big)
\bigcup \Big(\bigcup_{\alpha_j\in |\Gamma^\nw|}R^{\dc}(\alpha_j)\Big)
\bigcup \Big(\bigcup_{\alpha_j\in |\Gamma^\nw|}R^{\dc}((\alpha_j')^\ast)\Big),
\end{equation}
\begin{equation}\label{description-R''}
R^{\mc}=\Big(\bigcup_{\alpha_i\in |\Gamma^w|}R^{\mc}(\alpha_i)\Big)
\bigcup \Big(\bigcup_{\alpha_j\in |\Gamma^\nw|}R^{\mc}(\alpha_j)\Big)
\bigcup \Big(\bigcup_{\alpha_j\in |\Gamma^\nw|}R^{\mc}((\alpha_j')^\ast)\Big),
\end{equation}
where, for $\alpha_i\in |\Gamma^\w|$ and $\alpha_j\in |\Gamma^\nw|$, 
\[R^{\dc}(\alpha_i):=R(\alpha_i),\quad  R^{\dc}(\alpha_j):=R(\alpha_j),\]
\[R^{\dc}((\alpha_j')^\ast):=\begin{cases}
R((\alpha_j')^\ast) & \text{if $\alpha_j\in |\Gamma_m^\nw|_g^g$},\\
\emptyset & \text{if $\alpha_j\in |\Gamma_m^\nw|_g^b$},\\
\bigcup_{w\in \W[\Pi^{\dc}]}
(2w(\alpha_j)+(2\Z +1)k(\alpha_j)a)  & \text{if $\alpha_j\in |\Gamma_m^\nw|_b^g$},\\
\emptyset & \begin{array}{l}
\!\!\!\text{if $\alpha_j\in |\Gamma_m^\nw|_b^b$}\\ 
\; \text{or $|(\Gamma^{\downarrow}\setminus\Gamma_m^\downarrow)^\nw|$},
\end{array}
\end{cases}\]
\[R^{\mc}(\alpha_i):=R(\alpha_i),\]
\[ 
R^{\mc}(\alpha_j):=\begin{cases}
R(\alpha_j), & \text{if $\alpha_j\in |\Gamma_m^\nw|_g^g$},\\
\bigcup_{w\in \W[\Pi^{\dc}]}
(w(\alpha_j)+2\Z k(\alpha_j)a) & \text{if $\alpha_j\in |\Gamma_m^\nw|_g^b$},\\
\emptyset & 
\text{if $\alpha_j\in |\Gamma^\nw|_b^g$},\\
\emptyset & \begin{array}{l}
\!\!\!\text{if $\alpha_j\in |\Gamma_m^\nw|_b^b$}\\ 
\; \text{or $|(\Gamma^{\downarrow}\setminus\Gamma_m^\downarrow)^\nw|$},
\end{array}
\end{cases}\]
\[R^{\mc}((\alpha_j')^\ast):=R((\alpha_j')^\ast).\]
\end{prop}
\begin{proof}
Recalling the definitions of $R^{\dc}$ and $R^{\mc}$, it is immediate to see the following
equalities: 
\begin{align*}
R^{\dc}(\alpha_i)=
&R(\alpha_i)\cap R^{\dc}, \qquad R^{\dc}(\alpha_j)=R(\alpha_j)\cap R^{\dc}, \\
R^{\dc}((\alpha_j')^\ast)=
&R((\alpha_j')^\ast)\cap R^{\dc},\\
R^{\mc}(\alpha_i)=
&R(\alpha_i)\cap R^{\mc}, \qquad 
R^{\mc}(\alpha_j)=R(\alpha_j)\cap R^{\mc},\\
R^{\mc}(\alpha_j')^\ast)=
&R((\alpha_j')^\ast)\cap R^{\mc}.
\end{align*}
Hence, the proposition is obtained.
\end{proof} 

\begin{cor}\label{cor:reducedness}
For a mERS $(R,G)$, the following are equivalent{\rm :}
\begin{itemize}
\item[(a)] $R$ is reduced.
\item[(b)] There is no black node in its elliptic diagram $\Gamma(R,G)$. 
\end{itemize}.
\end{cor}
\begin{proof}
If $\Gamma(R,G)$ has a black node, it is obvious that $R$ is not reduced. 
Thus, we have (a) $\Longrightarrow$ (b). Conversely, assume $\Gamma(R,G)$ 
has no black node. By the proposition above, we have
\[R^{\dc}(\alpha_i)=R^{\mc}(\alpha_i)\quad \text{for every }\alpha_i\in |\Gamma^\w|,\]
\[R^{\dc}(\alpha_j)=R^{\mc}(\alpha_j)\quad \text{and}\quad 
R^{\dc}((\alpha_j')^\ast)=R^{\mc}((\alpha_j')^\ast)\quad 
\text{for every }\alpha_j\in |\Gamma^\nw|.\]
Thus, it is obtained that $R^{\dc}=R^{\mc}$. This means $R$ is reduced. 
\end{proof}

By this corollary and the list of the isomorphism classes on mERSs 
in the previous subsection, the list of reduced mERS is obtained:
\begin{center}
No.1 and No. 6 in Table \ref{table:BCC}, \quad 
No.1 and No. 6 in Table \ref{table:CBC},\\[3pt]
No.1 in Table \ref{table:BB-general}, \quad  No.3 in Table \ref{table:BB2},\quad
No.1 in Table \ref{table:CC}.
\end{center}
Note that we have not proved the existence of such root systems yet (c.f. 
Problem \ref{prob:existence} in the next subsection).

In the rest of this subsection, assume $(R,G)$ is non-reduced. 
Note that this assumption is equivalent to  that
there is a black node in $\Gamma(R,G)$ by Corollary \ref{cor:reducedness}.
In the following, we study the reduced mERSs 
$(R^{\dc},G)$ and $(R^{\mc},G)$ individually. 
For describing the elliptic diagrams
$\Gamma(R^{\dc},G)$ and $\Gamma(R^{\mc},G)$, we need to determine
the following data: 
\begin{itemize}
\item the shape of these diagrams (nodes and edges);
\vskip 1mm
\item the color of nodes in them.
\end{itemize}
Especially, to determine the shapes of them, we need to compute the exponents 
of the nodes in 
$|\Gamma^{\downarrow}(R^{\dc},G)|$ and $|\Gamma^{\downarrow}(R^{\mc},G)|$, 
respectively. Furthermore, to compute the exponents, we have to determine 
the non-reduced counting numbers for these nodes 
(c.f. Definition \ref{defn:exponent-non-reduced}). 
For $(R^{\dc},G)$, every data is determined in $\S$ \ref{sect_R'}
(see {Proposition} \ref{prop:Gamma(R',G)}), and for $(R^{\mc},G)$, 
they are done in $\S$ \ref{sect_R''} (see Proposition \ref{prop:Gamma(R'',G)}). 
Compiling these results, we have an answer for Problem 2 in  
$\S$ \ref{sect:ans-Problem2} (Proposition \ref{prop:red-pair}). 
\vskip 5mm

\subsection{ Structure of $R^{\dc}$}\label{sect_R'}
First, we treat the reduced mERS $(R^{\dc},G)$. 
Let $(\Pi_{R^{\dc}}^{\dc},\Pi_{R^{\dc}}^{\mc})$ be a paired simple system of 
$(R^{\dc},G)$ and $\Gamma(R^{\dc},G)$ be its elliptic diagram.
Denote $\Pi_{R^{\dc}}^{\dc}=\{\beta_0,\ldots,\beta_l\}$. 
By the definition, the linearly independent subset 
$\Pi_{R^{\dc}}^{\mc}=\big\{(\beta_0)_{R^{\dc}}',\ldots,
(\beta_l)_{R^{\dc}}'\big\}$ is given by
\[(\beta_i)_{R^{\dc}}'=\big(\beta_i\big)^{pr}_{R^{\dc}}
\quad \text{for }0\leq i\leq l,\]
where $(\,\cdot\,)^{pr}_{R^{\dc}}:R^{\dc}\to R^{\dc}$ is the prime map for $R^{\dc}$. 

The subdiagram of $\Gamma(R^{\dc},G)$ whose nodes consist of all elements of 
$\Pi_{R^{\dc}}^{\dc}$ is denoted by $\Gamma^{\downarrow}(R^{\dc},G)$. Then, we 
have
\begin{equation}\label{decomp_Gamma_R'}
|\Gamma(R^{\dc},G)|=|\Gamma^{\downarrow}(R^{\dc},G)|\amalg 
|\Gamma_m^{\uparrow}(R^{\dc},G)^\ast|.
\end{equation}

On the other hand, let $(\Pi^{\dc},\Pi^{\mc})$ be a paired simple system of 
$(R,G)$ and denote $\Pi^{\dc}=\{\alpha_0,\ldots,\alpha_l\}$. Note that
$\Pi^{\dc}$ is a simple system of $(R^{\dc},G)$. Since $R^{\dc}$ is reduced, one has 
$R^{\dc} \cap \frac12 (R^{\dc})_l=\emptyset$, and
$\Pi^{\dc}$ satisfies the condition \eqref{condition-Pi'} as a simple 
system of $(R^{\dc},G)$, automatically. Therefore, by 
Proposition \ref{prop:nr-counting-Autom1}, 
we may assume 
\[\Pi^{\dc}_{R^{\dc}}=\Pi^{\dc}\quad \text{and}\quad 
\beta_i=\alpha_i\quad \text{for }0\leq i\leq l.\]
That is, for a fixed simple system 
$\Pi^{\dc}=\{\alpha_0,\ldots,\alpha_l\}$, we may assume that  
a paired simple system $(\Pi^{\dc}_{R^{\dc}},\Pi^{\dc}_{R^{\mc}})$ has the following forms:
\[\Pi^{\dc}_{R^{\dc}}=\{\alpha_0,\ldots,\alpha_l\}\quad\text{and}\quad
\Pi^{\dc}_{R^{\mc}}=\big\{(\alpha_0)_{R^{\dc}}',\ldots,(\alpha_l)_{R^{\dc}}'\big\}.\]
Note that, since $(\alpha_i)'_{R^{\dc}}=(\alpha_i)^{pr}_{R^{\dc}}$ 
is the image of $\alpha_i\in \Pi^{\dc}$ under the prime map 
$(\,\cdot\,)^{pr}_{R^{\dc}}:R^{\dc}\to R^{\dc}$ for $R^{\dc}$, 
\index[index]{prime map!prime map dc@-- for $R^{\dc}$}
it {\it does not} coincide
with $\alpha_i'=(\alpha_i)^{pr}\in R^{\mc}$ in general, 
where $(\,\cdot\,)^{pr}:R\to R^{\mc}$ is the prime map for $R$.
{Here and after, when we regard $\alpha_i\in \Pi^{\dc}\, (0\leq i\leq l)$ 
as an element of $\Pi_{R^{\dc}}^{\dc}$, we denote it $(\alpha_i)_{R^{\dc}}$.}
\index[notations]{a911@$(\alpha_i)_{R^{\dc}}$}
\\

The goal of $\S$ \ref{sect_R'} is the following proposition.

\begin{prop}\label{prop:Gamma(R',G)}
{\rm (1)} The explicit forms of $(\alpha_i)_{R^{\dc}}'\in \Pi^{\mc}_{R^{\dc}}\, (0\leq i\leq l)$ are 
given by
\[(\alpha_i)_{R^{\dc}}'=\begin{cases}
2\alpha_i-k(\alpha_i)a & \text{if $\frac12 (\alpha_i')^\ast\not\in R$ and 
$\frac12 \overline{(\alpha_i')^\ast}\in R/G$},\\
\alpha_i & \text{otherwise}.
\end{cases}\]
{\rm (2)} Let $k_{R^{\dc}}(\alpha)$ be the counting number of 
$\alpha\in R^{\dc}$ as an element of the reduced mERS $(R^{\dc},G)$. One has
\[\begin{aligned}
k_{R^{\dc}}\big((\alpha_i)_{R^{\dc}}\big)&=k(\alpha_i)\quad\text{for every }0\leq i\leq l, \\
k_{R^{\dc}}\big(\big((\alpha_i)_{R^{\dc}}'\big)^\ast\big)&=\begin{cases}
2k(\alpha_i) & \text{if $\frac12 (\alpha_i')^\ast\not\in R$ and 
$\frac12 \overline{(\alpha_i')^\ast}\in R/G$},\\
k(\alpha_i) & \text{otherwise}.
\end{cases}
\end{aligned}\]
{\rm (3)} The explicit forms of $\big((\alpha_i)_{R^{\dc}}'\big)^\ast\in 
\big(\Pi^{\mc}_{R^{\dc}}\big)^\ast\, (0\leq i\leq l)$ are given by
\[\big((\alpha_i)_{R^{\dc}}'\big)^\ast=\begin{cases}
\frac12 (\alpha_i')^\ast & \text{if $\frac12 (\alpha_i')^\ast\in R$},\\
(\alpha_i')^\ast & \text{otherwise},
\end{cases} \]
{\rm (4)} For $0\leq i\leq l$, let $m_{(\alpha_i)_{R^{\dc}}}$ be the exponent of 
$(\alpha_i)_{R^{\dc}}$. One has
\[m_{(\alpha_i)_{R^{\dc}}}=m_{\alpha_i}\quad \text{for every }0\leq i\leq l.\]
As a by-product, it follows that the correspondence
\[\begin{array}{lllllllll}
\alpha_i & \longmapsto &  (\alpha_i)_{R^{\dc}}\quad &\text{for }\alpha_i\in|\Gamma^{\downarrow}(R,G)|,
\\[5pt]
(\alpha_j')^\ast & \longmapsto & \big((\alpha_j)_{R^{\dc}}'\big)^\ast \quad &\text{for }
(\alpha_j')^\ast\in |\Gamma_m^{\uparrow}(R,G)^\ast|
\end{array}\]
gives a bijection from $|\Gamma(R,G)|$ to $|\Gamma(R^{\dc},G)|$.
\vskip 1mm
\noindent
{\rm (5)} The color of a node in $|\Gamma(R^{\dc},G)|$ is determined by the following 
rules. 
\begin{itemize}
\item For a white node in $|\Gamma(R,G)|$, the corresponding node in 
$|\Gamma(R^{\dc},G)\vert$ is a white node also.
\item If $\alpha_i\in |\Gamma_m^\nw|_\sharp^g\,\, (\sharp=g \text{ or }b)$, both 
$(\alpha_i)_{R^{\dc}}$ and $\big((\alpha_i)_{R^{\dc}}'\big)^\ast$ are gray nodes.
\item If $\alpha_i\in |\Gamma_m^\nw|_\sharp^b\,\, (\sharp=g \text{ or }b)$, both 
$(\alpha_i)_{R^{\dc}}$ and $\big((\alpha_i)_{R^{\dc}}'\big)^\ast$ are white nodes. 
\item If $\alpha_i\in |(\Gamma^\downarrow
\setminus\Gamma_m^\downarrow)^\nw|$,  
$(\alpha_i)_{R^{\dc}}$ is a white node. 
\end{itemize}
\end{prop}

After some preparations, we will prove this proposition at the end of 
$\S$ \ref{sect_R'}. 
Let us start with the next lemma. 

\begin{lemma}\label{lemma:black-ast}
For $\alpha_i\in |\Gamma^\downarrow|$, the following two conditions are equivalent.
\begin{itemize}
\item[(a)] $\frac12 (\alpha_i')^\ast\in R$.
\vskip 1mm
\item[(b)] $\alpha_i\in |\Gamma^\nw_m|_g^b\amalg |\Gamma^\nw_m|_b^b\amalg 
|(\Gamma^{\downarrow}\setminus \Gamma^{\downarrow}_m)^\nw|$.
\end{itemize}
\end{lemma}
\begin{proof}
If $\alpha_i\in |\Gamma_m^\downarrow|$, the condition (a) is equivalent to 
$\alpha_i\in |\Gamma^\nw_m|_g^b\amalg |\Gamma^\nw_m|_b^b$. 
Otherwise, $\alpha_i$ belongs to 
$|\Gamma^\downarrow\setminus \Gamma_m^\downarrow|$. 
By \eqref{nr-countings-cor1}, condition (a) 
implies $\overline{\alpha_i'}=2\overline{\alpha_i}\in R/G$. 
Therefore, we have 
$\alpha_i\in |\Gamma^\downarrow\setminus \Gamma_m^\downarrow|\cap 
|\Gamma^\nw|=
|(\Gamma^{\downarrow}\setminus \Gamma^{\downarrow}_m)^\nw|$. Conversely,
assume $\alpha_i\in 
|(\Gamma^{\downarrow}\setminus \Gamma^{\downarrow}_m)^\nw|$. 
By Lemma \ref{lemma:9.5} (2), Proposition \ref{prop:shape-diagram2} (1) (c) and 
(3) (b), we have $\frac12 (\alpha_i')^\ast\in R$. Thus, we completed  the proof.
\end{proof}

\begin{lemma}\label{lemma:color-Gamma-R'}
Assume $\frac12 \overline{(\alpha_i')^\ast}\in R/G$. 
If $\frac12 (\alpha_i')^\ast\not\in R$,
the node $(\alpha_i)_{R^{\dc}}\in |\Gamma^{\downarrow}(R^{\dc},G)|$ is a gray node in $\Gamma^\downarrow(R^{\dc},G)$. 
Otherwise, it is a white node. 
\end{lemma}
\begin{proof}
First, assume $\frac12(\alpha_i')^\ast\not\in R$ and 
$\frac12 \overline{(\alpha_i')^\ast}\in R/G$.
By \eqref{defn:nr-countings_alpha'-ast}, \eqref{nr-countings_alpha3} and 
\eqref{countings-alpha-prime}, we have
\[(\alpha_i')^\ast=2\alpha_i+k(\alpha_i)a.\]
This formula says that $(\alpha_i')^\ast$ belongs to $R^{\dc}$. As a 
by-product, we get 
\[2\overline{(\alpha_i)_{R^{\dc}}}=2\overline{\alpha_i}=\overline{(\alpha_i')^\ast}\in R^{\dc}/G.\]
Therefore, the node $(\alpha_{i})_{R^{\dc}}\in |\Gamma^{\downarrow}(R^{\dc},G)|$ 
is a colored node in $|\Gamma^\downarrow(R^{\dc},G)|$. Since $R^{\dc}$ is reduced, 
it should be a gray node. 

Second, assume $\frac12(\alpha_i')^\ast\in R$. Since 
$(\alpha_i)_{R^{\dc}}=\alpha_i\in R^{\dc}$ and $R^{\dc}$ is reduced, we have
$2(\alpha_i)_{R^{\dc}}=2\alpha_i\not\in R^{\dc}$.
Therefore, $(\alpha_{i})_{R^{\dc}}=\alpha_i$ is a white node
in $|\Gamma^{\downarrow}(R^{\dc},G)|$, and the lemma is obtained.
\end{proof}

\begin{lemma}\label{lemma:data-Gamma-R'}
{\rm (1)} One has $k_{R^{\dc}}(\alpha_i)=k(\alpha_i)$ for every $0\leq i\leq l$.
\vskip 1mm
\noindent
{\rm (2)} For every $0\leq i\leq l$,  
\begin{equation}\label{stared-node-R'}
\big((\alpha_i)_{R^{\dc}}'\big)^\ast=\begin{cases}
\frac12 (\alpha_i')^\ast & \text{if $\frac12 (\alpha_i')^\ast\in R$},\\
(\alpha_i')^\ast & \text{otherwise}.
\end{cases}
\end{equation}
{\rm (3)} One has $m_{(\alpha_i)_{R^{\dc}}}=m_{\alpha_i}$ for every $0\leq i\leq l$. 
\end{lemma}

\begin{proof}
(1) Since $R(\alpha_i)=R^{\dc}(\alpha_i)$ for every $0\leq i\leq l$ (see Proposition 
\ref{prop:description-R'-R''}), we have statement (1), immediately. 
\vskip 1mm
\noindent
(2) By \eqref{defn:nr-countings_alpha'-ast} and Lemma
\ref{lemma:nr-countings1} (1), we have
\begin{equation}\label{nr-countings-R'1}
\big((\alpha_i)_{R^{\dc}}'\big)^\ast
=\big(\big((\alpha_i)_{R^{\dc}}'\big)^\ast:(\alpha_i)_{R^{\dc}}\big)_G\big\{
(\alpha_i)_{R^{\dc}}+k_{R^{\dc}}^\nr \big((\alpha_i)_{R^{\dc}}\big)a\big\}.
\end{equation}
On the other hand, we get 
\[\begin{aligned}
&(\alpha_i)_{R^{\dc}}+k_{R^{\dc}}^\nr \big((\alpha_i)_{R^{\dc}}\big)a \\
=
&\alpha_i+k^\nr(\alpha_i)a\quad (\because\ \eqref{nr-counting-number})\\
=
&\dfrac{1}{((\alpha_i')^\ast:\alpha_i)_G}\, (\alpha_i')^\ast
\quad (\because\ \text{\eqref{defn:nr-countings_alpha'-ast} and Lemma
\ref{lemma:nr-countings1} (1)}).
\end{aligned}\]
Hence, by \eqref{nr-countings-R'1}, we have
\[\big((\alpha_i)_{R^{\dc}}'\big)^\ast= 
\dfrac{\big(\big((\alpha_i)_{R^{\dc}}'\big)^\ast:(\alpha_i)_{R^{\dc}}\big)_G}
{((\alpha_i')^\ast:\alpha_i)_G}\, (\alpha_i')^\ast.\]
By the definition of the two prime maps and Lemma 
\ref{lemma:color-Gamma-R'}, we get  
\begin{align*}
\dfrac{\big(\big((\alpha_i)_{R^{\dc}}'\big)^\ast:(\alpha_i)_{R^{\dc}}\big)_G}
{((\alpha_i')^\ast:\alpha_i)_G}
&=\begin{cases}
\frac12 & \!\!\! \begin{array}{ll}
\text{if $(\alpha_i)_{R^{\dc}}\in |\Gamma^\downarrow (R^{\dc},G)|$ is a white node}\\
\text{\quad and $\alpha_i\in |\Gamma^\downarrow (R,G)|$ is not a white node},
\end{array}
\\
1 & \text{otherwise}
\end{cases}\\
&=\begin{cases}
\frac12 & \text{if $\alpha_i\in |\Gamma^\nw_m|_g^b,\, 
|\Gamma^\nw_m|_b^b$ or 
$|(\Gamma^{\downarrow}\setminus \Gamma^{\downarrow}_m)^\nw|$
},\\
1 & \text{otherwise}.
\end{cases}
\end{align*}
By Lemma \ref{lemma:black-ast}, we have the desired result.
\vskip 1mm
\noindent
(3) Note that $I_{R^{\dc}}\big((\alpha_i)_{R^{\dc}},(\alpha_i)_{R^{\dc}}\big)
=I_R(\alpha_i,\alpha_i)$. By the definition of the exponents, we have
\begin{align*}
m_{(\alpha_i)_{R^{\dc}}}
=
&\dfrac{I_{R^{\dc}}\big((\alpha_i)_{R^{\dc}},(\alpha_i)_{R^{\dc}}\big)}
{2 k_{R^{\dc}}^\nr \big((\alpha_i)_{R^{\dc}}\big)}n_{\alpha_i} \\
=
&\dfrac{I_R(\alpha_i,\alpha_i)}{2 k_{R^{\dc}}^\nr \big((\alpha_i)_{R^{\dc}}\big)}
n_{\alpha_i}
=\dfrac{k_{R^{\dc}}^\nr \big((\alpha_i)_{R^{\dc}}\big)}{k^\nr(\alpha_i)}
m_{\alpha_i}.
\end{align*}
Here, $k_{R^{\dc}}^\nr \big((\alpha_i)_{R^{\dc}}\big)$ is the non-reduced counting number
of $(\alpha_i)_{R^{\dc}}=\alpha_i$. 
Therefore, for proving statement (2), it is enough to show that
\begin{equation}\label{nr-counting-number}
k_{R^{\dc}}^\nr \big((\alpha_i)_{R^{\dc}}\big)=k^\nr(\alpha_i).
\end{equation}

Applying Corollary \ref{cor:nr-countings1} to the mERSs $(R^{\dc},G)$, we get 
\begin{align*}
& k_{R^{\dc}}^\nr \big((\alpha_i)_{R^{\dc}}\big) \\
=
&\begin{cases}
\frac12 k_{R^{\dc}}\big((\alpha_i)_{R^{\dc}}\big)
& \text{if $\frac12\big((\alpha_i)_{R^{\dc}}'\big)^\ast\not\in R^{\dc}$ and 
$\frac12 \overline{\big((\alpha_i)_{R^{\dc}}'\big)^\ast}\in R^{\dc}/G$},\\
k_{R^{\dc}}\big((\alpha_i)_{R^{\dc}}\big)
& \text{otherwise},
\end{cases}\\
=
&\begin{cases}
\frac12 k(\alpha_i)
& \text{if $\frac12\big((\alpha_i)_{R^{\dc}}'\big)^\ast\not\in R^{\dc}$ and 
$\frac12 \overline{\big((\alpha_i)_{R^{\dc}}'\big)^\ast}\in R^{\dc}/G$},\\
k(\alpha_i)
& \text{otherwise}.
\end{cases}
\end{align*}
Note that the second equality follows from statement (1). 
\vskip 3mm
\noindent
{\bf Claim.} {\it For $(\alpha_i)_{R^{\dc}}\in |\Gamma^{\downarrow}(R^{\dc},G)|$,
the following two conditions are equivalent}:
\begin{itemize}
\item[(a)] $\frac12\big((\alpha_i)_{R^{\dc}}'\big)^\ast\not\in R^{\dc}$ {\it and} 
$\frac12 \overline{\big((\alpha_i)_{R^{\dc}}'\big)^\ast}\in R^{\dc}/G$.
\vskip 1mm
\item[(b)] {\it $(\alpha_i)_{R^{\dc}}$ is a gray node.}
\end{itemize}
\begin{proof}
First, since $R^{\dc}$ is reduced, note that the condition 
$\frac12\big((\alpha_i)_{R^{\dc}}'\big)^\ast\not\in R^{\dc}$
is satisfied for every $(\alpha_i)_{R^{\dc}}\in |\Gamma^{\downarrow}(R^{\dc},G)|$.
Second, by 
the equivalence \eqref{nr-countings-cor1} for
$(R^{\dc},G)$, we have
\[\frac12 \overline{\big((\alpha_i)_{R^{\dc}}'\big)^\ast}\in R^{\dc}/G\quad 
\Longleftrightarrow\quad 2\overline{(\alpha_i)_{R^{\dc}}}
\in R^{\dc}/G.\]
The right hand side means that the node 
$(\alpha_i)_{R^{\dc}}\in |\Gamma^{\downarrow}(R^{\dc},G)|$ is not 
a white node. Since $R^{\dc}$ is reduced, it should be a gray node. 
\end{proof}

Lemma \ref{lemma:color-Gamma-R'} asserts that the following 
two conditions
are equivalent:
\begin{itemize}
\item[(i)] $(\alpha_i)_{R^{\dc}}\in |\Gamma^{\downarrow}(R^{\dc},G)|$ is a gray node;
\vskip 1mm
\item[(ii)] $(\alpha_i')^\ast\in |(\Gamma_m^{\nw})^*|$ is a gray node.
\end{itemize}
Since the later means that $\frac12(\alpha_i')^\ast\not\in R$ and 
$\frac12 \overline{(\alpha_i')^\ast}\in R/G$, we get  
\begin{align*}
 k_{R^{\dc}}^\nr \big((\alpha_i)_{R^{\dc}}\big)
&=\begin{cases}
\frac12 k(\alpha_i)
& \text{if $\frac12(\alpha_i')^\ast\not\in R$ and 
$\frac12 \overline{(\alpha_i')^\ast}\in R/G$},\\
k(\alpha_i) & \text{otherwise}.
\end{cases}
\end{align*}
By Corollary \ref{cor:nr-countings1} for the mERS $(R,G)$, one has that 
the right hand side
of the quality above coincides with $k^\nr (\alpha_i)$. 
Thus,  \eqref{nr-counting-number} is obtained, and the proof of statement
(3) is completed. 
\end{proof}

Now, let us prove Proposition \ref{prop:Gamma(R',G)}. 
\vskip 3mm
\noindent
{\it Proof of Proposition \ref{prop:Gamma(R',G)}}. 
The first half of statement (2), statements (3) and (4) are already proved 
in Lemma \ref{lemma:data-Gamma-R'} (1), (2) and (3), respectively.  Statement 
(5) is an immediate consequence of Lemma \ref{lemma:color-Gamma-R'}. 
Applying Proposition \ref{prop:countings-prime} to the reduced mERS $(R^{\dc},G)$, 
the second half of statement (2) is obtained. 
Thus, the remaining is to prove  statement (1). 

By the definition of the
$\ast$-operation, we have
\begin{equation}\label{prime-R'}
(\alpha_i)_{R^{\dc}}'=\big((\alpha_i)_{R^{\dc}}'\big)^\ast-k_{R^{\dc}}
\big(\big((\alpha_i)_{R^{\dc}}'\big)^\ast\big)a.
\end{equation}
The explicit form of the first term of the right hand side is already known in 
 statement (3). 
The explicit form of the second term is also known in 
statement (2). 
Substituting them to \eqref{prime-R'}, statement (1) is obtained
by Corollary \ref{cor:nr-countings1} and Proposition 
\ref{prop:countings-prime}. 
Thus, we completed  the proof of the proposition. \hfill $\square$
\vskip 5mm
\subsection{ Structure of $R^{\mc}$}\label{sect_R''} 
Next, we study the reduced mERS $(R^{\mc},G)$. Define
\[\begin{array}{lll}
|\Gamma^\nw|_b&\!\!\!\!=|\Gamma_m^\nw|_b^g\amalg 
|\Gamma_m^\nw|_b^b\amalg
|(\Gamma^\downarrow\setminus\Gamma_m^\downarrow)^\nw|&
\text{(the set of all black nodes in $|\Gamma^\downarrow|$)},\\[3pt]
|\Gamma^\nw|_g&\!\!\!\!=|\Gamma_m^\nw|_g^g\amalg 
|\Gamma^\nw_m|_g^b&
\text{(the set of all gray nodes in $|\Gamma^\downarrow|$)}.
\end{array}\]
\index[notations]{g916@$\vert\Gamma^\nw\vert_b$}
\index[notations]{g917@$\vert\Gamma^\nw\vert_g$}
Here, we recall that $\Gamma^\downarrow=
\Gamma^{\downarrow}(R,G)$.
Then, the following decomposition of $|\Gamma^\downarrow|$ is obtained:
\begin{equation}\label{decomp-Gamma}
|\Gamma^\downarrow|
=|\Gamma^\w|\amalg |\Gamma^\nw|_b\amalg |\Gamma^\nw|_g.
\end{equation}
Set
\[\begin{array}{lllll}
R^{\mc}(i)&\!\!\!\!=R^{\mc}(\alpha_i) & \text{if $\alpha_i\in|\Gamma^\w|$},\\[3pt]
R^{\mc}(j)&\!\!\!\!=R^{\mc}(\alpha_j)\cup R^{\mc}((\alpha_j')^\ast) & 
\text{if $\alpha_j\in|\Gamma^\nw|=|\Gamma^\nw|_b\amalg |\Gamma^\nw|_g$},
\end{array}\] 
where $R^{\mc}(\alpha_i)$, $R^{\mc}(\alpha_j)$ and 
$R^{\mc}((\alpha_j')^\ast)$ are subsets of $R$ introduced  in Proposition 
\ref{prop:description-R'-R''}.
By \eqref{description-R''}, we have
\begin{equation}\label{eq:description-Rmc}
R^{\mc}=\Big(\bigcup_{\alpha_i\in |\Gamma^\w|}R^{\mc}(i)\Big)\bigcup
\Big(\bigcup_{\alpha_j\in |\Gamma^\nw|_b}R^{\mc}(j)\Big)\bigcup
\Big(\bigcup_{\alpha_j\in |\Gamma^\nw|_g}R^{\mc}(j)\Big).
\end{equation}
Take the image of both sides under the canonical projection $\pi_G:F\to F/G$.
By using the explicit descriptions of $R^{\mc}(i)$ and $R^{\mc}(j)$ in Proposition
\ref{prop:description-R'-R''}, we have
\begin{equation}\label{description-R''/G}
\begin{aligned}
& R^{\mc}/G \\
=
&\Big(\bigcup_{\alpha_i\in |\Gamma^\w|}\overline{R^{\mc}(i)}\Big)\bigcup
\Big(\bigcup_{\alpha_j\in |\Gamma^\nw|_b}\overline{R^{\mc}(j)}\Big)\bigcup
\Big(\bigcup_{\alpha_j\in |\Gamma^\nw|_g}\overline{R^{\mc}(j)}\Big)\\
=&\Big(\bigcup_{\alpha_i\in |\Gamma^\w|}W(R/G)(\overline{\alpha_i})\Big)\bigcup
\Big(\bigcup_{\alpha_j\in |\Gamma^\nw|_b}W(R/G)(2\overline{\alpha_j})\Big)\\
&\quad  \;  \bigcup
\Big(\bigcup_{\alpha_j\in |\Gamma^\nw|_g}
\Big(W(R/G)(\overline{\alpha_j})
\bigcup W(R/G)(2\overline{\alpha_j})\Big)\Big)
\quad (\because\ \text{Lemma \ref{lemma:paired-reduced-basis} (1)}).
\end{aligned}
\end{equation}
Recall the reduced pair $((R^{\mc}/G)^{\dc},(R^{\mc}/G)^{\mc})$ of the affine root system $R^{\mc}/G$. 
The formula \eqref{description-R''/G} implies that
\begin{itemize}
\item the set $\{\overline{\alpha_i}\}_{\alpha_i\in|\Gamma^\w|}\amalg 
\{2\overline{\alpha_j}\}_{\alpha_j\in|\Gamma^\nw|_b}\amalg
\{\overline{\alpha_j}\}_{\alpha_j\in|\Gamma^\nw|_g}$ gives a simple system of 
$(R^{\mc}/G)^{\dc}$;
\item the set $\{\overline{\alpha_i}\}_{\alpha_i\in|\Gamma^\w|}\amalg 
\{2\overline{\alpha_j}\}_{\alpha_j\in|\Gamma^\nw|_b}\amalg
\{2\overline{\alpha_j}\}_{\alpha_j\in|\Gamma^\nw|_g}$ gives a simple system of 
$(R^{\mc}/G)^{\mc}$.
\end{itemize}

Motivated by these facts, we define a linearly independent subset 
$\{(\alpha_i)_{R^{\mc}}\}_{0\leq i\leq l}$ of $R^{\mc}$ by
\begin{equation}\label{eq:defn-(alpha)R''}
(\alpha_i)_{R^{\mc}}=\begin{cases}
2\alpha_i & \text{if $\alpha_i\in |\Gamma^\nw|_b$},\\
\alpha_i & \text{otherwise}. 
\end{cases}
\end{equation}
\index[notations]{a912@$(\alpha_i)_{R^{\mc}}$}
The above facts assert that the image 
$\{\overline{(\alpha_i)_{R^{\mc}}}\}_{0\leq i\leq l}\subset R^{\mc}/G$ 
is a simple system of $R^{\mc}/G$. Similar to the case of $R^{\dc}$, one has 
$R^{\mc} \cap \frac12 (R^{\mc})_l=\emptyset$, because $R^{\mc}$ is reduced. 
Therefore, the set $\{(\alpha_i)_{R^{\mc}}\}_{0\leq i\leq l}$ is a simple system of 
$R^{\mc}$ which satisfies condition \eqref{condition-Pi'}. Define
\[(\alpha_i)_{R^{\mc}}'=\big((\alpha_i)_{R^{\mc}}\big)^{pr}_{R^{\mc}}
\quad \text{for }0\leq i\leq l,\]
where $(\,\cdot\,)^{pr}_{R^{\mc}}:R^{\mc}\to R^{\mc}$ 
is the prime map for $R^{\mc}$. 
\index[index]{prime map!prime map mc@-- for $R^{\mc}$}

Then, the pair
\[\big(\{(\alpha_i)_{R^{\mc}}\}_{0\leq i\leq l},
\{(\alpha_i)_{R^{\mc}}'\}_{0\leq i\leq l}\big)\] 
is a paired simple system of $(R^{\mc},G)$. 
By Proposition \ref{prop:nr-counting-Autom1}, any paired simple 
system $(\Pi^{\dc}_{R^{\mc}},\Pi^{\mc}_{R^{\mc}})$ of $(R^{\mc},G)$
is the image of the pair $\big(\{(\alpha_i)_{R^{\mc}}\}_{0\leq i\leq l}, \\
\{(\alpha_i)_{R^{\mc}}'\}_{0\leq i\leq l}\big)$ under an automorphism of $(R^{\mc},G)$.
Thus, we may assume that 
\[(\Pi^{\dc}_{R^{\mc}},\Pi^{\mc}_{R^{\mc}})
=\big(\{(\alpha_i)_{R^{\mc}}\}_{0\leq i\leq l},
\{(\alpha_i)_{R^{\mc}}'\}_{0\leq i\leq l}\big).\]

Let $a^{R^{\mc}}$ be a generator of the lattice $Q(R^{\mc})\cap G$ of rank $1$. 
We may assume that $a^{R^{\mc}}\in \Z_{>0}a$. 
More precisely, we have
\begin{equation}\label{eq:a-Rmc}
a^{R^{\mc}}=[Q(R)\cap G:Q(R^{\mc})\cap G]a.
\end{equation}
\vskip 5mm

In the following, we determine the elliptic diagram $\Gamma(R^{\mc},G)$ of 
$(R^{\mc},G)$. Let us start with the following lemma.

\begin{lemma}\label{lemma:Gamma(R'',G)}
{\rm (1)} Let $k_{R^{\mc}}(\alpha)$ be the counting number of 
$\alpha\in R^{\mc}$ as an element of the reduced mERS $(R^{\mc},G)$. One has
\begin{equation}\label{eq:counting-Rmc-1}
[Q(R)\cap G:Q(R^{\mc})\cap G]\,
k_{R^{\mc}}\big((\alpha_i)_{R^{\mc}}\big)=\begin{cases}
2k(\alpha_i) & \text{if $\frac12 (\alpha_i')^\ast\in R$},\\
k(\alpha_i) & \text{otherwise},
\end{cases}
\end{equation}
\begin{equation}\label{eq:counting-Rmc-2}
[Q(R)\cap G:Q(R^{\mc})\cap G]\,
k_{R^{\mc}}\big(\big((\alpha_i)_{R^{\mc}}'\big)^\ast\big)=
k(\alpha_i')\quad\text{for every}\quad 0\leq i\leq l, 
\end{equation}
{\rm (2)} The explicit forms of $\big((\alpha_i)_{R^{\mc}}'\big)^\ast\in 
\big(\Pi^{\mc}_{R^{\mc}}\big)^\ast\, (0\leq i\leq l)$ are given by
\[\big((\alpha_i)_{R^{\mc}}'\big)^\ast=
(\alpha_i')^\ast.\]
{\rm (3)} The explicit forms of 
$(\alpha_i)_{R^{\mc}}'\in \Pi^{\mc}_{R^{\mc}}\, (0\leq i\leq l)$ 
are given by
\[(\alpha_i)_{R^{\mc}}'=\alpha_i'.\]
\end{lemma}

\begin{proof}
(1) The formula \eqref{eq:counting-Rmc-1} is an easy consequence of the description 
\eqref{eq:description-Rmc} of $R^{\mc}$. Let us prove \eqref{eq:counting-Rmc-2}. 
Note that, since $(R^{\mc},G)$ is reduced, there is no black node in 
$|\Gamma(R^{\mc},G)|$. 
Applying Proposition \ref{prop:countings-prime} for $(R^{\mc},G)$, 
we have
\begin{equation}\label{eq:couting-Rmc-ast}
\begin{split}
&k_{R^{\mc}}\big(\big((\alpha_i)_{R^{\mc}}'\big)^\ast\big) \\
=
&
\begin{cases}
k_{R^{\mc}}\big((\alpha_i)_{R^{\mc}}\big) & 
\text{if $(\alpha_i)_{R^{\mc}}$ is a white node in $|\Gamma(R^{\mc},G)|$},\\
2k_{R^{\mc}}\big((\alpha_i)_{R^{\mc}}\big) & 
\text{if $(\alpha_i)_{R^{\mc}}$ is a gray node in $|\Gamma(R^{\mc},G)|$}.
\end{cases}
\end{split}
\end{equation}
By the definition of $(\alpha_i)_{R^{\mc}}$, we have
\begin{equation}\label{eq:color-Rmc}
\left\{\begin{array}{lllllll}
\text{$(\alpha_i)_{R^{\mc}}$ is a white node in $\Gamma(R^{\mc},G)$} 
& \Longleftrightarrow & \alpha_i\in |\Gamma^\w|\amalg |\Gamma^\nw|_b,\\[3pt]
\text{$(\alpha_i)_{R^{\mc}}$ is a gray node in $\Gamma(R^{\mc},G)$} 
& \Longleftrightarrow & \alpha_i\in |\Gamma^\nw|_g.
\end{array}\right.
\end{equation}
Therefore, by \eqref{eq:couting-Rmc-ast} and
\eqref{eq:color-Rmc}, we have
\begin{align*}
&[Q(R)\cap G:Q(R^{\mc})\cap G]\,
k_{R^{\mc}}\big(\big((\alpha_i)_{R^{\mc}}'\big)^\ast\big)\\
&\quad =[Q(R)\cap G:Q(R^{\mc})\cap G]\,
k_{R^{\mc}}\big((\alpha_i)_{R^{\mc}}\big)\times
\begin{cases}
1 & \text{if $\alpha_i\in |\Gamma^\w|\amalg |\Gamma^\nw|_b$},\\
2 & \text{if $\alpha_i\in |\Gamma^\nw|_g$}.
\end{cases}
\end{align*}
Furthermore, \eqref{eq:counting-Rmc-1} and Lemma \ref{lemma:black-ast}
imply that
\begin{align*}
&\text{the right hand side} =\begin{cases}
k(\alpha_i) & \text{if } \alpha_i\in |\Gamma^\w|\amalg |\Gamma_m^\nw|_b^g ,\\
2k(\alpha_i) & \text{if } \alpha_i\in |\Gamma_m^\nw|_g^g\amalg |\Gamma_m^\nw|_b^b
\amalg |(\Gamma^{\downarrow}\setminus\Gamma^{\downarrow}_m)^\nw|,\\
4k(\alpha_i) & \text{if }\alpha_i\in|\Gamma_m^\nw|_g^b.
\end{cases}
\end{align*}
By Proposition \ref{prop:countings-prime} for $(R,G)$, this coincides with 
$k((\alpha_i')^\ast)=k(\alpha_i')$. Thus, we have \eqref{eq:counting-Rmc-2}.
\vskip 1mm
\noindent
(2) Note that the formula
\begin{equation}\label{eq:n-red-coungting_in_Rmc}
k_{R^{\mc}}^{nr}\big(\big((\alpha_i)_{R^{\mc}}'\big)^\ast\big)=
k_{R^{\mc}}\big((\alpha_i)_{R^{\mc}}\big)
\end{equation}
holds. Indeed, by Corollary \ref{cor:nr-countings1} for $(R^{\mc},G)$, we have
\begin{align*}
&k_{R^{\mc}}\big(\big((\alpha_i)_{R^{\mc}}'\big)^\ast\big) \\
=
&\begin{cases}
k_{R^{\mc}}^{nr}\big(\big((\alpha_i)_{R^{\mc}}'\big)^\ast\big) & 
\text{if $(\alpha_i)_{R^{\mc}}$ is a white node in $\Gamma(R^{\mc},G)$},\\
2k_{R^{\mc}}^{nr}\big(\big((\alpha_i)_{R^{\mc}}'\big)^\ast\big) &
\text{if $(\alpha_i)_{R^{\mc}}$ is a gray node in $\Gamma(R^{\mc},G)$}.
\end{cases}
\end{align*}
This formula and \eqref{eq:couting-Rmc-ast} imply
\eqref{eq:n-red-coungting_in_Rmc}.

Hence, by \eqref{defn:nr-countings_alpha'-ast} and \eqref{eq:n-red-coungting_in_Rmc},
we get 
\begin{equation}\label{eq:prime-ast-Rmc}
\begin{aligned}
\big((\alpha_i)_{R^{\mc}}'\big)^*=
\big(\big((\alpha_i)_{R^{\mc}}'\big)^*:(\alpha_i)_{R^{mc}}\big)_G(\alpha_i)_{R^{\mc}}
+k_{R^{\mc}}\big((\alpha_i)_{R^{\mc}}\big)a^{R^{\mc}}.
\end{aligned}
\end{equation}
By \eqref{eq:color-Rmc} and the definition \eqref{eq:defn-(alpha)R''} 
of $(\alpha_i)_{R^{\mc}}$, we have
\begin{align*}
\big(\big((\alpha_i)_{R^{\mc}}'\big)^*:(\alpha_i)_{R^{mc}}\big)_G(\alpha_i)_{R^{\mc}}
&=\begin{cases}
(\alpha_i)_{R^{\mc}} & \text{if $(\alpha_i)_{R^{\mc}}$ is a white node},\\
2(\alpha_i)_{R^{\mc}} & \text{if $(\alpha_i)_{R^{\mc}}$ is a gray node},
\end{cases}\\
&=\begin{cases}
\alpha_i & \text{if $\alpha_i\in |\Gamma^\w|$},\\
2\alpha_i & \text{otherwise},
\quad 
\end{cases}
\end{align*}
Furthermore, by \eqref{eq:a-Rmc} and \eqref{eq:counting-Rmc-1}, we get 
\begin{align*}
k_{R^{\mc}}\big((\alpha_i)_{R^{\mc}}\big)a^{R^{\mc}} 
&=[Q(R)\cap G:Q(R^{\mc})\cap G]k_{R^{\mc}}\big((\alpha_i)_{R^{\mc}}\big)a\\
&=\begin{cases}
2k(\alpha_i)a & \text{if $\alpha_i\in 
|\Gamma_m^\nw|_g^b
\amalg |\Gamma_m^\nw|_b^b\amalg 
|(\Gamma^{\downarrow}\setminus 
\Gamma_m^{\downarrow})^\nw|$},\\
k(\alpha_i)a & \text{if $|\Gamma^\w|\amalg 
|\Gamma_m^\nw|_g^g\amalg |\Gamma_m^\nw|_b^g$}.
\end{cases}
\end{align*}
Substituting the above two formulas to \eqref{eq:prime-ast-Rmc}, we have
\begin{align*}
\big((\alpha_i)_{R^{\mc}}'\big)^*
&=
\begin{cases}
\alpha_i+k(\alpha_i)a & \text{if $\alpha_i\in |\Gamma^\w|$},\\
2\alpha_i+k(\alpha_i)a & \text{if $\alpha_i\in 
|\Gamma_m^\nw|_b^g\amalg |\Gamma_m^\nw|_g^g$},\\
2\alpha_i+2k(\alpha_i)a & \text{if $\alpha_i\in 
|\Gamma_m^\nw|_g^b\amalg |\Gamma_m^\nw|_b^b \amalg
|(\Gamma^{\downarrow}\setminus \Gamma_m^{\downarrow})^\nw|$}.\\
\end{cases}
\end{align*}
This coincides with $(\alpha_i')^\ast$, as desired.
\vskip 1mm
\noindent
(3) 
By the definition of the $\ast$-operation and \eqref{k=kast}, we have
\[\alpha_i'=(\alpha_i')^\ast-k((\alpha_i')^\ast)a\quad{and}\quad
(\alpha_i)_{R^{\mc}}'
=\big((\alpha_i)_{R^{\mc}}'\big)^\ast
-k_{R^{\mc}}\big(\big((\alpha_i)_{R^{\mc}}'\big)^\ast\big)a^{R^{\mc}}.\]
Moreover, by statement (2), \eqref{eq:a-Rmc} and 
\eqref{eq:counting-Rmc-2}, we get 
\[\big((\alpha_i)_{R^{\mc}}'\big)^\ast=(\alpha_i')^\ast\quad\text{and}\quad
k_{R^{\mc}}\big(\big((\alpha_i)_{R^{\mc}}'\big)^\ast\big)a^{R^{\mc}}
=k((\alpha_i')^\ast)a.\]
Compiling these four formulas, we have the desired result.
 \end{proof}

For $0\leq i\leq l$, let $m_{(\alpha_i)_{R^{\mc}}}^{R^{\mc}}$ be the exponent 
of $(\alpha_i)_{R^{\mc}}\in\Pi_{R^{\mc}}^{\dc}$, as an element of 
the reduced root system $(R^{\mc},G)$. Its explicit form is given by
\begin{equation}\label{eq:exponent-Rmc}
m_{(\alpha_i)_{R^{\mc}}}^{R^{\mc}}
=\dfrac{I_{R^{^mc}}\big((\alpha_i)_{R^{\mc}},(\alpha_i)_{R^{\mc}}\big)}
{2k_{R^{\mc}}^{nr}\big((\alpha_i)_{R^{\mc}}\big)}n_{(\alpha_i)_{R^{\mc}}}^{R^{\mc}},
\end{equation}
where $\big\{n_{(\alpha_i)_{R^{\mc}}}^{R^{\mc}}\big\}_{0\leq i\leq l}$ is the family
of positive integers such that
\[\delta_b^{R^{\mc}}:=\sum_{0\leq i\leq l} 
n_{(\alpha_i)_{R^{\mc}}}^{R^{\mc}} (\alpha_i)_{R^{\mc}}
\in Q(R^{\mc})\cap \mathrm{rad}(I)\]
is primitive. To determine the shape of the elliptic diagram $\Gamma(R^{\mc},G)$, 
we need to compute the exponents $m_{(\alpha_i)_{R^{\mc}}}^{R^{\mc}}\ (0\leq l\leq l)$
explicitly. \\

Our goal is the following proposition.
\begin{prop}\label{prop:Gamma(R'',G)}
{\rm (1)} One has
\[m^{R^{\mc}}_{(\alpha_i)_{R^{\mc}}}=
\dfrac{[Q(R)\cap \mathrm{rad}(I):Q(R^{\mc})\cap\mathrm{rad}(I)]}{(I_R:I_{R^{\mc}})}\,
m_{\alpha_i}\quad \text{for every }0\leq i\leq l. \]
As a by-product, it follows that the correspondence 
\[\begin{array}{lllllllll}
\alpha_i & \longmapsto &  (\alpha_i)_{R^{\mc}}\quad &\text{for }\alpha_i\in
|\Gamma^{\downarrow}(R,G)|,
\\[5pt]
(\alpha_j')^\ast & \longmapsto & \big((\alpha_j)_{R^{\mc}}'\big)^\ast \quad &\text{for }
(\alpha_j')^\ast\in |\Gamma_m^{\uparrow}(R,G)^\ast|
\end{array}\]
gives a bijection from $|\Gamma(R,G)|$ to $|\Gamma(R^{\mc},G)|$.
\vskip 1mm
\noindent
{\rm (2)} The color of a node in $|\Gamma(R^{\mc},G)|$ is determined by the following 
rules. 
\begin{itemize}
\item For a white node in $|\Gamma(R,G)|$, the corresponding node in 
$|\Gamma(R^{\mc},G)|$ is a white node also.
\item If $\alpha_i\in |\Gamma_m^\nw|_g^\sharp\,\, (\sharp=g \text{ or }b)$, both 
$(\alpha_i)_{R^{\mc}}$ and $\big((\alpha_i)_{R^{\mc}}'\big)^\ast$ are gray nodes.
\item If $\alpha_i\in |\Gamma_m^\nw|_b^\sharp\,\, (\sharp=g \text{ or }b)$, both 
$(\alpha_i)_{R^{\mc}}$ and $\big((\alpha_i)_{R^{\mc}}'\big)^\ast$ are white nodes. 
\item If $\alpha_i\in |(\Gamma^{\downarrow}\setminus 
\Gamma_m^{\downarrow})^\nw|$,  
$(\alpha_i)_{R^{\mc}}$ is a white node. 
\end{itemize}
\end{prop}
 
Before proving the proposition, we need some preparations. Take the image 
$\overline{\delta_b^{R^{\mc}}}=\pi_G(\delta_b^{R^{\mc}})$
of $\delta_b^{R^{\mc}}$ under the canonical projection $\pi_G:F\longrightarrow F/G$:
\[\overline{\delta_b^{R^{\mc}}}=\sum_{0\leq i\leq l}n_{(\alpha_i)_{R^{\mc}}}^{R^{\mc}} 
\overline{(\alpha_i)_{R^{\mc}}}\in Q(R^{\mc}/G)\cap \mathrm{rad}(I_{F/G}).\]
Recall that the element $\delta_b=\sum_{0\leq i\leq l}n_{\alpha_i} \alpha_i
\in Q(\W[\Pi^{\dc}].\Pi^{\dc})\cap \mathrm{rad}(I)$ introduced in 
\eqref{defn:delta_b},  and its image 
$\overline{\delta_b}=\pi_G(\delta_b)$ of $\delta_b$ under the 
canonical projection $\pi_G:$\footnote{In $\S$\ref{sect:radZ}, 
$\overline{\delta_b}=\pi_G(\delta_b)$ is denoted by $\delta$.}
\[\overline{\delta_b}=\sum_{0\leq i\leq l}n_{\alpha_i} \overline{\alpha_i}
\in Q(R/G)\cap \mathrm{rad}(I_{F/G}).\]
Since $Q(R^{\mc}/G)\cap \mathrm{rad}(I_{F/G})$ is a 
sublattice of the lattice $Q(R^{\mc}/G)\cap \mathrm{rad}(I_{F/G})$ of rank $1$, 
there exists a positive integer $(\delta_b^{R^{\mc}}:\delta_b)_G$
such that 
\begin{equation}\label{eq:delta-Rmc}
\delta_b^{R^{\mc}}\equiv \big(\delta_b^{R^{\mc}}:\delta_b\big)_G\, \delta_b\
(\mathrm{mod}\, G).
\end{equation}
\vskip 3mm
\begin{lemma}
For every $0\leq i\leq l$, we have
\begin{equation}\label{eq:Kac-label-Rmc}
n_{(\alpha_i)_{R^{\mc}}}^{R^{\mc}}=\dfrac{\big(\delta_b^{R^{\mc}}:\delta_b\big)_G}
{\big((\alpha_i)_{R^{\mc}}:\alpha_i\big)_G}\, n_{\alpha_i}.
\end{equation}
\end{lemma}
\begin{proof}
Since $\alpha_i'\equiv (\alpha_i':\alpha_i)_G\alpha_i \, (\mathrm{mod}\, G)$ 
for every $0\leq i\leq l$, we have
\begin{align*}
\delta_b&\equiv \sum_{0\leq i\leq l}\dfrac{n_{\alpha_i}}{(\alpha_i':\alpha_i)_G}
\alpha_i'\ (\mathrm{mod}\, G).
\end{align*}
Similarly, we have
\begin{align*}
\delta_b^{R^{\mc}}&\equiv 
\sum_{0\leq i\leq l}\dfrac{n_{(\alpha_i)_{R^{\mc}}}^{R^{\mc}}}
{\big((\alpha_i)_{R^{\mc}}':(\alpha_i)_{R^{\mc}}\big)_G}(\alpha_i)'_{R^{\mc}} \ 
(\mathrm{mod}\, G)\\
&=\sum_{0\leq i\leq l}\dfrac{n_{(\alpha_i)_{R^{\mc}}}^{R^{\mc}}}
{\big(\alpha_i':(\alpha_i)_{R^{\mc}}\big)_G}\alpha_i'\quad 
(\because\, \text{Lemma \ref{lemma:Gamma(R'',G)} (3)}).
\end{align*}
These two equalities and \eqref{eq:delta-Rmc} imply 
\begin{equation*}
n_{(\alpha_i)_{R^{\mc}}}^{R^{\mc}}=\big(\delta_b^{R^{\mc}}:\delta_b\big)_G
\dfrac{\big(\alpha_i':(\alpha_i)_{R^{\mc}}\big)_G}{(\alpha_i':\alpha_i)_G}\, n_{\alpha_i}
=\dfrac{\big(\delta_b^{R^{\mc}}:\delta_b\big)_G}
{\big((\alpha_i)_{R^{\mc}}:\alpha_i\big)_G}\, n_{\alpha_i},
\end{equation*}
as desired.
\end{proof}
\vskip 3mm
\noindent
\begin{lemma}
For every $0\leq i\leq l$, we have
\begin{equation}\label{eq:nr-counging-Rmc}
\dfrac{k_{R^{\mc}}^{nr}\big((\alpha_i)_{R^{\mc}}\big)}
{\big((\alpha_i)_{R^{\mc}}:\alpha_i\big)_G}
=\dfrac{k^{nr}(\alpha_i)}{[Q(R)\cap G:Q(R^{\mc})\cap G]}.
\end{equation}
\end{lemma}
\begin{proof}
By Lemma \ref{lemma:nr-countings1} (1) for 
the mERS $(R^{\mc},G)$ and \eqref{eq:n-red-coungting_in_Rmc}, we have
\begin{align*}
k_{R^{\mc}}^{nr}\big((\alpha_i)_{R^{\mc}}\big)&=
\dfrac{k_{R^{\mc}}^{nr}\big(\big((\alpha_i)_{R^{\mc}}'\big)^\ast\big)}
{\big(\big((\alpha_i)_{R^{\mc}}'\big)^\ast:(\alpha_i)_{R^{\mc}}\big)_G}
=\dfrac{k_{R^{\mc}}\big((\alpha_i)_{R^{\mc}}\big)}
{\big(\big((\alpha_i)_{R^{\mc}}'\big)^\ast:(\alpha_i)_{R^{\mc}}\big)_G}\\
&=\dfrac{k_{R^{\mc}}\big((\alpha_i)_{R^{\mc}}\big)}
{\big((\alpha_i')^\ast:(\alpha_i)_{R^{\mc}}\big)_G} \quad
(\because \ \text{Lemma \ref{lemma:Gamma(R'',G)} (2)}).
\end{align*}
Note that
\begin{align*}
\big((\alpha_i')^\ast:(\alpha_i)_{R^{\mc}}\big)_G=
&
\big((\alpha_i')^\ast:\alpha_i\big)_G
\big(\alpha_i:(\alpha_i)_{R^{\mc}}\big)_G, \\
\big((\alpha_i')^\ast : \alpha_i\big)_G
=
&\begin{cases}
1 & \text{if }\alpha_i\in |\Gamma^\w|,\\
2 & \text{otherwise}.
\end{cases}
\end{align*}
Therefore, by \eqref{eq:counting-Rmc-1}, we have
\begin{align*}
&\dfrac{k_{R^{\mc}}^{nr}\big((\alpha_i)_{R^{\mc}}\big)}
{\big((\alpha_i)_{R^{\mc}}:\alpha_i\big)_G}
=
\dfrac{k_{R^{\mc}}\big((\alpha_i)_{R^{\mc}}\big)}{\big((\alpha_i')^\ast,\alpha_i\big)_G}\\
=&\dfrac{k(\alpha_i)}{[Q(R)\cap G:Q(R^{\mc})\cap G]}\times
\begin{cases}
1&
\text{if $\alpha_i\in$} \begin{array}{l} |\Gamma^\w|\amalg
|\Gamma_m^\nw|_g^b\amalg |\Gamma_m^\nw|_b^b \\
\amalg |(\Gamma^\downarrow\setminus \Gamma_m^\downarrow)^\nw|,
\end{array}\\[3pt]
1/2 & 
\text{if $\alpha_i\in |\Gamma_m^\nw|_g^g \amalg |\Gamma_m^\nw|_b^g,$}
\end{cases}\\
=&\dfrac{k^{nr}(\alpha_i)}{[Q(R)\cap G:Q(R^{\mc})\cap G]}
\quad (\because\ \text{Proposition \ref{prop:nr-countings1}}),
\end{align*}
as desired. 
\end{proof}
\vskip 3mm

Now, let us prove Proposition \ref{prop:Gamma(R'',G)}.
\vskip 3mm
\noindent
{\it Proof of Proposition \ref{prop:Gamma(R'',G)}}. 
(1) One has
\begin{align*}
I_{R^{^mc}}\big((\alpha_i)_{R^{\mc}},(\alpha_i)_{R^{\mc}}\big)
&=\big((\alpha_i)_{R^{\mc}}:\alpha_i\big)_G^2I_{R^{^mc}}(\alpha_i,\alpha_i)\\
&=\big((\alpha_i)_{R^{\mc}}:\alpha_i\big)_G^2(I_{R^{\mc}}:I_R)I_R(\alpha_i,\alpha_i).
\end{align*}
Substitute this formula and \eqref{eq:Kac-label-Rmc}
to the right hand side of \eqref{eq:exponent-Rmc}, we have
\begin{align*}
&m_{(\alpha_i)_{R^{\mc}}}^{R^{\mc}} \\
=
&\dfrac{\big((\alpha_i)_{R^{\mc}}:\alpha_i\big)^2 \cdot I_R(\alpha_i,\alpha_i)}
{(I_R:I_{R^{\mc}})}
\times 
\dfrac{1}{2k_{R^{\mc}}^{nr}\big((\alpha_i)_{R^{\mc}}\big)}
\times
\dfrac{\big(\delta_b^{R^{\mc}}:\delta_b\big)_G}
{\big((\alpha_i)_{R^{\mc}}:\alpha_i\big)_G}\, n_{\alpha_i}\\
=
&\dfrac{[Q(R)\cap G:Q(R^{\mc})\cap G]\cdot\big(\delta_b^{R^{\mc}}:\delta_b\big)_G}
{(I_R:I_{R^{\mc}})}\cdot
\dfrac{I_R(\alpha_i,\alpha_i)}
{2k^{nr}(\alpha_i)}\,n_{\alpha_i}\quad (\because \ \eqref{eq:nr-counging-Rmc})\\
=
&\dfrac{[Q(R)\cap G:Q(R^{\mc})\cap G]\cdot\big(\delta_b^{R^{\mc}}:\delta_b\big)_G}
{(I_R:I_{R^{\mc}})}\, m_{\alpha_i}.
\end{align*}
Since
\[ 
[Q(R)\cap G:Q(R^{\mc})\cap G]\cdot\big(\delta_b^{R^{\mc}}:\delta_b\big)_G
=[Q(R)\cap \mathrm{rad}(I):Q(R^{\mc})\cap \mathrm{rad}(I)], 
\] 
we have statement (1). 
\vskip 1mm
\noindent 
(2) We already know that 
\begin{itemize}
\item[(i)] the color of $(\alpha_i)_{R^{\mc}}$ in 
$|\Gamma(R^{\mc},G)|$ is determined according to the rule in  statement (2);
\item[(ii)] the diagram $\Gamma(R^{\mc},G)$ has the same shape as $\Gamma(R,G)$ 
(by (1));
\item[(iii)] there is no black node in $|\Gamma(R^{\mc},G)|$. 
\end{itemize}
Under these conditions, the diagram $\Gamma(R^{\mc},G)$ is uniquely determined, 
with the rule in  statement (2). 
\hfill $\square$

\vskip 5mm
\subsection{Diagrammatic interpretation}\label{sect:ans-Problem2}
The results of Propositions  \ref{prop:Gamma(R',G)} and  \ref{prop:Gamma(R'',G)} are translated in diagrammatic language as
follows.
\begin{prop}\label{prop:red-pair} 
{\rm (1)} Both of the elliptic diagrams $\Gamma(R^{\dc},G)$ and 
$\Gamma(R^{\mc},G)$ have the same shape as $\Gamma(R,G)$.
\vskip 1mm
\noindent
{\rm (2)} The colors of nodes of them are given as follows{\rm :} 
transform the black nodes and the gray nodes by the following $4$ rules.
\begin{enumerate}
\item[(D0)] If a black node 
\begin{tikzpicture} 
\fill (0,0) circle(0.1); \draw (0,0.1) node[above] {\scriptsize $\alpha_i$};
\end{tikzpicture}
is only connected to white nodes, 
replace it 
\begin{enumerate}
\item[i)] with a white node 
\begin{tikzpicture} 
\draw (0,0) circle(0.1); \draw (0,0.1) node[above] {\scriptsize $\alpha_i$};
\end{tikzpicture}
 for $R^{\dc}$ and
\item[ii)] with a white node 
\begin{tikzpicture} 
\draw (0,0) circle(0.1); \draw (0,0.1) node[above] {\scriptsize $2\alpha_i$};
\end{tikzpicture}
for $R^{\mc}$.
\end{enumerate}
\vskip 1mm
\item[(D1)] Replace the unit
\begin{tikzpicture} 
\fill (0,0) circle(0.1); \draw (0,0.15) node[above] {\scriptsize $\alpha_i$};
\draw[dashed, double distance=2] (0.1,0) -- (0.9,0); 
{\color{red}\draw (0.45,0) -- (0.58,0.13); \draw (0.45,0) -- (0.58,-0.13);}
\fill (1,0) circle(0.1); \draw (1.1,0.1) node[above] {\scriptsize $(\alpha_i')^\ast$};
\end{tikzpicture} 
\begin{enumerate}
\item[i)] with 
\begin{tikzpicture} 
\draw (0,0) circle(0.1); \draw (0,0.15) node[above] {\scriptsize $\alpha_i$};
\draw[dashed, double distance=2] (0.1,0) -- (0.9,0); 
\draw (1,0) circle(0.1); 
\draw (1.05,0.05) node[above] {\scriptsize $\frac12 (\alpha_i')^\ast$};
\end{tikzpicture}
for $R^{\dc}$ and
\item[ii)] with
\begin{tikzpicture} 
\draw (0,0) circle(0.1); \draw (0,0.15) node[above] {\scriptsize $2\alpha_i$};
\draw[dashed, double distance=2] (0.1,0) -- (0.9,0); 
\draw (1,0) circle(0.1); \draw (1.1,0.1) node[above] {\scriptsize $(\alpha_i')^\ast$};
\end{tikzpicture}
for $R^{\mc}$.
\end{enumerate}
\vskip 1mm
\item[(D2)] Replace the unit 
\begin{tikzpicture} 
\fill (0,0) circle(0.1); \draw (0,0.15) node[above] {\scriptsize $\alpha_i$};
\draw[dashed, double distance=2] (0.1,0) -- (0.9,0); 
{\color{red}\draw (0.45,0) -- (0.58,0.13); \draw (0.45,0) -- (0.58,-0.13);}
\draw [black,fill=gray!70] (1,0) circle(0.1); 
\draw (1.1,0.1) node[above] {\scriptsize $(\alpha_i')^\ast$};
\end{tikzpicture} 
\begin{enumerate}
\item[i)] with 
\begin{tikzpicture} 
\draw [black,fill=gray!70]  (0,0) circle(0.1); 
\draw (0,0.15) node[above] {\scriptsize $\alpha_i$};
\draw[dashed, double distance=2] (0.1,0) -- (0.9,0); 
{\color{red}\draw (0.45,0) -- (0.58,0.13); \draw (0.45,0) -- (0.58,-0.13);}
\draw [black,fill=gray!70] (1,0) circle(0.1); 
\draw (1.1,0.1) node[above] {\scriptsize $(\alpha_i')^\ast$};
\end{tikzpicture}
for $R^{\dc}$ and 
\item[ii)] with
\begin{tikzpicture} 
\draw (0,0) circle(0.1); \draw (0,0.15) node[above] {\scriptsize $2\alpha_i$};
\draw[dashed, double distance=2] (0.1,0) -- (0.9,0); 
\draw (1,0) circle(0.1); \draw (1.1,0.1) node[above] {\scriptsize $(\alpha_i')^\ast$};
\end{tikzpicture}
for $R^{\mc}$.
\end{enumerate}
\vskip 1mm
\item[(D3)] Replace the unit 
\begin{tikzpicture} 
\draw [black,fill=gray!70]  (0,0) circle(0.1); 
\draw (0,0.15) node[above] {\scriptsize $\alpha_i$};
\draw[dashed, double distance=2] (0.1,0) -- (0.9,0); 
{\color{red}\draw (0.45,0) -- (0.58,0.13); \draw (0.45,0) -- (0.58,-0.13);}
\fill (1,0) circle(0.1); \draw (1.1,0.1) node[above] {\scriptsize $(\alpha_i')^\ast$};
\end{tikzpicture} 
\begin{enumerate}
\item[i)] with
\begin{tikzpicture} 
\draw (0,0) circle(0.1); \draw (0,0.15) node[above] {\scriptsize $\alpha_i$};
\draw[dashed, double distance=2] (0.1,0) -- (0.9,0); 
\draw (1,0) circle(0.1); \draw (1.1,0.05) node[above] 
{\scriptsize $\frac12 (\alpha_i')^\ast$};
\end{tikzpicture}
 for $R^{\dc}$ and 
\item[ii)] with 
\begin{tikzpicture} 
\draw [black,fill=gray!70]  (0,0) circle(0.1); 
\draw (0,0.15) node[above] {\scriptsize $\alpha_i$};
\draw[dashed, double distance=2] (0.1,0) -- (0.9,0); 
{\color{red}\draw (0.45,0) -- (0.58,0.13); \draw (0.45,0) -- (0.58,-0.13);}
\draw [black,fill=gray!70] (1,0) circle(0.1); 
\draw (1.1,0.1) node[above] {\scriptsize $(\alpha_i')^\ast$};
\end{tikzpicture}
for $R^{\mc}$.
\end{enumerate}
\end{enumerate}
Note that the newly obtained white nodes represent short roots for 
$R^{\dc}$ and long roots for $R^{\mc}$.
\end{prop}


\part{Summary of root data}
\section{Root data of mERSs with non-reduced affine quotient}\label{chapter:root-data}
Here, we collect several important root data for non-isomorphic mERS $(R,G)$ with non-reduced quotient $R/G$, belonging to 
$(F,I_F)$ where $I_F$ is a symmetric bilinear form with signature $(\overset{+}{l},\overset{0}{2},\overset{-}{0})$. 
Here, we fix a basis $(\vep_1, \vep_2,\cdots,\vep_l,b,a)$ of $F$ as in \eqref{def_base-F}.
\vskip 5mm
\begin{enumerate}
\item The tier numbers $t_1(R,G)$ and $t_2(R,G)$ and dual root system $(R^\vee,G)$,
\item Real Root system,
\item A paired simple system $(\Pi^{\dc}, \Pi^{\mc})$, 
\begin{enumerate} 
\item[i)] roots in $\Pi^{\dc}$,
\item[ii)] roots in $\Pi^{\mc} \setminus \Pi^{\dc}$.
\end{enumerate}
\item (Non-reduced) counting numbers: 
\begin{enumerate} 
\item[i)] $k(\alpha)$ for $\alpha \in \Pi^{\dc}$,
\item[ii)] $k^{\nr}((\alpha')^\ast)$ for $\alpha' \in \Pi^{\mc} \setminus \Pi^{\dc}$. 
\end{enumerate}
\item Exponents 
\item Elliptic diagram.
\item $W$-orbits on $R$.
\end{enumerate}
\vskip 0.2in 

To make the elliptic diagrams clearer, we use light blue edge to indicate 
\begin{tikzpicture} \draw [fill=gray!50] (0,0) circle(0.1); 
                                     \draw (0.1,0) -- (0.9,0); \draw (0.5,0.03) node[above] {\tiny{$\infty$}};
                                     \draw [fill=gray!50] (1,0) circle(0.1);
\end{tikzpicture} 
for some $\node{50} \in \{ \white, \gray, \black\}$
(cf. (Edge3) in Definition \ref{defn_Dynkin-diag}), dark blue edge to indicate
\begin{tikzpicture} \draw [fill=gray!50] (0,0) circle(0.1); 
                                     \draw (0.1,0) -- (0.9,0);  \draw (0.5,0.05) node[above] {\tiny{$2$}};
                                     \draw (0.55,0.1) -- (0.45,0); \draw (0.55,-0.1) -- (0.45,0); 
                                     \draw [fill=gray!50] (1,0) circle(0.1); \end{tikzpicture} 
for some $\node{50} \in \{ \white, \gray, \black\}$
(cf. $t=2$ in (Edge2) in Definition \ref{defn_Dynkin-diag})
red edge to indicate 
\begin{tikzpicture} \draw [fill=gray!50] (0,0) circle(0.1); 
                                     \draw (0.1,0) -- (0.9,0);  \draw (0.5,0.05) node[above] {\tiny{$4$}};
                                     \draw (0.55,0.1) -- (0.45,0); \draw (0.55,-0.1) -- (0.45,0); 
                                     \draw [fill=gray!50] (1,0) circle(0.1); \end{tikzpicture} 
for some $\node{50} \in \{ \white, \gray, \black\}$
(cf. $t=4$ in (Edge2) in Definition \ref{defn_Dynkin-diag}).


\subsection{Reduced mERSs}\label{sect_root-data-red}




\subsubsection{Type $BC_l^{(1,2)}\; (l\geq 1)$}\label{data:BC-1}
\begin{enumerate}
\item $t_1(BC_l^{(1,2)})=1$, \; $t_2(BC_l^{(1,2)})=2$ and 
$(BC_l^{(1,2)})^\vee=BC_l^{(4,2)}$,
\item $R(BC_l^{(1,2)})=(R(BC_l)_s+\Z a+\Z b) \cup (R(BC_l)_m+\Z a+\Z b)$ \\
\phantom{$R(BC_l^{(1,2)})=$} $\cup (R(BC_l)_l+(1+2\Z)a+\Z b)$. 
\item 
\begin{enumerate}
\item[i)] $\alpha_0=a+b-2\vep_1, \, \alpha_i=\vep_i-\vep_{i+1}\, (1\leq i\leq l-1), \, \alpha_l=\vep_l$, 
\item[ii)] $\alpha_l'= -a+2\vep_l$.
\end{enumerate}
\item
\begin{enumerate}
\item[i)] $k(\alpha_0)=2,\, k(\alpha_i)=1\, (1\leq i\leq l)$, 
\item[ii)] $k^{\nr}((\alpha_l')^\ast)=1$. 
\end{enumerate}
\item $m_{\alpha_0}=2, \; m_{\alpha_i}=4\, (1\leq i\leq l)$.
\item Elliptic diagram:
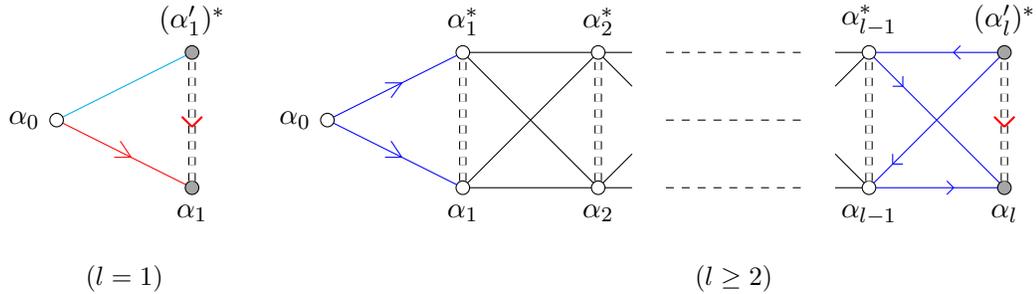
\input{Figure-new/diagram-BC12-red}
\item $W$-orbits: $R(BC_l^{(1,2)})_l=W(\alpha_0) \amalg W((\alpha_l')^\ast)$, \\
\phantom{$W$-orbits:} $R(BC_l^{(1,2)})_m$: single $W$-orbit, \quad $R(BC_l^{(1,2)})_s=W(\alpha_l)$.
\end{enumerate}

\subsubsection{Type $BC_l^{(1,1)\ast}\; (l\geq 1)$}\label{data:BC-2}
\begin{enumerate}
\item $t_1(BC_l^{(1,1)\ast})=1$, \; $t_2(BC_l^{(1,1)\ast})=1$ and 
$(BC_l^{(1,1)\ast})^\vee=BC_l^{(4,4)\ast}$,
\item $R(BC_l^{(1,1)\ast})=(R(BC_l)_s+\Z a+\Z b) \cup (R(BC_l)_m+\Z a+\Z b)$ \\
\phantom{$R(BC_l^{(1,1)\ast})=$} $\cup (R(BC_l)_l+\{\, ma+nb\,\vert\, (m-1)(n-1) \equiv 0\, [2]\, \})$.
\item 
\begin{enumerate}
\item[i)] $\alpha_0=b-2\vep_1, \, \alpha_i=\vep_i-\vep_{i+1}\, 
(1\leq i\leq l-1), \, \alpha_l=\vep_l$,
\item[ii)] $\alpha_l'= -a+2\vep_l$.
\end{enumerate}
\item
\begin{enumerate}
\item[i)] $k(\alpha_i)=1\, (0\leq i\leq l)$, 
\item[ii)]$k^{\nr}((\alpha_l')^\ast)=1$. 
\end{enumerate}
\item $m_{\alpha_i}=4\, (0\leq i\leq l)$.
\item Elliptic diagram:
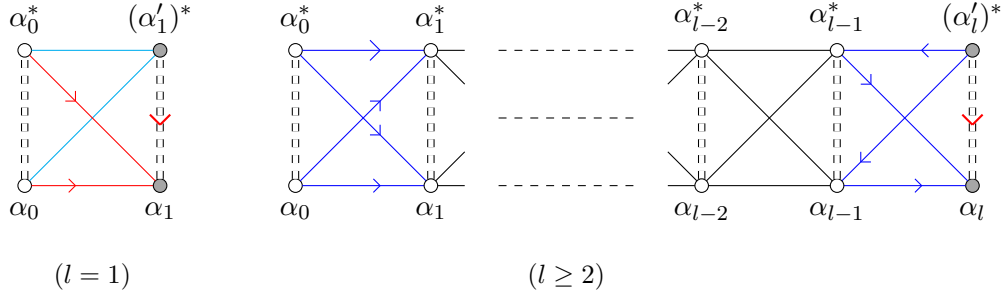
\input{Figure-new/diagram-BC11st-red}
\item $W$-orbits: $R(BC_l^{(1,1)\ast})_l=W(\alpha_0) \amalg W(\alpha_0^\ast) \amalg W((\alpha_l')^\ast)$, \\
\phantom{$W$-orbits:} $R(BC_l^{(1,1)\ast})_m$: single $W$-orbit, \quad $R(BC_l^{(1,1)\ast})_s=W(\alpha_l)$.
\end{enumerate}

\subsubsection{Type $BC_l^{(4,2)}\; (l\geq 1)$}\label{data:BC-3}
\begin{enumerate}
\item $t_1(BC_l^{(4,2)})=4$, \; $t_2(BC_l^{(4,2)})=2$ and  
$(BC_l^{(4,2)})^\vee=BC_l^{(1,2)}$,
\item $R(BC_l^{(4,2)})=(R(BC_l)_s+\Z a+\Z b) \cup (R(BC_l)_m+\Z a+2\Z b)$ \\
\phantom{$R(BC_l^{(4,2)})=$} $\cup (R(BC_l)_l+(1+2\Z)a+4\Z b)$.
\item 
\begin{enumerate}
\item[i)] $\alpha_0=b-\vep_1, \, \alpha_i=\vep_i-\vep_{i+1}\, 
(1\leq i\leq l-1), \, \alpha_l=\vep_l$, 
\item[ii)] $\alpha_l'= -a+2\vep_l$.
\end{enumerate}
\item
\begin{enumerate}
\item[i)] $k(\alpha_i)=1\, (0\leq i\leq l)$, 
\item[ii)] $k^{\nr}((\alpha_l')^\ast)=1$. 
\end{enumerate}
\item $m_{\alpha_0}=1, \; m_{\alpha_i}=2\, (1\leq i\leq l)$.
\item Elliptic diagram:
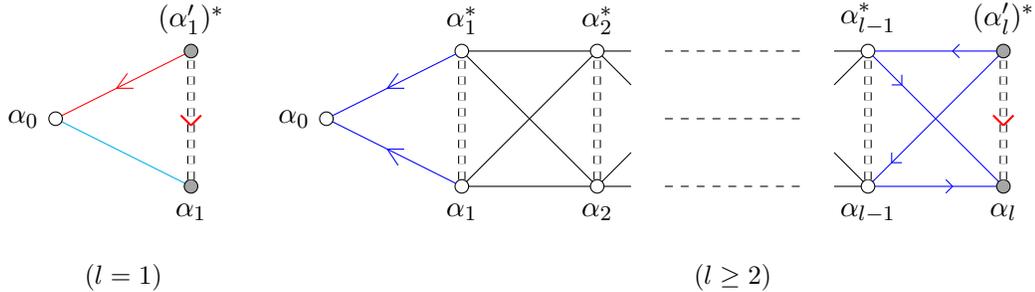
\input{Figure-new/diagram-BC42-red}
\item $W$-orbits:  $R(BC_l^{(4,2)})_l=W((\alpha_l')^\ast)$, \\ 
\phantom{$W$-orbits:} $R(BC_l^{(4,2)})_m$: single $W$-orbit,  \\
\phantom{$W$-orbits:} $R(BC_l^{(4,2)})_s=W(\alpha_0) \amalg W(\alpha_l)$.
\end{enumerate}

\subsubsection{Type $BC_l^{(4,4)\ast}\; (l\geq 1)$}\label{data:BC-4}
\begin{enumerate}
\item $t_1(BC_l^{(4,4)\ast})=4$, \; $t_2(BC_l^{(4,4)\ast})=4$ and  
$(BC_l^{(4,4)\ast})^\vee=BC_l^{(1,1)\ast}$,
\item $R(BC_l^{(4,4)\ast})=(R(BC_l)_s+\{\, ma+nb\, \vert\, (m-1)(n-1)\equiv 0\, [2]\, \})$ \\
\phantom{$R(BC_l^{(4,4)\ast})=$} $\cup (R(BC_l)_m+2\Z a+2\Z b) \cup (R(BC_l)_l+4\Z a+4\Z b)$. 
\item 
\begin{enumerate}
\item[i)] $\alpha_0=b-\vep_1, \, \alpha_i=\vep_i-\vep_{i+1} 
(1\leq i\leq l-1), \, \alpha_l=a+\vep_l$,
\item[ii)] $\alpha_l'=2\vep_l$.
\end{enumerate}
\item 
\begin{enumerate}
\item[i)] $k(\alpha_0)=1, \; k(\alpha_i)=2\,  (1\leq i\leq l)$, 
\item[ii)] $k^{\nr}((\alpha_l')^\ast)=2$. 
\end{enumerate}
\item $m_{\alpha_i}=1\, (0\leq i\leq l)$.
\item Elliptic diagram:
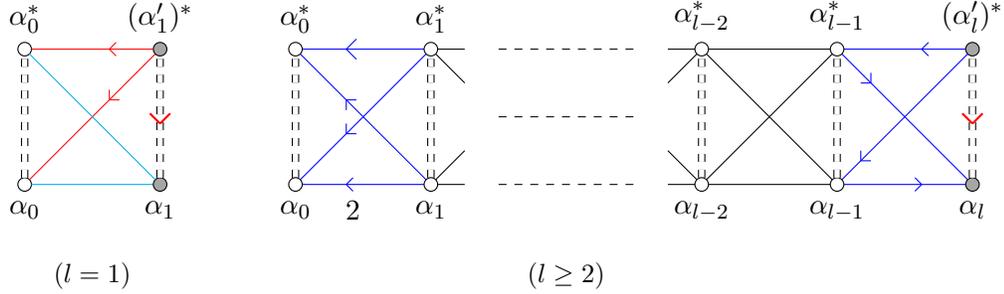
\input{Figure-new/diagram-BC44st-red}
\item $W$-orbits: $R(BC_l^{(4,4)\ast})_l=W((\alpha_l')^\ast)$, \\
\phantom{$W$-orbits:}  $R(BC_l^{(4,4)\ast})_m$: single $W$-orbit, \\
\phantom{$W$-orbits:} 
$R(BC_l^{(4,4)\ast})_s=W(\alpha_0) \amalg W(\alpha_0^\ast) \amalg W(\alpha_l)$.
\end{enumerate}

\subsubsection{Type $BC_l^{(2,2)\sigma}(1)\; (l\geq 2)$}\label{data:BC-5}
\begin{enumerate}
\item $t_1(BC_l^{(2,2)\sigma}(1))=2$, \; $t_2(BC_l^{(2,2)\sigma}(1))=2$ and  \\
$(BC_l^{(2,2)\sigma}(1))^\vee=BC_l^{(2,2)\sigma}(1)$,
\item $R(BC_l^{(2,2)\sigma}(1))=(R(BC_l)_s+\Z a+\Z b) \cup (R(BC_l)_m+\Z a+\Z b)$ \\
\phantom{$R(BC_l^{(2,2)\sigma}(1))=$} $\cup (R(BC_l)_l+(1+2\Z) a+2\Z b)$. 
\item 
\begin{enumerate}
\item[i)] $\alpha_0=b-(\vep_1+\vep_2), \, \alpha_i=\vep_i-\vep_{i+1}\, 
(1\leq i\leq l-1), \, \alpha_l=\vep_l$, 
\item[ii)] $\alpha_l'= -a+2\vep_l$.
\end{enumerate}
\item 
\begin{enumerate}
\item[i)] $k(\alpha_i)=1\, (0\leq i\leq l)$, 
\item[ii)] $k^{\nr}((\alpha_l')^\ast)=1$. 
\end{enumerate}
\item $m_{\alpha_0}=m_{\alpha_1}=2, \; m_{\alpha_i}=4\, 
(2\leq i\leq l)$.
\item Elliptic diagram:
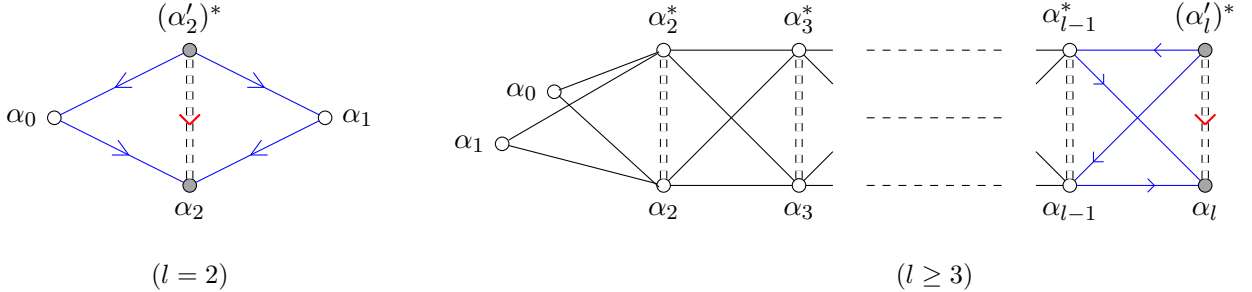
\input{Figure-new/diagram-BC22s1-red}
\item $W$-orbits:
\begin{enumerate}
\item[i)] $l=2$: $R(BC_2^{(2,2)\sigma}(1))_l=W((\alpha_2')^\ast),\quad R(BC_2^{(2,2)\sigma}(1))_s=W(\alpha_2)$, \\
\phantom{$l=2$:} $R(BC_2^{(2,2)\sigma}(1))_m=W(\alpha_0) \amalg W(\alpha_1)$.
\item[ii)] $l\geq 3$: $R(BC_l^{(2,2)\sigma}(1))_l=W((\alpha_l')^\ast),\quad   R(BC_l^{(2,2)\sigma}(1))_s=W(\alpha_l)$. \\
\phantom{$l\geq 3$:} $R(BC_l^{(2,2)\sigma}(1))_m$: single $W$-orbit.
\end{enumerate}
\end{enumerate}

\subsubsection{Type $BC_l^{(2,2)\sigma}(2)\; (l\geq 1)$}\label{data:BC-6}
\begin{enumerate}
\item $t_1(BC_l^{(2,2)\sigma}(2))=2$, \; $t_2(BC_l^{(2,2)\sigma}(2))=2$ and \\
$(BC_l^{(2,2)\sigma}(2))^\vee=BC_l^{(2,2)\sigma}(2)$,
\item $R(BC_l^{(2,2)\sigma}(2))=(R(BC_l)_s+\Z a+\Z b) \cup (R(BC_l)_m+\Z a+2\Z b)$ \\
\phantom{$R(BC_l^{(2,2)\sigma}(2))=$} $\cup (R(BC_l)_l+(1+2\Z) a+2\Z b)$. 
\item 
\begin{enumerate}
\item[i)] $\alpha_0=b-\vep_1, \, \alpha_i=\vep_i-\vep_{i+1}\, 
(1\leq i\leq l-1), \, \alpha_l=\vep_l$, 
\item[ii)] $\alpha_0'= -a+2b-2\vep_1$, \; $\alpha_l'=-a+2\vep_l$.
\end{enumerate}
\item 
\begin{enumerate}
\item[i)] $k(\alpha_i)=1\, (0\leq i\leq l)$, 
\item[ii)] $k^{\nr}((\alpha_0')^\ast)=k^{\nr}((\alpha_l')^\ast)=1$. 
\end{enumerate}
\item $m_{\alpha_i}=2\, (0 \leq i\leq l)$.
\item Elliptic diagram:

\input{Figure-new/diagram-BC22s2-red}

\item $W$-orbits: $R(BC_l^{(2,2)\sigma}(2))_l=W((\alpha_0')^\ast) \amalg W((\alpha_l')^\ast)$, \\
\phantom{$W$-orbits:} $R(BC_l^{(2,2)\sigma}(2))_m$: single $W$-orbit, \\
\phantom{$W$-orbits:} $R(BC_l^{(2,2)\sigma}(2))_s=W(\alpha_0) \amalg W(\alpha_l)$.
\end{enumerate}
\bigskip

\subsection{Non-Reduced mERSs with $R/G$ of type $\mathbf{BCC_l}$}
\label{sect_root-data-non-reduced-1}
\subsubsection{Type $BCC_l^{(1)}$ $(l\geq 1)$}\label{data:BCC-1}
\begin{enumerate}
\item $t_1(BCC_l^{(1)})=1,$ \quad $t_2(BCC_l^{(1)})=1,$ \quad $(BCC_l^{(1)})^\vee = C^\vee BC_l^{(4)}.$
\item $R(BCC_l^{(1)})=R(BC_l)+\Z a+\Z b$.
\item 
\begin{enumerate}
\item[i)]$\alpha_0=b-2\vep_1$, $\alpha_i=\vep_i-\vep_{i+1}\; 
(1\leq i\leq l-1)$, $\alpha_l=\vep_l$. 
\item[ii)] $\alpha_l'=2\vep_l$.
\end{enumerate}
\item 
\begin{enumerate}
\item[i)] $k(\alpha_i)=1\; (0\leq i\leq l)$,
\item[ii)] $k^{\nr}((\alpha_l')^\ast)=1$. 
\end{enumerate}
\item $m_{\alpha_i}=4\; (0\leq i\leq l)$.
\item Elliptic diagram:
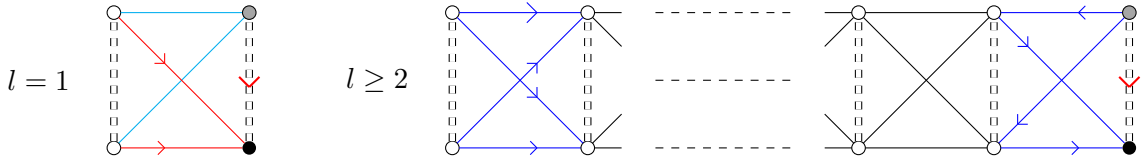
\input{Figure-new/diagram-1-nr}
\item $W$-orbits: $R(BCC_l^{(1)})_l=W(\alpha_0) \amalg W(\alpha_0^\ast) \amalg
W(\alpha_l') \amalg W((\alpha_l')^\ast)$, \\
\phantom{$W$-orbits:} $R(BCC_l^{(1)})_m$: single $W$-orbit, \quad $R(BCC_l^{(1)})_s=W(\alpha_l)$.
\end{enumerate}

\subsubsection{Type $BCC_l^{(1)\ast_0}$ $(l\geq 1)$}\label{data:BCC-2}
\begin{enumerate}
\item $t_1(BCC_l^{(1)\ast_{0}})=1,$ \quad $t_2(BCC_l^{(1)\ast_{0}})=1,$ \quad $(BCC_l^{(1)\ast_{0}})^\vee = C^\vee BC_l^{(4)\ast_{0}}.$
\item $R(BCC_l^{(1)\ast_0})=(R(BC_l)_s+\Z a+\Z b) \cup (R(BC_l)_m+\Z a+\Z b)$ \\
\phantom{$R(BCC_l^{(1)\ast_0})=$} $\cup (R(BC_l)_l+\{\, ma+nb\, \vert \, mn \equiv 0\, [2]\, \})$.
\item 
\begin{enumerate}
\item[i)]$\alpha_0=b-2\vep_1$, $\alpha_i=\vep_i-\vep_{i+1}\; 
(1\leq i\leq l-1)$, $\alpha_l=\vep_l$. 
\item[ii)] $\alpha_l'=2\vep_l$.
\end{enumerate}
\item 
\begin{enumerate}
\item[i)] $k(\alpha_0)=2$, \; $k(\alpha_i)=1\; (1\leq i\leq l)$, 
\item[ii)] $k^{\nr}((\alpha_l')^\ast)=1$. 
\end{enumerate}
\item  $m_{\alpha_0}=2$, \; $m_{\alpha_i}=4\; (1\leq i\leq l)$.
\item Elliptic diagram:
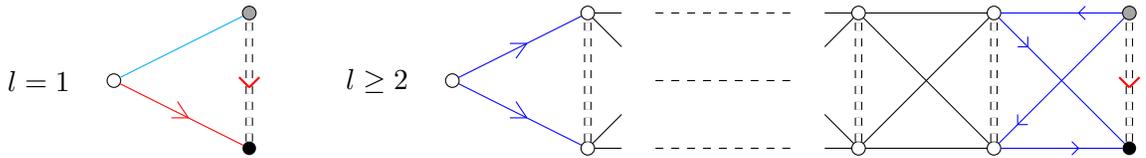
\input{Figure-new/diagram-2-nr} 
\item $W$-orbits: $R(BCC_l^{(1)\ast_{0}})_l=W(\alpha_0) \amalg 
W(\alpha_l') \amalg W((\alpha_l')^\ast)$, \\
\phantom{$W$-orbits:} $R(BCC_l^{(1)\ast_0})_m$: single $W$-orbit, \quad $R(BCC_l^{(1)\ast_0})_s=W(\alpha_l)$.
\end{enumerate}

\subsubsection{Type $BCC_l^{(1)\ast_{0'}}$ $(l\geq 1)$}\label{data:BCC-3}
\begin{enumerate}
\item $t_1(BCC_l^{(1)\ast_{0'}})=1,$ \quad $t_2(BCC_l^{(1)\ast_{0'}})=1,$ \quad $(BCC_l^{(1)\ast_{0'}})^\vee = C^\vee BC_l^{(4)\ast_{0'}}.$
\item $R(BCC_l^{(1)\ast_{0'}})=(R(BC_l)_s+\Z a+\Z b) \cup (R(BC_l)_m+\Z a+\Z b)$ \\
\phantom{$R(BCC_l^{(1)\ast_{0'}})=$} $\cup (R(BC_l)_l+\{\, ma+nb\, \vert \, m(n-1) \equiv 0\, [2]\, \})$.
\item 
\begin{enumerate}
\item[i)] $\alpha_0=b-2\vep_1$, $\alpha_i=\vep_i-\vep_{i+1}\; 
(1\leq i\leq l-1)$, $\alpha_l=\vep_l$. 
\item[ii)] $\alpha_l'=2\vep_l$.
\end{enumerate}
\item 
\begin{enumerate}
\item[i)] $k(\alpha_i)=1\; (0\leq i\leq l)$, 
\item[ii)] $k^{\nr}((\alpha_l')^\ast)=2$. 
\end{enumerate}
\item $m_{\alpha_i}=4\; (0\leq i\leq l-1)$, \;  $m_{\alpha_l}=2$.
\item Elliptic diagram:
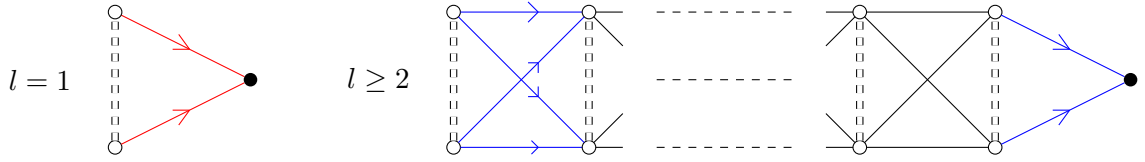
\input{Figure-new/diagram-3-nr}
\item $W$-orbits: $R(BCC_l^{(1)\ast_{0'}})_l=W(\alpha_0) \amalg W(\alpha_0^\ast) \amalg
W(\alpha_l')$, \\
\phantom{$W$-orbits:} $R(BCC_l^{(1)\ast_{0'}})_m$: single $W$-orbit, \quad $R(BCC_l^{(1)\ast_{0'}})_s=W(\alpha_l)$.
\end{enumerate}

\subsubsection{Type $BCC_l^{(2)}(1)$ $(l\geq 2)$}\label{data:BCC-4}
\begin{enumerate}
\item $t_1(BCC_l^{(2)}(1))=1,$ \quad $t_2(BCC_l^{(2)}(1))=2,$ \quad $(BCC_l^{(2)}(1))^\vee = C^\vee BC_l^{(2)}(1).$
\item $R(BCC_l^{(2)}(1))=(R(BC_l)_s+\Z a+\Z b) \cup (R(BC_l)_m+\Z a+\Z b)$ \\
\phantom{$R(BCC_l^{(2)}(1))=$} $\cup (R(BC_l)_l+2\Z a+\Z b)$.
\item 
\begin{enumerate}
\item[i)] $\alpha_0=b-2\vep_1$, $\alpha_i=\vep_i-\vep_{i+1}\; 
(1\leq i\leq l-1)$, $\alpha_l=\vep_l$. 
\item[ii)] $\alpha_l'=2\vep_l$.
\end{enumerate}
\item 
\begin{enumerate}
\item[i)] $k(\alpha_0)=2$, \; $k(\alpha_i)=1\; (1\leq i\leq l)$, 
\item[ii)] $k^{\nr}((\alpha_l')^\ast)=2$. 
\end{enumerate}
\item $m_{\alpha_0}=2$, \; $m_{\alpha_i}=4\; 
(1\leq i\leq l-1)$, \; $m_{\alpha_l}=2$.
\item Elliptic diagram:
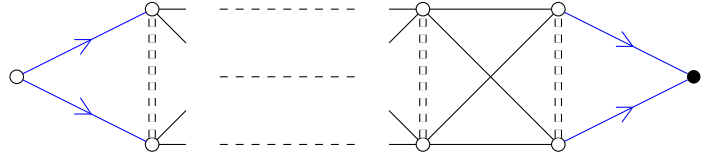
\input{Figure-new/diagram-4-nr}
\item $W$-orbits: 
\begin{enumerate}
\item[i)] $l=2$: $R(BCC_2^{(2)}(1))_l=W(\alpha_0) 
\amalg W(\alpha_2')$, \\
\phantom{$l=2$:} $R(BCC_2^{(2)}(1))_m=W(\alpha_1) \amalg W(\alpha_1^\ast)$, \\
\phantom{$l=2$:} $R(BCC_2^{(2)}(1))_s=W(\alpha_2)$.
\item[ii)] $l\geq 3$: $R(BCC_l^{(2)}(1))_l=W(\alpha_0) 
\amalg W(\alpha_l')$, \\
\phantom{$l=2$:} $R(BCC_l^{(2)}(1))_m$:  single $W$-orbit, \\
\phantom{$l=2$:} $R(BCC_2^{(2)}(1))_s=W(\alpha_l)$.
\end{enumerate}
\end{enumerate}

\subsubsection{Type $BCC_l^{(2)}(2)$ $(l\geq 1)$}\label{data:BCC-5}
\begin{enumerate}
\item $t_1(BCC_l^{(2)}(2))=1,$ \quad $t_2(BCC_l^{(2)}(2))=2,$ \quad $(BCC_l^{(2)}(2))^\vee = C^\vee BC_l^{(2)}(2).$
\item $R(BCC_l^{(2)}(2))=(R(BC_l)_s+\Z a+\Z b) \cup (R(BC_l)_m+2\Z a+\Z b)$ \\
\phantom{$R(BCC_l^{(2)}(2))=$} $\cup (R(BC_l)_l+2\Z a+\Z b)$.
\item 
\begin{enumerate}
\item[i)] $\alpha_0=b-2\vep_1$, $\alpha_i=\vep_i-\vep_{i+1}\; 
(1\leq i\leq l-1)$, $\alpha_l=\vep_l$. 
\item[ii)] $\alpha_l'=2\vep_l$.
\end{enumerate}
\item 
\begin{enumerate}
\item[i)] $k(\alpha_i)=2\; (0\leq i\leq l-1)$, \; $k(\alpha_l)=1$, 
\item[ii)] $k^{\nr}((\alpha_l')^\ast)=2$. 
\end{enumerate}
\item $m_{\alpha_i}=2\; (0\leq i\leq l)$.
\item Elliptic diagram:
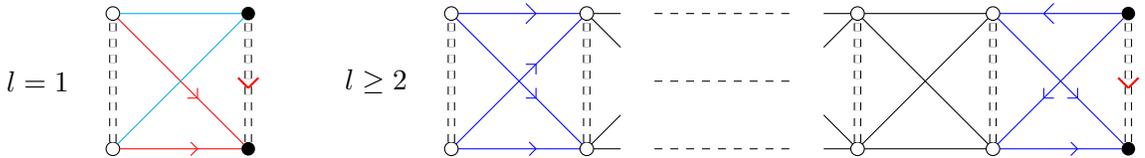
\input{Figure-new/diagram-5-nr}
\item $W$-orbits: $R(BCC_l^{(2)}(2))_l=W(\alpha_0) \amalg W(\alpha_0^\ast)
\amalg W(\alpha_l') \amalg W((\alpha_l')^\ast)$, \\
\phantom{$W$-orbits:} $R(BCC_l^{(2)}(2))_m$: single $W$-orbit, \\
\phantom{$W$-orbits:} $R(BCC_l^{(2)}(2))_s=W(\alpha_l) \amalg W(\alpha_l^\ast)$.
\end{enumerate}

\subsubsection{Type $BCC_l^{(2)\ast_0}$} \label{data:BCC-6}
\begin{enumerate}
\item $t_1(BCC_l^{(2)\ast_0})=1,$ \quad $t_2(BCC_l^{(2)\ast_0})=2,$ \quad 
$(BCC_l^{(2)\ast_0})^\vee = C^\vee BC_l^{(2)\ast_0}.$
\item $R(BCC_l^{(2)\ast_0})=(R(BC_l)_s+\Z a+\Z b) \cup (R(BC_l)_m+2\Z a+\Z b)$ \\
\phantom{$R(BCC_l^{(2)\ast_0})=$} $\cup (R(BC_l)_l+\{\, 2ma+nb\, \vert\, mn 
\equiv 0\, [2]\, \})$. 
\item 
\begin{enumerate}
\item[i)] $\alpha_0=b-2\vep_1$, $\alpha_i=\vep_i-\vep_{i+1}\; 
(1\leq i\leq l-1)$, 
$\alpha_l=\vep_l$. 
\item[ii)] $\alpha_l'=2\vep_l$.
\end{enumerate}
\item 
\begin{enumerate}
\item[i)] $k(\alpha_0)=4$, \; $k(\alpha_i)=2\; (1\leq i\leq l-1)$, \; 
$k(\alpha_l)=1$, 
\item[ii)] $k^{\nr}((\alpha_l')^\ast)=2$.
\end{enumerate}
\item $m_{\alpha_0}=1$, \; $m_{\alpha_i}=2\; (1\leq i\leq l)$.
\item Elliptic diagram:
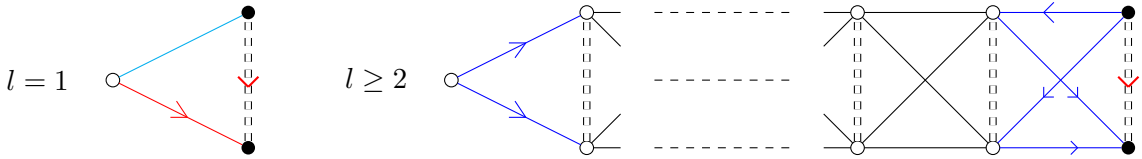
\input{Figure-new/diagram-6-nr}
\item $W$-orbits: $R(BCC_l^{(2)\ast_0})_l=W(\alpha_0) \amalg W(\alpha_l') \amalg
W((\alpha_l')^\ast)$, \\
\phantom{$W$-orbits:} $R(BCC_l^{(2)\ast_0})_m$: single $W$-orbit, \\
\phantom{$W$-orbits:} $R(BCC_l^{(2)\ast_0})_s=W(\alpha_l) \amalg W(\alpha_l^\ast)$.
\end{enumerate}

\subsubsection{Type $BCC_l^{(2)\ast_1}$ $(l\geq 1)$}\label{data:BCC-7}
\begin{enumerate}
\item $t_1(BCC_l^{(2)\ast_1})=1,$ \quad $t_2(BCC_l^{(2)\ast_1})=2,$ \quad $(BCC_l^{(2)\ast_1})^\vee = C^\vee BC_l^{(2)\ast_1}.$
\item $R(BCC_l^{(2)\ast_1})=(R(BC_l)_s+\Z a+\Z b) \cup (R(BC_l)_m+2\Z a+\Z b)$ \\
\phantom{$R(BCC_l^{(2)\ast_1})=$} $\cup (R(BC_l)_l+\{\, 2ma+nb\, \vert\, (m-1)(n-1) \equiv 0\, [2]\, \})$. 
\item 
\begin{enumerate}
\item[i)] $\alpha_0=-2a+b-2\vep_1$, $\alpha_i=\vep_i-\vep_{i+1}\; 
(1\leq i\leq l-1)$, $\alpha_l=\vep_l$. 
\item[ii)] $\alpha_l'= -2a+2\vep_l$.
\end{enumerate}
\item 
\begin{enumerate}
\item[i)] $k(\alpha_i)=2\; (0\leq i\leq l-1)$, \; $k(\alpha_l)=1$, 
\item[ii)] $k^{\nr}((\alpha_l')^\ast)=2$. 
\end{enumerate}
\item $m_{\alpha_i}=2\; (0\leq i\leq l)$.
\item Elliptic diagram:
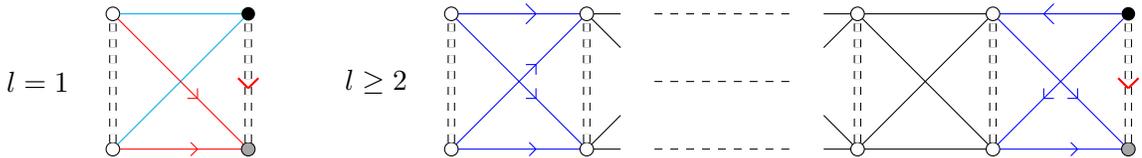
\input{Figure-new/diagram-7-nr}
\item $W$-orbits: $R(BCC_l^{(2)\ast_1})_l=W(\alpha_0) \amalg W(\alpha_0^\ast) \amalg W((\alpha_l')^\ast)$, \\
\phantom{$W$-orbits:} $R(BCC_l^{(2)\ast_1})_m$: single $W$-orbit, \\
\phantom{$W$-orbits:}  $R(BCC_l^{(2)\ast_1})_s=W(\alpha_l) \amalg W(\alpha_l^\ast)$.
\end{enumerate}


\subsubsection{Type $BCC_l^{(4)}$ $(l\geq 1)$}\label{data:BCC-8}
\begin{enumerate}
\item $t_1(BCC_l^{(4)})=1,$ \quad $t_2(BCC_l^{(4)})=4,$ \quad $(BCC_l^{(4)})^\vee = C^\vee BC_l^{(1)}.$
\item $R(BCC_l^{(4)})=(R(BC_l)_s+\Z a+\Z b) \cup (R(BC_l)_m+2\Z a+\Z b)$ \\
\phantom{$R(BCC_l^{(4)})=$} $\cup (R(BC_l)_l+4\Z a+\Z b)$. 
\item 
\begin{enumerate}
\item[i)] $\alpha_0=b-2\vep_1$, $\alpha_i=\vep_i-\vep_{i+1}\; 
(1\leq i\leq l-1)$, $\alpha_l=a+\vep_l$. 
\item[ii)] $\alpha_l'=2\vep_l$.
\end{enumerate}
\item 
\begin{enumerate}
\item[i)] $k(\alpha_0)=4$, \; $k(\alpha_i)=2\; (1\leq i\leq l-1)$, \; 
$k(\alpha_l)=1$, 
\item[ii)] $k^{\nr}((\alpha_l')^\ast)=2$. 
\end{enumerate}
\item $m_{\alpha_0}=1$, \; $m_{\alpha_i}=2\; (1\leq i\leq l)$.
\item Elliptic diagram:
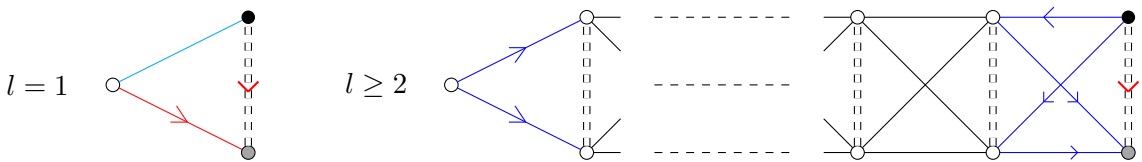
\input{Figure-new/diagram-8-nr}
\item $W$-orbits: $R(BCC_l^{(4)})_l=W(\alpha_0) \amalg W((\alpha_l')^\ast)$, \\
\phantom{$W$-orbits:} $R(BCC_l^{(4)})_m$: single $W$-orbit, \\
\phantom{$W$-orbits:} $R(BCC_l^{(4)})_s=W(\alpha_l) \amalg W(\alpha_l^\ast)$.
\end{enumerate}
\bigskip

\subsection{Non-Reduced mERSs with $R/G$ of type $\mathbf{C^\vee BC_l}$}
\label{sect_root-data-non-reduced-2}
\subsubsection{Type $C^\vee BC_l^{(1)}$ $(l\geq 1)$}\label{data:CBC-1}
\begin{enumerate}
\item $t_1(C^\vee BC_l^{(1)})=4,$ \quad $t_2(C^\vee BC_l^{(1)})=1,$ \quad $(C^\vee BC_l^{(1)})^\vee = BCC_l^{(4)}.$
\item $R(C^\vee BC_l^{(1)})=(R(BC_l)_s+\Z a+\Z b) \cup (R(BC_l)_m+\Z a+2\Z b)$ \\
\phantom{$R(C^\vee BC_l^{(1)})=$}$\cup (R(BC_l)_l+\Z a+4\Z b)$. 
\item 
\begin{enumerate}
\item[i)] $\alpha_0=b-\vep_1$, $\alpha_i=\vep_i-\vep_{i+1}\; 
(1\leq i\leq l-1)$, $\alpha_l=\vep_l$. 
\item[ii)] $\alpha_l'=2\vep_l$.
\end{enumerate}
\item 
\begin{enumerate}
\item[i)] $k(\alpha_i)=1\; (0\leq i\leq l)$,
\item[ii)] $k^{\nr}((\alpha_l')^\ast)=1$. 
\end{enumerate}
\item $m_{\alpha_0}=1$, \; $m_{\alpha_i}=2\; (1\leq i \leq l)$.
\item Elliptic diagram:
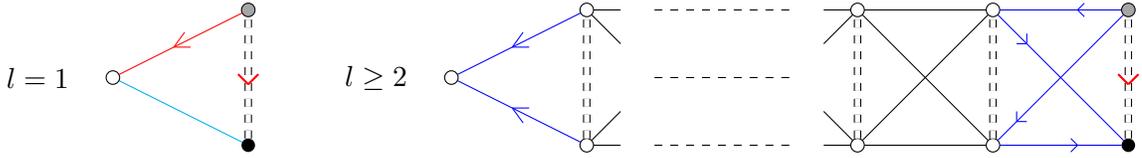
\input{Figure-new/diagram-9-nr}
\item $W$-orbits: $R(C^\vee BC_l^{(1)})_l=W(\alpha_l') \amalg W((\alpha_l')^\ast)$, \\
\phantom{$W$-orbits:} $R(C^\vee BC_l^{(1)})_m$: single $W$-orbit, \\
\phantom{$W$-orbits:}  $R(C^\vee BC_l^{(1)})_s=W(\alpha_0) \amalg W(\alpha_l)$.
\end{enumerate}

\subsubsection{Type $C^\vee BC_l^{(2)}(1)$ $(l\geq 2)$}\label{data:CBC-2}
\begin{enumerate}
\item $t_1(C^\vee BC_l^{(2)}(1))=4,$ \quad $t_2(C^\vee BC_l^{(2)}(1))=2,$ \\ $(C^\vee BC_l^{(2)}(1))^\vee = BCC_l^{(2)}(1).$
\item $R(C^\vee BC_l^{(2)}(1))=(R(BC_l)_s+\Z a+\Z b) \cup (R(BC_l)_m+\Z a+2\Z b)$
\phantom{$R(C^\vee BC_l^{(2)}(1))=$} $\cup (R(BC_l)_l+2\Z a+4\Z b)$. 
\item 
\begin{enumerate}
\item[i)] $\alpha_0=b-\vep_1$, $\alpha_i=\vep_i-\vep_{i+1}\; 
(1\leq i\leq l-1)$, $\alpha_l=\vep_l$. 
\item[ii)] $\alpha_l'=2\vep_l$.
\end{enumerate}
\item 
\begin{enumerate}
\item[i)] $k(\alpha_i)=1\; (0\leq i\leq l)$,
\item[ii)] $k^{\nr}((\alpha_l')^\ast)=2$. 
\end{enumerate}
\item $m_{\alpha_0}=1$, \; $m_{\alpha_i}=2\; 
(1\leq i\leq l-1)$, \;  $m_{\alpha_l}=1$.
\item Elliptic diagram:
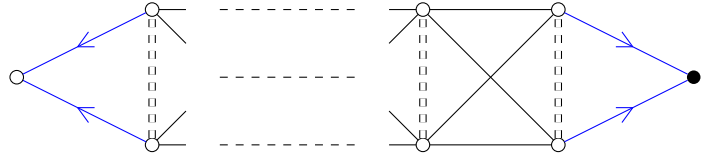
\input{Figure-new/diagram-10-nr}
\item $W$-orbits: \begin{enumerate}
\item[i)] $l=2$: $R(C^\vee BC_2^{(2)}(1))_l=W(\alpha_2')$, \\
\phantom{$l=2$:}  $R(C^\vee BC_2^{(2)}(1))_m=W(\alpha_1) \amalg W(\alpha_1^\ast)$, \\
\phantom{$l=2$:}  $R(C^\vee BC_2^{(2)}(1))_s=W(\alpha_0) \amalg W(\alpha_2)$.
\item[ii)] $l\geq 3$: $R(C^\vee BC_l^{(2)}(1))_l=W(\alpha_l')$, \\ 
\phantom{$l\geq 3$:} $R(C^\vee BC_l^{(2)}(1))_m$: single $W$-orbit, \\
\phantom{$l\geq 3$:} $R(C^\vee BC_l^{(2)}(1))_s=W(\alpha_0) \amalg W(\alpha_l)$.
\end{enumerate}
\end{enumerate}

\subsubsection{Type $C^\vee BC_l^{(2)}(2)$ $(l\geq 1)$}\label{data:CBC-3}
\begin{enumerate}
\item $t_1(C^\vee BC_l^{(2)}(2))=4,$ \quad $t_2(C^\vee BC_l^{(2)}(2))=2,$ \\ $(C^\vee BC_l^{(2)}(2))^\vee = BCC_l^{(2)}(2).$
\item $R(C^\vee BC_l^{(2)}(2))=(R(BC_l)_s+\Z a+\Z b) \cup (R(BC_l)_m+2\Z a+2\Z b)$  \\
\phantom{$R(C^\vee BC_l^{(2)}(2))=$} $\cup (R(BC_l)_l+2\Z a+4\Z b)$.
\item 
\begin{enumerate}
\item[i)] $\alpha_0=b-\vep_1$, $\alpha_i=\vep_i-\vep_{i+1}\; 
(1\leq i\leq l-1)$, $\alpha_l=\vep_l$. 
\item[ii)] $\alpha_l'=2\vep_l$.
\end{enumerate}
\item 
\begin{enumerate}
\item[i)] $k(\alpha_0)=1$, \; $k(\alpha_i)=2\; 
(1\leq i\leq l-1)$,\; $k(\alpha_l)=1$, 
\item[ii)] $k^{\nr}((\alpha_l')^\ast)=2$.
\end{enumerate}
\item $m_{\alpha_i}=1\; (0\leq i \leq l)$. 
\item Elliptic diagram:
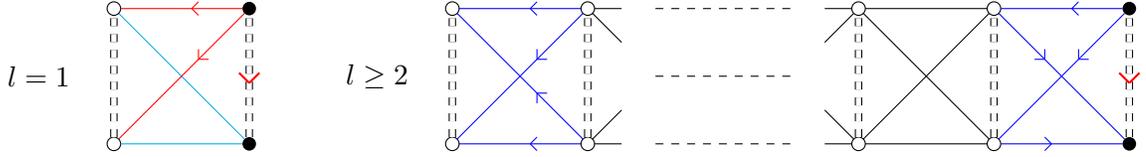
\input{Figure-new/diagram-11-nr}
\item $W$-orbits: $R(C^\vee BC_l^{(2)}(2))_l=W(\alpha_l') \amalg W((\alpha_l')^\ast)$, \\
\phantom{$W$-orbits:}  $R(C^\vee BC_l^{(2)}(2))_m$: single $W$-orbit, \\
\phantom{$W$-orbits:} 
$R(C^\vee BC_l^{(2)}(2))_s=W(\alpha_0) \amalg W(\alpha_0^\ast)
\amalg W(\alpha_l) \amalg W(\alpha_l^\ast)$.
\end{enumerate}

\subsubsection{Type $C^\vee BC_l^{(2)\ast_0}$ $(l\geq 1)$}\label{data:CBC-4}
\begin{enumerate}
\item $t_1(C^\vee BC_l^{(2)\ast_0})=4,$ \quad $t_2(C^\vee BC_l^{(2)\ast_0})=2,$ \quad $(C^\vee BC_l^{(2)\ast_0})^\vee = BCC_l^{(2)\ast_0}.$
\item $R(C^\vee BC_l^{(2)\ast_0})=(R(BC_l)_s+\{\, ma+nb\, \vert\, mn \equiv 0\, [2]\, \})$ \\
\phantom{$R(C^\vee BC_l^{(2)\ast_0})=$} $\cup (R(BC_l)_m+2\Z a+2\Z b) \cup (R(BC_l)_l+2\Z a+4\Z b)$. 
\item 
\begin{enumerate}
\item[i)] $\alpha_0=b-\vep_1$, $\alpha_i=\vep_i-\vep_{i+1}\; 
(1\leq i\leq l-1)$, $\alpha_l=\vep_l$. 
\item[ii)] $\alpha_l'=2\vep_l$.
\end{enumerate}
\item 
\begin{enumerate}
\item[i)] $k(\alpha_i)=2\; 
(0\leq i\leq l-1)$, \; $k(\alpha_l)=1$, 
\item[ii)] $k^{\nr}((\alpha_l')^\ast)=2$. 
\end{enumerate}
\item $m_{\alpha_0}=\frac{1}{2}$, \; $m_{\alpha_i}=1\; (1\leq i\leq l)$.
\item Elliptic diagram:
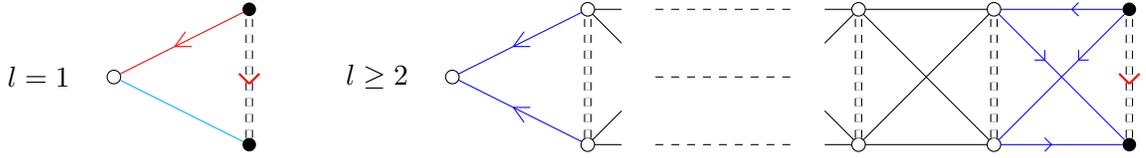
\input{Figure-new/diagram-12-nr}
\item $W$-orbits: 
$R(C^\vee BC_l^{(2)\ast_0})_l=W(\alpha_l') \amalg W((\alpha_l')^\ast)$, \\
\phantom{$W$-orbits:}  $R(C^\vee BC_l^{(2)\ast_0})_m$: single $W$-orbit, \\
\phantom{$W$-orbits:} 
$R(C^\vee BC_l^{(2)\ast_0})_s=W(\alpha_0) \amalg W(\alpha_l) \amalg W(\alpha_l^\ast)$.
\end{enumerate}

\subsubsection{Type $C^\vee BC_l^{(2)\ast_1}$ $(l\geq 1)$}\label{data:CBC-5}
\begin{enumerate}
\item $t_1(C^\vee BC_l^{(2)\ast_1})=4,$ \quad $t_2(C^\vee BC_l^{(2)\ast_1})=2,$ \quad $(C^\vee BC_l^{(2)\ast_1})^\vee = BCC_l^{(2)\ast_1}.$
\item $R(C^\vee BC_l^{(2)\ast_1})=(R(BC_l)_s+\{\, ma+nb\, \vert\, (m-1)(n-1) \equiv 0\, [2]\, \})$ \\
\phantom{$R(C^\vee BC_l^{(2)\ast_0})=$} $\cup (R(BC_l)_m+2\Z a+2\Z b) \cup (R(BC_l)_l+2\Z a+4\Z b)$. 
\item
\begin{enumerate}
\item[i)] $\alpha_0=-a+b-\vep_1$, $\alpha_i=\vep_i-\vep_{i+1}\; 
(1\leq i\leq l-1)$, $\alpha_l=a+\vep_l$. 
\item[ii)] $\alpha_l'=2a+2\vep_l$.
\end{enumerate}
\item 
\begin{enumerate}
\item[i)] $k(\alpha_0)=1$, \; $k(\alpha_i)=2\; (1\leq i\leq l)$, 
\item[ii)] $k^{\nr}((\alpha_l')^\ast)=2$. 
\end{enumerate}
\item $m_{\alpha_i}=1\; (0\leq i \leq l)$.
\item Elliptic diagram:
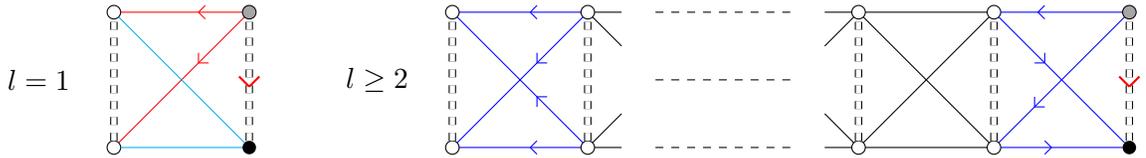
\input{Figure-new/diagram-13-nr}
\item $W$-orbits: 
$R(C^\vee BC_l^{(2)\ast_1})_l=W(\alpha_l') \amalg W((\alpha_l')^\ast)$, \\
\phantom{$W$-orbits:} $R(C^\vee BC_l^{(2)\ast_1})_m$: single $W$-orbit, \\
\phantom{$W$-orbits:} 
$R(C^\vee BC_l^{(2)\ast_1})_s=W(\alpha_0) \amalg W(\alpha_0^\ast)\amalg W(\alpha_l)$.
\end{enumerate}

\subsubsection{Type $C^\vee BC_l^{(4)}$ $(l\geq 1)$}\label{data:CBC-6}
\begin{enumerate}
\item $t_1(C^\vee BC_l^{(4)})=4,$ \quad $t_2(C^\vee BC_l^{(4)})=4,$ \quad $(C^\vee BC_l^{(4)})^\vee = BCC_l^{(1)}.$
\item $R(C^\vee BC_l^{(4)})=(R(BC_l)_s+\Z a+\Z b) \cup (R(BC_l)_m+2\Z a+2\Z b)$ \\ 
\phantom{$R(C^\vee BC_l^{(4)})=$} $\cup (R(BC_l)_l+4\Z a+4\Z b)$. 
\item 
\begin{enumerate}
\item[i)] $\alpha_0=b-\vep_1$, $\alpha_i=\vep_i-\vep_{i+1}\; 
(1\leq i\leq l-1)$, $\alpha_l=a+\vep_l$. 
\item[ii)] $\alpha_l'=2\vep_l$.
\end{enumerate}
\item 
\begin{enumerate}
\item[i)] $k(\alpha_0)=1$, \; $k(\alpha_i)=2\; 
(1\leq i\leq l-1)$, \; $k(\alpha_l)=1$, 
\item[ii)] $k^{\nr}((\alpha_l')^\ast)=2$.
\end{enumerate}
\item $m_{\alpha_i}=1\; (0\leq i\leq l)$.
\item Elliptic diagram: 
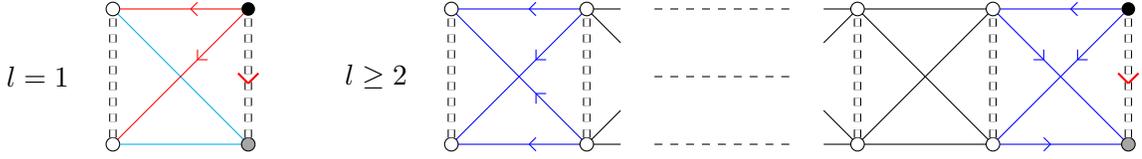
\input{Figure-new/diagram-14-nr}
\item $W$-orbits:  $R(C^\vee BC_l^{(4)})_l=W((\alpha_l')^\ast)$, \\
\phantom{$W$-orbits:}  $R(C^\vee BC_l^{(4)})_m$: single $W$-orbit, \\
\phantom{$W$-orbits:} 
$R(C^\vee BC_l^{(4)})_s=W(\alpha_0) \amalg W(\alpha_0^\ast)
\amalg W(\alpha_l) \amalg W(\alpha_l^\ast)$.
\end{enumerate}

\subsubsection{Type $C^\vee BC_l^{(4)\ast_0}$ $(l\geq 1)$}\label{data:CBC-7}
\begin{enumerate}
\item $t_1(C^\vee BC_l^{(4)\ast_{0}})=4,$ \quad $t_2(C^\vee BC_l^{(4)\ast_{0}})=4,$ \quad $(C^\vee BC_l^{(4)\ast_{0}})^\vee = BCC_l^{(1)\ast_{0}}.$
\item $R(C^\vee BC_l^{(4)\ast_0})=(R(BC_l)_s+\{\, ma+nb\, \vert\, mn \equiv 0\, [2]\, \})$ \\
\phantom{$R(C^\vee BC_l^{(4)\ast_0})=$} $\cup (R(BC_l)_m+2\Z a+2\Z b) \cup (R(BC_l)_l+4\Z a+4\Z b)$. 
\item 
\begin{enumerate}
\item[i)] $\alpha_0=b-\vep_1$, $\alpha_i=\vep_i-\vep_{i+1}\; 
(1\leq i\leq l-1)$, $\alpha_l=a+\vep_l$. 
\item[ii)] $\alpha_l'=2\vep_l$.
\end{enumerate}
\item 
\begin{enumerate}
\item[i)] $k(\alpha_i)=2\; (0\leq i \leq l-1)$, \; $k(\alpha_l)=1$, 
\item[ii)] $k^{\nr}((\alpha_l')^\ast)=2$. 
\end{enumerate}
\item $m_{\alpha_0}=\frac{1}{2}$, \; $m_{\alpha_i}=1\; (1\leq i \leq l)$.
\item Elliptic diagram:
\input{Figure-new/diagram-15-nr}
\item $W$-orbits: $R(C^\vee BC_l^{(4)\ast_{0}})_l=W((\alpha_l')^\ast)$, \\
\phantom{$W$-orbits:}  $R(C^\vee BC_l^{(4)\ast_{0}})_m$: single $W$-orbit, \\
\phantom{$W$-orbits:} $R(C^\vee BC_l^{(4)\ast_{0}})_s=W(\alpha_0) \amalg 
W(\alpha_l) \amalg W(\alpha_l^\ast)$. 
\end{enumerate}

\subsubsection{Type $C^\vee BC_l^{(4)\ast_{0'}}$ $(l\geq 1)$}\label{data:CBC-8}
\begin{enumerate}
\item $t_1(C^\vee BC_l^{(4)\ast_{0'}})=4,$ \quad $t_2(C^\vee BC_l^{(4)\ast_{0'}})=4,$ \\ $(C^\vee BC_l^{(4)\ast_{0'}})^\vee = BCC_l^{(1)\ast_{0'}}.$
\item $R(C^\vee BC_l^{(4)\ast_{0'}})=(R(BC_l)_s+\{\, ma+nb\, \vert\, m(n-1) \equiv 0\, [2]\, \})$ \\
\phantom{$R(C^\vee BC_l^{(4)\ast_{0'}})=$} $\cup (R(BC_l)_m+2\Z a+2\Z b) \cup (R(BC_l)_l+4\Z a+4\Z b)$. 
\item 
\begin{enumerate}
\item[i)] $\alpha_0=b-\vep_1$, $\alpha_i=\vep_i-\vep_{i+1}\; 
(1\leq i\leq l-1)$, $\alpha_l=\vep_l$. 
\item[ii)] $\alpha_l'=2\vep_l$.
\end{enumerate}
\item 
\begin{enumerate}
\item[i)] $k(\alpha_0)=1$, \; $k(\alpha_i)=2 \; (1\leq i\leq l)$, 
\item[ii)] $k^{\nr}((\alpha_l')^\ast)=4$. 
\end{enumerate}
\item $m_{\alpha_i}=1\; (0\leq i\leq l-1)$, \; $m_{\alpha_l}=\frac{1}{2}$.
\item Elliptic diagram:
\input{Figure-new/diagram-16-nr}
\item $W$-orbits:  $R(C^\vee BC_l^{(4)\ast_{0'}})_l=W(\alpha_l')$, \\
\phantom{$W$-orbits:}  $R(C^\vee BC_l^{(4)\ast_{0'}})_m$: single $W$-orbit, \\
\phantom{$W$-orbits:} $R(C^\vee BC_l^{(4)\ast_{0'}})_s=W(\alpha_0) \amalg W(\alpha_0^\ast) \amalg
W(\alpha_l)$. 
\end{enumerate}
\bigskip
\subsection{Non-Reduced mERSs with $R/G$ of type $\mathbf{BB_l^\vee}$}
\label{sect_root-data-non-reduced-3}
\subsubsection{Type $BB_l^{\vee (1)}$ $(l\geq 2)$}\label{data:BB-1}
\begin{enumerate}
\item $t_1(BB_l^{\vee (1)})=2,$ \quad $t_2(BB_l^{\vee (1)})=1,$ \quad $(BB_l^{\vee (1)})^\vee = BB_l^{\vee (4)}.$
\item $R(BB_l^{\vee (1)})=(R(BC_l)_s+\Z a+\Z b) \cup (R(BC_l)_m+\Z a+\Z b)$ \\
\phantom{$R(BB_l^{\vee (1)})=$} $\cup (R(BC_l)_l+\Z a+2\Z b)$. 
\item 
\begin{enumerate}
\item[i)] $\alpha_0=b-\vep_1-\vep_2$, $\alpha_i=\vep_i-\vep_{i+1}\; 
(1\leq i \leq l-1)$, $\alpha_l=\vep_l$. 
\item[ii)] $\alpha_l'=2\vep_l$.
\end{enumerate}
\item 
\begin{enumerate}
\item[i)] $k(\alpha_i)=1\; (0\leq i\leq l)$, 
\item[ii)] $k^{\nr}((\alpha_l')^\ast)=1$. 
\end{enumerate}
\item $m_{\alpha_0}=m_{\alpha_1}=2$, \; $m_{\alpha_i}=4\; 
(2\leq i \leq l)$.
\item Elliptic diagram:
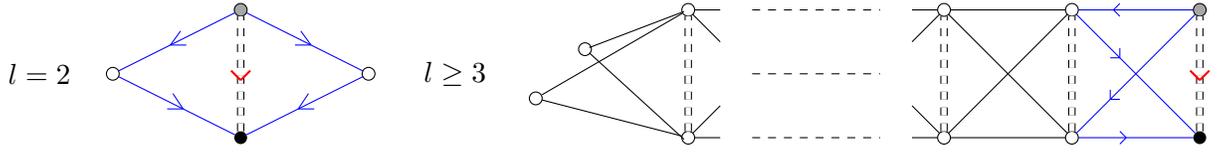
\input{Figure-new/diagram-17-nr}
\item $W$-orbits:
\begin{enumerate}
\item[i)] $l=2$: $R(BB_2^{\vee (1)})_l=W(\alpha_2') \amalg W((\alpha_2')^\ast)$, 
\\ 
\phantom{$l=2$:} $R(BB_2^{\vee (1)})_m=W(\alpha_0) \amalg W(\alpha_1)$, \\
\phantom{$l=2$:} $R(BB_2^{\vee (1)})_s=W(\alpha_2)$.
\item[ii)] $l\geq 3$: $R(BB_l^{\vee (1)})_l=W(\alpha_l') \amalg W((\alpha_l')^\ast)$, \\
\phantom{$l\geq 3$:} $R(BB_l^{\vee (1)})_m$: single $W$-orbit, \\
\phantom{$l\geq 3$:} $R(BB_l^{\vee (1)})_s=W(\alpha_l)$.
\end{enumerate}
\end{enumerate}

\subsubsection{Type $BB_l^{\vee (2)}(1)$ $(l\geq 2)$}\label{data:BB-2}
\begin{enumerate}
\item $t_1(BB_l^{\vee (2)}(1))=2,$ \quad $t_2(BB_l^{\vee (2)}(1))=2,$ \quad $(BB_l^{\vee (2)}(1))^\vee = BB_l^{\vee (2)}(1).$
\item $R(BB_l^{\vee (2)}(1))=(R(BC_l)_s+\Z a+\Z b) \cup (R(BC_l)_m+\Z a+\Z b)$ \\
\phantom{$R(BB_l^{\vee (2)}(1))=$} $\cup (R(BC_l)_l+2\Z a+2\Z b)$. 
\item 
\begin{enumerate}
\item[i)] $\alpha_0=b-\vep_1-\vep_2$, $\alpha_i=\vep_i-\vep_{i+1}\; 
(1\leq i \leq l-1)$, $\alpha_l=\vep_l$. 
\item[ii)] $\alpha_l'=2\vep_l$.
\end{enumerate}
\item 
\begin{enumerate}
\item[i)] $k(\alpha_i)=1\; (0\leq i\leq l)$, 
\item[ii)] $k^{\nr}((\alpha_l')^\ast)=2$. 
\end{enumerate}
\item $m_{\alpha_0}=m_{\alpha_1}=2$, \; $m_{\alpha_i}=4\; 
(2\leq i\leq l-1)$, \;  $m_{\alpha_l}=2$.
\item Elliptic diagram:
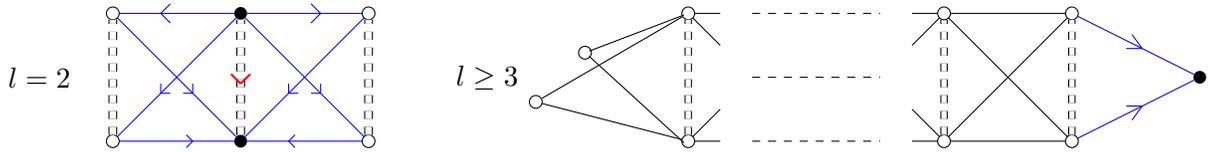
\input{Figure-new/diagram-18-nr}
\item $W$-orbits:
\begin{enumerate}
\item[i)] $l=2$: $R(BB_2^{\vee (2)}(1))_l=W((\alpha_2')^\ast), \quad R(BB_2^{\vee (2)}(1))_s=W(\alpha_2)$, \\
\phantom{$l=2$:} $R(BB_2^{\vee (2)}(1))_m=W(\alpha_0) \amalg W(\alpha_0^\ast) \amalg W(\alpha_1) \amalg W(\alpha_1^\ast)$.
\item[ii)] $l\geq 3$:  $R(BB_l^{\vee (2)}(1))_l=W(\alpha_l'), \quad R(BB_l^{\vee (2)}(1))_s=W(\alpha_l)$, \\
\phantom{$l\geq 3$:} $R(BB_l^{\vee (2)}(1))_m$: single $W$-orbit.
\end{enumerate}
\end{enumerate}

\subsubsection{Type $BB_l^{\vee (2)}(2)$ $(l\geq 2)$}\label{data:BB-3}
\begin{enumerate}
\item $t_1(BB_l^{\vee (2)}(2))=2,$ \quad $t_2(BB_l^{\vee (2)}(2))=2,$ \quad $(BB_l^{\vee (2)}(2))^\vee = BB_l^{\vee (2)}(2).$
\item $R(BB_l^{\vee (2)}(2))=(R(BC_l)_s+\Z a+\Z b) \cup (R(BC_l)_m+2\Z a+\Z b)$ \\
\phantom{$R(BB_l^{\vee (2)}(2))=$} $\cup (R(BC_l)_l+2\Z a+2\Z b)$. 
\item 
\begin{enumerate}
\item[i)] $\alpha_0=b-\vep_1-\vep_2$, $\alpha_i=\vep_i-\vep_{i+1}\; 
(1\leq i\leq l-1)$, $\alpha_l=\vep_l$. 
\item[ii)] $\alpha_l'=2\vep_l$.
\end{enumerate}
\item 
\begin{enumerate}
\item[i)] $k(\alpha_i)=2\; (0\leq i\leq l-1)$, \; $k(\alpha_l)=1$, 
\item[ii)] $k^{\nr}((\alpha_l')^\ast)=2$. 
\end{enumerate}
\item $m_{\alpha_0}=m_{\alpha_1}=1$, \; $m_{\alpha_i}=2\; 
(2\leq i \leq l)$.
\item Elliptic diagram:
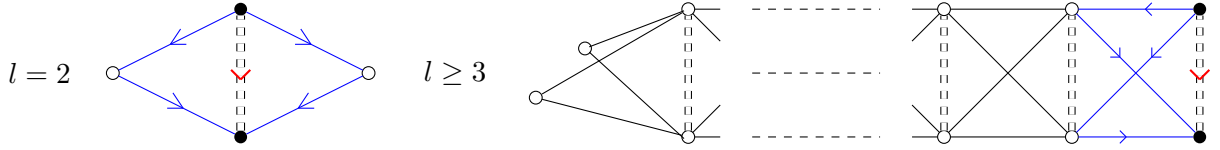
\input{Figure-new/diagram-19-nr}
\item $W$-orbits:
\begin{enumerate}
\item[i)] $l=2$: $R(BB_2^{\vee (2)}(2))_l=W(\alpha_2') \amalg W((\alpha_2')^\ast)$, \\
\phantom{$l=2$:} $R(BB_2^{\vee (2)}(2))_m=W(\alpha_0) \amalg W(\alpha_1)$, \\
\phantom{$l=2$:} $R(BB_2^{\vee (2)}(2))_s=W(\alpha_2) \amalg W(\alpha_2^\ast)$.
\item[ii)] $l\geq 3$: $R(BB_l^{\vee (2)}(2))_l =W(\alpha_l') \amalg W((\alpha_l')^\ast)$, \\
\phantom{$l\geq 3$:}  $R(BB_l^{\vee (2)}(2))_m$: single $W$-orbit, \\
\phantom{$l\geq 3$:}  $R(BB_l^{\vee (2)}(2))_s=W(\alpha_l) \amalg W(\alpha_l^\ast)$.
\end{enumerate}
\end{enumerate}
\subsubsection{Type $BB_2^{\vee (2) \ast}$}\label{data:BB-4}
\begin{enumerate}
\item $t_1(BB_2^{\vee (2)\ast})=2,$ \quad $t_2(BB_2^{\vee (2)\ast})=2,$ \quad $(BB_2^{\vee (2)\ast})^\vee = BB_2^{\vee (2)\ast}.$
\item $R(BB_2^{\vee (2)\ast})=(R(BC_2)_s+\Z a+\Z b)$ \\
\phantom{$R(BB_2^{\vee ()\ast})=\,$}$\cup (R(BC_2)_m+\{\, ma+nb\, \vert\, mn \equiv 0\, [2]\, \})$ \\
\phantom{$R(BB_2^{\vee ()\ast})=\,$}$\cup (R(BC_2)_l+2\Z a+2\Z b)$. 
\item 
\begin{enumerate}
\item[i)] $\alpha_0=b-\vep_1-\vep_2$, $\alpha_1=\vep_1-\vep_2$, $\alpha_2=\vep_2$. 
\item[ii)] $\alpha_2'=2\vep_2$.
\end{enumerate}
\item 
\begin{enumerate}
\item[i)] $k(\alpha_0)=2$, \; $k(\alpha_i)=1\; (i=1,2)$, 
\item[ii)] $k^{\nr}((\alpha_2')^\ast)=2$. 
\end{enumerate}
\item $m_{\alpha_0}=1$, \; $m_{\alpha_1}=m_{\alpha_2}=2$.
\item Elliptic diagram:
\input{Figure-new/diagram-20-nr}
\item $W$-orbits: $R(BB_2^{\vee (2)\ast})_l=W((\alpha_2')^\ast)$, \quad $R(BB_2^{\vee (2)\ast})_s=W(\alpha_2)$, \\
\phantom{$W$-orbits:\,} $R(BB_2^{\vee (2)\ast})_m=W(\alpha_0)\amalg W(\alpha_1) \amalg W(\alpha_1^\ast).$
\end{enumerate}
\subsubsection{Type $BB_l^{\vee (4)}$ $(l\geq 2)$}\label{data:BB-5}
\begin{enumerate}
\item $t_1(BB_l^{\vee (4)})=2,$ \quad $t_2(BB_l^{\vee (4)})=4,$ \quad $(BB_l^{\vee (4)})^\vee = BB_l^{\vee (1)}.$
\item $R(BB_l^{\vee (4)})=(R(BC_l)_s+\Z a+\Z b) \cup (R(BC_l)_m+2\Z a+\Z b)$ \\
\phantom{$R(BB_l^{\vee (4)})=$} $\cup (R(BC_l)_l+4\Z a+2\Z b)$. 
\item 
\begin{enumerate}
\item[i)] $\alpha_0=b-\vep_1-\vep_2$, $\alpha_i=\vep_i-\vep_{i+1}\; 
(1\leq i\leq l-1)$, $\alpha_l=a+\vep_l$. 
\item[ii)] $\alpha_l'=2\vep_l$.
\end{enumerate}
\item 
\begin{enumerate}
\item[i)] $k(\alpha_i)=2\; (0\leq i\leq l-1)$, \; $k(\alpha_l)=1$, 
\item[ii)] $k^{\nr}((\alpha_l')^\ast)=2$. 
\end{enumerate}
\item $m_{\alpha_0}=m_{\alpha_1}=1$, \; $m_{\alpha_i}=2\; 
(2\leq i \leq l)$.
\item Elliptic diagram:
\input{Figure-new/diagram-21-nr}
\item $W$-orbits:
\begin{enumerate}
\item[i)] $l=2$: $R(BB_2^{\vee (4)})_l=W((\alpha_2')^\ast)$, \\
\phantom{$l=2$:} $R(BB_2^{\vee (4)})_m=W(\alpha_0) \amalg W(\alpha_1)$, \\
\phantom{$l=2$:} $R(BB_2^{\vee (4)})_s=W(\alpha_2) \amalg W(\alpha_2^\ast)$.
\item[ii)] $l\geq 3$: $R(BB_l^{\vee (4)})_l=W((\alpha_l')^\ast)$, \\
 \phantom{$l\geq 3$:} $R(BB_l^{\vee (4)})_m$: single $W$-orbit, \\
\phantom{$l\geq 3$:}  $R(BB_l^{\vee (4)})_s=W(\alpha_l) \amalg W(\alpha_l^\ast)$.
\end{enumerate}
\end{enumerate}
\bigskip 
\subsection{Non-Reduced mERSs with $R/G$ of type $\mathbf{C^\vee C_l}$}
\label{sect_root-data-non-reduced-4}

\subsubsection{Type $C^\vee C_l^{(1)}$ $(l\geq 1)$}\label{data:CC-1}
\begin{enumerate}
\item $t_1(C^\vee C_l^{(1)})=2,$ \quad $t_2(C^\vee C_l^{(1)})=1,$ \quad $(C^\vee C_l^{(1)})^\vee = C^\vee C_l^{(4)}.$
\item $R(C^\vee C_l^{(1)})=(R(BC_l)_s+\Z a+\Z b) \cup (R(BC_l)_m+\Z a+2\Z b)$\\
\phantom{$R(C^\vee C_l^{(1)})=$}  $\cup (R(BC_l)_l+\Z a+2\Z b)$. 
\item 
\begin{enumerate}
\item[i)] $\alpha_0=b-\vep_1$, $\alpha_i=\vep_i-\vep_{i+1}\; 
(1\leq i \leq l-1)$, $\alpha_l=\vep_l$. 
\item[ii)] $\alpha_0'=2b-2\vep_1$, \; $\alpha_l'=2\vep_l$.
\end{enumerate}
\item 
\begin{enumerate}
\item[i)] $k(\alpha_i)=1\; (0\leq i\leq l)$, 
\item[ii)] $k^{\nr}((\alpha_0')^\ast)=k^{\nr}((\alpha_l')^\ast)=1$. 
\end{enumerate}
\item $m_{\alpha_i}=2\; (0\leq i \leq l)$. 
\item Elliptic diagram: 
\input{Figure-new/diagram-22-nr}
\item $W$-orbits: $R(C^\vee C_l^{(1)})_l=W(\alpha_0') \amalg W((\alpha_0')^\ast)
\amalg W(\alpha_l') \amalg W((\alpha_l')^\ast)$, \\
\phantom{$W$-orbits:}
$R(C^\vee C_l^{(1)})_m$: single $W$-orbit, \\
\phantom{$W$-orbits:}
$R(C^\vee C_l^{(1)})_s=W(\alpha_0) \amalg W(\alpha_l)$.
\end{enumerate}

\subsubsection{Type $C^\vee C_l^{(1)\ast_0}$ $(l\geq 1)$}\label{data:CC-2}
\begin{enumerate}
\item $t_1(C^\vee C_l^{(1)\ast_0})=2,$ \quad $t_2(C^\vee C_l^{(1)\ast_0})=1,$ \quad 
$(C^\vee C_l^{(1)\ast_0})^\vee = C^\vee C_l^{(4)\ast_0}.$
\item $R(C^\vee C_l^{(1)\ast_0})=(R(BC_l)_s+\Z a+\Z b) \cup 
(R(BC_l)_m+\Z a+2\Z b)$ \\
\phantom{$R(C^\vee C_l^{(1)\ast_0})=$} 
$\cup (R(BC_l)_l+\{ \,ma+2nb\, \vert\, mn \equiv 0\, [2]\,\})$. 
\item 
\begin{enumerate}
\item[i)] $\alpha_0=b-\vep_1$, $\alpha_i=\vep_i-\vep_{i+1}\; 
(1\leq i\leq l-1)$, $\alpha_l=\vep_l$. 
\item[ii)] $\alpha_0'=2b-2\vep_1$, \; $\alpha_l'=2\vep_l$.
\end{enumerate}
\item 
\begin{enumerate}
\item[i)] $k(\alpha_i)=1\; (0\leq i\leq l)$, 
\item[ii)] $k^{\nr}((\alpha_0')^\ast)=2, \; k^{\nr}((\alpha_l')^\ast)=1$. 
\end{enumerate}
\item $m_{\alpha_0}=1$, \; $m_{\alpha_i}=2\; (1\leq i \leq l)$. 
\item Elliptic diagram: 
\input{Figure-new/diagram-23-nr}
\item $W$-orbits: $R(C^\vee C_l^{(1)\ast_0})_l=W(\alpha_0') 
\amalg W(\alpha_l') \amalg W((\alpha_l')^\ast)$, \\
\phantom{$W$-orbits:}
$R(C^\vee C_l^{(1)\ast_0})_m$: single $W$-orbit, \\
\phantom{$W$-orbits:}
$R(C^\vee C_l^{(1)\ast_0})_s=W(\alpha_0) \amalg W(\alpha_l)$.
\end{enumerate}

\subsubsection{Type $C^\vee C_l^{(1)\ast_1}$ $(l\geq 1)$}\label{data:CC-3}
\begin{enumerate}
\item $t_1(C^\vee C_l^{(1)\ast_1})=2,$ \quad $t_2(C^\vee C_l^{(1)\ast_1})=1,$ \quad $(C^\vee C_l^{(1)\ast_1})^\vee = C^\vee C_l^{(4)\ast_1}.$
\item $R(C^\vee C_l^{(1)\ast_1})=(R(BC_l)_s+\Z a+\Z b) \cup (R(BC_l)_m+\Z a+2\Z b)$ \\
\phantom{$R(C^\vee C_l^{(1)\ast_1})=$} $\cup (R(BC_l)_l+\{ \,ma+2nb\, \vert\, (m-1)(n-1) \equiv 0\, [2]\,\})$. 
\item 
\begin{enumerate}
\item[i)] $\alpha_0=b-\vep_1$, $\alpha_i=\vep_i-\vep_{i+1}\; 
(1\leq i\leq l-1)$, $\alpha_l=\vep_l$. 
\item[ii)] $\alpha_0'=2b-2\vep_1$, \; $\alpha_l'= -a+2\vep_l$.
\end{enumerate}
\item 
\begin{enumerate}
\item[i)] $k(\alpha_i)=1\; (0\leq i\leq l)$, 
\item[ii)] $k^{\nr}((\alpha_0')^\ast) = k^{\nr}((\alpha_l')^\ast)=1$. 
\end{enumerate}
\item $m_{\alpha_i}=2\; (0 \leq i \leq l)$. 
\item Elliptic diagram: 
\input{Figure-new/diagram-24-nr}
\item $W$-orbits: $R(C^\vee C_l^{(1)\ast_1})_l=W(\alpha_0') \amalg W((\alpha_0')^\ast)
\amalg W((\alpha_l')^\ast)$, \\
\phantom{$W$-orbits:}
$R(C^\vee C_l^{(1)\ast_1})_m$: single $W$-orbit, \\
\phantom{$W$-orbits:}
$R(C^\vee C_l^{(1)\ast_1})_s=W(\alpha_0) \amalg W(\alpha_l)$.
\end{enumerate}

\subsubsection{Type $C^\vee C_l^{(2)}(1)$ $(l\geq 2)$}\label{data:CC-4}
\begin{enumerate}
\item $t_1(C^\vee C_l^{(2)}(1))=2,$ \quad $t_2(C^\vee C_l^{(2)}(1))=2,$ \quad $(C^\vee C_l^{(2)}(1))^\vee = C^\vee C_l^{(2)}(1).$
\item $R(C^\vee C_l^{(2)}(1))=(R(BC_l)_s+\Z a+\Z b) \cup (R(BC_l)_m+\Z a+2\Z b)$ \\
\phantom{$R(C^\vee C_l^{(2)}(1))=$} $\cup (R(BC_l)_l+2\Z a+2\Z b)$. 
\item 
\begin{enumerate}
\item[i)] $\alpha_0=b-\vep_1$, $\alpha_i=\vep_i-\vep_{i+1}\; 
(1\leq i\leq l-1)$, $\alpha_l=\vep_l$. 
\item[ii)] $\alpha_0'=2b-2\vep_1$, \; $\alpha_l'=2\vep_l$.
\end{enumerate}
\item 
\begin{enumerate}
\item[i)] $k(\alpha_i)=1\; (0\leq i\leq l)$, 
\item[ii)] $k^{\nr}((\alpha_0')^\ast)=k^{\nr}((\alpha_l')^\ast)=2$. 
\end{enumerate}
\item $m_{\alpha_0}=1$, \; $m_{\alpha_i}=2\; 
(1\leq i\leq l-1)$, \; $m_{\alpha_l}=1$. 
\item Elliptic diagram: 
\input{Figure-new/diagram-25-nr}
\item $W$-orbits:
\begin{enumerate}
\item[i)] $l=2$: $R(C^\vee C_2^{(2)}(1))_l=W(\alpha_0') \amalg W(\alpha_2')$, \\
\phantom{$l=2$:} $R(C^\vee C_2^{(2)}(1))_m=W(\alpha_1) \amalg W(\alpha_1^\ast)$, \\
\phantom{$l=2$:} $R(C^\vee C_2^{(2)}(1))_s=W(\alpha_0) \amalg W(\alpha_2)$. 
\item[ii)] $l\geq 3$: $R(C^\vee C_l^{(2)}(1))_l=W(\alpha_0') \amalg W(\alpha_l')$, \\ 
\phantom{$l\geq 3$:} $R(C^\vee C_l^{(2)}(1))_m$: single $W$-orbit, \\
\phantom{$l\geq 3$:} $R(C^\vee C_l^{(2)}(1))_s=W(\alpha_0) \amalg W(\alpha_l)$.
\end{enumerate}
\end{enumerate}

\subsubsection{Type $C^\vee C_l^{(2)}(2)$ $(l\geq 1)$}\label{data:CC-5}
\begin{enumerate}
\item $t_1(C^\vee C_l^{(2)}(2))=2,$ \quad $t_2(C^\vee C_l^{(2)}(2))=2,$ \quad $(C^\vee C_l^{(2)}(2))^\vee = C^\vee C_l^{(2)}(2).$
\item $R(C^\vee C_l^{(2)}(2))=(R(BC_l)_s+\Z a+\Z b) \cup (R(BC_l)_m+2\Z a+2\Z b)$ \\
\phantom{$R(C^\vee C_l^{(2)}(2))=$} $\cup (R(BC_l)_l+2\Z a+2\Z b)$. 
\item 
\begin{enumerate}
\item[i)] $\alpha_0=b-\vep_1$, $\alpha_i=\vep_i-\vep_{i+1}\; 
(1\leq i\leq l-1)$, $\alpha_l=\vep_l$.
\item[ii)] $\alpha_0'=2b-2\vep_1$, \; $\alpha_l'=2\vep_l$.
\end{enumerate} 
\item 
\begin{enumerate}
\item[i)] $k(\alpha_0)=1$, \; $k(\alpha_i)=2\; 
(1\leq i \leq l-1)$, \; $k(\alpha_l)=1$, 
\item[ii)] $k^{\nr}((\alpha_0')^\ast)=k^{\nr}((\alpha_l')^\ast)=2$. 
\end{enumerate}
\item $m_{\alpha_i}=1\; (0 \leq i \leq l)$. 
\item Elliptic diagram: 
\input{Figure-new/diagram-26-nr}
\item $W$-orbits: $R(C^\vee C_l^{(2)}(2))_l=W(\alpha_0') \amalg W((\alpha_0')^\ast) \amalg W(\alpha_l') \amalg W((\alpha_l')^\ast)$, \\
\phantom{$W$-orbits:} $R(C^\vee C_l^{(2)}(2))_m$: single $W$-orbit, \\
\phantom{$W$-orbits:} $R(C^\vee C_l^{(2)}(2))_s=W(\alpha_0) \amalg W(\alpha_0^\ast) \amalg W(\alpha_l) \amalg W(\alpha_l^\ast)$.
\end{enumerate}

\subsubsection{Type $C^\vee C_l^{(2)\ast_s}$ $(l\geq 1)$} \label{data:CC-6}
\begin{enumerate}
\item $t_1(C^\vee C_l^{(2)\ast_s})=2,$ \quad $t_2(C^\vee C_l^{(2)\ast_s})=2,$ \quad $(C^\vee C_l^{(2)\ast_s})^\vee = C^\vee C_l^{(2)\ast_l}.$
\item $R(C^\vee C_l^{(2)\ast_s})=(R(BC_l)_s+\{ \,ma+nb\, \vert\, mn \equiv 0\, [2]\,\})$ \\ 
\phantom{$R(C^\vee C_l^{(2)\ast_s})=$} $\cup (R(BC_l)_m+2\Z a+2\Z b) \cup (R(BC_l)_l+2\Z a+2\Z b)$. 
\item 
\begin{enumerate}
\item[i)] $\alpha_0=b-\vep_1$, $\alpha_i=\vep_i-\vep_{i+1}\; 
(1\leq i \leq l-1)$, $\alpha_l=\vep_l$. 
\item[ii)] $\alpha_0'=2b-2\vep_1$, \; $\alpha_l'=2\vep_l$.
\end{enumerate}
\item 
\begin{enumerate}
\item[i)] $k(\alpha_i)=2\; (0\leq i\leq l-1)$, \; $k(\alpha_l)=1$, 
\item[ii)] $k^{\nr}((\alpha_0')^\ast)=k^{\nr}((\alpha_l')^\ast)=2$. 
\end{enumerate}
\item $m_{\alpha_i}=1\; (0 \leq i \leq l)$. 
\item Elliptic diagram: 
\input{Figure-new/diagram-27-nr}
\item $W$-orbits: $R(C^\vee C_l^{(2)\ast_s})_l=W(\alpha_0') \amalg W((\alpha_0')^\ast) \amalg W(\alpha_l') \amalg W((\alpha_l')^\ast)$,
 \\
\phantom{$W$-orbits:} $R(C^\vee C_l^{(2)\ast_s})_m$: single $W$-orbit, \\
\phantom{$W$-orbits:} $R(C^\vee C_l^{(2)\ast_s})_s=W(\alpha_0) \amalg W(\alpha_l) \amalg W(\alpha_l^\ast)$.
\end{enumerate}

\subsubsection{Type $C^\vee C_l^{(2)\ast_l}$ $(l\geq 1)$}\label{data:CC-7}
\begin{enumerate}
\item $t_1(C^\vee C_l^{(2)\ast_l})=2,$ \quad $t_2(C^\vee C_l^{(2)\ast_l})=2,$ \quad $(C^\vee C_l^{(2)\ast_l})^\vee = C^\vee C_l^{(2)\ast_s}.$
\item $R(C^\vee C_l^{(2)\ast_l})=(R(BC_l)_s+\Z a+\Z b)\cup (R(BC_l)_m+2\Z a+2\Z b) $ \\ 
\phantom{$R(C^\vee C_l^{(2)\ast_l})=$} $\cup (R(BC_l)_l+\{ \,2ma+2nb\, \vert\, mn \equiv 0\, [2]\,\})$. 
\item 
\begin{enumerate}
\item[i)] $\alpha_0= a+b-\vep_1$, $\alpha_i=\vep_i-\vep_{i+1}\; 
(1\leq i \leq l-1)$, $\alpha_l=\vep_l$. 
\item[ii)] $\alpha_0'=2b-2\vep_1$, \; $\alpha_l'=2\vep_l$.
\end{enumerate}
\item 
\begin{enumerate}
\item[i)] $k(\alpha_0)=1$, \; $k(\alpha_i)=2\; 
(1\leq i \leq l-1)$, \; $k(\alpha_l)=1$, 
\item[ii)] $k^{\nr}((\alpha_0')^\ast)= k^{\nr}((\alpha_l')^\ast)=2$. 
\end{enumerate}
\item $m_{\alpha_i}=1\; (0 \leq i \leq l)$. 
\item Elliptic diagram: 
\input{Figure-new/diagram-28-nr}
\item $W$-orbits: $R(C^\vee C_l^{(2)\ast_l})_l=W((\alpha_0')^\ast) \amalg W(\alpha_l') \amalg W((\alpha_l')^\ast)$,  \\
\phantom{$W$-orbits:} $R(C^\vee C_l^{(2)\ast_l})_m$: single $W$-orbit, \\
\phantom{$W$-orbits:} $R(C^\vee C_l^{(2)\ast_l})_s=W(\alpha_0) \amalg W(\alpha_0^\ast) \amalg W(\alpha_l) \amalg W(\alpha_l^\ast)$.
\end{enumerate}

\subsubsection{Type $C^\vee C_l^{(2)\ast_{0}}$ $(l\geq 1)$}\label{data:CC-8}
\begin{enumerate}
\item $t_1(C^\vee C_l^{(2)\ast_{0}})=2,$ \quad $t_2(C^\vee C_l^{(2)\ast_{0}})=2,$ \quad $(C^\vee C_l^{(2)\ast_{0}})^\vee = C^\vee C_l^{(2)\ast_{0}}.$
\item $R(C^\vee C_l^{(2)\ast_{0}})=(R(BC_l)_s+\{ \,ma+nb\, \vert\, mn \equiv 0\, [2]\,\})$ \\ 
\phantom{$R(C^\vee C_l^{(2)\ast_{0}})=$} $\cup (R(BC_l)_m+2\Z a+2\Z b)$ \\
\phantom{$R(C^\vee C_l^{(2)\ast_{0}})=$} $\cup (R(BC_l)_l+\{ \,2ma+2nb\, \vert\, mn \equiv 0\, [2]\,\})$. 
\item 
\begin{enumerate}
\item[i)] $\alpha_0=b-\vep_1$, $\alpha_i=\vep_i-\vep_{i+1}\; 
(1\leq i \leq l-1)$, $\alpha_l=\vep_l$. 
\item[ii)] $\alpha_0'=2b-2\vep_1$, \; $\alpha_l'=2\vep_l$.
\end{enumerate}
\item 
\begin{enumerate}
\item[i)] $k(\alpha_i)=2\; (0\leq  i \leq l-1)$, \; $k(\alpha_l)=1$,
\item[ii)] $k^{\nr}((\alpha_0')^\ast)=4$, \; $k^{\nr}((\alpha_l')^\ast)=2$. 
\end{enumerate}
\item $m_{\alpha_0}=\frac{1}{2}$, \; $m_{\alpha_i}=1\; (1\leq i \leq l)$. 
\item Elliptic diagram: 
\input{Figure-new/diagram-29-nr}
\item $W$-orbits: $R(C^\vee C_l^{(2)\ast_{0}})_l=W(\alpha_0') \amalg W(\alpha_l') \amalg W((\alpha_l')^\ast)$, \\
\phantom{$W$-orbits:} $R(C^\vee C_l^{(2)\ast_{0}})_m$: single $W$-orbit,  \\ 
\phantom{$W$-orbits:} $R(C^\vee C_l^{(2)\ast_{0}})_s=W(\alpha_0) \amalg W(\alpha_l) \amalg W(\alpha_l^\ast)$.
\end{enumerate}

\subsubsection{Type $C^\vee C_l^{(2)\ast_{1'}}$ $(l\geq 1)$}\label{data:CC-9}
\begin{enumerate}
\item $t_1(C^\vee C_l^{(2)\ast_{1'}})=2,$ \quad $t_2(C^\vee C_l^{(2)\ast_{1'}})=2,$ \quad $(C^\vee C_l^{(2)\ast_{1'}})^\vee = C^\vee C_l^{(2)\ast_{1'}}.$
\item  $R(C^\vee C_l^{(2)\ast_{1'}})=(R(BC_l)_s+\{ \,ma+nb\, \vert\, mn \equiv 0\, [2]\,\})$ \\
\phantom{$R(C^\vee C_l^{(2)\ast_{1'}})=$} $\cup (R(BC_l)_m+2\Z a+2\Z b)$ \\
\phantom{$R(C^\vee C_l^{(2)\ast_{1'}})=$} $\cup (R(BC_l)_l+\{ \,2ma+2nb\, \vert\, (m-1)n \equiv 0\, [2]\,\})$. 
\item 
\begin{enumerate}
\item[i)] $\alpha_0=b-\vep_1$, $\alpha_i=\vep_i-\vep_{i+1}\; 
(1\leq i\leq l-1)$, $\alpha_l=\vep_l$. 
\item[ii)] $\alpha_0'= -2a+2b-2\vep_1$, \; $\alpha_l'=2\vep_l$.
\end{enumerate} 
\item 
\begin{enumerate}
\item[i)] $k(\alpha_i)=2\; (0\leq  i \leq l-1)$, \; $k(\alpha_l)=1$, 
\item[ii)] $k^{\nr}((\alpha_0')^\ast)=k^{\nr}((\alpha_l')^\ast)=2$. 
\end{enumerate}
\item $m_{\alpha_i}=1\; (0 \leq i \leq l)$.  
\item Elliptic diagram: 
\input{Figure-new/diagram-30-nr}
\item $W$-orbits: $R(C^\vee C_l^{(2)\ast_{1'}})_l=W((\alpha_0')^\ast) \amalg W(\alpha_l') \amalg W((\alpha_l')^\ast)$, \\
\phantom{$W$-orbits:} $R(C^\vee C_l^{(2)\ast_{1'}})_m$: single $W$-orbit, \\
\phantom{$W$-orbits:} $R(C^\vee C_l^{(2)\ast_{1'}})_s=W(\alpha_0) \amalg W(\alpha_l) \amalg W(\alpha_l^\ast)$.
\end{enumerate}

\subsubsection{Type $C^\vee C_l^{(2)\ast_{1}}$ $(l\geq 1)$}\label{data:CC-10}
\begin{enumerate}
\item $t_1(C^\vee C_l^{(2)\ast_{1}})=2,$ \quad $t_2(C^\vee C_l^{(2)\ast_{1}})=2,$ \quad $(C^\vee C_l^{(2)\ast_{1}})^\vee = C^\vee C_l^{(2)\ast_{1}}.$
\item $R(C^\vee C_l^{(2)\ast_{1}})=(R(BC_l)_s+\{ \,ma+nb\, \vert\, mn \equiv 0\, [2]\,\})$ \\
\phantom{$R(C^\vee C_l^{(2)\ast_{1}})=$} $\cup (R(BC_l)_m+2\Z a+2\Z b)$ \\
\phantom{$R(C^\vee C_l^{(2)\ast_{1}})=$} $\cup (R(BC_l)_l+\{ \,2ma+2nb\, \vert\, (m-1)(n-1) \equiv 0\, [2]\,\})$. 
\item 
\begin{enumerate}
\item[i)] $\alpha_0=b-\vep_1$, $\alpha_i=\vep_i-\vep_{i+1}\; 
(1\leq i \leq l-1)$, $\alpha_l=\vep_l$. 
\item[ii)] $\alpha_0'=2b-2\vep_1$, \; $\alpha_l'= -2a+2\vep_l$.
\end{enumerate}
\item 
\begin{enumerate}
\item[i)] $k(\alpha_i)=2\; (0\leq  i \leq l-1)$, \; $k(\alpha_l)=1$, 
\item[ii)] $k^{\nr}((\alpha_0')^\ast)=k^{\nr}((\alpha_l')^\ast)=2$. 
\end{enumerate}
\item $m_{\alpha_i}=1\; (0 \leq i \leq l)$. 
\item Elliptic diagram: 
\input{Figure-new/diagram-31-nr}
\item $W$-orbits: $R(C^\vee C_l^{(2)\ast_{1}})_l=W(\alpha_0') \amalg W((\alpha_0')^\ast) \amalg W((\alpha_l')^\ast)$, \\
\phantom{$W$-orbits:} $R(C^\vee C_l^{(2)\ast_{1}})_m$: single $W$-orbit, \\
\phantom{$W$-orbits:} $R(C^\vee C_l^{(2)\ast_{1}})_s=W(\alpha_0) \amalg W(\alpha_l) \amalg W(\alpha_l^\ast)$.
\end{enumerate}

\subsubsection{Type $C^\vee C_l^{(2)\diamond}$ $(l\geq 1)$}\label{data:CC-11}
\label{subsubsec:diamond}
\begin{enumerate}
\item $t_1(C^\vee C_l^{(2)\diamond})=2,$ \quad $t_2(C^\vee C_l^{(2)\diamond})=2,$ \quad $(C^\vee C_l^{(2)\diamond})^\vee = C^\vee C_l^{(2)\diamond}.$
\item $R(C^\vee C_l^{(2)\diamond})=(R(BC_l)_s+\Z a+\Z b) \cup (R(BC_l)_m+\Z a+2\Z b)$ \\
\phantom{$R(C^\vee C_l^{(2)\diamond})=$} $\cup (R(BC_l)_l+\{ \,ma+2nb\, \vert\, m-n \equiv 0\, [2]\,\})$. 
\item 
\begin{enumerate}
\item[i)] $\alpha_0=b-\vep_1$, $\alpha_i=\vep_i-\vep_{i+1}\; 
(1\leq i \leq l-1)$, $\alpha_l=\vep_l$. 
\item[ii)] $\alpha_0'= -a+2b-2\vep_1$, \; $\alpha_l'=2\vep_l$.
\end{enumerate}
\item 
\begin{enumerate}
\item[i)] $k(\alpha_i)=1\; (0\leq i\leq l)$, 
\item[ii)] $k^{\nr}((\alpha_0')^\ast)=1$, \; $k^{\nr}((\alpha_l')^\ast)=2$. 
\end{enumerate}
\item $m_{\alpha_i}=2\; (0 \leq i \leq l-1)$, \;  $m_{\alpha_l}=1$. 
\item Elliptic diagram: 
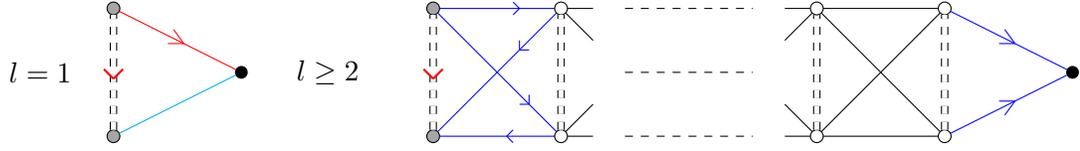
\input{Figure-new/diagram-32-nr}
\item $W$-orbits: $R(C^\vee C_l^{(2)\diamond})_l=W((\alpha_0')^\ast) \amalg W(\alpha_l')$, \\
\phantom{$W$-orbits:}
$R(C^\vee C_l^{(2)\diamond})_m$: single $W$-orbit, \\
\phantom{$W$-orbits:}
$R(C^\vee C_l^{(2)\diamond})_s=W(\alpha_0) \amalg W(\alpha_l)$.
\end{enumerate}

\subsubsection{Type $C^\vee C_l^{(4)}$ $(l\geq 1)$}\label{data:CC-12}
\begin{enumerate}
\item $t_1(C^\vee C_l^{(4)})=2,$ \quad $t_2(C^\vee C_l^{(4)})=4,$ \quad $(C^\vee C_l^{(4)})^\vee = C^\vee C_l^{(1)}.$
\item $R(C^\vee C_l^{(4)})=(R(BC_l)_s+\Z a+\Z b) \cup (R(BC_l)_m+2\Z a+2\Z b)$ \\
\phantom{$R(C^\vee C_l^{(4)})=$} $\cup (R(BC_l)_l+4\Z a+2\Z b)$. 
\item 
\begin{enumerate}
\item[i)] $\alpha_0=a+b-\vep_1$,  \; $\alpha_i=\vep_i-\vep_{i+1}\; 
(1\leq i \leq l-1)$, $\alpha_l= a+\vep_l$. 
\item[ii)] $\alpha_0'=2b-2\vep_1$, \; $\alpha_l'=2\vep_l$.
\end{enumerate}
\item 
\begin{enumerate}
\item[i)] $k(\alpha_0)=1$, \; $k(\alpha_i)=2\; 
(1\leq i \leq l-1)$, \; $k(\alpha_l)=1$, 
\item[ii)] $k^{\nr}((\alpha_0')^\ast)=k^{\nr}((\alpha_l')^\ast)=2$. 
\end{enumerate}
\item $m_{\alpha_i}=1\; (0 \leq i \leq l)$. 
\item Elliptic diagram: 
\input{Figure-new/diagram-33-nr}
\item $W$-orbits: $R(C^\vee C_l^{(4)})_l=W((\alpha_0')^\ast) \amalg W((\alpha_l')^\ast)$, \\
\phantom{$W$-orbits:} $R(C^\vee C_l^{(4)})_m$: single $W$-orbit, \\
\phantom{$W$-orbits:} $R(C^\vee C_l^{(4)})_s=W(\alpha_0) \amalg W(\alpha_0^\ast) \amalg W(\alpha_l) \amalg W(\alpha_l^\ast)$.
\end{enumerate}

\subsubsection{Type $C^\vee C_l^{(4)\ast_0}$ $(l\geq 1)$}\label{data:CC-13}
\begin{enumerate}
\item $t_1(C^\vee C_l^{(4)\ast_0})=2,$ \quad $t_2(C^\vee C_l^{(4)\ast_0})=4,$ \quad $(C^\vee C_l^{(4)\ast_0})^\vee = C^\vee C_l^{(1)\ast_0}.$
\item $R(C^\vee C_l^{(4)\ast_{0}})=(R(BC_l)_s+\{ \,ma+nb\, \vert\, mn \equiv 0\, [2]\,\})$ \\
\phantom{$R(C^\vee C_l^{(4)\ast_{0}})=$} $ \cup (R(BC_l)_m+2\Z a+2\Z b) \cup (R(BC_l)_l+4\Z a+2\Z b)$. 
\item 
\begin{enumerate}
\item[i)] $\alpha_0=b-\vep_1$, $\alpha_i=\vep_i-\vep_{i+1}\; 
(1\leq i \leq l-1)$, $\alpha_l=a+\vep_l$. 
\item[ii)] $\alpha_0'=2b-2\vep_1$, \; $\alpha_l'=2\vep_l$.
\end{enumerate}
\item 
\begin{enumerate}
\item[i)] $k(\alpha_i)=2\; (0\leq i\leq l-1)$, \; $k(\alpha_l)=1$, 
\item[ii)] $k^{\nr}((\alpha_0')^\ast)=4$, \; $k^{\nr}((\alpha_l')^\ast)=2$. 
\end{enumerate}
\item $m_{\alpha_0}=\frac{1}{2}$, \; $m_{\alpha_i}=1\; (1\leq i \leq l)$. 
\item Elliptic diagram: 
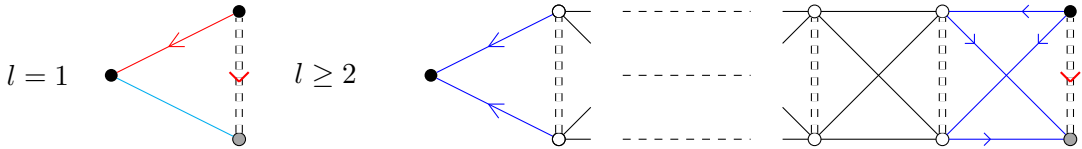
\input{Figure-new/diagram-34-nr}
\item $W$-orbits: $R(C^\vee C_l^{(4)\ast_0})_l=W(\alpha_0') \amalg W((\alpha_l')^\ast)$, \\
\phantom{$W$-orbits:}  $R(C^\vee C_l^{(4)\ast_0})_m$: single $W$-orbit, \\
\phantom{$W$-orbits:} $R(C^\vee C_l^{(4)\ast_0})_s=W(\alpha_0) \amalg W(\alpha_l) \amalg W(\alpha_l^\ast)$.
\end{enumerate}

\subsubsection{Type $C^\vee C_l^{(4)\ast_1}$ $(l\geq 1)$}\label{data:CC-14}
\begin{enumerate}
\item $t_1(C^\vee C_l^{(4)\ast_1})=2,$ \quad $t_2(C^\vee C_l^{(4)\ast_1})=4,$ \quad $(C^\vee C_l^{(4)\ast_1})^\vee = C^\vee C_l^{(1)\ast_1}.$
\item $R(C^\vee C_l^{(4)\ast_{1}})=(R(BC_l)_s+\{ \,ma+nb\, \vert\, (m-1)(n-1) \equiv 0\, [2]\,\})$ \\
\phantom{$R(C^\vee C_l^{(4)\ast_{1}})=$} $ \cup (R(BC_l)_m+2\Z a+2\Z b) \cup (R(BC_l)_l+4\Z a+2\Z b)$. 
\item 
\begin{enumerate}
\item[i)] $\alpha_0= a+b-\vep_1$, $\alpha_i=\vep_i-\vep_{i+1}\; 
(1\leq i \leq l-1)$, $\alpha_l=a+\vep_l$. 
\item[ii)] $\alpha_0'=2b-2\vep_1$, \; $\alpha_l'=2\vep_l$.
\end{enumerate}
\item 
\begin{enumerate}
\item[i)]  $k(\alpha_0)=1$, \; $k(\alpha_i)=2\; (1\leq i \leq l)$, 
\item[ii)] $k^{\nr}((\alpha_0')^\ast)=k^{\nr}((\alpha_l')^\ast)=2$. 
\end{enumerate}
\item $m_{\alpha_i}=1\; (0 \leq i \leq l)$. 
\item Elliptic diagram: 
\input{Figure-new/diagram-35-nr}
\item $W$-orbits: $R(C^\vee C_l^{(4)\ast_1})_l=W((\alpha_0')^\ast) \amalg W((\alpha_l')^\ast)$, \\
\phantom{$W$-orbits:} $R(C^\vee C_l^{(4)\ast_1})_m$: single $W$-orbit, \\
\phantom{$W$-orbits:} $R(C^\vee C_l^{(4)\ast_1})_s=W(\alpha_0) \amalg W(\alpha_0^\ast) \amalg W(\alpha_l)$.
\end{enumerate}

\setcounter{section}{0}
\renewcommand{\thesection}{\Alph{section}}

\part*{Appendix}

In this appendix, for the sake of reader's convenience, we collect some root data of finite and affine root systems.
In addition, we also provide some root data of a marked elliptic root system $(R,G)$ 
whose affine quotient is reduced.
For an affine and elliptic root system $R$, we also describe the decomposition of $R$ into $W(R)$-orbits. Such information might be useful in applications. 

\section{Finite root systems} \label{sect_finite-root-system}
In this section, we provide necessary informations on finite root systems, including the unique non-reduced one. 

Each root system of rank $l$ is a subset of the vector space $\R^l$ with the scalar product $I$, and let ${\vep_i, 1\leq i \leq l}$ be an orthonormal basis with respect to $I$.
\subsection{Reduced types}
\subsubsection{$A_l$\; $(l\geq 1)$}\label{sect_A}
\begin{enumerate}
\item Roots: 
\[ R(A_l)=\{\pm (\vep_i- \vep_j)\, \vert 1\leq i<j\leq l\}. \]
\item Simple roots: \; $\alpha_i=\vep_i-\vep_{i+1}\; (1\leq i\leq l)$. 
\item Highest root: $\theta=\vep_1-\vep_{l+1}=\sum_{i=1}^l \alpha_i$.
\item Dynkin diagram:
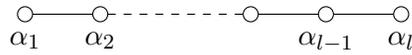
\begin{center}
\scalebox{0.9}{
\begin{tikzpicture}
\draw (1,0) circle (0.1);  \draw (1,-0.1) node[below] {$\alpha_1$};
\draw (1.1,0) -- (1.9,0); 
\draw (2,0) circle (0.1);  \draw (2,-0.1) node[below] {$\alpha_2$};
\draw [dashed] (2.1,0) -- (3.9,0);
\draw (4,0) circle (0.1); 
\draw (4.1,0) -- (4.9,0); 
\draw (5,0) circle (0.1); \draw (5,-0.1) node[below] {$\alpha_{l-1}$};
\draw (5.1, 0) -- (5.9,0); 
\draw (6,0) circle (0.1);  \draw (6,-0.1) node[below] {$\alpha_{l}$};
\end{tikzpicture}}
\end{center}
\end{enumerate}

\subsubsection{$B_l$\; $(l\geq 2)$}\label{sect_B}
\begin{enumerate}
\item Roots: 
\[ R(B_l)=\{\pm\vep_i\pm \vep_j\, \vert 1\leq i<j\leq l\} 
\cup \{\pm\vep_i \, \vert\, 1\leq i\leq l\}, \]
\item Simple roots: \; $\alpha_i=\vep_i-\vep_{i+1}\; (1\leq i<l)$ and $\alpha_l=\vep_l$. 
\item Highest root: $\theta=\vep_1+\vep_{2}=\alpha_1+2(\alpha_2+\cdots+\alpha_l)$.
\item Dynkin diagram:
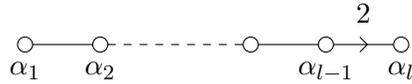
\begin{center}
\scalebox{0.9}{
\begin{tikzpicture}
\draw (1,0) circle (0.1);  \draw (1,-0.1) node[below] {$\alpha_1$};
\draw (1.1,0) -- (1.9,0); 
\draw (2,0) circle (0.1);  \draw (2,-0.1) node[below] {$\alpha_2$};
\draw [dashed] (2.1,0) -- (3.9,0);
\draw (4,0) circle (0.1); 
\draw (4.1,0) -- (4.9,0); 
\draw (5,0) circle (0.1); \draw (5,-0.1) node[below] {$\alpha_{l-1}$};
\draw (5.1, 0) -- (5.9,0); 
\draw (6,0) circle (0.1);  \draw (6,-0.1) node[below] {$\alpha_{l}$};
\draw (5.55,0) -- (5.45,0.1) ;
\draw (5.55,0) -- (5.45,-0.1) ;
\draw (5.5,0.15) node[above] {$2$};
\end{tikzpicture}}
\end{center}
\end{enumerate}

\subsubsection{$C_l$\; $(l\geq 2)$}\label{sect_C}
\begin{enumerate}
\item Roots: 
\[ R(C_l)=\{\pm\vep_i\pm \vep_j\, \vert 1\leq i<j\leq l\} 
\cup \{\pm2\vep_i\, \vert\, 1\leq i\leq l\}, \]
\item Simple roots: \; $\alpha_i=\vep_i-\vep_{i+1}\; (1\leq i<l)$ and $\alpha_l=2\vep_l$. 
\item Highest root: $\theta=2\vep_1=2(\alpha_1+\cdots+\alpha_{l-1})+\alpha_l$.
\item Dynkin diagram:
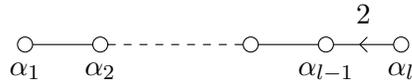
\begin{center}
\scalebox{0.9}{
\begin{tikzpicture}
\draw (1,0) circle (0.1);  \draw (1,-0.1) node[below] {$\alpha_1$};
\draw (1.1,0) -- (1.9,0); 
\draw (2,0) circle (0.1);  \draw (2,-0.1) node[below] {$\alpha_2$};
\draw [dashed] (2.1,0) -- (3.9,0);
\draw (4,0) circle (0.1); 
\draw (4.1,0) -- (4.9,0); 
\draw (5,0) circle (0.1); \draw (5,-0.1) node[below] {$\alpha_{l-1}$};
\draw (5.1, 0) -- (5.9,0); 
\draw (6,0) circle (0.1);  \draw (6,-0.1) node[below] {$\alpha_{l}$};
\draw (5.45,0) -- (5.55,0.1) ;
\draw (5.45,0) -- (5.55,-0.1) ;
\draw (5.5,0.15) node[above] {$2$};
\end{tikzpicture}}
\end{center}
\end{enumerate}

\subsubsection{$D_l$\; $(l\geq 3)$}\label{sect_D}
\begin{enumerate}
\item Roots: 
\[ R(D_l)=\{\pm\vep_i\pm \vep_j\, \vert 1\leq i<j\leq l\}, \]
\item Simple roots: \; $\alpha_i=\vep_i-\vep_{i+1}\; (1\leq i<l)$ and $\alpha_l=\vep_{l-1}+\vep_l$. 
\item Highest root: $\theta=\vep_1+\vep_2=\alpha_1+2(\alpha_1+\cdots+\alpha_{l-2})+\alpha_{l-1}+\alpha_l$.
\item Dynkin diagram:
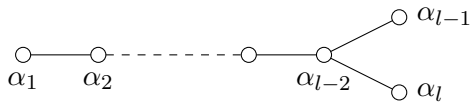
\begin{center}
\scalebox{0.9}{
\begin{tikzpicture}
\draw (1,0) circle (0.1);  \draw (1,-0.1) node[below] {$\alpha_1$};
\draw (1.1,0) -- (1.9,0); 
\draw (2,0) circle (0.1);  \draw (2,-0.1) node[below] {$\alpha_2$};
\draw [dashed] (2.1,0) -- (3.9,0);
\draw (4,0) circle (0.1); 
\draw (4.1,0) -- (4.9,0); 
\draw (5,0) circle (0.1); \draw (5,-0.1) node[below] {$\alpha_{l-2}$};
\draw (5.09,0.045) -- (5.91,0.455);
\draw (5.09,-0.045) -- (5.91,-0.455);
\draw (6,0.5) circle (0.1);  \draw (6.1,0.5) node[right] {$\alpha_{l-1}$};
\draw (6,-0.5) circle (0.1);  \draw (6.1,-0.5) node[right] {$\alpha_{l}$};
\end{tikzpicture}}
\end{center}
\end{enumerate}

\subsubsection{$E_6$}\label{sect_E6}
\begin{enumerate}
\item Roots: 
\begin{align*}
R(E_6)=
&\{\pm\vep_i\pm \vep_j\, \vert 1\leq i<j\leq 5\}  \\
&\cup
\left\{ \frac{1}{2}\left(\vep_8-\vep_7-\vep_6+\sum_{i=1}^5(-1)^{\nu(i)}\vep_i\right) \, 
\left\vert \, \sum_{i=1}^5 \nu(i):\, \text{even}\, \right\}, \right.
\end{align*}
\item Simple roots: \; $\alpha_1=\frac{1}{2}(\vep_1+\vep_8)-\frac{1}{2}(\vep_2+\vep_3+\cdots+\vep_7), \; \alpha_2=\vep_1+\vep_2, $ \\
\phantom{simple r roots \; }$\alpha_i=\vep_{i-1}-\vep_{i-2}\quad (3\leq i\leq 6)$.
\item Highest root: 
\begin{align*}
\theta=
&\frac{1}{2}(\vep_1+\vep_2+\vep_3+\vep_4+\vep_5-\vep_6-\vep_7+\vep_8) \\
=
&\alpha_1+2\alpha_2+2\alpha_3+3\alpha_4+2\alpha_5+\alpha_6.
\end{align*}
\item Dynkin diagram:
\begin{center}
\scalebox{0.9}{
\begin{tikzpicture}
\draw (1,0) circle (0.1);  \draw (1,-0.1) node[below] {$\alpha_1$};
\draw (1.1,0) -- (1.9,0); 
\draw (2,0) circle (0.1);  \draw (2,-0.1) node[below] {$\alpha_3$};
\draw (2.1, 0) -- (2.9,0);
\draw (3,0) circle (0.1); \draw (3,-0.1) node[below] {$\alpha_4$};
\draw (3.1,0) -- (3.9,0); 
\draw (4,0) circle (0.1); \draw (4,-0.1) node[below] {$\alpha_5$};
\draw (4.1,0)-- (4.9,0);
\draw (5,0) circle (0.1); \draw (5,-0.1) node[below] {$\alpha_6$};
\draw (3,0.1) -- (3,0.9); \draw (3,1) circle (0.1); 
\draw (3,1.1) node[above] {$\alpha_2$};
\end{tikzpicture}}
\end{center}
\end{enumerate}

\subsubsection{$E_7$}\label{sect_E7}

\begin{enumerate}
\item Roots: 
\begin{align*}
R(E_7)=
&\{\pm\vep_i\pm \vep_j\, \vert 1\leq i<j\leq 6\} \cup \{ \pm (\vep_7-\vep_8)\}  \\
&
\cup \left\{ \frac{1}{2}\left(\vep_8-\vep_7+\sum_{i=1}^6(-1)^{\nu(i)}\vep_i\right) \, 
\left\vert \, \sum_{i=1}^6 \nu(i):\, \text{odd}\, \right\}, \right.
\end{align*}
\item Simple roots: \; $\alpha_1=\frac{1}{2}(\vep_1+\vep_8)-\frac{1}{2}(\vep_2+\vep_3+\cdots+\vep_7), \; \alpha_2=\vep_1+\vep_2, $ \\
\phantom{simple r roots \; }$\alpha_i=\vep_{i-1}-\vep_{i-2}\quad (3\leq i\leq 7)$.
\item Highest root: 
$
\theta=\vep_8-\vep_7
=
2\alpha_1+2\alpha_2+3\alpha_3+4\alpha_4+3\alpha_5+2\alpha_6+\alpha_7.
$
\item Dynkin diagram:
\begin{center}
\scalebox{0.9}{
\begin{tikzpicture}
\draw (1,0) circle (0.1);  \draw (1,-0.1) node[below] {$\alpha_1$};
\draw (1.1,0) -- (1.9,0); 
\draw (2,0) circle (0.1);  \draw (2,-0.1) node[below] {$\alpha_3$};
\draw (2.1, 0) -- (2.9,0);
\draw (3,0) circle (0.1); \draw (3,-0.1) node[below] {$\alpha_4$};
\draw (3.1,0) -- (3.9,0); 
\draw (4,0) circle (0.1); \draw (4,-0.1) node[below] {$\alpha_5$};
\draw (4.1,0)-- (4.9,0);
\draw (5,0) circle (0.1); \draw (5,-0.1) node[below] {$\alpha_6$};
\draw (5.1,0)-- (5.9,0);
\draw (6,0) circle (0.1); \draw (6,-0.1) node[below] {$\alpha_7$};
\draw (3,0.1) -- (3,0.9); \draw (3,1) circle (0.1); 
\draw (3,1.1) node[above] {$\alpha_2$};
\end{tikzpicture}}
\end{center}
\end{enumerate}

\subsubsection{$E_8$}\label{sect_E8}

\begin{enumerate}
\item Roots: 
\begin{align*}
R(E_8)=
&\{\pm\vep_i\pm \vep_j\, \vert 1\leq i<j\leq 8\} \\
&\cup  
\left\{ \frac{1}{2}\sum_{i=1}^8(-1)^{\nu(i)}\vep_i \, 
\left\vert \, \sum_{i=1}^8 \nu(i):\, \text{even}\, \right\}. \right.
\end{align*}
\item Simple roots: \; $\alpha_1=\frac{1}{2}(\vep_1+\vep_8)-\frac{1}{2}(\vep_2+\vep_3+\cdots+\vep_7), \; \alpha_2=\vep_1+\vep_2, $ \\
\phantom{simple r roots \; }$\alpha_i=\vep_{i-1}-\vep_{i-2}\quad (3\leq i\leq 8)$.
\item Highest root: 
$
\theta=\vep_7+\vep_8
=
2\alpha_1+3\alpha_2+4\alpha_3+6\alpha_4+5\alpha_5+4\alpha_6+3\alpha_7+2\alpha_8.
$
\item Dynkin diagram:
\begin{center}
\scalebox{0.9}{
\begin{tikzpicture}
\draw (1,0) circle (0.1);  \draw (1,-0.1) node[below] {$\alpha_1$};
\draw (1.1,0) -- (1.9,0); 
\draw (2,0) circle (0.1);  \draw (2,-0.1) node[below] {$\alpha_3$};
\draw (2.1, 0) -- (2.9,0);
\draw (3,0) circle (0.1); \draw (3,-0.1) node[below] {$\alpha_4$};
\draw (3.1,0) -- (3.9,0); 
\draw (4,0) circle (0.1); \draw (4,-0.1) node[below] {$\alpha_5$};
\draw (4.1,0)-- (4.9,0);
\draw (5,0) circle (0.1); \draw (5,-0.1) node[below] {$\alpha_6$};
\draw (5.1,0)-- (5.9,0);
\draw (6,0) circle (0.1); \draw (6,-0.1) node[below] {$\alpha_7$};
\draw (6.1,0)-- (6.9,0);
\draw (7,0) circle (0.1); \draw (7,-0.1) node[below] {$\alpha_8$};
\draw (3,0.1) -- (3,0.9); \draw (3,1) circle (0.1); 
\draw (3,1.1) node[above] {$\alpha_2$};
\end{tikzpicture}}
\end{center}
\end{enumerate}

\subsubsection{$F_4$}\label{sect_F4}

\begin{enumerate}
\item Roots: 
\begin{align*}
R(F_4)
=
&
\{\pm\vep_i\, \vert\, 1\leq i\leq 4\, \} \cup \{\pm\vep_i\pm \vep_j\, \vert 1\leq i<j\leq 4\} \\
&\cup  \left\{ \frac{1}{2}(\pm \vep_1\pm \vep_2\pm \vep_3\pm \vep_4) \right\}. 
\end{align*}
\item Simple roots: \; $\alpha_1=\vep_2-\vep_3, \; \alpha_2=\vep_3-\vep_4, \; \alpha_3=\vep_4$, \\ 
\phantom{Simple roots: \;\,} $\alpha_4=\frac{1}{2}(\vep_1-\vep_2-\vep_3-\vep_4)$. 
\item Highest root: 
$
\theta=\vep_1+\vep_2
=
2\alpha_1+3\alpha_2+4\alpha_3+2\alpha_4.
$
\item Dynkin diagram:
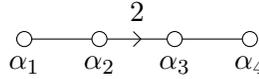
\begin{center}
\scalebox{0.9}{
\begin{tikzpicture}
\draw (1,0) circle (0.1);  \draw (1,-0.1) node[below] {$\alpha_1$};
\draw (1.1,0) -- (1.9,0); 
\draw (2,0) circle (0.1);  \draw (2,-0.1) node[below] {$\alpha_2$};
\draw (2.1, 0) -- (2.9,0);
\draw (3,0) circle (0.1); \draw (3,-0.1) node[below] {$\alpha_3$};
\draw (3.1,0) -- (3.9,0); 
\draw (4,0) circle (0.1); \draw (4,-0.1) node[below] {$\alpha_4$};
\draw (2.55,0) -- (2.45,0.1);
\draw (2.55,0) -- (2.45,-0.1); \draw (2.5,0.1) node[above] {$2$};
\end{tikzpicture}}
\end{center}
\end{enumerate}

\subsubsection{$G_2$}\label{sect_G2}

\begin{enumerate}
\item Roots: 
\begin{align*}
R(G_2)=
&\{\pm(\vep_i- \vep_j)\, \vert 1\leq i<j\leq 3\} \\
&\cup  
\left\{ \left. \pm \left(3\vep_i -\sum_{k=1}^3 \vep_k\right) \, 
\right\vert \, 1\leq i\leq 3 \, \right\}. 
\end{align*}
\item Simple roots: \; $\alpha_1=\vep_1-\vep_2, \; \alpha_2=-2\vep_1+\vep_2+\vep_3.$ \item Highest root: 
$
\theta=-\vep_1-\vep_2+2\vep_3
=
3\alpha_1+2\alpha_2.
$
\item Dynkin diagram:
\begin{center}
\scalebox{0.9}{
\begin{tikzpicture}
\draw (1,0) circle (0.1);  \draw (1,-0.1) node[below] {$\alpha_1$};
\draw (1.1,0) -- (1.9,0); 
\draw (2,0) circle (0.1);  \draw (2,-0.1) node[below] {$\alpha_2$};
\draw (1.45,0) -- (1.55,0.1);
\draw (1.45,0) -- (1.55,-0.1); \draw (1.5,0.1) node[above] {$3$};
\end{tikzpicture}}
\end{center}
\end{enumerate}

\subsection{Non-reduced type}
\subsubsection{$BC_l$\; $(l\geq 1)$}\label{sect_BC}

\begin{enumerate}
\item Roots: 
\[ R(BC_l)=\{\pm\vep_i\pm \vep_j\, \vert 1\leq i<j\leq l\} 
\cup \{\pm\vep_i, \; \pm2\vep_i\, \vert\, 1\leq i\leq l\}, \]
\item Simple roots: \; $\alpha_i=\vep_i-\vep_{i+1}\; (1\leq i<l)$ and $\alpha_l=\vep_l$. 
\item Dynkin diagram:
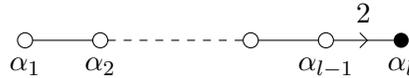
\begin{center}
\scalebox{0.9}{
\begin{tikzpicture}
\draw (1,0) circle (0.1);  \draw (1,-0.1) node[below] {$\alpha_1$};
\draw (1.1,0) -- (1.9,0); 
\draw (2,0) circle (0.1);  \draw (2,-0.1) node[below] {$\alpha_2$};
\draw [dashed] (2.1,0) -- (3.9,0);
\draw (4,0) circle (0.1); 
\draw (4.1,0) -- (4.9,0); 
\draw (5,0) circle (0.1); \draw (5,-0.1) node[below] {$\alpha_{l-1}$};
\draw (5.1,0) -- (5.9,0); 
\draw (5.45, -0.1) -- (5.55,0); \draw (5.45, 0.1) -- (5.55,0); 
\fill (6,0) circle (0.1);  \draw (6,-0.1) node[below] {$\alpha_{l}$};
\draw (5.5,0.1)  node[above] {$2$};
\end{tikzpicture}}
\end{center}
\end{enumerate}


\section{Affine root systems} \label{sect_affine-root-system}
In this section, we provide necessary informations on the affine root systems.
In particular, this includes the four non-reduced cases classified by I. Macdonald \cite{Macdonald1972}. For the reduced affine root systems, we follow the nomenclatures due to Moody \cite{Moody1969}, except for the case $BC_1^{(2)}$ which was called $A_1^{(2)}$ there.

Each affine root system of rank ${l+1}$ is a subset of the vector space $\R^{l+1}$ with the symmetric bilinear form $I$ with signature $(l,1,0)$. Let $\vep_i \,(1 \leq i \leq l)\,  \text{and}\,  b$ be a basis of $\R^{l+1}$ with the pairing given by
\[
I(\vep_i,\vep_j)=\delta_{i,j}\,, I(\vep_i,b)=0\,, I(b,b)=0.
\]

\newpage
\subsection{Reduced types}\label{sect_red-affine}

\subsubsection{$A_l^{(1)}$\; $(l\geq 1)$}\label{sect_A^1}
\begin{enumerate}
\item Real roots: $R(A_l^{(1)})=R(A_l)+\Z b$. 
\item Simple roots: \; $\alpha_0=b-\vep_1+\vep_{l+1}, \; \alpha_i=\vep_i-\vep_{i+1}\; (1\leq i\leq l)$. 
\item $b=\sum_{i=0}^l \alpha_i$.
\item Dynkin diagram:
\begin{enumerate}
\item[i)] $l=1$:
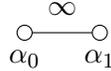
\begin{center}
\begin{tikzpicture}
\draw (0,0) circle (0.1); \draw (0,-0.1) node[below] {$\alpha_0$};
\draw (0.1, 0) -- (0.9,0); 
\draw (0.5, 0.1) node[above] {$\infty$}; 
\draw (1,0) circle (0.1);  \draw (1,-0.1) node[below] {$\alpha_1$};
\end{tikzpicture}
\end{center}
\item[ii)] $l\geq 2$:
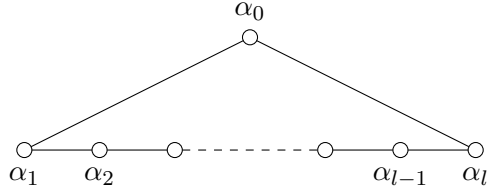
\begin{center}
\begin{tikzpicture}
\draw (0,0) circle (0.1);  \draw (0,-0.1) node[below] {$\alpha_1$};
\draw (0.1,0) -- (0.9,0); 
\draw (1,0) circle (0.1);  \draw (1,-0.1) node[below] {$\alpha_2$};
\draw (1.1,0) -- (1.9,0); 
\draw (2,0) circle (0.1);  
\draw [dashed] (2.1,0) -- (3.9,0);
\draw (4,0) circle (0.1); 
\draw (4.1,0) -- (4.9,0); 
\draw (5,0) circle (0.1); \draw (5,-0.1) node[below] {$\alpha_{l-1}$};
\draw (5.1, 0) -- (5.9,0); 
\draw (6,0) circle (0.1);  \draw (6,-0.1) node[below] {$\alpha_{l}$};
\draw (3,1.5) circle (0.1); \draw (3,1.6) node[above] {$\alpha_0$};
\draw (0.09,0.045) -- (2.91,1.455); 
\draw (5.91,0.045) -- (3.09,1.455); 
\end{tikzpicture}
\end{center}
\end{enumerate}
\item $W$-orbits:
\begin{enumerate}
\item[i)] $l=1$: $R(A_1^{(1)})=W(\alpha_0)\amalg W(\alpha_1)$.
\item[ii)] $l\geq 2$: single $W$-orbit.
\end{enumerate} 
\end{enumerate}

\subsubsection{$B_l^{(1)}$\; $(l\geq 3)$}\label{sect_B^1}

\begin{enumerate}
\item Real roots: $R(B_l^{(1)})=R(B_l)+\Z b$. 
\item Simple roots: \; $\alpha_0=b-\vep_1-\vep_2, \; \alpha_i=\vep_i-\vep_{i+1}\; (1\leq i<l)$ \\
\phantom{Simple roots: \;\,}  $\alpha_l=\vep_l$. 
\item $b=\alpha_0+\alpha_1+2(\alpha_2+\cdots+\alpha_l)$.
\item Dynkin diagram:
\begin{center}
\begin{tikzpicture}
\draw (1,0.5) circle (0.1);  \draw (0.9,0.5) node[left] {$\alpha_1$};
\draw (1,-0.5) circle (0.1);  \draw (0.9,-0.5) node[left] {$\alpha_0$};
\draw (1.09,0.455) -- (1.91,0.045); \draw (1.09,-0.455) -- (1.91,-0.045); 
\draw (2,0) circle (0.1);  \draw (2.1,-0.1) node[below] {$\alpha_2$};
\draw [dashed] (2.1,0) -- (3.9,0);
\draw (4,0) circle (0.1); 
\draw (4.1,0) -- (4.9,0); 
\draw (5,0) circle (0.1); \draw (5,-0.1) node[below] {$\alpha_{l-1}$};
\draw (5.1, 0) -- (5.9,0); 
\draw (6,0) circle (0.1);  \draw (6,-0.1) node[below] {$\alpha_{l}$};
\draw (5.55,0) -- (5.45,0.1) ;
\draw (5.55,0) -- (5.45,-0.1) ;
\draw (5.5,0.15) node[above] {$2$};
\end{tikzpicture}
\end{center}
\item $W$-orbits: $R(B_l^{(1)})_l \amalg R(B_l^{(1)})_s$.
\end{enumerate}

\subsubsection{$C_l^{(1)}$\; $(l\geq 2)$}\label{sect_C^1}

\begin{enumerate}
\item Real roots: $R(C_l^{(1)})=R(C_l)+\Z b$.
\item Simple roots: \; $\alpha_0=b-2\vep_1, \; \alpha_i=\vep_i-\vep_{i+1}\; (1\leq i<l)$, \\
\phantom{Simple roots: \;\,} $\alpha_l=2\vep_l$. 
\item $b=\alpha_0+2(\alpha_1+\cdots+\alpha_{l-1})+\alpha_l$.
\item Dynkin diagram:
\begin{center}
\begin{tikzpicture}
\draw (0,0) circle (0.1); \draw (0,-0.1) node[below] {$\alpha_0$};
\draw (0.1,0) -- (0.9,0);
\draw (0.45,0.1) -- (0.55,0);
\draw (0.45,-0.1) -- (0.55,0);
\draw (0.5,0.1) node[above] {$2$};
\draw (1,0) circle (0.1);  \draw (1,-0.1) node[below] {$\alpha_1$};
\draw (1.1,0) -- (1.9,0); 
\draw (2,0) circle (0.1);  \draw (2,-0.1) node[below] {$\alpha_2$};
\draw [dashed] (2.1,0) -- (3.9,0);
\draw (4,0) circle (0.1); 
\draw (4.1,0) -- (4.9,0); 
\draw (5,0) circle (0.1); \draw (5,-0.1) node[below] {$\alpha_{l-1}$};
\draw (5.1, 0) -- (5.9,0); 
\draw (6,0) circle (0.1);  \draw (6,-0.1) node[below] {$\alpha_{l}$};
\draw (5.45,0) -- (5.55,0.1) ;
\draw (5.45,0) -- (5.55,-0.1) ;
\draw (5.5,0.15) node[above] {$2$};
\end{tikzpicture}
\end{center}
\item $W$-orbits: $R(C_l^{(1)})_l=W(\alpha_0) \amalg W(\alpha_l),$ \\
\phantom{$W$-orbits:} $R(C_l^{(1)})_s$: single $W$-orbit.
\end{enumerate}


\subsubsection{$D_l^{(1)}$\; $(l\geq 4)$}\label{sect_D^1}

\begin{enumerate}
\item Real roots: $R(D_l^{(1)})=R(D_l)+\Z b$.
\item Simple roots: \; $\alpha_0=b-\vep_1-\vep_2, \; \alpha_i=\vep_i-\vep_{i+1}\; (1\leq i<l)$,  \\
\phantom{Simple roots: \;\,} $\alpha_l=\vep_{l-1}+\vep_l$. 
\item $b=\alpha_0+\alpha_1+2(\alpha_1+\cdots+\alpha_{l-2})+\alpha_{l-1}+\alpha_l$.
\item Dynkin diagram:
\begin{center}
\begin{tikzpicture}
\draw (0,0.5) circle (0.1);  \draw (-0.1,0.5) node[left] {$\alpha_1$};
\draw (0,-0.5) circle (0.1);  \draw (-0.1,-0.5) node[left] {$\alpha_0$};
\draw (0.09,0.455) -- (0.91,0.045); \draw (0.09,-0.455) -- (0.91,-0.045); 
\draw (1,0) circle (0.1);  \draw (1.1,-0.1) node[below] {$\alpha_2$};
\draw (1.1,0) -- (1.9,0); \draw (2,0) circle (0.1);
\draw [dashed] (2.1,0) -- (3.9,0);
\draw (4,0) circle (0.1); 
\draw (4.1,0) -- (4.9,0); 
\draw (5,0) circle (0.1); \draw (5,-0.1) node[below] {$\alpha_{l-2}$};
\draw (5.09,0.045) -- (5.91,0.455);
\draw (5.09,-0.045) -- (5.91,-0.455);
\draw (6,0.5) circle (0.1);  \draw (6.1,0.5) node[right] {$\alpha_{l-1}$};
\draw (6,-0.5) circle (0.1);  \draw (6.1,-0.5) node[right] {$\alpha_{l}$};
\end{tikzpicture}
\end{center}
\item $W$-orbits: single $W$-orbit.
\end{enumerate}

\subsubsection{$E_6^{(1)}$}\label{sect_E6^1}

\begin{enumerate}
\item Real roots: $R(E_6^{(1)})=R(E_6)+\Z b$.
\item Simple roots: \; $\alpha_0=b-\frac{1}{2}(\vep_1+\vep_2+\vep_3+\vep_4+\vep_5-\vep_6-\vep_7+\vep_8)$, \\
\phantom{simple r roots \; }$\alpha_1=\frac{1}{2}(\vep_1+\vep_8)-\frac{1}{2}(\vep_2+\vep_3+\cdots+\vep_7)$, \\ 
\phantom{simple r roots \; }$\alpha_2=\vep_1+\vep_2, \; \alpha_i=\vep_{i-1}-\vep_{i-2}\quad (3\leq i\leq 6)$.
\item $b=\alpha_0+\alpha_1+2\alpha_2+2\alpha_3+3\alpha_4+2\alpha_5+\alpha_6.$
\item Dynkin diagram:
\begin{center}
\begin{tikzpicture}
\draw (1,0) circle (0.1);  \draw (1,-0.1) node[below] {$\alpha_1$};
\draw (1.1,0) -- (1.9,0); 
\draw (2,0) circle (0.1);  \draw (2,-0.1) node[below] {$\alpha_3$};
\draw (2.1, 0) -- (2.9,0);
\draw (3,0) circle (0.1); \draw (3,-0.1) node[below] {$\alpha_4$};
\draw (3.1,0) -- (3.9,0); 
\draw (4,0) circle (0.1); \draw (4,-0.1) node[below] {$\alpha_5$};
\draw (4.1,0)-- (4.9,0);
\draw (5,0) circle (0.1); \draw (5,-0.1) node[below] {$\alpha_6$};
\draw (3,0.1) -- (3,0.9); \draw (3,1) circle (0.1); 
\draw (3.1,1) node[right] {$\alpha_2$};
\draw (3,1.1) -- (3,1.9); \draw (3,2) circle (0.1); 
\draw (3.1,2) node[right] {$\alpha_0$};
\end{tikzpicture}
\end{center}
\item $W$-orbits: single $W$-orbit.
\end{enumerate}

\subsubsection{$E_7^{(1)}$}\label{sect_E7^1}

\begin{enumerate}
\item Real roots: $R(E_7^{(1)})=R(E_7)+\Z b$.
\item Simple roots: \; $\alpha_0=b+\vep_7-\vep_8$, \\ 
\phantom{simple rroots \;\,} $\alpha_1=\frac{1}{2}(\vep_1+\vep_8)-\frac{1}{2}(\vep_2+\vep_3+\cdots+\vep_7)$,  \\
\phantom{simple rroots \;\,} $\alpha_2=\vep_1+\vep_2, \; \alpha_i=\vep_{i-1}-\vep_{i-2}\quad (3\leq i\leq 7)$.
\item 
$
b=
\alpha_0+2\alpha_1+2\alpha_2+3\alpha_3+4\alpha_4+3\alpha_5+2\alpha_6+\alpha_7.
$
\item Dynkin diagram:
\begin{center}
\begin{tikzpicture}
\draw (0,0) circle (0.1); \draw (0,-0.1) node[below] {$\alpha_0$};
\draw (0.1,0) -- (0.9,0);
\draw (1,0) circle (0.1);  \draw (1,-0.1) node[below] {$\alpha_1$};
\draw (1.1,0) -- (1.9,0); 
\draw (2,0) circle (0.1);  \draw (2,-0.1) node[below] {$\alpha_3$};
\draw (2.1, 0) -- (2.9,0);
\draw (3,0) circle (0.1); \draw (3,-0.1) node[below] {$\alpha_4$};
\draw (3.1,0) -- (3.9,0); 
\draw (4,0) circle (0.1); \draw (4,-0.1) node[below] {$\alpha_5$};
\draw (4.1,0)-- (4.9,0);
\draw (5,0) circle (0.1); \draw (5,-0.1) node[below] {$\alpha_6$};
\draw (5.1,0)-- (5.9,0);
\draw (6,0) circle (0.1); \draw (6,-0.1) node[below] {$\alpha_7$};
\draw (3,0.1) -- (3,0.9); \draw (3,1) circle (0.1); 
\draw (3,1.1) node[above] {$\alpha_2$};
\end{tikzpicture}
\end{center}
\item $W$-orbits: single $W$-orbit.
\end{enumerate}


\subsubsection{$E_8^{(1)}$}\label{sect_E8^1}

\begin{enumerate}
\item Real roots: $R(E_8^{(1)})=R(E_8)+\Z b$.
\item Simple roots: \; $\alpha_0=b-\vep_7-\vep_8$, \\ 
\phantom{SImple roots: \;} $\alpha_1=\frac{1}{2}(\vep_1+\vep_8)-\frac{1}{2}(\vep_2+\vep_3+\cdots+\vep_7),$ \\
\phantom{simple r roots \;} $\alpha_2=\vep_1+\vep_2, \; \alpha_i=\vep_{i-1}-\vep_{i-2}\quad (3\leq i\leq 8)$.
\item 
$
b=\alpha_0+2\alpha_1+3\alpha_2+4\alpha_3+6\alpha_4+5\alpha_5+4\alpha_6+3\alpha_7+2\alpha_8.
$
\item Dynkin diagram:
\begin{center}
\begin{tikzpicture}
\draw (1,0) circle (0.1);  \draw (1,-0.1) node[below] {$\alpha_1$};
\draw (1.1,0) -- (1.9,0); 
\draw (2,0) circle (0.1);  \draw (2,-0.1) node[below] {$\alpha_3$};
\draw (2.1, 0) -- (2.9,0);
\draw (3,0) circle (0.1); \draw (3,-0.1) node[below] {$\alpha_4$};
\draw (3.1,0) -- (3.9,0); 
\draw (4,0) circle (0.1); \draw (4,-0.1) node[below] {$\alpha_5$};
\draw (4.1,0)-- (4.9,0);
\draw (5,0) circle (0.1); \draw (5,-0.1) node[below] {$\alpha_6$};
\draw (5.1,0)-- (5.9,0);
\draw (6,0) circle (0.1); \draw (6,-0.1) node[below] {$\alpha_7$};
\draw (6.1,0)-- (6.9,0);
\draw (7,0) circle (0.1); \draw (7,-0.1) node[below] {$\alpha_8$};
\draw (7.1,0) -- (7.9,0); 
\draw (8,0) circle (0.1); \draw (8,-0.1) node[below] {$\alpha_0$};
\draw (3,0.1) -- (3,0.9); \draw (3,1) circle (0.1); 
\draw (3,1.1) node[above] {$\alpha_2$};
\end{tikzpicture}
\end{center}
\item $W$-orbits: single $W$-orbit.
\end{enumerate}


\subsubsection{$F_4^{(1)}$}\label{sect_F4^1}

\begin{enumerate}
\item Real roots: $R(F_4^{(1)})=R(F_4)+\Z b$.
\item Simple roots: \; $\alpha_0=b-\vep_1-\vep_2, \; \alpha_1=\vep_2-\vep_3, \; \alpha_2=\vep_3-\vep_4, \; \alpha_3=\vep_4,$ \\
\phantom{Simple roots\; r\,\,}$\alpha_4=\frac{1}{2}(\vep_1-\vep_2-\vep_3-\vep_4)$. 
\item 
$
b=
\alpha_0+2\alpha_1+3\alpha_2+4\alpha_3+2\alpha_4.
$
\item Dynkin diagram:
\begin{center}
\begin{tikzpicture}
\draw (0,0) circle (0.1);  \draw (0,-0.1) node[below] {$\alpha_0$};
\draw (0.1,0) -- (0.9,0);
\draw (1,0) circle (0.1);  \draw (1,-0.1) node[below] {$\alpha_1$};
\draw (1.1,0) -- (1.9,0); 
\draw (2,0) circle (0.1);  \draw (2,-0.1) node[below] {$\alpha_2$};
\draw (2.1, 0) -- (2.9,0);
\draw (3,0) circle (0.1); \draw (3,-0.1) node[below] {$\alpha_3$};
\draw (3.1,0) -- (3.9,0); 
\draw (4,0) circle (0.1); \draw (4,-0.1) node[below] {$\alpha_4$};
\draw (2.55,0) -- (2.45,0.1);
\draw (2.55,0) -- (2.45,-0.1); \draw (2.5,0.1) node[above] {$2$};
\end{tikzpicture}
\end{center}
\item $W$-orbits: $R(F_4^{(1)})_l \amalg R(F_4^{(1)})_s$.
\end{enumerate}

\subsubsection{$G_2^{(1)}$}\label{sect_G2^1}

\begin{enumerate}
\item Real roots: $R(G_2^{(1)})=R(G_2)+\Z b$. 
\item Simple roots: \; $\alpha_0=b+\vep_1+\vep_2-2\vep_3, \; \alpha_1=\vep_1-\vep_2$, \\ 
\phantom{Simple roots: \;\,} $\alpha_2=-2\vep_1+\vep_2+\vep_3.$ 
\item 
$
b=
\alpha_0+3\alpha_1+2\alpha_2.
$
\item Dynkin diagram:
\begin{center}
\begin{tikzpicture}
\draw (1,0) circle (0.1);  \draw (1,-0.1) node[below] {$\alpha_1$};
\draw (1.1,0) -- (1.9,0); 
\draw (2,0) circle (0.1);  \draw (2,-0.1) node[below] {$\alpha_2$};
\draw (1.45,0) -- (1.55,0.1);
\draw (1.45,0) -- (1.55,-0.1); \draw (1.5,0.1) node[above] {$3$};
\draw (2.1,0) -- (2.9,0);
\draw (3,0) circle (0.1); \draw (3,-0.1) node[below] {$\alpha_0$};
\end{tikzpicture}
\end{center}
\item $W$-orbits: $R(G_2^{(1)})_l \amalg R(G_2^{(1)})_s$.
\end{enumerate}

\subsubsection{$B_l^{(2)}$\; $(l\geq 2)$}\label{sect_B^2}

\begin{enumerate}
\item Real roots: $R(B_l^{(2)})=(R(B_l)_s+\Z b) \cup (R(B_l)_l+2\Z b)$.
\item Simple roots: \; $\alpha_0=b-\vep_1, \; \alpha_i=\vep_i-\vep_{i+1}\; (1\leq i<l)$ and $\alpha_l=\vep_l$. 
\item $b=\sum_{i=0}^l\alpha_i$.
\item Dynkin diagram:
\begin{center}
\begin{tikzpicture}
\draw (0,0) circle (0.1);  \draw (0,-0.1) node[below] {$\alpha_0$};
\draw (0.1,0) -- (0.9,0); 
\draw (0.45,0) -- (0.55,0.1) ;
\draw (0.45,0) -- (0.55,-0.1) ;
\draw (0.5,0.15) node[above] {$2$};
\draw (1,0) circle (0.1);  \draw (1,-0.1) node[below] {$\alpha_1$};
\draw (1.1,0) -- (1.9,0); 
\draw (2,0) circle (0.1);  \draw (2,-0.1) node[below] {$\alpha_2$};
\draw [dashed] (2.1,0) -- (3.9,0);
\draw (4,0) circle (0.1); 
\draw (4.1,0) -- (4.9,0); 
\draw (5,0) circle (0.1); \draw (5,-0.1) node[below] {$\alpha_{l-1}$};
\draw (5.1, 0) -- (5.9,0); 
\draw (6,0) circle (0.1);  \draw (6,-0.1) node[below] {$\alpha_{l}$};
\draw (5.55,0) -- (5.45,0.1) ;
\draw (5.55,0) -- (5.45,-0.1) ;
\draw (5.5,0.15) node[above] {$2$};
\end{tikzpicture}
\end{center}
\item $W$-orbits: $R(B_l^{(2)})_l$: single $W$-orbit, \\
\phantom{$W$-orbits:} $R(B_l^{(2)})_s=W(\alpha_0) \amalg W(\alpha_l)$.
\end{enumerate}

\subsubsection{$C_l^{(2)}$\; $(l\geq 3)$}\label{sect_C^2}

\begin{enumerate}
\item Real roots: $R(C_l^{(2)})=(R(C_l)_s+\Z b) \cup (R(C_l)_l+2\Z b)$.
\item Simple roots: \; $\alpha_0=b-\vep_1-\vep_2, \; \alpha_i=\vep_i-\vep_{i+1}\; (1\leq i<l)$, \\
\phantom{Simple roots:\;\, } $\alpha_l=2\vep_l$. 
\item $b=\alpha_0+\alpha_1+2(\alpha_2+\cdots+\alpha_{l-1})+\alpha_l$.
\item Dynkin diagram:
\begin{center}
\begin{tikzpicture}
\draw (1,0.5) circle (0.1);  \draw (0.9,0.5) node[left] {$\alpha_1$};
\draw (1,-0.5) circle (0.1);  \draw (0.9,-0.5) node[left] {$\alpha_0$};
\draw (1.09,0.455) -- (1.91,0.045); 
\draw (1.09,-0.455) -- (1.91,-0.045); 
\draw (2,0) circle (0.1);  \draw (2.1,-0.1) node[below] {$\alpha_2$};
\draw [dashed] (2.1,0) -- (3.9,0);
\draw (4,0) circle (0.1); 
\draw (4.1,0) -- (4.9,0); 
\draw (5,0) circle (0.1); \draw (5,-0.1) node[below] {$\alpha_{l-1}$};
\draw (5.1, 0) -- (5.9,0); 
\draw (6,0) circle (0.1);  \draw (6,-0.1) node[below] {$\alpha_{l}$};
\draw (5.45,0) -- (5.55,0.1) ;
\draw (5.45,0) -- (5.55,-0.1) ;
\draw (5.5,0.15) node[above] {$2$};
\end{tikzpicture}
\end{center}
\item $W$-orbits: $R(C_l^{(2)})_l \amalg R(C_l^{(2)})_s$.
\end{enumerate}

\subsubsection{$F_4^{(2)}$}\label{sect_F4^2}

\begin{enumerate}
\item Real roots: $R(F_4^{(2)})=(R(F_4)_s+\Z b) \cup( R(F_4)_l+2\Z b)$.
\item Simple roots: \; $\alpha_0=b-\vep_1, \; \alpha_1=\vep_2-\vep_3, \; \alpha_2=\vep_3-\vep_4, \; \alpha_3=\vep_4,$ \\
\phantom{Simple roots:  rr}$\alpha_4=\frac{1}{2}(\vep_1-\vep_2-\vep_3-\vep_4)$. 
\item 
$
b=\alpha_0+\alpha_1+2\alpha_2+3\alpha_3+2\alpha_4.
$
\item Dynkin diagram:
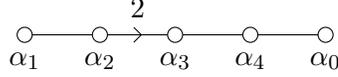
\begin{center}
\begin{tikzpicture}
\draw (1,0) circle (0.1);  \draw (1,-0.1) node[below] {$\alpha_1$};
\draw (1.1,0) -- (1.9,0); 
\draw (2,0) circle (0.1);  \draw (2,-0.1) node[below] {$\alpha_2$};
\draw (2.1, 0) -- (2.9,0);
\draw (3,0) circle (0.1); \draw (3,-0.1) node[below] {$\alpha_3$};
\draw (3.1,0) -- (3.9,0); 
\draw (4,0) circle (0.1); \draw (4,-0.1) node[below] {$\alpha_4$};
\draw (4.1,0) -- (4.9,0);
\draw (5,0) circle (0.1); \draw (5,-0.1) node[below] {$\alpha_0$};
\draw (2.55,0) -- (2.45,0.1);
\draw (2.55,0) -- (2.45,-0.1); \draw (2.5,0.1) node[above] {$2$};
\end{tikzpicture}
\end{center}
\item $W$-orbits: $R(F_4^{(2)})_l \amalg R(F_4^{(2)})_s$.
\end{enumerate}


\subsubsection{$G_2^{(3)}$}\label{sect_G2^3}

\begin{enumerate}
\item Real roots: $R(G_2^{(3)})=(R(G_2)_s+\Z b) \cup (R(G_2)_l+3\Z b)$. 
\item Simple roots: \; $\alpha_0=b+\vep_2-\vep_3, \; \alpha_1=\vep_1-\vep_2, \; \alpha_2=-2\vep_1+\vep_2+\vep_3.$ 
\item 
$
b=
\alpha_0+2\alpha_1+\alpha_2.
$
\item Dynkin diagram:
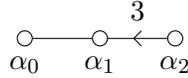
\begin{center}
\begin{tikzpicture}
\draw (0,0) circle (0.1); \draw (0,-0.1) node[below] {$\alpha_0$};
\draw (0.1,0) -- (0.9,0);
\draw (1,0) circle (0.1);  \draw (1,-0.1) node[below] {$\alpha_1$};
\draw (1.1,0) -- (1.9,0); 
\draw (2,0) circle (0.1);  \draw (2,-0.1) node[below] {$\alpha_2$};
\draw (1.45,0) -- (1.55,0.1);
\draw (1.45,0) -- (1.55,-0.1); \draw (1.5,0.1) node[above] {$3$};
\end{tikzpicture}
\end{center}
\item $W$-orbits: $R(G_2^{(3)})_l \amalg R(G_2^{(3)})_s$.
\end{enumerate}


\subsubsection{$BC_l^{(2)}$\; $(l\geq 1)$}\label{sect_BC^2}

\begin{enumerate}
\item Real roots: 
\[ R(BC_l^{(2)})=
\begin{cases} \;  \begin{array}{l} (R(BC_l)_s+\Z b) \;\cup \; (R(BC_l)_m+\Z b) \\ 
                                                     \phantom{(R(BC_l)_s+\Z b)} \; \cup \; (R(BC_l)_l+(1+2\Z)b) \end{array} \quad  & l>1, \\
\; (R(BC_1)_s+\Z b) \cup (R(BC_1)_l+(1+2\Z)b) \qquad & l=1. \end{cases}
\]
\item Simple roots: \; $\alpha_0=b-2\vep_1, \; \alpha_i=\vep_i-\vep_{i+1}\; (1\leq i<l)$ and $\alpha_l=\vep_l$. 
\item $b=\alpha_0+2(\alpha_1+\cdots +\alpha_l)$.
\item Dynkin diagram:
\begin{enumerate}
\item[i)] $l=1$:
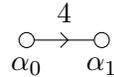
\begin{center}
\begin{tikzpicture}
\draw (0,0) circle (0.1); \draw (0,-0.1) node[below] {$\alpha_0$};
\draw (0.1, 0) -- (0.9,0); 
\draw (0.45, 0.1) -- (0.55,0); 
\draw (0.45, -0.1) -- (0.55,0); 
\draw (0.5, 0.1) node[above] {$4$}; 
\draw (1,0) circle (0.1);  \draw (1,-0.1) node[below] {$\alpha_1$};
\end{tikzpicture}
\end{center}
\item[ii)] $l\geq 2$:
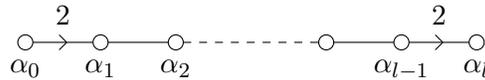
\begin{center}
\begin{tikzpicture}
\draw (0,0) circle (0.1); \draw (0,-0.1) node[below] {$\alpha_0$};
\draw (0.1,0) -- (0.9,0); 
\draw (0.45,0.1) -- (0.55,0); 
\draw (0.45, -0.1) -- (0.55,0); \draw (0.5,0.1) node[above] {$2$};
\draw (1,0) circle (0.1);  \draw (1,-0.1) node[below] {$\alpha_1$};
\draw (1.1,0) -- (1.9,0); 
\draw (2,0) circle (0.1);  \draw (2,-0.1) node[below] {$\alpha_2$};
\draw [dashed] (2.1,0) -- (3.9,0);
\draw (4,0) circle (0.1); 
\draw (4.1,0) -- (4.9,0); 
\draw (5,0) circle (0.1); \draw (5,-0.1) node[below] {$\alpha_{l-1}$};
\draw (5.1,0) -- (5.9,0); 
\draw (5.45, -0.1) -- (5.55,0); \draw (5.45, 0.1) -- (5.55,0); 
\draw (6,0) circle (0.1);  \draw (6,-0.1) node[below] {$\alpha_{l}$};
\draw (5.5,0.1)  node[above] {$2$};
\end{tikzpicture}
\end{center}
\end{enumerate}
\item $W$-orbits:
\begin{enumerate}
\item[i)] $l=1$: $R(BC_1^{(2)})_l \amalg R(BC_1^{(2)})_s$.
\item[ii)] $l\geq 2$: $R(BC_l^{(2)})_l \amalg R(BC_l^{(2)})_m \amalg R(BC_l^{(2)})_s$.
\end{enumerate}
\end{enumerate}


\subsection{Non-reduced types}\label{sect_non-red-affine}

\subsubsection{$BCC_l\; (l\geq 1)$}\label{sect_data-BCC}

\begin{enumerate}
\item Real roots: $R(BCC_l)=R(BC_l)+\Z b$. 
\item Simple roots: \; $\alpha_0=b-2\vep_1$, $\alpha_i=\vep_i-\vep_{i+1}\; (1\leq i<l)$ and $\alpha_l=\vep_l$. 
\item $b=\alpha_0+2(\alpha_1+\cdots +\alpha_l)$ 
\item Dynkin diagram
\begin{enumerate}
\item[i)] $l=1$:
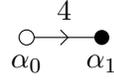
\begin{center}
\begin{tikzpicture}
\draw (0,0) circle (0.1); \draw (0,-0.1) node[below] {$\alpha_0$};
\draw (0.1, 0) -- (0.9,0); 
\draw (0.45, 0.1) -- (0.55,0); 
\draw (0.45, -0.1) -- (0.55,0); 
\draw (0.5, 0.1) node[above] {$4$}; 
\fill (1,0) circle (0.1);  \draw (1,-0.1) node[below] {$\alpha_1$};
\end{tikzpicture}
\end{center}
\item[ii)] $l\geq 2$:
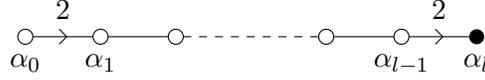
\begin{center}
\begin{tikzpicture}
\draw (0,0) circle (0.1); \draw (0,-0.1) node[below] {$\alpha_0$};
\draw (0.45, 0.1) -- (0.55,0); 
\draw (0.45, -0.1) -- (0.55,0); 
\draw (0.5,0.1)  node[above] {$2$};
\draw (0.1,0) -- (0.9,0) ;
\draw (1,0) circle (0.1);  \draw (1,-0.1) node[below] {$\alpha_1$};
\draw (1.1,0) -- (1.9,0); 
\draw (2,0) circle (0.1); 
\draw [dashed] (2.1,0) -- (3.9,0);
\draw (4,0) circle (0.1); 
\draw (4.1,0) -- (4.9,0); 
\draw (5,0) circle (0.1); \draw (5,-0.1) node[below] {$\alpha_{l-1}$};
\draw (5.1,0) -- (5.9,0);
\fill (6,0) circle (0.1);  \draw (6,-0.1) node[below] {$\alpha_{l}$};
\draw (5.45,0.1) -- (5.55,0) ;
\draw (5.45,-0.1) -- (5.55,0) ;
\draw (5.5,0.1) node[above] {$2$};
\end{tikzpicture}
\end{center}
\end{enumerate}
\item $W$-orbits:
\begin{enumerate}
\item[i)] $l=1$: $R(BCC_1)_l=W(\alpha_0) \amalg W(2\alpha_1),$ \\
 \phantom{$l=1$\, } 
 $R(BCC_1)_s$: single $W$-orbit.
\item[ii)] $l\geq 2$: $R(BCC_l)_l=W(\alpha_0) \amalg W(2\alpha_l),$ \\
\phantom{$l\geq 2$\;\, }$R(BCC_l)_m, \; R(BCC_l)_s$: single $W$-orbits.
\end{enumerate}
\end{enumerate}

\subsubsection{$C^\vee BC_l\; (l\geq 1)$}

\begin{enumerate}
\item Real roots: $R(C^\vee BC_l)=(R(BC_l)_s+\Z b) \cup (R(BC_l)_m+2\Z b)$ \\
\phantom{Real roots: $R(C^\vee BC_l)=$} $\cup (R(BC_l)_l+4\Z b)$.
\item Simple roots: \; $\alpha_0=b-\vep_1$, $\alpha_i=\vep_i-\vep_{i+1}\; (1\leq i<l)$ and $\alpha_l=\vep_l$. 
\item $b=\alpha_0+\alpha_1+\cdots +\alpha_l$ 
\item Dynkin diagram
\begin{enumerate}
\item[i)] $l=1$:
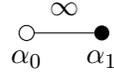
\begin{center}
\begin{tikzpicture}
\draw (0,0) circle (0.1); \draw (0,-0.1) node[below] {$\alpha_0$};
\draw (0.1, 0) -- (0.9,0); 
\draw (0.5, 0.1) node[above] {$\infty$}; 
\fill (1,0) circle (0.1);  \draw (1,-0.1) node[below] {$\alpha_1$};
\end{tikzpicture}
\end{center}
\item[ii)] $l\geq 2$:
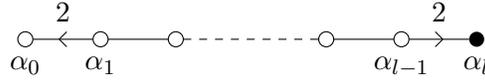
\begin{center}
\begin{tikzpicture}
\draw (0,0) circle (0.1); \draw (0,-0.1) node[below] {$\alpha_0$};
\draw (0.1, 0) -- (0.9,0); 
\draw (0.45, 0) -- (0.55,-0.1); 
\draw (0.45, 0) -- (0.55,0.1) ;
\draw (0.5,0.1) node[above] {$2$} ;
\draw (1,0) circle (0.1);  \draw (1,-0.1) node[below] {$\alpha_1$};
\draw (1.1,0) -- (1.9,0); 
\draw (2,0) circle (0.1); 
\draw [dashed] (2.1,0) -- (3.9,0);
\draw (4,0) circle (0.1); 
\draw (4.1,0) -- (4.9,0); 
\draw (5,0) circle (0.1); \draw (5,-0.1) node[below] {$\alpha_{l-1}$};
\draw (5.1, 0) -- (5.9,0); 
\draw (5.45, -0.1) -- (5.55,0); 
\draw (5.45, 0.1) -- (5.55,0); 
\draw (5.5,0.1) node[above] {$2$};
\fill (6,0) circle (0.1);  \draw (6,-0.1) node[below] {$\alpha_{l}$};
\end{tikzpicture}
\end{center}
\end{enumerate}
\item $W$-orbits:
\begin{enumerate}
\item[i)] $l=1$: $R(C^\vee BC_1)_l$: single $W$-orbit, \\
\phantom{$l=1$:\;} $R(C^\vee BC_1)_s=W(\alpha_0) \amalg W(\alpha_1)$.
\item[ii)] $l\geq 2$: $R(C^\vee BC_l)_l, \; R(C^\vee BC_l)_m$: single $W$-orbits, \\
\phantom{$l=1$\, } $R(C^\vee BC_l)_s=W(\alpha_0) \amalg W(\alpha_l)$.
\end{enumerate}
\end{enumerate}

\subsubsection{$BB_l^\vee\; (l\geq 2)$}

\begin{enumerate}
\item Real roots: $R(BB_l^\vee)=(R(BC_l)_s+\Z b) \cup (R(BC_l)_m+\Z b)$ \\
\phantom{Real roots: $R(BB_l^\vee)=$} $\cup (R(BC_l)_l+2\Z b)$.
\item Simple roots: \; $\alpha_0=b-\vep_1-\vep_2$, $\alpha_i=\vep_i-\vep_{i+1}\; (1\leq i<l)$, \\
          \phantom{Simple roots:\;\, }    $\alpha_l=\vep_l$. 
\item $b=\alpha_0+\alpha_1+2(\alpha_2+\cdots +\alpha_l)$ 
\item Dynkin diagram
\begin{enumerate}
\item[i)] $l=2$:
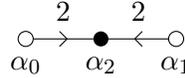
\begin{center}
\begin{tikzpicture}
\draw (0,0) circle (0.1); \draw (0,-0.1) node[below] {$\alpha_0$};
\draw (0.1,0) -- (0.9,0);
\draw (0.45, 0.1) -- (0.55,0); 
\draw (0.45, -0.1) -- (0.55,0); 
\draw (0.5,0.1) node[above] {$2$} ;
\fill (1,0) circle (0.1); \draw (1,-0.1) node[below] {$\alpha_2$};
\draw (1.1, 0) -- (1.9,0); 
\draw (1.45, 0) -- (1.55,-0.1); 
\draw (1.45, 0) -- (1.55,0.1); 
\draw (1.5,0.1) node[above] {$2$};
\draw (2,0) circle (0.1);  \draw (2,-0.1) node[below] {$\alpha_1$};
\end{tikzpicture}
\end{center}
\item[ii)] $l\geq 3$:
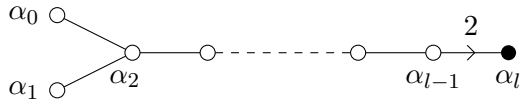
\begin{center}
\begin{tikzpicture}
\draw (0,0.5) circle (0.1); \draw (-0.1,0.5) node[left] {$\alpha_0$};
\draw (0,-0.5) circle (0.1); \draw (-0.1,-0.5) node[left] {$\alpha_1$};
\draw (1,0) circle (0.1);  \draw (0.9,-0.1) node[below] {$\alpha_2$};
\draw (1.1,0) -- (1.9,0);
\draw (0.09,0.455) -- (0.91,0.045) ;
\draw (0.09,-0.455) -- (0.91,-0.045) ; 
\draw (2,0) circle (0.1); 
\draw [dashed] (2.1,0) -- (3.9,0);
\draw (4,0) circle (0.1); \draw (4.1,0) -- (4.9,0); 
\draw (5,0) circle (0.1); \draw (5,-0.1) node[below] {$\alpha_{l-1}$};
\draw (5.1, 0) -- (5.9,0); 
\draw (5.45,0.1) -- (5.55,0);
\draw (5.45,-0.1) -- (5.55,0);
\draw (5.5,0.1) node[above] {$2$}; 
\fill (6,0) circle (0.1);  \draw (6,-0.1) node[below] {$\alpha_{l}$};
\end{tikzpicture}
\end{center}
\end{enumerate}
\item $W$-orbits: 
\begin{enumerate}
\item[i)] $l=2$: $R(BB_2^\vee)_m=W(\alpha_0) \amalg W(\alpha_1)$, \\
\phantom{$l=2$\, } $R(BB_2^{\vee})_l, \; R(BB_2^{\vee})_s$:\; single $W$-orbits.
\item[ii)] $l\geq 3$: $R(BB_l^{\vee})_l \amalg R(BB_l^{\vee})_m \amalg R(BB_l^{\vee})_s$.
\end{enumerate}
\end{enumerate}

\subsubsection{$C^\vee C_l\; (l\geq 1)$}\label{sect:CveeC_l}

\begin{enumerate}
\item Real roots: $R(C^\vee C_l)=(R(BC_l)_s+\Z b) \cup (R(BC_l)_m+2\Z b)$ \\
\phantom{Real roots: $R(C^\vee C_l)=$} $\cup (R(BC_l)_l+2\Z b)$.
\item Simple roots: \; $\alpha_0=b-\vep_1$, $\alpha_i=\vep_i-\vep_{i+1}\; (1\leq i<l)$ and $\alpha_l=\vep_l$. 
\item $b=\alpha_0+\alpha_1+\cdots +\alpha_l$ 
\item Dynkin diagram
\begin{enumerate}
\item[i)] $l=1$:
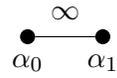
\begin{center}
\begin{tikzpicture}
\fill (0,0) circle (0.1); \draw (0,-0.1) node[below] {$\alpha_0$};
\draw (0.1, 0) -- (0.9,0); 
\draw (0.5, 0.1) node[above] {$\infty$}; 
\fill (1,0) circle (0.1);  \draw (1,-0.1) node[below] {$\alpha_1$};
\end{tikzpicture}
\end{center}
\item[ii)] $l\geq 2$:
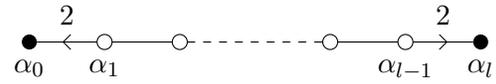
\begin{center}
\begin{tikzpicture}
\fill (0,0) circle (0.1); \draw (0,-0.1) node[below] {$\alpha_0$};
\draw (0.1,0) -- (0.9,0);
\draw (0.45, 0) -- (0.55,0.1); 
\draw (0.45, 0) -- (0.55,-0.1); 
\draw (0.5,0.1) node[above] {$2$} ;
\draw (1,0) circle (0.1);  \draw (1,-0.1) node[below] {$\alpha_1$};
\draw (1.1,0) -- (1.9,0); 
\draw (2,0) circle (0.1); 
\draw [dashed] (2.1,0) -- (3.9,0);
\draw (4,0) circle (0.1); 
\draw (4.1,0) -- (4.9,0); 
\draw (5,0) circle (0.1); \draw (5,-0.1) node[below] {$\alpha_{l-1}$};
\draw (5.1,0) -- (5.9,0);
\draw (5.45, 0.1) -- (5.55,0); 
\draw (5.45, -0.1) -- (5.55,0); 
\draw (5.5,0.1) node[above] {$2$};
\fill (6,0) circle (0.1);  \draw (6,-0.1) node[below] {$\alpha_{l}$};
\end{tikzpicture}
\end{center}
\end{enumerate}
\item $W$-orbits:
\begin{enumerate}
\item[i)] $l=1$: $R(C^\vee C_1)_l=W(2\alpha_0) \amalg W(2\alpha_1)$, \\ 
\phantom{$l=1$:} $R(C^\vee C_1)_s=W(\alpha_0) \amalg W(\alpha_1)$.
\item[ii)] $l\geq 2$: $R(C^\vee C_l)_l \,= W(2\alpha_0) \amalg W(2\alpha_l)$, \\
\phantom{$l\geq 2$\;\, }$R(C^\vee C_l)_m$: single $W$-orbit, \\
\phantom{$l\geq 2$\;\, }$R(C^\vee C_l)_s =W(\alpha_0) \amalg W(\alpha_l)$. 
\end{enumerate}
\end{enumerate}


\section{MERSs with reduced affine quotient}\label{sect_clasfficiation-red-quot}
We obtained the classification of mERSs $(R,G)$, belonging to $(F,I)$, with reduced $R/G$. As a consequence, one sees that the classification due to K. Saito \cite{Saito1985} is complete in this case. For the sake of completeness, let us collect several important root data for each marked elliptic root system $(R,G)$:
\begin{enumerate}
\item The tier numbers $t_1(R,G)$ and $t_2(R,G)$ and the marked dual root system $(R^\vee,G)$,
\item Root system $R$,
\item Simple roots of $(R,G)$,
\item Counting and Exponents,
\item Elliptic diagram,
\item $W$-orbits.
\end{enumerate}
Here, we fix a basis $(\vep_1, \vep_2,\cdots,\vep_l,b,a)$ of $F$ as in \eqref{def_base-F}. 
\setcounter{subsection}{1}

\subsubsection{Type $A_l^{(1,1)}\; (l\geq 1)$}
\begin{enumerate}
\item $t_1(A_l^{(1,1)})=1$, \;  $t_2(A_l^{(1,1)})=1$ and 
$(A_l^{(1,1)})^\vee=A_l^{(1,1)}$.
\item $R(A_l^{(1,1)})=R(A_l)+\Z a+\Z b$.
\item $\alpha_0=-\vep_1+\vep_{l+1}+b, \; \alpha_i=\vep_i-\vep_{i+1}\, (1\leq i\leq l)$.
\item $k(\alpha_i)=1$ \quad ($0\leq i\leq l$). \\
         $m_i=1$ \quad ($0\leq i\leq l$).
\item Elliptic diagram:
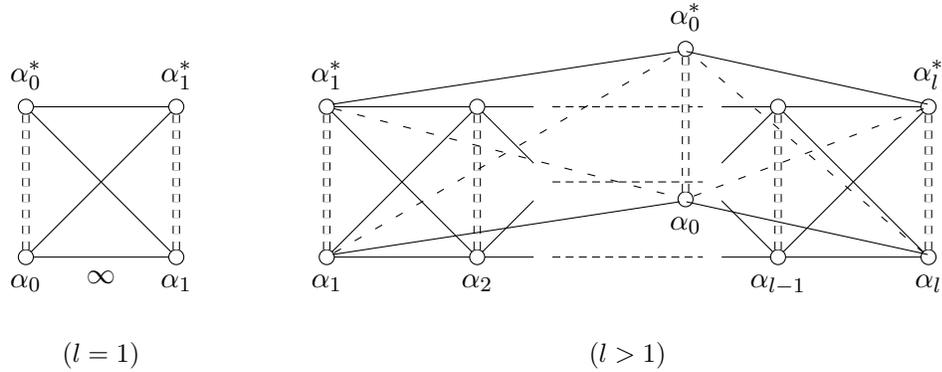
\input{Figure-red/diagram-A11}
\item $W$-orbits:
\begin{enumerate}
\item[i)] $l=1$: $R(A_1^{(1,1)})=\amalg_{i \in \{0,1\}} \left(W(\alpha_i) \amalg W(\alpha_i^\ast)\right)$.
\item[ii)] $l\geq 2$: single $W$-orbit.
\end{enumerate}
\end{enumerate}

\subsubsection{Type $A_1^{(1,1)\ast}$}
\begin{enumerate}
\item $t_1(A_l^{(1,1)\ast})=1$, \; $t_2(A_l^{(1,1)\ast})=1$ and 
$(A_1^{(1,1)\ast})^\vee=A_1^{(1,1)\ast}$.
\item $R(A_1^{(1,1)\ast})=R(A_1)+\{\, ma+nb\,\vert\, m,n \in \Z, \, mn \equiv 0\, [2]\, \}$.
\item $\alpha_0=-\vep+b, \; \alpha_1=\vep$.  \quad   ($\vep \in F$ is a vector with $I(\vep,\vep)=2$.)
\item $k(\alpha_0)=2, \quad k(\alpha_1)=1$. \\
         $m_0=\dfrac{1}{2}, \quad m_1=1$.
\item Elliptic diagram:
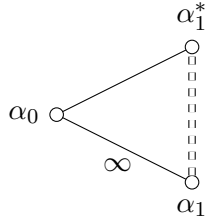
\input{Figure-red/diagram-A11st}
\item $W$-orbits: $R(A_1^{(1,1)\ast})=W(\alpha_0) \amalg W(\alpha_1) \amalg W(\alpha_1^\ast)$.
\end{enumerate}

\subsubsection{Type $B_l^{(1,1)}\; (l\geq 3)$}
\begin{enumerate}
\item $t_1(B_l^{(1,1)})=1$, \;  $t_2(B_l^{(1,1)})=1$ and  
$(B_l^{(1,1)})^\vee=C_l^{(2,2)}$,
\item $R(B_l^{(1,1)})=R(B_l)+\Z a+\Z b$. 
\item $\alpha_0=-\vep_1-\vep_2+b, \; \alpha_i=\vep_i-\vep_{i+1}\, (1\leq i< l), \; \alpha_l=\vep_l$.
\item $k(\alpha_i)=1 \; (0\leq i\leq l).$ \\
         $m_0=m_1=2,  \quad m_i=4\; (1< i<l), \quad m_l=2$.
\item Elliptic diagram:
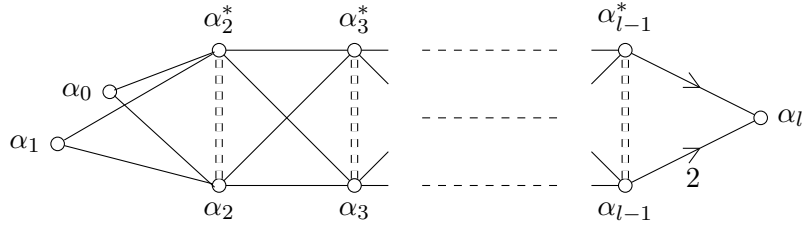
\input{Figure-red/diagram-B11}
\item $W$-orbits: $R(B_l^{(1,1)})_l \amalg R(B_l^{(1,1)})_s$.
\end{enumerate}

\subsubsection{Type $B_l^{(1,2)}\; (l\geq 3)$}
\begin{enumerate}
\item $t_1(B_l^{(1,2)})=1$, \; $t_2(B_l^{(1,2)})=2$ and  
$(B_l^{(1,2)})^\vee=C_l^{(2,1)}$,
\item $R(B_l^{(1,2)})=(R(B_l)_s+\Z a+\Z b)\cup (R(B_l)_l+2\Z a+\Z b)$.
\item $\alpha_0=-\vep_1-\vep_2+b, \; \alpha_i=\vep_i-\vep_{i+1}\, (1\leq i< l), \; \alpha_l=\vep_l$.
\item $k(\alpha_i)=2 \; (0\leq i< l), \quad k(\alpha_l)=1$. \\
         $m_0=m_1=1,  \quad m_i=2\; (1< i\leq l)$.
\item Elliptic diagram:
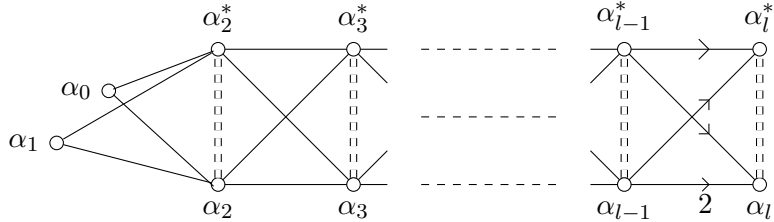
\input{Figure-red/diagram-B12}
\item $W$-orbits:
$R(B_l^{(1,2)})_l$: single $W$-orbit, \\ 
\phantom{$W$-orbits:}
$R(B_l^{(1,2)})_s=W(\alpha_l) \amalg W(\alpha_l^\ast)$.
\end{enumerate}

\subsubsection{Type $B_l^{(2,1)}\; (l\geq 2)$}
\begin{enumerate}
\item $t_1(B_l^{(2,1)})=2$, \; $t_2(B_l^{(2,1)})=1$ and  $(B_l^{(2,1)})^\vee=C_l^{(1,2)}$,
\item $R(B_l^{(2,1)})=(R(B_l)_s+\Z a+\Z b) \cup (R(B_l)_l+\Z a+2\Z b)$.
\item $\alpha_0=-\vep_1+b, \; \alpha_i=\vep_i-\vep_{i+1}\, (1\leq i< l), \; \alpha_l=\vep_l$.
\item $k(\alpha_i)=1 \; (0\leq i\leq  l)$. \\
         $m_0=1,  \quad m_i=2\; (0< i< l), \quad m_l=1$.
\item Elliptic diagram:
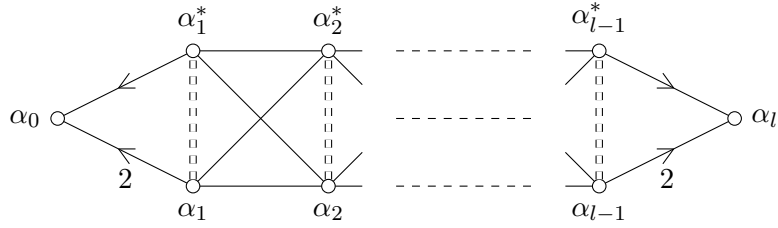
\input{Figure-red/diagram-B21}
\item $W$-orbits:
\begin{enumerate}
\item[i)] $l=2$: $R(B_2^{(2,1)})_l=W(\alpha_1) \amalg W(\alpha_1^\ast)$, \\ 
\phantom{$l=2$:} $R(B_2^{(2,1)})_s=W(\alpha_0) \amalg W(\alpha_2)$.
\item[ii)] $l\geq 3$: $R(B_l^{(2,1)})_l$: single $W$-orbit, \\ 
\phantom{$l\geq 3$:} $R(B_l^{(2,1)})_s=W(\alpha_0) \amalg W(\alpha_l)$.
\end{enumerate}
\end{enumerate}

\subsubsection{Type $B_l^{(2,2)}\; (l\geq 2)$}
\begin{enumerate}
\item $t_1(B_l^{(2,2)})=2$, \;  $t_2(B_l^{(2,2)})=2$ and $(B_l^{(2,2)})^\vee=C_l^{(1,1)}$.
\item $R(B_l^{(2,2)})=(R(B_l)_s+\Z a+\Z b)\cup (R(B_l)_l+2\Z a+2\Z b)$.
\item $\alpha_0=-\vep_1+b, \; \alpha_i=\vep_i-\vep_{i+1}\, (1\leq i< l), \; \alpha_l=\vep_l$.
\item $k(\alpha_0)=1, \quad k(\alpha_i)=2 \; (0<i<l), \quad k(\alpha_l)=1$. \\
         $m_i=1\; (0\leq  i\leq  l)$.
\item Elliptic diagram:
\input{Figure-red/diagram-B22}
\item $W$-orbits: $R(B_l^{(2,2)})_l$: single $W$-orbit, \\
\phantom{$W$-orbits:}
$R(B_l^{(2,2)})_s=\amalg_{i \in \{0,l\}} \left( W(\alpha_i) \amalg W(\alpha_i^\ast)\right)$.
\end{enumerate}

\subsubsection{Type $C_l^{(1,1)}\; (l\geq 2)$}
\begin{enumerate}
\item $t_1(C_l^{(1,1)})=1$, \;  $t_2(C_l^{(1,1)})=1$ and  $(C_l^{(1,1)})^\vee=B_l^{(2,2)}$.
\item $R(C_l^{(1,1)})=R(C_l)+\Z a+\Z b.$
\item $\alpha_0=-2\vep_1+b, \; \alpha_i=\vep_i-\vep_{i+1}\, (1\leq i< l), \; \alpha_l=2\vep_l$.
\item $k(\alpha_i)=1 \; (0\leq i\leq l)$. \\
         $m_i=2\; (0\leq  i\leq  l)$.
\item Elliptic diagram:
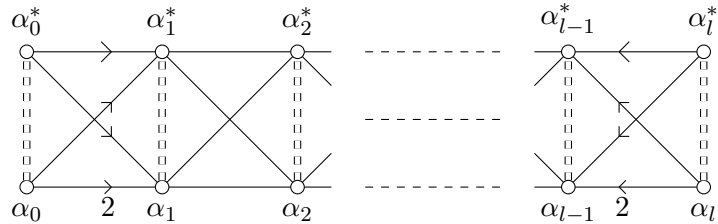
\input{Figure-red/diagram-C11}
\item $W$-orbits: $R(C_l^{(1,1)})_l=\amalg_{i \in \{0,l\}} \left( W(\alpha_i) \amalg W(\alpha_i^\ast)\right)$, \\
\phantom{$W$-orbits:}
$R(C_l^{(1,1)})_s$: : single $W$-orbit.
\end{enumerate}

\subsubsection{Type $C_l^{(1,2)}\; (l\geq 2)$}
\begin{enumerate}
\item $t_1(C_l^{(1,2)})=1$, \;  $t_2(C_l^{(1,2)})=2$ and  $(C_l^{(1,2)})^\vee=B_l^{(2,1)}$.
\item $R(C_l^{(1,2)})=(R(C_l)_s+\Z a+\Z b)\cup (R(C_l)_l+2\Z a+\Z b)$.
\item $\alpha_0=-2\vep_1+b, \; \alpha_i=\vep_i-\vep_{i+1}\, (1\leq i< l), \; \alpha_l=2\vep_l$.
\item $k(\alpha_0)=2, \quad k(\alpha_i)=1 \; (0< i< l), \quad k(\alpha_l)=2$. \\
         $m_0=1, \quad m_i=2\; (0<  i<  l), \quad m_l=1$.
\item Elliptic diagram:
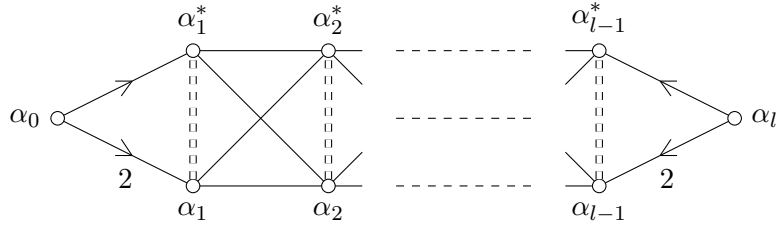
\input{Figure-red/diagram-C12}
\item $W$-orbits:
\begin{enumerate}
\item[i)] $l=2$: $R(C_2^{(1,2)})_l=W(\alpha_0) \amalg W(\alpha_2)$, \\ 
\phantom{$l=2$:} $R(C_2^{(1,2)})_s=W(\alpha_1) \amalg W(\alpha_1^\ast)$. 
\item[ii)] $l\geq 3$: $R(C_l^{(1,2)})_l=W(\alpha_0) \amalg W(\alpha_l)$, \\
\phantom{$l\geq 3$:} $R(C_l^{(1,2)})_s$: single $W$-orbit.
\end{enumerate}
\end{enumerate}

\subsubsection{Type $C_l^{(2,1)}\; (l\geq 3)$}
\begin{enumerate}
\item $t_1(C_l^{(2,1)})=2$, \; $t_2(C_l^{(2,1)})=1$ and $(C_l^{(2,1)})^\vee=B_l^{(1,2)}$.
\item $R(C_l^{(2,1)})=(R(C_l)_s+\Z a+\Z b) \cup (R(C_l)_l+\Z a+2\Z b)$.
\item $\alpha_0=-\vep_1-\vep_2+b, \; \alpha_i=\vep_i-\vep_{i+1}\, (1\leq i< l), \; \alpha_l=2\vep_l$.
\item $k(\alpha_i)=1 \; (0\leq  i\leq  l)$. \\
         $m_0=m_1=1, \quad m_i=2\; (1<  i\leq   l)$.
\item Elliptic diagram:
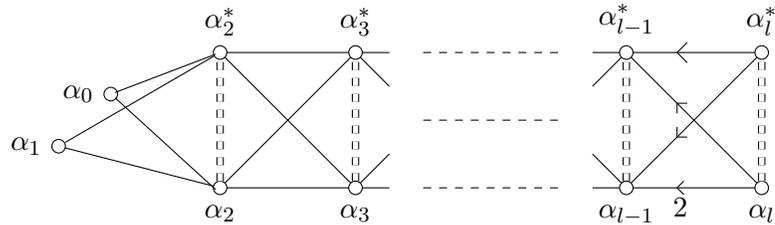
\input{Figure-red/diagram-C21}
\item $W$-orbits:
$R(C_l^{(2,1)})_l=W(\alpha_l) \amalg W(\alpha_l^\ast)$, \\
\phantom{$W$-orbits:} 
$R(C_l^{(2,1)})_s$: single $W$-orbit.
\end{enumerate}

\subsubsection{Type $C_l^{(2,2)}\; (l\geq 3)$}
\begin{enumerate}
\item $t_1(C_l^{(2,2)})=2$,  \; $t_2(C_l^{(2,2)})=2$ and  $(C_l^{(2,2)})^\vee=B_l^{(1,1)}$.
\item $R(C_l^{(2,2)})=(R(C_l)_s+\Z a+\Z b) \cup (R(C_l)_l+2\Z a+2\Z b)$.
\item $\alpha_0=-\vep_1-\vep_2+b, \; \alpha_i=\vep_i-\vep_{i+1}\, (1\leq i< l), \; \alpha_l=2\vep_l$.
\item $k(\alpha_i)=1 \; (0\leq  i< l), \quad k(\alpha_l)=2$. \\
         $m_0=m_1=1, \quad m_i=2\; (1<  i<   l), \quad m_l=1$.
\item Elliptic diagram:
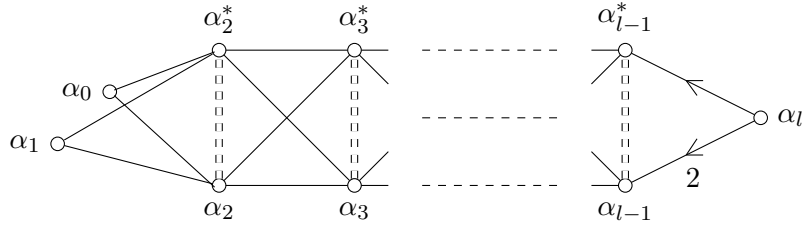
\input{Figure-red/diagram-C22}
\item $W$-orbits: $R(C_l^{(2,2)})_l \amalg R(C_l^{(2,2)})_s$.
\end{enumerate}

\subsubsection{Type $B_l^{(2,2)\ast}\; (l\geq 2)$}
\begin{enumerate}
\item $t_1(B_l^{(2,2)\ast})=2$, \; $t_2(B_l^{(2,2)\ast})=2$ and $(B_l^{(2,2)\ast})^\vee=C_l^{(1,1)\ast}$.
\item $R(B_l^{(2,2)\ast})=(R(B_l)_s+\{\,ma+nb\,\vert\, mn \equiv 0\, [2]\, \})$ \\
\phantom{$R(B_l^{(2,2)\ast})=$} $\cup (R(B_l)_l+2\Z a+2\Z b)$.
\item $\alpha_0=-\vep_1+b, \; \alpha_i=\vep_i-\vep_{i+1}\, (1\leq i< l), \; \alpha_l=\vep_l$.
\item $k(\alpha_i)=2 \; (0\leq i<l), \quad k(\alpha_l)=1$. \\
         $m_0=\dfrac{1}{2}, \quad m_i=1\; (0<  i\leq  l)$.
\item Elliptic diagram:
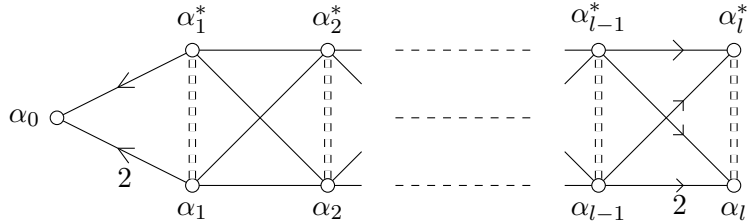
\input{Figure-red/diagram-B22st}
\item $W$-orbits: $R(B_l^{(2,2)\ast})_l$: single $W$-orbit, \\
\phantom{$W$-orbits:} 
$R(B_l^{(2,2)\ast})_s=W(\alpha_0) \amalg W(\alpha_l) \amalg W(\alpha_l^\ast)$.
\end{enumerate}

\subsubsection{Type $C_l^{(1,1)\ast}\; (l\geq 2)$}
\begin{enumerate}
\item $t_1(C_l^{(1,1)\ast})=1$, \; $t_2(C_l^{(1,1)\ast})=1$ and  $(C_l^{(1,1)\ast})^\vee=B_l^{(2,2)\ast}$,
\item $R(C_l^{(1,1)\ast})=(R(C_l)_s+\Z a+\Z b) \cup (R(C_l)_l+\{\, ma+nb\, \vert\, mn \equiv 0\, [2]\, \})$.
\item $\alpha_0=-2\vep_1+b, \; \alpha_i=\vep_i-\vep_{i+1}\, (1\leq i< l), \; \alpha_l=2\vep_l$.
\item $k(\alpha_0)=2, \quad k(\alpha_i)=1 \; (0< i\leq l)$. \\
         $m_0=1, \quad m_i=2\; (0<  i\leq  l)$.
\item Elliptic diagram:
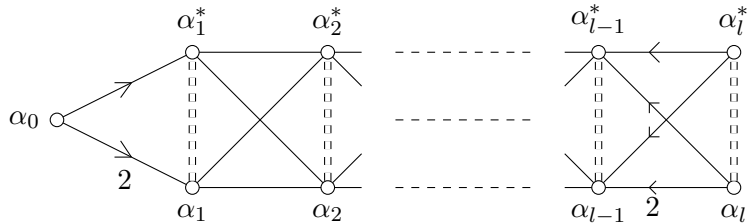
\input{Figure-red/diagram-C11st}
\item $W$-orbits: $R(C_l^{(1,1)\ast})_l=W(\alpha_0) \amalg W(\alpha_l) \amalg W(\alpha_l^\ast)$, \\
\phantom{$W$-orbits:} 
$R(C_l^{(1,1)\ast})_s$: single $W$-orbit.
\end{enumerate}

\subsubsection{Type $BC_l^{(2,1)}\; (l\geq 1)$}
\begin{enumerate}
\item $t_1(BC_l^{(2,1)})=2$, \;  $t_2(BC_l^{(2,1)})=1$ and $(BC_l^{(2,1)})^\vee=BC_l^{(2,4)}$,
\item $R(BC_l^{(2,1)})=(R(BC_l)_s+\Z a+\Z b) \cup (R(BC_l)_m+\Z a+\Z b)$ \\
\phantom{$R(BC_l^{(2,1)})=$} $\cup (R(BC_l)_l+\Z a+(1+2\Z)b)$.
\item $\alpha_0=-2\vep_1+b, \; \alpha_i=\vep_i-\vep_{i+1}\, (1\leq i< l), \; \alpha_l=\vep_l$.
\item $k(\alpha_i)=1 \; (0\leq i\leq l)$. \\
         $m_i=4\; (0\leq  i<  l), \quad m_l=2$.
\item Elliptic diagram:
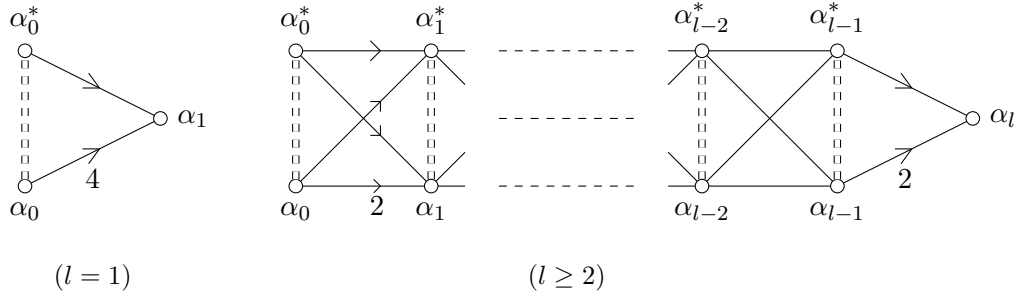
\input{Figure-red/diagram-BC21}
\item $W$-orbits: $R(BC_l^{(2,1)})_l=W(\alpha_0) \amalg W(\alpha_0^\ast)$, \\
\phantom{$W$-orbits:} $R(BC_l^{(2,1)})_m$: single $W$-orbit, \;\; $R(BC_l^{(2,1)})_s=W(\alpha_l)$.
\end{enumerate}

\subsubsection{Type $BC_l^{(2,4)}\; (l\geq 1)$}
\begin{enumerate}
\item $t_1(BC_l^{(2,4)})=2$, \;  $t_2(BC_l^{(2,4)})=4$ and $(BC_l^{(2,4)})^\vee=BC_l^{(2,1)}$.
\item $R(BC_l^{(2,4)})=(R(BC_l)_s+\Z a+\Z b) \cup (R(BC_l)_m+2\Z a+\Z b)$ \\
\phantom{$R(BC_l^{(2,4)})=$} $\cup (R(BC_l)_l+4\Z a+(1+2\Z)b)$.
\item $\alpha_0=-2\vep_1+b, \; \alpha_i=\vep_i-\vep_{i+1}\, (1\leq i< l), \; \alpha_l=\vep_l$.
\item $k(\alpha_0)=4, \quad k(\alpha_i)=2 \; (0< i< l), \quad k(\alpha_l)=1$. \\
         $m_0=1, \quad m_i=2\; (0<  i\leq l)$.
\item Elliptic diagram:
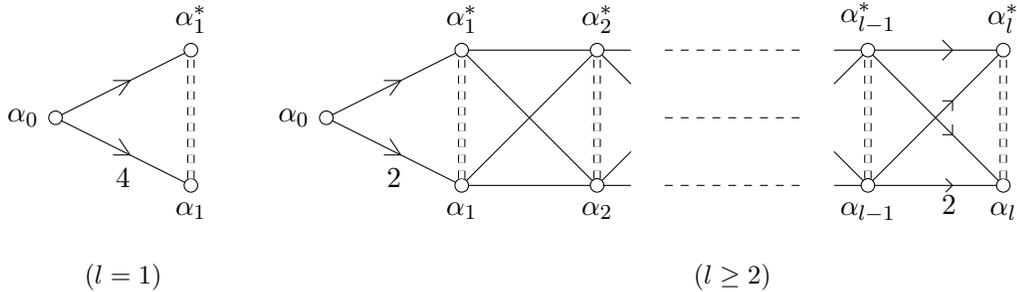
\input{Figure-red/diagram-BC24}
\item $W$-orbits: $R(BC_l^{(2,4)})_l= W(\alpha_0)$, \;\; $R(BC_l^{(2,4)})_m$: single $W$-orbit,  \\
\phantom{$W$-orbits:} $R(BC_l^{(2,4)})_s=W(\alpha_l) \amalg W(\alpha_l^\ast)$.
\end{enumerate}

\subsubsection{Type $BC_l^{(2,2)}(1)\; (l\geq 2)$}
\begin{enumerate}
\item $t_1(BC_l^{(2,2)}(1))=2$, \; $t_2(BC_l^{(2,2)}(1))=2$ and $(BC_l^{(2,2)}(1))^\vee=BC_l^{(2,2)}(1)$.
\item $R(BC_l^{(2,2)}(1))=(R(BC_l)_s+\Z a+\Z b) \cup (R(BC_l)_m+\Z a+\Z b)$ \\
\phantom{$R(BC_l^{(2,2)}(1))=$} $\cup (R(BC_l)_l+2\Z a+(1+2\Z)b)$.
\item $\alpha_0=-2\vep_1+b, \; \alpha_i=\vep_i-\vep_{i+1}\, (1\leq i< l), \; \alpha_l=\vep_l$.
\item $k(\alpha_0)=2, \quad k(\alpha_i)=1 \; (0< i\leq  l)$. \\
         $m_0=2, \quad m_i=4\; (0<  i< l), \quad m_l=2$.
\item Elliptic diagram:
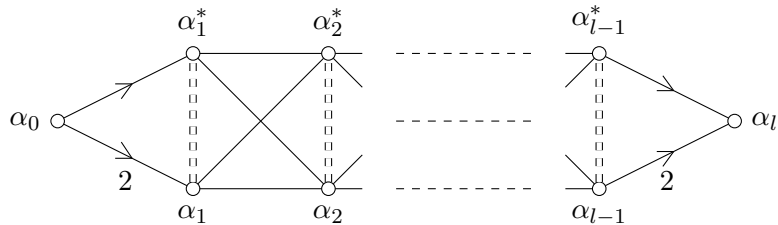
\input{Figure-red/diagram-BC221}
\item $W$-orbits:
\begin{enumerate}
\item[i)] $l=2$: $R(BC_2^{(2,2)}(1))_l=W(\alpha_0)$,\;\;  $R(BC_2^{(2,2)}(1))_s=W(\alpha_2)$, \\
\phantom{$l=2$:} $R(BC_2^{(2,2)}(1))_m=W(\alpha_1) \amalg W(\alpha_1^\ast)$.
\item[ii)] $l\geq 3$: $R(BC_2^{(2,2)}(1))_l=W(\alpha_0)$,\;\;  $R(BC_2^{(2,2)}(1))_s=W(\alpha_l)$, \\ 
\phantom{$l\geq 3$:} $R(BC_l^{(2,2)}(1))_m$: single $W$-orbit. 
\end{enumerate}
\end{enumerate}

\subsubsection{Type $BC_l^{(2,2)}(2)\; (l\geq 1)$}
\begin{enumerate}
\item $t_1(BC_l^{(2,2)}(2))=2$, \;  $t_2(BC_l^{(2,2)}(2))=2$ and $(BC_l^{(2,2)}(2))^\vee=BC_l^{(2,2)}(2)$.
\item $R(BC_l^{(2,2)}(2))=(R(BC_l)_s+\Z a+\Z b) \cup (R(BC_l)_m+2\Z a+\Z b)$ \\ 
\phantom{$R(BC_l^{(2,2)}(2))=$} $\cup (R(BC_l)_l+2\Z a+(1+2\Z)b)$.
\item $\alpha_0=-2\vep_1+b, \; \alpha_i=\vep_i-\vep_{i+1}\, (1\leq i< l), \; \alpha_l=\vep_l$.
\item $k(\alpha_i)=2 \; (0\leq  i<  l), \quad k(\alpha_l)=1$. \\
         $m_i=2\; (0\leq  i\leq  l)$.
\item Elliptic diagram:
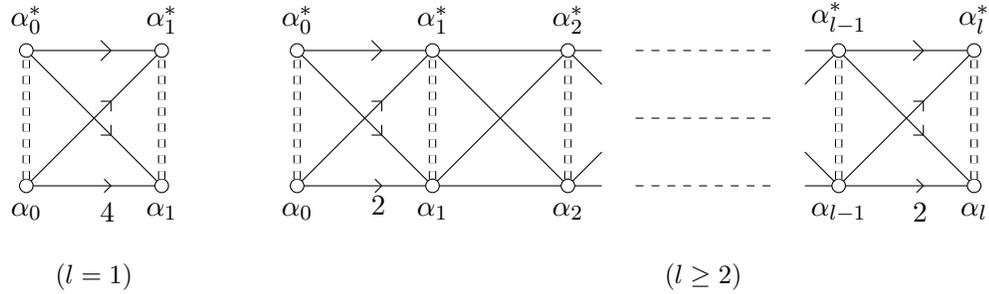
\input{Figure-red/diagram-BC222}
\item $W$-orbits: $R(BC_l^{(2,2)}(2))_l=W(\alpha_0) \amalg W(\alpha_0^\ast)$, \\
\phantom{$W$-orbits:} $R(BC_l^{(2,2)}(2))_m$: single $W$-orbit, \\
\phantom{$W$-orbits:} $R(BC_l^{(2,2)}(2))_s=W(\alpha_l) \amalg W(\alpha_l^\ast)$.
\end{enumerate}

\subsubsection{Type $D_l^{(1,1)}\; (l\geq 4)$}
\begin{enumerate}
\item $t_1(D_l^{(1,1)})=1$, \; $t_2(D_l^{(1,1)})=1$ and $(D_l^{(1,1)})^\vee=D_l^{(1,1)}$.
\item $R(D_l^{(1,1)})=R(D_l)+\Z a+\Z b$.
\item $\alpha_0=-\vep_1-\vep_2+b, \; \alpha_i=\vep_i-\vep_{i+1}\, (1\leq i< l), \; \alpha_l=\vep_{l-1}+\vep_l$.
\item $k(\alpha_i)=1 \; (0\leq  i\leq   l)$. \\
         $m_0=m_1=1, \quad m_i=2\; (1<  i< l-1), \quad m_{l-1}=m_l=1$.
\item Elliptic diagram:
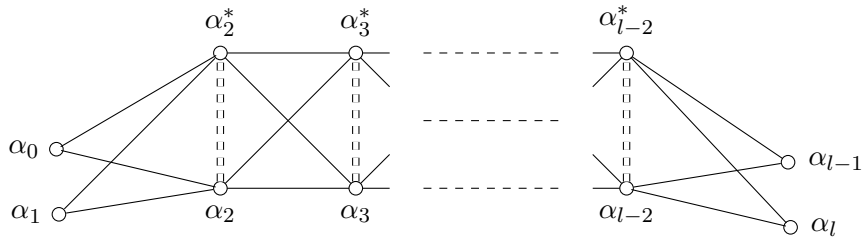
\input{Figure-red/diagram-D11}
\item $W$-orbits: single $W$-orbit. 
\end{enumerate}

\subsubsection{Type $E_6^{(1,1)}$}
\begin{enumerate}
\item $t_1(E_6^{(1,1)})=1$, \; $t_2(E_6^{(1,1)})=1$ and $(E_6^{(1,1)})^\vee=E_6^{(1,1)}$,
\item $R(E_6^{(1,1)})=R(E_6)+\Z a+\Z b$. 
\item $\alpha_0=-\frac{1}{2}(\vep_1+\vep_2+\vep_3+\vep_4+\vep_5-\vep_6-\vep_7+\vep_8) +b,$\\
$\alpha_1=\frac{1}{2}(\vep_1+\vep_8)-\frac{1}{2}(\vep_2+\vep_3+\cdots+\vep_7), \; \alpha_2=\vep_1+\vep_2$, \\ $\alpha_i=\vep_{i-1}-\vep_{i-2}\quad (3\leq i\leq 6)$.
\item $k(\alpha_i)=1 \; (0\leq  i\leq   6)$. \\
         $m_0=m_1=1, \quad m_2=m_3=2, \quad m_4=3, \quad m_5=2, \quad m_6=1$.
\item Elliptic diagram:
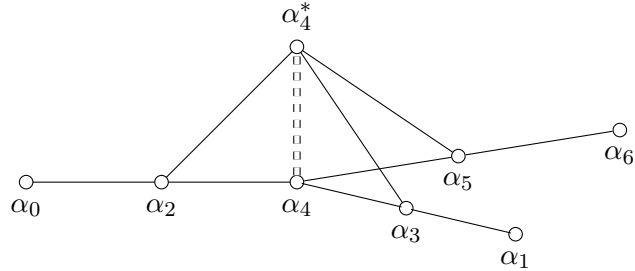
\input{Figure-red/diagram-E611-cl}
\item $W$-orbits: single $W$-orbit. 
\end{enumerate}

\subsubsection{Type $E_7^{(1,1)}$}
\begin{enumerate}
\item $t_1(E_7^{(1,1)})=1$, \;  $t_2(E_7^{(1,1)})=1$ and $(E_7^{(1,1)})^\vee=E_7^{(1,1)}$.
\item $R(E_7^{(1,1)})=R(E_7)+\Z a+\Z b$. 
\item $\alpha_0=-(\vep_8-\vep_7)+b, \; \alpha_1=\frac{1}{2}(\vep_1+\vep_8)-\frac{1}{2}(\vep_2+\vep_3+\cdots+\vep_7)$, \\
$\alpha_2=\vep_1+\vep_2, \; \alpha_i=\vep_{i-1}-\vep_{i-2}\quad (3\leq i\leq  7)$. 
\item $k(\alpha_i)=1 \; (0\leq  i\leq   7)$. \\
         $m_0=1, \quad m_1=m_2=2 \quad m_3=3, \quad m_4=4, \quad m_5=3$, \\ $m_6=2, \quad m_7=1$.
\item Elliptic diagram:
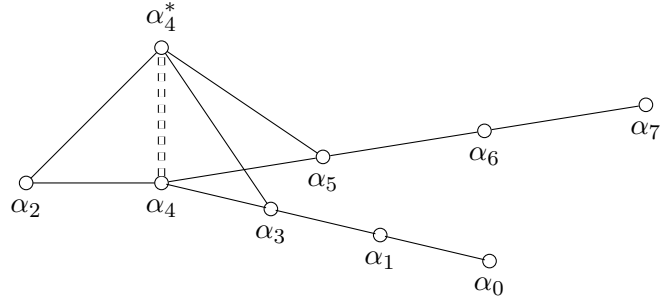
\input{Figure-red/diagram-E711-cl}
\item $W$-orbits: single $W$-orbit. 
\end{enumerate}

\subsubsection{Type $E_8^{(1,1)}$}
\begin{enumerate}
\item $t_1(E_8^{(1,1)})=1$, \;  $t_2(E_8^{(1,1)})=1$ and $(E_8^{(1,1)})^\vee=E_8^{(1,1)}$.
\item $R=R(E_8)+\Z a+\Z b$.
\item $\alpha_0=-(\vep_7+\vep_8)+b, \; \alpha_1=\frac{1}{2}(\vep_1+\vep_8)-\frac{1}{2}(\vep_2+\vep_3+\cdots+\vep_7)$, \\
$\alpha_2=\vep_1+\vep_2, \; \alpha_i=\vep_{i-1}-\vep_{i-2}\quad (3\leq i\leq 8)$.
\item $k(\alpha_i)=1 \; (0\leq  i\leq  8)$. \\
         $m_0=1, \quad m_1=2, \quad m_2=3, \quad m_3=4, \quad m_4=6$, \\ $m_5=5, \quad m_6=4, \quad m_7=3, \quad m_8=2$.
\item Elliptic diagram:
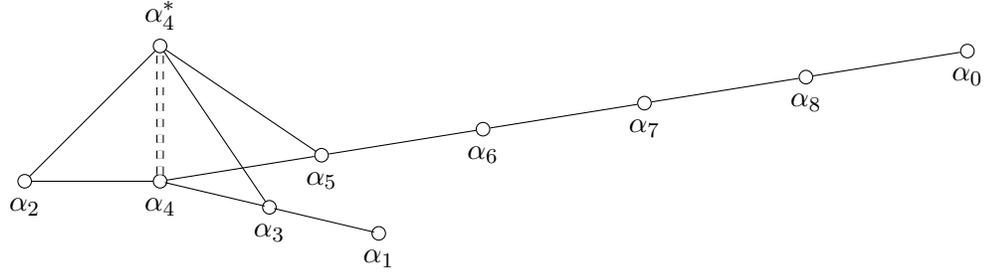
\input{Figure-red/diagram-E811-cl}
\item $W$-orbits: single $W$-orbit. 
\end{enumerate}

\subsubsection{Type $F_4^{(1,1)}$}
\begin{enumerate}
\item $t_1(F_4^{(1,1)})=1$, \; $t_2(F_4^{(1,1)})=1$ and $(F_4^{(1,1)})^\vee=F_4^{(2,2)}$.
\item $R(F_4^{(1,1)})=R(F_4)+\Z a+\Z b$. 
\item $\alpha_0=-\vep_1-\vep_2+b, \; \alpha_1=\vep_2-\vep_3, \; \alpha_2=\vep_3-\vep_4, \; \alpha_3=\vep_4$, \\ $\alpha_4=\dfrac{1}{2}(\vep_1-\vep_2-\vep_3-\vep_4)$. 
\item $k(\alpha_i)=1 \; (0\leq  i\leq  4)$. \\
         $m_0=2, \quad m_1=4, \quad m_2=6 \quad m_3=4, \quad m_4=2$.
\item Elliptic diagram:
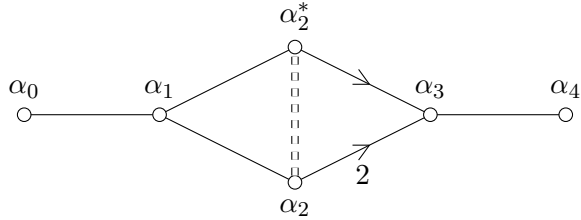
\input{Figure-red/diagram-F411}
\item $W$-orbits: $R(F_4^{(1,1)})_l \amalg R(F_4^{(1,1)})_s$.
\end{enumerate}

\subsubsection{Type $F_4^{(1,2)}$}
\begin{enumerate}
\item $t_1(F_4^{(1,2)})=1$, \;  $t_2(F_4^{(1,2)})=2$ and $(F_4^{(1,2)})^\vee=F_4^{(2,1)}$.
\item $R(F_4^{(1,2)})=(R(F_4)_s+\Z a+\Z b)  \cup (R(F_4)_l+2\Z a+\Z b)$.
\item $\alpha_0=-\vep_1-\vep_2+b, \; \alpha_1=\vep_2-\vep_3, \; \alpha_2=\vep_3-\vep_4, \; \alpha_3=\vep_4$, \\ $\alpha_4=\dfrac{1}{2}(\vep_1-\vep_2-\vep_3-\vep_4)$. 
\item $k(\alpha_i)=2 \; (0\leq  i\leq  2), \quad k(\alpha_3)=k(\alpha_4)=1$. \\
         $m_0=1, \quad m_1=2, \quad m_2=3 \quad m_3=4, \quad m_4=2$.
\item Elliptic diagram:
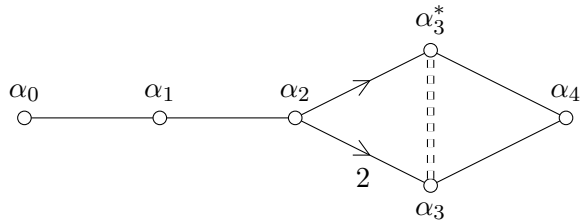
\input{Figure-red/diagram-F412}
\item $W$-orbits: $R(F_4^{(1,2)})_l \amalg R(F_4^{(1,2)})_s$.
\end{enumerate}

\subsubsection{Type $F_4^{(2,1)}$}
\begin{enumerate}
\item $t_1(F_4^{(2,1)})=2$, \; $t_2(F_4^{(2,1)})=1$ and $(F_4^{(2,1)})^\vee=F_4^{(1,2)}$.
\item $R(F_4^{(2,1)})=(R(F_4)_s+\Z a+\Z b) \cup (R(F_4)_l+\Z a+2\Z b)$.
\item $\alpha_0=-\vep_1+b, \; \alpha_1=\vep_2-\vep_3, \; \alpha_2=\vep_3-\vep_4, \; \alpha_3=\vep_4$, \\ $\alpha_4=\dfrac{1}{2}(\vep_1-\vep_2-\vep_3-\vep_4)$.
\item $k(\alpha_i)=1 \; (0\leq  i\leq  4)$. \\
         $m_0=1, \quad m_1=2, \quad m_2=4 \quad m_3=3, \quad m_4=2$.
\item Elliptic diagram:
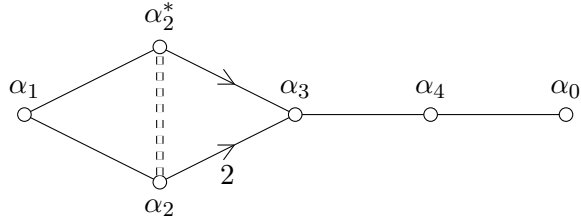
\input{Figure-red/diagram-F421-tr}
\item $W$-orbits: $R(F_4^{(2,1)})_l \amalg R(F_4^{(2,1)})_s$.
\end{enumerate}

\subsubsection{Type $F_4^{(2,2)}$}
\begin{enumerate}
\item $t_1(F_4^{(2,2)})=2$, \;  $t_2(F_4^{(2,2)})=2$ and $(F_4^{(2,2)})^\vee=F_4^{(1,1)}$,
\item $R(F_4^{(2,2)})=(R(F_4)_s+\Z a+\Z b) \cup (R(F_4)_l+2\Z a+2\Z b)$.
\item $\alpha_0=-\vep_1+b, \; \alpha_1=\vep_2-\vep_3, \; \alpha_2=\vep_3-\vep_4, \; \alpha_3=\vep_4$, \\ $\alpha_4=\dfrac{1}{2}(\vep_1-\vep_2-\vep_3-\vep_4)$.
\item  $k(\alpha_0)=1, \quad k(\alpha_1)=k(\alpha_2)=2,\quad k(\alpha_3)=k(\alpha_4)=1$. \\
         $m_0=1, \quad m_1=2, \quad m_2=3 \quad m_3=2, \quad m_4=1$.
\item Elliptic diagram:
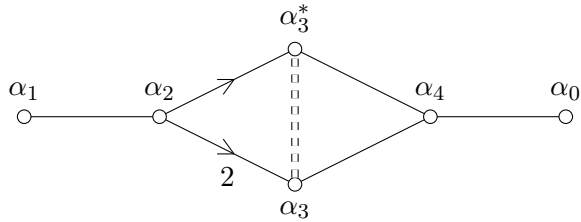
\input{Figure-red/diagram-F422-tr}
\item $W$-orbits: $R(F_4^{(2,2)})_l \amalg R(F_4^{(2,2)})_s$.
\end{enumerate}

\subsubsection{Type $G_2^{(1,1)}$}
\begin{enumerate}
\item $t_1(G_2^{(1,1)})=1$, \; $t_2(G_2^{(1,1)})=1$ and  $(G_2^{(1,1)})^\vee=G_2^{(3,3)}$.
\item  $R(G_2^{(1,1)})=R(G_2)+\Z a+\Z b$.
\item $\alpha_0=\vep_1+\vep_2-2\vep_3+b, \; \alpha_1=\vep_1-\vep_2, \; \alpha_2=-2\vep_1+\vep_2+\vep_3$.
\item $k(\alpha_i)=1 \; (0\leq  i\leq  2)$, \; 
         $m_0=3, \; m_1=3, \; m_2=6$.
\item Elliptic diagram:
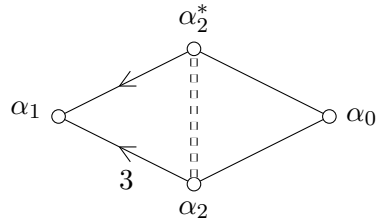
\input{Figure-red/diagram-G211-tr}
\item $W$-orbits: $R(G_2^{(1,1)})_l \amalg R(G_2^{(1,1)})_s$.
\end{enumerate}

\subsubsection{Type $G_2^{(1,3)}$}
\begin{enumerate}
\item $t_1(G_2^{(1,3)})=1$, \; $t_2(G_2^{(1,3)})=3$ and $(G_2^{(1,3)})^\vee=G_2^{(3,1)}$.
\item $R(G_2^{(1,3)})=(R(G_2)_s+\Z a+\Z b) \cup (R(G_2)_l+3\Z a+\Z b)$.
\item $\alpha_0=\vep_1+\vep_2-2\vep_3+b, \; \alpha_1=\vep_1-\vep_2, \; \alpha_2=-2\vep_1+\vep_2+\vep_3$.
\item $k(\alpha_0)=k(\alpha_2)=3, \; k(\alpha_1)=1$, \; 
         $m_0=1, \quad m_1=3, \; m_2=2$.
\item Elliptic diagram:
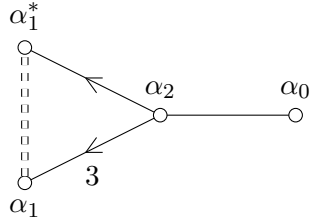
\input{Figure-red/diagram-G213-tr}
\item $W$-orbits: $R(G_2^{(1,3)})_l \amalg R(G_2^{(1,3)})_s$.
\end{enumerate}

\subsubsection{Type $G_2^{(3,1)}$}
\begin{enumerate}
\item $t_1(G_2^{(3,1)})=3$, \; $t_2(G_2^{(3,1)})=1$ and $(G_2^{(3,1)})^\vee=G_2^{(1,3)}$.
\item $R(G_2^{(3,1)})=(R(G_2)_s+\Z a+\Z b) \cup (R(G_2)_l+\Z a+3\Z b)$.
\item $\alpha_0=\vep_2-\vep_3+b, \; \alpha_1=\vep_1-\vep_2, \; \alpha_2=-2\vep_1+\vep_2+\vep_3$.
\item $k(\alpha_i)=1\; (0\leq i\leq 2)$, \; 
         $m_0=1, \quad m_1=2, \quad m_2=3$.
\item Elliptic diagram:
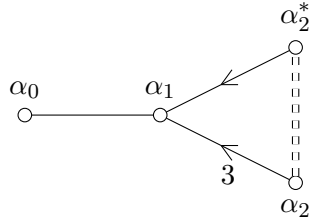
\input{Figure-red/diagram-G231}
\item $W$-orbits: $R(G_2^{(3,1)})_l \amalg R(G_2^{(3,1)})_s$.
\end{enumerate}

\subsubsection{Type $G_2^{(3,3)}$}
\begin{enumerate}
\item $t_1(G_2^{(3,3)})=3$, \; $t_2(G_2^{(3,3)})=3$ and $(G_2^{(3,3)})^\vee=G_2^{(1,1)}$.
\item $R(G_2^{(3,3)})=(R(G_2)_s+\Z a+\Z b) \cup (R(G_2)_l+3\Z a+3\Z b)$.
\item $\alpha_0=\vep_2-\vep_3+b, \; \alpha_1=\vep_1-\vep_2, \; \alpha_2=-2\vep_1+\vep_2+\vep_3$.
\item $k(\alpha_0)=k(\alpha_1)=1, \; k(\alpha_2)=3,$ \; 
         $m_0=1, \; m_1=2, \; m_2=1$.
\item Elliptic diagram:
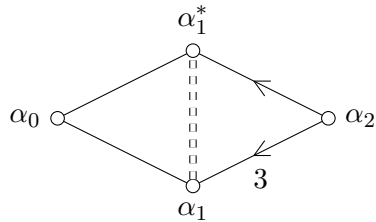
\input{Figure-red/diagram-G233}
\item $W$-orbits: $R(G_2^{(3,3)})_l \amalg R(G_2^{(3,3)})_s$.
\end{enumerate}

\part*{}


\begin{thebibliography}{AAB+97}
\bibitem[AAB+97]{AllisonAzamBermanGaoPianzola1997}
Bruce N. Allison, Saeid Azam, Stephen Berman, Yun Gao, and Arturo Pianzola, 
\underline{Extended affine Lie algebras and their root systems}, Mem. Amer. Math. Soc. 
{\bf 126} (1997), no. 603, x+122. MR 1376741
\bibitem[ABG02]{AllisonBenkartGao2002}
Bruce Allison, Georgia Benkart, and Yun Gao, 
\underline{Lie algebras graded by the root}
\underline{systems $BC_r,\, r\geq 2$}, Mem. Amer. Math. Soc. 
{\bf 158} (2002), no. 751, x+158. MR 1902499
\bibitem[ABP14]{AllisonBermanPianzola2014}
Bruce Allison, Stephen Berman, and Arturo Pianzola, 
\underline{Multiloop algebras,} 
\underline{iterated loop algebras and extended affine Lie algebras of nullity 2}, 
J. Eur. Math. Soc. (JEMS) {\bf 16} (2014), no. 2, 327–385. MR 3161286
\bibitem[AKY05]{AzamKhaliliYousofzadeh2005}
S. Azam, V. Khalili, and M. Yousofzadeh, 
\underline{Extended affine root systems of} \underline{type $BC$}, 
J. Lie Theory {\bf 15} (2005), no. 1, 145–181. 
MR 2115234
\bibitem[Aza02]{Azam2002}
S. Azam, 
\underline{Extended affine root systems}, 
J. Lie Theory {\bf 12} (2002), no. 2, 515– 527. MR 1923782
\bibitem[BS03]{BenkartSmirnov2003}
Georgia Benkart and Oleg Smirnov, 
\underline{Lie algebras graded by the root system}
\underline{$BC_1$}, 
J. Lie Theory {\bf 13} (2003), no. 1, 91–132. MR 1958577
\bibitem[Bil99]{Billig1999}
Y. Billig, 
\underline{An extension of the Korteweg-de Vries hierarchy arising from a rep-}
\underline{resentation of a toroidal Lie algebra}, 
J. Algebra {\bf 217} (1999), no. 1, 40–64. MR 1700475
\bibitem[Bou81]{Bourbaki(Lie)4-6}
N. Bourbaki,
\underline{\'{E}l\'{e}ments de math\'{e}matique}, 
Masson, Paris, 1981, Groupes et alg\`{e}bres de Lie. Chapitres 4, 5 et 6. 
(Lie groups and Lie algebras. Chapters 4, 5 and 6.) MR 647314
\bibitem[Bri71]{Brieskorn1971}
E. Brieskorn, 
\underline{Singular elements of semi-simple algebraic groups}, 
Actes du Congr\`{e}s International des Math\'{e}maticiens (Nice, 1970), 
Tome 2, Gauthier-Villars, Paris, 1971, pp. 279–284. MR 0437798
\bibitem[BT72]{BruhatTits1972}
F. Bruhat and J. Tits, 
\underline{Groupes r\'{e}ductifs sur un corps local}, 
Inst. Hautes \'{E}tudes Sci. Publ. Math. (1972), no. 41, 5–251. MR 327923
\bibitem[Car94]{Cartan1894}
\'{E}. Cartan, 
\underline{Sur la structure des groupes de transformations finis et continus.}, 
Th\`{e}se pr\'{e}sent\'{e}e \`{a} la Facult\'{e} des Sciences. $4^{\circ}$. 156 S. Paris. 
Nony et Co. (1894)., 1894.
\bibitem[Car05]{Carter2005}
R. W. Carter, 
\underline{Lie algebras of finite and affine type}, 
Cambridge Studies in Advanced Mathematics, vol. 96, Cambridge University Press, 
Cambridge, 2005. MR 2188930
\bibitem[Col89]{Coleman1989}
A. J. Coleman, 
\underline{Killing and the Coxeter transformation of Kac-Moody algebras}, 
Invent. Math. {\bf 95} (1989), no. 3, 447–477. MR 979360
\bibitem[Cox34]{Coxeter1934}
H. S. M. Coxeter,
\underline{Discrete groups generated by reflections}, 
Ann. of Math. (2) {\bf 35} (1934), no. 3, 588–621. MR 1503182
\bibitem[Cox51]{Coxeter1951}
H. S. M. Coxeter, 
\underline{The product of the generators of a finite group generated by}
\underline{reflections}, 
Duke Math. J. {\bf 18} (1951), 765–782. MR 45109
\bibitem[DV34a]{DuVal1934-I}
P. Du Val, 
\underline{On isolated singularities of surfaces which do not affect the} 
\underline{conditions of adjunction. I}, 
Proc. Camb. Philos. Soc. {\bf 30} (1934), 453–459 (English).
\bibitem[DV34b]{DuVal1934-II}
P. Du Val, 
\underline{On isolated singularities of surfaces which do not affect the} 
\underline{conditions of adjunction. II.}, 
Proc. Camb. Philos. Soc. {\bf 30} (1934), 460–465 (English).
\bibitem[DV34c]{DuVal1934-III}
P. Du Val, 
\underline{On isolated singularities of surfaces which do not affect the} 
\underline{conditions of adjunction. III.}, 
Proc. Camb. Philos. Soc. {\bf 30} (1934), 483–491 (English).
\bibitem[HS04]{HelmkeSlodowy2004}
S. Helmke and P. Slodowy, 
\underline{Loop groups, elliptic singularities and principal}
\underline{bundles over elliptic curves}, 
Geometry and topology of caustics—CAUSTICS ’02, Banach Center Publ., vol. 62, 
Polish Acad. Sci. Inst. Math., Warsaw, 2004, pp. 87–99. MR 2055872
\bibitem[HS05]{HelmkeSlodowy2005}
S. Helmke and P. Slodowy, 
\underline{Singular elements of affine Kac-Moody groups}, 
European Congress of Mathematics, Eur. Math. Soc., Zürich, 2005, pp. 155–172. 
MR 2185743
\bibitem[IM65]{IwahoriMatsumoto1965}
N. Iwahori and H. Matsumoto, 
\underline{On some Bruhat decomposition and the} 
\underline{structure of the Hecke rings of $p$-adic Chevalley groups}, 
Inst. Hautes \'{E}tudes Sci. Publ. Math. (1965), no. 25, 5–48. MR 185016
\bibitem[ISW99a]{IoharaSaitoWakimoto1999a}
K. Iohara, Y. Saito, and M. Wakimoto, 
\underline{Hirota bilinear forms with 2-toroidal} 
\underline{symmetry}, Phys. Lett. A {\bf 254} (1999), no. 1-2, 37–46. MR 1688100
\bibitem[ISW99b]{IoharaSaitoWakimoto1999b}
K. Iohara, Y. Saito, and M. Wakimoto, 
\underline{Notes on differential equations arising} 
\underline{from a representation of 2-toroidal Lie algebras}, 
Progr. Theoret. Phys. Suppl. (1999), no. 135, 166–181, 
Gauge theory and integrable models (Kyoto, 1999). MR 1747990
\bibitem[Kac69]{Kac1969}
V. G. Kac, 
\underline{Automorphisms of finite order of semisimple Lie algebras}, 
Funkcional. Anal. i Prilo\v{z}en. {\bf 3} (1969), no. 3, 94–96. MR 0251091
\bibitem[Kac90]{Kac1990}
V. G. Kac, 
\underline{Infinite-dimensional Lie algebras, third ed}., Cambridge University
Press, Cambridge, 1990. MR 1104219
\bibitem[KP84]{KacPeterson1984}
V. G. Kac and D. H. Peterson, 
\underline{Infinite-dimensional Lie algebras, theta} 
\underline{functions and modular forms}, 
Adv. in Math. 53 (1984), no. {\bf 2}, 125–264. MR 750341
\bibitem[Kil88]{Killing1888}
W. Killing, 
\underline{Die Zusammensetzung der stetigen endlichen Transformations-}
\underline{gruppen. I.}, Math. Ann. {\bf 31} (1888), 252–290 (German).
\bibitem[Kil89a]{Killing1889a}
W. Killing, 
\underline{Die Zusammensetzung der stetigen endlichen Transformations-}
\underline{gruppen. II.}, Math. Ann. {\bf 33} (1889), 1–48 (German).
\bibitem[Kil89b]{Killing1889b}
W. Killing,
\underline{Die Zusammensetzung der stetigen endlichen Transformations-}
\underline{gruppen. III.}, Math. Ann. {\bf 34} (1889), 57–122 (German).
\bibitem[Kil90]{Killing1890}
W. Killing,
\underline{Die Zusammensetzung der stetigen endlichen Transformations-}
\underline{gruppen. (IV, Schluss.).}, Math. Ann. {\bf 36} (1890), 161–189 (German).
\bibitem[Kob93]{Koblitz1993}
N. Koblitz, 
\underline{Introduction to elliptic curves and modular forms}, 2nd ed., 
GTM {\bf 97}, Springer-Verlag, New York, 1993. x+248 pp. MR1216136
\bibitem[Kos59]{Kostant1959}
B. Kostant,
\underline{The principal three-dimensional subgroup and the Betti numbers}
\underline{of a complex simple Lie group}, Amer. J. Math. {\bf 81} (1959), 973–1032. 
MR 114875
\bibitem[Lek83]{Lek1983}
H. van der Lek, 
\underline{The homotopy type of complex hyperplane complements}, 
Doctor's Thesis, Nijmegen 1983.
\bibitem[Loo76]{Looijenga1976}
E. Looijenga, 
\underline{Root Systems and elliptic curves}, 
Invent. Math., {\bf 38} (1976), no. 1, 17–32. MR0466134
\bibitem[Loo78]{Looijenga1978}
E. Looijenga, 
\underline{On the semi-universal deformation of a simple elliptic singu-}
\underline{larity II}, 
Topology {\bf 17} (1978), 23-40. MR0492380
\bibitem[Loo80]{Looijenga1980}
E. Looijenga, 
\underline{Invariant theory for generalized root systems}, 
Invent. Math., {\bf 61} (1980), no. 1, 1–32. MR0587331
\bibitem[Mac72]{Macdonald1972}
I. G. Macdonald, 
\underline{Affine root systems and Dedekind’s $\eta$-function}, 
Invent. Math. {\bf 15} (1972), 91–143. MR 357528
\bibitem[Mac03]{Macdonald2003}
I. G. Macdonald,
\underline{Affine Hecke algebras and orthogonal polynomials}, 
Cambridge Tracts in Mathematics, vol. 157, Cambridge University Press, 
Cambridge, 2003. MR 1976581
\bibitem[Moo68]{Moody1968}
R. V. Moody, 
\underline{A new class of Lie algebras}, 
J. Algebra {\bf 10} (1968), 211–230. MR 229687
\bibitem[Moo69]{Moody1969}
R. V. Moody,
\underline{Euclidean Lie algebras}, 
Canadian J. Math. {\bf 21} (1969), 1432–1454. MR 255627
\bibitem[Pol94]{Pollmann1994}
U. Pollmann,
\underline{Realisation der biaffinen Wurzelsysteme von Saito in Lie-}
\underline{Algebren}, 
Master’s thesis, University of Hamburg, Germany, 1994.
\bibitem[Sai74]{Saito1974}
K. Saito, 
\underline{Einfach-elliptische Singularit\"{a}ten}, 
Invent. Math. {\bf 23} (1974), 289–325. MR 354669
\bibitem[Sai85]{Saito1985}
K. Saito, 
\underline{Extended affine root systems. I. Coxeter transformations}, 
Publ. Res. Inst. Math. Sci. {\bf 21} (1985), no. 1, 75–179. MR 780892
\bibitem[Sai90]{Saito1990}
K. Saito, 
\underline{Extended affine root systems. II. Flat Invarinats}, 
Publ. Res. Inst. Math. Sci. {\bf 26} (1990), no. 1, 15–78. MR 1053908
\bibitem[Sai01]{Saito2001}
K. Saito, 
\underline{Extended affine root systems. V. Elliptic eta-products and their} 
\underline{Dirichlet series}, 
Proceedings on Moonshine and related topics (Montréal, QC, 1999), 
CRM Proc. Lecture Notes, vol. 30, Amer. Math. Soc., Providence, RI, 2001, 
pp. 185–222. MR 1881609
\bibitem[Sat95]{Satake1995}
I. Satake, 
\underline{Automorphisms of the extended affine root system and modular} 
\underline{property for the flat theta invariants}, 
Publ. Res. Inst. Math. Sci. {\bf 31} (1995), no. 1, 1–32. MR 1317520
\bibitem[Slo80a]{Slodowy1980a}
P. Slodowy, 
\underline{Four lectures on simple groups and singularities}, 
Communications of the Mathematical Institute, Rijksuniversiteit Utrecht, vol. 11, 
Rijksuniver- siteit Utrecht, Mathematical Institute, Utrecht, 1980. MR 563725
\bibitem[Slo80b]{Slodowy1980b}
P. Slodowy, 
\underline{Simple singularities and simple algebraic groups}, 
Lecture Notes in Mathematics, vol. 815, Springer, Berlin, 1980. MR 584445
\bibitem[Slo85]{Slodowy1985}
P. Slodowy, 
\underline{A character approach to Looijenga's invariant theory for genera-}
\underline{lized root systems}, 
Compositio Math. {\bf 55} (1985), no. 1, 3–32. MR0791645
\bibitem[SS09]{SaitoShiota2009}
Y. Saito and M. Shiota, 
\underline{On Hecke algebras associated with elliptic root systems} 
\underline{and the double affine Hecke algebras}, 
Publ. Res. Inst. Math. Sci. 45 (2009), no. {\bf 3}, 845–905. MR 2569569
\bibitem[ST97]{SaitoTakebayashi1997}
K. Saito and T. Takebayashi, 
\underline{Extended affine root systems. III. Elliptic Weyl} 
\underline{groups}, 
Publ. Res. Inst. Math. Sci. {\bf 33} (1997), no. 2, 301–329. MR 1442503
\bibitem[Tit79]{Tits1979}
J. Tits, 
\underline{Reductive groups over local fields}, 
Automorphic forms, representations and $L$-functions (Proc. Sympos. Pure Math., 
Oregon State Univ., Corvallis, Ore., 1977), Part 1, Proc. Sympos. Pure Math., XXXIII, 
Amer. Math. Soc., Providence, R.I., 1979, pp. 29–69. MR 546588
\bibitem[Vin71]{Vinberg1971}
\`{E}. B. Vinberg, 
\underline{Discrete linear groups that are generated by reflections}, 
Izv. Akad. Nauk SSSR Ser. Mat. {\bf 35} (1971), 1072–1112. MR 0302779
\end{thebibliography}

\printindex[notations]
\printindex[index]

\end{document}